\documentclass[a4paper,twoside]{article}

\usepackage[utf8]{inputenc}
\usepackage[top=17mm,bottom=28mm,hmarginratio={1:1}]{geometry}
\usepackage{hyperref}
\usepackage{amsmath,amsfonts,amsthm,amssymb}
\usepackage{graphicx}
\usepackage{array,paralist}
\usepackage{caption}
\usepackage{subcaption}
\usepackage{float}
\usepackage{multirow}
\usepackage{quiver}
\usepackage{rotating}
\usepackage{xspace}

\theoremstyle{plain}
\newtheorem{theorem}{Theorem}[section]
\newtheorem{proposition}[theorem]{Proposition}
\newtheorem{lemma}[theorem]{Lemma}
\newtheorem{corollary}[theorem]{Corollary}

\theoremstyle{definition}
\newtheorem{definition}[theorem]{Definition}

\newcommand{\OR}{orientation-reversing\xspace}
\newcommand{\OP}{orientation-preserving\xspace}
\newcommand{\ena}{enantiomorphic\xspace}
\newcommand{\cry}{crystallographic\xspace}
\newcommand\R{\ensuremath{\mathbb{R}}}
\newcommand\T{\ensuremath{\mathbb{T}}}
\renewcommand\H{\ensuremath{\mathcal{H}}}
\renewcommand\O{\ensuremath{\mathrm{O}}}
\newcommand\Orth[1]{\ensuremath{\mathrm{O}(#1)}}
\newcommand\SO{\ensuremath{\mathrm{SO}}}
\newcommand\dist{\ensuremath{\mathrm{dist}}}
\newcommand\orbit{\ensuremath{\mathrm{orbit}}}
\newcommand\id{\ensuremath{\mathrm{id}}}
\newcommand\w{\ensuremath{\omega}}
\let \phi=\varphi
\DeclareMathOperator{\lcm}{lcm}
\DeclareMathOperator{\vor}{Vor}
\DeclareMathOperator{\diag}{diag}

\newdimen\symsize
\newbox\slashbox
\newbox\backslashbox

\newcommand\symB[1]{\hbox{$\vcenter{\hrule
      \hbox{\vrule height \symsize \kern\symsize \vrule}\hrule}\kern-0.4pt \kern-\symsize
          \hbox to \symsize{\hss$
      \if .#1\cdot\else
      \if X#1\times\else
      \if 1#1\else 
      \if /#1\vcenter{\copy\slashbox}\else
      \if |#1\vcenter{\hrule height\symsize width 0,4pt}\else
      \if \setminus#1\vcenter{\copy\backslashbox}
           \else
      \if *#1\hbox to \wd0{\hss$+$\hss}\kern-\wd0
             \hbox to \wd0{\hss$\times$\hss}\else
      \if L#1\raise 0,55pt\hbox{$\scriptstyle\circlearrowleft$}\else
      \if R#1\raise 0,55pt\hbox{$\scriptstyle\circlearrowright$}\else
      #1\fi\fi\fi\fi\fi\fi\fi\fi\fi
      $\hss}\kern0.4pt$}}

\newcommand\sym{\raise 0,5pt \symB} 

\newcommand\grp[1]{\ensuremath{\mskip 1mu\vcenter{\hrule
      \hbox{\vrule \kern 0,6pt
        $\vcenter{\kern 0,6pt \symB#1 \kern 0,6pt}$\kern
        0,6pt\vrule}\hrule}\mskip 1mu}}

\newcommand*\samethanks[1][\value{footnote}]{\footnotemark[#1]}

\hypersetup{
    pdftitle = {Towards a Geometric Understanding of the 4-Dimensional Point Groups},
    pdfauthor = {Laith Rastanawi and Günter Rote}
  }

\pdfoutput=1
\newif\iffigurescompact
\figurescompacttrue

\renewcommand{\subsectionmark}[1]{}

\begin{document}

\setbox\slashbox=\hbox{\rotatebox{-45}{\vrule height 1,4143\symsize
    width 0,42pt }}
\setbox\backslashbox=\hbox{\rotatebox{45}{\vrule height 1,4143\symsize
   width 0,42pt }}

\title{Towards a Geometric Understanding\\of the 4-Dimensional Point Groups}
\author{Laith Rastanawi and Günter Rote}

\begin{small}
\author{Laith Rastanawi\thanks{
            Supported by the DFG Research Training Group
            GRK 2434 ``Facets of Complexity''
        }\\[4pt]
\parbox{0.36\linewidth}{\normalsize\centering
        Institut f\"ur Mathematik\\
        Freie Universität Berlin\\
        Arnimallee 2\\
        14195 Berlin, Germany\\[1pt]
        \nolinkurl{laith.rastanawi@fu-berlin.de}}
        \and
        G\"unter Rote\samethanks\\[4pt]
\parbox{0.36\linewidth}{\normalsize\centering
        Institut f\"ur Informatik\\
        Freie Universität Berlin\\
        Takustra\ss e 9\\
        14195 Berlin, Germany\\[1pt]
        \nolinkurl{rote@inf.fu-berlin.de}}
    }
\end{small}

\maketitle

\begin{abstract}
We classify the finite groups of orthogonal
transformations in 4-space, and we study these groups from the
viewpoint of their geometric action,
using polar orbit polytopes.
For one type of groups (the toroidal groups), we develop a new
classification based on their action on an invariant torus, while we
rely on classic results for the remaining groups.

As a tool, we develop a convenient parameterization of the oriented great
circles on the 3-sphere, which leads to (oriented) Hopf fibrations in a
natural way.
\end{abstract}

 \tableofcontents
\markboth{Laith Rastanawi and Günter Rote: 4-Dimensional Point Groups}
         {Laith Rastanawi and Günter Rote: 4-Dimensional Point Groups}
\listoftables
\markboth{Laith Rastanawi and Günter Rote: 4-Dimensional Point Groups}
         {Laith Rastanawi and Günter Rote: 4-Dimensional Point Groups}
         
\section{Introduction and Results}
\label{sec:intro}

A \emph{$d$-dimensional point group} is a finite group of orthogonal
transformations in $\R^d$, or in other words, a finite subgroup of $\Orth{d}$.
We propose the following classification for the $4$-dimensional point groups.

\begin{theorem}
\label{classification}
The 4-dimensional point groups can be classified into
\begin{itemize}
    \item 25 polyhedral groups (Table~\ref{tab:polyhedral}),
    \item 21 axial groups (7 pyramidal groups, 7 prismatic groups,
        and 7 hybrid groups, Table~\ref{tbl:axial_full}),
    \item 22 one-parameter families of tubical groups
        (11 left tubical groups and 11 right tubical groups,
        Table~\ref{tbl:left_tubical_groups}), and
    \item 25 infinite families of toroidal groups
        (Table~\ref{tab:overview}), among them
        \begin{itemize}
            \item 2 three-parameter families,
            \item 19 two-parameter families, and
            \item 4 one-parameter families.
        \end{itemize}
\end{itemize}
\end{theorem}

In contrast to earlier classifications of these groups
(notably by Du Val in 1962 \cite{duval} and by
Conway and Smith in 2003 \cite{CS}), see Section~\ref{sec:previous}),
we emphasize a geometric viewpoint,
trying to visualize and understand actions of these groups.
Besides, we correct some omissions, duplications,
and mistakes in these classifications.

\paragraph{Overview of the groups.}
The 25 \emph{polyhedral} groups are related to the regular polytopes.
The symmetries of the regular polytopes are well understood,
because they are generated by reflections,
and the classification of such groups as Coxeter groups is classic.
We will deal with these groups only briefly,
dwelling a little on just a few groups that come in \ena pairs
(i.e., groups that are not equal to their own mirror.)

The 21 \emph{axial} groups are those that keep one axis fixed.
Thus, they essentially operate in the three dimensions
perpendicular to this axis
(possibly combined with a flip of the axis),
and they are easy to handle,
based on the well-known classification of
the three-dimensional point groups.

The \emph{tubical} groups are characterized as those that
have (exactly) one Hopf bundle invariant.
They come in left and right versions
(which are mirrors of each other)
depending on the Hopf bundle they keep invariant.
They are so named because they arise with
a decomposition of the 3-sphere into tube-like structures
(discrete Hopf fibrations).

The \emph{toroidal} groups are characterized as having an invariant torus.
This class of groups is where our main contribution in terms of
the completeness of the classification lies.
We propose a new, geometric, classification of these groups.
Essentially, it boils down to classifying the isometry groups of
the two-dimensional square flat torus.

We emphasize that, 
regarding the completeness of the classification,
in particular concerning the polyhedral and tubical groups,
we rely on the classic approach (see Section~\ref{sec:classic}).
Only for the toroidal and axial groups,
we supplant the classic approach by our geometric approach.

\paragraph{Hopf fibrations.}
We give a self-contained presentation of Hopf fibrations (Section~\ref{sec:hopf}).
In many places in the literature,
one particular Hopf map is introduced as ``the Hopf map'',
either in terms of four real coordinates or two complex coordinates,
leading to ``the Hopf fibration''.
In some sense, this is justified,
as all Hopf bundles are (mirror-)congruent.
However, for our characterization,
we require the full generality of Hopf bundles.
As a tool for working with Hopf fibrations, 
we introduce a parameterization for great circles in $S^3$,
which might be useful elsewhere.

\paragraph{Orbit polytope.}
Our main tool to understand tubical groups are polar orbit polytopes.
(Section~\ref{sec:orbit-polytopes}).
In particular,
we study the symmetries of a cell of the polar orbit polytope
 for different starting points.

\section{Orbit Polytopes}
\label{sec:orbit-polytopes}

\subsection{Geometric understanding through orbit polytopes: the
  pyritohedral group}
\label{subsec:orbit_polytopes}

One can try to visualize a point group $G \leqslant \O(d)$
by looking at the orbit of some point $0 \ne v \in \R^d$ and
taking the convex hull.
This is called the \emph{$G$-orbit polytope} of $v$.
For an in-depth study of orbit polytopes and their symmetries,
refer to \cite{FL16,FL18}.

The {orbit polytope} will usually depend on the choice of $v$,
and it may have other symmetries in addition to those of $G$.
For example, the $C_n$-orbit polytope in the plane is
always a regular $n$-gon, and this orbit polytope has the larger
dihedral group
$D_{2n}$ as its symmetry group.

\label{sec:pyrit}

We will illustrate the usefulness of orbit polytopes
with a three-dimensional example.
The pyritohedral group
 is perhaps the most interesting among the point groups in 3 dimensions.
It is generated by a cyclic rotation of the
coordinates $(x_1,x_2,x_3)\mapsto(x_2,x_3,x_1)$ and by the coordinate
reflection $(x_1,x_2,x_3)\mapsto(-x_1,x_2,x_3)$.
It has order 24.
Figure~\ref{fig:pyrit} shows a few examples of orbit polytopes for this
group, and
their polars.
The elements of the pyritohedral group are simultaneously symmetries of the octahedron
(where it is an index-2 subgroup of the full symmetry group)
and the icosahedron (an index-5 subgroup),
and of course of their polars, the cube and the dodecahedron.
The group contains reflections, but it is not generated by its reflections.

The orbit of the points $(1,0,0)$ and $(1,1,1)$ generate the regular
octahedron and the cube, respectively.
These are each other's polars,
but they don't give any specific information about the
pyritohedral group.

Figure~\ref{subfig:pyrit_generic}
shows the orbit polytope (in yellow) of a generic point
$(\frac23,\frac12,1)$, and its polar (in orange).
The symmetries of these polytopes are exactly the pyritohedral group.
That orbit polytope has 6 rectangular faces
(lying in planes of the faces of a cube),
8 equilateral triangles (lying in the faces of an octahedron),
and 12 trapezoids
(going through the %
edges of some cube, but not of some regular octahedron).
The polar has 24 quadrilateral faces,
corresponding to the 24 group elements.
For any pair of faces,
there is a unique symmetry of the polytope that
maps one face to the other.\footnote{In mineralogy,
this shape is sometimes called a \emph{diploid}, and
\emph{diploidal symmetry} is an alternative name for pyritohedral symmetry.
In our context, the term diploid will show up in a different sense.}

If we choose one coordinate of the starting point to be 0,
the rectangles shrink to line segments,
and the trapezoids become isosceles triangles.
See Figure~\ref{subfig:pyrit_generic_mirror}.
The orbit polytope is an icosahedron with 20 triangular faces:
8 equilateral triangles and 12 isosceles triangles.
The polar polytope is a \emph{pyritohedron}, that is,
a dodecahedron with 12 equal but not necessarily regular pentagons.
For this choice, the orbit contains only 12 points,
but the polytope gains no additional symmetries beyond the
pyritohedral symmetries.
However, for $(0, \frac{\sqrt{5}-1}{2}, 1)$,
we get the regular icosahedron and the regular
dodecahedron.
For the specific choice $(0, \frac{1}{2}, 1)$,
the polar orbit polytope is
one of the crystal forms of the mineral pyrite,
which gave the polytope and group its name,
see Figure~\ref{subfig:pyrit_generic_mirror}.
This polytope is also an \emph{alternahedron} on $4$ symbols~\cite{CK2006_f2}.
An alternahedron can be constructed as the orbit of a generic point
$(x_1,x_2,x_3,x_4) \in \mathbb R^4$
under all even permutations.
Since the points lie in a hyperplane
$x_1+x_2+x_3+x_4 = \mathrm{const}$,
this is a three-dimensional polytope.
For the starting point $(0, 1,2)$,
we obtain %
the alternahedron
that results from the canonical choice
$(x_1,x_2,x_3,x_4)=(1,2,3,4)$, a scaled copy of
Figure~\ref{subfig:pyrit_generic_mirror}.\footnote
{The illustration of this polytope in
\cite[Fig.~1]{CK2006_f2} may make the wrong impression of consisting of
equilateral triangles only. However,
its isosceles faces have base length $2$ and two
equal legs of length $\sqrt6 \approx 2.45$.}

The pyritohedral group differs from the symmetries of the cube (or the
octahedron) by allowing only even permutations of the coordinates
$x_1,x_2,x_3$.
When two coordinates are equal, this distinction plays no role, and
the resulting polyhedron will have all symmetries of the cube,
see Figure~\ref{subfig:pyrit_nice_triangulation}.
(We mention that some
special starting points of this form lead to Archimedean polytopes:
The starting point $(1, 1, \sqrt{2}+1)$ generates
a rhombicuboctahedron with 8 regular triangles and 18 squares;
 $(0, 1, 1)$ generates the cuboctahedron
with 8 regular triangles and 6 squares;
with $(\frac{1}{\sqrt{2}+1}, 1, 1)$,
we get the truncated cube with 8 regular triangles and 8 regular
octagons,
similar to the yellow polytope in Figure~\ref{subfig:pyrit_nice_triangulation}.)

For the purpose of visualizing the pyritohedral group,
we will try to keep the three coordinates distinct.
By choosing the point close to $(1,1,1)$ or $(0,0,1)$, we can
emphasize the cube-like or the octahedron-like appearance of the orbit
polytope or its polar.
For example, the polar orbit polytope for $(0,\frac{1}{10},1)$
resembles a cube whose squares are subdivided into rectangles,
like the orange polytope in Figure~\ref{subfig:pyrit_rectangles}.
(Actually, the mineral pyrite has sometimes a cubic crystal form in which
the faces carry parallel thin grooves, %
so-called \emph{striations}.\footnote{See
\url{http://www.mineralogische-sammlungen.de/Pyrit-gestreift-engl.html}})
See also Figure~\ref{subfig:pyrit_close_cube}
for $(\frac{2}{10}, \frac{1}{10}, 1)$.
The orbit polytope in Figure~\ref{subfig:pyrit_rectangles} appears like an octahedron
whose edges have been shaved off, but in an asymmetric way that provides
a direction for the edges
(see Figure~\ref{fig:truncated-cube}a on
p.~\pageref{fig:truncated-cube} in Section~\ref{sec:24-cell}).

On the other hand, the polar orbit polytope for
$(\frac{8}{10},\frac{9}{10},1)$ resembles an octahedron, carrying
a pinwheel-like structure on every face.
See Figure~\ref{subfig:pyrit_close_octahedron}.

\begin{figure}[htb]
    \begin{subfigure}{0.5\textwidth}
        \centering
        \begin{minipage}{0.5\textwidth}
        \includegraphics[width=3cm]{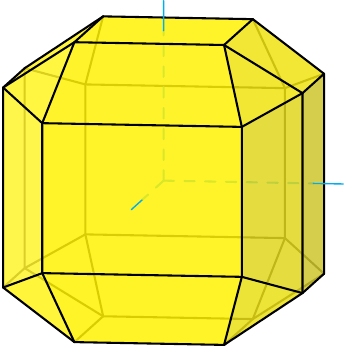}
        \end{minipage}%
        \begin{minipage}{0.5\textwidth}
        \includegraphics[width=3cm]{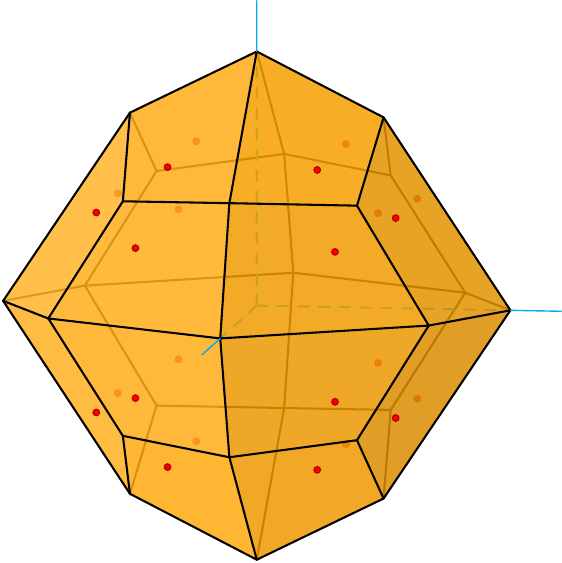}
        \end{minipage}
        \caption{$(\frac23, \frac12, 1)$}
        \label{subfig:pyrit_generic}
    \end{subfigure}%
    \begin{subfigure}{0.5\textwidth}
        \centering
        \begin{minipage}{0.5\textwidth}
        \includegraphics[width=3cm]{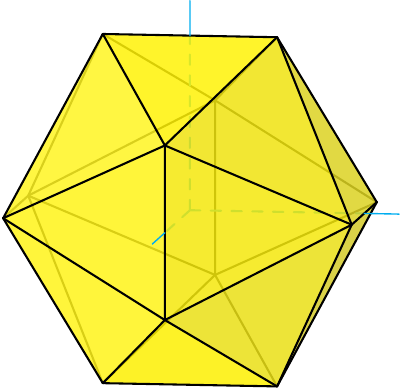}
        \end{minipage}%
        \begin{minipage}{0.5\textwidth}
        \includegraphics[width=3cm]{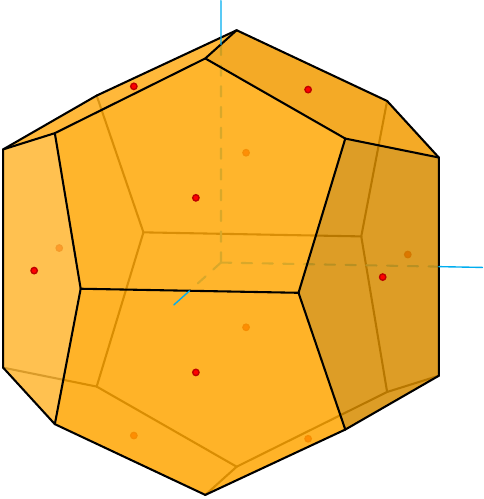}
        \end{minipage}
        \caption{$(0, \frac12, 1)$}
        \label{subfig:pyrit_generic_mirror}
    \end{subfigure}
    \begin{subfigure}{0.5\textwidth}
        \centering
        \begin{minipage}{0.5\textwidth}
        \includegraphics[width=3cm]{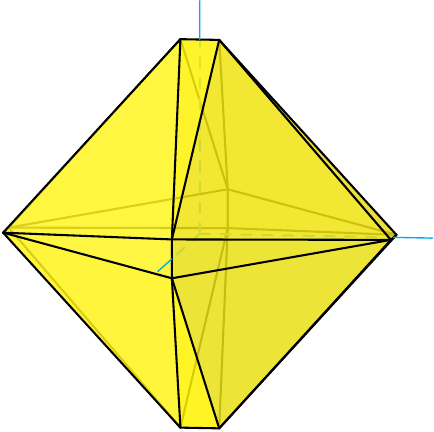}
        \end{minipage}%
        \begin{minipage}{0.5\textwidth}
        \includegraphics[width=3cm]{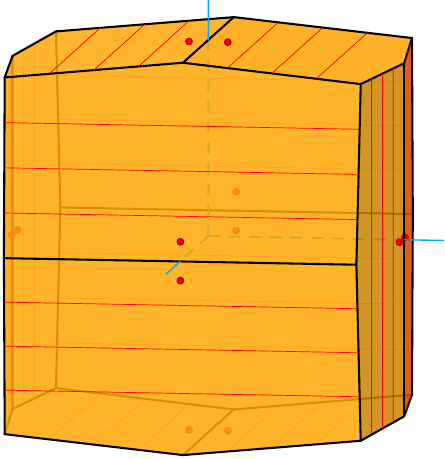}
        \end{minipage}
        \caption{$(0, \frac{1}{10}, 1)$}
        \label{subfig:pyrit_rectangles}
    \end{subfigure}%
    \begin{subfigure}{0.5\textwidth}
        \centering
        \begin{minipage}{0.5\textwidth}
        \includegraphics[width=3cm]{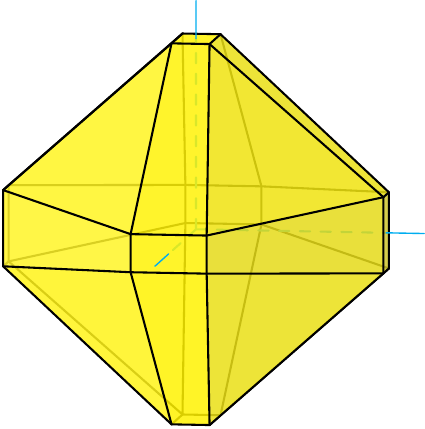}
        \end{minipage}%
        \begin{minipage}{0.5\textwidth}
        \includegraphics[width=3cm]{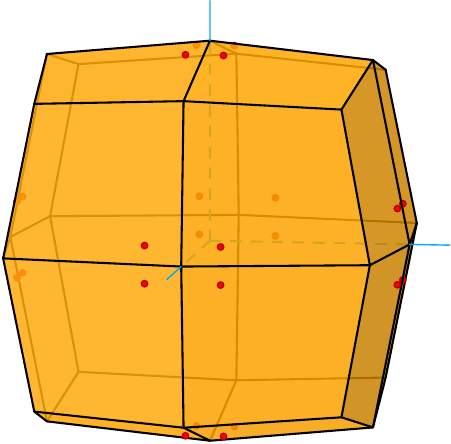}
        \end{minipage}
        \caption{$(\frac{2}{10}, \frac{1}{10}, 1)$}
        \label{subfig:pyrit_close_cube}
    \end{subfigure}
    \begin{subfigure}{0.5\textwidth}
        \centering
        \begin{minipage}{0.5\textwidth}
        \includegraphics[width=3cm]{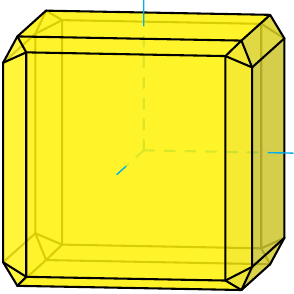}
        \end{minipage}%
        \begin{minipage}{0.5\textwidth}
        \includegraphics[width=3cm]{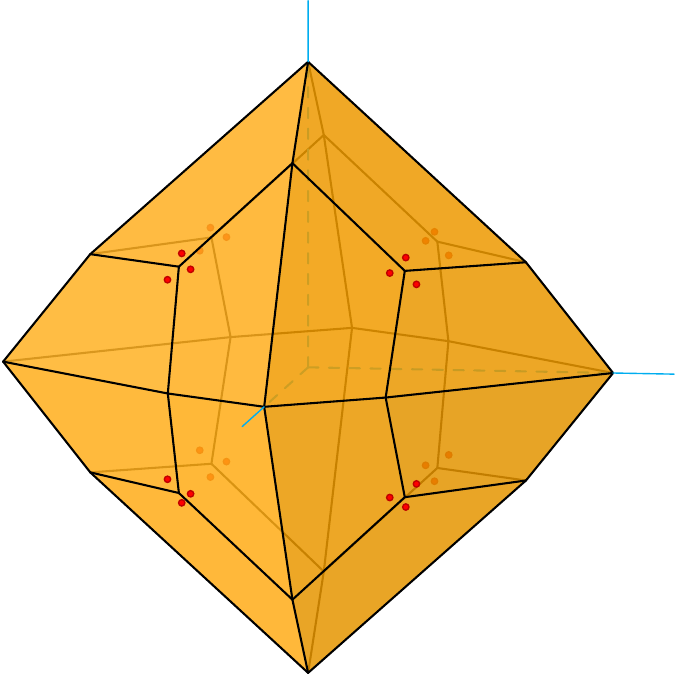}
        \end{minipage}
        \caption{$(\frac{8}{10}, \frac{9}{10}, 1)$}
        \label{subfig:pyrit_close_octahedron}
    \end{subfigure}
    \begin{subfigure}{0.5\textwidth}
        \centering
        \begin{minipage}{0.5\textwidth}
        \includegraphics[width=3cm]{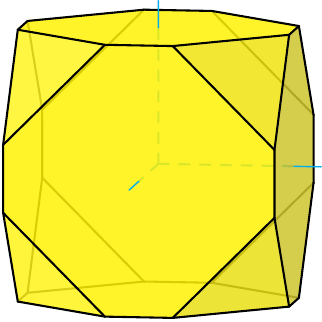}
        \end{minipage}%
        \begin{minipage}{0.5\textwidth}
        \includegraphics[width=3cm]{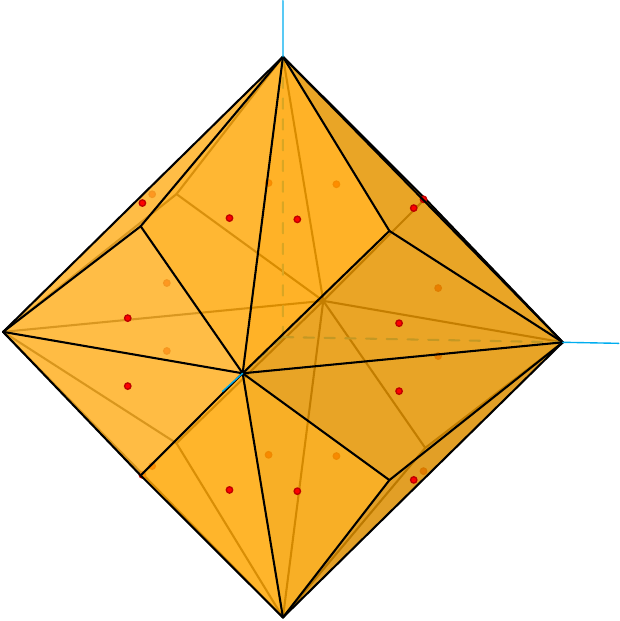}
        \end{minipage}
        \caption{$(\frac{1}{4}, 1, 1)$}
        \label{subfig:pyrit_nice_triangulation}
    \end{subfigure}%
    \caption{Orbit polytopes of the pyritohedral group (yellow, on the
      left)
    and their polar polytopes (orange, on the right) for various starting
    points. The pictures are rescaled to uniform size; the scale is
    not maintained between the pictures.}
  \label{fig:pyrit}
\end{figure}

\subsubsection{The pyritohedral group for flatlanders}
\label{sec:pyrit-flat}
We will be in the situation that we try to visualize 4-dimensional
point groups through orbit polytopes or their polars.
So let us go one dimension lower and imagine that we,
as ordinary three-dimensional people, would like to
explain the pyritohedral group to flatlanders.
We will see that different options have different merits, and
there may be no unique best way of visualizing a group.

Assuming that flatlanders accept the notions of a cube or an
octahedron,
we could tell them that we build a cube whose squares are striped in
such a way that the patterns on adjacent squares never abut, similar to
 the orange polytope in Figure~\ref{subfig:pyrit_rectangles}.
It is allowed to map any square to any other square (6 possibilities) in
such a way that the stripes match (the dihedral group $D_4$ with 4
possibilities, for a total of 24 transformations).

Alternatively, we could tell them that
the edges of an octahedron are oriented such that
each triangle forms a directed cycle
(Figure~\ref{fig:truncated-cube}a on
p.~\pageref{fig:truncated-cube}). %
It is allowed to map any triangle to any other triangle (8 possibilities) in
such a way that edge directions are preserved (the cyclic group $C_3$ with 3
possibilities, for a total of 24 transformations).

Another option is the polar of $(c,1,1)$, where $c \not\in \{0, 1\}$,
see the orange polytope in Figure~\ref{subfig:pyrit_nice_triangulation}.
It has 24 isosceles triangles, one per group element,
As $c$ approaches 1 or 0,
the polar orbit polytope converges to an octahedron
or to a rhombic dodecahedron.
As a shape, the triangle does not reveal much about the group,
so we have to add the information that the base edge acts as a mirror,
and the opposite vertex is a 3-fold \emph{gyration point},
i.e., there are three rotated copies that fit together.
(This is essentially what is expressed in the orbifold notation $3{*}2$.)
We are not allowed to use the reflection that maps the triangle to itself,
and we might indicate this by placing an arrow along the base edge.

In most cases, it was advantageous to describe the group in terms of
the polar orbit polytope: We have many copies of one shape, and any
shape can be mapped to any other. It is not necessarily the best
option to insist that all points of the orbit are distinct.
Sometimes it is
preferable to allow also symmetries within each face. In this case,
the information, which of these symmetries are in the group must be
conveyed as side information, for example by decorations or patterns that should
be left invariant, such as the stripes
 in Figure~\ref{subfig:pyrit_rectangles}.

\begin{figure}[htb]\centering
  \fboxsep=5mm
  \fbox
  {\qquad
    \includegraphics{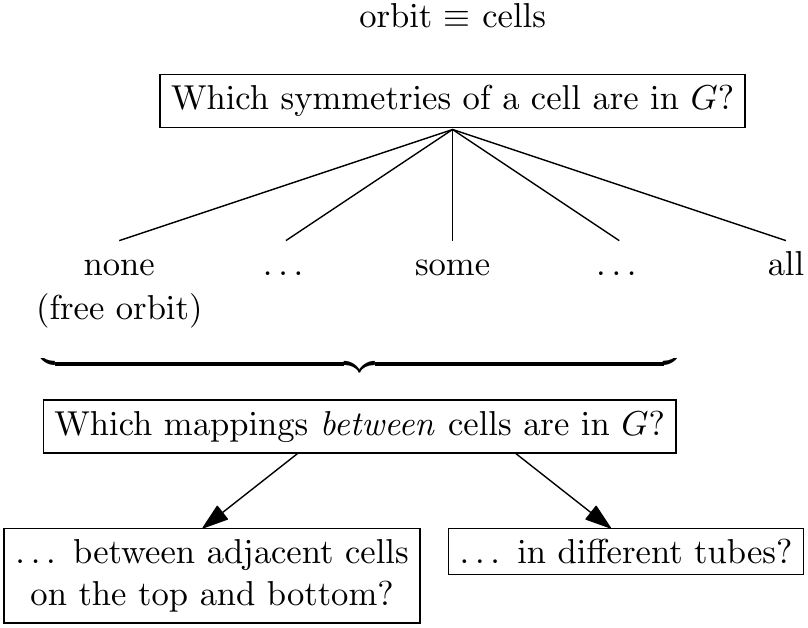}
    \qquad
  }
  \caption{Geometric understanding of a group $G$ through is polar orbit polytope}
  \label{fig:geometric-understanding}
\end{figure}

Figure~\ref{fig:geometric-understanding} summarizes the relation
between a polar orbit polytope and its group $G$. All cells are equal,
and the cells correspond to the points of the orbit. We know that
between any two cells, there is at least one transformation in $G$
that carries one cell to the other. However, it is not directly apparent
\emph{which} transformations
 carry one cell to another cell, or to itself. If all symmetries of a
 cell belong to the group, the answer is clear; otherwise we have to
 discuss this question and describe the answer separately.

 The bottom row of
 Figure~\ref{fig:geometric-understanding} splits this question into
 two subproblems that are relevant only for tubical groups (Section~\ref{sec:tubical}), namely the
 relation between adjacent cells in a tube, and between cells of
 different tubes.

\begin{figure}[htb]
    \begin{subfigure}{0.5\textwidth}
    \centering
    \includegraphics[width=4cm]{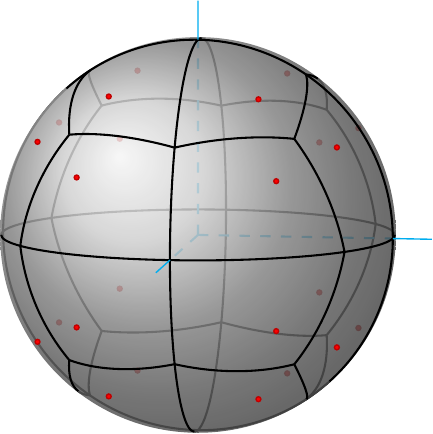}
    \end{subfigure}%
    \begin{subfigure}{0.5\textwidth}
    \centering
    \includegraphics[width=4cm]{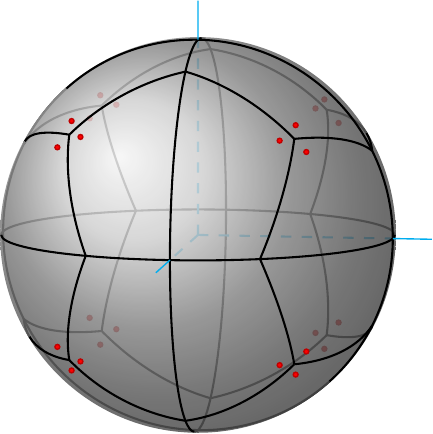}
    \end{subfigure}
    \caption{Spherical Voronoi diagrams of the orbits in
    Figure~\ref{subfig:pyrit_generic} and
    Figure~\ref{subfig:pyrit_close_octahedron}.
    }
    \label{fig:svornoi_example}
\end{figure}

\subsubsection{Polar orbit polytopes and Voronoi diagrams}
\label{sec:voronoi}

There is a well-known connection between polar orbit polytopes and spherical
Voronoi diagrams, or more generally, between polytopes whose facets
are tangent to a sphere and spherical Voronoi diagrams: The central
projection of the polytope to the sphere gives the spherical Voronoi
diagram of the tangency points (the orbit points).
Figure~\ref{fig:svornoi_example} shows spherical Voronoi diagrams for two
orbits of Figure~\ref{fig:pyrit}.

Thus, when we look at polar orbit polytopes, we may think about partitioning the
sphere according to the closest point from the orbit.  The orbit
polytope and the spherical Voronoi diagram have the same combinatorial
structure, but the faces of the orbit polytope are true Euclidean
polytopes, whereas the faces of the Voronoi diagram are spherical
polytopes.
The closer the orbit points are together, the smaller the distortion
will be, and the more the orbit polytope will represent the true
metric situation of the Voronoi diagram.

In our illustrations of 4-dimensional groups, we will prefer to show
orbit polytopes, because these are easier to compute.

\subsection{Fundamental domains and orbifolds}

For comparison, we mention
another way to characterize geometric groups, namely by showing a
fundamental domain of the group, possibly extended by additional
information that characterizes the type of rotations that fix an edge,
such as in an orbifold.
This is particularly appropriate for
Coxeter groups, which are generated by reflections and for which the
choice of fundamental domain is canonical.

Dunbar \cite{dunbar94_f2} studied \OP 4-dimensional point groups.
He constructed fundamental
domains for 10 out of the 14 \OP polyhedral groups
(omitting $\pm[I\times T]$ and $\pm[I\times O]$ and their mirrors).
For each of the 21
\OP polyhedral and axial groups,
he showed the structure of the singular set (fixpoints of some group
elements) of the corresponding
orbifold,
which is a 3-valent
graph where each edge is labeled with the order of the rotational
symmetry around the edge.\footnote
{In the
  list of \OR polyhedral groups that are Coxeter
  groups
  \cite[Figure~17]{dunbar94_f2},   the 6th and 8th entries, which are
 the Coxeter-Dynkin diagrams for the \OR extensions of ${T}\times_{{C}_3}T$
  and $J\times _J^* J^1$, must be exchanged.}

The fundamental domain, possibly enriched by additional information,
is a concise way for representing some groups, but it does not have
the immediate visual appeal of polar orbit polytopes.
For example, the fundamental domain of every Coxeter group is a
simplex,
and the distinctions between different groups lies
only in the dihedral angles at the edges.

\subsection{Left or right orientation of projected images: view from outside}
\label{sec:view-orientation}
We will illustrate many situations in 4-space by %
three-dimensional
graphics that are derived through projection. Just as a plane in space has no
preferred orientation, a 3-dimensional hyperplane in 4-space has no intrinsic orientation.
It depends on from which side we look at it.
Hence, it is important to establish a convention about the orientation,
in order to distinguish a
situation from its mirror image.

Let us look at plane images of the familiar three-dimensional space
``for orientation'' in this matter.
For a polytope or a sphere, we follow the convention that we want to
look at it \emph{from outside}, as for a map of some part of the Earth.
Accordingly, when we interpret a plane picture with
an $x_1,x_2$-coordinate system (with $x_2$ counterclockwise from $x_1$),
the usual convention is to think of the third coordinate $x_3$ as the
``vertical upward'' direction that is facing %
us, leading to a right-handed coordinate system $x_1,x_2,x_3$.

Similarly, when we deal with a 4-polytope and want to show a picture
of one of its facets, which is a three-dimensional polytope~$F$, we
use a
right-handed
orthonormal $x_1,x_2,x_3$-coordinate system in the space of $F$
that can be extended to a positively oriented coordinate system
$x_1,x_2,x_3,x_4$ of 4-space such that $x_4$ points outward from the
4-polytope.

We use the same convention when drawing a cluster of adjacent facets,
or when illustrating situations in the 3-sphere, either through
central projection or through parallel projection.
For example, a small region in the 3-sphere can be visualized as 3-space,
with some distortion, and we will be careful to ensure that this corresponds
to a view on the sphere ``from outside''.

There are other contexts that favor the opposite convention.
For example, stereographic projection is often done
from the North Pole $(x_1,x_2,x_3,x_4)=(0,0,0,1)$ of $S^3$,
and this yields a view %
``from inside'' in the $(x_1,x_2,x_3)$-hyperplane. See for example
\cite[\S 7]{ThrS-I}, or also \cite[p.~123]{dunbar94_f2} for a different
ordering of the coordinates with the same effect.

\section{Point groups}

The 2-dimensional point groups are the
cyclic groups $C_n$ and the dihedral groups $D_{2n}$, for $n\ge 1$.
For $n\ge 3$, they can be visualized, respectively,
as the $n$ rotations of the regular $n$-gon, and
the $2n$ symmetries (rotations and reflections) of the regular $n$-gon.
See Figure~\ref{fig:symmetries_pentagon}.

\begin{figure}[ht]
  \centering
\begin{subfigure}{0.25\textwidth}
    \centering
    \includegraphics[scale=0.8]{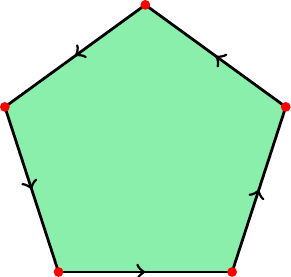}
    \caption*{$C_5$}
\end{subfigure}%
\begin{subfigure}{0.25\textwidth}
    \centering
    \includegraphics[scale=0.8]{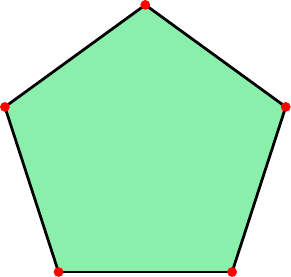}
    \caption*{$D_{10}$}
\end{subfigure}
\caption{The group $C_5$ consisting of the rotational symmetries of the
  regular pentagon, and the group %
  $D_{10}$ of all symmetries of the regular pentagon. %
}
\label{fig:symmetries_pentagon}
\end{figure}

The 3-dimensional point groups are well-studied (see
Section~\ref{sec:3-d-groups} below).
In one sentence, they can be characterized as
 the symmetry groups of the five Platonic solids and
of the  regular $n$-side prisms, and their subgroups.
This gives a frame for classifying these groups, but it does not
give the full information. It remains to work out what the subgroups are,
and moreover, there are duplications, for example:
certain Platonic solids are polar to each other;
the vertices of the cube are contained in the vertices of an icosahedron;
and in turn, they contain the vertices of a tetrahedron;
a cube is a special quadrilateral prism.

\subsection{The 4-dimensional orthogonal transformations}
\label{sec:the-4d-transformations}
\subsubsection{Orientation-preserving transformations}%
We call a 4-dimensional \OP\ transformation a \emph{rotation}.
In some appropriate basis with coordinates $x_1,x_2,x_3,x_4$,
every rotation has the form
\begin{equation}
\label{rotation}
R_{\alpha_1,\alpha_2} = 
    \begin{pmatrix}
    \cos\alpha_1 & -\sin\alpha_1 & 0 & 0 \\
    \sin\alpha_1 & \cos\alpha_1 & 0 & 0 \\
    0 & 0 & \cos\alpha_2 & -\sin\alpha_2  \\
    0 & 0 & \sin\alpha_2 & \cos\alpha_2  \\
    \end{pmatrix},
\text{ or }
R_{\alpha_1,\alpha_2} =
    \begin{pmatrix}
    R_{\alpha_1}&0\\
    0&R_{\alpha_2}\\
    \end{pmatrix} = \diag(R_{\alpha_1},R_{\alpha_2})
\end{equation}
in block form, using the rotation matrices
$R_\alpha =
\left(
\begin{smallmatrix}
    \cos\alpha & -\sin\alpha \\
    \sin\alpha & \cos\alpha
\end{smallmatrix}
\right)
$
as building blocks \cite[\S12.1]{Coxeter}.
If $\alpha_2=0$, we have a \emph{simple rotation}:
a rotation in the $x_1x_2$-plane by the angle $\alpha_1$,
leaving the complementary %
$x_3x_4$-plane fixed.
Thus, the general rotation is the product of two simple rotations in
two %
orthogonal planes,
and we call it more specifically a \emph{double rotation}.
If $\alpha_2\ne\pm\alpha_1$ then the two planes are uniquely determined.
Each plane is an \emph{invariant plane}:
as a set, it is fixed by the operation.

If $\alpha_1=\alpha_2=\pi$, the matrix is the negative identity
matrix, and we have the \emph{central inversion} or \emph{antipodal map}, which we denote by
$-\id$. In $\R^4$, this is an \OP\ transformation.

\subsubsection{Absolutely orthogonal planes and circles}
When we speak of orthogonal planes in 4-space,
we always mean ``absolutely'' orthogonal,
in the sense that every vector in one plane is orthogonal
to every vector in the other plane.%

We will mostly study the situation on the sphere.
Here, an invariant plane becomes an \emph{invariant great circle}, and
there are
\emph{absolutely orthogonal great circles}.

\subsubsection{Left and right rotations}
The rotations with $\alpha_2=\pm\alpha_1$ play a special role: Every
point is moved by the same angle $|\alpha_1|$, and there is no unique
pair of invariant planes.
The rotations with $\alpha_2=\alpha_1$ are \emph{left rotations}, and
the rotations with $\alpha_2=-\alpha_1$ are \emph{right rotations}.
It is easy to see that every rotation
$R_{\alpha_1,\alpha_2}$
is the product of a left and a right rotation
(with angles
$(\alpha_1\pm\alpha_2)/2$).
This representation is unique,
up to a multiplication of both factors with $-\id$.
Left rotations commute with right rotations. These facts are not
straightforward, but they follow easily from the quaternion representation
that is discussed below.
The product of a left rotation by $\beta_L$
and a right rotation by $\beta_R$ is a rotation
$R_{\beta_L+\beta_R,\beta_L-\beta_R}$.

\subsubsection{Orientation-reversing transformations}
An \OR\ transformation has the following form,
in some appropriate basis with coordinates $x_1,x_2,x_3,x_4$:
\begin{equation}\label{eq:OR}
\bar R_{\alpha}=
    \begin{pmatrix}
    \cos\alpha & -\sin\alpha & 0 & 0 \\
    \sin\alpha & \cos\alpha & 0 & 0 \\
     0 & 0& -1&0\\
     0 & 0& 0&1
   \end{pmatrix}
    = \diag(R_{\alpha},-1,1)
\end{equation}
It operates in some three-dimensional subspace $x_1,x_2,x_3$
and leaves one axis $x_4$ fixed.
The $x_3$-axis is inverted.
For $\alpha=0$, we have a mirror reflection in a hyperplane,
$\bar R_0=\diag(1,1,-1,1)$.
For $\alpha=\pi$, we have $\bar R_{\pi}=\diag(-1,-1,-1,1)$,
which could be interpreted as a reflection in the $x_4$-axis.
In general, we have a rotary-reflection,
which has two unique invariant planes:
In one plane, it acts as a rotation by $\alpha$;
in the other plane, it has two opposite fixpoints in $S^3$,
and two other opposite points that are swapped.
The square of an \OR\ transformation $\bar R_{\alpha}$ is always a
simple rotation.

\subsubsection{Quaternion representation} %
The quaternions $x_1+x_2i+x_3j+x_4k$ are naturally identified with
the
vectors $x=(x_1,x_2,x_3,x_4) \in \mathbb R^4$.
We identify the set of unit quaternions with $S^3$, the 3-sphere,
and the set of pure unit quaternions
$v_1i+v_2j+v_3k$ with the points $(v_1,v_2,v_3)$ on $S^2$, the 2-sphere.

Every 4-dimensional rotation can be represented by a pair $[l,r]$
of unit quaternions $l, r \in S^3$.
See \cite[\S 4.1]{CS}.
The pair $[l,r]$ operates on
the vectors $x \in \mathbb R^4$, treated as quaternions,
 by the rule
\begin{displaymath}
  [l,r]\colon
  x \mapsto
  \bar lxr.
\end{displaymath}
The representation of rotations by quaternion pairs is unique
except that $[l,r]=[-l,-r]$.
The rotations $[l,1]$ are the \emph{left rotations},
and the rotations $[1,r]$ are the \emph{right rotations}:
They correspond to quaternion multiplication
from the left and from the right.
A left or right rotation moves every point by the same angular
distance $\alpha$. In fact, as we shall see
(Proportion~\ref{one-parameter-rotations}(ii)),
a left or right rotation
by an angle $\alpha$ other than 0 or $\pi$
defines a \emph{Hopf bundle}, a decomposition of the 3-sphere $S^3$
into circles, each of which is rotated in itself by~$\alpha$. As
transformations on $S^3$, they operate as left screws and right screws, respectively.
See Section~\ref{sec:screws}.

We compose transformations by writing them from left to right, i.e.
$[l_1,r_1][l_2,r_2]$ denotes the effect of first applying
$[l_1,r_1]$ and then $[l_2,r_2]$.\footnote{Du Val~\cite{duval} used the opposite convention, and accordingly his
 notation $[l,r]$ denotes the map
 $  x \mapsto  lx\bar r$.}
Accordingly, composition can be carried out as componentwise
quaternion multiplication: $[l_1,r_1][l_2,r_2]=
[l_1l_2,r_1r_2]$.

Every \OR\ transformation can be represented as
\begin{displaymath}
  {*}[l,r]\colon
  x \mapsto \bar l\bar xr.
\end{displaymath}
See \cite[\S 4.1]{CS}.
The stand-alone symbol
$*$ is alternate notation for quaternion conjugation $*[1,1]\colon x \mapsto  \bar x$.
Then $*[a,b]$ can be interpreted as a composition of the operations
$*$ and $[a,b]$.
Geometrically, the transformation ${*}$ maps $(x_1,x_2,x_3,x_4)$ to
$(x_1,-x_2,-x_3,-x_4)$, and it is a reflection in the $x_1$-axis.
The transformation $-{*}$ maps $(x_1,x_2,x_3,x_4)$ to
$(-x_1,x_2,x_3,x_4)$, and it is a reflection in the hyperplane $x_1=0$.

The inverse transformations are given by these formulas:
\begin{align} \nonumber
  [l,r]^{-1} &= 
  [\bar l, \bar r]\\
  ({*}[l,r])^{-1} &= 
  {*}[\bar r, \bar l]
  =  
  [\bar l, \bar r]{*}  \label{eq:inverse}
\end{align}
The last equation in \eqref{eq:inverse} is also interesting: We may
put the ${*}$ operation on the other side of a transformation $[l,r]$
after swapping the components $l$ and~$r$.

For $l=r$, it is easy to see that $[l,l]$ maps the point $1$ to itself,
and thus operates only on the pure quaternion part.
Thus, the
pairs $[l,l]$ act as 3-dimensional rotations.
For $l=\cos \alpha + \sin \alpha (ui+vj+wk)$,
$[l,l]$ performs a rotation by $2\alpha$ around
the axis with unit vector $(u,v,w)\in \mathbb R^3$.
We will denote $[l, l]$ by $[l]\colon x \mapsto \bar lxl$.
When viewed as an operation on the unit sphere $S^2$, 
$[l]$ is a \emph{clockwise} rotation by $2\alpha$ around the point
$(u,v,w)$.\footnote
{Measuring the rotation angle clockwise is opposite to
the usual convention of regarding
 the counterclockwise direction
 as the mathematically positive direction.
This is a consequence of writing the operation $[l]$
as $x \mapsto \bar lxl$ (as opposed to the alternative
$x \mapsto  lx\bar l$, which was chosen, for example,
by Du Val~\cite{duval}) and regarding
 the quaternion axes $i,j,k$ as a
right-handed coordinate frame of 3-space,
  see~\cite[Exercise 6.4 on p.~67, answer on
  pp.~189--190]{Cox-complex}.
  }
Note that, when the quaternion~$l$ is used as a left rotation $[l,1]$
or a right rotation $[1,l]$ in 4-space,
every point is rotated only by~$\alpha$, not by~$2\alpha$.

\subsection{The classic approach to the classification}
\label{sec:classic}

For a finite subgroup
$G \leqslant \SO(4)$, we can consider the group
\begin{displaymath}
  A = \{\,(l,r) \in S^3\times S^3\mid [l,r] \in G\,\},
\end{displaymath}
which is a two-fold cover of $G$, as 
each rotation $[l,r] \in G$ is represented by two quaternion pairs
$(l,r)$
and $(-l,-r)$ in~$A$.
The elements $l$ and $r$ of these pairs form the
 \emph{left and the right group} of~$G$:
 \begin{displaymath}
   L := \{\,l\mid (l,r) \in A\,\},\quad
   R := \{\,r\mid (l,r) \in A\,\}
 \end{displaymath}
 These are finite groups of quaternions.
 \begin{proposition}
   There is a one-to-one correspondence between
   \begin{enumerate}
   \item The finite subgroups $G$ of $\SO(4)$
   \item The subgroups $A$ of $L\times R$ that contain the element
     $(-1,-1)$, where $L$ and $R$ are finite groups of unit quaternions.
   \end{enumerate}
 \end{proposition}
Since there are only five possibilities for finite groups of unit
quaternions (including two infinite families,
see Section~\ref{sec:quaternion-groups}), this makes it easy, in
principle,
to determine the finite subgroups of $\SO(4)$.

One task of this program,
 the enumeration of the subgroups $A$
of a direct product $L\times R$ is guided by Goursat's Lemma, which
was established by Goursat~\cite{goursat-1889} in this very context:
 The groups
 \begin{displaymath}
   L_0 := \{\,l\mid (l,1) \in A\,\},\quad
   R_0 := \{\,r\mid (1,r) \in A\,\}
 \end{displaymath}
form normal subgroups of $L$ and $R$, which we call the \emph{left and
  right kernel} of $G$.
The group $A$, and hence $G$, is determined by $L,R,L_0,R_0$ and an isomorphism
$\Phi: L/L_0 \to R/R_0$ between the factor groups:
\begin{displaymath}
G = \{\,[l, r] \in \SO(4) \mid l \in L,\ r \in R,\ \Phi(lL_0) = rR_0 \,\}
\end{displaymath}
The task reduces to the enumeration of all possibilities
for the components $L,R,L_0,R_0,\Phi$, and to the less trivial task
of determining which parameters lead to geometrically equal groups.

This approach underlies all classifications so far,
and we call it the \emph{classic} classification.

\subsection{Previous classifications}
\label{sec:previous}

\begin{itemize}
\item Goursat~\cite{goursat-1889}, in 1889,
  classified the finite groups of motions of \emph{elliptic
    3-space}.
  Elliptic 3-space can be interpreted as the 3-sphere $S^3$ in which
  antipodal points are identified.
  Hence, these groups can be equivalently described as those groups
  in \SO(4) that contain the central inversion $-\id$
(the so-called diploid groups, see Section~\ref{sec:notations-diploid}).

\item Threlfall and Seifert~\cite{ThrS-I,ThrS-II}, in a series of two papers in 1931 and
  1933,
  extended this to the groups of \SO(4),
but they only concentrated on the chiral groups.
Their goal was to study the quotient spaces of the 3-sphere %
under fixpoint-free group actions, because these lead to %
\emph{space forms}, spaces of constant curvature without
singularities.\footnote
{%
  The term ``Diskontuinuitätsbereich'' in
  the title of~\cite{ThrS-I,ThrS-II}
   is used
like %
  a well-established concept that does not require a
definition.
In the contemporary literature, it means
 what we today call a {fundamental
   domain}. Seifert and Threlfall were in particular interested in its
 topological properties,
referring by ``Diskontuinuitätsbereich'' to the quotient space under a
group action,
with a specification how the boundary faces of the fundamental
   domain are to be pairwise identified.
  Du Val~\cite[\S\,30]{duval} also takes this interpretation and calls it a \emph{group-set space},
  where \emph{group-set} is his term for orbit.

   In modern usage, ``region of discontinuity'' has other meanings,
   closer to the literal meaning of the words, where discontinuity
   plays a role.
  }
  
\item Hurley~\cite{hurley}, in 1951, independently of  Threlfall and Seifert,
  built on Goursat's classification and extended it to $\O(4)$.
 However, he considered only the crystallographic groups, see
 Appendix~\ref{sec:crystallographic}.
\item Du Val~\cite{duval},
independently of Hurley,
  in a small monograph from 1964,
  took up Goursat's classification and extended it to all groups.
  From a geometric viewpoint,
  he extensively discussed the symmetries of the 4-dimensional regular
  polytopes.
\item Conway and Smith~\cite{CS} in a monograph from 2003,
  took up the classification task again, correcting some omissions and duplications
of the previous classifications.
They gave geometric
descriptions for the polyhedral and axial groups in terms of
Coxeter's notation.
  \end{itemize}
\subsubsection{Related work}
  \begin{itemize}
\item
  De Medeiros
  and Figueroa-O'Farrill~\cite{atmp/spin}, in 2012,
  classified the groups of
order pairs $(l,r)\in S^3\times S^3$ of unit quaternions
under componentwise multiplication
(using Goursat's Lemma again).
These form 
the 4-dimensional spin group
Spin(4). Since this is a double cover of $\SO(4)$,
the results should confirm the classification of the chiral point groups.
Indeed,
Tables 16--18 in  \cite[Appendix B]{atmp/spin} give
references to \SO(4) and the classification of~\cite{CS}.%
\footnote
{However, besides noticing a few typographical errors, we found some
  discrepancies in these tables: %
(i) The 6th entry in Table~18 lists a group $\pm [C_{2k+1}\times \bar D_{4m}]$.
We cannot match this with anything in the Conway--Smith classification, even
allowing for one typo.
(ii) The last entry in Table 4.2 of \cite{CS} is $+ \frac1f [C_{mf} \times C_{nf}]$.
 This group does not appear in the tables of~\cite{atmp/spin}.
We don't know whether these discrepancies arose
 in the translation
from the classification in
\cite{atmp/spin} to the notions of \SO(4)
or they indicate problems in the classification itself.
}

\item

  Marina
  Maerchik, in 1976 \cite{maerchik76},
  investigated the groups that are generated
  by
 reflections and simple rotations (also in higher dimensions),  
as reported in Lange and Mikha\^\i lova \cite{lange16},
(The term ``pseudoreflections'' in the
title of~%
\cite{maerchik76} refers to simple rotations.)

\item We mention that the approach of understanding the 4-dimensional
  groups through their orbits was pioneered by
  Robinson~\cite{robinson_1931}, who, in 1931, studied the orbits of
  the polyhedral groups. He focused on the orbits themselves
  and
  their convex hulls
  (and not
  on the polar orbit polytopes as we do).

\end{itemize}

\subsection{Conjugacy, geometrically equal groups}
\label{sec:conjugacy}

Conjugation with a rotation $[a,b]$ transforms a group into a
different group, which is
geometrically the same, but expressed in a different coordinate
system.
Conjugation transforms an \OP\ transformation $[l,r]$ as follows:
\begin{displaymath}
  [a,b]^{-1}[l,r][a,b] = [a^{-1}la, b^{-1}rb]
\end{displaymath}
Its effect is thus a conjugation of the left group by $a$
and an independent conjugation of the right group by $b$.
As a conclusion, we can represent the left group $L$ and the
right group~$R$ in any convenient coordinate system of our choice,
and it is no loss of generality to choose a particular representative
for each finite group of
quaternions. (Section~\ref{sec:quaternion-groups} specifies the
representatives that we use.)

\subsection{Obtaining the achiral groups}
\label{sec:achiral}

\newcommand\extendingelement{e}

The {classic} approach by Goursat's Lemma leads only to the chiral
groups.
Since the chiral part of an achiral group is an index-2 subgroup,
every achiral group $G$ is obtained by
extending a chiral group $H$ with some
\OR\ element 
\begin{displaymath}
 \extendingelement={*}[a,b].
\end{displaymath}
We will now derive some conditions on
$\extendingelement$, and possibly by modifying the group $G$ into a
geometrically conjugate group,
constrain
$\extendingelement$ to a finite number of possibilities.

Let $H$ be a chiral group with {left group}
 $L$
 and {right group}
$R$.
For each $[l,r] \in H$,
we must have $\extendingelement^{-1} [l,r]  \extendingelement\in H$,
i.e., $H$ is normalized
by $\extendingelement$:
\begin{displaymath}
\extendingelement^{-1} [l,r]  \extendingelement
  =  {*}[\bar b, \bar a] [l,r]  {*}[a,b]
  = [\bar a r a, \bar b l b] \in H %
\end{displaymath}
This means that
$\bar a r a \in L$ and $\bar b l b \in R$ for every
$[l,r] \in H$,
which implies
$\bar a R a = L$ and $\bar b L b = R$, i.e., $L$ and $R$ are
conjugate.

We conjugate $G$ with $[1,a]$, transforming $G$ to some
geometrically equivalent group $G'$ with left group $L'$ and right
group~$R'$.
Let us see what happens to an arbitrary element $[l,r]$:
\begin{equation}
  \label{eq:A}
  [1,\bar a] [l,r]  [1,a] = [l,\bar a r a]
\end{equation}
The set of values $\bar a r a$ forms the new right group $R'=\bar a R a=L$,
while the left group remains unchanged: $L'=L$.
Thus, we have achieved $L'=R'$, i.e., the left and right groups are
not just conjugate, but equal.

The extending element $\extendingelement={*}[a,b]$ is transformed as follows:
\begin{equation}
  \label{eq:B} \extendingelement' := 
  [1,\bar a] {*}[a,b]
   [1,a] = {*}[1,ba] = {*}[1,c]
\end{equation}
Thus we have simultaneously achieved $\extendingelement'={*}[1,c]$.
Moreover,
\begin{displaymath}
  \extendingelement'\extendingelement' =
  {*}[1,c]  {*}[1,c] = [c,c]\in H,
\end{displaymath}
and thus,
$c$ must be an element of $L=R$.

\begin{proposition}
  \label{1,c}
  W.l.o.g., we can assume $L=R$,
  and the extending element %
  is of the form
  $\extendingelement={*}[1,c]$, with $c\in L$.
\end{proposition}

This reduces the extending element to a finite number of
possibilities.  Conway and Smith \cite[p.~51]{CS} have sketched some
additional considerations, which allow to further restrict the
extending element, sometimes at the cost of giving up the condition
$L=R$, see Figure~\ref{fig:p51} on p.~\pageref{fig:p51}.

Conjugation by $[a,a]$ changes the transformations as follows:
\begin{align*}
  [a,a]^{-1}[l,r][a,a] &= [a^{-1}la, a^{-1}ra]\\
  [a,a]^{-1}{*}[l,r][a,a] &= *[a^{-1}la, a^{-1}ra]
\end{align*}
Its effect is thus a conjugation of the left and right group $L=R$ by $a$.
As %
for the chiral groups,
we can therefore choose any convenient representation
of the left and right group~$L$ in Proposition~\ref{1,c}.

\subsection{Point groups in 3-space and their quaternion
  representation}
\label{sec:3-d-groups}

Table~\ref{tab:3d-point-groups} lists the three-dimensional point
groups that we will use.
We will refer to them by the notation of Conway and Smith  \cite{CS},
given in the first column.
As alternate
notations, we give the orbifold notation, the Hermann-Mauguin notation or
international symbol~\cite{IT}, and the Coxeter notation, which we
will revisit in Section~\ref{sec:polyhedral}.
\begin{table}[htb]
  \centering
  \begin{tabular}[t]{|cccc|r|l|}
    \hline
   \multicolumn 6{|c|}{the chiral groups }
\\    \hline
    CS%
    & \hbox to 0pt{\hss orbifold\hss}& I.T. 
          & Coxeter\ name
          & {o}rder &\OP\ symmetries of \dots\\\hline
$+C_n$& $\mathbf{nn}$& $n$& $[n]^+$ & $n$& the $n$-sided pyramid
    $n\ge1$)
    \\
$+D_{2n}$& $\mathbf{22n}$& $n2$& $[2,n]^+$ & $2n$& the $n$-sided prism
        ($n\ge1$)
    \\
$+T$& $\mathbf{332}$& $23$& $[3,3]^+$ & $12$& the tetrahedron\\
$+O$& $\mathbf{432}$& $432$& $[3,4]^+$ & $24$& the octahedron / the cube\\
    $+I$& $\mathbf{532}$& $532$& $[3,5]^+$ & $60$& the icosahedron /
                                                   the dodecahedron\\
    \hline
    \multicolumn 6{|c|}{achiral polyhedral groups}\\\hline
    CS%
    & \hbox to 0pt{\hss orbifold\hss}& I.T. 
          & Coxeter\ name
                                               & \llap{o}rder
          &description of the group\\
    \hline
$TO$& $\mathbf{*332}$& $\raise1pt\strut\bar43m$& $[3,3]$ & $24$&all symmetries of the
    tetrahedron\\
$\pm T$& $\mathbf{3{*}2}$& $m\bar 3$& $[3^+,4]$ or $[^+3,4]$ & $24$&the pyritohedral group\\
$\pm O$& $\mathbf{*432}$& $m\bar 3m$&  $[3,4]$ & $48$&all symmetries of the
    octahedron\\
$\pm I$& $\mathbf{*532}$& $53m$&     $[3,5]$ & $120$&all symmetries of
                                                      the
                                                      icosahedron\\
    \hline
  \end{tabular}
  \caption
  [Point groups in 3 dimensions]
  {Some point groups in 3 dimensions}
  \label{tab:3d-point-groups}
\end{table}

  The table contains all polyhedral groups (3 chiral and 4 achiral
  ones): groups consisting of symmetries of regular polytopes.
 The groups that are not polyhedral (subgroups of the
symmetry groups of regular prisms, related to the
frieze groups) include, besides $+C_n$ and $+D_{2n}$, five additional
classes of \emph{achiral} groups, which
are not listed here.
In total, there are 14 types of three-dimensional point groups.
Note that the subscript $2n$ in $D_{2n}$ is always even; we follow the
convention of using
the \emph{order} of the group, not the number of sides of the polygon
or prism
of which it is the symmetry group.

The notations $+I,\pm I$, etc. for the polyhedral groups are easy to
remember.  The one that requires some attention is the full
symmetry group of the tetrahedron, which is denoted by $TO$,
as opposed to the pyritohedral group $\pm T$,
which is obtained by extending $+T$ by the central reflection,
and which we have discussed extensively in
Section~\ref{sec:pyrit}.

\subsection{Finite groups of quaternions}
\label{sec:quaternion-groups}

The finite groups of quaternions
are \cite[Theorem 12]{CS}:
\begin{align*}
2I &= \langle i_I, \w \rangle &
2D_{2n} &= \langle e_n, j \rangle \\
2O &= \langle i_O, \w \rangle &
2C_n &= \langle e_n \rangle \\
2T &= \langle i, \w \rangle 
&1C_n &= \langle e_{n/2} \rangle \ \ (n \text{ odd})
\end{align*}
The generators are defined
in terms of the following  quaternions, which we will use throughout:
\begin{equation}
\label{eq:defining_quaternions}
\begin{aligned}
    \w &= \tfrac12(-1+i+j+k) &&\text{(order 3)}\\
    i_O &= \tfrac1{\sqrt2}(j+k) &&\text{(order 4)}\\
    i_I &= \tfrac12\bigl(i + \tfrac{\sqrt5-1}2 j +
    \tfrac{\sqrt5+1}2k\bigr) &&\text{(order 4)}\\
    e_n &%
    = \cos \tfrac{\pi}{n} + i \sin\tfrac{\pi}{n}&&\text{(order $2n$)}
\end{aligned}
\end{equation}
We follow Conway and Smith's notation for these groups.
For each group $+G < \SO(3)$
(see the upper part of Table~\ref{tab:3d-point-groups}),
there is quaternion group $2G$ of
twice the size, containing the quaternions~$\pm l$ for which $[l]$
represents a rotation in $+G$.
All these groups contain the quaternion $-1$.
In addition, there are the odd cyclic groups $1C_n$, of order $n$.
 They cannot arise as left or right groups, because $(-1,-1)$ is
always contained in $A$ and hence the left and right groups
contain the quaternion~$-1$.

\subsection{Notations for the 4-dimensional point groups, diploid
  and haploid groups}
\label{sec:notations-diploid}

We use the notation by
Conway and Smith \cite{CS}
for 4-dimensional point groups $G$, except for the toroidal groups,
where we will replace it with our own notation.
If $L$ and $R$ are 3-dimensional \OP\ point groups,
$\pm[L\times R]$
denotes full product group $\{\,[l,r]\mid (l,r) \in 2L\times
2R\,\} $, of order $2|L|\cdot |R|$.
Note that the groups $2L$ and $2R$ that appear in the
definition are quaternion groups,
while the notation shows only the
corresponding rotation groups $L,R\in\SO(3)$.

A group that contains the negation $-\id =[1,-1]$ is called a
\emph{diploid} group.
A diploid index-$f$ subgroup of $\pm[L\times R]$
is denoted by
$\pm\frac1f[L\times R]$. It is defined by two %
normal
subgroups %
of $2L$ and %
of $2R$ of index~$f$.
Different possibilities for the
normal
subgroups %
and for
the isomorphism~$\Phi$ are distinguished by various
ornamentations of the notation,
see Appendix~\ref{conway-smith} for %
some of these cases.

A \emph{haploid} group, which does not contain the negation $-\id$,
is denoted by
$+\frac1f[L\times R]$, and it is %
an index-2 subgroup of
the corresponding diploid group
$\pm\frac1f[L\times R]$.
Achiral groups are index-2 extensions of
chiral groups, and they are also denoted by various decorations.

Du Val \cite{duval} writes the groups as $(\mathbf{L}/\mathbf{L_0};
\mathbf{R}/\mathbf{R_0})$,
where the
 boldface letters distinguish
 quaternion groups from the corresponding 3-dimensional rotation groups.
Again, various ornamentations denote different cases of normal
subgroups and
the isomorphism~$\Phi$.
Achiral extensions are denoted by a star.
We will not work with this notation except for reference in our tables, and then we will
omit the boldface font.
In some cases, we had to adapt Du Val's names, see
Table~\ref{tbl:axial_full} and footnote~\ref{duval-41-42}.

\section{Hopf fibrations}
\label{sec:hopf}
We give a self-contained presentation of Hopf fibrations. %
In many places in the literature,
one particular Hopf map is introduced as ``the Hopf map'',
either in terms of four real coordinates
or two complex coordinates,
leading to ``the Hopf fibration''.
In some sense, this is justified,
as %
all Hopf bundles are (mirror-)congruent.
However, for our characterization,
we need the full generality of Hopf bundles.

Our treatment was inspired by Lyons~\cite{lyons03_f2},
but we did not see it anywhere in this generality.
As a %
tool, we introduce a parameterization of
the great circles in $S^3$, which might be useful elsewhere.
We also define \emph{oriented} Hopf bundles: families of
consistently oriented great circles.

We summarize the main statements:
\begin{itemize}
\item The great circles in $S^3$ can be parameterized by pairs $p,q$
  of pure unit quaternions, or equivalently, by pairs of points
  $p,q\in S^2$
 (Section~\ref{great-circles}).
 The choice of parameters is unique except that
  $K_p^q = K_{-p}^{-q}$.
  The twofold
  ambiguity of the parameters can be used to specify an orientation
  of the circles
 (Section~\ref{sec:oriented}).

\item The great circles
  $K_p^q$ with fixed $q$ form a partition of $S^3$, which we call the
\emph{left Hopf bundle}~$\H^q$.
  It naturally comes with a \emph{left Hopf map} $h^q\colon S^3\to S^2$,
  which maps
  all points of $K_p^q$ to the point $p\in S^2$.
  
  This map provides a bijection between the
  circles of the {left Hopf bundle} $\H^q$ and the points on $S^2$.
  
  Similarly, the great circles
  $K_p^q$ with fixed $p$ form a
\emph{right Hopf bundle}~$\H_p$, with a  \emph{right Hopf map} $h_p$,
etc.
In the following, we will mention only the left Hopf bundles, but all
statements hold also with left and right reversed. %

\item Every great circle of $S^3$ belongs to a unique left Hopf
  bundle. In other words, the left Hopf bundles form a partition of
  the set of great circles of $S^3$.

\item
  For every left Hopf bundle $\H^q$,
  there is a one-parameter family of right rotations that maps every
  circle in $\H^q$ to itself, rotating each circle by the same
  angle~$\alpha$.

  Conversely,
  a right rotation by an angle $\alpha\notin\{ 0,\pi\}$
rotates every point of $S^3$ by the  same
  angle~$\alpha$, and the 
  set of circles along which these rotations happen form
 a left Hopf bundle
  (Proposition~\ref{one-parameter-rotations}).
  
    \item The following statements discuss the behavior of Hopf
      bundles under orthogonal transformations (Proposition~\ref{prop-congruent-bundles}):
    \begin{itemize}
\item Any left rotation leaves
  the
    left Hopf bundle $\H^q$ fixed, as a partition. It permutes the
    great circles of the bundle.
  \item Any rotation maps the
    left Hopf bundle $\H^q$ to another left Hopf bundle.
 Any two left Hopf bundles are congruent (by some
    right rotation).
\item 
   Left Hopf bundles and right Hopf bundles are mirrors of each other.

    \end{itemize}

\item The intersection of  a left Hopf bundle and a right Hopf bundle
  consists of two absolutely orthogonal circles
(Corollary~\ref{coro:hopf-intersect}).

\item Any two great circles in the same Hopf bundle are Clifford-parallel
  (Proposition~\ref{prop:clifford_parallel}).
  This means that a point moving on one circle maintains a constant
  distance to the other circle.

\end{itemize}

\subsection{Parameterizing the great circles in \texorpdfstring{$S^3$}{S3}}\label{great-circles}
\begin{definition}
  \label{definition-Kpq}
For any two pure unit quaternions $p, q \in S^2$,
we define the following subset of unit quaternions:
\begin{equation}
    \label{eq:great-circles}
    K_p^q := \{\, x \in S^3 \mid [x]p=q \,\}
\end{equation}
This %
can be interpreted as the set of rotations on $S^2$ that
map $p$ to $q$.
\end{definition}

\begin{proposition}
\label{hopf-alternate}
$K_p^q$ has an alternative representation
\begin{equation}
    \label{eq:great-circles2}
    K_p^q = \{\, x \in S^3 \mid [p, q]x = x \,\},
\end{equation}
and it forms a great circle in $S^3$.
Moreover, every great circle in $S^3$ can be represented in this way,
and the choice of parameters $p, q \in S^2$ is unique except that
$K_p^q = K_{-p}^{-q}.$
\end{proposition}

This %
gives a convenient parameterization
of the great circles in $S^3$ (or equivalently, the planes in~$\R^4$)
by pairs of points on $S^2$, which might be useful
in other contexts.
For example, they might be used to define a notion of distance between
great circles (or planes in~$\R^4$). (Other distance measures are
discussed in \cite{kr-ctps4-16,kr-ctps4-16t} and \cite{packing-planes}.
  The connection to these different distance notions remains to
  be explored.)

Before giving the proof, let us make a general remark about
quaternions.
Multiple %
meanings can be associated to a
unit quaternion~$x$: Besides treating it (i) as a point on $S^3$, we can regard it
(ii)
as a rotation $[x]$ of $S^2$,
or (iii)~as a left rotation $[x,1]$ of $S^3$,
or (iv)~as a right rotation $[1,x]$ of $S^3$.
Rather than fixing an opinion on what a %
quaternion really \emph{is} (cf.~\cite[p.~298]{Altman-Hamilton-1989}),
we capitalize on this ambiguity
and freely switch between
the definitions  %
\eqref{eq:great-circles} and \eqref{eq:great-circles2}.

\begin{proof}[Proof of Proposition~\ref{hopf-alternate}]
 The two expressions
\eqref{eq:great-circles} and \eqref{eq:great-circles2}
are equivalent by a simple rearrangement of terms:
\begin{displaymath}
    [x]p = q
    \iff \bar x p x = q
    \iff p x = x q
    \iff x = \bar p x q
    \iff x = [p, q]x
\end{displaymath}
The expression \eqref{eq:great-circles2} shows that  
$K_{p}^q$ is the set of fixpoints of the rotation $[p,q]$.
Since $p$ and $q$ are unit quaternions,
the rotation $[p,q]$ is a simple rotation by $180^\circ$ (a half-turn).
Its set of fixpoints is a two-dimensional plane, or when restricted
to unit quaternions, a great circle.

Conversely, if a great circle $K$ is given and we want to determine
$p$ and $q$, we know that we are looking for a simple rotation %
by $180^\circ$ whose set of fixpoints is $K$. This rotation is
uniquely determined, and its quaternion representation $[p,q]$
is unique up to flipping both signs simultaneously.
\end{proof}

The effect of orthogonal
transformations on great circles is expressed easily in our parameterization:
\begin{proposition}
\label{prop:relations_between_circles}
Let $p, q \in S^2$. Then for any $l, r \in S^3$,
\begin{enumerate}[\rm(i)]
    \item $[l, r] K_p^q = K_{[l]p}^{[r]q}$.
    \item $(*[l, r]) K_p^q = K_{[l]q}^{[r]p}$,
      and in particular,
      $* K_p^q = K_{q}^{p}$.
\end{enumerate}
\end{proposition}

\begin{proof}
The following calculation proves part (i).
\begin{multline*}
    [l, r] K_p^q
    = \{\,\bar{l} x r \mid \bar{x} p x = q \,\}
\\    = \{\,y \mid r \bar{y} \bar{l} p l y \bar{r} = q \,\}
    = \{\,y \mid \bar{y} \bar{l} p l y = \bar{r} q r \,\}
    = \{\,y \mid [y][l] p = [r] q \,\}
    = K_{[l]p}^{[r]q},
\end{multline*}
where we have substituted $x$ by $y := \bar{l} x r$.
Part (ii) follows from part (i) and $* K_p^q = K_{q}^{p}$.
This last statement expresses the fact that the inverse rotations
$[\bar x]$ of
the rotations $[x]$ that map $p$ to $q$ are the rotations mapping $q$
to $p$.  More formally,
\begin{displaymath}
    * K_p^q
    = \{\,\bar{x} \mid \bar{x}px = q \,\}
    = \{\,y \mid yp\bar{y} = q \,\}
    = \{\,y \mid p = \bar yq{y} \,\}
    = K_{q}^{p},
  \end{displaymath}
  with $y := \bar{x}$.
\end{proof}

The elements of $K_p^p$ form
 a subgroup of the quaternions \cite{dunbar94_f2}: %
 According to~%
\eqref{eq:great-circles},
 $K_p^p$ is the stabilizer of~$p$.
Its cosets
can be characterized
by Proposition~\ref{prop:relations_between_circles}(i):
\begin{corollary}
  \label{coro:circle_as_costs}
The left cosets of $K_p^p$ are the circles  $K_{p'}^p$, and
the right cosets of $K_p^p$ are the circles  $K_{p}^{p'}$, for
arbitrary $p'\in S^2$.
\qed
\end{corollary}

We emphasize that the two parameters $p$ and $q$ in $K_p^q$ ``live on
different spheres $S^2$'': Any relation between them has no intrinsic
geometric meaning, and will be changed by coordinate
transformations according to
 Proposition~\ref{prop:relations_between_circles}.
This is despite the fact that $p=q$ has an
algebraic significance, since the circle $K_p^p$
goes through the special
quaternion 1, which is one of the coordinate axes,
and hence $K_p^p$ forms a subgroup of quaternions.

\subsubsection{Keeping a circle invariant}

The following proposition characterizes the transformations that map a
given great circle %
to itself. Moreover, it describes the action of these transformations 
when restricted to that circle.
For a pure unit quaternion $p\in S^2$ and an angle $\theta\in\R$ we use the notation
\begin{displaymath}
  \exp p\theta := \cos\theta + p\sin\theta,
\end{displaymath}
so that
$[\exp p\theta]$ is a clockwise rotation around $p$ by $2\theta$ on $S^2.$

\begin{proposition}
\label{prop:action_on_its_circle}
Consider the circle $K_p^q$, for $p, q \in S^2$.
The rotations $[l, r]$ that leave $K_p^q$ invariant fall into two categories,
each of which is a
 two-parameter family.
\begin{enumerate}[\rm(a)]
\item
  \label{OP-circle-rotate}
The \OP\ case: $[l]p = p$ and $[r]q = q$.

Every transformation in this family can be written as
$[\exp p\phi , \exp q\theta]$ for $\phi, \theta \in \R$.
This transformation acts on the circle $K_p^q$ as rotation by
$|\theta-\phi|$.
\item
  The \OR\ case: $[l]p = -p$ and $[r]q = -q$.

After choosing two fixed quaternions $p', q' \in S^2$
 orthogonal to $p$ and $q$,
respectively, they can be written as the transformations  
 $[p'\exp p\phi, q'\exp q\theta]$ 
    for $\phi, \theta \in \R$,
and they act on $K_p^q$ as reflections.
\end{enumerate}
\end{proposition}
Note that the transformations that we consider are always {\OP} when
considered in 4-space; they can
be \OR\ when considered as (2-dimensional) operations on the
circle $K_p^q$.

\begin{proof}
Let $[l, r] \in \SO(4)$ be a rotation.
Then we have the following equivalences.
\begin{displaymath}
    [l, r]K_p^q = K_p^q
    \iff K_{[l]p}^{[r]q} = K_p^q
    \iff ([l]p = p \wedge [r]q = q) \lor ([l]p = -p \wedge [r]q = -q)
\end{displaymath}
For the first case,
the transformations $[l]$ on $S^2$ that leave the point $p$ fixed
are the rotations around $p$,
and they are given by the quaternions
$l=\exp p\phi$, and similarly for~$r$.
For the second case, the transformations $[l]$ on $S^2$
that map $p$ to $-p$ can be written as a composition of
$[p']$, which maps $p$ to $-p$,
and an arbitrary rotation around the axis through $p$ and $-p$,
which is expressed as $[\exp p\phi]$.
This establishes that $[l,r]$ can be written in the claimed form.

We now investigate the action of these rotations on $K_p^q$.
\begin{enumerate}[(a)]
\item
Let $x \in K_p^q$. 
Since $x q = p x$,
we have $x \exp q\theta = (\exp p\phi)x$.
In particular,
\begin{align*}
    [\exp p\phi, \exp q\theta]x = \exp(-p\phi)x\exp q\theta
    = \exp(-p\phi)(\exp p\theta)x
    = (\exp p(-\phi+\theta))x.
\end{align*}
Thus, $[\exp p\phi, \exp q\theta]$ acts on $K_p^q$
like the left multiplication with $\exp p(\theta-\phi)$, which
(being a left rotation) moves every point by the angle $|\theta-\phi|$.

\item
It is enough to show that $[p', q']$ acts as a reflection on $K_p^q$.
We will show that $K_p^q \cap K_{p'}^{q'} \not= \emptyset$
and $K_p^q \cap K_{p'}^{-q'} \not= \emptyset$.
Thus, there is a point $x \in K_p^q$
with $[p', q']x = x$ and another point
$y \in K_p^q$ with $[p', q']y = -y$, and
this means
 that $[p', q']$ fixes some, but not all, points on $K_p^q$,
and thus its action cannot be a rotation.

Let $[x_0]$ be a rotation that maps $p$ to $q$.
Then it maps $p'$ to some point $p''$ that is orthogonal to $q$.
Let $[y_0]$ be the rotation that fixes $q$ and maps $p''$ to $q'$.
The rotation $[x_0y_0]$ maps $p$ to $q$ and $p'$ to $q'$.
Thus, $x_0y_0 \in K_p^q \cap K_{p'}^{q'}$.
Similarly, if $[z_0]$ is the rotation that fixes $q$
and maps $p''$ to $-q'$,
then $x_0z_0 \in  K_p^q \cap K_{p'}^{-q'}$.
\qedhere
\end{enumerate}
\end{proof}

\begin{proposition}
The great circles $K_p^q$ and $K_p^{-q}=K_{-p}^{q}$ are absolutely orthogonal.
\end{proposition}

\begin{proof}
    The simple rotation $[p, -q]=[-p, q]$ maps $x \in K_p^q$
    to $-x \in K_p^q$. That is, $[p, -q]$ preserves (not pointwise) $K_p^q$.
    Since $K_p^{-q}$ is the fixed circle of $[p, -q]$ and the invariant
    circles of a simple rotation are absolutely orthogonal, we are done.
\end{proof}

\subsubsection{Oriented great circles}
\label{sec:oriented}
By Proportion~\ref{prop:action_on_its_circle},
the left rotation $[\exp (-p\theta),1]$ has the same effect on the circle $K_p^q$
as the
right rotation
 $[1,\exp q\theta]$. This
allows us to specify an
orientation for~$K_p^q$.
For some starting point
 $x\in K_p^q$, we write
\begin{equation}
  \label{eq:orient-circle}
  K_p^q= \{\, (\exp p\theta) x \mid \theta \in \R\,\}
  = \{\, x\exp q\theta \mid \theta \in \R\,\},
\end{equation}
and both parameterizations traverse the circle in the same sense, for
increasing $\theta$.
We may thus introduce the notation
$\vec K_p^q$ to denote an \emph{oriented great circle} on $S^3$.
If we use
$\vec K_{-p}^{-q}$ in \eqref{eq:orient-circle}, the same circle will be
traversed in the \emph{opposite sense}.
Thus,
we obtain
a notation for oriented great
circles on~$S^3$, and for this notation, the choice of parameters $p,q\in S^2$ is unique.
Only for an oriented circle,
the phrase ``rotation by $\pi/4$'' or  ``rotation by $-\pi/3$''
has a well-defined meaning, and
we can give a more specific version of
Proposition~\ref{prop:action_on_its_circle}\ref{OP-circle-rotate}:
The operation
$[\exp p\phi, \exp q\theta]$ rotates
$\vec K_p^q$ by
 $\theta-\phi$.

In Appendix~\ref{sec:oriented-circles-interpretation}, we give a
direct geometric view of this orientation, based on the
original interpretation of $K_p^q$ as the set of rotations on $S^2$ that map
$p$ to $q$ (Definition~\ref{definition-Kpq}).

Proposition~\ref{prop:relations_between_circles} extends to oriented
circles as follows:
\begin{proposition}
\label{prop:relations_between_oriented_circles}
 $[l, r] \vec K_p^q = \vec K_{[l]p}^{[r]q}$ and
      $* \vec K_p^q = \vec K_{-q}^{-p}$.
\end{proposition}
\begin{proof}
  For $x\in K_p^q$,
  \begin{displaymath}
  [l,r](x\exp q\theta)
  =
  \bar lx(\exp q\theta)r
  =
  \bar lx r \bar r(\exp q\theta)r
  =
  (\bar lx r) \exp (\bar r qr\theta)
  =
  y \exp (([r] q)\theta)
  \end{displaymath}
with $y = \bar lx r \in
[l, r] K_p^q = 
K_{[l]p}^{[r]q}$. Thus, the orientation that we get on
$[l, r] \vec K_p^q $ coincides with the orientation prescribed
 in \eqref{eq:orient-circle} for
 $\vec K_{[l]p}^{[r]q}$.
 Similarly,
  \begin{displaymath}
  {*}(x\exp q\theta)
  =
  (\exp \bar q\theta)\bar x
  =
  \exp (-q\theta)\,y
  \end{displaymath}
  with $y = \bar x \in {*} K_p^q = K_q^p=
  K_{-q}^{-p}$,
  and this is the correct orientation for  $\vec K_{-q}^{-p}$
  in accordance with~\eqref{eq:orient-circle}.
\end{proof}

\subsection{Hopf bundles}
Hopf bundles are %
families of circles $K_p^q$
with fixed $p$ or
with fixed $q$:

\begin{definition}
  Let $q_0 \in S^2$ be a pure unit quaternion.
  The \emph{left Hopf bundle} $\H^{q_0}$ %
is
$$\H^{q_0} := \{\, K_{q}^{q_0} \mid q \in S^2 \,\},$$
and the \emph{right Hopf bundle} $\H_{q_0}$ %
is %
$$\H_{q_0} := \{\, K_{q_0}^q \mid q \in S^2 \,\}.$$
The \emph{oriented} left and right Hopf bundles are defined
analogously:
\begin{align*}
  \vec \H^{q_0} &:= \{\, \vec K_{q}^{q_0} \mid q \in S^2 \,\}\\
  \vec  \H_{q_0} &:= \{\, \vec K_{q_0}^q \mid q \in S^2 \,\}
\end{align*}
\end{definition}
The convention for left and right was adopted from Dunbar~\cite{dunbar94_f2}:
According to Corollary~\ref{coro:circle_as_costs}, the circles $K_q^{q_0}$
of the
{left Hopf bundle} $\H^{q_0}$
are the \emph{left} cosets of the circle $K_{q_0}^{q_0}$.

We can naturally assign a Hopf map to each bundle, such that
 the circles of a bundle become the fibers of the associated Hopf
map:
\begin{definition}
Let $q_0 \in S^2$ be a pure unit quaternion.
The \emph{left Hopf map} associated with~$q_0$ is
\begin{align*}
    h^{q_0}\colon  S^3 &\to S^2 \\
    x &\mapsto [\bar{x}]q_0 = x q_0 \bar{x},
\end{align*}
and the \emph{right Hopf map} associated with $q_0$ is
\begin{align*}
    h_{q_0}\colon  S^3 &\to S^2 \\
    x &\mapsto [x]q_0 = \bar{x} q_0 x.
\end{align*}
\end{definition}

\begin{corollary}\label{coro:hopf-intersect}
The following statements are direct consequences of the definitions:
\begin{itemize}
    \item The choice of the parameter $q_0$ in the left Hopf bundle
      $\H^{q_0}$ is unique except that $\H^{q_0} = \H^{-q_0}$.
As oriented Hopf bundles,
$\vec\H^{q_0}$ and $\vec \H^{-q_0}$ contain the same circles in
opposite orientation.
      
         The same statement holds for right Hopf bundles.

    \item No two different left Hopf bundles share a circle.
        That is,
        \begin{displaymath}
            \H^{p_0} \cap \H^{p_1} = \emptyset\ \text{if}\ p_0 \neq \pm p_1.
        \end{displaymath}
        A similar statement holds for right Hopf bundles.

    \item A left Hopf bundle intersects a right Hopf bundle in exactly
      two circles, which %
      are absolutely orthogonal:
        \begin{displaymath}
        \H_{q_0} \cap \H^{p_0}
            = \{K_{q_0}^{p_0}, K_{q_0}^{-p_0}=K_{-q_0}^{p_0}\}.
        \end{displaymath}
    \item Every great circle $K_{q_0}^{p_0}$ in $S^3$ belongs
        to a unique left Hopf bundle $\H^{p_0}$ and
        to a unique right Hopf bundle $\H_{q_0}$.
\end{itemize}
\end{corollary}

From Proposition~\ref{prop:relations_between_oriented_circles}, we
can directly work out the effect of a transformation on an (oriented) Hopf bundle:
\begin{proposition}
  \label{prop:mappings_between_oriented_Hopf_bundles}
  (a)
 $[l, r] \vec \H^q = \vec \H^{[r]q}$ and
 $[l, r] \vec \H_p = \vec \H_{[l]p}$;
 \ (b)
 $* \vec \H^q = \vec \H_{-q}$ and
  $* \vec \H_p = \vec \H^{-p}$.
\end{proposition}
We get consequences about the operations that leave a
 Hopf bundle invariant and about mappings between Hopf bundles.
\begin{proposition}
  \label{prop-congruent-bundles}
 The following statements about the operations that leave a
left Hopf bundle invariant hold, and similar statements hold for
right Hopf bundles.
  \begin{enumerate}[\rm(i)]
\item Any left rotation leaves
an oriented left Hopf bundle $\vec \H^q$ invariant.
It permutes the
    great circles of the bundle.
  \item
    \label{Hopf-preserve}
    A right rotation $[1,r]$ leaves
    the oriented left Hopf bundle $\vec \H^q$ invariant iff $[r]q=q$.

\item
    \label{Hopf-reverse}
  A right rotation $[1,r]$ 
  maps the oriented left Hopf bundle $\vec \H^q$ to the opposite
  bundle $\vec \H^{-q}$ iff $[r]q=-q$.
\item Any two oriented left Hopf bundles are congruent,
    and can be mapped to each other by a right rotation.
\item Any oriented right Hopf bundle and any oriented left Hopf bundle
    are mirrors of each other.
    \qed
  \end{enumerate}
\end{proposition}

We can summarize properties (i)--(iii) in the following statement,
which characterizes the transformations that
leave a given left Hopf bundle invariant, in analogy to Proposition~\ref{prop:action_on_its_circle}.
\begin{proposition}
\ \label{prop:transformations_preserving_H2}
\begin{enumerate}[\rm(i)]
\item 
  A rotation  $[l, r]$ preserves $\H^{q_0}$ if and only if
$[r]q_0=\pm q_0$.
\item
More precisely,
these rotations come
in two %
families.
\begin{enumerate}[\rm(a)]
\item
The rotations with $[r]q_0=q_0$ can be written as
  $[l, \exp q_0\theta]$
  for $\theta \in \R$,
and they map $\vec\H^{q_0}$ to $\vec\H^{q_0}$,
preserving the orientation of the circles.
\item
The rotations with $[r]q_0=-q_0$ can be written as
  $[l, q'\exp q_0\theta]$
  for $\theta \in \R$, where
   $q' \in S^2$ is some fixed quaternion orthogonal to $q_0$.
They map $\vec\H^{q_0}$ to $\vec\H^{-q_0}$,
reversing the %
orientation of the circles.
\end{enumerate}
\end{enumerate}
\end{proposition}

Note that an \OR\ transformation sends a left Hopf bundle to a right one,
and those two share exactly two circles.
Thus, no \OR\ transformation can
preserve a Hopf bundle.

\subsubsection{Left and right screws}
\label{sec:screws}

A generic rotation has two circles that it leaves invariant. The left
and right rotations are special: they have infinitely many invariant
circles, and as we will see, these circles form a Hopf bundle.
In contrast to Proposition~\ref{prop:transformations_preserving_H2},
we now discuss rotations that leave {every} \emph{individual} circle of a
Hopf bundle invariant:
\begin{proposition}
  \ \label{one-parameter-rotations}
  \begin{enumerate} [\rm(i)]
  \item 
  For the oriented left Hopf bundle $\vec \H^{q_0}$,
  the one-parameter subgroup
  of right rotations $[1,\exp  q_0\phi]$ rotates every circle of
  $\vec \H^{q_0}$ in itself by the same angle~$\phi$.
\item
  Conversely, for a right rotation $[1,r]$
  with $r\ne 1,-1$, the set of circles that it leaves invariant forms a
  left Hopf bundle $\H^{q_0}$, and
 $[1,r]$ rotates every circle of $\vec H^{q_0}$ in itself by the same angle~$\phi$.
\end{enumerate}
\end{proposition}
\begin{proof}
Part~(i) is a direct consequence of
the definition~\eqref{eq:orient-circle} of oriented circles.

According to Proposition~\ref{prop:action_on_its_circle},
the right rotation $[1,r]$
leaves a circle $K_p^q$ invariant iff
 $[r]q = q$.
(Case~(b) of Proposition~\ref{prop:action_on_its_circle}, where
$[l]p = -p$, does not apply since $l=1$.)
After writing $r=\exp  q_0\phi$ %
with $\phi\ne 0,\pi$,
the condition
 $[r]q = q$ translates to $q=\pm q_0$, and the circles
 $\{\,K_p^{\pm q_0}\mid p\in S^2\,\}$ form the Hopf bundle
 $ \H^{q_0}$. The last part of the statement repeats~(i).
\end{proof}
 \begin{figure}[htb]
    \centering
    \includegraphics{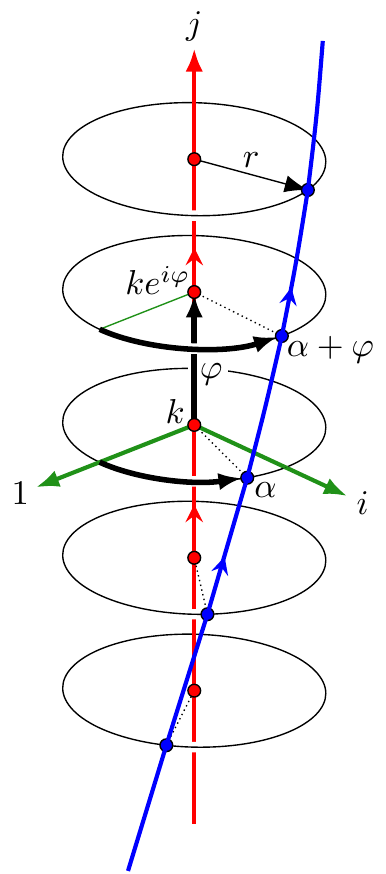}
    \caption{A right screw}
    \label{fig:right-screw}
  \end{figure}
Geometrically,
these rotations
are
\emph{screw motions}.
  If we look at one circle $K_p^{q_0}$ from the bundle,
  the adjacent circles %
form helices that wind around this circle,
see Figure~\ref{fig:right-screw}.
  The right multiplication by
$\exp  q_0\phi$ effects
a forward motion of $\phi$ \emph{along} every circle, and a simultaneous
clockwise rotation by the same angle $\phi$ \emph{around} the circle,
when seen in the direction of the forward movement,
and is thus a 
 \emph{right screw}.\footnote
 {While not everything that is associated to right rotations is ``right'',
   it is a lucky coincidence that at
   least right rotations effect right
   screws,
and left rotations effect left screws.
   This view depends on the convention that we have chosen
   in Section~\ref{sec:view-orientation}
    for
viewing parts of the 3-sphere as three-dimensional space.

 Here is a check of this fact at an example:   
 Figure~\ref{fig:right-screw}
 shows the situation around the point
 $(x_1,x_2,x_3,x_4)=(0,0,0,1)\equiv k \in K_{-i}^{i}$.
 According to our
 conventions from Section~\ref{sec:view-orientation}, we draw this in
 3-space by projecting to the tangent space $x_4=1$, i.e., omitting
 the $x_4$-coordinate, and drawing $(x_1,x_2,x_3)\equiv(1,i,j)$ as a
 right-handed coordinate system.
 The great circle
 $ K_{-i}^{i}$
 is invariant under the family of right rotations
 $[1,\exp i\phi]$, which move the point $k$ along the circle:
 \begin{displaymath}
   \vec K_{-i}^{i} =
   \{\,k\exp i\phi\,\} =
   \{\,k(\cos\phi +i\sin\phi)\,\} =
   \{\,k\cos\phi +j\sin\phi %
   \,\}
\end{displaymath}
The tangent vector at $\phi=0$ points in the direction $j\equiv(0,0,1,0)$.

Let us look at a small circle of radius $r$ around  $ K_{-i}^{i}$,
centered at~$k$:
It lies in a plane parallel to the $1,i$-plane and can be written as
\begin{displaymath}
\tfrac1{\sqrt{1+r^2}}  (k + r(\cos\alpha + i\sin\alpha))
=
\tfrac1{\sqrt{1+r^2}}
  (k + r\exp i\alpha).
\end{displaymath}
The right rotation
 $[1,\exp i\phi]$ maps this to
\begin{displaymath}
\tfrac1{\sqrt{1+r^2}}
  (k + r\exp i\alpha)\exp i\phi
=
\tfrac1{\sqrt{1+r^2}}
  (k\exp i\phi + r\exp i(\alpha+\phi))
\end{displaymath}
i.e., it increases $\alpha$ together with~$\phi$.
As can be seen in 
 Figure~\ref{fig:right-screw}, this is a right screw.
    
Du Val~\cite[\S\,14, p.~36]{duval}, for example, considers right quaternion
multiplications as left screws, without giving reasons for this
choice, and he draws his illustrations accordingly. %
On the other hand,
Coxeter \cite[Chapter~6,
p.~70]{Cox-complex} considers right
quaternion multiplications as right screws.
  }
In contrast to the situation in Euclidean 3-space, these screws have
no distinguished axis. The blue circle
{seems} to wind around the red circle,
but this is an artifact of the projection of this picture.
All circles are in fact equivalent, and the situation looks the
same for every circle of the bundle.

\subsubsection{Clifford-parallel circles}
\label{cliffor-parallel}
We measure the \emph{distance} between two points
 $p, q \in S^3$ as the geodesic distance on the sphere, which equals
 the angular {distance} %
along the great circle through $p$ and $q$:
$\dist(p, q) := \arccos\langle p, q\rangle$,
where $\langle p, q\rangle$ denotes the scalar product. %
The distance between two sets $K, K' \subseteq S^3$ is
$\dist(K, K') = \inf \{\,\dist(p, q) \mid p \in K, q \in K'\,\}.$

Two great circles $K$ and $K'$ in $S^3$ are called \emph{Clifford-parallel}
if $\dist(x, K')$ does not depend on $x \in K$.
See for example \cite[Section 18.8]{berger09_f2} for
 more information on Clifford parallelism.

\begin{proposition}[{\cite[Exercise 18.11.18]{berger09_f2}}]
  \label{prop:dist_bw_clifford_parallel}
  \label{prop:clifford_parallel}
  Great circles in the same Hopf bundle $\H^q$ are Clifford-parallel,
  and  
$\dist(K_p^q, K_r^q) = \dist(p, r)/2$.
\end{proposition}

\begin{proof}
By 
Proposition~\ref{prop:action_on_its_circle}\ref{OP-circle-rotate},
the right rotations $[1, \exp q\theta]$ rotate $x\in K_p^q$ along the
circle
$K_p^q$ while keeping $K_r^q$ invariant as a set. Thus,
$\dist(x, K_r^q)$ is constant as
$x$ moves on $K_p^q$, showing
that $K_p^q$ and $K_r^q$ are Clifford-parallel.
  
Since $K_r^q$ is a left coset of $K_q^q$, 
by applying some left rotation to $K_p^q$ and $K_r^q$,
we may assume that $r = q$.
That is, it is enough to show that
$\dist(K_p^q, K_q^q) = \dist(p, q)/2$.
Since $1 \in K_q^q$ and the circles
$K_p^q$ and $K_p^p$ are Clifford parallel,
it is enough to show that
$\dist(K_p^q, 1) = \dist(p, q)/2$.

The points $x = \cos{\alpha} + v\sin{\alpha}\in K_p^q$
represent the rotations
 $[x]$ on $S^2$ that map $p$ to~$q$,
and $\dist(x, 1) = \arccos \cos \alpha = \alpha$, assuming $0\le\alpha\le\pi$.
Thus,
we are trying to minimize $\alpha$, which is half the rotation angle of~$[x]$.
The rotation
that minimizes the rotation angle
is the one that
moves $p$ to $q$ along the great circle through $p$ and $q$,
and its rotation angle $2\alpha$ is $\dist(p, q)$.
\end{proof}
We mention that Clifford parallelism arises in two kinds: left and
right, accordingly as the circles belong to a common left or right
Hopf bundle.  Each kind of Clifford parallelism is transitive, but
Clifford parallelism in itself is not.

\section{Classification of the point groups}
\label{sec:classification}

We make a coarse classification of the groups by their invariant Hopf
bundles.
The following observation of Dunbar~\cite[p.~124]{dunbar94_f2}
characterizes this
in terms of the left and right groups.

\begin{proposition}
\label{prop:tubical_characterization}
A 4-dimensional point group leaves some left Hopf bundle  invariant
if and only if its right group is cyclic or dihedral.
A similar statement holds for right Hopf bundles and the left group.
\end{proposition}

\begin{proof}
By Proposition~\ref{prop:transformations_preserving_H2}(i), a transformation
$[l, r] \in \SO(4)$ preserves $\H^{q_0}$ if and only if
$[r]$ keeps the line through $q_0$ invariant.
The set of such $r$'s form an infinite group that is isomorphic to $\O(2)$.
Its finite subgroups are either cyclic or dihedral.
\end{proof}

As we have seen,
the left
and right groups $L$ and $R$
are one of the five classes $2I,2O,2T,2D_{2n}$, and $2C_n$.
Besides the infinite families
of cyclic groups $2C_n$ and dihedral groups
 $2D_{2n}$,
there are the three polyhedral groups
$2I,2O,2T$.
Accordingly, we get a rough classification into three classes of groups.

\begin{enumerate}
\item
  The left subgroup is cyclic or dihedral, and the right subgroup is
  polyhedral, or vice versa.

  These groups
  leave some left or right Hopf bundle invariant, and they
  are the \emph{tubical groups}, to be discussed in
  Section~\ref{sec:tubical}.
\item
  Both the left and right subgroup are cyclic or dihedral.

  These groups
  leave some both some left and some right Hopf bundle invariant.
  They form a large family,
 the \emph{toroidal groups}, to be discussed in
  Section~\ref{sec:toroidal}.
  
\item
  Both the left and right subgroup are polyhedral.

  These groups leave no Hopf bundle invariant.
  There are finitely many groups of this class: the polyhedral groups
  and the axial groups.
\end{enumerate}

For all classes except the tubical groups, there is the possibility that
 $L=R$, and hence we also consider the achiral extensions of these groups.

\subsection{The Clifford torus}%
\label{sec:clifford-tori}

The toroidal groups
are characterized as
leaving
both some left Hopf bundle
$\H_{p}$ and some right Hopf bundle
$\H^{q}$ invariant.
By Corollary~\ref{coro:hopf-intersect}, these two bundles intersect
in two orthogonal circles
$K_{p}^{q}\cup K_{p}^{-q}$, and hence these two circles must
also be invariant.
We conclude that the set
$\T_p^q$
of points that are equidistant from these two
circles is also invariant.
We will see that
this set is a \emph{Clifford torus}.
It has several alternative representations.
\begin{align}
     \T_p^q &
    = \{\,x \in S^3 \mid \dist(x, K_p^q) = \dist(x, K_p^{-q}) \,\}
             \label{eq:torus}
  \\&    = \{\,x \in S^3 \mid \dist(x, K_p^q) = \tfrac{\pi}{4} \,\}
   \nonumber 
\\&    = \{\,x \in S^3 \mid \dist(x, K_p^{-q}) = \tfrac{\pi}{4} \,\}
  \nonumber
\\&    = \{\,x \in S^3 \mid \dist(x, K_p^q) = \dist(x, K_{-p}^{q}) \,\}
  \nonumber
\end{align}
Proposition~\ref{prop:relations_between_circles} tells us how an
orthogonal transformation acts on the circle $K_p^q$ that defines the
torus $\T_p^q$.
As an immediate corollary, we obtain:
\begin{proposition}
\label{prop:relations_between_tori}
Let $p, q \in S^2$. Then for any $l, r \in S^3$,
\begin{enumerate}[\rm(i)]
    \item $[l, r] \T_p^q = \T_{[l]p}^{[r]q}$.
    \item $(*[l, r]) \T_p^q = \T_{[l]q}^{[r]p}$, and as a special case,
      $* \T_p^q = \T_{q}^{p}$.
\end{enumerate}
\end{proposition}
From $\T_p^q$, we can recover the two defining circles
$K_{p}^{q}\cup K_{p}^{-q}$ as those points whose distance from
$\T_p^q$ takes the extreme values $\pi/4$:
\begin{displaymath}
  K_{p}^{q}\cup K_{p}^{-q}
= \{x \in S^3 \mid \dist(x, \T_p^{q}) = \tfrac{\pi}{4} \}
\end{displaymath}
Since the choice of parameters $p,q$ for
circles $K_{p}^{q}$ is unique up to simultaneous sign changes,
 the choice of parameters $p,q\in S^2$ for the torus
 $\T_p^q$ is unique
up to independent sign changes:
$\T_p^q = \T_{-p}^{-q} = \T_{-p}^{q} = \T_p^{-q}$.

By Proposition~\ref{prop:relations_between_tori},
any two Clifford tori are related by an appropriate
\OP\ transformation.
There are no ``left'' or ``right'' Clifford tori.
Thus, it is sufficient to study one special torus.
In particular,
$\T_i^i$ is the ``standard'' Clifford torus:
\begin{equation}\label{standard-clifford}
    \T_i^i = 
    \{\,\tfrac{1}{\sqrt{2}}
        (\cos\theta, \sin\theta, \cos\phi, \sin\phi)
        \mid 0 \leq \theta, \phi < 2\pi
        \,\}
=\{\,x \in \mathbb{R}^4\mid x_1^2+y_1^2=   x_2^2+y_2^2=  \tfrac12\,\}
\end{equation}
It is a square flat torus, and we name the coordinates
$(x_1,y_1,x_2,y_2)$ to emphasize that it is the Cartesian product of
a circle of radius $\sqrt{1/2}$
in the $x_1,y_1$-plane and
a circle of radius $\sqrt{1/2}$
 in the $x_2,y_2$-plane.
For this torus,
the two circles of extreme distance are
$K_i^i$ and
$K_i^{-i}$, the great circles in the $x_1,y_1$-plane and
in the $x_2,y_2$-plane.

In Section~\ref{sec:dup-example}, we will see another torus, $\T^i_k$,
with a different, but equally natural equation~\eqref{clif2}.

\goodbreak
\section{The tubical groups}
\label{sec:tubical}

In this section we consider the point groups that preserve a 
left or a right Hopf bundle, but \emph{not both}.
By Proposition~\ref{prop:tubical_characterization}, these groups are characterized
as the groups for which the left or the right group, but not both,
is cyclic or dihedral.
These groups will be called \emph{tubical groups}.
We have chosen this name because,
as we will see (see for instance Figure~\ref{fig:smooth_tubes}),
for large enough order, the
polar orbit polytope consists of
intertwined congruent tube-like structures.\footnote{
There is a notion of \emph{tubular groups}, which is something completely different,
see for example \cite{cashen10-tubular}.} %

Since any two left (resp.\ right) Hopf bundles are congruent, it is enough to
consider the tubical groups that preserve a specific left (resp.\ right) Hopf
bundle.
We will call these the \emph{left tubical groups} and the \emph{right
tubical groups}.
Since left and right Hopf bundles are mirror-congruent, we can
restrict our attention
to the left tubical groups.

 The classic classification leads to 11 classes of left tubical groups.
 Table \ref{tbl:left_tubical_groups} lists them
with the notation from Conway and Smith~\cite[Table~4.1]{CS} in the
first column,
together with their generators.
In Appendix~\ref{sec:subgroups-tubical},
we depict subgroup relations between these groups.

\begin{table}[htb]
\centering
\setlength{\extrarowheight}{1.41pt}
\begin{tabular}{|r@{}l|c|l@{ }l@{ }l|r|c|}
\hline
\multicolumn2{|c|}{$G \leqslant \SO(4)$}& parameter $n$ &
\multicolumn3{l|}{generators}&orde\rlap r&
$G^h \leqslant \O(3)$\\
\hline
\multicolumn8{|c|}{cyclic type} \\ \hline
$\pm[I $&${}\times C_n]$ & $n\ge 1$ &
$[i_I, 1], [\w, 1];$&$ [1, e_n]$& &$120n$ &
$+I$ \\
\hline
$\pm[O $&${}\times C_n]$ & $n\ge 1$ &
$[i_O, 1], [\w, 1];$&$ [1, e_n]$& &$48n$ &
$+O$ \\
$\pm\frac{1}{2}[O$&${}\times C_{2n}]$ & $n\ge 1$ &
$[i, 1], [\w, 1];$&$ [1, e_n];$&$ [i_O, e_{2n}]$ &$48n$ &
$+O$\\
\hline
$\pm[T $&${}\times C_n]$ & $n\ge 1$ &
$[i, 1], [\w, 1];$&$ [1, e_n]$& & $24n$ &
$+T$ \\
$\pm\frac{1}{3}[T$&${}\times C_{3n}]$ & $n\ge 1$ &
$[i, 1];$&$ [1, e_n];$&$ [\w, e_{3n}]$ &$24n$ &
$+T$ \\
\hline
\multicolumn8{|c|}{dihedral type} \\ \hline
$\pm[I $&${}\times D_{2n}]$ & $n\ge 2$ &
$[i_I, 1], [\w, 1];$&$ [1, e_n], [1, j]$& &$240n$ &
$\pm I$ \\
\hline
$\pm[O $&${}\times D_{2n}]$ & $n\ge 2$ &
$[i_O, 1], [\w, 1];$&$ [1, e_n], [1, j]$& &$96n$ &
$\pm O$ \\
$\pm\frac{1}{2}[O$&${}\times\overline{D}_{4n}]$ & $n\ge 2$ &
$[i, 1], [\w, 1];$&$ [1, e_n], [1, j];$&$ [i_O, e_{2n}]$ &$96n$ &
$\pm O$\\
\hline
$\pm\frac{1}{2}[O$&${}\times D_{2n}]$ & $n\ge 2$ &
$[i, 1], [\w, 1];$&$ [1, e_n];$&$ [i_O, j]$ &$48n$ &
$TO$\\
$\pm\frac{1}{6}[O$&${}\times D_{6n}]$ & $n\ge 1$ &
  $[i, 1]%
  ;$&$ [1, e_n] ;$&$ [i_O, j], [\w, e_{3n}]$ &$48n$ &
$TO$\\
\hline
$\pm[T $&${}\times D_{2n}]$ & $n\ge 2$ &
$[i, 1], [\w, 1];$&$ [1, e_n], [1, j]$& &$48n$ &
$\pm T$ \\
\hline
\end{tabular}
\caption
[The 11 classes of left tubical groups]
{Left tubical groups \cite[Table~4.1]{CS}.
  See \eqref{eq:defining_quaternions}
  on p.~\pageref{eq:defining_quaternions}
  for definitions of the quaternions $i_I,i_O,\omega,e_n$.
}
\label{tbl:left_tubical_groups}
\end{table}

According to the right group, there
are 5 tubical group classes of \emph{cyclic type} and
6 tubical group classes of \emph{dihedral type}.
The left Hopf bundle that they leave invariant is $\H^i$.
This follows from Proposition~\ref{prop:transformations_preserving_H2}(ii)
and our choice for the generators of $2C_n$ and $2D_{2n}$.
The cyclic-type groups are those tubical groups that moreover preserve
the consistent orientation of the circles in $\H^i$.
That is, they preserve $\vec \H^i$.
Each of these classes is parameterized by a positive integer $n$,
which is the largest integer $n$ such that $[1, e_n]$ is in the group.

In some cases the parameter $n$ starts from 2 in order to exclude the
groups $D_2$, which is geometrically the same as $C_2$.
We also exclude $\pm\tfrac12[O\times \overline{D}_4]$ because
the notation $\overline{D}_{4n}$ indicates that the normal
subgroup $D_{2n}$ of $D_{4n}$ is used, and not $C_{2n}$.
For $n=1$, this
distinction disappears, and hence
$\pm\tfrac12[O\times \overline{D}_4]$ is geometrically the same as
$\pm\tfrac12[O\times D_4]$
(see also Appendix~\ref{sec:index4}).
In this case and in all other cases where
$C_2$ and $D_2$ are exchanged, the respective groups
are conjugate under $[1, \tfrac1{\sqrt2}(i+j)]$,
which exchanges $[1, i]$ with $[1, j]$. %

\paragraph{Convention.}

For ease of use,
we drop the word ``left'' from ``left tubical group''
and call it simply ``tubical group'' in this section.
We will denote $\H^i$ by $\H$ and call it
\emph{the} %
Hopf bundle.
We will also denote $h^i(x)=xi\bar x$ by $h(x)$ and call it
\emph{the} %
Hopf map.

\subsection{Orbit circles}
\label{sec:orbit-circles}

An element of a %
tubical group has one of the following
two forms, and
Proposition~\ref{prop:action_on_its_circle} describes its action on
the circles of~$\H$:
\begin{itemize}
    \item $[l, e_m^s]$, which maps $\vec K_p$ to $\vec K_{[l]p}$, and
    \item $[l, je_m^s]$, which maps
      $ K_p$ to $ K_{-[l]p}$ with a reversal of orientation.
      More precisely, this rotation maps
      $\vec K_p=\vec K_p^i$
      to $\vec K^{-i}_{[l]p}$, which is the reverse
      of 
      $\vec K^i_{-[l]p} = \vec K_{-[l]p}$.
      These elements occur only in the groups of dihedral type.
\end{itemize}
Thus, the rotations permute the Hopf circles of $\H$.
Via the
one-to-one correspondence of the
Hopf map, they induce mappings on the Hopf sphere $S^2$:

\begin{proposition}
A tubical group $G$ induces a $3$-dimensional point group $G^h$ via
the Hopf map~$h$.
This group $G^h$ is isomorphic to $G/\langle[1, e_n]\rangle$,
where $n$ is the largest integer such that $[1, e_n] \in G$.
\end{proposition}

\begin{proof}
  The above considerations show
that $[l, e_m^s]$ induces the \OP\ transformation $[l]$ on $S^2$,
and $[l, je_m^s]$ induces the \OR\ transformation $-[l]$ on $S^2$.
We are done since the image of $G$ in the %
homomorphism 
\begin{align*}
    G &\to \O(3) \\
    [l, e_m^s] &\mapsto [l]\\
    [l, je_m^s] &\mapsto -[l]
\end{align*}
is $G^h$, and the kernel is $\langle[1, e_n]\rangle$.
\end{proof}

The column ``$G^h \leqslant \O(3)$'' in Table~\ref{tbl:left_tubical_groups}
lists the induced group for each tubical group $G$.
Tubical groups of cyclic type induce chiral groups $G^h$, and
tubical groups of dihedral type induce achiral groups $G^h$.

    As a consequence, the orbit of some starting point $v\in S^3$ can be
    determined as follows:
    \begin{enumerate}
    \item The starting point lies on the circle $K_{h(v)}$.
      The subgroup $\langle[1, e_n]\rangle$
generates a regular $2n$-gon in this circle.
      
\item For each $t\in G^h$,
there is a coset of elements that map
$K_{h(v)}$ to the circle $K_{t(h(v))}$,
and these elements
generate a regular $2n$-gon in this circle.
\end{enumerate}
\begin{proposition}
  \label{prop:orbit-circles}
  Let $G$ be a tubical group.
  The orbit of a point $v \in S^3$ is the union
  of  regular $2n$-gons on the circles
  $K_{t(h(v))}$ for $t \in  G^h$.
  \qed
\end{proposition}
We call these circles the \emph{orbit circles} of~$G$.

If the $G^h$-orbit of $h(v)$ is not free, several of these $2n$-gons
will share the same circle, and they may overlap.  The $2n$-gons may
coincide, or they may form polygons with more vertices.
It turns out that they can intersperse to form a regular $2fn$-gon
or, in the case of dihedral-type groups, the union of two
 regular $2fn$-gons, for
some $1\le f\le 5$.  

The $G^h$-orbit of $h(v)$ is always free when
the starting point does not lie on a rotation
center or a mirror of $G^h$.
The following corollary follows directly 
from the previous proposition.

\begin{corollary}
\label{coro:free_orbit}
Let $G$ be a tubical group and let $v \in S^3$ be a point.
If the $G^h$-orbit of $h(v)$ is free,
then the $G$-orbit of $v$ is also free.
\qed
\end{corollary}

For tubical groups of cyclic type, the orbit has the following nice
property.%

\begin{proposition}
\label{prop:cyclic_type_orbit}
Let $G$ be a cyclic-type tubical group.
The $G$-orbit of a point $v \in S^3$, up to congruence,
depends only on the circle of $\H$ on which $v$ lies.
\end{proposition}

\begin{proof}
Rotation of $v$ along $K_{h(v)}$ can be performed by
a right rotation of the form $[1, \exp \theta i]$.
Since the right group of $G$ is cyclic,
elements of $G$ have the form $[l, e_m^s]$.
These elements commute with right rotations of the form $[1, \exp \theta i]$.
In particular,
\begin{displaymath}
  \orbit([1, \exp \theta i] v, G) = [1, \exp \theta i] \orbit(v, G).
  \qedhere
\end{displaymath}
\end{proof}

\subsection{Tubes}
\label{sec:tubes}

If $n$ is large, the orbit fills the orbit circles
densely.
Figure~\ref{subfig:smooth_tubes_discrete} shows
 the cells (i.e.\ facets) of the polar orbit polytope
 that correspond to orbit points on three orbit circles.
 Here orbit points form a regular 80-gon on each orbit circle.
We clearly see twisted and intertwined tubes, which are characteristic
for these groups, and which we have used to assign
their names.
Figures \ref{subfig:smooth_tubes_cell} and~\ref{subfig:smooth_tubes_top_view}
show a single cell. It has two
large flat faces, where successive cells are stacked on top of each
other with a slight twist.
On the boundary of the tubes in
Figure~\ref{subfig:smooth_tubes_discrete} we can distinguish two different sets of
``parallel'' curves. One set of curves comes from the boundaries between
successive \emph{slices} (cells) of the tubes, and the other set of curves is a
trace of the slices of the adjacent tubes.
At first sight, it is hard to know which of the two line patterns is
which.
In Figure~\ref{subfig:smooth_tubes_discrete_cut}, we have cut the
tubes open to show where the boundaries between the slices are,
revealing also the three orbit circles.

If we let $n$ grow to infinity, the tubes become smooth, see
Figure~\ref{subfig:smooth_tubes}.
We explore the limiting shape of these tubes in
Section~\ref{sec:geometry-tubes}. We will see that the
tubes are either 3-sided, 4-sided, or 5-sided, and their shape as
well as their structure, how they share common boundaries and how they meet
around edges, can be understood in terms of the spherical Voronoi
diagram on the Hopf sphere $S^2$.
Figure~\ref{subfig:smooth_tubes_voronoi} shows this Voronoi diagram
for our example.

We will show some more examples of cells below
(Figures~\ref{fig:IxCn_5fold} and~\ref{fig:hOxC2n_4fold}) and
in Appendix~\ref{sec:special_starting_points}.
In general, the cell of a polar orbit polytope
of a tubical group
for large enough $n$
will always exhibit the following characteristic features.
\begin{itemize}
\item It is a thin slice with a roughly polygonal shape.
\item The top and bottom faces are parallel.
\item Moreover, the top and bottom faces are congruent and slightly
  twisted with a right screw.  (There are, however exceptions
for tubical groups of dihedral type: With some choices
  of starting points, there is an
  alternative way of stacking the slices: every other slice is
  upside down, as in Figure~\ref{fig:pancakes}.)
\item The top and bottom faces approach the shape of a triangle,
  quadrilateral or pentagon with curved sides.
\item The sides are decorated with slanted patterns, which come from
  the boundaries of the adjacent tubes.
\item The tube twists around the orbit circle
by one full $360^\circ$ turn 
  as it closes up on itself.
\end{itemize}

If $n$ is small, these properties break down:
The circles are not filled densely enough to ensure that
the
cells are thin slices. %
Sometimes they are regular or Archimedean
polytopes,
and the
orbit polytopes coincide with those of
polyhedral groups, and
the ``tubes'' may even
be disconnected, see for example Figures~\ref{fig:IxCn_3fold}
or~\ref{fig:TxCn_2fold}
in Appendix~\ref{sec:special_starting_points}.
See~Section~\ref{small-n} for more examples.

Figure~\ref{fig:smooth_tubes} shows a case where the $2n$-gons lie
on different circles. Then the orbit is free: for any two cells, there is a
unique transformation in the group that moves one cell to the other.
If the starting point is generic enough, the cells have no
symmetries.
(See Proposition~\ref{prop:generic-start} below for a precise statement.)
Then the given group is the symmetry group of its orbit
polytope:  There is a unique transformation mapping one cell to the
other even among \emph{all} orthogonal transformations, not just the
group elements.

\subsubsection{Mapping between adjacent cells}

\begin{definition}
The \emph{cell axis} of a cell of the polar
orbit polytope is the orthogonal projection
of the orbit circle into the 3-dimensional hyperplane of the cell.
\end{definition}
The cell axis thus gives the direction in which
consecutive cells are stacked upon each other along
the orbit circle.
It is a line going through the orbit point.
Figure~\ref{subfig:smooth_tubes_cell} shows a cell together with its axis.
The cell axis is not necessarily a symmetry axis.
The cell axis intersects the boundary of the cell in two
\emph{poles}. %

This is where consecutive cells are attached to each other
(unless $n$ is too small and the tubes are disconnected.)
More precisely: For the orbit polytope of a generic starting point,
the next cell is attached as follows. We translate the cell $C$ from the
bottom pole to the top pole. Call the new cell $C'$.
We rotate $C'$ slightly until its bottom
face matches the top face of $C$, and we attach it there (with a bend
into the fourth dimension, as for every polytope).

\begin{figure}
\begin{subfigure}[b]{0.5\textwidth}
    \centering
    \includegraphics[height=75mm]{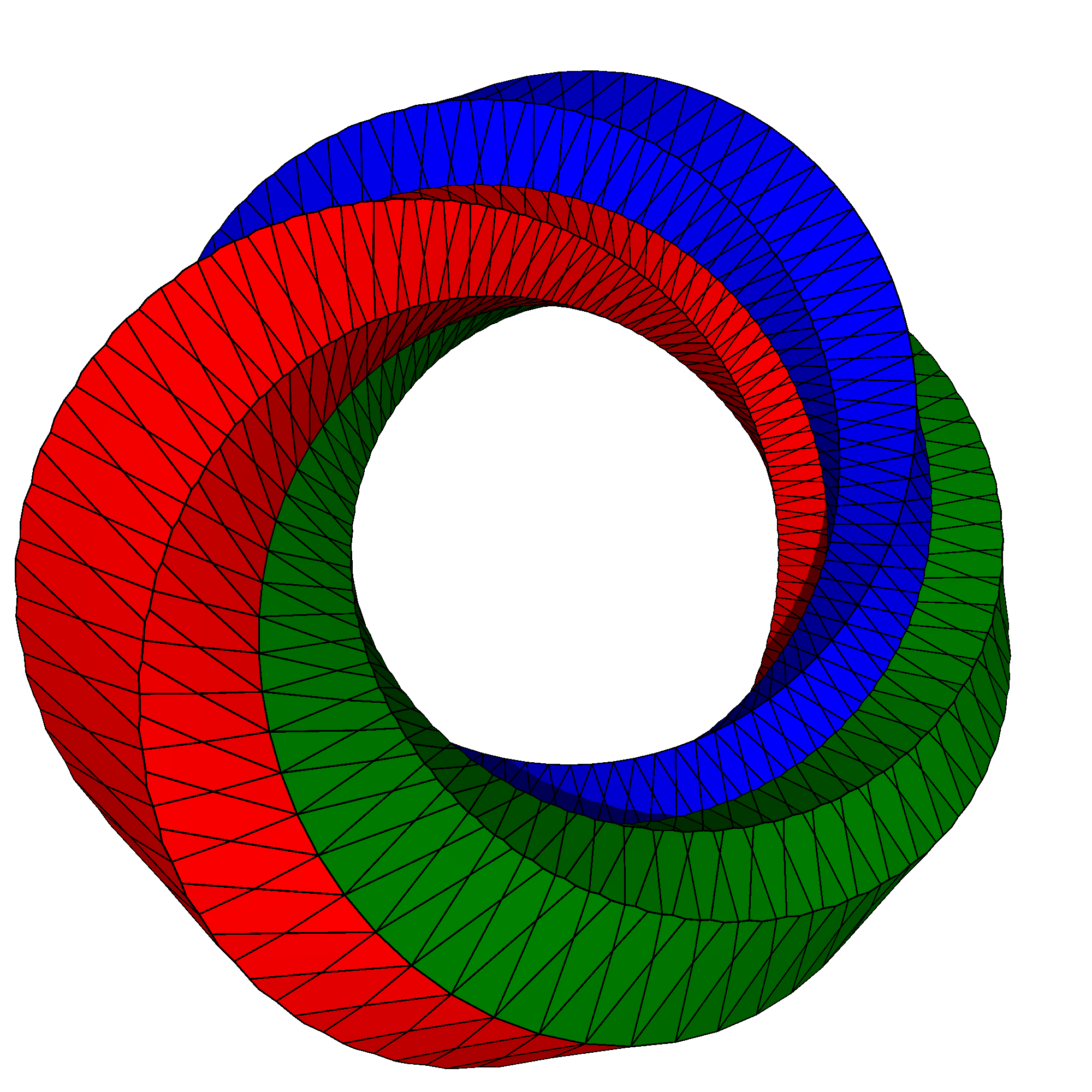}
    \caption{Three tubes}
    \label{subfig:smooth_tubes_discrete}
\end{subfigure}%
\begin{subfigure}[b]{0.5\textwidth}
    \centering
    \includegraphics[height=75mm]{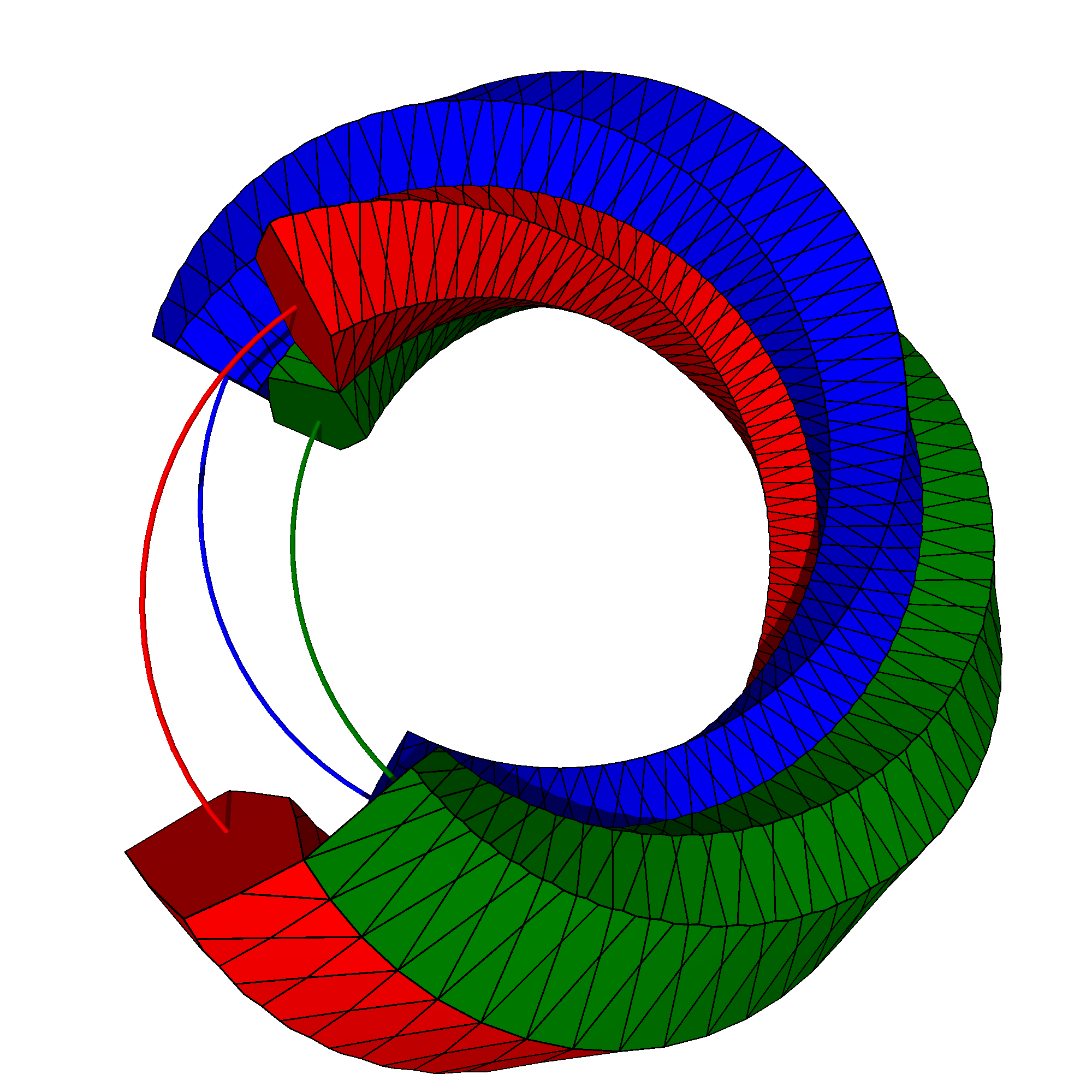}
    \caption{Three partial tubes}
    \label{subfig:smooth_tubes_discrete_cut}
\end{subfigure}%

\begin{subfigure}[b]{0.5\textwidth}
    \centering
    \noindent
    \raisebox{1.5cm}{
    \includegraphics[width=5.5cm]{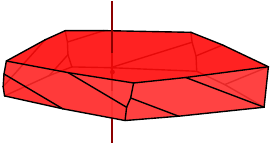}%
}
    \vskip 3mm
    \caption{A single cell}
    \label{subfig:smooth_tubes_cell}
\end{subfigure}
\begin{subfigure}[b]{0.5\textwidth}
    \centering
    \includegraphics[height=75mm]{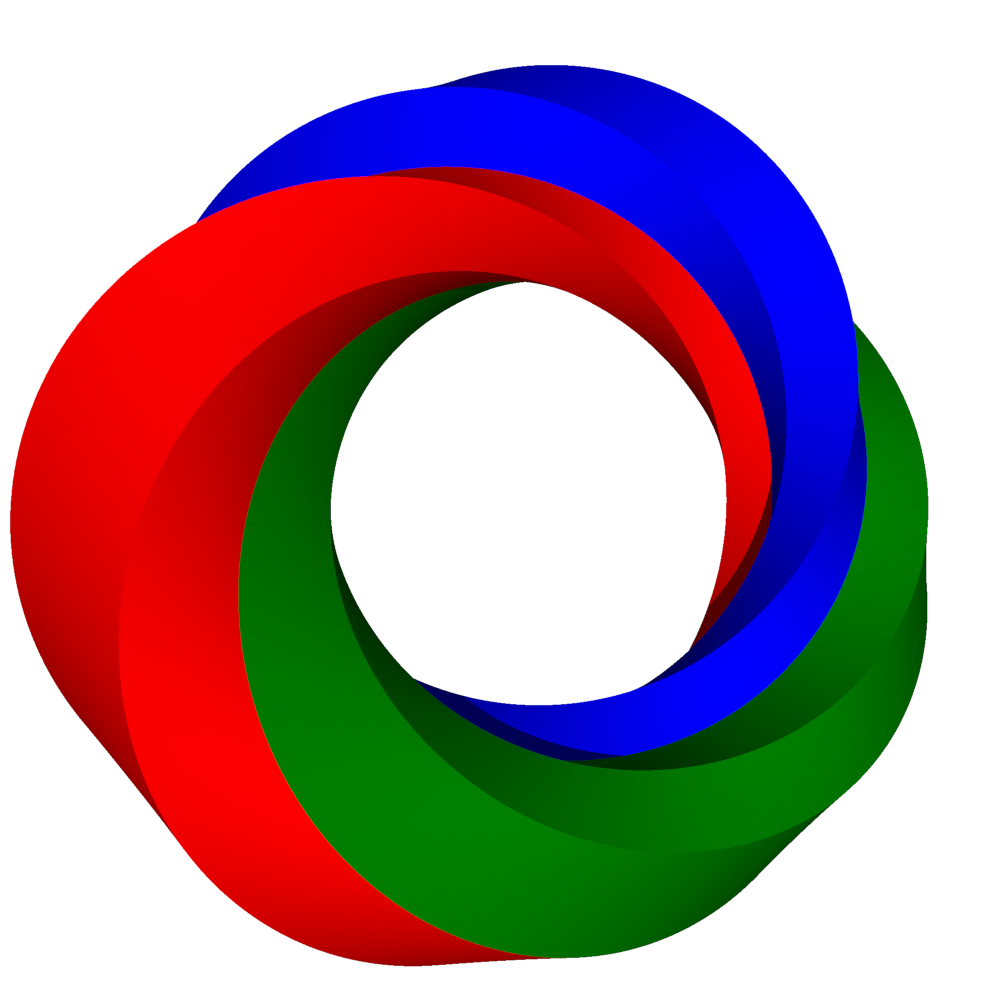}
    \caption{Three smooth tubes}
    \label{subfig:smooth_tubes}
  \end{subfigure}

  \bigskip

\begin{subfigure}[b]{0.5\textwidth}
    \centering
    \includegraphics[width=5.5cm]{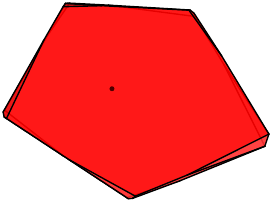}
    \caption{Top view of a cell in orthogonal
      projection
    }
    \label{subfig:smooth_tubes_top_view}
\end{subfigure}
\begin{subfigure}[b]{0.5\textwidth}
    \centering
    \includegraphics[scale=1.2]{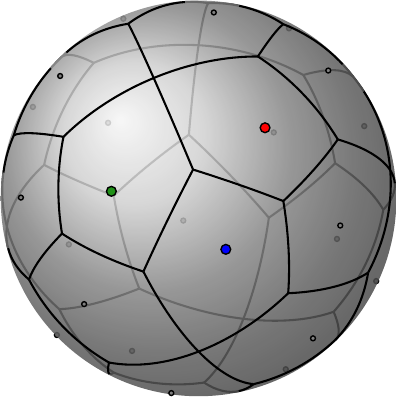}
    \caption{Voronoi diagram on the Hopf sphere}
    \label{subfig:smooth_tubes_voronoi}
\end{subfigure}%
\caption{
(a) Three tubes (out of twenty-four) of the polar $\pm[O\times C_n]$-orbit
polytope for a generic starting point $v$ and $n=40$.
Each tube consists of 80 cells (slices).
The tubes are shown in a central projection.
(b) Some of the cells are removed to make the slices visible.
We also show the corresponding orbit circles. %
(c) A single cell (with its cell axis) from those tubes, in a
perspective view from the side, and
(e) a top view in orthogonal projection.
(f) The spherical Voronoi diagram of the ${+O}$-orbit of $h(v)$.
The colored points correspond to the tubes of the same color.
(d) The tubes as $n$ goes to infinity.
}
\label{fig:smooth_tubes}
\end{figure}

\subsection{The geometry of the tubes}
\label{sec:geometry-tubes}

We investigate the structure of the tubes in the limiting case as
$n\to \infty$, where they become smooth objects.
As $n$ gets larger,
the orbit circle is filled more and more densely,
and the slices get thinner.
In the limit,
every slice becomes a flat plane convex region,
which we call a \emph{tangential slice}.
The tangential slices around an orbit circle sweep out
the \emph{tangential tube} as $v$ moves around the circle.
The limit of the polar orbit polytope consists of tangential tubes,
and this is what is shown in Figure~\ref{subfig:smooth_tubes}.
The central projections of these tubes and
slices to the sphere are the
\emph{spherical tubes} and the \emph{spherical slices}
of these tubes.
The spherical tubes are the Voronoi diagram on $S^3$ of the orbit circles.

This gives us a way to generalize these notations to 
any finite set of circles from a common Hopf bundle.
For that we first need the definition of the spherical Voronoi diagram.
Let $\mathcal{X}$ be a finite collection of nonempty subsets of $S^d$,
and let $X \in \mathcal{X}$ be one of these subsets.
The \emph{spherical Voronoi cell} of $X$
with respect to $\mathcal{X}$ is 
\begin{displaymath}
\vor_\mathcal{X}(X) := \{ x \in S^d \mid \dist(x, X) \leq \dist(x, Y) \text{ for all } Y \in \mathcal{X} \}.
\end{displaymath}
The spherical Voronoi cells of the subsets in $\mathcal{X}$
give a decomposition of $S^d$,
denoted by $\vor_\mathcal{X}$ and
called the \emph{spherical Voronoi diagram}.
If the subsets in $\mathcal{X}$ are singletons,
we get the usual spherical Voronoi diagram.

Let $\mathcal{C}$ be a finite set of at least two circles from a common
Hopf bundle,
and let $K \in \mathcal{C}$ be one of them.
We can assume that the common Hopf bundle is $\H$.
The Voronoi cell of~$K$ with respect to $\mathcal{C}$ is 
called a \emph{spherical tube}.
The intersection of $\vor_\mathcal{C}(K)$ with
the hyperplane perpendicular to $K$ at a point $v \in K$
gives two (2-dimensional) patches. One contains $v$ and one contains $-v$.
These are \emph{spherical slices}.
The \emph{tangential slices} and \emph{tangential tubes}
are defined as above in the special case of orbit circles.

 We will show that
the spherical tubes are bounded by patches of Clifford tori
(Theorem~\ref{thm:boundaries_of_tubes}),
and the tangential slices are polygons of circular arcs
(Theorem~\ref{thm:tangential_slice}).

\subsubsection{The spherical tubes}
Given that the circles belong to a common Hopf bundle and
the Hopf map transforms distances appropriately
(Proposition~\ref{prop:dist_bw_clifford_parallel}),
it is no surprise that %
the Voronoi diagram of
the set of \emph{circles} on $S^3$ is closely related to the Voronoi
diagram of the corresponding \emph{points} on $S^2$ (see Figure~\ref{subfig:smooth_tubes_voronoi}.)

\begin{proposition}
\label{prop:tube_circles}
Let $\mathcal{C} \subset \H$ be a finite set of circles from $\H$,
and let $K \in \mathcal{C}$ be one of them.
The spherical tube $\vor_\mathcal{C}(K)$ is the union of circles from $\H$
that are the preimages under $h$ of the points in
$\vor_{h(\mathcal{C})}(h(K))$, where
$h(\mathcal{C}):=\{\,h(C) \mid C \in \mathcal{C}\,\}$.
\end{proposition}

\begin{proof}
First we will show that for any point $x' \in \vor_\mathcal{C}(K)$,
the great circle $K'$ from $\H$ on which $x'$ lies is also in
$\vor_\mathcal{C}(K)$.
Since all the circles in $\H$ are Clifford-parallel
(Proposition~\ref{prop:clifford_parallel}),
$\dist(K', C) = \dist(x', C)$ for all $C \in \mathcal{C}$.
Thus, we get the following equivalence.
\begin{displaymath}
\dist(x', K) \leq \dist(x', C) \iff \dist(K', K) \leq \dist(K', C),
\end{displaymath}
for all $C \in \mathcal{C}$.
That is, $K' \subset \vor_\mathcal{C}(K)$.
By Proposition~\ref{prop:dist_bw_clifford_parallel}
we know that
\begin{displaymath}
\dist(K', K) \leq \dist(K', C) \iff
\dist(h(K'), h(K)) \leq \dist(h(K'), h(C)),
\end{displaymath}
for all $C \in \mathcal{C}$.
That is, $K' \in \vor_\mathcal{C}(K)$ if and only if
$h(K') \in \vor_{h(\mathcal{C})}\bigl(h(K)\bigr)$.
\end{proof}

\subsubsection{The spherical tube boundaries}

\begin{theorem}
\label{thm:boundaries_of_tubes}
Let $\mathcal{C} \subset \H$ be a finite set of circles from $\H$.
The boundaries of the corresponding spherical tubes consist of
patches of Clifford tori.
The edges of these tubes are great circles from~$\H$.
\end{theorem}
\begin{proof}
  As in Proposition~\ref{prop:tube_circles}, the boundary between two
  tubes is the preimage, under the Hopf map $h$, of the boundary
  between the two corresponding Voronoi regions in
  $\vor_{h(\mathcal{C})}$.
  Such a boundary edge on the Hopf sphere~$S^2$ is contained in a great
  circle.
  A great circle can be described as the points that are equidistant from two
  antipodal points $\pm p$ on~$S^2$, and under the inverse Hopf map, these
  become the points on $S^3$ that are equidistant from two absolutely
  orthogonal circles $K_p$ and $K_{-p}$, and this is, by definition, a Clifford torus.

The tube edges, where three or more tubes meet,
are the preimages of the Voronoi vertices of
$\vor_{h(\mathcal{C})}$.
Thus, they are circles from $\H$.
\end{proof}

\subsubsection{The tangential slices}

\begin{theorem}
\label{thm:tangential_slice}
Let $\mathcal{C} \subset \H$ be a finite set of circles from $\H$.
The corresponding tangential slices are
(flat) convex regions bounded by circular arcs.
\end{theorem}

\begin{proof}
Let $K \in \mathcal{C}$ be one of the circles.
We want to consider the tangential slice of $K$ at a point $v \in K$.
Without loss of generality, we may assume that $v = i$, because
the left rotation $[-vi, 1]$ preserves $\H$
(see Proposition~\ref{prop-congruent-bundles}(i))
and maps $v$ to $i$.
Then $K$ is actually $K_i$,
the great circle through the points $1$ and~$i$.

The tangent direction of $K$ at $v$ is the quaternion 1.
The hyperplane $Q$ perpendicular to $K$ at $v$ is 
spanned by $i$, $j$ and $k$,
which we represent in a 3-dimensional coordinate
system $\hat{x},\hat{y},\hat{z}$,
see Figure~\ref{subfig:S0}.
$Q$ intersects $S^3$ in a great 2-sphere $S_0$.
 The spherical tube $\vor_\mathcal{C}(K)$ cuts out two opposite patches from $S_0$:
the spherical slices.
Denote by $A$ the slice that contains $v$.
The slice $A$ intersects each circle of $\vor_\mathcal{C}(K)$.
Thus, by Proposition~\ref{prop:tube_circles},
$h(A)$ equals $\vor_{h(\mathcal{C})}(h(v))$,
which we will denote by~$B$.

Using spherical coordinates,
a point in $S_0$ has the form
$i\cos\theta  + p\sin\theta$,
where
the direction vector
$p$ is a unit vector in the $\hat{y},\hat{z}$-plane that plays the role of the longitude,
and $\theta \in \R$ is the angular distance
on $S_0$ between that point and $i$.
See Figure~\ref{fig:radial_contraction}.
Since $p$ and $i$ are pure unit quaternions, they anticommute, and
in particular, $pip %
= -i{p}p = i$.
We will now apply the Hopf map $h$ to a point in $S_0$:
\begin{align*}
h(i\cos\theta + p\sin\theta ) &=
(i\cos\theta  + p\sin\theta  )\, i\, (-i\cos\theta  -p\sin\theta  ) \\
&= i\cos^2{\theta} -pip\sin^2{\theta} + p \cos\theta\sin\theta  + p\cos\theta\sin\theta  \\
&= i(\cos^2{\theta} - \sin^2{\theta})  + 2p\cos\theta\sin\theta  \\
&= i\cos{2\theta} + p\sin{2\theta} .
\end{align*}
That is, $h$ maps a point whose angular distance from $i$ is $\theta$
to the point in the same direction but with angular distance $2\theta$.
Thus, if we identify $S_0$ with $S^2$ using the natural identification
(on $S^2$, we denote the $i$, $j$ and $k$ directions
by $x$, $y$ and $z$, respectively),
we see that $A$ is obtained %
from $B$ by
a \emph{radial contraction}. That is, we look from $i$ in all directions
and multiply the angular distance between $i$ and each point in $B$ by~$1/2$.

\begin{figure}
\null  \hskip -5mm
\begin{subfigure}{.58\textwidth}
    \centering
    \includegraphics[scale=1.1]{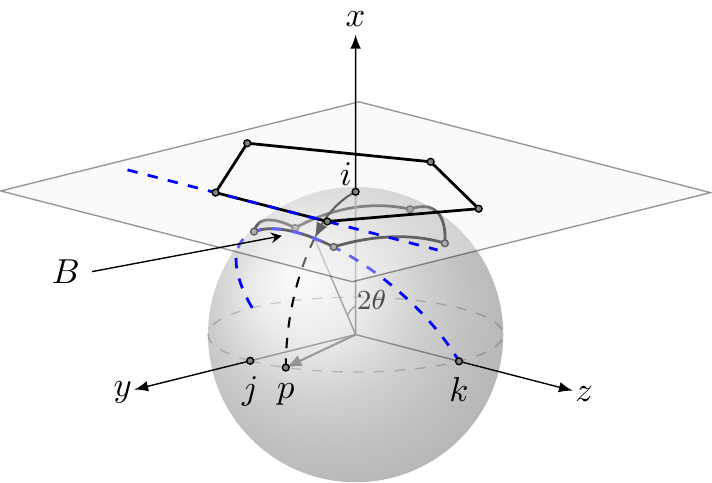}
    \caption{$S^2$}
    \label{subfig:S2}
\end{subfigure}%
\begin{subfigure}{.45\textwidth}
    \centering
    \includegraphics[scale=1.1]{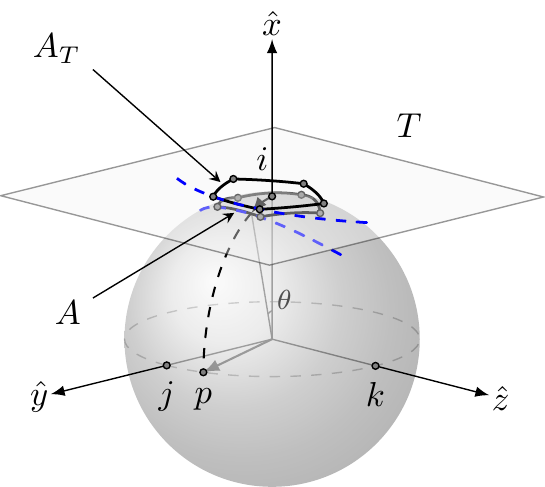}
    \caption{$S_0 \subset Q$}
    \label{subfig:S0}
\end{subfigure}
\caption{
The procedure to get the tangential slice $A_T$.
(a)
The spherical pentagon $B$ is the Voronoi cell of the
point~$i=h(K_i)$ with respect to $h(\mathcal{C})$.
The pentagon in the plane passing through~$i$ is
the central projection of $B$ onto that plane.
(b)
The spherical pentagon $A$ is the spherical slice at~$i$,
which we get from a radial contraction of $B$.
The circular-arc pentagon $A_T$ in the tangent plane $T$ passing through~$i$ is
the corresponding tangential slice,
which we get from a central projection of $A$ to~$T$. %
This example is constructed from
the orbit circles of %
Figure~\ref{fig:smooth_tubes}.
}
\label{fig:radial_contraction}
\end{figure}

The intersection of $Q$ with
the (3-dimensional) tangent space of $S^3$ at $v$ is
the 2-dimensional tangent plane $T$ of $S_0$ at $v$.
For our choice $v=i$, $T$ is the plane in $Q$ defined by $\hat{x}=1$.
 The tangential slice lies in this plane.%

So to get the tangential slice $A_T$ at $v$,
we radially contract $B$ to get $A$,
and then centrally project $A$ to $T$.
We will describe this procedure algebraically.
The radial contraction towards $i$ is the map
\begin{displaymath}
    i\cos\theta  + p\sin\theta 
    \mapsto 
    i\cos\tfrac\theta2 + p \sin\tfrac\theta2.
  \end{displaymath}
This map is not uniquely determined at the South Pole ($\theta= \pi$),
and we will tacitly exclude this point from further consideration.
Writing $p$ as $j\cos\phi  + k\sin\phi $,
the map can be described
as follows:
\begin{align*}
    \begin{pmatrix}
        x \\ y \\ z
    \end{pmatrix}
    =
    \begin{pmatrix}
        \cos\theta \\
        \cos\phi\sin\theta \\
        \sin\phi\sin\theta \\ 
    \end{pmatrix}
    \mapsto
    \begin{pmatrix}
        \hat{x} \\ \hat{y} \\ \hat{z}
    \end{pmatrix}
    =
    \begin{pmatrix}
        \cos\frac\theta2 \\
        \cos\phi\sin\frac\theta2 \\
        \sin\phi\sin\frac\theta2 \\ 
    \end{pmatrix}
\end{align*}
Using the identities
$\cos\frac\theta2 = \frac{\sqrt{1+\cos\theta}}{\sqrt2}$ and
$\sin\theta = 2 \sin\frac\theta2\cos\frac\theta2$,
the map is written as follows.
\begin{displaymath}
    (x, y, z) \mapsto
    (\hat{x}, \hat{y}, \hat{z}) = 
    \frac{1}{\sqrt{2}}
    \Bigl(\sqrt{1+x}, \frac{y}{\sqrt{1+x}}, \frac{z}{\sqrt{1+x}}\Bigr)
\end{displaymath}
Combining this with the central projection
from the origin onto $T$
gives the following map $f$.
\begin{displaymath}
    f\colon (x, y, z) \mapsto
    (\hat{x}, \hat{y}, \hat{z}) = 
    \Bigl(1, \frac{y}{1+x}, \frac{z}{1+x}\Bigr)
    =
    \Bigl(1, \frac{y/x}{1+1/x}, \frac{z/x}{1+1/x}\Bigr)
\end{displaymath}

If we apply %
$f$ to a boundary edge of $B$,
it will turn out the resulting curve is part of a circle.
The boundary edges of $B$ are arcs of great circles on $S^2$.
 We obtain such an arc by centrally projecting to $S^2$
a straight segment in the tangent plane of $S^2$ at $h(v)=i$.
Without loss of generality suppose that one of these 
segments lies on the line
$(x,y,z)=%
(1, c_0, t)$, %
$t \in \R %
$,
for some constant $c_0 \neq 0$,
see the blue line in Figure~\ref{subfig:S2}.
The central projection of this line to $S^2$ lies on the great circle
\begin{displaymath}
\Bigl\{\,
\frac{\pm1}{\sqrt{c_0^2+t^2+1}}(1, c_0, t) \Bigm| t \in \R
\,\Bigr\}.
\end{displaymath}
See the blue curve in Figure~\ref{subfig:S2}.
The map $f$ transforms this great circle into the set
\begin{equation}\label{eq:blue-curve}
\Bigl\{\,
\Bigl(1, \frac{ c_0}{1\pm\sqrt{c_0^2+t^2+1}},
\frac{ t}{{1\pm \sqrt{c_0^2+t^2+1}}}
\Bigr)
\Bigm| t \in \R
\,\Bigr\}.
\end{equation}
See the blue curve in the tangent plane in Figure~\ref{subfig:S0}.
Straightforward manipulations show that this set is a circle:
\begin{multline*}
    \hat{y} = \frac{c_0}{1\pm\sqrt{c_0^2 + t^2 + 1}}
    \iff
    \pm \hat{y}\sqrt{c_0^2 + t^2 + 1} = c_0 - \hat{y}  \\
    \iff
    \hat{y}^2c_0^2 + \hat{y}^2t^2 + \hat{y}^2 = \hat{y}^2 -2c_0\hat{y} + \hat{y}_0^2
    \iff
    \hat{y}^2c_0^2 + \hat{y}^2t^2 +2c_0\hat{y} = c_0^2
\end{multline*}
Dividing both sides by $c_0^2$ and then substituting the relation
$
\frac{\hat{z}}{\hat{y}}
=
\frac{t}{c_0}
$, which follows from~\eqref{eq:blue-curve},
gives
\begin{equation}\label{circle}
    \hat{y}^2 + \hat{z}^2 +\frac{2}{c_0}\hat{y} = 1 \iff
    \Bigl(\hat{y}+\frac1{c_0}\Bigr)^2 + \hat{z}^2 = \frac{c_0^2+1}{c_0^2},
\end{equation}
which is the equation of a circle.
\end{proof}
The circle defined in~\eqref{circle}
belongs to the pencil of circles through the points
$(\hat{x},\hat{y},\hat{z})=(1,0,\pm 1)$, because these points fulfill the
equations~\eqref{circle}.
The center
$(\hat{x},\hat{y},\hat{z})=(1,-\frac1{c_0},0)$ lies on the axis
$(\hat{x},\hat{y},\hat{z})=\lambda(c_0,-1,0)$ perpendicular to the plane
$c_0x=y$ containing the great circle and the line that started the
construction.

If the set of great circles $\mathcal{C}$ in the previous theorem
are the orbit circles of a tubical group $G$,
then the spherical Voronoi cell $B$ on $S^2$
can have 3, 4 or 5 sides, because the cells form a tiling of the
sphere with equal cells.
Thus, the spherical slice is also 3, 4 or 5 sided.
In particular, we get the following corollary.

\begin{corollary}
  The tangential slice
  of an orbit of a tubical group
  is a convex plane region whose boundary consists of
 3, 4, or 5 circular arcs.
\end{corollary}

\subsubsection{The tangential tube boundaries}

 The boundary surfaces of the
tangential tubes (shown in Figure~\ref{subfig:smooth_tubes})
carry some interesting
structures, but
we don't know what these surfaces are.

The points on such a surface are equidistant from two circles $K$ and
$K'$,
and we denote the surface by $B(K,K')$.
We know from
Theorem~\ref{thm:boundaries_of_tubes}
that its central projection to the sphere is a Clifford
torus $\T$, whose image $h(\T)$ is the bisector between $h(K)$ and
$h(K')$ on~$S^2$.
According to the relation between Voronoi diagrams and polar orbit polytopes
(as briefly discussed in Section~\ref{sec:voronoi}), a circle $K\in
\H$ that belongs to $\T$ is expanded by some factor, depending on the
distance to $K$ and $K'$, to become a circle on
 $B(K,K')$. Thus,
 the surface $B(K,K')$
 is fibered by circles (of different radii)
around the origin.

Another fibration by circles, this time of equal radii, can be obtained
by taking the circular arc %
that forms the boundary of the tangential slice %
towards $K'$,
and sweeping it along the
circle $K$.  In Figure~\ref{subfig:S0}, the circle $K$ proceeds from
the point $i$ into the fourth dimension, and the circular boundary arc
must simultaneously wind around $K$ as it moves along~$K$.
A third fibration, by circles of the same radius, is obtained in 
an analogous way from $K'$. Each of these fibrations leads to a straightforward parametric
description of $B(K,K')$.

Alternatively,
an implicit description $B(K,K')$ by two equations can be obtained
as the intersection of two ``tangential hypercylinders'' in which the two
tangential tubes of $K$ and $K'$ lie. (If the circle $K$ is
described by the
system $x_1^2+x_2^2=1$,
$x_3=x_4=0$ in an appropriate coordinate system,
its tangential hypercylinder is obtained by omitting the equations
$x_3=x_4=0$.)

\subsection{Generic starting points}
\label{sec:generic-start}

We return to the analysis of the polar orbit polytope, and start with
the easy generic case.

\begin{proposition}
\label{prop:generic-start} 
Let $G$ be a tubical group whose right group is $C_{n}$ or $D_{n}$ 
for $n \ge 6$.
Let $v \in S^3$ be a point.
If the $G^h$-orbit of $h(v)$
has no symmetries other than $G^h$,
then the same holds for the $G$-orbit of $v$:
the symmetry group of this orbit is $G$.
\end{proposition}
\begin{proof}
Since no $C_{n}$ or $D_{n}$ for $n\ge 6$ is contained in a polyhedral group,
the only groups containing $G$ are tubical.
In particular, the symmetry group $H$ of the $G$-orbit of $v$ is tubical.
Since the symmetry group of the $G^h$-orbit of $h(v)$ is $G^h$
by assumption,
the point $h(v)$ does not lie on any rotation center or
a mirror of a supergroup of $G^h$.
In particular, the $H^h$-orbit of $h(v)$ is free.
Thus, by Corollary~\ref{coro:free_orbit}, the $H$-orbit of $v$ is free.
So $G$ and $H$ have the same order.
Since $G \leqslant H$, we get $G = H$.
\end{proof}

According to our goal of obtaining a geometric understanding through
the orbit polytope, as described in
Figure~\ref{fig:geometric-understanding} in
Section~\ref{sec:orbit-polytopes}, we are done, in principle.
Since the cell has no nontrivial symmetries, \emph{all}
symmetries of a cell are in $G$.
We are in the branch of
Figure~\ref{fig:geometric-understanding} that requires no further action.
Every cell can be mapped to every other cell in a unique way.

In particular,
for two consecutive cells on a tube it is obvious what the
transformation between them is: a small translation along the orbit
circle combined with a slight twist around the orbit circle, or in other
words, a right screw, effected by the right rotation $[1,e_n]$.

Between cells on different tubes, the transformation is not so
obvious.  For example, in Figure~\ref{subfig:smooth_tubes_cell}, we
see a vertical zigzag of three short edges between the front
corner of the upper (roughly pentagonal) face and the corresponding
corner of the lower face.
These edges are part of a longer sequence of edges, where 3 tubes meet,
and which closes in a circular way.
How are the cells arranged around this
``axis'', and how does the group map between them? To
investigate this question, it is helpful to move the starting point
closer to the axis to look what happens there. In particular, this
will help us to distinguish different classes of groups $G$ with the same
group $G^h$. 
We will see an example in Section~\ref{sec:same-symmetry}.
Eventually, we will  also consider starting points \emph{on} the axis.

\subsection{Starting point close to a mirror}
\label{sec:near-mirror}
Let $G$ be a dihedral-type tubical group,
and $p \in S^2$ be a point close to the mirror of a reflection of $G^h$.
Moreover, assume that $p$ does not
lie on any rotation center of $G^h$.
The point $p$ has a \emph{neighboring partner} $p'$, which is obtained 
from $p$ by reflecting it across that mirror.
We call the corresponding circles
$K_p$ and $K_{p'}$ \emph{neighboring circles}.
The red point and the blue point in Figure~\ref{fig:neighboring_pair}
form a neighboring pair for the group $\pm T$.

We will now discuss the $G$-orbit under different
choices for the starting point $v$ on $K_p$.

\begin{enumerate}[\textbf{Case}~1.]
\item
Choose $v \in K_p$ such that for each orbit point,
the closest point on the neighboring circle is also in the orbit.
See Figure~\ref{subfig:near_mirror_min_cell}.
Thus, in the polar $G$-orbit polytope,
each cell has a ``big'' face
that directly faces the closest point on the neighboring circle.

\item
If we move $v$ in one direction,
the orbit points on the neighboring circle move in the opposite direction.
We choose $v$ such that the orbit points on neighboring circles are in
``alternating positions''.
That is, the distance between orbit points on
neighboring circles is maximized.
See Figure~\ref{subfig:near_mirror_max_cell}.
Thus, in every cell of the polar $G$-orbit polytope,
the side that is close to the neighboring circle is divided
into two faces, on each a cell of the neighboring tube is stacked.

\item
Figure~\ref{subfig:near_mirror_inb_cell}
shows an intermediate situation.
\end{enumerate}

\iffigurescompact
  \begin{figure}[tb]
\else
  \begin{figure}
\fi
\begin{subfigure}{0.5\textwidth}
    \centering
    \includegraphics[scale=1]{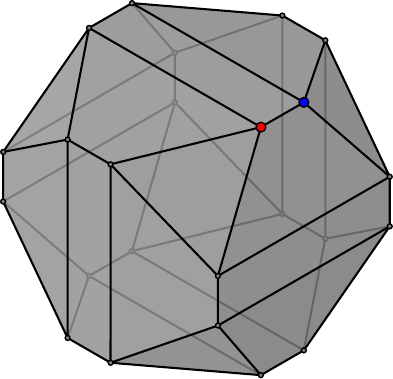}
    \caption{The ${\pm T}$-orbit polytope of $p$.}
    \label{fig:neighboring_pair}
\end{subfigure}%
\begin{subfigure}{0.5\textwidth}
    \centering
    \includegraphics[scale=1]{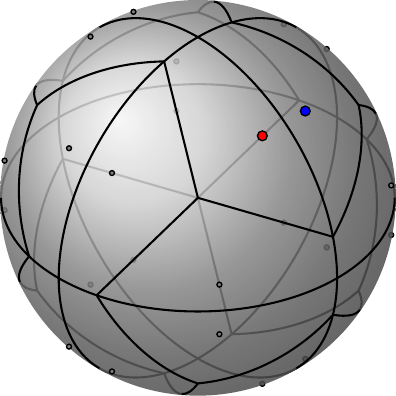}
    \caption{The spherical Voronoi diagram of the orbit.}
\end{subfigure}
\smallskip

\begin{subfigure}{0.33\textwidth}
    \centering
    \includegraphics[scale=1.2]{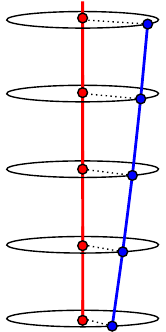}\\
\end{subfigure}%
\begin{subfigure}{0.33\textwidth}
    \centering
    \includegraphics[scale=1.2]{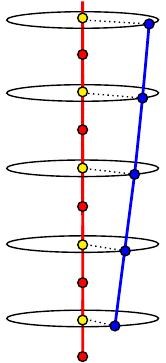}\\
\end{subfigure}%
\begin{subfigure}{0.33\textwidth}
    \centering
    \includegraphics[scale=1.2]{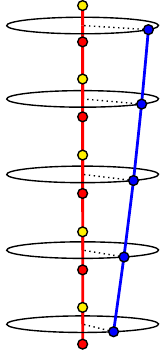}\\
\end{subfigure}
\smallskip

\begin{subfigure}{0.33\textwidth}
    \centering
    \includegraphics[scale=1]{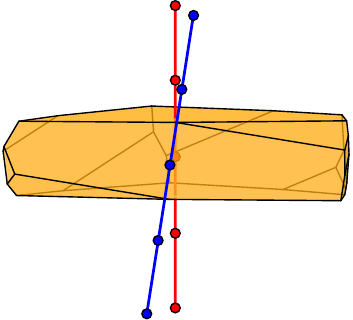}\\
    \caption{}
    \label{subfig:near_mirror_min_cell}
\end{subfigure}%
\begin{subfigure}{0.33\textwidth}
    \centering
    \includegraphics[scale=1]{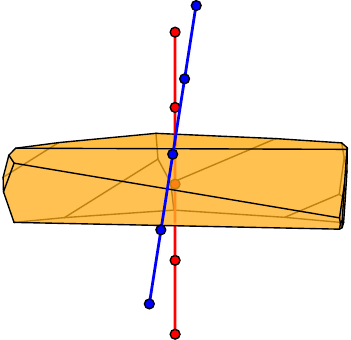}\\
    \caption{}
    \label{subfig:near_mirror_max_cell}
\end{subfigure}%
\begin{subfigure}{0.33\textwidth}
    \centering
    \includegraphics[scale=1]{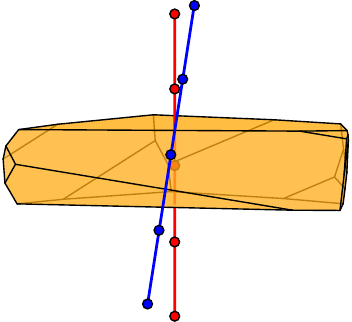}\\
    \caption{}
    \label{subfig:near_mirror_inb_cell}
\end{subfigure}
\caption{Orbits of the group $G=\pm[T\times D_{20}]$
for a starting point $v$ whose image $p:=h(v)$
lies near a mirror of $G^h= {\pm T}$.
The top row shows the three-dimensional ${\pm T}$-orbit polytope of $p$
and the corresponding spherical Voronoi diagram.
The red and the blue points form a neighboring pair.
The next row shows different possible configurations for orbit points
on the corresponding neighboring circles.
Red points and blue points are orbit points on the two neighboring circles.
Yellow points are midpoints of orbit points on the red circle.
They are not orbit points.
The third row shows a cell of the corresponding polar orbit polytope.
}
\label{fig:near_mirror}
\end{figure}

\subsection{Starting point on a mirror}
It is also interesting to see what happens 
if we move $p$ to lie on that mirror of $G^h$.
We still assume that $p$ does not
lie on any rotation center of $G^h$.
In this case, the neighboring pairs on $S^2$ coincide,
and thus the corresponding neighboring circles also coincide.
We describe next what happens in each of the previous cases.

\begin{enumerate}[\textbf{Case}~1.]
\item
The orbit points coincide in pairs,
and thus they form a regular $2n$-gon on $K_p$.
Each orbit point can be mapped to any other orbit point
by two different elements of $G$,
one of which rotates $K_p$ and one of which reverses the orientation of~$K_p$.
Thus, in the polar orbit polytope, 
each cell has a half-turn symmetry that
flips the direction of the cell axis, and
exchanges the top and bottom faces.
We call it a \emph{flip symmetry}.
(For small $n$, top and bottom faces might not be defined.)

It is interesting to notice that for this choice of the starting point,
the $G$-orbit of $v$ coincides with the orbit of $v$ under the 
cyclic-type index-2 subgroup $G_C$ of $G$.
Since the $G_C$-orbit is the same up to congruence for any starting point on $K_p$
(Proposition~\ref{prop:cyclic_type_orbit}),
the $G_C$-orbit of \emph{any} starting point on $K_p$ has the extra
symmetries coming from a dihedral-type group that is
 geometrically equal to $G$.
(This geometrically equal group has the generators of $G$
with $j$ replaced by a different unit quaternion $q'$ orthogonal to~$i$,
which is the quaternion $q'$ from Proposition~\ref{prop:action_on_its_circle}(b).)
We put this in a proposition since we will need it later.

\begin{proposition}
\label{prop:cylic_on_mirror}
Let $G_C$ be a cyclic-type tubical group,
and let $G_D$ be a dihedral-type tubical group
containing $G_C$ as an index-2 subgroup.
If $p$ lies on a mirror of $G_D^h$,
then the $G_C$-orbit of any point on $K_p$ has
the symmetries from (a geometrically equal copy of)~$G_D$.
\end{proposition}

\item
Orbit points on $K_p$ form a regular $4n$-gon.
Each orbit point can be mapped to any other orbit point
by a unique element of $G$.
However, this orbit has extra symmetries, which
 come from the supergroup of $G$ that we obtain by
extending $G$ by the new symmetry $[1, e_{2n}]$.
This orbit of the supergroup follows the behavior described in Case~1.
Accordingly, each cell of the polar $G$-orbit polytope has a flip symmetry.

In almost all choices for $G$,
the supergroup has the same class as $G$
but with twice the parameter $n$.
The only exceptional case is $G = \pm\frac12[O\times \overline{D}_{4n}]$.
In this case, the supergroup is $\pm[O\times D_{4n}]$.

\item
Orbit points on $K_p$ form two regular $2n$-gons
whose union is a $4n$-gon with equal angles,
and side lengths alternating between two values.
The orbit points come in close pairs.
Accordingly, the cells of the polar orbit polytope
come in a sequence of alternating ``up-and-down pancakes''
stacked upon each other. See the two cells in Figure~\ref{fig:pancakes}.

\begin{figure}[ht]
\begin{subfigure}{0.33\textwidth}
\centering
\includegraphics{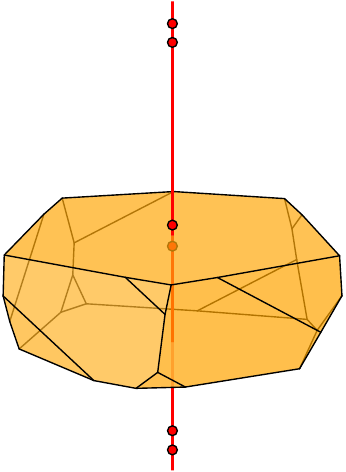}
\end{subfigure}
\begin{subfigure}{0.33\textwidth}
\centering
\includegraphics{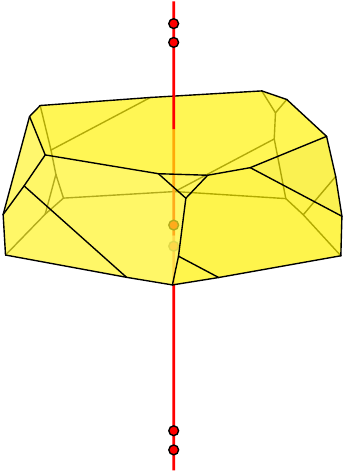}
\end{subfigure}
\begin{subfigure}{0.33\textwidth}
\centering
\includegraphics{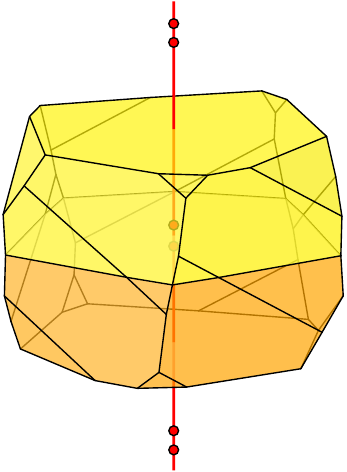}
\end{subfigure}
\caption{Two cells %
stacked upon each other with a $180^\circ$ rotation.
The two left figures show each cell individually.
}
\label{fig:pancakes}
\end{figure}
\end{enumerate}

\iffigurescompact
  \begin{figure}[tb]
\else
  \begin{figure}
\fi
\begin{subfigure}{0.5\textwidth}
    \centering
    \includegraphics[scale=1]{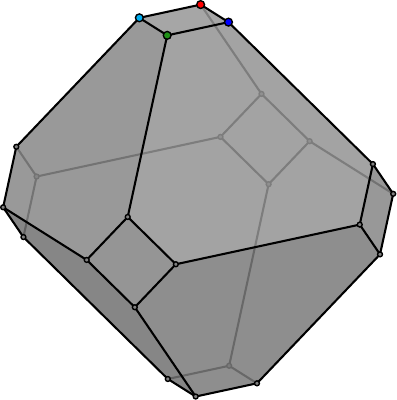}
    \caption{The ${+O}$-orbit polytope of $p$.}
    \label{fig:near_4fold_orbit}
\end{subfigure}%
\begin{subfigure}{0.5\textwidth}
    \centering
    \includegraphics[scale=1]{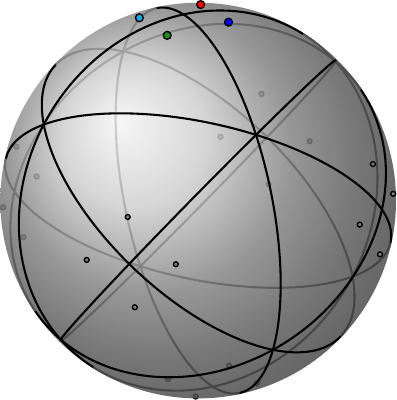}
    \caption{The spherical Voronoi diagram of the orbit.}
\end{subfigure}
\smallskip

\begin{subfigure}[c]{0.25\textwidth}
    \centering
    \includegraphics[scale=1.5]{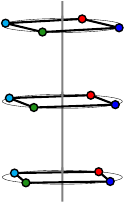}
\end{subfigure}%
\begin{subfigure}[c]{0.25\textwidth}
    \centering
    \includegraphics[scale=1.5]{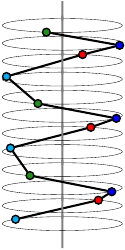}
\end{subfigure}%
\begin{subfigure}[c]{0.25\textwidth}
    \centering
    \includegraphics[scale=1.5]{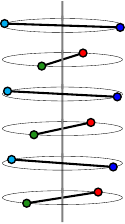}
\end{subfigure}%
\begin{subfigure}[c]{0.25\textwidth}
    \centering
    \includegraphics[scale=1.5]{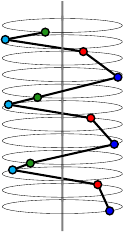}
\end{subfigure}
\iffigurescompact\else\smallskip\fi

\begin{subfigure}[c]{0.25\textwidth}
    \centering
    \includegraphics{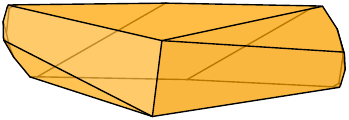}
\end{subfigure}%
\begin{subfigure}[c]{0.25\textwidth}
    \centering
    \includegraphics{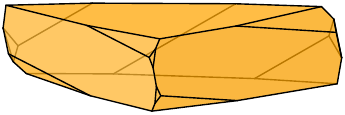}
\end{subfigure}%
\begin{subfigure}[c]{0.25\textwidth}
    \centering
    \includegraphics{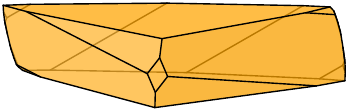}
\end{subfigure}%
\begin{subfigure}[c]{0.25\textwidth}
    \centering
    \includegraphics{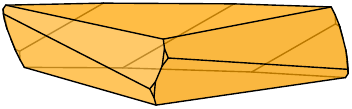}
\end{subfigure}
\iffigurescompact\else\smallskip\fi

\begin{subfigure}[c]{0.25\textwidth}
    \centering
    \includegraphics{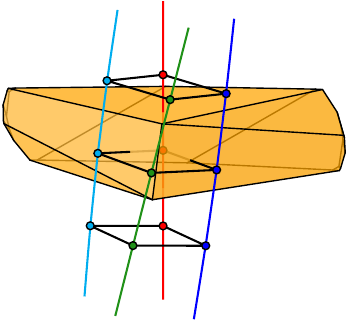}
    \caption{$\pm[O\times C_{20}]$\\ squares}
    \label{subfig:near_4fold_OxC20}
\end{subfigure}%
\begin{subfigure}[c]{0.25\textwidth}
    \centering
    \includegraphics{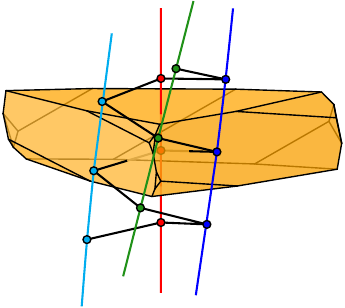}
    \caption{$\pm[O\times C_{21}]$\\ 3/4 (right) staircase}
    \label{subfig:near_4fold_OxC21}
\end{subfigure}%
\begin{subfigure}[c]{0.25\textwidth}
    \centering
    \includegraphics{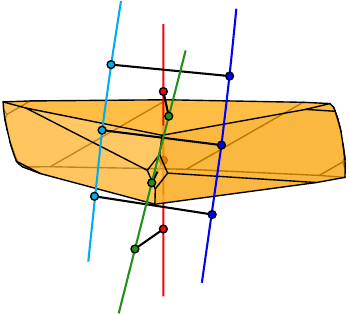}
    \caption{$\pm[O\times C_{22}]$\\ pairs}
    \label{subfig:near_4fold_OxC22}
\end{subfigure}%
\begin{subfigure}[c]{0.25\textwidth}
    \centering
    \includegraphics{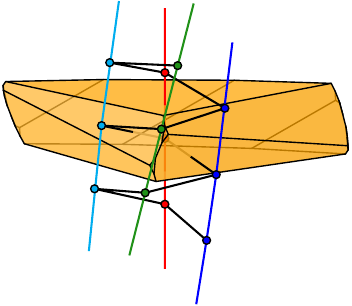}
    \caption{$\pm[O\times C_{23}]$\\ 1/4 (left) staircase}
    \label{subfig:near_4fold_OxC23}
\end{subfigure}
\caption{
  Orbits of the groups
  $G=\pm[O\times C_{n}]$ for a starting point $v$ whose image $p:=h(v)$
  lies  near a 4-fold rotation center of $G^h={+O}$.
The top row shows the three-dimensional ${+O}$-orbit polytope of $p$
and the corresponding spherical Voronoi diagram.
The four images of $p$ under the 4-fold rotation are colored.
The next row shows all possible configurations for orbit points
on the corresponding colored circles.
The vertical line in each figure is the great circle of $\H$
that correspond to the rotation center.
The third row shows a cell of the corresponding polar orbit polytope,
and
the bottom row combines the previous two rows.
}
\label{fig:near_4fold}
\end{figure}

\subsection{Starting point close to a rotation center}
\label{sec:near-axis}

Let $G$ be a cyclic-type tubical group,
and let $p$ be an $f$-fold rotation center\footnote{
We call $p$ an \emph{$f$-fold} rotation center of some 3-dimensional point group
if $f$ is the largest order of a rotation around
$p$ in that group. Hence, a 4-fold rotation center of a group
 is \emph{not} a 2-fold rotation center of that group.}
of $G^h$.
Let $[g] \in G^h$ be the clockwise rotation of $G^h$ around $p$ by $\frac{2\pi}{f}$.
That is, $g = \cos\frac\pi{f} + p\sin\frac\pi{f}$.

Choose a point $p_1 \in S^2$ close to $p$.
Since $p_1$ avoids rotation centers of $G^h$,
its images under $[g]$ are all distinct:
$$p_1,\ p_2:=[g]p_1,\ \ldots,\ p_f := [g]^{f-1}p_1$$
Figure~\ref{fig:near_4fold_orbit} and
Figure~\ref{fig:near_5fold_orbit} show these
points around a 4-fold rotation center and a 5-fold rotation center,
respectively.

We want to describe the $G$-orbit for a starting point on $K_{p_1}$.
By Proposition~\ref{prop:cyclic_type_orbit},
any point on $K_{p_1}$ will give the same $G$-orbit, up to congruence.
Thus, let $v \in K_{p_1}$ be any point on $K_{p_1}$ and consider its $G$-orbit.

We will now discuss the $G$-orbit of $v$ under different
assumptions on the subgroup $H$ of elements of $G$ that preserve $K_{p}$.

\begin{enumerate}[\textbf{Case}~1.]
\item
$H$ contains a simple rotation fixing $K_{p}$ of order $f$:
Orbit points around $K_{p}$ can be grouped into regular $f$-gons
(if $f \geq 3$) or pairs (if $f=2$).
See
Figure~\ref{subfig:near_4fold_OxC20} and
Figure~\ref{subfig:near_5fold_IxC20}.

\item 
$H$ contains no simple rotation fixing $K_{p}$:
Orbit points around $K_{p}$ form different types of staircases.
See Figures~\ref{subfig:near_4fold_OxC21} and
\ref{subfig:near_4fold_OxC23}, and
 Figures~\ref{subfig:near_5fold_IxC21}--\ref{subfig:near_5fold_IxC24}.

\item
$H$ contains a simple rotation fixing $K_{p}$ of order not equal to $f$:
This case can only occur when $f=4$ and the order of that simple rotation
is 2.
Orbit points around $K_{p}$ can be grouped into pairs.
See Figure~\ref{subfig:near_4fold_OxC22}.
\end{enumerate}

\iffigurescompact
  \begin{figure}[tb]
\else
  \begin{figure}
\fi
\begin{subfigure}{0.5\textwidth}
    \centering
    \includegraphics[scale=0.85]{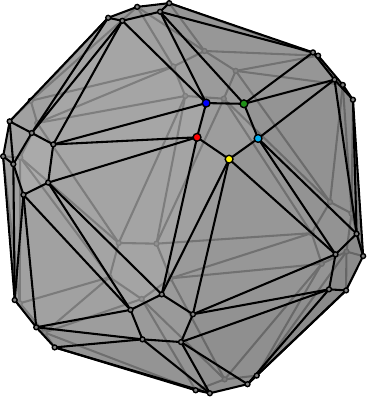}
    \caption{The ${+I}$-orbit polytope of $p$.}
    \label{fig:near_5fold_orbit}
\end{subfigure}%
\begin{subfigure}{0.5\textwidth}
    \centering
    \includegraphics[scale=0.85]{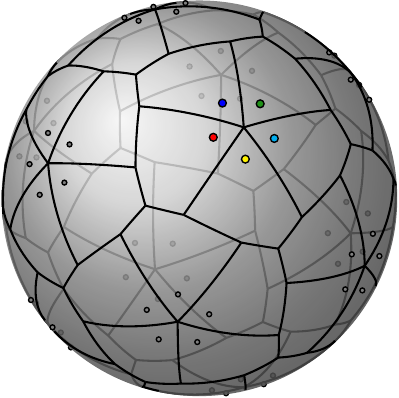}
    \caption{The spherical Voronoi diagram of the orbit.}
\end{subfigure}
\iffigurescompact\else\smallskip\fi

\begin{subfigure}[c]{0.2\textwidth}
    \centering
    \includegraphics[scale=1.4]{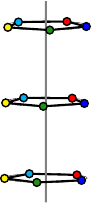}
\end{subfigure}%
\begin{subfigure}[c]{0.2\textwidth}
    \centering
    \includegraphics[scale=1.4]{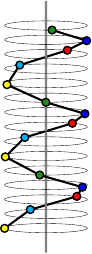}
\end{subfigure}%
\begin{subfigure}[c]{0.2\textwidth}
    \centering
    \includegraphics[scale=1.4]{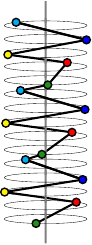}
\end{subfigure}%
\begin{subfigure}[c]{0.2\textwidth}
    \centering
    \includegraphics[scale=1.4]{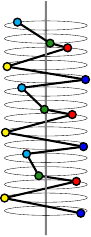}
\end{subfigure}%
\begin{subfigure}[c]{0.2\textwidth}
    \centering
    \includegraphics[scale=1.4]{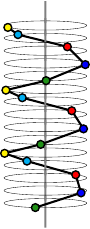}
\end{subfigure}
\iffigurescompact\vskip-1mm \else\smallskip\fi

\begin{subfigure}[c]{0.2\textwidth}
    \centering
    \includegraphics[scale=1.4]{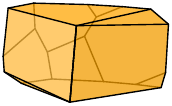}
\end{subfigure}%
\begin{subfigure}[c]{0.2\textwidth}
    \centering
    \includegraphics[scale=1.4]{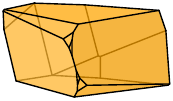}
\end{subfigure}%
\begin{subfigure}[c]{0.2\textwidth}
    \centering
    \includegraphics[scale=1.4]{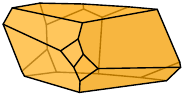}
\end{subfigure}%
\begin{subfigure}[c]{0.2\textwidth}
    \centering
    \includegraphics[scale=1.4]{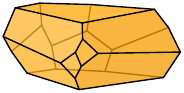}
\end{subfigure}%
\begin{subfigure}[c]{0.2\textwidth}
    \centering
    \includegraphics[scale=1.4]{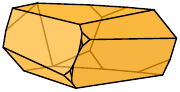}
\end{subfigure}
\iffigurescompact\vskip-1mm \else\smallskip\fi

\begin{subfigure}[c]{0.2\textwidth}
    \centering
    \includegraphics[scale=1.4]{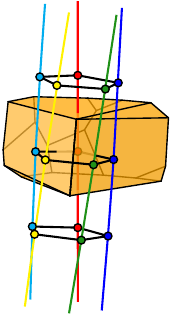}
\end{subfigure}%
\begin{subfigure}[c]{0.2\textwidth}
    \centering
    \includegraphics[scale=1.4]{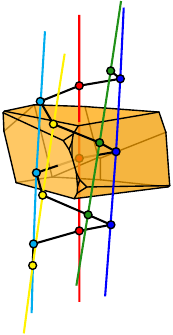}
\end{subfigure}%
\begin{subfigure}[c]{0.2\textwidth}
    \centering
    \includegraphics[scale=1.4]{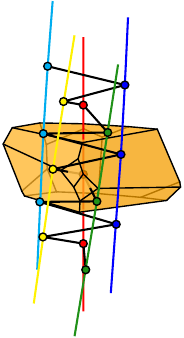}
\end{subfigure}%
\begin{subfigure}[c]{0.2\textwidth}
    \centering
    \includegraphics[scale=1.4]{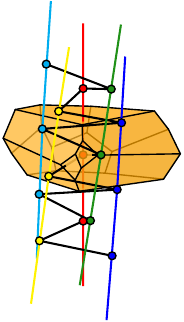}
\end{subfigure}%
\begin{subfigure}[c]{0.2\textwidth}
    \centering
    \includegraphics[scale=1.4]{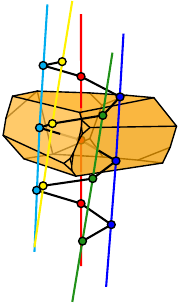}
\end{subfigure}
\iffigurescompact\vskip -7mm
\vbox to 3pt{\textcolor{white}{\hrule height 0pt depth 15pt
    width\textwidth \vss}}

\else\smallskip\fi

\begin{subfigure}[c]{0.2\textwidth}
    \centering
    \caption{$\pm[I\times C_{20}]$\\ pentagons}
    \label{subfig:near_5fold_IxC20}
\end{subfigure}%
\begin{subfigure}[c]{0.2\textwidth}
    \centering
    \caption{$\pm[I\times C_{21}]$\\ $4/5$ staircase}
    \label{subfig:near_5fold_IxC21}
\end{subfigure}%
\begin{subfigure}[c]{0.2\textwidth}
    \centering
    \caption{$\pm[I\times C_{22}]$\\ $2/5$ staircase}
    \label{subfig:near_5fold_IxC22}
\end{subfigure}%
\begin{subfigure}[c]{0.2\textwidth}
    \centering
    \caption{$\pm[I\times C_{23}]$\\ $3/5$ staircase}
    \label{subfig:near_5fold_IxC23}
\end{subfigure}%
\begin{subfigure}[c]{0.2\textwidth}
    \centering
    \caption{$\pm[I\times C_{24}]$\\ $1/5$ staircase}
    \label{subfig:near_5fold_IxC24}
\end{subfigure}
\iffigurescompact
\hrule height 0pt
\else\smallskip\fi
\caption{
Orbits of the groups $G=\pm[I\times C_{n}]$
for a starting point $v$ whose image $p:=h(v)$
lies near a 5-fold rotation center of $G^h= {+I}$.
The top row shows the three-dimensional ${+I}$-orbit polytope of $p$
and the corresponding spherical Voronoi diagram.
The five images of $p$ under the 5-fold rotation are colored.
The next row shows all possible configurations for orbit points
on the corresponding colored circles.
The vertical line in each figure is the great circle of $\H$
that correspond to the rotation center.
The third row shows a cell of the corresponding polar orbit polytope,
and the bottom row combines the previous two rows.
}
\label{fig:near_5fold}
\end{figure}

\subsection{Starting point on a rotation center}
\label{sec:on-axis}
It is also interesting to see what happens 
if we move $p_1$ to~$p$.
In this case, the points $p_1, \ldots, p_{f}$ coincide with $p$,
and thus the corresponding circles
$K_{p_1}, \ldots, K_{p_f}$ coincide with $K_{p}$.
We describe next what happens in each of the previous cases.

\begin{enumerate}[\textbf{Case}~1.]
\item
 The orbit points coincide in groups of size $f$,
and thus they form a regular $2n$-gon on $K_{p}$.
Each orbit point can be mapped to itself
by $f$ different elements of $G$.
Thus, in the polar orbit polytope,
each cell has an $f$-fold rotational symmetry
whose axis is the cell axis.

\item
Orbit points on $K_p$ form a regular $2fn$-gon.
Each orbit point can be mapped to itself by a unique element of $G$.
However, the orbit has extra symmetries,
which come from the supergroup of $G$ that we obtain by
extending $G$ by the new symmetry $[1, e_{fn}]$.
Thus, in total, each orbit point can be mapped to itself by $f$ symmetries.
Accordingly, in the polar orbit polytope,
each cell has an $f$-fold rotational symmetry
whose axis is the cell axis.

\item
Orbit points on $K_p$ form a regular $4n$-gon. 
Each orbit point can be mapped to itself
by 2 different elements of $G$.
However, the orbit has extra symmetries, which
 come from the supergroup of $G$ that we obtain by
extending $G$ by the new symmetry $[1, e_{2n}]$.
Thus, each orbit point can be mapped to itself by \emph{extra} 2 symmetries.
Accordingly, in the polar orbit polytope,
each cell has a 4-fold rotational symmetry
whose axis is the cell axis.
\end{enumerate}

See Section~\ref{subsec:examples} for particular examples
and Appendix~\ref{sec:special_starting_points} for a coverage of all groups.

\subsubsection{Supergroups of cyclic type}

The cyclic-type supergroups described in Case~2 and Case~3 are listed in
Table~\ref{tbl:tubical_supergroups} for each group class and each type of rotation
center.
For large enough $n$, this supergroup is the largest cyclic-type symmetry group
of the orbit.
In most cases, this is the same class of group with a larger parameter~$n$.
The only exception are the groups $G=\pm[T\times C_n]$ when $p$ is a
2-fold rotation center of $G^h=+T$.
As can be seen in
Table~\ref{tbl:tubical_supergroups},
the symmetry groups of cyclic type of the orbit are then of the form
$\pm[O\times C_{n'}]$ or
$\pm\frac12[O\times C_{n'}]$.

The reason for this exceptional behavior can already be seen at the level of the groups $G^h$ in
three dimensions:
On $S^2$, the group ${+T}$ is an index-2 subgroup of ${+O}$.
The 2-fold rotation centers $p$
of ${+T}$ coincide with the 4-fold rotation centers of ${+O}$, and 
the orbit has size 6 in both cases.

The group $G_1:= \pm[T\times C_n]$ is an index-2 subgroup of $G_2: = \pm[O\times C_n]$.
One can show that when $n \equiv 0\bmod 4$, the orbits of both
groups %
have a simple rotation fixing
$K_{p}$ of order 2 (for $G_1$) and
of order 4 (for $G_2$).
In particular, both orbits follow Case~1 above and they form a regular
$2n$-gon on each orbit circle.
Since they also have the same orbit circles, these two orbits coincide.
The other cases
($n \equiv 2\bmod 4$, and $n$ odd) are similar.

Accordingly, all cells of the groups $\pm[T\times C_n]$
when $p$ is a 2-fold rotation center (Section~\ref{T-2fold}),
appear also as cells of the groups
$\pm\frac12[O\times C_{n'}]$
when $p$ is a 4-fold rotation center (Figure~\ref{fig:hOxC2n_4fold}),
and those when $n$ is a multiple of~$4$ also appear for the groups
$\pm[O\times C_{n'}]$ (Section~\ref{O-4fold}).

It is perhaps instructive to look at a particular example and compare the groups
 $\pm[T\times C_{24}]$ (Figure~\ref{fig:TxCn_2fold}) and
 $\pm\frac12[O\times C_{24}]$
 (Figure~\ref{fig:hOxC2n_4fold} for $n=12$), which have equal, 4-sided cells.
 The allowed rotations between consecutive cells, apart from the necessary
 adjustment of $\pi/24$,
 are $0^\circ$ and  $180^\circ$
  in the first case
  and  $\pm 90^\circ$ in the second case.
  The common supergroup that has all four rotations is
 $\pm[O\times C_{24}]$
 (Figure~\ref{fig:OxCn_4fold}).

 \subsubsection{Supergroups of dihedral type, and flip symmetries}
 \label{sec:flip}
 For each cyclic-type tubical group and
for each rotation center $p$ of its induced group on $S^2$,
there is a dihedral-type tubical group
whose induced group on $S^2$ has a mirror through $p$,
and the cyclic-type group is an index-2 subgroup of the dihedral-type group.
Thus, by Proposition~\ref{prop:cylic_on_mirror},
the orbit of the cyclic-type group for a starting point on $K_p$
has extra symmetries coming from (a geometrically equal copy of)
that dihedral-type tubical group.
In particular, each cell of the polar orbit polytope
will have a flip symmetry.
See the figures in Section~\ref{subsec:examples}
and Appendix~\ref{sec:special_starting_points}.
The dihedral-type supergroups are listed in Table~\ref{tbl:tubical_supergroups}.
\begin{table}
\centering
\setlength{\extrarowheight}{1.41pt}
\begin{tabular}{|c|c|c|c|c|c|c|}
\hline
center & \multirow{2}{*}{\#tubes}
& \multirow{2}{*}{$n$} & orbit &
cyclic-type& dihedral-type & \multirow{2}{*}{figure}\\[-1.4pt]
type &&& size& supergroup & supergroup & \\
\hline
\multicolumn7{|c|} {$\raise3pt\strut\lower2pt\strut \pm[I\times C_n]$} \\
\hline
\multirow{2}{*}{5-fold} &\multirow{2}{*}{12} & $0 \bmod 5$ & $24n$ & -- & $\pm[I\times D_{2n}]$ & \multirow2*{\ref{fig:IxCn_5fold}}\\
&& else & $120n$ & $\pm[I \times C_{5n}]$ & $\pm[I\times D_{10n}]$ &\\
\hline
\multirow{2}{*}{3-fold} & \multirow{2}{*}{20} & $0 \bmod 3$ & $40n$ & -- & $\pm[I\times D_{2n}]$ & \multirow2*{\ref{fig:IxCn_3fold}} \\
&& else & $120n$ & $\pm[I \times C_{3n}]$ & $\pm[I\times D_{6n}]$ &\\
\hline
\multirow{2}{*}{2-fold} & \multirow{2}{*}{30} & $0 \bmod 2$ & $60n$ & -- & $\pm[I\times D_{2n}]$ & \multirow2*{\ref{fig:IxCn_2fold}} \\
& & else & $120n$ & $\pm[I \times C_{2n}]$ & $\pm[I \times D_{4n}]$ &\\
\hline

\multicolumn7{|c|} {$\raise3pt\strut\lower2pt\strut \pm[O\times C_n]$} \\
\hline
\multirow{3}{*}{4-fold} &\multirow{3}{*}{6}& $0 \bmod 4$ & $12n$ & -- & $\pm[O\times D_{2n}]$ & \multirow3*{\ref{fig:OxCn_4fold}} \\
&& $2 \bmod 4$ & $24n$ & $\pm[O\times C_{2n}]$ & $\pm[O\times D_{4n}]$& \\
&& else & $48n$ & $\pm[O \times C_{4n}]$ & $\pm[O\times D_{8n}]$& \\
\hline
\multirow{2}{*}{3-fold} &\multirow{2}{*}{8} & $0 \bmod 3$ & $16n$ & -- & $\pm[O\times D_{2n}]$ & \multirow2*{\ref{fig:OxCn_3fold}}\\
&& else & $48n$ & $\pm[O \times C_{3n}]$ & $\pm[O\times D_{6n}]$& \\
\hline
\multirow{2}{*}{2-fold} &\multirow{2}{*}{12} & $0 \bmod 2$ & $24n$ & -- & $\pm[O\times D_{2n}]$ & \multirow2*{\ref{fig:OxCn_2fold}}\\
&& else & $48n$ & $\pm[O \times C_{2n}]$ & $\pm[O\times D_{4n}]$& \\
\hline

\multicolumn7{|c|} {$\raise3pt\strut\lower2pt\strut \pm\frac12[O\times C_{2n}]$} \\
\hline
\multirow{3}{*}{4-fold} & \multirow{3}{*}{6} & $2 \bmod 4$ & $12n$ & -- & $\pm\frac12[O\times \overline{D}_{4n}]$ & \multirow3*{\ref{fig:hOxC2n_4fold}} \\
&& $0 \bmod 4$ & $24n$ & $\pm[O\times C_{2n}]$ & $\pm[O\times D_{4n}]$& \\
&& else & $48n$ & $\pm[O\times C_{4n}]$ & $\pm[O\times D_{8n}]$& \\
\hline
\multirow{2}{*}{3-fold} &\multirow{2}{*}{8} & $0 \bmod 3$ & $16n$ & -- & $\pm\frac12[O\times \overline{D}_{4n}]$ & \multirow2*{\ref{fig:hOxC2n_3fold}}\\
&& else & $48n$ & $\pm\frac12[O \times C_{6n}]$ & $\pm\frac12[O\times \overline{D}_{12n}]$& \\
\hline
\multirow{2}{*}{2-fold} & \multirow{2}{*}{12} & $0 \bmod 2$ & $24n$ & -- & $\pm\frac12[O\times \overline{D}_{4n}]$ & \multirow2*{\ref{fig:hOxC2n_2fold}}\\
&& else & $48n$ & $\pm\frac12[O \times C_{4n}]$ & $\pm\frac12[O\times \overline{D}_{8n}]$& \\
\hline

\multicolumn7{|c|} {$\raise3pt\strut\lower2pt\strut \pm[T\times C_n]$} \\
\hline
\multirow{2}{*}{3-fold} &\multirow{2}{*}{4} & $0 \bmod 3$ & $8n$ & -- & $\pm\frac12[O\times D_{2n}]$ & \multirow2*{\ref{fig:TxCn_3fold}}\\
&& else & $24n$ & $\pm[T \times C_{3n}]$ & $\pm\frac12[O\times D_{6n}]$& \\
\hline
\multirow{3}{*}{2-fold} & \multirow{3}{*}{6} & $0 \bmod 4$ & $12n$ & $\pm[O\times C_n]$ & $\pm[O\times D_{2n}]$ & \multirow3*{\ref{fig:TxCn_2fold}}\\
&& $2 \bmod 4$ & $12n$ & $\pm\frac12[O\times C_{2n}]$ & $\pm\frac12[O\times \overline{D}_{4n}]$ &\\
&& else & $24n$ & $\pm\frac12[O \times C_{4n}]$ & $\pm\frac12[O\times \overline{D}_{8n}]$& \\
\hline

\multicolumn7{|c|} {$\raise3pt\strut\lower2pt\strut \pm\frac13[T\times C_{3n}]$} \\
\hline
\multirow{2}{*}{3-fold I} & \multirow{2}{*}{4} & $1 \bmod 3$ & $8n$ & -- & $\pm\frac16[O\times D_{6n}]$ & \multirow2*{\ref{fig:tTxC3n_3pfold}} \\
&& else & $24n$ & $\pm[T \times C_{3n}]$ & $\pm\frac12[O\times D_{6n}]$& \\
\hline
\multirow{2}{*}{3-fold II} &\multirow{2}{*}{4} & $2 \bmod 3$ & $8n$ & -- & $\pm\frac16[O\times D_{6n}]$ & \multirow2*{\ref{fig:tTxC3n_3fold}}\\
&& else & $24n$ & $\pm[T \times C_{3n}]$ & $\pm\frac12[O\times D_{6n}]$& \\
\hline
\multirow{2}{*}{2-fold} & \multirow{2}{*}{6}& $0 \bmod 2$ & $12n$ & -- & $\pm\frac16[O\times D_{6n}]$ & \multirow2*{\ref{fig:tTxC3n_2fold}}\\
&& else & $24n$ & $\pm\frac13[T \times C_{6n}]$ & $\pm\frac16[O\times D_{12n}]$& \\
\hline
\end{tabular}
\caption
[Relations among tubical groups]
{The columns ``cyclic-type supergroup'' and ``dihedral-type
supergroup'' indicate the largest symmetry group of the orbit that
is tubical of that type.
In Section~\ref{subsec:examples}, we extensively discuss two cases
from the table. %
For the other cases, we summarize the results in 
Appendix~\ref{sec:special_starting_points}.
The last column %
refers to the figure that
shows cells of the corresponding polar orbit polytope
with different values for $n$.
The two types of 3-fold rotation centers
for $\pm\frac13[T\times C_{3n}]$
(3-fold~I and 3-fold~II)
are defined in Section~\ref{sec:same-symmetry}.
}
\label{tbl:tubical_supergroups}
\end{table}

\subsection{Two examples of special starting points}
\label{subsec:examples}

In this section we will discuss two cases of non-generic starting points.
In particular, we want to consider orbits of cyclic-type tubical groups
where the image of the starting point under $h$ is a rotation center of
the induced group.
In Table~\ref{tbl:tubical_supergroups} and
Appendix~\ref{sec:special_starting_points}, we summarize the results
for the remaining groups and rotation centers.

\subsubsection{\texorpdfstring{$\pm[I\times C_n]$}{+-[IxCn]}, 5-fold rotation center}

Let $G = \pm[I\times C_n]$.
We want to consider the $G$-orbit of a point whose image under $h$
is a 5-fold rotation center $p$ of ${+I}$.
By Proposition~\ref{prop:cyclic_type_orbit},
any starting point on $K_p$ will give the same orbit, up to congruence.
Notice also that the other orbit circles correspond to
the other 5-fold rotation centers of ${+I}$.
Thus, choosing $p$ to be an arbitrary 5-fold rotation center will yield 
the same orbit, up to congruence.

So let $p$ be the 5-fold rotation center
$p = \frac1{\sqrt{\varphi^2+1}}(0, 1, \varphi)$,
where $\varphi=\frac{1+\sqrt5}{2}$.
Then
$g = -\w i_I = \cos\frac\pi5 + p \sin\frac\pi5 \in 2I$
defines the $72^\circ$ clockwise rotation $[g]\in {+I}$ around $p$.
By Proposition~\ref{prop:action_on_its_circle},
we know the elements of $G$ that preserve $K_p$.
These elements form a subgroup
$H = \langle [g, 1], [1, e_n] \rangle$
of order $10n$.
Proposition~\ref{prop:action_on_its_circle} also
tells us the $H$ acts on $K_p$ as a 2-dimensional cyclic group.

The rotation $[g, 1]$ rotates $\vec K_p$ by $-\frac\pi5$,
while $[1, e_n]$ rotates it by $\frac\pi{n}$.
Thus, the $G$-orbit of a point on $K_p$
forms a regular $\lcm(2n, 10)$-gon on $K_p$.
We will discuss the orbit of a point $v \in K_p$ depending on the value of $n$.
 Figure~\ref{fig:IxCn_5fold}
shows cells of the polar orbit polytopes for different values of $n$.

\begin{enumerate}
\item[$\bullet$]
If $n$ is a multiple of 5,
then the orbit points form a regular $2n$-gon
on each orbit circle.
So, every orbit point can be mapped to itself by 5 different elements of $G$.
This is reflected on the cells of the polar orbit polytope where each cell
has a 5-fold rotational symmetry whose axis is the cell axis.

This case corresponds to Case~1 in Section~\ref{sec:on-axis},
where $H$ contains a simple rotation of order~$5$ fixing $K_p$.

The element $[1, e_n]$ of $G$ maps an orbit point
to an adjacent one on the same circle.
Correspondingly, on each tube,
the cells of the polar orbit polytope are stacked upon each other
with a right screw by~$\frac\pi{n}$.

\item[$\bullet$]
If $n$ is not a multiple of 5,
then the orbit points form a regular $10n$-gon
on each orbit circle.
That is, the orbit is free.
So, every orbit point can be mapped to itself by a unique element of $G$.
However, this orbit has extra symmetries. %
In particular, the rotation $[1, e_{5n}]$ maps each orbit point
to an adjacent one on the same circle.
Adjoining $[1, e_{5n}]$ to $G$ gives the supergroup $\pm[I\times C_{5n}]$,
whose orbit of $n$ follows the first case.
Accordingly, each cell of the polar orbit polytope has a 5-fold symmetry
whose axis is the cell axis.

This case corresponds to Case~2 in Section~\ref{sec:on-axis},
where $H$ does not contain any simple rotation fixing $K_p$.

The symmetry $[1, e_{5n}]$ (which is not in $G$) maps an orbit point
to an adjacent one on the same circle.
Correspondingly, on each tube,
the cells of the polar orbit polytope are stacked upon each other
with a right screw by $\frac\pi{5n}$.
\end{enumerate}

In accordance with Section~\ref{sec:flip},
every cell has a flip symmetry,
which is not included in~$G$.
It comes from (a group geometrically equal to)
the group $\pm[I\times D_{2n}]$,
which contains $G$ as an index-2 subgroup.

The top and bottom faces in each cell are congruent.
They resemble the shape of a pentagon.
This corresponds to the fact that the spherical
Voronoi cell of the ${+I}$-orbit of $p$
on the 2-sphere is a spherical regular pentagon, as shown in
 the top right picture of Figure~\ref{fig:IxCn_5fold}.
(Refer to the discussion in Section~\ref{sec:geometry-tubes}.)

Since the ${+I}$-orbit of $p$ has size 12,
the $G$-orbit of $v$ lies on 12 orbit circles.
Accordingly, the cells of the polar orbit polytope can
be decomposed into 12 tubes, each with $\lcm(2n, 10)$ cells.
In the PDF-file of this article, the interested reader can click on the pictures in
Figure~\ref{fig:IxCn_5fold} for an interactive visualization of these tubes.
We refer to Section~\ref{sec:gallery} for more details.

In accordance with the program
set out in
Figure~\ref{fig:geometric-understanding} in
Section~\ref{sec:orbit-polytopes} to understand the group
by its action on the orbit polytope, we will now work out how each cell
is mapped to the adjacent cell in the same tube.
This requires a small number-theoretic calculation.
The mapping between adjacent cells is obtained in cooperation between
the right group and the left group.
In particular,
to get a rotation by $\tfrac {2\pi}{\lcm(2n, 10)}$
along the orbit circle $\vec K_p$, we have to combine
a left rotation by $-a\cdot\frac{\pi}5$
with a right rotation by $b\cdot\frac{\pi}{n}$, resulting in the angle
\begin{equation}
    \label{eq:b-a}
    \frac{b\pi}{n}
    -
    \frac{a\pi}5
    = \frac {2\pi}{\lcm(2n, 10)}.
\end{equation}
For example, for $n=12$ we can solve this by $a=2, b=5$.
The right screw angle between consecutive slices (or orbit points)
is then
$\frac{b\pi}{n} + \frac{a\pi}5$.
Using \eqref{eq:b-a}, this can be rewritten as
\begin{equation}
\frac{a\pi}5 + \frac{b\pi}{n}
=
\frac{2a\pi}5 + \frac{2\pi}{\lcm(2n, 10)}
=
\left(\frac{a}5 + \frac{1}{\lcm(2n, 10)}\right) \cdot 2\pi,
\end{equation}
which is
$(\frac{2}5+\frac \pi{120})\cdot 2\pi$ in our example.
This angle is always of the form
$(\frac{a}5 + \frac{1}{\lcm(2n, 10)}) \cdot 2\pi$ for some integer $a$,
in accordance
with the requirement to match the pentagonal shape.
The value $a$ can never be 0.
The rotation angles for different values of $n$ are listed in
Figure~\ref{fig:IxCn_5fold}.

When $n$ is not a multiple of 5, there is one element of the group
that maps a cell to the upper adjacent one. Thus, $a$ has a unique value.
When $n$ is a multiple of 5, each cell has a 5-fold symmetry included in 
the group. Thus, all values of $a$ are permissible.

\begin{figure}[htp]
\begin{minipage}{0.5\textwidth}
\centering
\includegraphics[scale=1]{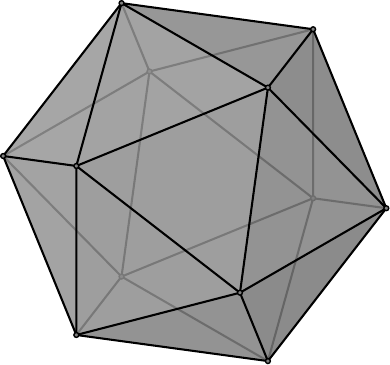}
\end{minipage}%
\begin{minipage}{0.5\textwidth}
\centering
\includegraphics[scale=1]{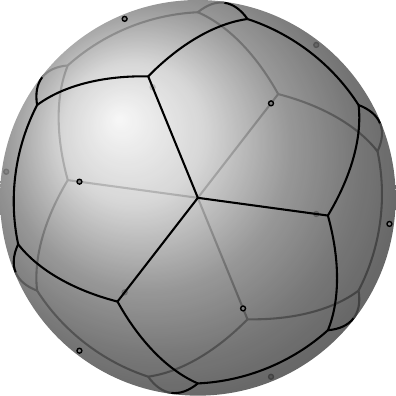}
\end{minipage}
\vskip 5pt plus 0.5fill
\begingroup
\setlength{\tabcolsep}{0pt}
\noindent\begin{tabular}{
>{\centering\arraybackslash}m{0.33\textwidth}
>{\centering\arraybackslash}m{0.33\textwidth}
>{\centering\arraybackslash}m{0.33\textwidth}}
\href{https://www.inf.fu-berlin.de/inst/ag-ti/software/DiscreteHopfFibration/gallery.html?f=IxCn/5/120cells_12tubes}{\includegraphics[scale=1.1]{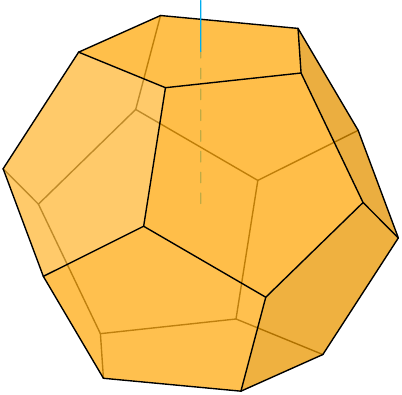}}
&
\href{https://www.inf.fu-berlin.de/inst/ag-ti/software/DiscreteHopfFibration/gallery.html?f=IxCn/5/240cells_12tubes}{\includegraphics[scale=1.1]{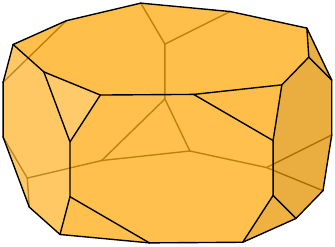}}
&
\href{https://www.inf.fu-berlin.de/inst/ag-ti/software/DiscreteHopfFibration/gallery.html?f=IxCn/5/360cells_12tubes}{\includegraphics[scale=1.1]{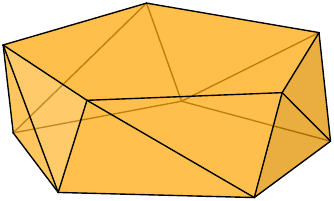}}
\\[1.5mm]
$n = 1, 5$\break
$(\frac{4}{5} + \frac{1}{10})\cdot 2\pi$
&$n = 2, 10$\break
$(\frac{2}{5} + \frac{1}{20})\cdot 2\pi$
&$n = 3, 15$\break
$(\frac{3}{5} + \frac{1}{30})\cdot 2\pi$
\end{tabular}\nobreak

\vfill\nobreak
\noindent\begin{tabular}{
>{\centering\arraybackslash}m{0.33\textwidth}
>{\centering\arraybackslash}m{0.33\textwidth}
>{\centering\arraybackslash}m{0.33\textwidth}}
\href{https://www.inf.fu-berlin.de/inst/ag-ti/software/DiscreteHopfFibration/gallery.html?f=IxCn/5/480cells_12tubes}{\includegraphics[scale=1.1]{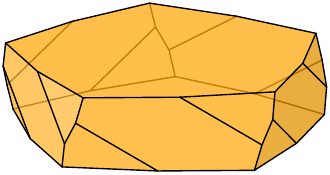}}
&
\href{https://www.inf.fu-berlin.de/inst/ag-ti/software/DiscreteHopfFibration/gallery.html?f=IxCn/5/600cells_12tubes}{\includegraphics[scale=1.1]{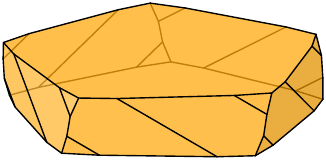}}
&
\href{https://www.inf.fu-berlin.de/inst/ag-ti/software/DiscreteHopfFibration/gallery.html?f=IxCn/5/720cells_12tubes}{\includegraphics[scale=1.1]{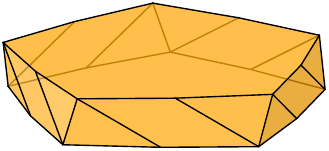}}
\\[1.5mm]
$n = 4, 20$\break
$(\frac{1}{5} + \frac{1}{40})\cdot 2\pi$
&$n = 25$\break
$(\frac{k}{5} + \frac{1}{50})\cdot 2\pi$
&$n = 6, 30$\break
$(\frac{4}{5} + \frac{1}{60})\cdot 2\pi$
\end{tabular}\nobreak

\vfill\nobreak
\noindent\begin{tabular}{
>{\centering\arraybackslash}m{0.33\textwidth}
>{\centering\arraybackslash}m{0.33\textwidth}
>{\centering\arraybackslash}m{0.33\textwidth}}
\href{https://www.inf.fu-berlin.de/inst/ag-ti/software/DiscreteHopfFibration/gallery.html?f=IxCn/5/840cells_12tubes}{\includegraphics[scale=1.1]{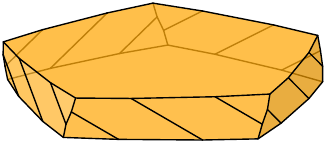}}
&
\href{https://www.inf.fu-berlin.de/inst/ag-ti/software/DiscreteHopfFibration/gallery.html?f=IxCn/5/960cells_12tubes}{\includegraphics[scale=1.1]{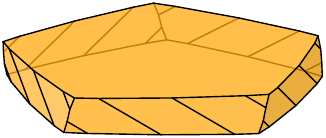}}
&
\href{https://www.inf.fu-berlin.de/inst/ag-ti/software/DiscreteHopfFibration/gallery.html?f=IxCn/5/1080cells_12tubes}{\includegraphics[scale=1.1]{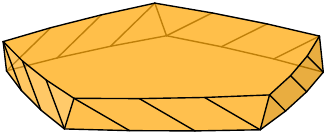}}
\\[1.5mm]
$n = 7, 35$\break
$(\frac{2}{5} + \frac{1}{70})\cdot 2\pi$
&$n = 8, 40$\break
$(\frac{3}{5} + \frac{1}{80})\cdot 2\pi$
&$n = 9, 45$\break
$(\frac{1}{5} + \frac{1}{90})\cdot 2\pi$
\end{tabular}\nobreak

\vfill\nobreak
\noindent\begin{tabular}{
>{\centering\arraybackslash}m{0.33\textwidth}
>{\centering\arraybackslash}m{0.33\textwidth}
>{\centering\arraybackslash}m{0.33\textwidth}}
\href{https://www.inf.fu-berlin.de/inst/ag-ti/software/DiscreteHopfFibration/gallery.html?f=IxCn/5/1200cells_12tubes}{\includegraphics[scale=1.1]{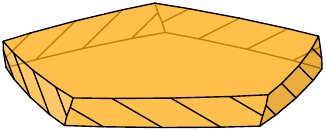}}
&
\href{https://www.inf.fu-berlin.de/inst/ag-ti/software/DiscreteHopfFibration/gallery.html?f=IxCn/5/1320cells_12tubes}{\includegraphics[scale=1.1]{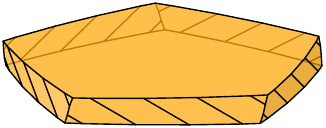}}
&
\href{https://www.inf.fu-berlin.de/inst/ag-ti/software/DiscreteHopfFibration/gallery.html?f=IxCn/5/1440cells_12tubes}{\includegraphics[scale=1.1]{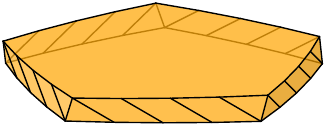}}
\\[1.5mm]
$n = 50$\break
$(\frac{k}{5} + \frac{1}{100})\cdot 2\pi$
&$n = 11, 55$\break
$(\frac{4}{5} + \frac{1}{110})\cdot 2\pi$
&$n = 12, 60$\break
$(\frac{2}{5} + \frac{1}{120})\cdot 2\pi$
\end{tabular}\nobreak

\vfill\nobreak
\noindent\begin{tabular}{
>{\centering\arraybackslash}m{0.33\textwidth}
>{\centering\arraybackslash}m{0.33\textwidth}
>{\centering\arraybackslash}m{0.33\textwidth}}
\href{https://www.inf.fu-berlin.de/inst/ag-ti/software/DiscreteHopfFibration/gallery.html?f=IxCn/5/1560cells_12tubes}{\includegraphics[scale=1.1]{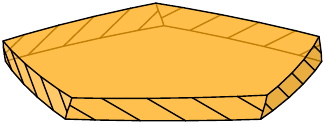}}
&
\href{https://www.inf.fu-berlin.de/inst/ag-ti/software/DiscreteHopfFibration/gallery.html?f=IxCn/5/1680cells_12tubes}{\includegraphics[scale=1.1]{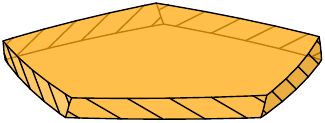}}
&
\href{https://www.inf.fu-berlin.de/inst/ag-ti/software/DiscreteHopfFibration/gallery.html?f=IxCn/5/1800cells_12tubes}{\includegraphics[scale=1.1]{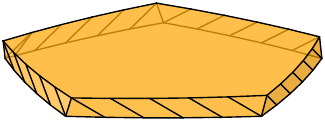}}
\\[1.5mm]
$n = 13, 65$\break
$(\frac{3}{5} + \frac{1}{130})\cdot 2\pi$
&$n = 14, 70$\break
$(\frac{1}{5} + \frac{1}{140})\cdot 2\pi$
&$n = 75$\break
$(\frac{k}{5} + \frac{1}{150})\cdot 2\pi$
\end{tabular}\nobreak

\begin{figure}[H]
\caption{
The ${+I}$-orbit polytope of the 5-fold rotation center
$p=(1/\sqrt{\phi^2+1})(0, 1, \phi)$ of ${+I}$,
where $\phi=(1+\sqrt5)/2$ (top left),
and the spherical Voronoi diagram of that orbit (top right).
The remaining pictures show cells of polar
$\pm[I\times C_n]$-orbit polytopes
for a starting point on $K_p$
for different values of $n$. %
In addition we indicate
the counterclockwise angle (as seen from the top) by which the
group rotates as it proceeds from a cell to the consecutive cell above.
When the same orbit arises for several values of $n$, then the indicated
angle is the unique valid angle only for the smallest value $n_0$ that
is specified. For a larger value $n=5n_0$, this can be combined with
arbitrary multiples of a $5$-fold rotation.
The polar orbit polytope
can be decomposed into 12 tubes,
each with $\lcm(2n, 10)$ cells.
The blue vertical line
indicates the cell axis, the direction towards the next cell along~$K_p$.
For an appropriate choice of starting point on $K_p$,
the group $\pm[I\times D_{2n}]$
produces the same orbit.
}

\label{fig:IxCn_5fold}
\end{figure}\endgroup
\end{figure}

\subsubsection{\texorpdfstring{$\pm\frac12[O\times C_{2n}]$}{+-1/2[OxC2n]}, 4-fold rotation center}
\label{ssubsec:hOxC2n}

Let $G = \pm\frac12[O\times C_{2n}]$.
We want to consider the $G$-orbit of a point whose image under $h$
is a 4-fold rotation center $p$ of ${+O}$.
The discussion will closely parallel that of the group from the previous section,
but in connection with the 4-fold rotation, we will also meet Case~3.
Any of the 4-fold rotation centers $p$ gives the same orbit.
So let $p$ be the 4-fold rotation center
$p = (0, 1, 0)$.
Then
$g = -\w i_O = \cos\frac\pi4 + p \sin\frac\pi4 \in 2O$
defines the $90^\circ$ clockwise rotation $[g]\in {+O}$ around~$p$.
By Proposition~\ref{prop:action_on_its_circle},
we determine the elements of $G$ that preserve $K_p$ as
the subgroup
$H = \langle [g, e_{2n}], [1, e_n] \rangle$
of order $8n$, which
acts on $K_p$ as a 2-dimensional cyclic group.
The rotation $[g, e_{2n}]$ rotates $\vec K_p$ by
$-\frac\pi4 + \frac\pi{2n} = -\frac{(n-2)\pi}{4n}$.
Its order is
\begin{displaymath}
    \frac{2\pi}{\gcd(\frac{(n-2)\pi}{4n}, 2\pi)} = 
    \frac{2\pi}{\frac{\pi}{4n}\gcd(n-2, 8n)} = 
    \frac{8n}{\gcd(n-2, 8n-8(n-2))} = 
    \frac{8n}{\gcd(n-2, 16)}.
\end{displaymath}
The other operation, $[1, e_n]$ rotates it by $\frac\pi{n}$.
Thus, the $G$-orbit of a point on $K_p$
forms a regular polygon with
$\lcm(2n, \frac{8n}{\gcd(n-2, 16)})$ sides on $K_p$.
The denominator $\gcd(n-2, 16)$ can take the values $1,2,4,8,16$, but
in the overall expression, the values $4,8,16$ make no distinction,
and thus we can simplify the expression for the number of sides
to $\frac{8n}{\gcd(n-2, 4)}$.

The structure of the orbit of a point $v \in K_p$ depends on $n$.
Cells of the polar orbit polytopes for different values of $n$ are
shown in Figure~\ref{fig:hOxC2n_4fold}.

\begin{enumerate}
\item[$\bullet$]
If $n - 2$ is a multiple of $4$,
then $\gcd(n-2, 4) = 4$ and $\frac{8n}{\gcd(n-2, 4)} = 2n$.
The orbit points form a regular $2n$-gon on each orbit circle,
and every point can be mapped to itself by 4 different elements of $G$.
This is reflected on the polar orbit polytope
where each cell has a 4-fold symmetry
whose axis is the cell axis.

This corresponds to Case~1 in Section~\ref{sec:on-axis},
where $H$ contains a simple rotation of order $4$ fixing $K_p$.

The element $[1, e_n]$ of $G$ maps an orbit point
to an adjacent one on the same circle.
Correspondingly, on each tube,
the cells of the polar orbit polytope are stacked upon each other
with a right screw by $\frac\pi{2n}$.

\item[$\bullet$]
If $n - 2 \equiv 2 \bmod 4$,
then $\gcd(n-2, 4) = 2$ and $\frac{8n}{\gcd(n-2, 4)} = 4n$.
The orbit points form a regular $4n$-gon on each orbit circle,
and every point can be mapped to itself by 2 different elements of $G$.
However, this orbit has extra symmetries. %
In particular, the rotation $[1, e_{2n}]$ maps each orbit point
to an adjacent one on the same circle.
Adjoining $[1, e_{2n}]$ to $G$ gives the supergroup $\pm[O\times C_{2n}]$,
which contains $G$ as an index-2 subgroup.
Thus, each orbit point can be mapped to itself by 2 extra symmetries
that are not in~$G$.
Accordingly, as in the first case, every cell of the polar orbit polytope has a 4-fold symmetry
whose axis is the cell axis.

This corresponds to Case~3 in Section~\ref{sec:on-axis},
where $H$ contains a simple rotation of order $2$ fixing $K_p$.

The symmetry $[1, e_{2n}]$ (which is not in $G$)
maps an orbit point to adjacent one on the same circle.
Correspondingly, on each tube,
the cells of the polar orbit polytope are stacked upon each other
with a right screw by $\frac\pi{2n}$.

\item[$\bullet$]
If $n - 2$ is odd,
then $\gcd(n-2, 4) = 1$ and $\frac{8n}{\gcd(n-2, 4)} = 8n$.
The orbit is free.
The orbit forms a regular $8n$-gon on each orbit circle.
Every point can be mapped to any other point by a unique element of $G$.
Again, the orbit has extra symmetries. %
In particular, the rotation $[1, e_{4n}]$ maps each orbit point
to an adjacent one on the same circle.
Adjoining $[1, e_{4n}]$ to $G$ gives the supergroup $\pm[O\times C_{4n}]$,
which contains $G$ as an index-4 subgroup.
Thus, each orbit point can be mapped to itself by 4 symmetries.
Accordingly, as in the other cases, every cell of the polar orbit polytope has a 4-fold symmetry
whose axis is the cell axis.

This corresponds to Case~2 in Section~\ref{sec:on-axis},
where $H$ does not contain a simple rotation fixing $K_p$.

The symmetry $[1, e_{4n}]$ (which is not in $G$)
maps an orbit point to the next one on the same circle.
Correspondingly, on each tube,
the cells of the polar orbit polytope are stacked upon each other
with a right screw by $\frac\pi{4n}$.
\end{enumerate}

In accordance with Section~\ref{sec:flip},
every cell has a flip symmetry,
which is not included in $G$.
It comes from (a group geometrically equal to)
the group $\pm\frac12[O\times \overline{D}_{4n}]$,
which contains $G$ as an index-2 subgroup.

The top and bottom faces in each cell are congruent.
They resemble the shape of a rounded square, in agreement
with %
the quadrilateral %
Voronoi cell %
on the 2-sphere, as shown in
the top right figure in Figure~\ref{fig:hOxC2n_4fold}.

Since the ${+O}$-orbit of $p$ has size 6,
the $G$-orbit of $v$ lies on 6 orbit circles.
Accordingly, the cells of the polar orbit polytope can
be decomposed into 6 tubes, each with $\frac{8n}{\gcd(n-2, 4)}$ cells.

Similar to the previous section,
one can work out the right screw angle (in $G$)
between consecutive slices.
To summarize:
When $n - 2$ is odd, there is a unique angle of the form:
$(\frac{k_0}4 + \frac{1}{8n})\cdot 2\pi$ (with specific $k_0=1$, $2$, or~$3$).
When $n - 2 \equiv 2 \bmod 4$,
there are two angles:
$(\frac{2k+1}4 + \frac{1}{4n})\cdot 2\pi$ (with arbitrary $k$).
When $n - 2$ is a multiple of 4,
there are four angles:
$(\frac{k}4 + \frac{1}{2n})\cdot 2\pi$ (with arbitrary $k$).

\begin{figure}[htp]
\begin{minipage}{0.5\textwidth}
\centering
\includegraphics[scale=1]{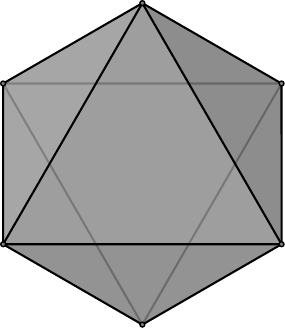}
\end{minipage}%
\begin{minipage}{0.5\textwidth}
\centering
\includegraphics[scale=1]{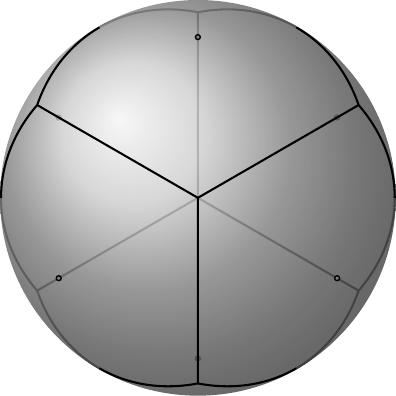}
\end{minipage}
\vskip 5pt plus 0.5fill
\begingroup
\setlength{\tabcolsep}{0pt}
\noindent\begin{tabular}{
>{\centering\arraybackslash}m{0.33\textwidth}
>{\centering\arraybackslash}m{0.33\textwidth}
>{\centering\arraybackslash}m{0.33\textwidth}}
\href{https://www.inf.fu-berlin.de/inst/ag-ti/software/DiscreteHopfFibration/gallery.html?f=hOxC2n/4/24cells_6tubes}{\vbox{\vskip-3mm \includegraphics[scale=0.5]{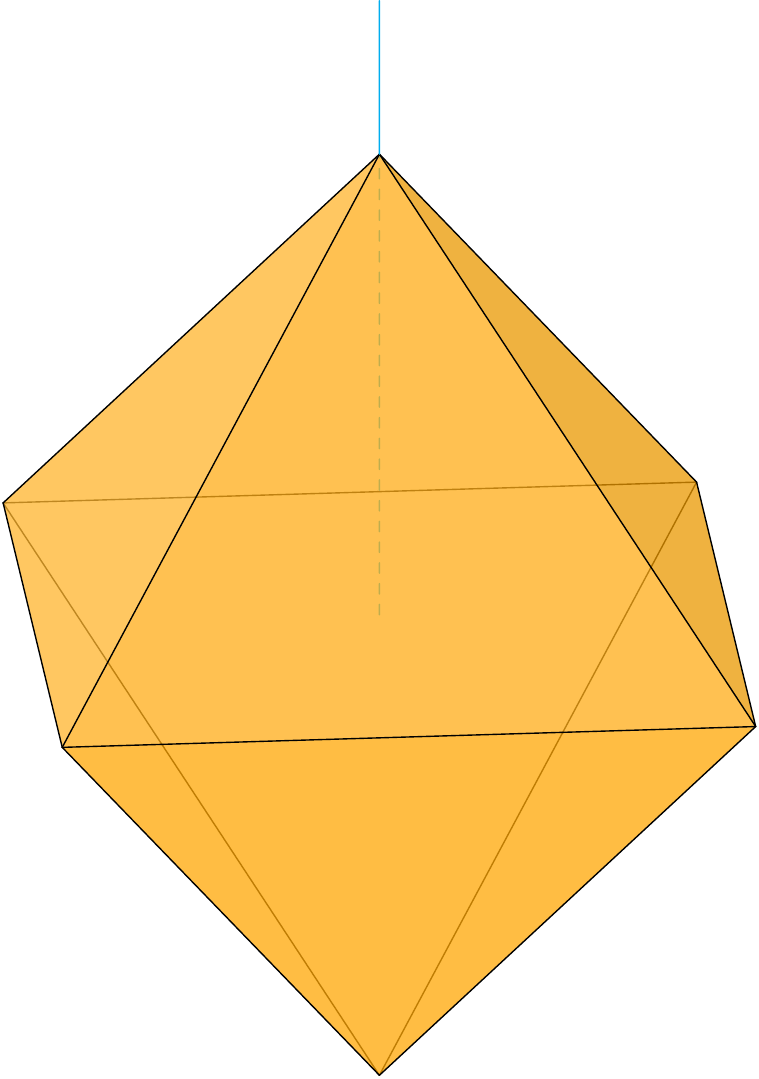}}}
&
\href{https://www.inf.fu-berlin.de/inst/ag-ti/software/DiscreteHopfFibration/gallery.html?f=hOxC2n/4/48cells_6tubes}{\vbox{\vskip-3mm \includegraphics[scale=0.8]{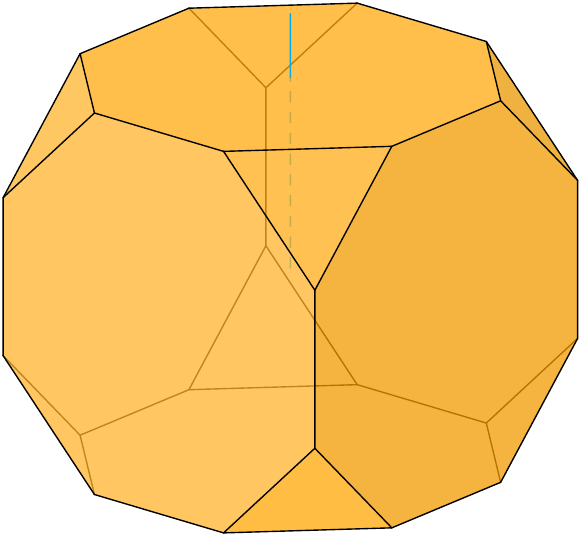}}}
&
\href{https://www.inf.fu-berlin.de/inst/ag-ti/software/DiscreteHopfFibration/gallery.html?f=hOxC2n/4/72cells_6tubes}{\includegraphics[scale=0.8]{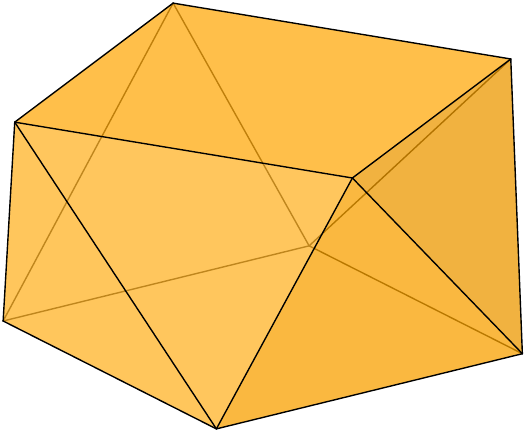}}
\\[1.5mm]
$n = 2$\break
$(\frac{k}{4} + \frac{1}{4})\cdot 2\pi$
&$n = 1$\break
$(\frac{1}{4} + \frac{1}{8})\cdot 2\pi$
&$n = 6$\break
$(\frac{k}{4} + \frac{1}{12})\cdot 2\pi$
\end{tabular}\nobreak

\vfill\nobreak
\noindent\begin{tabular}{
>{\centering\arraybackslash}m{0.33\textwidth}
>{\centering\arraybackslash}m{0.33\textwidth}
>{\centering\arraybackslash}m{0.33\textwidth}}
\href{https://www.inf.fu-berlin.de/inst/ag-ti/software/DiscreteHopfFibration/gallery.html?f=hOxC2n/4/96cells_6tubes}{\includegraphics[scale=0.8]{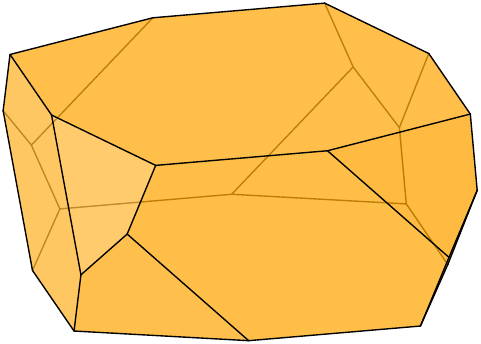}}
&
\href{https://www.inf.fu-berlin.de/inst/ag-ti/software/DiscreteHopfFibration/gallery.html?f=hOxC2n/4/120cells_6tubes}{\includegraphics[scale=0.8]{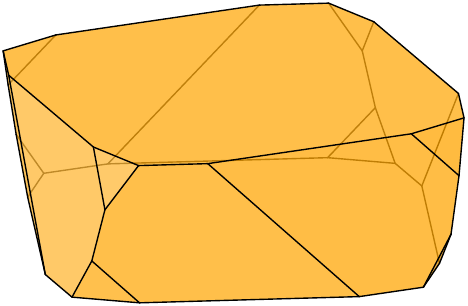}}
&
\href{https://www.inf.fu-berlin.de/inst/ag-ti/software/DiscreteHopfFibration/gallery.html?f=hOxC2n/4/144cells_6tubes}{\includegraphics[scale=0.8]{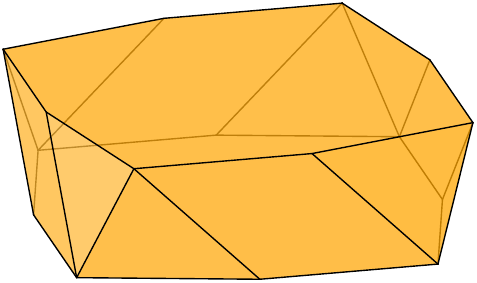}}
\\[1.5mm]
$n = 4$\break
$(\frac{2k+1}{4} + \frac{1}{16})\cdot 2\pi$
&$n = 10$\break
$(\frac{k}{4} + \frac{1}{20})\cdot 2\pi$
&$n = 3$\break
$(\frac{3}{4} + \frac{1}{24})\cdot 2\pi$
\end{tabular}\nobreak

\vfill\nobreak
\noindent\begin{tabular}{
>{\centering\arraybackslash}m{0.33\textwidth}
>{\centering\arraybackslash}m{0.33\textwidth}
>{\centering\arraybackslash}m{0.33\textwidth}}
\href{https://www.inf.fu-berlin.de/inst/ag-ti/software/DiscreteHopfFibration/gallery.html?f=hOxC2n/4/168cells_6tubes}{\includegraphics[scale=0.8]{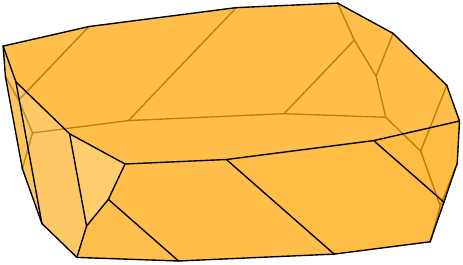}}
&
\href{https://www.inf.fu-berlin.de/inst/ag-ti/software/DiscreteHopfFibration/gallery.html?f=hOxC2n/4/192cells_6tubes}{\includegraphics[scale=0.8]{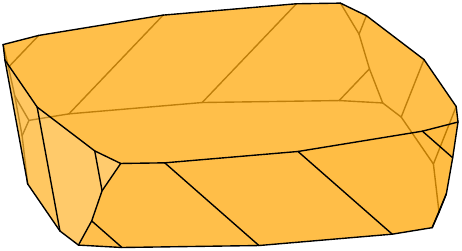}}
&
\href{https://www.inf.fu-berlin.de/inst/ag-ti/software/DiscreteHopfFibration/gallery.html?f=hOxC2n/4/216cells_6tubes}{\includegraphics[scale=0.8]{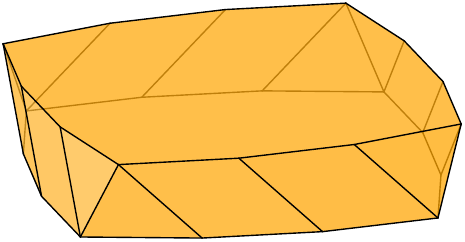}}
\\[1.5mm]
$n = 14$\break
$(\frac{k}{4} + \frac{1}{28})\cdot 2\pi$
&$n = 8$\break
$(\frac{2k+1}{4} + \frac{1}{32})\cdot 2\pi$
&$n = 18$\break
$(\frac{k}{4} + \frac{1}{36})\cdot 2\pi$
\end{tabular}\nobreak

\vfill\nobreak
\noindent\begin{tabular}{
>{\centering\arraybackslash}m{0.33\textwidth}
>{\centering\arraybackslash}m{0.33\textwidth}
>{\centering\arraybackslash}m{0.33\textwidth}}
\href{https://www.inf.fu-berlin.de/inst/ag-ti/software/DiscreteHopfFibration/gallery.html?f=hOxC2n/4/240cells_6tubes}{\includegraphics[scale=0.8]{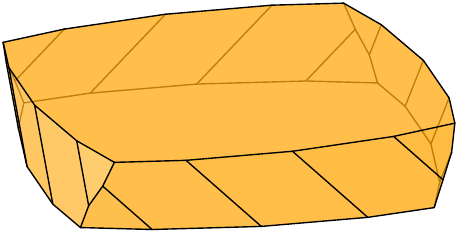}}
&
\href{https://www.inf.fu-berlin.de/inst/ag-ti/software/DiscreteHopfFibration/gallery.html?f=hOxC2n/4/264cells_6tubes}{\includegraphics[scale=0.8]{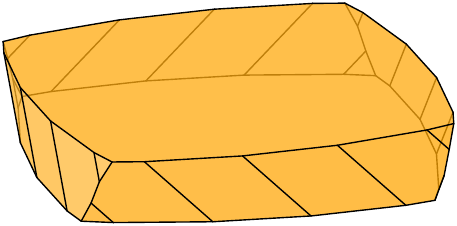}}
&
\href{https://www.inf.fu-berlin.de/inst/ag-ti/software/DiscreteHopfFibration/gallery.html?f=hOxC2n/4/288cells_6tubes}{\includegraphics[scale=0.8]{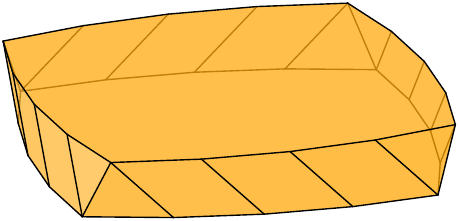}}
\\[1.5mm]
$n = 5$\break
$(\frac{1}{4} + \frac{1}{40})\cdot 2\pi$
&$n = 22$\break
$(\frac{k}{4} + \frac{1}{44})\cdot 2\pi$
&$n = 12$\break
$(\frac{2k+1}{4} + \frac{1}{48})\cdot 2\pi$
\end{tabular}\nobreak

\begin{figure}[H]
\caption{
The ${+O}$-orbit polytope of the 4-fold rotation center
$p = (0, 1, 0)$ of ${+O}$ (top left),
and the spherical Voronoi diagram of that orbit (top right).
The remaining pictures show cells of polar
$\pm\frac12[O\times C_{2n}]$-orbit polytopes
for a starting point on $K_p$
for different values of $n$. %
In addition we indicate
the counterclockwise angle (as seen from the top) by which the
group rotates as it proceeds from a cell to the consecutive cell above.
The polar orbit polytope
can be decomposed into 6 tubes,
each with $\frac{8n}{\gcd(n-2, 4)}$ cells.
The blue vertical line
indicates the cell axis, the direction towards the next cell along~$K_p$.
For an appropriate choice of starting point on $K_p$,
the group $\pm\frac12[O\times \overline{D}_{4n}]$
produces the same orbit.
When $n = 2$,
the cells that should form a tube touch each other only in a vertex.
}
\label{fig:hOxC2n_4fold}
\end{figure}\endgroup
\end{figure} %

\subsection{Consequences for starting points near rotation centers}
\label{relation-near-on}
In Sections~\ref{sec:near-axis} and~\ref{sec:on-axis} we have
discussed the different cases that can arise for an orbit
\emph{near} a rotation axis and \emph{on} a rotation axis. Indeed, we
can confirm this relation by comparing
Figure~\ref{fig:near_5fold} and
Figure~\ref{fig:IxCn_5fold}.
By the analysis that lead to Figure~\ref{fig:near_5fold},
 an orbit of
$\pm[I\times C_{n}]$
 near a 5-fold rotation axis forms a
 $4/5$,
 $2/5$,
 $3/5$, or
 $1/5$
 staircase
 if $n\equiv 1,2,3,4 \bmod 5$, respectively, and it forms
 pentagons if $n$ is a multiple of~$5$. We can check in
  Figure~\ref{fig:IxCn_5fold} that these values are precisely the
  specified rotations (up to the twist by $\frac\pi{5n}$), except when 
  $n$ is a multiple of~$5$, and in that case all five rotations are allowed.
Similarly, Figure~\ref{fig:near_4fold} corresponds with
Figure~\ref{fig:OxCn_4fold}.
  
  Conversely, we can consult the appropriate entries in
  Appendix~\ref{sec:special_starting_points} for orbits \emph{on} a
  rotation axis to conclude what type of pentagons, quadrilaterals,
  triangles, pairs, or staircases to expect for an orbit \emph{near}
  this rotation axis.

\subsection{Mappings between different tubes}
Continuing the discussion
of the tubes
for the groups $G = \pm\frac12[O\times C_{2n}]$,
from Section~\ref{ssubsec:hOxC2n},
we will now continue with the program
set out in
Figure~\ref{fig:geometric-understanding} in
Section~\ref{sec:orbit-polytopes}, by asking, for this example, how cells in different
tubes are mapped to each other.
The cells in
Figure~\ref{fig:hOxC2n_4fold} have a roughly four-sided shape.
At \emph{corners} of these quadrilaterals, three tubes meet.

To understand what is happening there, we imagine putting a starting
point $v'$ near a corner. Then $h(v')$ is near a three-fold rotation
center of $+O$. Near such a rotation center, the orbit forms either a
set of triangles, or a left or right
staircase. As just discussed, we can check this by consulting the
pictures for the orbit \emph{on} a three-fold rotation axis:
Figure~\ref{fig:hOxC2n_3fold}.

We see that those cells of
Figure~\ref{fig:hOxC2n_4fold} that have a straight line segment $A$
between the top and the bottom face at the corners ($n=6,3,18,12$)
correspond to cases where the orbit of $v'$ consists of triangles.
Indeed, one can imagine three cells arranges around a common edge $A$.
(The cells don't lie perpendicular to the axis $A$, but they are twisted.)

For the remaining cases ($n=1,4,10,14,8,5,22$) the edge is broken into
three parts between the top and the bottom face, and this is where the
cells are arranged in a staircase-like fashion.

\subsection{Small values of $n$}
\label{small-n}

For small values of $n$,
some of the cyclic-type tubical groups
recover well-known decompositions of regular/uniform polytopes
into tubes (or more commonly knows as rings).
These appear in various places in the literature.
We list some of the references.
Next to each group, we state the rotation center of the induced group
that is the image of the starting point.

\begin{itemize}
\item
$\pm[I\times C_1]$ and 5-fold rotation center
(Figure~\ref{fig:IxCn_5fold}):
We get the 
decomposition of the 120-cell into
12 tubes, each with 10 regular dodecahedra.\footnote
{A remarkable paper model of a Schlegel diagram with two rings
was produced by Robert Webb, \url{https://youtu.be/2nTLI89vdzg}.
An interesting burr puzzle was made in \cite{SS15}
using pieces of these rings.}.
Figure~\ref{fig:inscribed_tetrahedra} shows a picture of
three dodecahedra from one tube, see also \cite[Figure~21]{duval},
\cite[p.~75]{coxeter1959symmetrical} and Coxeter~\cite[p.~53]{Cox-complex}.

\item
$\pm[O\times C_1]$ and 4-fold rotation center
(Figure~\ref{fig:OxCn_4fold}):
We get the 
decomposition of the bitruncated 24-cell (the 48-cell) into
6 tubes, each with 8 truncated cubes,
stacked upon the octagonal faces.

\item
$\pm[O\times C_1]$ and 3-fold rotation center
(Figure~\ref{fig:OxCn_3fold}):
We get the 
decomposition of the bitruncated 24-cell (the 48-cell) into
8 tubes, each with 6 truncated cubes,
stacked upon the triangular faces.
\cite[p.~75-76]{coxeter1959symmetrical}.

\item
$\pm[T\times C_1]$ and 3-fold rotation center
(Figure~\ref{fig:TxCn_3fold}):
We get the 
decomposition of the 24-cell into
4 tubes, each with 6 octahedra
\cite[p.~74]{coxeter1959symmetrical},
\cite{banchoff2013}.

\item
$\pm[T\times C_1]$ and 2-fold rotation center
(Figure~\ref{fig:TxCn_2fold}):
We get the 
decomposition of the 24-cell into
6 tubes, each with 4 octahedra,
touching each other via vertices.

\item
$\pm\frac13[T\times C_3]$ and 3-fold (type I) rotation center
(Figure~\ref{fig:tTxC3n_3fold}):
This is a degenerate  case. We get the
decomposition of the hypercube into
4 ``tubes'', but each ``tube'' is just a pair of opposite cube faces.

\end{itemize}

We remark that the orbit of $G = \pm[L\times C_1]$,
is the same, up to congruence, for any starting point.
This follows since the $G$-orbit of a point $v \in \R^4$
can be obtained from the $G$-orbit of the quaternion $1$ by applying
the rotation $[1, v]$:
\begin{displaymath}
    \orbit(v, G)
    = \{\, \bar{l} v \mid l \in L \}
    = [1, v] \{\, \bar{l} \mid l \in L \}
    = [1, v] \orbit(1, G).
\end{displaymath}

\subsection{Online gallery of polar orbit polytopes}
\label{sec:gallery}

 The interested reader can explore
 polar orbit polytopes for
the cyclic-type tubical groups with
all special choices of starting points
in an online gallery that provides interactive three-dimensional views.\footnote
{\url{https://www.inf.fu-berlin.de/inst/ag-ti/software/DiscreteHopfFibration/}.
  In the PDF-file of this article,
the pictures of the cells in the figures in
Section~\ref{subsec:examples} and Appendix~\ref{sec:special_starting_points}
are linked to the corresponding entries in the gallery.}

The polytopes are shown in a central projection
to the three-dimensional tangent space at the starting point $v$ of the orbit.
The projection center lies
outside the polytope, close to the cell $F_0$ %
opposite to $v$.
In the projection, $F_0$ becomes the outer cell
that (almost) encloses all remaining projected cells.
The orientation of the outer cell is reversed with respect to the other cells.
We are mostly interested not in $F_0$
but in the cells near $v$, %
which are distorted the least in the projection,
and as a consequence,
we go with the majority and ensure that \emph{these}
cells are oriented according to our convention
(Section~\ref{sec:view-orientation}).
For large values on $n$, 
we have refrained from constructing true
Schlegel diagrams, because this would have resulted
in tiny inner cells.
As a result, cells near the boundary of the projection wrap around and
overlap.

The goal of the gallery is to show the decomposition of the polytopes
into tubes, and how these tubes are structured and interact with each other.
It is possible to remove cells one by one to see more structure.
The order of the cells is based on the distances of their orbit points to
the starting point $v$. %

\subsection{\texorpdfstring{$\pm[T\times C_n]$}{+-[TxCn]}
versus
\texorpdfstring{$\pm\frac13[T\times C_{3n}]$}{+-1/3[TxC3n]}}
\label{sec:same-symmetry}
Looking at the tubical groups in Table~\ref{tbl:left_tubical_groups},
we see that there are groups $G$ with the same induced symmetry
group $G^h$ on $S^2$. %
Thus, for the same starting point,
these groups have the same orbit circles.
However, they differ in the way how the points
on different circles are arranged relative to each other.

In this section we will consider the case where the induced group is ${+T}$.
For the same $n$, we will compare the actions of 
$\pm[T\times C_n]$ and $\pm\frac13[T\times C_{3n}]$
on and around the circles of $\H$ 
that correspond to rotation centers of ${+T}$.
We will see that these two groups
have different sets of fixed circles of $\H$,
which correspond to 3-fold rotation centers of ${+T}$.
On such a fixed circle, the size of the orbit is reduced by a factor
of~3
(from $24n$ to $8n$, see Table~\ref{tbl:tubical_supergroups}).
In Figures~\ref{fig:near_3fold}~and~\ref{fig:near_3pfold},
we visualize the effect of that difference on the orbit points
and the cells of the polar orbit polytope around these circles.
We will see that triangles and both types of staircases appear in
$\pm[T\times C_n]$ and $\pm\frac13[T\times C_{3n}]$,
depending on $n$.
In this sense, there is no sharp geometric distinction between the two families.

\paragraph{2-fold rotation center.}
Let $p \in S^2$ be a $2$-fold rotation center of ${+T}$
and let $[g]\in +T$ be the $180^\circ$ rotation around $p$.
If $n$ is even, then $[g, e_2]$ is in both groups,
and it is a simple rotation that fixes $K_p$.
If $n$ is odd, then $K_p$ is not fixed.
Thus, for the same $n$,
$\pm[T\times C_n]$ and $\pm\frac13[T\times C_{3n}]$
have the same set of fixed circles
that correspond to $2$-fold rotation centers of ${+T}$.

\paragraph{3-fold rotation center.}
The eight 3-fold rotation centers of ${+T}$ belong to two conjugacy classes,
depending on which ${+T}$-orbit they are in.
The rotation centers of \emph{type I},
are the ones in the orbit of $p_0 = (-1, -1, -1)$,
and the rotation centers of \emph{type II},
are the ones in the orbit of $-p_0 = (1, 1, 1)$.
We will see that the group $\pm[T\times C_n]$
does not distinguish between the circles $K_{p_0}$ and $K_{-p_0}$.
In particular, the orbit of a starting point on $p_0$
is congruent to the one of a starting point on $-p_0$.
However, this is not the case for $\pm\frac13[T\times C_{3n}]$.

The quaternion $-\w \in 2T$ defines the $120^\circ$ clockwise rotation
$[-\w]$ around $p_0$.
That is $-\w = \cos\frac\pi3 +p_0 \sin\frac\pi3$.
The quaternion $-\w^2 \in 2T$ defines the $120^\circ$ clockwise rotation
$[-\w^2]$ around $-p_0$.
That is $-\w^2 = \cos\frac\pi3 -p_0 \sin\frac\pi3$.

By Proposition~\ref{prop:action_on_its_circle},
the set of rotations that preserve $K_{p_0}$ is the 
same as the set of rotations that preserve $K_{-p_0}$.
Let's look at these rotations inside each of the two groups.

\begin{itemize}
\item
The elements of $\pm[T\times C_n]$ that
preserve $K_{p_0}$ (and $K_{-p_0}$)
form the subgroup
\begin{displaymath}
\langle [-\w, 1], [1, e_n] \rangle
=
\langle [-\w^2, 1], [1, e_n] \rangle
\end{displaymath}
of order $6n$.
The rotation $[-\w, 1]$ rotates $K_{p_0}$
by $\frac\pi3$ in one direction,
while $[1, e_n]$ rotates it by $\frac\pi n$ in the other direction.
Thus, the $\pm[T\times C_n]$-orbit of 
a starting point on $K_{p_0}$ forms a regular
$\lcm(2n, 3)$-gon on $K_{p_0}$.
Similarly, the $\pm[T\times C_n]$-orbit of 
a starting point on $K_{-p_0}$ forms a regular
$\lcm(2n, 3)$-gon on $K_{-p_0}$.
In particular, if $n$ is a multiple of $3$,
$\pm[T\times C_n]$ has 
a simple rotation ($[-\w, e_3]$) fixing $K_p$ and
a simple rotation ($[-\w^2, e_3]$) fixing $K_{-p_0}$.
If $n$ is not a multiple of $3$,
$\pm[T\times C_n]$ has no simple rotation fixing $K_{p_0}$ or
$K_{-p_0}$,
and the orbit points on the three circles form a left or right staircase.

\item
The elements of $\frac13[T\times C_{3n}]$ that
preserve $K_{p_0}$ (and $K_{-p_0}$)
form the subgroup
\begin{displaymath}
\langle[-\w, e_{3n}], [1, e_n]\rangle
=
\langle[-\w^2, e_{3n}^2], [1, e_n]\rangle
\end{displaymath}
of order $6n$.
We will now consider the action of this subgroup
on the circles $K_{p_0}$ and $K_{-p_0}$.
On $K_{p_0}$, the rotation $[-\w, e_{3n}]$ rotates $K_{p_0}$ by
$\frac\pi3 - \frac\pi{3n} = \frac{(n-1)\pi}{3n}$.
Its order is 
\begin{displaymath}
\frac{2\pi}{\gcd(\frac{(n-1)\pi}{3n}, 2\pi)} = 
\frac{2\pi}{\gcd(\frac{\pi}{3n}(n-1), 6n\frac{\pi}{3n})} = 
\frac{2\pi}{\frac{\pi}{3n}\gcd(n-1, 6n)} = 
\frac{6n}{\gcd(n-1, 6)}.
\end{displaymath}
Thus, the $\pm\frac13[T\times C_{3n}]$-orbit of
a starting point on $K_{p_0}$
forms a regular polygon %
with
$\lcm(2n, \frac{6n}{\gcd(n-1, 6)}) = \frac{6n}{\gcd(n-1, 3)}$
sides.
In particular, if $n-1$ is a multiple of $3$,
$\pm\frac13[T\times C_{3n}]$ has a simple rotation fixing $K_{p_0}$.
Otherwise, $G$ has no simple rotation fixing $K_{p_0}$.
On $K_{-p_0}$, the rotation $[-\w^2, e_{3n}^2]$ rotates $K_{-p_0}$ by
$\frac{\pi}3 - \frac{2\pi}{3n} = \frac{(n-2)\pi}{3n}$.
Its order is 
\begin{displaymath}
\frac{2\pi}{\gcd(\frac{(n-2)\pi}{3n}, 2\pi)} = 
\frac{2\pi}{\gcd(\frac{\pi}{3n}(n-2), 6n\frac{\pi}{3n})} = 
\frac{2\pi}{\frac{\pi}{3n}\gcd(n-2, 6n)} = 
\frac{6n}{\gcd(n-2, 12)}.
\end{displaymath}
Thus, the $\pm\frac13[T\times C_{3n}]$-orbit of
a starting point on $K_{-p_0}$
forms a regular polygon %
with
$\lcm(2n, \frac{6n}{\gcd(n-2, 12)}) = \frac{6n}{\gcd(n-2, 3)}$ sides.
In particular, if $n-2$ is a multiple of $3$,
$\pm\frac13[T\times C_{3n}]$ has a simple rotation
fixing $K_{-p_0}$.
Otherwise, $G$ has no simple rotation fixing $K_{-p_0}$.
\end{itemize}

To summarize,
$\pm[T\times C_n]$ fixes $K_{p_0}$ and $K_{-p_0}$
if and only if $n \equiv 0 \bmod 3$.
While,
$\pm\frac13[T\times C_{3n}]$ fixes $K_{p_0}$
if and only if $n \equiv 1 \bmod 3$,
and it fixes $K_{-p_0}$
if and only if $n \equiv 2 \bmod 3$.

Here, we have discussed the situation in terms of orbits near the
axis.
As discussed in Section~\ref{relation-near-on}, the results can be
checked against Figures~\ref{fig:TxCn_3fold},
\ref{fig:tTxC3n_3pfold}, and~\ref{fig:tTxC3n_3fold}.

\begin{figure}[htbp]
\begin{subfigure}{0.5\textwidth}
    \centering
    \includegraphics[scale=1]{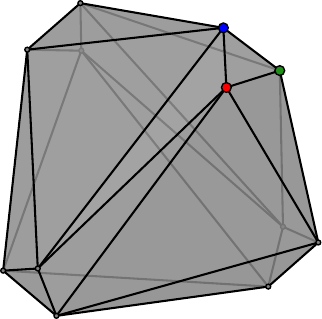}
\end{subfigure}%
\begin{subfigure}{0.5\textwidth}
    \centering
    \includegraphics[scale=1]{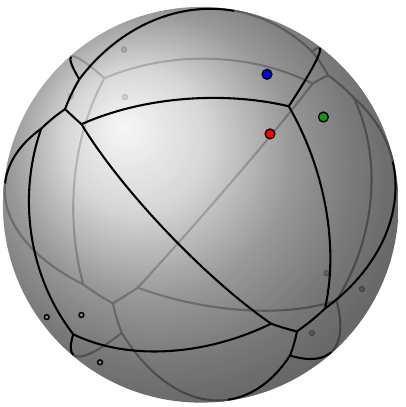}
\end{subfigure}
\caption{
The ${+T}$-orbit polytope of a starting point near
a %
3-fold rotation center
of ${+T}$ (left), and
the spherical Voronoi diagram of this orbit (right).
The picture looks the same for a Type I or a Type II center.
}
\label{fig:near_3fold_orbit}
\begin{subfigure}{.333\textwidth}
    \centering
    \includegraphics[width=4cm]{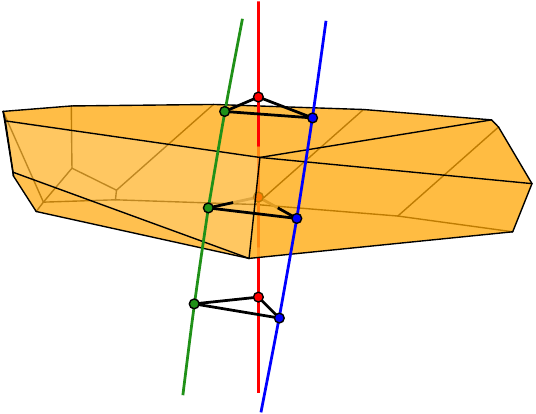}
    \caption{$\pm[T\times C_{15}]$}
    \label{subfig:near_3fold_TxC15}
\end{subfigure}%
\begin{subfigure}{.333\textwidth}
    \centering
    \includegraphics[width=4cm]{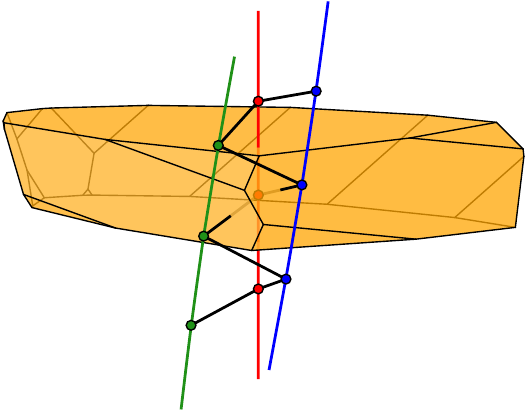}
    \caption{$\pm[T\times C_{16}]$}
    \label{subfig:near_3fold_TxC16}
\end{subfigure}%
\begin{subfigure}{.333\textwidth}
    \centering
    \includegraphics[width=4cm]{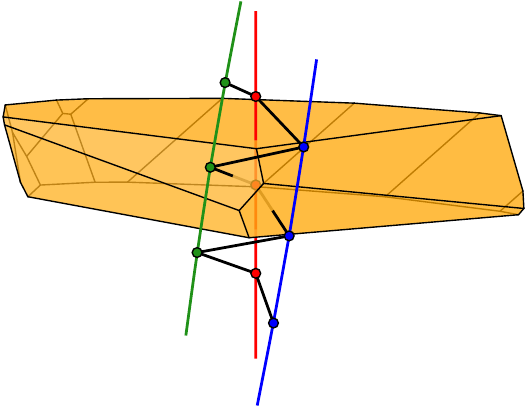}
    \caption{$\pm[T\times C_{17}]$}
    \label{subfig:near_3fold_TxC17}
\end{subfigure}
\begin{subfigure}{.333\textwidth}
    \centering
    \includegraphics[width=4cm]{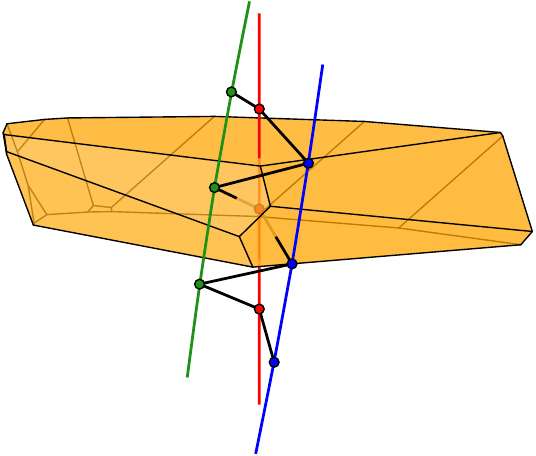}
    \caption{$\pm\frac13[T\times C_{45}]$}
    \label{subfig:near_3fold_tTxC45}
\end{subfigure}%
\begin{subfigure}{.333\textwidth}
    \centering
    \includegraphics[width=4cm]{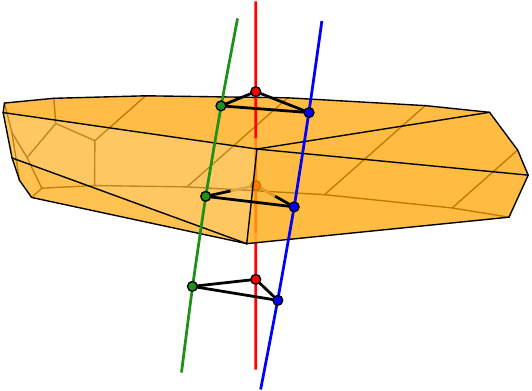}
    \caption{$\pm\frac13[T\times C_{48}]$}
    \label{subfig:near_3fold_tTxC48}
\end{subfigure}%
\begin{subfigure}{.333\textwidth}
    \centering
    \includegraphics[width=4cm]{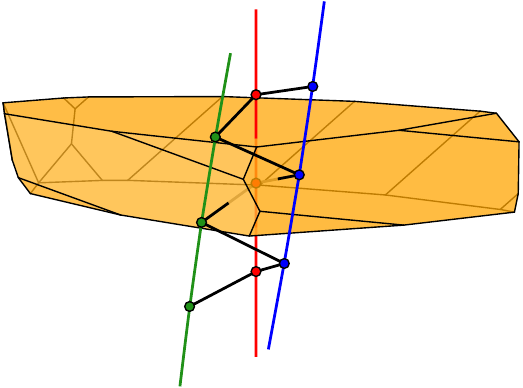}
    \caption{$\pm\frac13[T\times C_{51}]$}
    \label{subfig:near_3fold_tTxC51}
\end{subfigure}
\caption{
Cells of polar orbit polytopes of the corresponding groups,
where the image of the starting point lies near a 3-fold
rotation center of  type~I.
The colors are in correspondence with Figure~\ref{fig:near_3fold_orbit}.
}
\label{fig:near_3fold}

\begin{subfigure}{.333\textwidth}
    \centering
    \includegraphics[width=4cm]{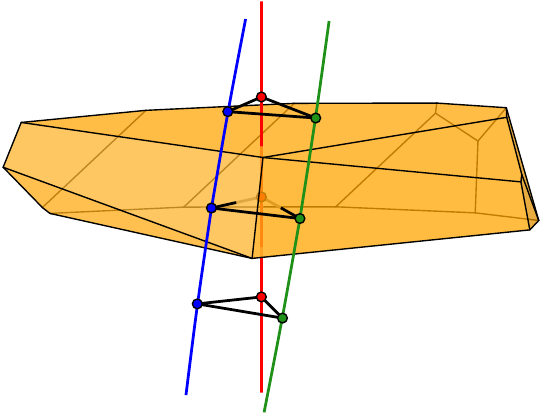}
    \caption{$\pm[T\times C_{15}]$}
    \label{subfig:near_3pfold_TxC15}
\end{subfigure}%
\begin{subfigure}{.333\textwidth}
    \centering
    \includegraphics[width=4cm]{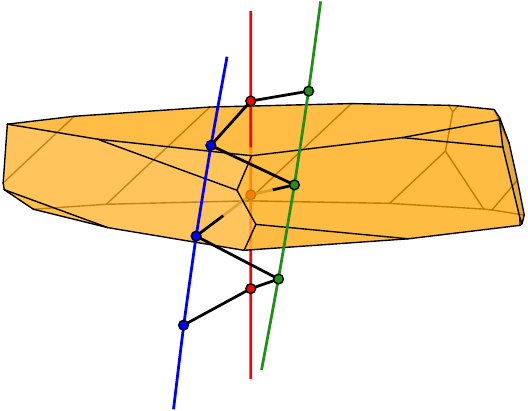}
    \caption{$\pm[T\times C_{16}]$}
    \label{subfig:near_3pfold_TxC16}
\end{subfigure}%
\begin{subfigure}{.333\textwidth}
    \centering
    \includegraphics[width=4cm]{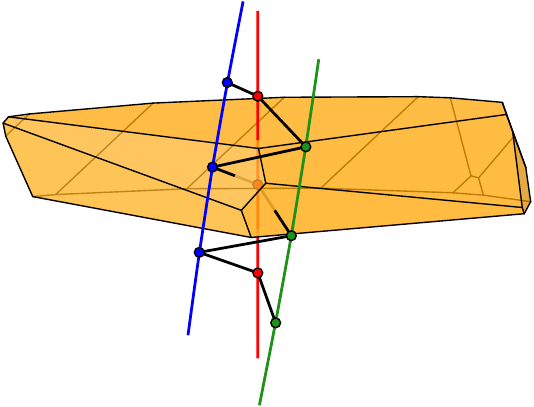}
    \caption{$\pm[T\times C_{17}]$}
    \label{subfig:near_3pfold_TxC17}
\end{subfigure}
\begin{subfigure}{.333\textwidth}
    \centering
    \includegraphics[width=4cm]{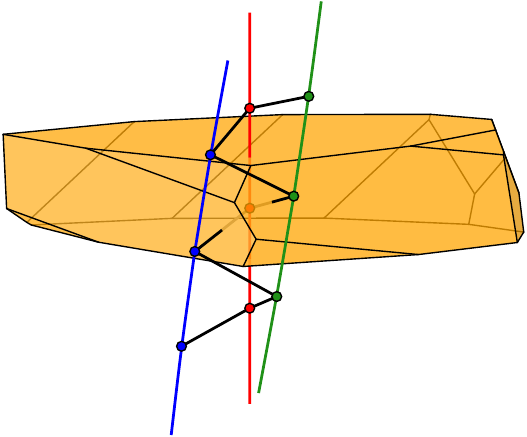}
    \caption{$\pm\frac13[T\times C_{45}]$}
    \label{subfig:near_3pfold_tTxC45}
\end{subfigure}%
\begin{subfigure}{.333\textwidth}
    \centering
    \includegraphics[width=4cm]{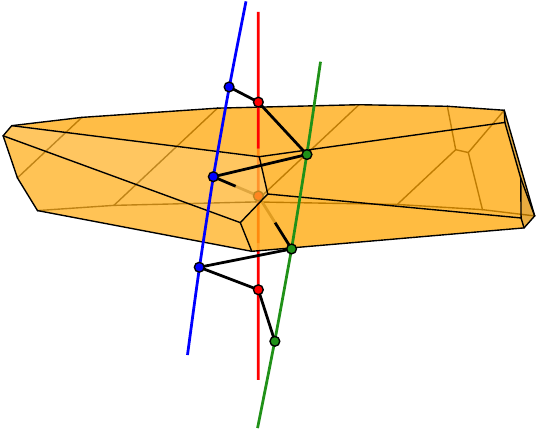}
    \caption{$\pm\frac13[T\times C_{48}]$}
    \label{subfig:near_3pfold_tTxC48}
\end{subfigure}%
\begin{subfigure}{.333\textwidth}
    \centering
    \includegraphics[width=4cm]{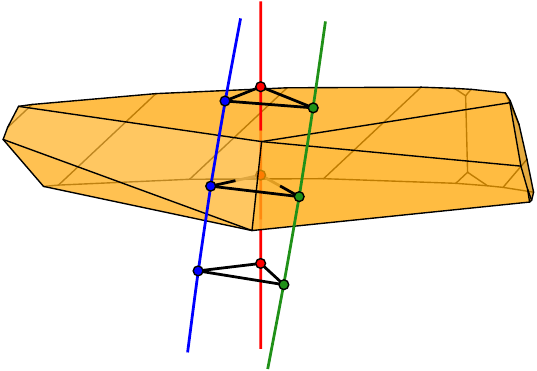}
    \caption{$\pm\frac13[T\times C_{51}]$}
    \label{subfig:near_3pfold_tTxC51}
\end{subfigure}
\caption{
Cells of polar orbit polytopes of the corresponding groups,
where the image of the starting point lies near a 3-fold
rotation center of type~II.
The colors are in correspondence with Figure~\ref{fig:near_3fold_orbit}.
}
\label{fig:near_3pfold}
\end{figure}

\section{The toroidal groups}
\label{sec:toroidal}

\subsection{The invariant Clifford torus}
\label{sec:Clifford-torus-tentative}

We will now study the large class of groups
of type $[D\times D]$ or
 $[C\times C]$ or
 $[C\times D]$,
where {both} the left {and} the right group are cyclic or
dihedral.
At the beginning of
Section~\ref{sec:clifford-tori}, we have seen that these groups have
an invariant
Clifford torus $\T_p^q$.
 All 
 tori $\T^q_p$ are the same up to orthogonal transformations. We can
 thus, without loss of generality, restrict our attention to
 the standard torus
 $\T^i_i$. Indeed this is the torus that is left invariant by the
 left and right multiplication with the groups
 $\pm[D_{2m}\times D_{2n}]$ and their subgroups, as follows from
Proposition~\ref{prop:transformations_preserving_H2}.
When we speak of \emph{the torus} in this section, we mean the torus
$\T^i_i$ and we denote it by $\T$.

Since we also have cases where the left and right subgroup are equal,
we also have to deal with their achiral extensions.
According to
Proposition~\ref{1,c},
the extending element can be taken as
$\extendingelement=*[1,c]$, which is a composition of
${*}\colon
(x_1,y_1,x_2,y_2)
\mapsto (x_1,-y_1,-x_2,-y_2)$,
which leaves the torus fixed, with $[1,c]$, for an element $c$ of the
right group,
which also leaves the torus fixed.
This means that the achiral extensions can also be found among the
groups that leave the torus fixed.

We call these groups, namely the subgroups $\pm[D_{2m}\times
D_{2n}]$ and their achiral extensions,
the \emph{toroidal groups}.

\emph{We will study and classify these groups by focusing on their action
  on $\T$}.
In particular, it will be of secondary interest whether the groups are
chiral or achiral, or which Hopf bundles they preserve. These
properties were important to derive the existence of the invariant torus,
but we will not use them for the classification.

Since $\T$ is a two-dimensional flat surface, the symmetry groups
acting on $\T$ bear much resemblance to the discrete symmetry groups of the
plane, i.e., the wallpaper groups. These groups are
well-studied and
intuitive.
All wallpaper groups except those that contain
3-fold rotations will make their appearance (12 out of the 17 wallpaper groups).
The reason for excluding 3-fold rotations is that
a Clifford torus has two distinguished directions, which are
perpendicular to each other, and these directions must be preserved.
We don't assume familiarity with the classification of the wallpaper groups.
We will develop the classification as we go and adapt it to our needs.

\begin{figure}[tb]
  \centering
 \includegraphics{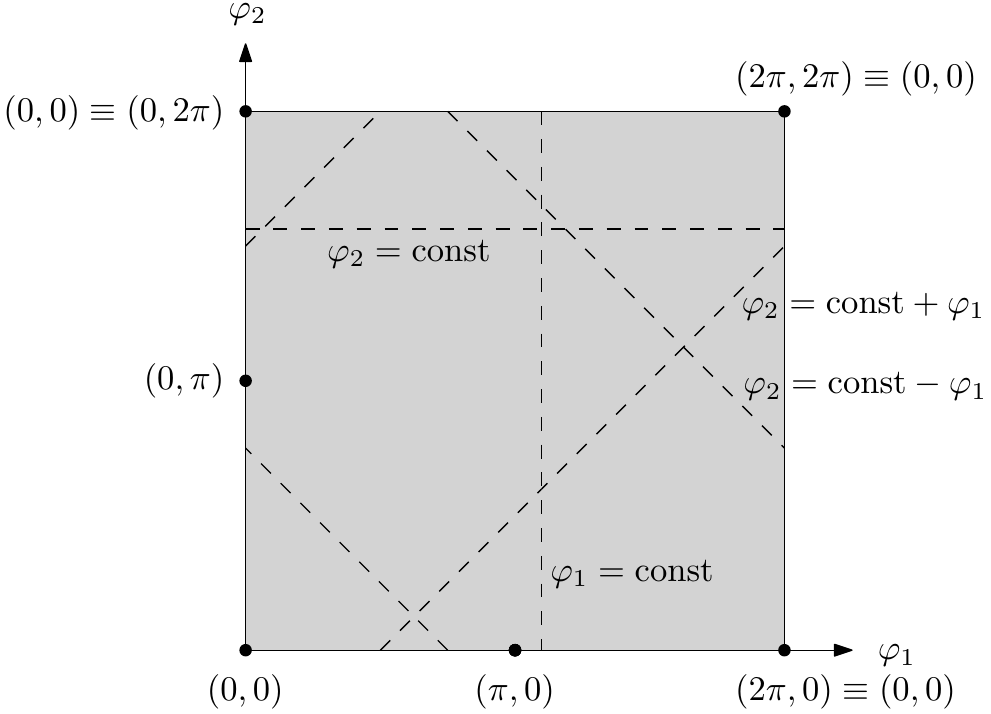}
  \caption{Torus coordinates for the Clifford torus}
  \label{fig:torus-coordinates}
\end{figure}

\subsection{Torus coordinates and the torus foliation}
\label{sec:torus-coordinates}

The Clifford torus
belongs to a foliation of $S^3$ by a family of tori, which,
in terms of Cartesian coordinates $(x_1,y_1,x_2,y_2)$,
have the equations
\begin{equation}
  \label{eq:torus-equation}
  x_1^2+y_1^2 = r_1^2
  ,\
  x_1^2+y_2^2 = r_2^2
\end{equation}
for fixed radii $r_1,r_2$ with $0<r_1,r_2<1$ and $r_1^2+r_2^2=1$.
The standard Clifford torus has the parameters
$r_1=r_2=\sqrt{1/2}$. As limiting cases,
$r_1=1$ gives the great
circle
in the $x_1,y_1$-plane, and
$r_1=0$ gives the great
circle
in the $x_2,y_2$-plane. Every torus in this family is the Cartesian product of two
circles, and thus is a flat torus, with a locally Euclidean metric,
forming a
$2\pi r_1 \times 2\pi r_2 $
rectangle with opposite sides identified.

\begin{figure}[htbp]
  \centering
  \includegraphics{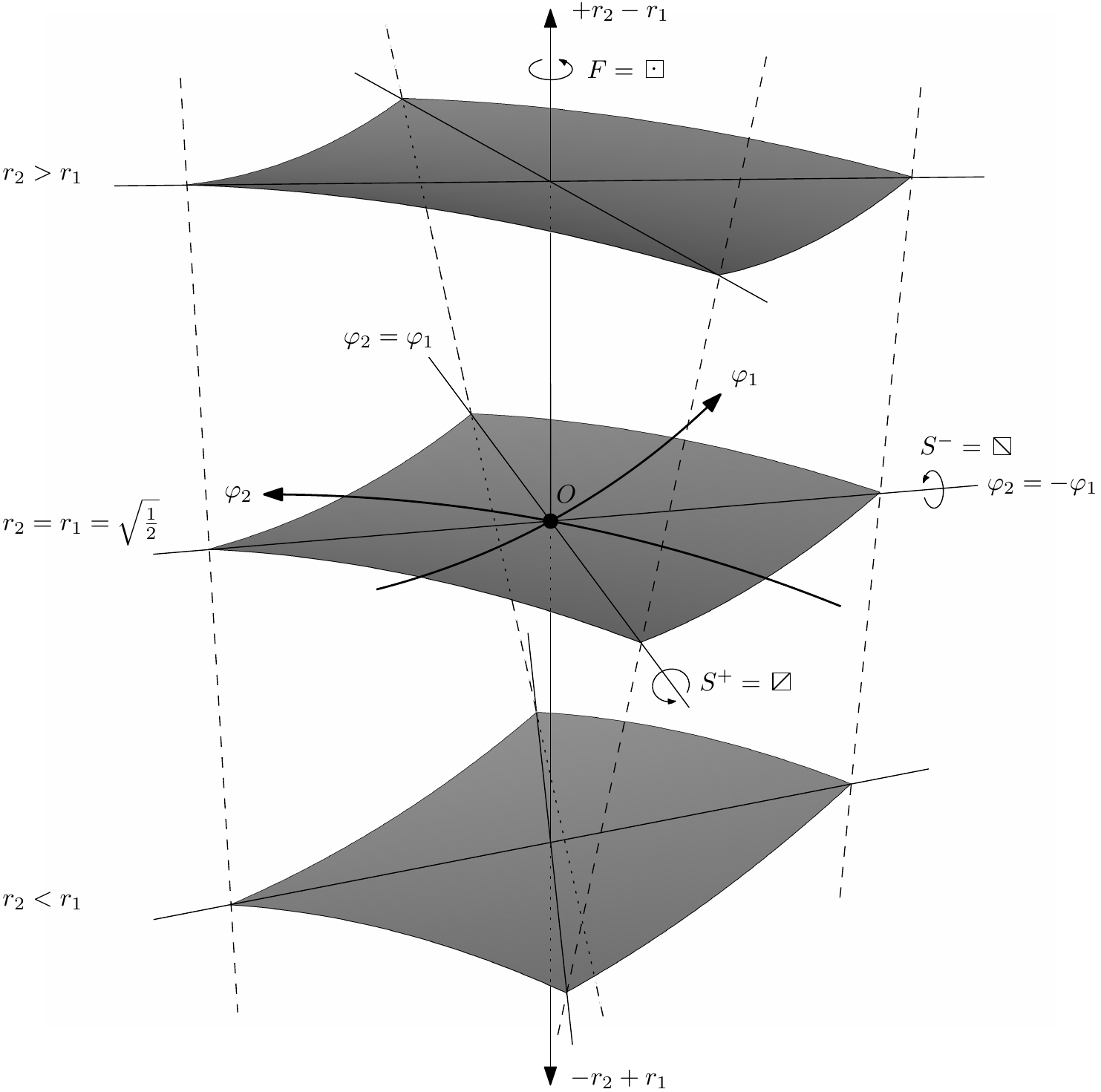}
  \caption{Patches of flat tori in the 3-sphere.
    This %
    illustration is %
    a central
    projection from the 3-sphere to the 3-dimensional
    tangent hyperplane at the point $O=(\sqrt{1/2},0,\sqrt{1/2},0)$,
    which is the marked point in the center.
    Great circles, i.e. geodesics on the 3-sphere,
    appear as straight lines.
    The axes of the \emph{flip} half-turns $F$
    and the \emph{swap} half-turns $S^+$ and $S^-$ are indicated.
\\
The tangent in direction $\phi_1$ points in the direction $(0,1,0,0)$
and 
the tangent vector in direction $\phi_2$ points in the direction
$(0,0,0,1)$.
The ``perpendicular direction'', which is the vertical axis $+r_2-r_1$
in the figure, is the direction
 $(-\sqrt{1/2},0,\sqrt{1/2},0)$.
}
  \label{fig:torus-patch}
\end{figure}

The best way to see the mapping to the rectangle is to use double
polar coordinates:
   \begin{equation}
     \label{eq:torus-coordinates}
        \begin{pmatrix}
     x_1\\ y_1\\
     x_2\\ y_2
   \end{pmatrix}
=   \begin{pmatrix}
     r_1\cos \phi_1\\
     r_1\sin \phi_1\\
     r_2\cos \phi_2\\
     r_2\sin \phi_2
   \end{pmatrix}
 \end{equation}
Then $\phi_1$ and $\phi_2$ (appropriately scaled) can be used as
rectangular two-dimensional coordinates, see
Figure~\ref{fig:torus-coordinates}.

The lines with $\phi_1=\mathrm{const}$ and $\phi_2=\mathrm{const}$ are
what we would normally call meridian
circles and parallel circles of the torus, except that there is no natural way to
distinguish the two classes. These circles have radius $\sqrt{1/2}$.
The $45^\circ$ lines with $\phi_2=\mathrm{const}+\phi_1$
and $\phi_2=\mathrm{const}-\phi_1$
are great
circles. They are the circles from the Hopf bundles $\H_i$ and $\H^i$.

Figure~\ref{fig:torus-patch} gives a picture of
corresponding patches around the origin $\phi_1=\phi_2=0$ for three
tori. The middle one is the Clifford torus with $r_1=r_2=\sqrt{1/2}\approx0.7$, the top one has
$r_1=0.55<r_2\approx 0.835$, and the bottom one has the reversed
values $r_1$ and~$r_2$.

Each torus is intrinsically flat, i.e., isometric to the Euclidean
plane in every small patch, but, as the figure suggests, it is
embedded as a ``curved'' surface inside $S^3$. The only
``lines'' in the torus that are geodesics of
 $S^3$ are those that are parallel to the diagonal lines $\phi_2=\pm\phi_1$. The dotted
 ``vertical'' lines connect points with the same
 $\phi_1,\phi_2$-coordinates on different tori. They are great
 circles, and they intersect every torus of the family orthogonally.

 In Section~\ref{sec:dup-example}, we will see the easy equation
 $x_1x_3=x_2x_4$~\eqref{clif2a} for the same torus in a different
 coordinate system.

\subsection{Symmetries of the torus}
\label{sec:torus-symmetries}

Since the torus is locally like the Euclidean plane, and the
plane is the universal covering space of the torus, we can investigate
the isometric symmetries of the torus by studying the isometries of the plane.
However, not every isometry of the plane can be used as a symmetry of
the torus; it must be ``compatible'' with the torus structure.
The following theorem makes this precise:

\begin{theorem} \label{theorem-torus-symmetries}
  There is a one-to-one correspondence between
  \begin{itemize}
  \item groups $G$ of isometries of the torus $[0,2\pi) \times [0,2\pi)$,
  \item groups $\hat G$ of isometries $x\mapsto Ax+t$
    of the $(\phi_1,\phi_2)$-plane
    with the following properties:
    \begin{enumerate}[\rm(i)]
    \item The directional part $A$ of every isometry
      in $\hat G$
      keeps the integer grid $\mathbb Z^2$
      invariant.
    \item The group contains the two
    translations $\phi_1 \mapsto \phi_1+2\pi$ and
    $\phi_2 \mapsto \phi_2+2\pi$.
    \end{enumerate}
  \end{itemize}
\end{theorem}

The proof uses the following lemma, which shows how to lift %
torus isometries to plane isometries:
\begin{lemma}
  \label{toroidal-symmetry-constraints}
  Let $\Lambda$ denote the scaled integer grid $\{\,(k_12\pi,k_22\pi) \mid k_1,k_2\in
  \mathbb Z\,\}$, and let $p\colon
\mathbb R^2 \to  \mathbb R^2|_\Lambda$ be the quotient map
from the plane to the torus $[0,2\pi) \times [0,2\pi)$:
 $$p (\phi_1,\phi_2) =
  (\phi_1\bmod 2\pi,\phi_2\bmod 2\pi)$$ 
  For every isometry $T$ of the torus  $[0,2\pi) \times [0,2\pi)$,
  there is an isometry $\hat T$ of the plane with the following
  properties.
  \begin{enumerate}[\rm(a)]
  \item 
    $T(p(x)) = p(\hat T(x))$ for all $x \in \mathbb R^2$.
  \item $\hat T$ maps the grid $\Lambda$ to a translate of $\Lambda$.
  \end{enumerate}
 The isometry $\hat T$ is unique up to translation by a grid vector
  $t\in\Lambda$.
\end{lemma}
\begin{proof}
  Pick some point $y_0$ of the torus and let $T(y_0)=y_0'$. Find points $x_0,x_0'\in \mathbb
R^2$
with $y_0=p(x_0)$ and $y_0'=p(x_0')$.
Since $p$ is locally injective, the mapping $T$ can be lifted to a
mapping $\hat T(x)=p^{-1}(T(p(x)))$ 
in some neighborhood $N(x_0)$ of $x_0\in \mathbb R^2$:
\begin{equation}
  \label{eq:lift}
\begin{tikzcd}
	{\mathbb R^2\colon} & {x_0} & {x_0'} \\
	{\mathbb T\colon} & {y_0} & {y_0'}
	\arrow["{\hat T}", from=1-2, to=1-3]
	\arrow["p", from=1-3, to=2-3]
	\arrow["T"', from=2-2, to=2-3]
	\arrow["p"', from=1-2, to=2-2]
      \end{tikzcd}
\end{equation}
In other words,
 $\hat
T(x_0)=x_0'$, and
for all $x\in N(x_0)$:
\begin{equation}
  \label{eq:commute}
p(\hat T(x)) = T(p(x)) 
\end{equation}
 Moreover,
since both $p$ and $T$ are
locally isometries,
$\hat T$ is an isometry in
$N(x_0)$. This isometry can be extended to a unique isometry $\hat T$
of the plane.

To extend the validity of~\eqref{eq:commute} from $N(x_0)$ to the whole
plane, we look at a path $x_0+\lambda t$ from $x_0$ to an arbitrary
point $x_0+t$ of the plane, where ($0\le \lambda \le 1$). On the torus, it
corresponds to a path $p(x_0+\lambda t)$, which is mapped to an image path
$T(p(x_0+\lambda t))$, which in turn can be lifted to a path on $\R^2$. Since
$p$ is locally invertible and an isometry, \eqref{eq:commute} must
hold along the whole path, and therefore for an arbitrary point
$x_0+t$ of the plane.  This is claim~(a).

To show claim~(b),
consider any $t\in \Lambda$.
By \eqref{eq:commute},
\begin{displaymath}
  p(\hat T(t)) = T(p(t)) = T(p(0))
\end{displaymath}
that is, 
all values $\hat T(t)$ for $t\in \Lambda$ project to the same point
$T(p(0))$ on
the torus.
It follows that the image of $\Lambda$ under
$\hat T$ is contained in a translate of $\Lambda$. But then it must be
\emph{equal} to this translate of~$\Lambda$.

Once $x_0$ and $x_0'$ have been chosen, 
the construction gives a unique transformation $\hat T$.
The result %
can be varied by adding an arbitrary translation
$t\in \Lambda$ to $x_0$ (before applying $\hat T$)
or $t'\in \Lambda$ to $x_0'$ (after applying $\hat T$).
By property~(b), it makes no difference whether we are allowed
to translate by an element of $\Lambda$ before
applying $\hat T$ or after (or both).
This proves the uniqueness claim of the lemma.
\end{proof}

As a consequence, we can write a torus isometry like a plane isometry
in the form
$x\mapsto Ax+t$ with an orthogonal matrix $A$ and a translation vector $t$,
bearing in mind that $t$ is unique only up to grid translations.

\begin{proof}[Proof of Theorem~\ref{theorem-torus-symmetries}]
Given a group $G$, we can construct the lifted group $\hat G$ as the
set of lifted isometries $\hat T$ of the transformations $T\in G$
according to the lemma.
The group property of $\hat G$ can be easily shown by extending the diagram
\eqref{eq:lift}:
\begin{displaymath}
\begin{tikzcd}
	{\mathbb R^2\colon} & {x_0} & {x_0'} & {x_0''} \\
	{\mathbb T\colon} & {y_0} & {y_0'}& {y_0''}
	\arrow["\hat T\hat T'"', %
        curve={height=-18pt}, swap, from=1-2, to=1-4]
	\arrow["{\hat T}", from=1-2, to=1-3]
	\arrow["{\hat T}'", from=1-3, to=1-4]
	\arrow["p", from=1-3, to=2-3]
	\arrow["T"', from=2-2, to=2-3]
	\arrow["T'"', from=2-3, to=2-4]
	\arrow["TT'"', curve={height=18pt}, from=2-2, to=2-4]
	\arrow["p"', from=1-2, to=2-2]
	\arrow["p"', from=1-4, to=2-4, swap]
      \end{tikzcd}
\end{displaymath}

The translations $\phi_1 \mapsto \phi_1+2\pi$ and
    $\phi_2 \mapsto \phi_2+2\pi$ arise as lifts of the identity $\id
    \in G$.
It is clear that a matrix $A$
  keeps the scaled integer grid
  $\Lambda := \{\,(k_12\pi,k_22\pi) \mid k_1,k_2\in \mathbb Z\,\}$
  invariant (Property~(b)) if and only if it keeps the standard integer grid $\mathbb Z^2$
  invariant (Property~(i)).

  Conversely, given a transformation $\hat T$ in the group $\hat G$,
  we can define
$T$ as follows:
 For a point $y_0$ of the torus, pick a point $x_0$ with $p(x_0)=y_0$,
 and define $T(y_0)$ through the relation
\eqref{eq:lift}:
$T(y_0) := p(\hat T(x_0))$.
 The choice of $x_0$ is ambiguous. %
It is
determined only up a translation by $t\in \Lambda$, but
we see that this has no effect on $T(y_0)$:
 \begin{displaymath}
p(\hat T(x_0+t))=
p(\hat T(x_0)+t')=
p(\hat T(x_0))
 \end{displaymath}
 By
property~(i), or
property~(b), $t'\in \Lambda$, and
therefore the ambiguity evaporates through the projection $p$.
\end{proof}

\subsubsection{Torus translations}
\label{sec:torus-translations}
The simplest operations are the ones that appear as translations on
the torus, modulo $2\pi$. We denote them by
$$R_{\alpha_1,\alpha_2}\colon (\phi_1,\phi_2) \mapsto
(\phi_1+\alpha_1,\phi_2+\alpha_2)$$
in accordance with \eqref{rotation}.
In this notation,
a left rotation $[\exp \alpha i,1]$ turns out to be a negative translation along
the $45^\circ$ direction: $T_{-\alpha,-\alpha}$.
A right rotation $[1,\exp \alpha i]$ is
a translation in the $-45^\circ$ direction:
$R_{\alpha,-\alpha}$.
Arbitrary torus translations can be composed from left and right rotations,
and the general translation is written in quaternion notation as
$$R_{\alpha_1,\alpha_2}=\left[\exp(\tfrac{-\alpha_1-\alpha_2}2i),\, \exp(\tfrac{\alpha_1-\alpha_2}2i)\right].
$$
The torus translations
$R_{\alpha,0}$ and
$R_{0,\alpha}$ along the $\phi_1$ and $\phi_2$-axis are simple
rotations, leaving
the $x_2,y_2$-plane or
the $x_1,y_1$-plane fixed,
respectively.

One should bear in mind that all ``translations'', as they appear on
the torus, are actually rotations of $S^3$. (Only the left and right
rotations among them
may be called \emph{translations of~$S^3$} with some justification, because
they correspond to the translations in elliptic 3-space.)

\subsubsection{The directional group: symmetries with a fixed point}
\label{sec:fixed-point}

We pick the point $O=(\sqrt{1/2},0,\sqrt{1/2},0)$ with torus coordinates
$\phi_1=\phi_2=0$ as a reference point or origin on~$\T$.
Every isometry of $\T$ can be decomposed in a unique way into a symmetry that leaves
$O$ fixed (the \emph{directional part}), plus a torus translation (the
\emph{translational part}).

Let us therefore study the symmetries that leave $O$ fixed. In the plane, these
would be all rotations and reflections. However,
according to Theorem~\ref{theorem-torus-symmetries}
we can only use symmetries that
leave the standard square grid
$\mathbb Z^2$ invariant,
apart from a translation. This allows
rotations by multiples of $90^\circ$, as well as reflections in the
coordinate axes and in the $45^\circ$-lines.

In the plane, 
these seven operations together with the identity form the dihedral group $D_8$, the symmetries
of the square.
We denote the group by $D_8^\T$, to indicate that we think of the
transformations of $S^3$ that leave the torus $\T$ invariant.
Table~\ref{tab:symmetries} summarizes these operations and their
properties.
For each operation, we have chosen a symbol indicating the
axis direction in case of a reflection, or otherwise some suggestive
sign, and
a name. We also give the quaternion representation, the effect in
terms of the $\phi_1,\phi_2$-coordinates, and the order of the group element.

Some transformations may \emph{swap} the two sides of
$\T$, exchanging the tori with parameters $r_1,r_2$
and $r_2,r_1$. This is indicated by a ``$-$'' in the column ``side'',
and the names of these operations include the term ``swap''.
The nonswapping operations leave every torus of the foliation~\eqref{eq:torus-equation}
invariant,
not just the ``central'' Clifford torus.

The column ``det''\ indicates whether the operation is \OP\ ($+$) or
\OR~($-$).
One must keep in mind that the operation on the torus $\T$ induces a
transformation of the whole $S^3$, and what appears as a
reflection in the planar $\phi_1,\phi_2$-picture of $\T$ may or may not be an \OR\
transformation of $S^3$.
Thus, it may at first sight come as a surprise that the \emph{torus swap} \sym/ is \OP.
The reason is that it goes together with a swap of the sides.
As shown in Figure~\ref{fig:torus-patch}, it is actually a half-turn
around the axis $S^+$.
(The product of the signs in the ``side'' and ``det'' columns tells
whether the operation is \OP\ when considered purely in the plane.)

  Figure~\ref{fig:torus-patch}
 makes it clear why there is no ``pure swap'', no ``inversion'' at the
central torus that would keep the torus pointwise fixed and swap the
two sides of the torus: such a mapping would flip the dashed
perpendicular lines and thus map the long side of the rectangular
patch on the top to the short side of the rectangular
patch at the bottom. We see that a swap is only possible if it goes hand in hand with an
exchange of the $\phi_1$ and $\phi_2$ axes. In particular, such an exchange comes with
the rotations by $\pm90^\circ$, the right and left \emph{swapturn}
operations, which are accordingly \OR.

The column ``conj.''\ indicates operations that are conjugate to each
other, i.e., geometrically equivalent. Thus, for example, the
operation $\sym|$ may, in a different coordinate system, appear
as the operation $\sym-$. By contrast,
\sym/ and
\sym\setminus\ are distinguished: the axis of \sym/ belongs to the
invariant left Hopf bundle $\H^i$,
and
 the axis of \sym\setminus\ belongs to the
invariant right Hopf bundle $\H_i$. %
The operations
\sym/ and
\sym\setminus\ are mirrors of each other, i.e., conjugate under an
\OR\ transformation. This is indicated in the last column.

When viewed in isolation, the half-turns $S^+=\sym/$, $S^-=\sym\setminus$, and $F=\sym.$
are conjugate to each other. However, they are distinct when
considering only transformations that leave the torus invariant.

\begin{table}
  \centering
\begin{tabular}{|c|l|c|c|c|c|c|c|c|}
  \hline
  \!\!symbol\!\!&name&$[l,r] %
               $&$(\phi_1,\phi_2)\to$&\!order\!&side&det&\!conj.\!%
  &\!mirror\!\\\hline
  \sym1&identity&$[1,1]$&$(\phi_1,\phi_2)$&1&$+$&$+$&--&\sym1\\
  \sym|&horizontal reflection &${*}[i,i]$
                          &$(-\phi_1,\phi_2)$&2&$+$&$-$&\sym-%
  &--\\
  \sym-&vertical reflection &${*}[k,k]$
                          &$(\phi_1,-\phi_2)$&2&$+$&$-$&\sym|&--\\
  \sym.&torus flip $F =\sym-\cdot\sym|%
         $&$[j,j]$
                          &$(-\phi_1,-\phi_2)$&2&$+$&$+$&--&\sym.\\
  \sym/&torus swap $S^+$&$[i,k]$
                          &$(\phi_2,\phi_1)$&2&$-$&$+$&--&\sym\setminus\\
  \sym\setminus&alternate torus swap $S^-$&$[-k,i]$
                          &$(-\phi_2,-\phi_1)$&2&$-$&$+$&--&\sym/\\
  \sym L&left swapturn $\sym/\cdot\sym|$ %
                     &${*}[-j,1]$ %
                          &$(\phi_2,-\phi_1)$&4&$-$&$-$&\sym R%
  &--\\
  \sym R&right swapturn $\sym L^{\,-1}$  %
          &${*}[1,j]$
                &$(-\phi_2,\phi_1)$&4&$-$&$-$&\sym L%
  &--\\
  \hline
\end{tabular}
\caption[The group $D_8^\mathbb{T}$, the directional parts of the torus symmetries]
{The directional parts of the torus symmetries,
  the elements of the group $D_8^\mathbb{T}$. Some come in conjugate
  pairs, as indicated in the column ``conj.'', meaning that they are
  geometrically equivalent. The conjugacy is
  established by any of the operations \hbox{\sym/}
  or \sym{\setminus} in these cases.
  The torus flip \sym. commutes with all other operations.
The last column shows the mirror transformation for each
transformation of determinant $+1$ (the \OP\ transformations).
}
  \label{tab:symmetries}
\end{table}

\subsubsection{Choice of coordinate system}
\label{sec:coordinate-system}
 The conjugacies discussed above introduces ambiguities in the
 representation of torus translations, which depend on the choice of
 the coordinate system for a given invariant torus.
 $R_{\alpha_1,\alpha_2}$ may, in a different coordinate system, appear
 as
 $R_{-\alpha_1,-\alpha_2}$ (conjugacy by \sym.), or as
 $R_{\alpha_2,\alpha_1}$ (conjugacy by \sym/), or as
 $R_{-\alpha_2,-\alpha_1}$ (conjugacy by \sym\setminus).
 (The operation  $R_{\alpha_1,-\alpha_2}$ or $R_{-\alpha_1,\alpha_2}$
 is its mirror operation.)
 The choice of origin in the $\phi_1,\phi_2$-plane, on the other
 hand, has no influence on the torus translations. It only affects the
 other operations.

\subsubsection{The directional group and the translational subgroup}

We have mentioned that every symmetry of the torus can be decomposed in
a unique way (after fixing an origin) into a
directional part and a translational part.

For a group $G$, the torus translations contained in it form
 a normal subgroup,
the
\emph{translational subgroup}, which we denote by $G_{\Box}$.
The directional parts of the group operations form the
\emph{directional group} of $G$.  It is a subgroup of $D_8^\T$, and we
will use it as a coarse classification of the toroidal groups.
(The directional group is isomorphic to the factor group $G/G_{\Box}$.)

The ten subgroups of $D_8^\T$ are listed in
Table~\ref{tab:subgroups}, together with a characteristic symbol and a name.
 Figure~\ref{fig:subgroups} shows their pictorial representation.

\begin{table}
  \centering
  \begin{tabular}{|l|l|c|c|c|c|}
    \hline
    group &name&chirality&swapping&conjugate & mirror
    \\\hline
    $\grp 1 = \{ \sym 1 \}$ &translation &chiral&no&--&\grp1\\
    $\grp | = \{ \sym 1, \sym | \}$ &reflection &achiral&no&\grp -, by \sym /&--\\
    $\grp - = \{ \sym 1, \sym - \}$ &reflection&achiral&no&\grp |, by \sym /&--\\
    $\grp . = \{ \sym 1, \sym . \}$ &flip&chiral&no&--&\grp .\\
    $\grp + = \{ \sym 1, \sym|,\sym-,\sym . \}$ &full reflection&achiral&no&--&--\\
    $\grp / = \{ \sym 1, \sym / \}$ &swap&chiral&yes&--&\grp\setminus\\
    $\grp \setminus = \{ \sym 1, \sym \setminus \}$ &swap&chiral&yes&--&\grp/\\
    $\grp X = \{ \sym 1, \sym /, \sym \setminus,\sym . \}$ &full swap&chiral&yes
      &--&\grp X\\
    $\grp L = \{ \sym 1, \sym L, \sym .,\sym R \}\cong C_4$ &swapturn&achiral&yes&--&--\\
    $\grp * = \{ \sym 1, \sym L, \sym .,\sym
    R,\sym|,\sym-,\sym/,\sym\setminus \}$ &full torus&achiral&yes&--&--\\
    \hline
  \end{tabular}
  \caption[The 10 subgroups of $%
    D_8^\T$]
{The 10 subgroups of $%
    D_8^\T$.
    A group is achiral if it contains an \OR\ transformation.
    A group is swapping if it contains a transformation that swaps the
    two sides of the torus.
  The fifth column shows to which other groups the group is
  conjugate by an \OP\ transformation.
  The last column shows the mirror group of each chiral group,
i.e., the conjugate group by an \OR\ transformation.  
(Each achiral group in this list is its own mirror image.)}

  \label{tab:subgroups}
\end{table}

\begin{figure}[htb]
  \centering
  
  \begin{minipage}{0.20\textwidth}
    \centering
    \includegraphics[scale=.16]{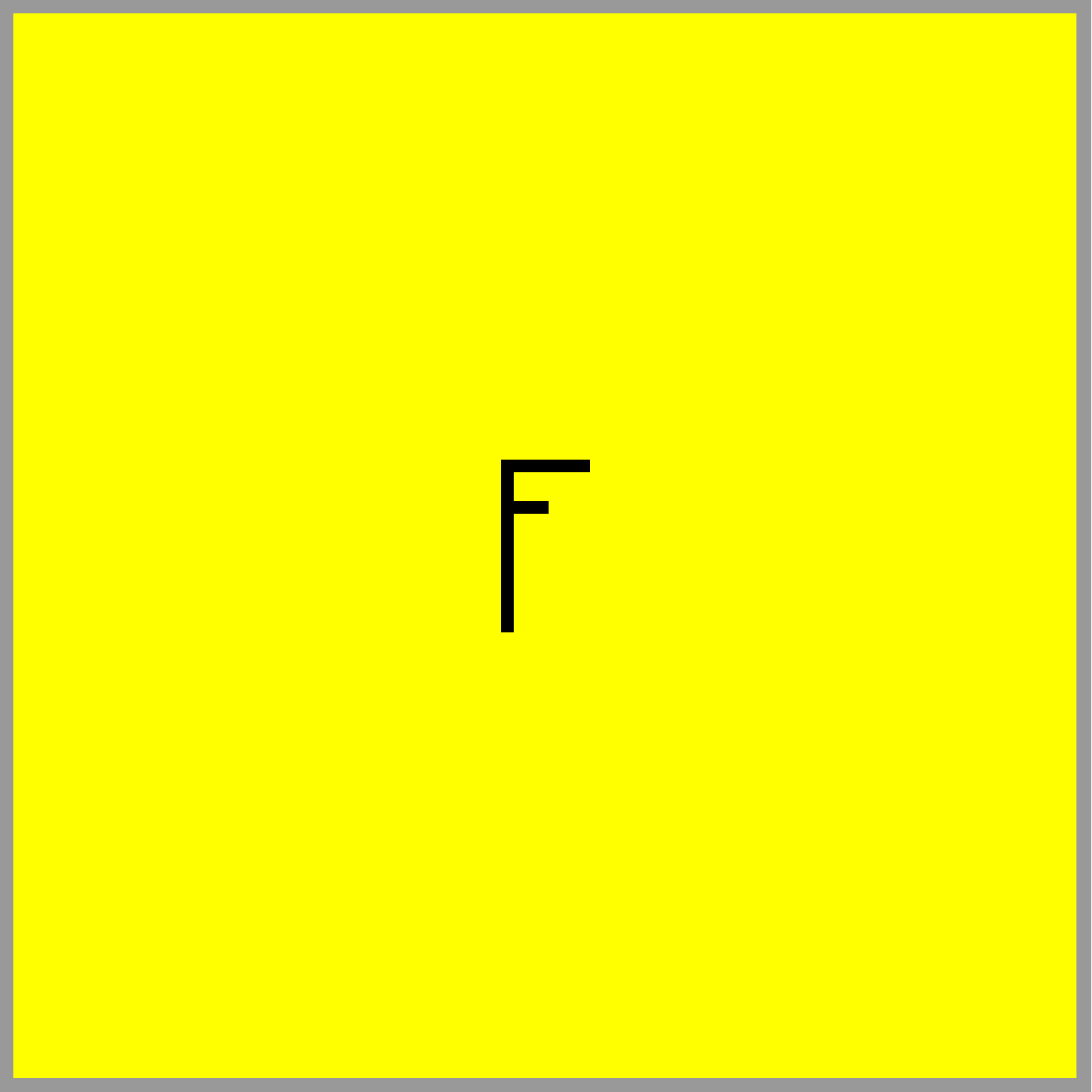} \\
    $\grp 1$
  \end{minipage}%
  \begin{minipage}{0.20\textwidth}
    \centering
    \includegraphics[scale=.16]{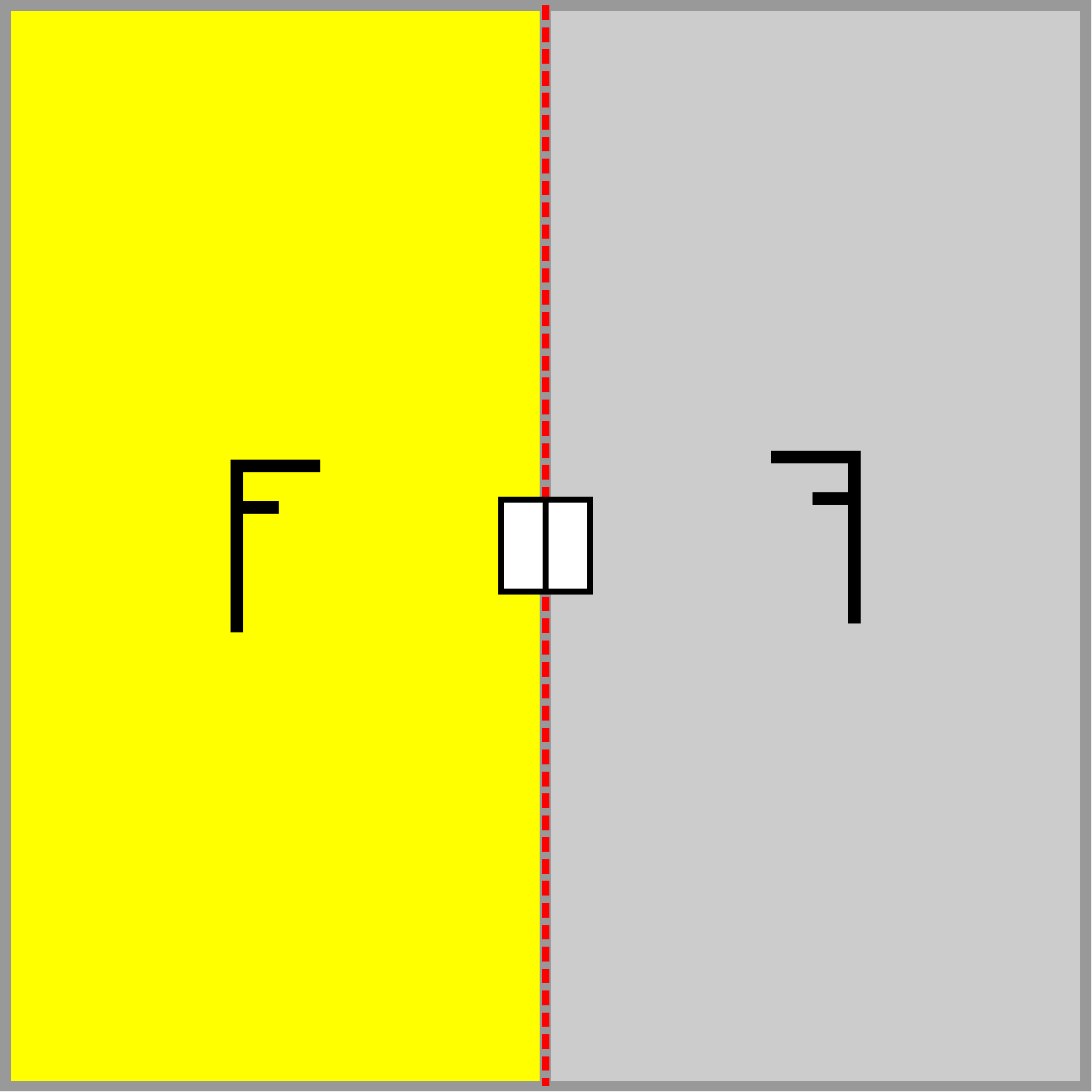} \\
    $\grp |$
  \end{minipage}%
  \begin{minipage}{0.20\textwidth}
    \centering
    \includegraphics[scale=.16]{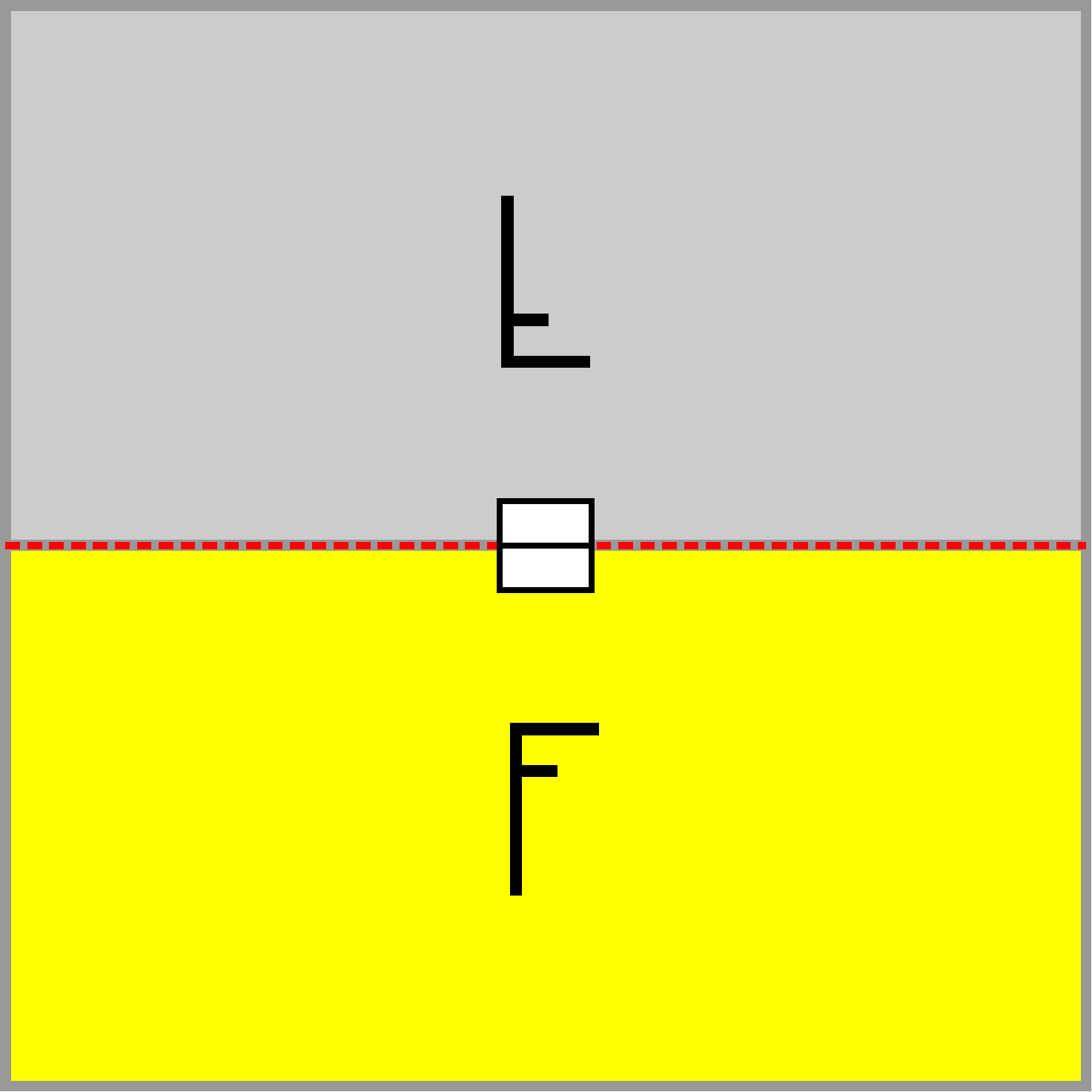} \\
    $\grp -$
  \end{minipage}%
  \begin{minipage}{0.20\textwidth}
    \centering
    \includegraphics[scale=.16]{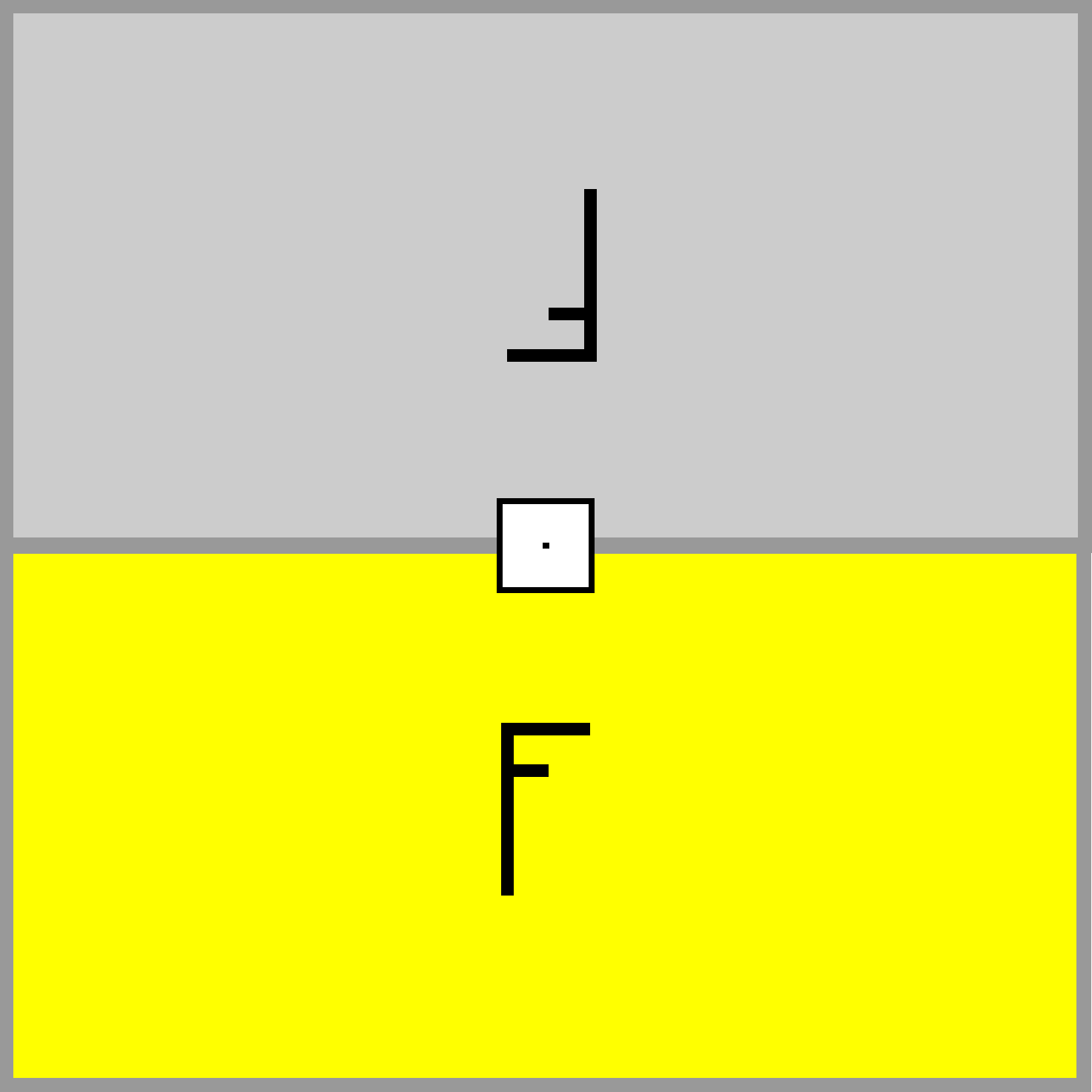} \\
    $\grp .$
  \end{minipage}%
  \begin{minipage}{0.20\textwidth}
    \centering
    \includegraphics[scale=.16]{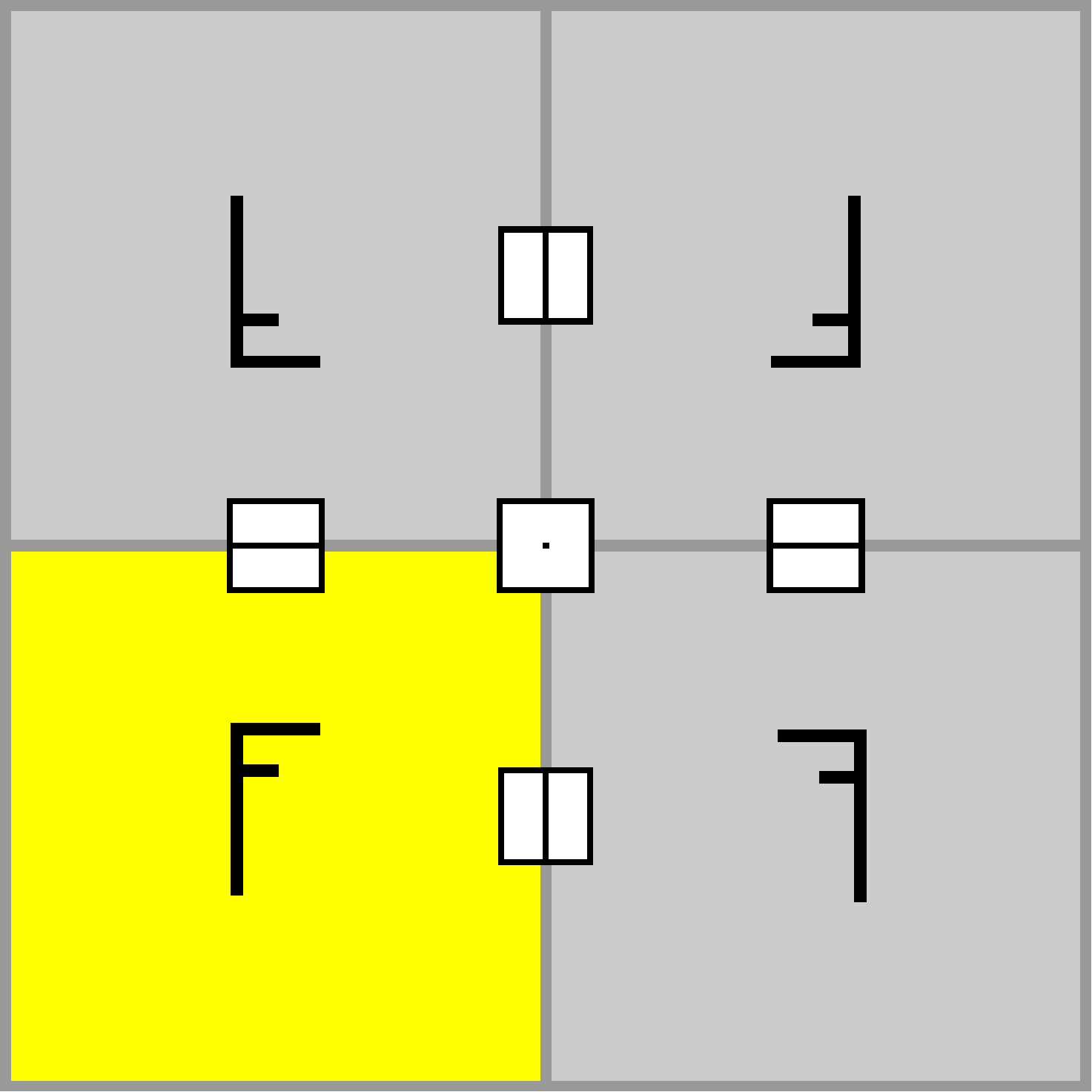} \\
    $\grp +$
  \end{minipage}
    \vspace{.5cm}
  
  \begin{minipage}{0.20\textwidth}
    \centering
    \includegraphics[scale=.16]{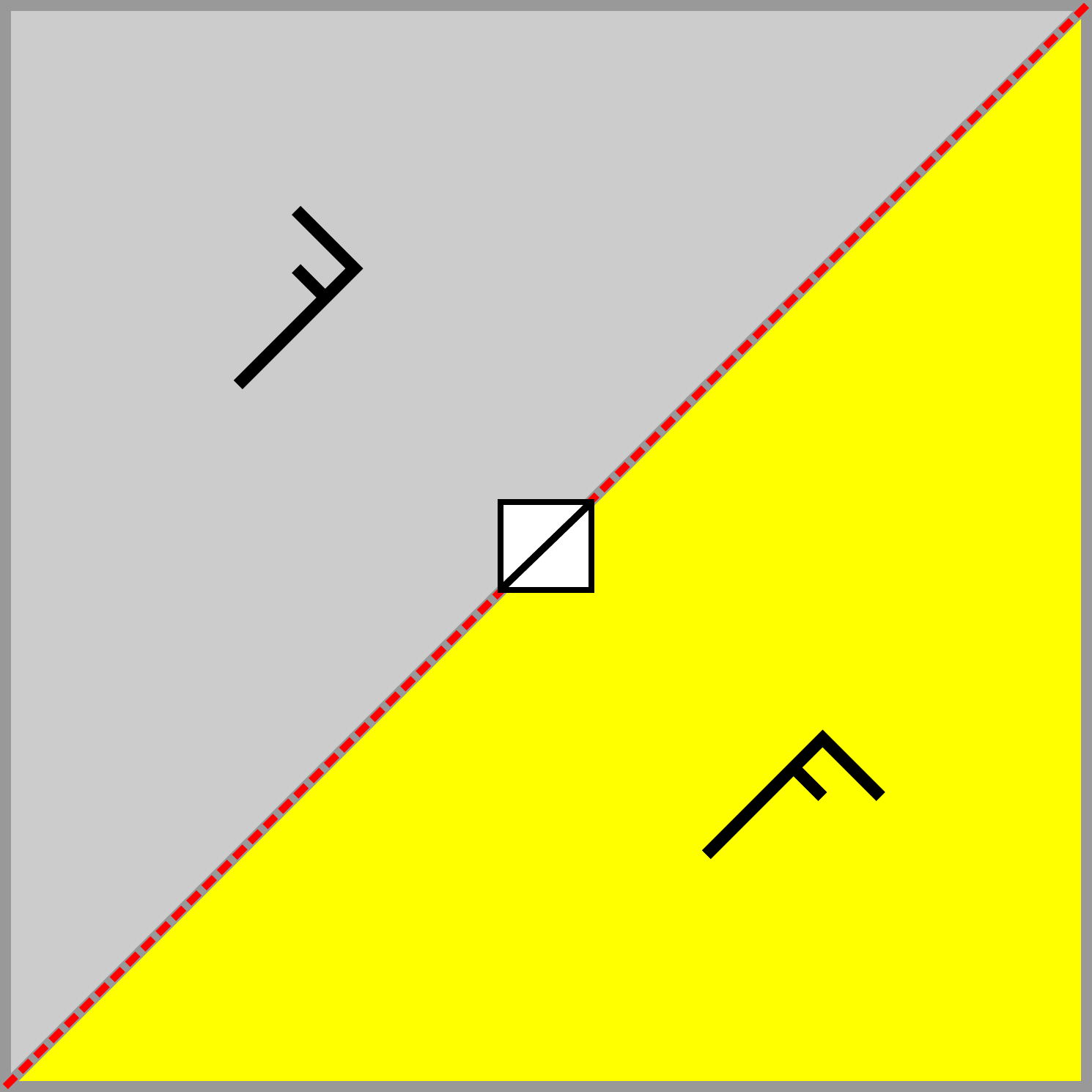} \\
    $\grp /$
  \end{minipage}%
  \begin{minipage}{0.20\textwidth}
    \centering
    \includegraphics[scale=.16]{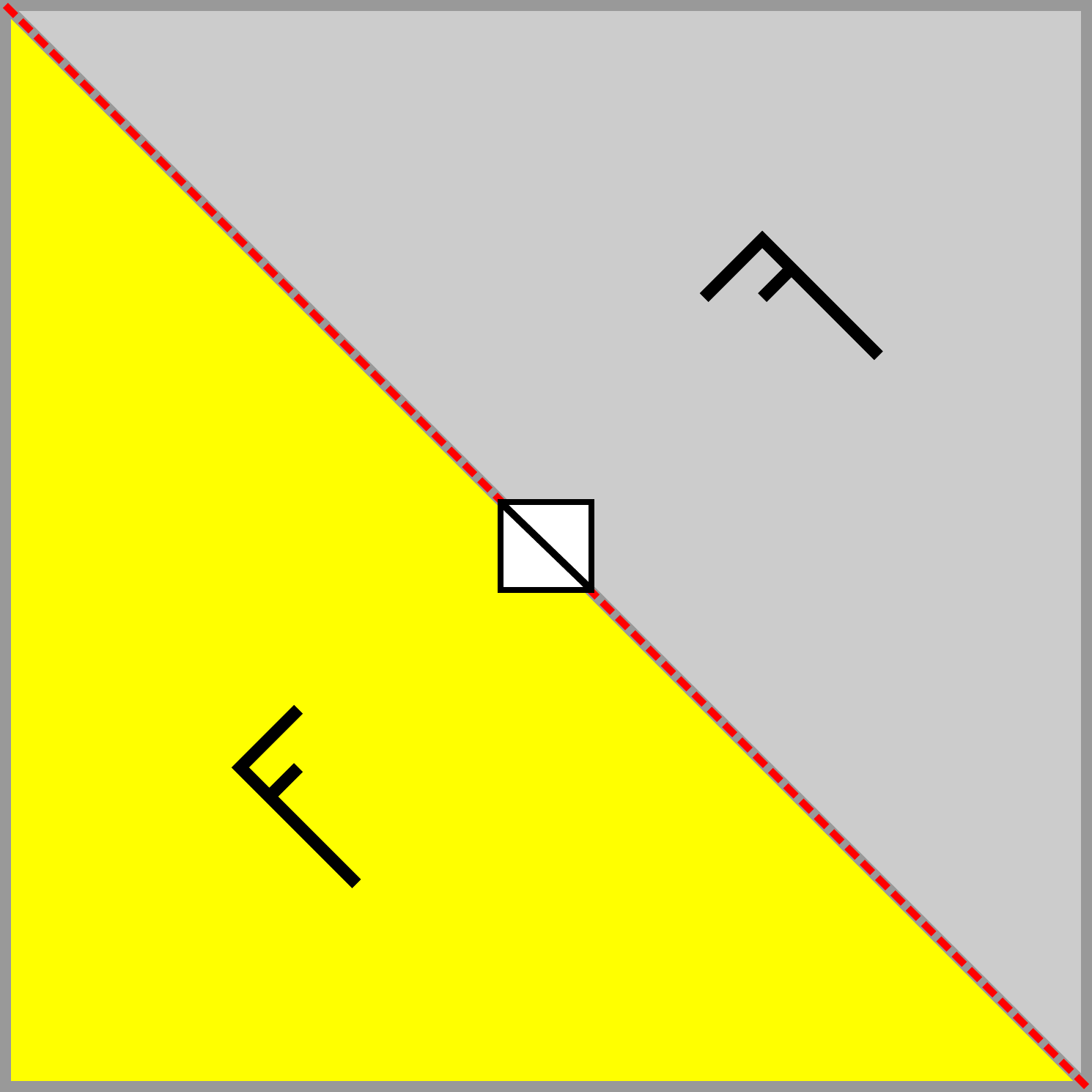} \\
    $\grp \setminus$
  \end{minipage}%
  \begin{minipage}{0.20\textwidth}
    \centering
    \includegraphics[scale=.16]{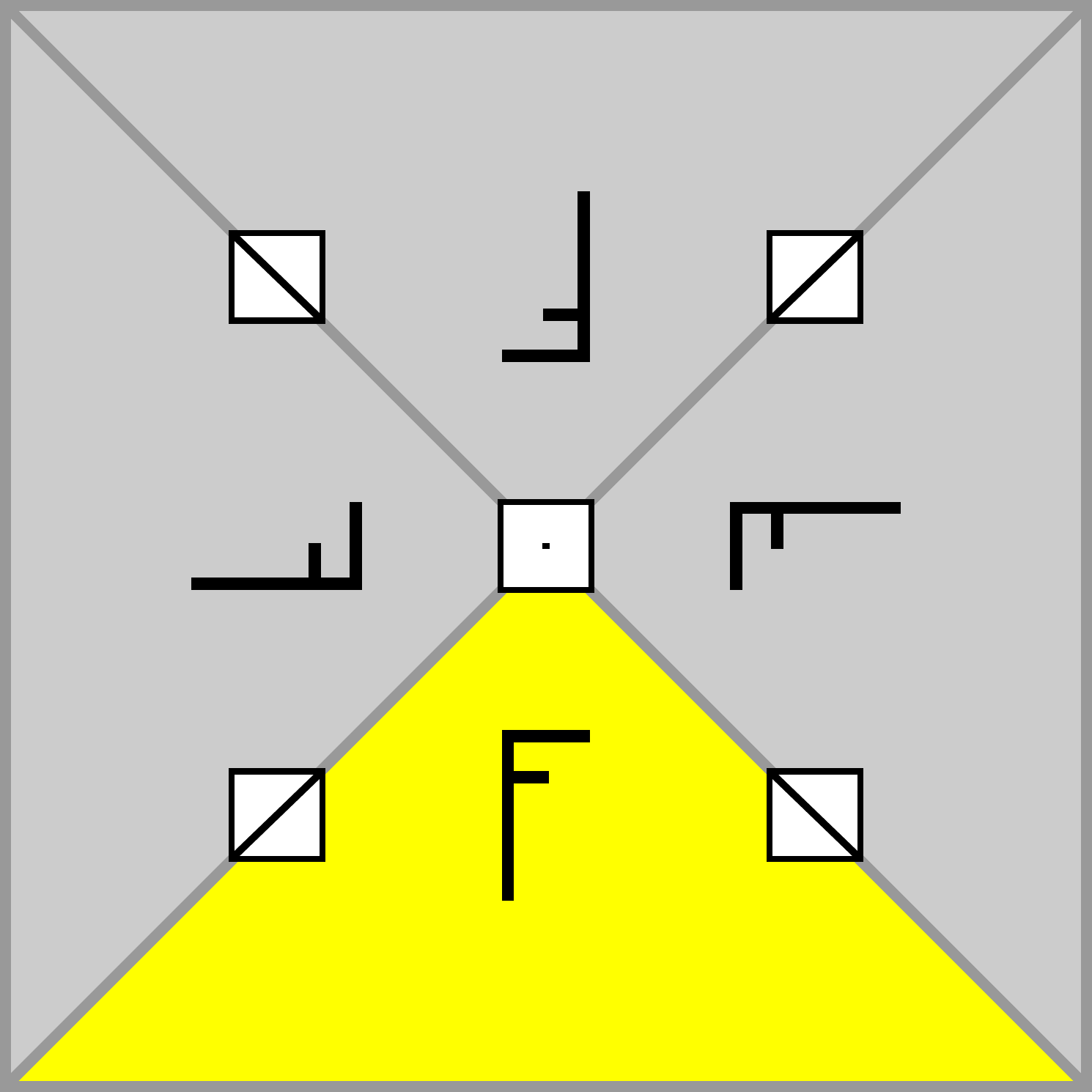} \\
    $\grp X$
  \end{minipage}%
  \begin{minipage}{0.20\textwidth}
    \centering
    \includegraphics[scale=.16]{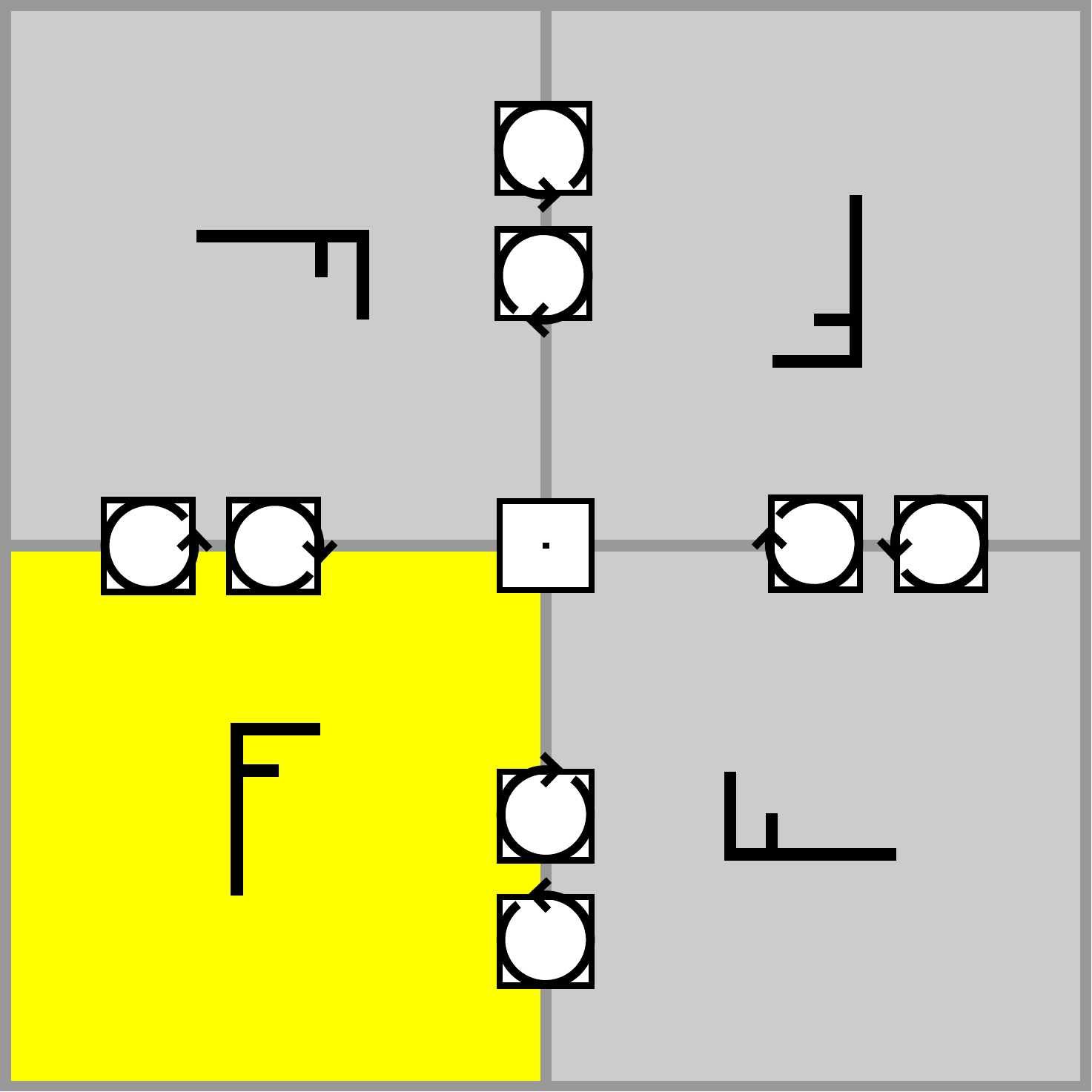} \\
    $\grp L$
  \end{minipage}%
  \begin{minipage}{0.20\textwidth}
    \centering
    \includegraphics[scale=.16]{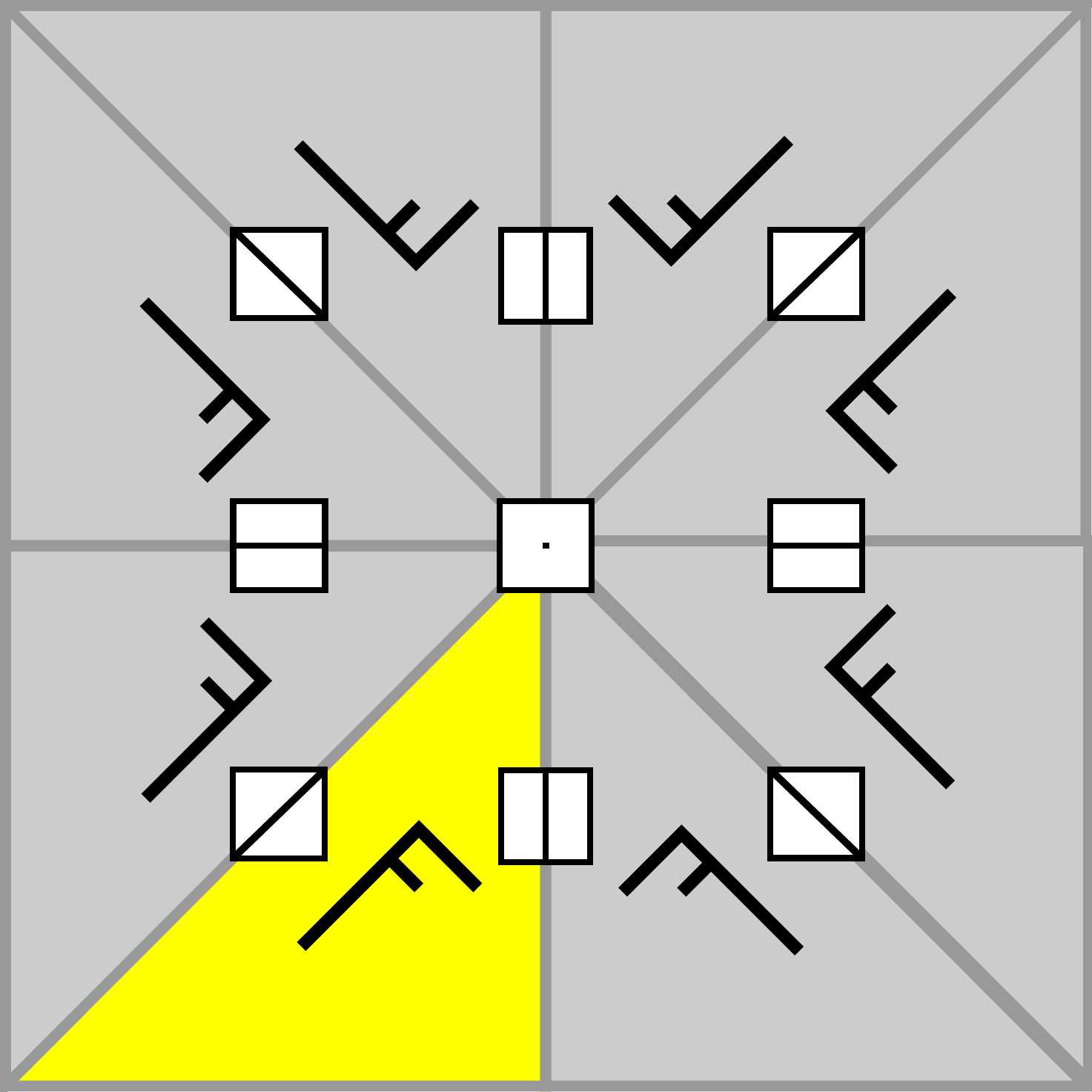} \\
    $\grp *$
  \end{minipage}
    \caption{The 10 subgroups of $D_8^{\mathbb{T}}$. See Table \ref{tab:subgroups}.}
  \label{fig:subgroups}
\end{figure}

The following lemma is useful in order to restrict the translational
subgroup for a given directional group.

\begin{lemma}\label{closed}
  For a group $G$ of torus symmetries,
  the translational subgroup $G_\Box$ is closed under every symmetry
  in the directional group of $G$.
\end{lemma}
\begin{proof}
  Assume that $t\in G_\Box$, and we have an operation in
  $G/G_\Box$ that is represented by an orthogonal $2\times 2$ matrix
  $A$.
This means that $G$ contains some transformation $x\mapsto
Ax+b$. If we conjugate the translation $x\mapsto x+t$ with this
transformation, we get
$x\mapsto A(A^{-1}(x-b)+t)+b=x+At$, i.e., a translation by $At$.
\end{proof}

 \subsection{Overview of the toroidal groups}
\label{sec:toroidal-overview}

After fixing the directional group, we have to look at the
translational subgroup, and the interaction between the two.
The result is %
summarized as follows.

\begin{proposition}\label{classification-toroidal}
  The 4-dimensional point groups that have an invariant torus can be classified into 25 infinite families of toroidal groups, among them
    \begin{itemize}
    \item 2 three-parameter families
    \item 19 two-parameter families
    \item 4 one-parameter families
    \end{itemize}
    as shown in  Table~\ref{tab:overview}.
\end{proposition}
The last column of  Table~\ref{tab:overview} shows the names of these groups in
the classification of Conway and Smith.\footnote
{\label{fn-mirror}%
To get a closer correspondence with our parameterization
for the groups of type
\grp1 and \grp. in the first two rows, we
swap the role of the left and right factors in
the generators given in Conway and Smith.
Effectively, we consider the mirror groups.
Accordingly,
we have adapted the Conway--Smith convention of writing $\frac1f[C_m\times
C_n^{(s)}]$, by decorating the \emph{left}
factor with the parameter $s$.
More details are given in Appendix~\ref{conway-smith}.}
We make a comparison in
Section~\ref{sec:compare-CS}.

There is one difficulty that we have not addressed: We look at the groups
that leave one particular Clifford torus invariant. However, there are
some groups, in particular small groups, that have several
invariant Clifford tori. This leads to ambiguities. For example, a torus translation by
$180^\circ$ on one torus may appear as a swapturn $\sym/$ on a
different torus. We %
investigate these cases %
in detail in Section~\ref{dupli}.

The natural constraint on the parameters $m$ and $n$ is
$m,n\ge 1$ in all cases
 of  Table~\ref{tab:overview}, in the sense that all these choices
(in a few cases under the additional constraint that $m\equiv n\pmod 2$)
lead to valid groups. (But note that some extra evenness constraints are
already built into the notation, for example, when we write
$\grp\setminus_{2m,2n}^{\textbf{pm}}$ instead of
$\grp\setminus_{m,n}^{\textbf{pm}}$.)
For the swapturn
groups
$\grp L_{a,b}$, the natural choices are $a,b\ge 0$ except for $(a,b)=(0,0)$.
The stricter conditions on $m$ and $n$ in Table~\ref{tab:overview}
are imposed in order to
exclude duplications.

\begin{table}
  \centering
  \setlength{\extrarowheight}{2.1pt}
  \begin{tabular}{|@{\,}l@{\,}|@{\,}c@{\,}|@{ }l@{\,}|@{\,}l@{ }|@{ }l@{\,}|@{ }l@{\,}|}
    \hline
    group& order & parameters &\multicolumn3{l|}{names in Conway--%
                                Smith \cite[Tables 4.1--4.3]{CS}
    }\\
    \hline \multicolumn 6{|@{\,}l|}{torus translation groups (chiral,
    wallpaper group $\mathbf{p1}$)}\\\hline
    \raise2pt\hbox{\strut}%
    $\grp1_{m,n}^{(s)}$  & $mn$
                 &
                   $m,n\ge1$,
                   $- \frac m2  \le s \le \frac{n-m}2$
 &\multicolumn3{l|}{
$m\equiv 0 \pmod 2$:
 $\pm\frac1{f}[C_{m\!f/2}^{(s')}\times C_{n}]$
}\\                                              
    &&%
 &\multicolumn3{l|}{
$m\equiv 1 \pmod 2$:
$+\frac1{f}[C_{m\!f}^{(s')}\times C_{n}]$}\\
    \hline \multicolumn 6{|@{\,}l|}{torus flip groups (chiral,
        wallpaper group $\mathbf{p2}$)}\\\hline
    \raise2pt\hbox{\strut}%
    $\grp._{m,n}^{(s)}$ & $2mn$ &
 $m,n\ge1$,
 $- \frac m2  \le s \le \frac{n-m}2$
                              &\multicolumn3{l|}{
$m\equiv 0 \pmod 2$:
    $\pm\frac{1}{2f}[D_{m\!f}^{(f-s')}\times D_{2n}]$}\\
         &&\quad %
$(m,n)\ne(1,1),(2,1)$
                 &\multicolumn3{l|}{%
$m\equiv 1 \pmod 2$:
    $+\frac{1}{2f}[D_{2m\!f}^{(2f-s')}\times D_{2n}]$}\\
    \hline \multicolumn 3{|@{\,}l|@{\,}}{torus swap groups (chiral)}
        & \multicolumn2{l|@{ }}{$n$ even}&$n$ odd\\\hline
    $\grp\setminus_{2m,2n}^{\textbf{pm}}$ & $4mn$ &$m,n\ge2$&\multicolumn2{l|@{ }}{ $\pm[D_{2m}\times C_{n}]$}&
$\pm\frac{1}{2}[D_{2m}\times C_{2n}]$\\
    $\grp\setminus_{2m,2n}^{\textbf{pg}}$ & $4mn$
                 &$m\ge2,n\ge1$&\multicolumn2{l|@{ }}{$\pm\frac{1}{2}[D_{2m}\times
                             C_{2n}]$}&$\pm[D_{2m}\times C_{n}]$ \\
    $\grp\setminus_{m,n}^{\textbf{cm}}$ & $2mn$ &
$m\ge3,n\ge2$, $m-n$ even&\multicolumn2{l|@{ }}{
 $\pm\frac{1}{2}[\overline{D}_{2m} \times C_{n}]$}%
    &    \raise-1pt\hbox{\strut}%
      $+\frac{1}{2}[D_{2m}\times C_{2n}]$\\
    \hline
    \raise1pt\hbox{\strut}%
    $\grp/_{2m,2n}^{\textbf{pm}}$
         & $4mn$ &$m,n\ge2$&\multicolumn3{l|}{
        mirrors of the groups $\grp\setminus_{2n,2m}^{\textbf{pm}}$}\\
    $\grp/_{2m,2n}^{\textbf{pg}}$ & $4mn$ &$m\ge1,n\ge2$&\multicolumn3{l|}{
        mirrors of the groups $\grp\setminus_{2n,2m}^{\textbf{pg}}$}\\
    $\grp/_{m,n}^{\textbf{cm}}$ & $2mn$ &$m\ge2,n\ge3$, $m-n$ even&\multicolumn3{l|}{        mirrors of the groups $\grp\setminus_{n,m}^{\textbf{cm}}$}\\   
    \hline \multicolumn 3{|@{\,}l|@{\,}}{full torus swap groups (chiral)}&$m
  $&$n$ even&$n$ odd\\\hline
    \raise1pt\hbox{\strut}%
    $\grp X_{2m,2n}^{\textbf{p2mm}}$ & $8mn$ &$m,n\ge2$
      &$m$ even& $\pm[D_{2m} \times D_{2n}]$
&
             $\pm\frac12[D_{2m} \times \overline D_{4n}]$\\
    &&&$m$ odd& $\pm\frac12[\overline D_{4m} \times D_{2n}]$
   & $\pm\frac14[D_{4m} \times \overline D_{4n}]$\\
$\grp X_{2m,2n}^{\textbf{p2mg}
    }$ & $8mn$ &$m,n\ge2$
     &$m$ even& $\pm\frac12[D_{2m} \times \overline D_{4n}]$
 & $\pm[D_{2m} \times D_{2n}]$\\
    &&&$m$ odd& $\pm\frac14[D_{4m} \times \overline D_{4n}]$
 & $\pm\frac12[\overline D_{4m} \times D_{2n}]$\\
$\grp X_{2m,2n}^{\textbf{p2gm}
    }$ & $8mn$ &$m,n\ge2$
      &$m$ even& $\pm\frac12[\overline D_{4m} \times D_{2n}]$
 & $\pm\frac14[D_{4m} \times \overline D_{4n}]$\\
    &&&$m$ odd& $\pm[D_{2m} \times D_{2n}]$
 & $\pm\frac12[D_{2m} \times \overline D_{4n}]$\\

$\grp X_{2m,2n}^{\textbf{p2gg}}$ & $8mn$ &$m,n\ge2$
      &$m$ even& $\pm\frac14[D_{4m} \times \overline D_{4n}]$
& $\pm\frac12[\overline D_{4m} \times D_{2n}]$\\
    &&&$m$ odd& $\pm\frac12[D_{2m} \times \overline D_{4n}]$
& $\pm[D_{2m} \times D_{2n}]$\\[0.5ex]
    $\grp X_{m,n}^{\textbf{c2mm}}$ & $4mn$ &$m,n\ge3$, $m-n$ even
  &$m\equiv n$&%
  $\pm\frac12[\overline{D}_{2m} \times \overline D_{2n}]$
    &    \raise-1pt\hbox{\strut}%
$+\frac14[D_{4m} \times \overline D_{4n}]$ \\
    \hline \multicolumn 6{|@{\,}l|}{torus reflection groups (achiral)}\\\hline
    \raise1pt\hbox{\strut}%
    $\grp|_{m,n}^{\textbf{pm}}$ & $2mn$ &$m,n\ge1$&
        \multicolumn3{@{ }l|}{
        $\left\{\vbox to 3,8ex{}\right.%
   \begin{aligned}
&\textstyle
{+}\text{ or }
{\pm}\frac{1}{f}[C_{n'\!f} \times C_{n'\!f}^{(s)}] \cdot 2^{(0)}\text{ or }\\
&\textstyle
{+}\frac{1}{f}[C_{n'\!f} \times C_{n'\!f}^{(s)}] \cdot 2^{(2)}
\end{aligned}%
$
}\\
    $\grp|_{m,n}^{\textbf{pg}}$ & $2mn$ &$m,n\ge1$&\multicolumn3{@{ }l|}{%
        $\left\{\vbox to 3,8ex{}\right.%
   \begin{aligned}
&\textstyle
{+}\text{ or }
{\pm}\frac{1}{f}[C_{n'\!f} \times C_{n'\!f}^{(s)}] \cdot 2^{(1)}\text{ or }\\
&\textstyle
{+}\frac{1}{f}[C_{n'\!f} \times C_{n'\!f}^{(s)}] \cdot 2^{(0)}
\end{aligned}$
}\\
    $\grp|_{m,n}^{\textbf{cm}}$ & $4mn$ &$m,n\ge1$
      &\multicolumn3{@{\,}l@{\,}|}{
    \raise-2pt\hbox{\strut}%
$\textstyle{\pm}
\frac{1}{f}[C_{n'\!f} \times C_{n'\!f}^{(s)}] \cdot 2^{(0)}
\text{ or }{+}\frac{1}{f}[C_{n'\!f} \times C_{n'\!f}^{(s)}] \cdot 2^{(0)}$}\\
    \hline \multicolumn 6{|@{\,}l|}{full torus reflection groups (achiral)}\\\hline
    \raise1pt\hbox{\strut}%
    $\grp+_{m,n}^{\textbf{p2mm}}$ & $4mn$ &$m \geq n \geq 1$, $(m,n) \not= (1,1)$&\multicolumn3{l|}{}\\
    $\grp+_{m,n}^{\textbf{p2mg}}$ & $4mn$ &$m,n\ge1$, $(m,n) \not= (1,1)$&
\multicolumn3{l|}{
\smash{%
\hbox{$\left.\vbox to 5ex{}\right\}
\begin{aligned}
&\textstyle\pm \frac{1}{2f}[D_{2n'f} \times D_{2n'f}^{(s)}]\cdot
2^{(\alpha,\beta)}
\text{ or}%
\\&\textstyle+ \frac{1}{2f}[D_{2n'f} \times D_{2n'f}^{(s)}]\cdot 2^{(\alpha,\beta)}
\end{aligned}
    $}}}\\
    $\grp+_{m,n}^{\textbf{p2gg}}$ & $4mn$ &$m \geq n \geq 1$, $(m,n) \not= (1,1)$&
 \multicolumn3{l|}{}\\
    $\grp+_{m,n}^{\textbf{c2mm}}$ & $8mn$ &$m \geq n \geq 1$, $(m,n) \not= (1,1)$
&
 \multicolumn3{l|}{
    \raise-2pt\hbox{\strut}%
$\textstyle\pm\text{ or }{+\frac{1}{2f}}[D_{2n'f} \times D_{2n'f}^{(s)}]\cdot 2^{(0,0)}$
}\\
    \hline \multicolumn 6{|@{\,}l|}{torus swapturn groups (achiral,
        wallpaper group $\mathbf{p4}$)}\\\hline
    \raise1pt\hbox{\strut}%
    $\grp L_{a,b}$ & $4(a^2{+}b^2)$ &
       $a \ge b \ge 0$%
     &\multicolumn3{l|}{
       $a\equiv b\pmod 2$:
      $\pm\frac{1}{2f}[D_{2nf}\times D_{2nf}^{(s)}]\cdot \overline{2}$
}\\
    && $a\ge2$, $(a,b) \not= (2,0)$ &%
    \multicolumn3{l|}{ $a\not\equiv b\pmod 2$:
    \raise-2pt\hbox{\strut}%
    $+\frac{1}{2f}[D_{2nf}\times D_{2nf}^{(s)}]\cdot \overline{2}$}
    \\
    \hline \multicolumn 3{|@{\,}l|@{\,}}{full torus groups (achiral)}
    & \multicolumn2{l|@{ }}{$n$ even}&$n$ odd
\\\hline
    \raise1pt\hbox{\strut}%
    $\grp*_{n}^{\textbf{p4mm}\mathrm{U}}$ & $8n^2$ &$n\ge3$&\multicolumn2{l|@{ }}{$\pm\frac{1}{2}[\overline{D}_{2n} \times \overline{D}_{2n}] \cdot 2$}%
& $+\frac{1}{4}[D_{4n} \times \overline{D}_{4n}]\cdot 2_1$\\
    $\grp*_{n}^{\textbf{p4gm}\mathrm{U}}$ & $8n^2$&$n\ge3$&\multicolumn2{l|@{ }}{$\pm\frac{1}{2}[\overline{D}_{2n} \times \overline{D}_{2n}] \cdot \overline{2}$}%
&$+\frac{1}{4}[D_{4n} \times \overline{D}_{4n}]\cdot 2_3$\\
    $\grp*_{n}^{\textbf{p4mm}\mathrm{S}}$ & $16n^2$ &$n\ge2$&\multicolumn2{l|@{ }}{$\pm[D_{2n}\times D_{2n}]\cdot 2$}%
& $\pm\frac{1}{4}[D_{4n} \times \overline{D}_{4n}]\cdot 2$\\
    $\grp*_{n}^{\mathbf{p4gm}\textrm{S}}$ & $16n^2$ &$n\ge2$
& \multicolumn2{l|@{ }}{    \raise-1pt\hbox{\strut}%
$\pm\frac{1}{4}[D_{4n} \times \overline{D}_{4n}]\cdot 2$}%
& $\pm[D_{2n}\times D_{2n}]\cdot 2$\\
    \hline
  \end{tabular}
  \caption[Overview of the 25 classes of toroidal groups]{Overview of the toroidal groups. In the Conway--Smith names, we write
    $n'$ and $s'$ when these parameters don't directly
    correspond to our parameters $n,s$.}
  \label{tab:overview}
\end{table}

We will now go through the categories one by one.
This closely parallels the classification of the wallpaper groups.
When appropriate, we use the established notations for wallpaper groups to distinguish
the torus groups.
We have to choose suitable parameters for the different dimensions of each
 wallpaper group,
and in some cases,
we have to refine the classification of wallpaper 
groups because
different axis directions are distinguished.

\subsection{The torus translation groups\texorpdfstring{, type \grp1}{}}
\label{sec:translations}

\begin{figure}
  \centering
  \includegraphics[scale=1]{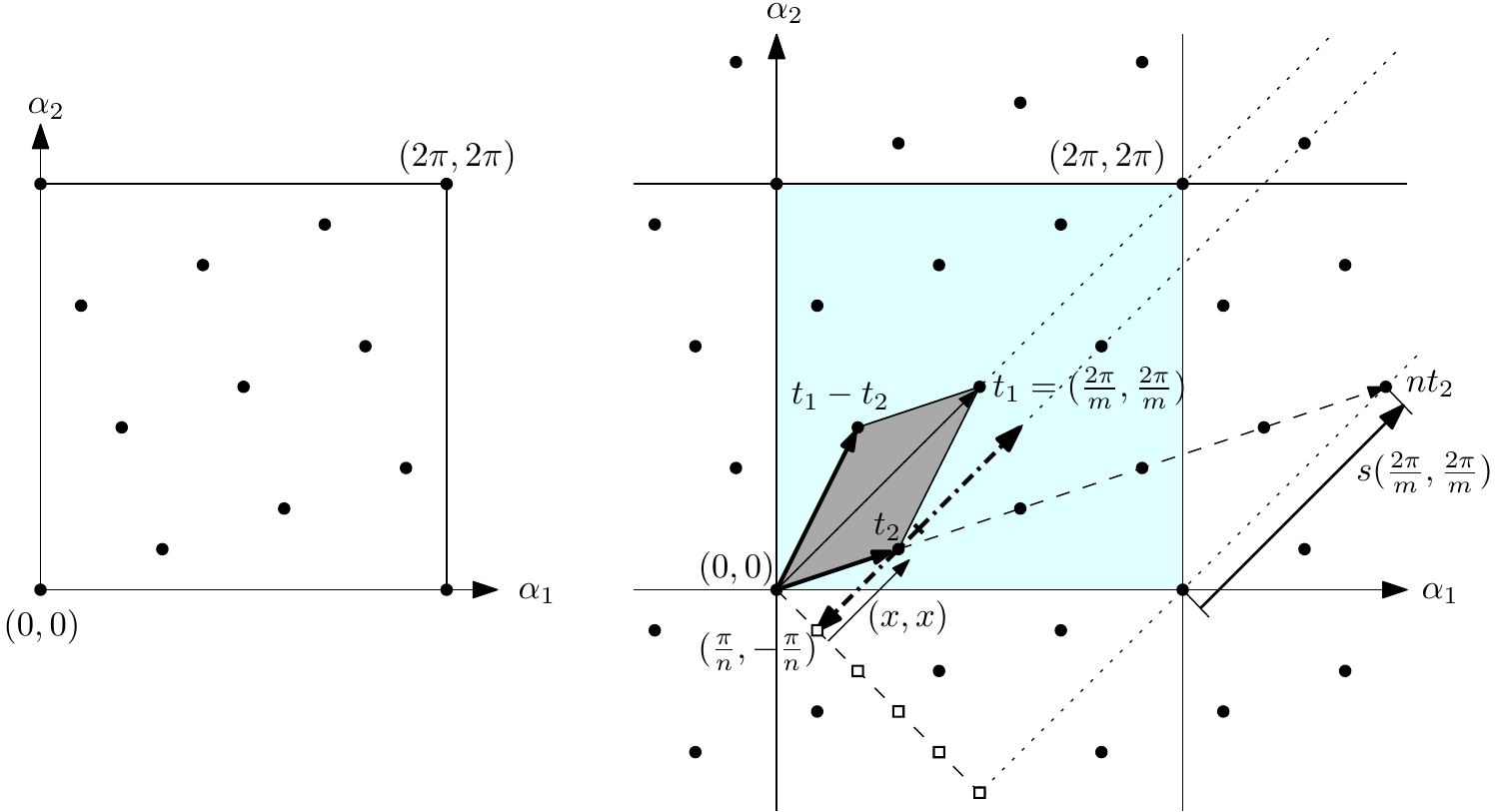}
  \caption{A lattice of torus translations.
    In the right part, we see that it is given by the parameters $m=2$, $n=5$, and
    $s=1$. The vectors $t_1=(\pi,\pi)$ and
    $t_2=(\frac\pi5,-\frac\pi5)+
    (\frac{2\pi}{5},\frac{2\pi}{5})
    =(\frac{3\pi}{5},\frac\pi{5})
    $ generate
    the group $\grp1_{2,5}^{(1)}$. This lattice happens to be
    a square lattice, but this %
    plays no role.}
  \label{fig:translations}
\end{figure}

These are the groups that contain only torus translations.
The pure translation groups are the
simplest class, but they are also the richest type of groups, requiring
three parameters for their description.
The translations $(\alpha_1,\alpha_2)$ with
$R_{\alpha_1,\alpha_2}\in G$
form an additive group
modulo
$(2\pi,2\pi)$, and hence a lattice
modulo
$(2\pi,2\pi)$.
In accordance with Theorem~\ref{theorem-torus-symmetries} %
we can also view it as a lattice in the plane that contains all
points whose coordinates are multiples of $2\pi$,
see Figure~\ref{fig:translations}.

We parameterize these lattices with three parameters $m,n,s$:
The lattice subdivides the principal diagonal from $(0,0)$ to $(2\pi,2\pi)$
into some number $m\ge1$ of segments.
Then we choose $t_1=%
({\frac{2\pi}m,\frac{2\pi}m})$ as the first
generator of the lattice.
The second parameter $n\ge1$ is the number of lattice lines
parallel to the principal diagonal that run between $(0,0)$ and $(2\pi,0)$,
including the last one through~$(2\pi,0)$.
In the figure, we have $m=2$ and $n=5$.
On each such line, the points are equidistant with distance
$\frac{2\pi}m\cdot \sqrt2$.
The first parallel lattice line thus contains a unique point
$t_2=(\frac\pi n,-\frac\pi n)+(x,x)$ with $0\le x <\frac{2\pi}m$,
and we choose $x$ as the third parameter. The range from
which $t_2$ can be chosen is indicated by a double arrow in the figure.

We still have to take into account the ambiguity from the choice of the coordinate
system (Section~\ref{sec:coordinate-system}).
The choice of origin is no problem, since a translation does not
depend on the origin.
Also,
the ``flip'' ambiguity from \sym. is no problem at all:
Rotating the coordinate system by $180^\circ$ maps the lattice to itself.
The ``swap'' ambiguity from \sym/, however, is more serious, as
it exchanges the coordinate axes:
$\alpha_1\leftrightarrow\alpha_2$.
(From \sym\setminus, we get no extra ambiguity, since
$\sym\setminus = \sym. \cdot \sym/$.)

To eliminate this  ambiguity, we look at
the vectors $t_1-t_2$ and $t_2$. They form also a lattice basis, and they
span a parallelogram whose diagonal $t_1$ lies on the
$\alpha_1=\alpha_2$ axis. %
The alternate choice of the basis will
reflect the parallelogram at this diagonal.
Thus, the choices $x$ and $\frac{2\pi}m-x$ will lead to the same group.
We can achieve a
unique representative by stipulating that $t_2$ is not longer than
$t_1-t_2$.
This means that we restrict $t_2$ to the lower half of the
range,
including the midpoint, which is marked in the figure:
$0\le x \le\frac{\pi}m$.\footnote
{This easy way of dealing with the duplications caused by \sym/ is the
  reason for preferring the
  oblique axes of
  Figure~\ref{fig:translations} for measuring the parameters $m$ and
  $n$
  over %
  the more natural
  $\alpha_1,\alpha_2$-axes.
  This oblique system is also aligned with the specification of the group
  by its
  left and right group (of left translations and right translations)
  that underlies the classic classification, see Appendix~\ref{conway-smith}.
  Curiously, %
  \label{fn-CS-escaped}
  these duplications
 caused by \sym/
  were overlooked by Conway and Smith~\cite{CS},
  although they had escaped %
  {none} of the previous classifications
\cite[p.~62, groupe~I]{goursat-1889},
\cite[p.~20, item \S1, formula~(2)]{ThrS-I},
 \cite[p.~55, first paragraph]{duval}.
}

Finally, we look at the point $nt_2$, which lies on the $45^\circ$
line through~$(2\pi,0)$.  We have to ensure that it is one of the
existing lattice points on this line because additional points would contradict
the choice of~$m$.  Thus
$$nt_2=(\pi,-\pi)+(nx,nx)=(2\pi,0)+s(\tfrac{2\pi}m,\tfrac{2\pi}m)$$ for
some integer $s$, or in other words
$$
x = \frac\pi n + s\cdot \frac {2\pi}{mn}
$$
Combining this with the constraint $0\le x \le\frac{\pi}m$,
we get
\begin{equation}
  \label{eq:ks}
  -\frac m2  \le s \le %
  -\frac{m}2+\frac{n}2
\end{equation}
This range contains
$\lceil \frac {n}2\rceil$ integers if $m$ is odd and
$\lceil \frac {n+1}2\rceil$ integers if $m$ is even.
In particular, there is always
at least one possible value~$s$.

\begin{proposition}
\label{prop:torus_translations}
  The point groups that contain only torus translations
  can be classified as follows:

    For any integers $m,n\ge 1$ and
    any integer~$s$ in the range~\eqref{eq:ks}, there is one such
    group, the torus translation group
    $\grp1_{m,n}^{(s)}$,
of order $mn$.
It is generated by
  $R_{\frac{2\pi}m,\frac{2\pi}m}$ 
and
  $R_{\frac{2\pi}n+\frac{2s\pi}{mn},\frac{2s\pi}{mn}}$.
\end{proposition}
In terms of quaternions, these generators %
are
$
[\exp (-\tfrac{2\pi}{m}i), 1]\text{ and }
[\exp(-\tfrac{(m+2s)\pi}{mn}i), \exp \tfrac{\pi i}{n}].
$
We emphasize that the two parameters $m$ and $n$ play different roles in this
parameterization, and there is no straightforward way to
read off
the parameters of the mirror group from the original parameters
$m,n,s$.  (See for example the entries 11/01 and 11/02 in
Table~\ref{tab:crystallographic}.)

We have observed above that $x$ and
$x'=\frac{2\pi}m-x$ lead to the same group, and the same is true for
$x'=\frac{2\pi}m+x$. In terms of $s$ this means that
the parameters $s'=-m-s$ and $s'=s+n$ lead to the same group as~$s$.
In Section~\ref{dupli}, when we discuss duplications, it will be convenient to allow values $s$ outside the 
range~\eqref{eq:ks}. In particular, it is good to remember that $s=0$
corresponds to a generating point on the $\alpha_1$-axis.

\subsubsection{Dependence on the starting point}

\begin{proposition}
Any two full-dimensional orbits %
of a toroidal translation group 
are linearly equivalent.
\end{proposition}

\begin{proof}
Let $G$ be a toroidal translation group.
We will show that any full-dimensional $G$-orbit
can be obtained from the $G$-orbit of the point 
$(\frac1{\sqrt2}, 0, \frac1{\sqrt2}, 0)$
by applying an invertible linear transformation.

Let $v\in \R^4$ be a point whose $G$-orbit is full-dimensional.
This is equivalent to requiring that the projections of $v$
to the $x_1,y_1$-plane and to the $x_2,y_2$-plane
are not zero.
We can map $v$ to a point $v'$ of the form $(r_1, 0, r_2, 0)$,
with $r_1 \ne 0$ and $r_2 \ne 0$,
by applying a rotation of the form 
\begin{equation}
R_{\alpha_1,\alpha_2} =
\begin{pmatrix}
    R_{\alpha_1} & 0 \\
    0 & R_{\alpha_2}
\end{pmatrix}.
\label{transformation_1}
\end{equation}
The new point $v'$ can be mapped to the point
$(\frac1{\sqrt2}, 0, \frac1{\sqrt2}, 0)$
by applying a matrix of the form
\begin{equation}
\diag(\lambda_1, \lambda_1, \lambda_2, \lambda_2) =
\begin{pmatrix}
    \lambda_1 & 0 & 0 & 0 \\
    0 & \lambda_1 & 0 & 0 \\
    0 & 0 & \lambda_2 & 0\\
    0 & 0 & 0 & \lambda_2  \\
\end{pmatrix}.
\label{transformation_2}
\end{equation}
Since torus translations commute with the linear
transformations \eqref{transformation_1}~and~\eqref{transformation_2},
we are done.
\end{proof}

Frieder and Ladisch~\cite[Proposition~6.3 and Corollary~8.4]{FL16} proved that the
same conclusion holds 
for any abelian group: All full-dimensional orbits
are linearly equivalent to each other in this case.

\subsection{The torus flip groups\texorpdfstring{, type \grp .}{}}
\label{sec:translations+flip}

These groups are generated by torus translations together with a single torus flip.
Adding the flip operation is completely harmless.
Conjugation with a flip changes $R_{\alpha_1,\alpha_2}$
to
$R_{-\alpha_1,-\alpha_2}$, and therefore does not change the
translation lattice at all.
The order of the group doubles.

If we choose the origin at the center of a $2$-fold rotation induced by
a torus flip, then  $\grp._{m,n}^{(s)}$ is generated by
$$
[\exp(-\tfrac{2\pi i}{m}), 1],
[\exp(-\tfrac{(m+2s)\pi i}{mn}), \exp \tfrac{\pi i}{n}],
[j, j].
$$

\subsection{Groups that contain only one type
  of reflection}
\label{sec:translations+reflect1}

These are the torus reflection groups  $\grp|$ and $\grp-$,
as well as the torus swap groups
$\grp/$ and $\grp\setminus$.
The groups of type $\grp|$ and $\grp-$ are geometrically the same,
because $\sym/$ (or $\sym\setminus$) exchanges vertical mirrors with
horizontal mirrors. Thus, Table~\ref{tab:overview} contains no entries
for~\grp-.
The groups $\grp/$ and $\grp\setminus$ are mirrors, and their
treatment is similar.

If the directional part of a transformation is a reflection (in the
plane), the transformation itself can be either a reflection or a
glide reflection. In both cases there is an invariant line.
We will classify the groups by placing a letter F on the invariant
line
and looking at its orbit.

We need a small lemma that is familiar from the classification of the
wallpaper groups:
\begin{lemma}\label{lattice-with-axis}
  If a two-dimensional lattice has
  an axis of symmetry, then the lattice is either
  \begin{enumerate}
  \item[(1)] a rectangular lattice that is aligned with the axis, or
  \item[(2)] a rhombic lattice, which contains in addition the midpoints of
    the rectangles.
  \end{enumerate}
  In case (1), the symmetry axis goes through a lattice line or
  half-way between two lattice lines.
   In case (2), the symmetry axis goes through a lattice line.
 \end{lemma}
 For an example,
see the upper half of Figure~\ref{fig:vertical-axis}, where the mirror
lines are drawn as solid lines.
\begin{proof}
  Assume without loss of generality that the symmetry axis is the
  $y$-axis. (We may have to translate the lattice so that it no
  longer contains the origin.)
  With every lattice point $(x,y)$,
  the lattice contains also the mirror point
  $(-x,y)$, and thus $(2x,0)$ is a horizontal lattice vector.
  It follows that there must be a lattice point
  $(x_0,y_0)$ with smallest positive $x$-coordinate, since otherwise
  there would be arbitrarily short lattice vectors.
  
  \begin{figure}[htb]
    \centering
    \includegraphics{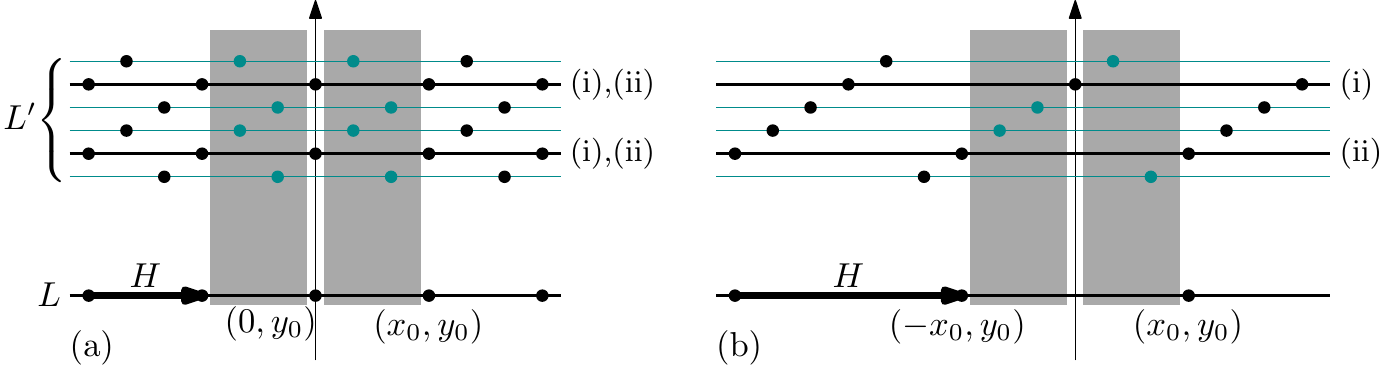}
    \caption{%
      Different
      possibilities for the lattice line $L'$. The gray area is forbidden.}
    \label{fig:lattice-with-axis}
  \end{figure}
  Consider the horizontal lattice line $L$ through $(x_0,y_0)$. 
  There are two cases, see Figure~\ref{fig:lattice-with-axis}.
  (a) $(0,y_0)$ is also a lattice point, and $(H,0)=(x_0,0)$ is a lattice basis vector.
  (b) $(0,y_0)$ is not a lattice point, and $(H,0)=(2x_0,0)$ is a
  lattice basis vector.
  Now look at 
  the next-higher
  horizontal lattice line $L'$ above $L$,
  and choose a lattice point  $(x',y')$ on $L'$.
  $L'$ contains the points $(x'+kH,y')$ for $k\in \mathbb Z$, and
  therefore
  a point $(x,y')$ in the interval $-H/2\le x \le H/2$.
  The value of $x$ cannot be in the range
  $-x_0<x<0$ or
  $0<x<x_0$
  because this would contradict the choice
  of $(x_0,y_0)$. Thus, either (i) $x=0$ or (ii) both points $(\pm
  x_0,y')$ are in the lattice. In case~(a), both possibilities (i) and (ii)
hold simultaneously, and this leads to a rectangular lattice with the axis
  through lattice points. If (b) and (ii) holds,
  we have a rectangular lattice with the axis between lattice lines.
 If (b) and (i) holds, we have a rhombic lattice.
\end{proof}

 \subsubsection{The torus reflection groups\texorpdfstring{, type
     \grp|}{}}
 \label{sec:torus-reflection}

We distinguish two major cases.
 \begin{compactenum}
 \item [M)]The group contains a mirror reflection.
 \item [G)]The group contains only glide reflections.
 \end{compactenum}
In both cases, every \OR\ transformation has a vertical invariant line.
(Actually,
since the translation $\phi_1 \mapsto \phi_1+2\pi$ is always an
element of the group, by Theorem~\ref{theorem-torus-symmetries},
the invariant lines come in pairs $\phi_1=\beta$ and
$\phi_1=\beta+\pi$.)

As announced, we observe the orbit of the letter~F.
We put the bottom endpoint of the F on an invariant line $\ell$.
First we look at the orbit under those transformations that leave
$\ell$ invariant, see the left side of Figure~\ref{fig:vertical-axis}.
In case G, the images with and without reflection alternate along~$\ell$.
In case M, they are mirror images of each other.

\begin{figure}[tbh]
  \centering
  \includegraphics {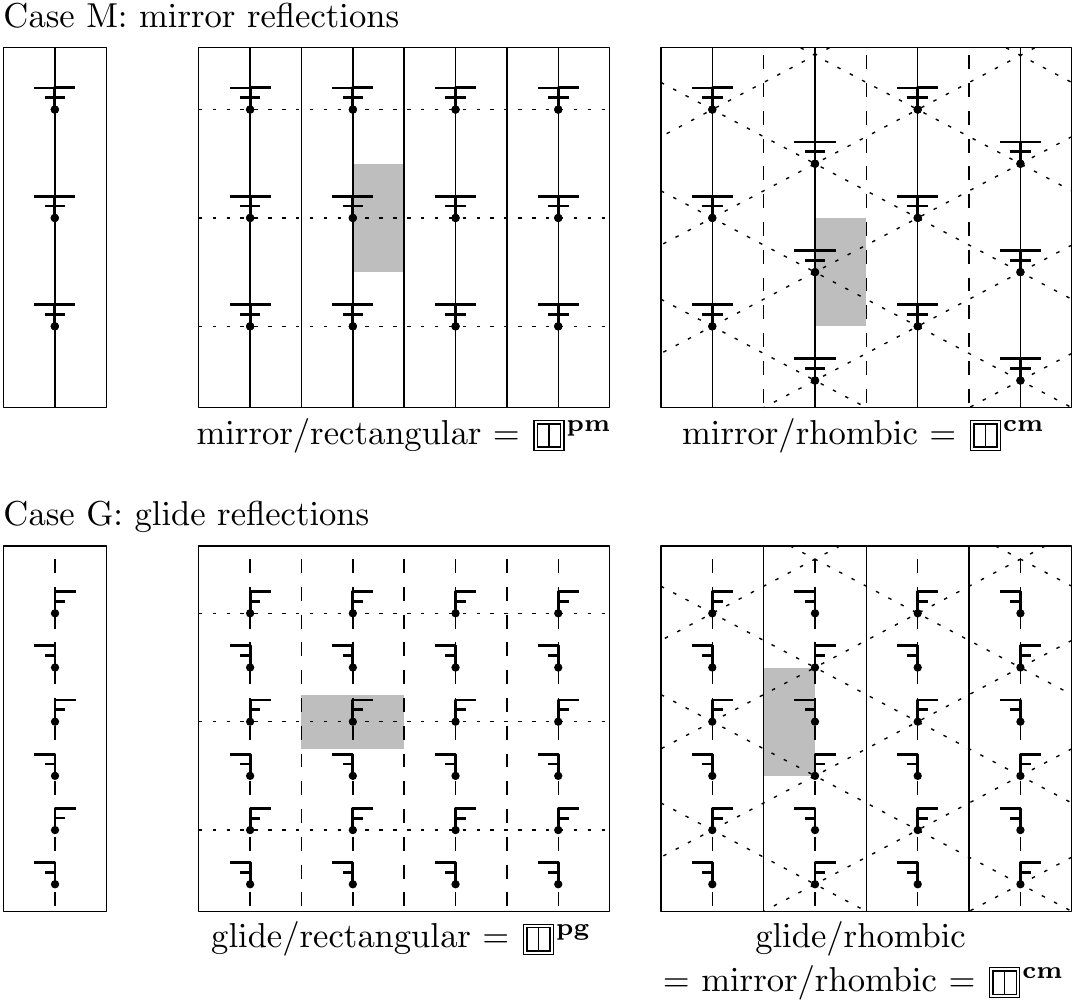}
  \caption{Torus reflection groups, type \grp|.
Combinations of a vertical mirror/glide reflection axis with either a
rectangular or a rhombic grid.
    Invariant lines are shown as solid lines if they act as
    mirrors, otherwise dashed.
    The dotted lines indicate the lattice of translations, and
    the shaded area is a fundamental
    domain.
  }
  \label{fig:vertical-axis}
\end{figure}

In case M, we have a mirror symmetry, and by
Lemma~\ref{closed}, the translational subgroup must be closed under
the mirror symmetry.
 Lemma~\ref{lattice-with-axis} gives the two possibilities of
a rectangular or a rhombic
translational subgroup. Combining these translations with the mirror
operations leads to the two cases
in the top
row
of Figure~\ref{fig:vertical-axis}.

In case G, we cannot apply
Lemma~\ref{lattice-with-axis} right away.
Let $H$ be the vertical distance between consecutive points on the
axis.
If we combine each glide reflection with a vertical translation by
$-H$, we get mirror reflections, as in case M.
To this modified group, we can apply Lemma~\ref{lattice-with-axis},
and we conclude that the translational group must
either form
a rectangular or a rhombic pattern.
Adding back the translation by $H$ to the \OR\ transformations leads to
the
 results in the lower
row
of Figure~\ref{fig:vertical-axis}.
In the rhombic case in the lower right picture
we see that,
when we try to combine  glide reflections with
a rhombic
translational subgroup, we generate mirror symmetries, and thus, this case
really belongs to case M. The picture looks different from the corresponding
picture in the upper row because there are two
alternating types of invariant lines: mirror lines, and lines with a
glide reflection. Depending on where we put the F, we get different pictures.

We are thus left with three cases, which we denote by superscripts
that are chosen
in accordance with the International Notation for these wallpaper groups:
\begin{compactitem}
\item mirror/rectangular: $\grp|^{\textbf{pm}}$,
\item  mirror/rhombic: $\grp|^{\textbf{cm}}$, and
\item glide/rectangular: $\grp|^{\textbf{pg}}$.
\end{compactitem}

The groups are parameterized by two parameters $m\ge1$ and $n\ge1$, the
dimensions of the rectangular grid of translations in the $\phi_1$ and
$\phi_2$ directions, see the left part of
Figure~\ref{fig:grid-parameters}.

Since the invariant lines give a distinguished direction, we need not
worry about duplications when exchanging $m$ and $n$.
The order of each group $G$ is twice the order of the translational subgroup $G_{\Box}$.

\begin{figure}
  \centering
  \includegraphics {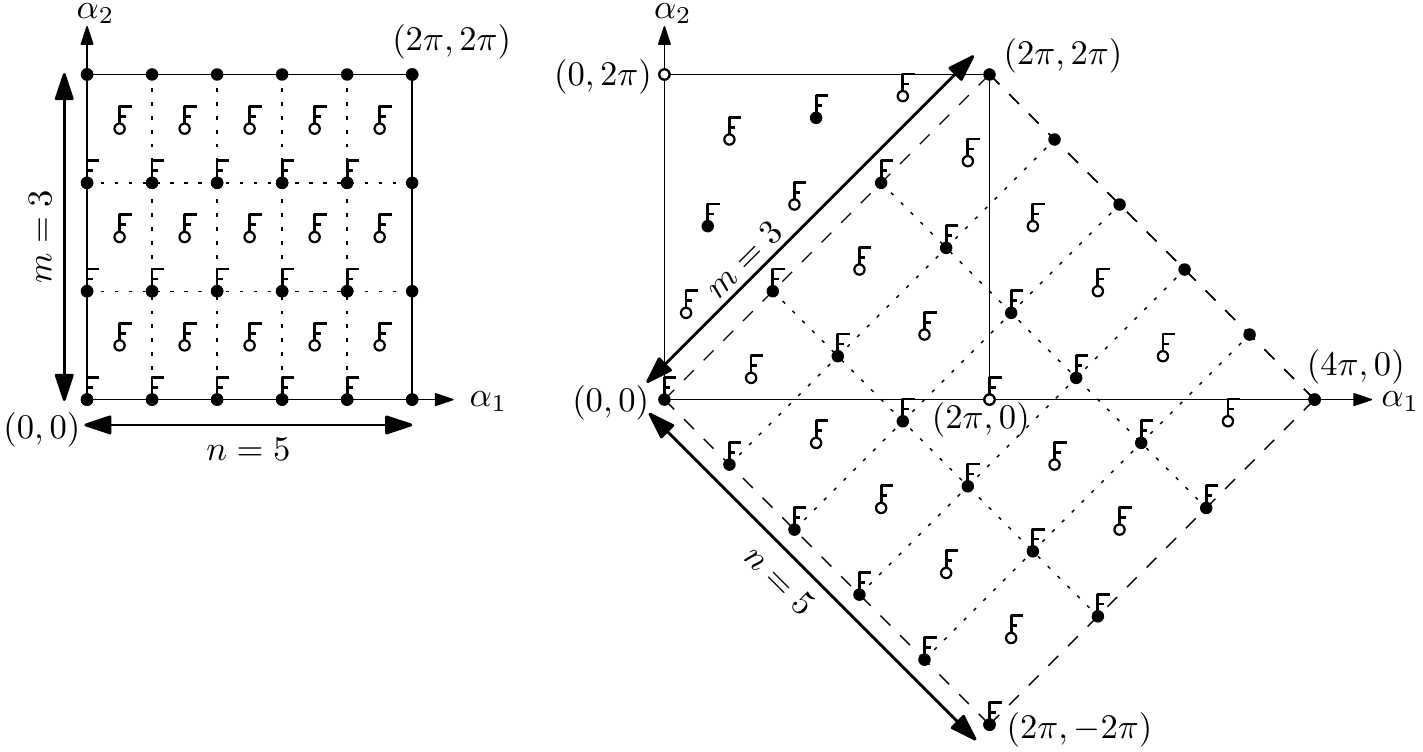}
  \caption{Left: Parameters for the translational subgroup of the groups
    with
    vertical invariant lines, type \grp|.
    We divide the vertical axis into $m$ equal parts and
    the horizontal axis into $n$ equal parts.
      In the rectangular case, the grid consists only of
    the $mn$ black points.
    In the rhombic case, the white points are
    also present, for $2mn$ translations in total.
\hfill\break
Right: For the groups of type $\grp/$, %
the axes are tilted clockwise by $45^\circ$ and longer by the factor~$\sqrt2$.
}
  \label{fig:grid-parameters}
\end{figure}

\subsubsection{The torus swap groups}

For the groups of type \grp/, we have to
turn the picture by $45 ^\circ$.
We have the same three cases,
$\grp/^{\textbf{pm}}$,
$\grp/^{\textbf{cm}}$, and
$\grp/^{\textbf{pg}}$, but we must adapt the definition of $m$ and $n$,
see the right part of
Figure~\ref{fig:grid-parameters}.
We divide the principal diagonal from $(0,0)$ to $(2\pi,2\pi)$
into $m$ parts and the secondary diagonal
from $(0,0)$ to $(2\pi,-2\pi)$ into $n$ parts.
We cannot choose $m$ and $n$ freely because the midpoint $(2\pi,0)$ of
the square spanned by these two diagonal directions, which represents the identity mapping, is always part of the lattice.
Therefore, for the rectangular
lattice cases
$\grp/^{\textbf{pm}}$ and
$\grp/^{\textbf{pg}}$,
 $m$ and $n$ must be even, and
the number of lattice points on the torus is $mn/2$.
(We loose a factor of 2 compared to \grp|, because the tilted square in the figure
covers the torus twice.)
 For the rhombic
lattice case $\grp/^{\textbf{cm}}$,
$m$ and $n$ must have the same parity, and
the number of lattice points on the torus is $mn$.

We mention that the parameter $m$ in this case coincides with the parameter $m$ for the
translations-only case \grp1 of Figure~\ref{fig:translations}. The parameter $n$ coincides
in the rhombic case; in the rectangular case, it is twice as big.

As mentioned,
the groups of type \grp\setminus\ are mirrors of
the groups of type \grp/, and we need not discuss them separately.

\paragraph{Generators for $\grp|$, $\grp/$ and $\grp\setminus$.}
Whenever a mirror line exists 
($\textbf{cm}$ and $\textbf{pm}$), we choose the
origin
of the coordinate system on such a line; otherwise
($\textbf{pg}$), we place it
on an axis of glide reflection.
With these conventions, the groups can be generated by the generators listed
in %
Table~\ref{tab:swap-generators}.

\begin{table}[htb]
  \centering

  \begin{tabular}{|l|l|}
    \hline
    group& generators \\ %
    \hline
        $\grp|_{m,n}^{\textbf{pm}}$&
            $[e^{\frac{\pi i}{m}}, e^{\frac{\pi i}{m}}],
            [e^{\frac{\pi i}{n}}, e^{-\frac{\pi i}{n}}],
            *[i,i]$\\
        $\grp|_{m,n}^{\textbf{pg}}$&
            $[e^{\frac{\pi i}{m}}, e^{\frac{\pi i}{m}}],
            [e^{\frac{\pi i}{n}}, e^{-\frac{\pi i}{n}}],
            {*}[i,i] [e^{\frac{\pi i}{2m}}, e^{\frac{\pi i}{2m}}] $\\
        $\grp|_{m,n}^{\textbf{cm}}$&
            $[e^{\frac{\pi i}{m}}, e^{\frac{\pi i}{m}}],
            [e^{\frac{\pi i}{n}}, e^{-\frac{\pi i}{n}}],
            [e^{\frac{\pi i}{2m}+\frac{\pi i}{2n}}, e^{\frac{\pi i}{2m}-\frac{\pi i}{2n}}],
            *[i,i]$ \\
    \hline
        $\grp\setminus_{2m,2n}^{\textbf{pm}}$&
        $[e^{\frac{\pi i}{m}}, 1],
        [1, e^{\frac{\pi i}{n}}], 
        [-k,i]$\\
        $\grp\setminus_{2m,2n}^{\textbf{pg}}$&
        $[e^{\frac{\pi i}{m}}, 1],
        [1, e^{\frac{\pi i}{n}}], 
        [1,e^{\frac{\pi i}{2n}}] [-k,i] $\\
        $\grp\setminus_{m,n}^{\textbf{cm}}$&
        $[e^{\frac{i2\pi}{m}}, 1],
        [1, e^{\frac{i2\pi}{n}}], 
        [e^{\frac{\pi i}{n}}, e^{\frac{\pi i}{m}}],
        [-k,i]$\\
        \hline
        $\grp/_{2m,2n}^{\textbf{pm}}$&
        $[e^{\frac{\pi i}{m}}, 1],
        [1, e^{\frac{\pi i}{n}}], 
        [i, k]$\\
        $\grp/_{2m,2n}^{\textbf{pg}}$&
        $[e^{\frac{\pi i}{m}}, 1],
        [1, e^{\frac{\pi i}{n}}], 
        [e^{\frac{\pi i}{2m}}, 1] [i, k] $\\
        $\grp/_{m,n}^{\textbf{cm}}$&
        $[e^{\frac{i2\pi}{m}}, 1],
        [1, e^{\frac{i2\pi}{n}}], 
        [e^{\frac{\pi i}{n}}, e^{\frac{\pi i}{m}}],
        [i, k]$\\
    \hline
  \end{tabular}

  \caption{Generators for torus reflection groups and torus swap groups}
  \label{tab:swap-generators}
\end{table}

\subsection{The torus swapturn groups\texorpdfstring{, type \grp L}{}}
\label{sec:swapturn}
\label{sec:rotation}

\begin{figure}[tb]
  \centering
 \noindent\includegraphics[scale=0.9]{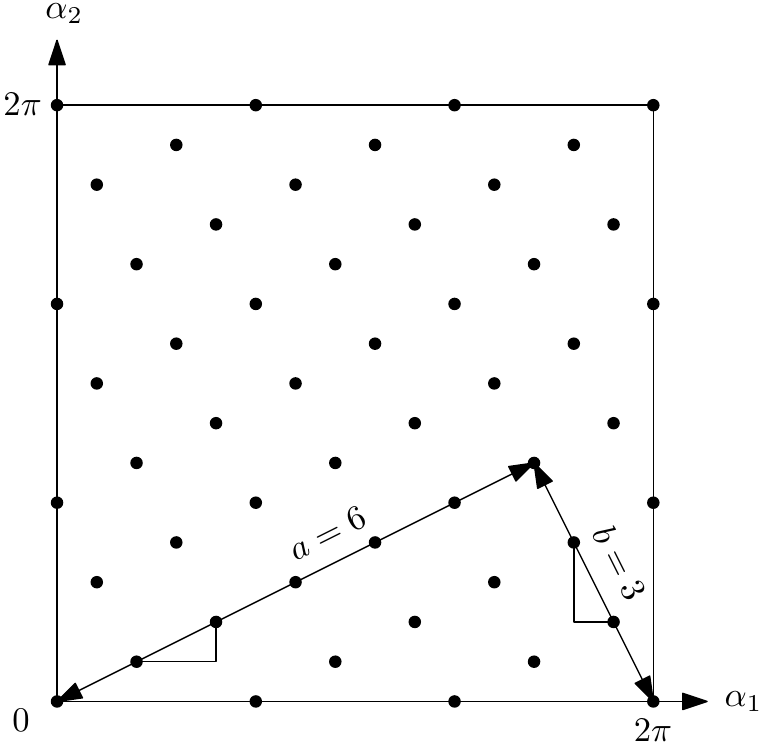}\hfill
  \raise
  0mm\hbox{\includegraphics[scale=1.2]{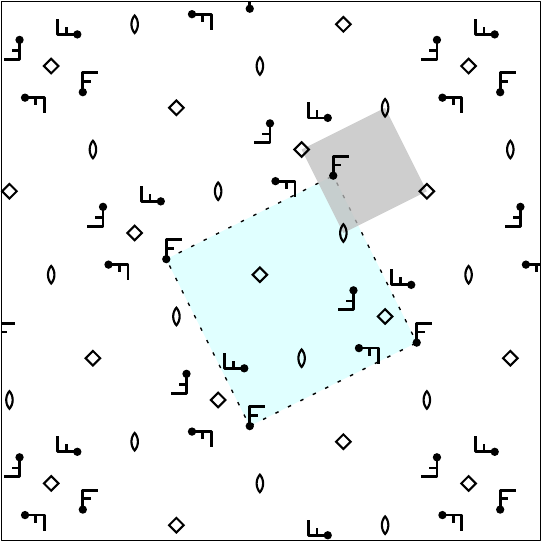}}
  \caption{Left: Parameterizing a square grid.
    Right:
    The wallpaper group \textbf{p4} corresponding to the groups \grp L.
    The centers of 4-fold rotations are
    marked by diamonds, the centers of 2-fold rotations are marked by
    ``digons'' in the form of a lense. The dotted light-blue square indicates the square lattice of the subgroup
    of translations,
    arbitrarily anchored at an upright F.
  }
  \label{fig:square-grid}
  \label{fig:square-rotation}
\end{figure}
By Lemma~\ref{closed}, the lattice of translations must be a square grid.
The left part of Figure~\ref{fig:square-grid} shows how we parameterize a square grid on the
torus. We take the
sides $a\ge0$ and $b\ge 0$ of the grid rectangle spanned by the
two points $(0,0)$ and $(2\pi,0)$, measured in grid units.
Since $(0,b)$ leads to the same grid rectangle as $(b,0)$, we require
$a\ge 1$.

Conjugation by $\sym/$ reflects the grid at the principal diagonal. Since
the grid is symmetric under $90^\circ$ rotations, this has the same
effect as reflection at a vertical axis, and it is easy to see
that such a reflection swaps the parameters $a$ and $b$.
Thus, $(a,b)$ and $(b,a)$ describe the same group, and we can assume
$a\ge b$ without loss of generality.

The number of grid points, i.e., the size of the translational subgroup,
is $a^2+b^2$, and the order is $4(a^2+b^2)$.
The right part of
Figure~\ref{fig:square-rotation} shows the various centers of 2-fold
and 4-fold rotations, and a typical orbit. This corresponds to the
wallpaper group \textbf{p4}.

The grid is generated by the two orthogonal vectors
$(\alpha_1,\alpha_2) = 2\pi(\frac
{a}{{a^2+b^2}},
  \frac {b}{{a^2+b^2}})$
    and
$(\alpha_1,\alpha_2) = 2\pi(\frac
{b}{{a^2+b^2}},
-\frac {-a}{a^2+b^2})$,
with $c=\sqrt{a^2+b^2}$.
If we choose the origin at the center of a $4$-fold rotation induced by
a swapturn, then  $\grp{L}_{a,b}$ can be generated by
$$
    [\exp \tfrac{(-a-b)\pi i}{a^2 + b^2}, \exp\tfrac{(a-b)\pi i}{a^2+b^2}],
    [\exp\tfrac{(a-b)\pi }{a^2 + b^2}, \exp\tfrac{(a+b)\pi i}{a^2+b^2}],
    *[-j, 1].
$$

\subsection{Groups that contain two orthogonal reflections%
\texorpdfstring{, type \grp+ and \grp X}{}}
\label{sec:translations+reflect2}

As in the case of \grp|, we distinguish, for each axis separately,
whether there are mirror reflections or only glide reflections.
We know that the glide reflection case is inconsistent with the
rhombic lattice (cf.~Section~\ref{sec:torus-reflection}).
Hence, we have the following cases, see Figure~\ref{fig:2reflections}.
\begin{itemize}
\item
  The grid of translations is a rhombic grid.
  In this case, both axes directions must be mirrors:
 \textbf{c2mm}.
\item
  The grid of translations is a rectangular grid.
  In this case each axis direction can be a mirror direction or a
  glide reflection
  \begin{itemize}
  \item
     \textbf{p2mm}. Two mirror directions
  \item
     \textbf{p2mg}. One mirror direction and one glide direction
  \item
     \textbf{p2gg}. Two glide directions
  \end{itemize}
\end{itemize}

In \textbf{p2mg}, the two families of invariant lines are distinguishable: one family of parallel lines consists of mirror lines,
whereas the perpendicular family has only glide reflections.
Thus, there are two different types, where the two directions change roles.

However, for $\grp+$, we need not distinguish two versions
of $\grp +^{\textbf{p2mg}}$,
because
conjugation with $\sym/$ maps one to the other.
For $\grp X$, on the other hand, the two versions are distinct. They are mirror
images.
We distinguish
$\grp X^{\textbf{p2mg}}$,
where the mirror lines are parallel to
    the principal diagonal $\phi_2=+\phi_1$, and
    $\grp X^{\textbf{p2gm}}$, where the mirror lines are parallel to
    the secondary diagonal direction $\phi_2=-\phi_1$.\footnote{
This is in accordance with previous editions of the International
Tables of X-Ray Crystallography, which
explicitly provided variations of the symbols for different ``settings''
\cite[Table 6.1.1, p.~542 in the 1952/1969 edition]{IT}:
short symbol \textbf{pmg}, full symbol \textbf{p2mg}, or \textbf{p2gm} for other setting.
}
The parameters $m$ and $n$ have the same meaning as in the
corresponding groups \grp| and \grp/.

These groups contain torus flips, as 
the product of two perpendicular reflections.
We choose the origin on the center of a $2$-fold rotation induced by
a torus flip.
For the groups \textbf{c2mm}, we place origin at the
intersection of two mirror lines. %
Then the groups can be generated by the generators given in
Table~\ref{tab:full-swap-generators}. %

\begin{figure}[htb]
  \centering

  \includegraphics{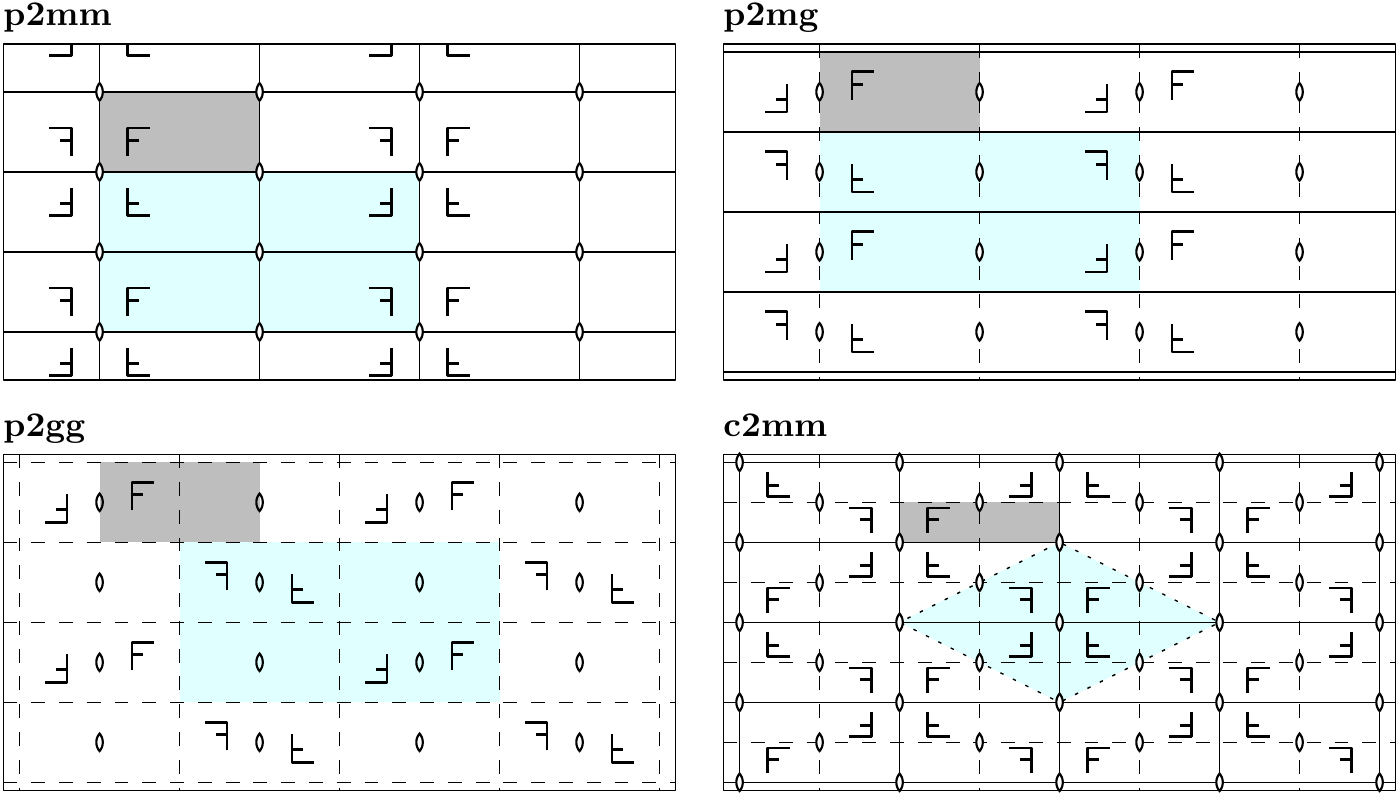}
  
  \caption{The four types of groups with two orthogonal families
    invariant lines.
    The light-blue region indicates the lattice of translations.
    For better visibility, the letter F is
    moved away from the mirror lines.
    Axes of mirror reflection are shown as solid lines, and
    axes of glide reflection are
    dashed.
As in Figure~\ref{fig:square-rotation},
 lenses
    mark centers of 2-fold rotations.   
}
  \label{fig:2reflections}
\end{figure}

\begin{table}
  \centering

    \begin{tabular}{|l|l|}
    \hline
        group& generators \\ %
        \hline
        $\grp+_{m,n}^{\textbf{p2mm}}$&
        $[e^{\frac{\pi i}{m}}, e^{\frac{\pi i}{m}}],
        [e^{\frac{\pi i}{n}}, e^{-\frac{\pi i}{n}}],
        *[i, i],
        *[k, k]$\\
        $\grp+_{m,n}^{\textbf{p2mg}}$&
        $[e^{\frac{\pi i}{m}}, e^{\frac{\pi i}{m}}],
        [e^{\frac{\pi i}{n}}, e^{-\frac{\pi i}{n}}],
        {*}[i,i]
        [e^{\frac{\pi i}{2n}}, e^{-\frac{\pi i}{2n}}],
        {*}[k,k]
        [e^{\frac{\pi i}{2n}}, e^{-\frac{\pi i}{2n}}]$\\
        $\grp+_{m,n}^{\textbf{p2gg}}$&
        $[e^{\frac{\pi i}{m}}, e^{\frac{\pi i}{m}}],
        [e^{\frac{\pi i}{n}}, e^{-\frac{\pi i}{n}}],
        {*}[i,i]
        [e^{\frac{\pi i}{2m}+\frac{\pi i}{2n}}, e^{\frac{\pi i}{2m}-\frac{\pi i}{2n}}],
        {*}[k,k]
        [e^{\frac{\pi i}{2m}+\frac{\pi i}{2n}}, e^{\frac{\pi i}{2m}-\frac{\pi i}{2n}}]$\\
        $\grp+_{m,n}^{\textbf{c2mm}}$&
        $[e^{\frac{\pi i}{m}}, e^{\frac{\pi i}{m}}],
        [e^{\frac{\pi i}{n}}, e^{-\frac{\pi i}{n}}],
        [e^{\frac{\pi i}{2m}+\frac{\pi i}{2n}},
                                      e^{\frac{\pi i}{2m}-\frac{\pi i}{2n}}],
        *[i, i],
        *[k, k]$\\
        \hline
        $\grp{X}_{2m,2n}^{\textbf{p2mm}}$&
        $[e^{\frac{\pi i}{m}}, 1],
        [1, e^{\frac{\pi i}{n}}],
        [i, k],
        [-k, i]$\\
        $\grp{X}_{2m,2n}^{\textbf{p2mg}}$&
        $[e^{\frac{\pi i}{m}}, 1],
        [1, e^{\frac{\pi i}{n}}],
        [1, e^{\frac{\pi i}{2n}}][i, k],
        [1, e^{\frac{\pi i}{2n}}][-k, i]$\\
        $\grp{X}_{2m,2n}^{\textbf{p2gm}}$&
        $[e^{\frac{\pi i}{m}}, 1],
        [1, e^{\frac{\pi i}{n}}],
        [e^{\frac{\pi i}{2m}}, 1] [i, k],
        [e^{\frac{\pi i}{2m}}, 1] [-k, i]$\\
        $\grp{X}_{2m,2n}^{\textbf{p2gg}}$&
        $[e^{\frac{\pi i}{m}}, 1],
        [1, e^{\frac{\pi i}{n}}],
        [e^{\frac{\pi i}{2m}}, e^{\frac{\pi i}{2n}}] [i, k], 
        [e^{\frac{\pi i}{2m}}, e^{\frac{\pi i}{2n}}] [-k, i]$\\ 
        $\grp{X}_{m,n}^{\textbf{c2mm}}$&
        $[e^{\frac{i2\pi}{m}}, 1],
        [1, e^{\frac{i2\pi}{n}}],
        [e^{\frac{\pi i}{m}}, e^{\frac{\pi i}{n}}],
        [i, k],
        [-k, i] %
         $\\
        \hline
        $\grp*_{n}^{\textbf{p4mm}\mathrm{U}}$&
        $[e^{\frac{\pi i}{n}}, e^{\frac{\pi i}{n}}],
        [e^{\frac{\pi i}{n}}, e^{-\frac{\pi i}{n}}],
        [i, k],
        *[i, i]$\\ 
        $\grp*_{n}^{\textbf{p4gm}\mathrm{U}}$&
        $[e^{\frac{\pi i}{n}}, e^{\frac{\pi i}{n}}],
        [e^{\frac{\pi i}{n}}, e^{-\frac{\pi i}{n}}],
        [i, k] [e^{\frac{\pi i}{n}}, 1], 
         {*}[i, i][e^{\frac{\pi i}{n}}, 1]$\\ 
        $\grp*_{n}^{\textbf{p4mm}\mathrm{S}}$&
        $[e^{\frac{\pi i}{n}}, 1],
        [1, e^{\frac{\pi i}{n}}],
        [i, k],
        {*}[i, i]$\\ 
        $\grp*_{n}^{\textbf{p4gm}\mathrm{S}}$&
        $[e^{\frac{\pi i}{n}}, 1],
        [1, e^{\frac{\pi i}{n}}],
        [i, k] [e^{\frac{\pi i}{2n}}, e^{\frac{\pi i}{2n}}],
        {*}[i, i]
        [e^{\frac{\pi i}{2n}}, e^{\frac{\pi i}{2n}}] 
                                               $\\
        \hline
  \end{tabular}

  \caption
  [Generators for full torus reflection/swap groups and full torus groups]
  {Generators for full torus reflection groups, full torus swap groups,
    and full torus groups}
  \label{tab:full-generators}
  \label{tab:full-swap-generators}
\label{tab:full-reflection-generators}
\end{table}

\subsection{The full torus groups\texorpdfstring{, type {\grp*}}{}}
\label{sec:everything}

Finally, we have the groups where all directional transformations are combined.
The conditions of \grp+ and \grp X force the lattice to be a
rectangular lattice both in the $\phi_1,\phi_2$ axis direction and in
the $\pm 45^\circ$ direction, possibly with added midpoints (rhombic case).
This means that the lattice is a \emph{square} lattice.
It appears as a
\emph{rectangular} lattice in
one pair of
perpendicular directions and as a %
\emph{rhombic} lattice in the
other directions.

Thus, there are only two cases for the translation lattice: The square $n\times n$ lattice with
$n^2$ translations (the upright grid ``U'',
Figure~\ref{fig:everything}a),
and its rhombic extension with $2 n^2$
translations (the slanted grid ``S'', Figure~\ref{fig:everything}b).

Let us first consider the slanted case,
 see
Figure~\ref{fig:everything}b.
The lattice appears as a rhombic lattice for the \grp+ directions.
From the point of view of the subgroups of type \grp+, we know that 
this means that the ``glide reflection'' case is excluded
(cf.~the
discussion in Section~\ref{sec:torus-reflection}).
There must
be mirror reflections in the horizontal and vertical axes.

\begin{figure}
  \centering
  \includegraphics{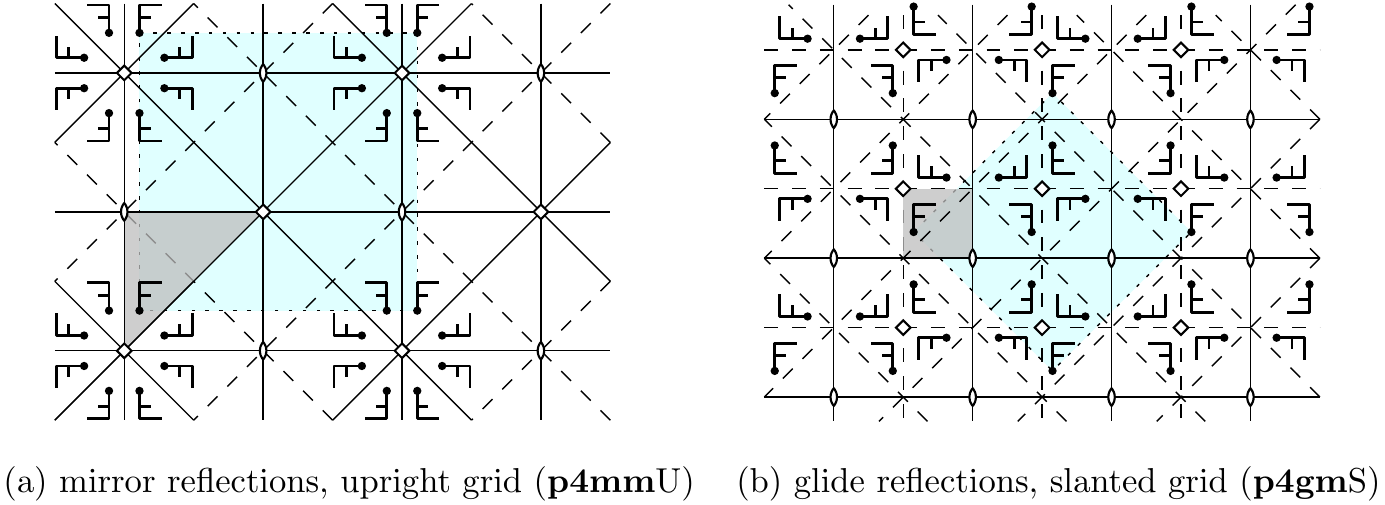}
  \caption{Two of the four types of groups \grp*.
Small squares denote centers of 4-fold rotations.
  For each figure, there exists a rotated version by
  $45^\circ$, where $\grp*^{\textbf{p4mm}\mathrm{U}}$ becomes
  $\grp*^{\textbf{p4mm}\mathrm{S}}$,
  and
  $\grp*^{\textbf{p4gm}\mathrm{S}}$ becomes
  $\grp*^{\textbf{p4gm}\mathrm{U}}$.}
  \label{fig:everything}
\end{figure}

For the \grp X directions, the lattice appears as a rectangular lattice.
According to Section~\ref{sec:translations+reflect2} we can have the
cases mirror/mirror, mirror/glide, glide/glide.
But since $90^\circ$ rotations are included, the mixed mirror/glide case is
impossible.
Two cases remain, which we call $\grp*^{\mathbf{p4mm}\mathrm{S}}$ and
$\grp*^{\mathbf{p4gm}\mathrm{S}}$. The latter is shown in
Figure~\ref{fig:everything}b.
When
the lattice appears as a square lattice for the \grp+ directions, the
two pairs of directions
\grp+ 
and \grp X
change roles, and we have
 two more groups,
 $\grp*^{\mathbf{p4mm}\mathrm{U}}$ and
 $\grp*^{\mathbf{p4gm}\mathrm{U}}$.
 The first one is shown in
 Figure~\ref{fig:everything}a.
The groups $\grp*^{\mathbf{p4mm}}$ have mirrors in all four
directions, whereas the groups
 $\grp*^{\mathbf{p4gm}}$ have mirrors in two directions only.

To list the generators for the full torus groups, we choose the origin of the
coordinate system on the center of a $4$-fold rotation induced by 
a swapturn, 
see Table~\ref{tab:full-generators}. %

This concludes the discussion of the toroidal groups. The reader who
wishes to practice the understanding of these classes might try to
count, as an exercise, all groups of order~100, see
Appendix~\ref{sec:counting}.

\subsection{Duplications}
\label{dupli}

As we have seen, every subgroup of a group
$\pm[D_{2m}\times D_{2n}]$ has an invariant %
torus.
So far, we have analyzed the groups that leave a \emph{fixed} torus invariant.
We have already mentioned that some subgroups have more than one invariant
Clifford torus, and this leads to duplications.
Unfortunately,
when it comes to weeding out duplications, all classifications
(including the classic classification) become messy.\footnote
{The difficulty caused by these ambiguous transformations, in
particular in connection with achiral groups, was already acknowledged by
Hurley~\cite[p.~656--7]{hurley}.}

We analyze the situation
as follows.
Every \OP\ transformation is of the form $R_{\alpha_1,\alpha_2}$,
with $-\pi \le \alpha_1, \alpha_2\le \pi$.
If
$\alpha_1 \ne \pm \alpha_2$, there is a unique pair of absolutely
orthogonal invariant planes, and hence, there is a unique invariant
Clifford torus \emph{on which the transformation appears as a torus
  translation}.
We call this torus the \emph{primary} invariant torus.

Our strategy is to analyze the situation
backwards. We look at all \OP\ transformations that are not torus
translations, we write them in the form
$R_{\alpha_1,\alpha_2}$ and determine the translation vector
$(\alpha_1,\alpha_2)$ by which they would appear on their primary
invariant torus.
The result is summarized in the following proposition.
The torus translations that lead to  ambiguity
are shown in
Figure~\ref{fig:ambiguities}:

\begin{proposition}
  The \OP\ transformations that have more than one invariant torus are
  the following:
\begin{enumerate}[\rm (a)]

  \item Simple half-turns of the form $\diag(-1,-1,1,1)$.
    
On their primary torus, they appear
as torus translation by $(\pi,0)$
or $(0,\pi)$. There is an infinite family of alternate tori
for which they are interpreted as torus flips or torus swaps.

\item 
  Double rotations
$R_{\alpha,\pi\pm\alpha}$.

On an alternate torus, they appear as reflections or
glide reflections associated to torus swaps \sym/ or~\sym\setminus.

\item Left and right rotations
 $R_{\alpha,\pm\alpha}$, including $\id$ and $-\id$.
(For $\alpha=\pm \pi/2$, these fall also under case~(b).)
 
A  left rotation $R_{\alpha,\alpha}$
with $\alpha\ne\pm \pi/2$ %
  appears as a torus translation
  by $(\alpha,\alpha)$ or
  by $(-\alpha,-\alpha)$
  on every invariant torus.

  Similarly, a right rotation $R_{\alpha,-\alpha}$
with $\alpha\ne\pm \pi/2$ %
  appears as a torus translation
  by $(\alpha,-\alpha)$ or
  by $(-\alpha,\alpha)$
  on every invariant torus.
\end{enumerate}
\end{proposition}

\begin{figure}[htb]
  \centering
  \includegraphics{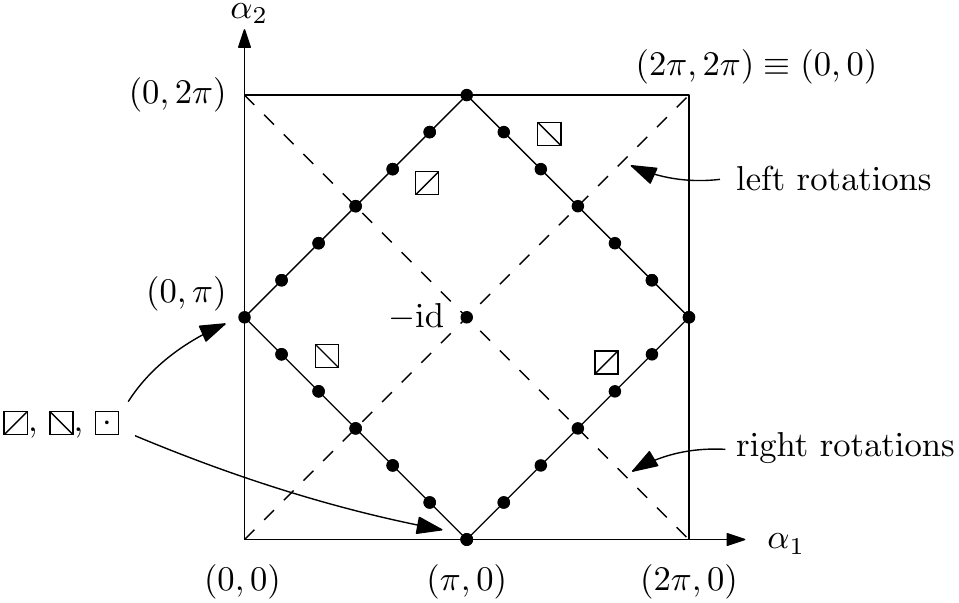}
  \caption{%
The torus translations on the tilted square are ambiguous: they can appear as rotations
of different types, as indicated.
Left and right rotations (on the diagonal) also have no unique
invariant torus, but they appear as left and right rotations on any
invariant torus.}
  \label{fig:ambiguities}
\end{figure}

\begin{proof}
  The \OP\ transformations that are not torus translations are
  \sym. (torus flips) and \sym/ and \sym\setminus\ (reflections and
  glide reflections associated to torus swaps).

  Every torus flip is a half-turn, and these are covered in case~(a).

Let us look at   
reflections and
glide reflections associated to the torus swaps \sym/.
The torus swap \sym/ at the principal diagonal is the transformation
$[i,k]$. Both $i$ and $k$ are pure quaternions, in accordance with the
fact that
\sym/ is a half-turn.
The general torus swap of type $\sym/$ is obtained by combining
$[i,k]$ with an arbitrary torus translation
$[\exp\beta_li,\exp\beta_ri]$:
\begin{displaymath}
[i\exp\beta_li,
k\exp\beta_ri] =
[\exp(\tfrac\pi2i)\exp\beta_li,
k(\cos\beta_r +i\sin\beta_r)]=
[\exp((\tfrac\pi2+\beta_l)i),
k\cos\beta_r +j\sin\beta_r]
\end{displaymath}
The right component
$k\cos\beta_r +j\sin\beta_r$
is still a unit quaternion (rotation angle $\pi/2$), and hence
the right rotation
$[1,\exp\beta_ri]$ has no effect on the type of the transformation.
This is in accordance with the fact that,
on the
$\phi_1,\phi_2$-torus, a right rotation is a translation perpendicular to the
reflection axis of \sym/, whose effect is just to move the reflection
axis.
The left rotation, however, changes the rotation angle from $\pi/2$ to
$\pi/2+\beta_l$. The result is a rotation of type
$R_{\pi+\beta_l,\beta_l}$.
As a torus translation $R_{\alpha_1,\alpha_2}$, it lies on the line
$\alpha_1=\alpha_2+\pi$
(and $\alpha_1=\alpha_2-\pi$, considering that angles are taken modulo
$2\pi$),
see Figure~\ref{fig:ambiguities}.

The operations of type
\sym\setminus\
are the mirrors of \sym/, and hence they appear on the reflected lines
$\alpha_1=-(\alpha_2\pm \pi)$.

Left and right rotations have infinitely many
 invariant tori, but cause no confusion for our classification,
because
a left rotation will appear as the same left rotation on \emph{any}
invariant torus (possibly with an inverted angle), except when it
falls under case~(b).
\end{proof}
We note the curious fact that the operations that don't have a unique
invariant torus coincide with the operations whose squares
 are left or right rotations.

\begin{corollary}
  \label{cor:several-invariant-tori}
  A group may have more than one invariant torus only if the
  translational subgroup contains only elements on the diagonals and
  on the tilted square in Figure~\ref{fig:ambiguities}.
\end{corollary}

This excludes from the search for duplications those groups for which
the translational subgroup is sufficiently rich, i.e., when both
parameters $m$ and $n$ are large. Still it leaves a large number of
cases where one of the parameters is small.
We present the list of duplications below.

\subsubsection{List of Duplications}
\label{sec:duplications-list}

As mentioned, we have imposed
the stricter conditions on $m$ and $n$ (and $a$ and $b$) in Table~\ref{tab:overview} in order to
exclude all duplications.
As a rule, among equal groups, we have chosen
the group with the larger subgroup of torus translations
(%
with the chosen invariant torus)
to stay in the table.

Table~\ref{tab:duplications} lists every group $G_1$ that is excluded
from Table~\ref{tab:overview}, together with a
group $G_2$ to which it is conjugate, and a conjugation
that converts the second group to the first one.
The conjugations %
depend on the specific parameterizations that we have
chosen and that were given with each class of groups discussed above,
in particular in Tables~\ref{tab:swap-generators}
and~\ref{tab:full-reflection-generators}.

In this section, we use the notation
$G_1\doteq G_2$ for groups that are geometrically the same, i.e., conjugate
under an \OP\ transformation, and we reserve the sign
``$=$'' for groups that are equal in our chosen coordinate system.

In some classes, the choice of the two parameters $m$ and $n$ is
symmetric (e.g., $
\grp+^{\mathbf{p2mm}}_{m,n}
\doteq \grp+^{\mathbf{p2mm}}_{n,m}$).
In those cases, we have achieved uniqueness by requiring $m\ge n$
in  Table~\ref{tab:overview}.
Such symmetries between the
parameters, %
and other general relations are listed
first
for each type of group in
 Table~\ref{tab:duplications}.
This is followed by a list of groups with small parameters that
are explicitly excluded in Table~\ref{tab:overview}.

We have made some simplifications to keep the table compact.
As mentioned previously, we sometimes refer to groups
$\grp{1}_{m,n}^{(s)}$ or
$\grp{.}_{m,n}^{(s)}$
where the parameter $s$  lies outside
 the ``legal''
 range (\ref{eq:ks}), in order to avoid case distinctions.
 The parameter $s$ can be brought into that range by using the
 equalities
 $\grp{1}_{m,n}^{(s)}=\grp{1}_{m,n}^{(s\pm m)}\doteq\grp{1}_{m,n}^{(-n-s)}$,
 and similarly for \grp..
If the permissible range of
parameters $s$ contains only one integer, we omit the parameter and denote the group simply by
$\grp{1}_{m,n}$ or
$\grp._{m,n}$.
In such a case, any choice of $s$ will lead to the same group.

 We have a few cases with more than two equal groups:
\begin{gather*}
    \grp{/}^{\mathbf{cm}}_{1,1}\doteq
    \grp{\setminus}^{\mathbf{cm}}_{1,1}\doteq
    \grp{.}^{{(0)}}_{1,1}\doteq \grp{1}^{{(0)}}_{1,2}
    = \langle\diag(1,1,-1,-1)\rangle
    \text{  (order 2)}
\\
    \grp{/}^{\mathbf{pm}}_{2,2}\doteq
    \grp{\setminus}^{\mathbf{pm}}_{2,2}\doteq
    \grp{.}^{{(-1)}}_{2,1}\doteq \grp{1}^{{(0)}}_{2,2}
    = \langle\diag(1,1,-1,-1),\diag(-1,-1,1,1)
    \rangle
    \cong D_4
    \text{  (order 4)}\\
    \grp{X}^{\mathbf{p2gg}}_{2,2}\doteq
    \grp{/}^{\mathbf{pm}}_{4,2}\doteq
    \grp{\setminus}^{\mathbf{pm}}_{2,4}\doteq \grp1^{(-2)}_{4,2}
    = \langle\diag(R_{\pi/2},R_{\pi/2}),\diag(R_{\pi/2},R_{-\pi/2})
    \rangle
    \text{  (order 8)}\\
    \grp{X}^{\mathbf{p2gm}}_{2,2}\doteq
    \grp{/}^{\mathbf{cm}}_{2,2}\doteq
    \grp{/}^{\mathbf{pm}}_{2,4}\doteq \grp.^{(-1)}_{2,2}
    \doteq \langle-\id,\diag(1,-1,1,-1),\diag(R_{\pi/2},R_{-\pi/2})
    \rangle
    \text{  (order 8)}\\
    \grp{X}^{\mathbf{p2mg}}_{2,2}\doteq
    \grp{\setminus}^{\mathbf{cm}}_{2,2}\doteq
    \grp{\setminus}^{\mathbf{pm}}_{4,2}\doteq \grp{.}^{{(-2)}}_{4,1}
    \doteq \langle-\id,\diag(1,-1,1,-1),\diag(R_{\pi/2},R_{\pi/2})
    \rangle
    \text{  (order 8)}\\
    \grp{*}^{\mathbf{p4gm\textrm{U}}}_{1}\doteq
    \grp{+}^{\mathbf{p2gg}}_{2,1}\doteq \grp{+}^{\mathbf{p2gg}}_{1,2}
    \doteq \langle
    \diag(-1,-1,1,1),
    \diag(1,1,-1,1),
    \diag(1,1,1,-1)
    \rangle
    \text{  (order 8)}\\
    \grp{X}^{\mathbf{c2mm}}_{2,2}\doteq
    \grp{X}^{\mathbf{p2mm}}_{4,2}\doteq
    \grp{X}^{\mathbf{p2mm}}_{2,4}\doteq \grp{.}^{{(-2)}}_{4,2}
    \text{  (order 16)}
\end{gather*}
To reduce case distinctions, some of these groups
$G_1$ point to groups $G_2$
that are themselves excluded in Table~\ref{tab:overview}, and which
must be looked up again
in Table~\ref{tab:duplications}.

The conjugations 
in Table~\ref{tab:duplications}
 were found
 by computer search for %
 particular values of $m$.
In many cases, the conjugate group or the conjugacy mapping depends
on the parity of some parameter.
We tried to simplify the
entries of the table
by manually adjusting them. %
All conjugations %
were
 checked by computer for $m\le100$.

 When the groups are translated to the Conway-Smith classification using
  Table~\ref{tab:overview}, the duplications have easy algebraic
  justifications: 
  For example, $C_2$ and $D_2$ are obviously the same group.
  Also, $\bar D_4$ can be replaced by $D_4$, see
  Appendix~\ref{sec:index4} for more information.

\begin{table}
\let\,\relax %

\begin{tabular}[htb]{|l|l|l||l|l|l|}
\hline
$G_1$ & $G_2$ & $[\hat{l}, \hat{r}]$%
& $G_1$ & $G_2$& $[\hat{l}, \hat{r}]$%
\\
\hline
\multicolumn6{|l|}{chiral groups}\\
\hline
$\grp{1}^{{(s)}}_{m,n}$& $\grp{1}^{(s+n)}_{m,n}$& $[1,1]$ (equal) &$\grp{.}^{{(s)}}_{m,n}$& $\grp{.}^{(s+n)}_{m,n}$& $[1,1]$ (equal) \\
$\grp{1}^{{(s)}}_{m,n}$& $\grp{1}^{(-m-s)}_{m,n}$&$[i,k] = \sym{/}$ &$\grp{.}^{{(s)}}_{m,n}$& $\grp{.}^{(-m-s)}_{m,n}$&$[i,k] = \sym{/}$ \\
&&&$\grp._{1,1}              $& $\grp1_{1,2}              $& $[i + j, 1 + k]$\\
&&&$\grp._{2,1}              $& $\grp1_{2,2}^{(0)}        $& $[i + j, i + j]$\\
\hline
$\grp\setminus^{\mathbf{pm}}_{4 \, m - 2,2}$& $\grp._{4 \, m - 2,1}     $& $[j + k, i + j]$&$\grp/^{\mathbf{pm}}_{2,4 \, m - 2}$& $\grp._{2,2 \, m - 1}^{(-1)}$& $[i + j, j + k]$\\
$\grp\setminus^{\mathbf{pm}}_{4 \, m,2}$& $\grp._{4 \, m,1}         $& $[1, i + j]$&$\grp/^{\mathbf{pm}}_{2,4 \, m}$& $\grp._{2,2 \, m}^{(-1)}  $& $[i + j, 1]$\\
$\grp\setminus^{\mathbf{pm}}_{2,4 \, m - 2}$& $\grp1_{2,4 \, m - 2}^{(2 \, m - 2)}$& $[i + k, 1]$&$\grp/^{\mathbf{pm}}_{2 \, m,2}$& $\grp1_{2 \, m,2}^{(0)}   $& $[1, i + k]$\\
$\grp\setminus^{\mathbf{pm}}_{2,4 \, m}$& $\grp1_{4,2 \, m}^{(-2)}  $& $[i + k, 1]$&&&\\
\hline
$\grp\setminus^{\mathbf{pg}}_{2,4 \, m - 2}$& $\grp1_{4,2 \, m - 1}^{(-2)}$& $[i + k, 1]$&$\grp/^{\mathbf{pg}}_{2 \, m,2}$& $\grp1_{2 \, m,2}^{(1)}   $& $[1, i + k]$\\
$\grp\setminus^{\mathbf{pg}}_{2,4 \, m}$& $\grp1_{2,4 \, m}^{(2 \, m - 1)}$& $[i + k, 1]$&&&\\
\hline
$\grp\setminus^{\mathbf{cm}}_{2 \, m + 1,1}$& $\grp._{2 \, m + 1,1}     $& $[j + k, 1 - k]$&$\grp/^{\mathbf{cm}}_{1,2 \, m + 1}$& $\grp._{1,2 \, m + 1}^{(m)}$& $[i + j, j + k]$\\
$\grp\setminus^{\mathbf{cm}}_{1,4 \, m - 3}$& $\grp1_{1,8 \, m - 6}^{(2 \, m - 2)}$& $[i + k, 1]$&$\grp/^{\mathbf{cm}}_{4 \, m - 3,1}$& $\grp1_{4 \, m - 3,2}     $& $[1, 1 - j]$\\
$\grp\setminus^{\mathbf{cm}}_{1,4 \, m - 1}$& $\grp1_{1,8 \, m - 2}^{(2 \, m - 1)}$& $[1 - j, 1]$&$\grp/^{\mathbf{cm}}_{4 \, m - 1,1}$& $\grp1_{4 \, m - 1,2}     $& $[1, i + k]$\\
$\grp\setminus^{\mathbf{cm}}_{2,4 \, m - 2}$& $\grp\setminus^{\mathbf{pm}}_{4,4 \, m - 2}$& $[i + j, 1]$&$\grp/^{\mathbf{cm}}_{4 \, m - 2,2}$& $\grp/^{\mathbf{pm}}_{4 \, m - 2,4}$& $[1, i + j]$\\
$\grp\setminus^{\mathbf{cm}}_{2,4 \, m}$& $\grp\setminus^{\mathbf{pg}}_{4,4 \, m}$& $[i + k, 1]$&$\grp/^{\mathbf{cm}}_{4 \, m,2}$& $\grp/^{\mathbf{pg}}_{4 \, m,4}$& $[1, i + k]$\\
\hline
$\grp X^{\mathbf{p2mm}}_{2 \, m,2}$& $\grp._{2 \, m,2}^{(0)}   $& $[1, i + k]$&$\grp X^{\mathbf{p2mm}}_{2,4 \, m - 2}$& $\grp._{2,4 \, m - 2}^{(2 \, m - 2)}$& $[i + k, 1]$\\
&&&$\grp X^{\mathbf{p2mm}}_{2,4 \, m}$& $\grp._{4,2 \, m}^{(-2)}  $& $[i + k, 1]$\\
\hline
$\grp X^{\mathbf{p2gm}}_{2 \, m,2}$& $\grp._{2 \, m,2}^{(1)}   $& $[1, i + k]$&$\grp X^{\mathbf{p2mg}}_{2,4 \, m - 2}$& $\grp._{4,2 \, m - 1}^{(-2)}$& $[i + k, 1]$\\
&&&$\grp X^{\mathbf{p2mg}}_{2,4 \, m}$& $\grp._{2,4 \, m}^{(2 \, m - 1)}$& $[i + k, 1]$\\
$\grp X^{\mathbf{p2gm}}_{2,4 \, m - 2}$& $\grp/^{\mathbf{cm}}_{2,4 \, m - 2}$& $[i + j, j + k]$&$\grp X^{\mathbf{p2mg}}_{4 \, m - 2,2}$& $\grp\setminus^{\mathbf{cm}}_{4 \, m - 2,2}$& $[j + k, i + j]$\\
$\grp X^{\mathbf{p2gm}}_{2,4 \, m}$& $\grp/^{\mathbf{pm}}_{4,4 \, m}$& $[i + j, 1]$&$\grp X^{\mathbf{p2mg}}_{4 \, m,2}$& $\grp\setminus^{\mathbf{pm}}_{4 \, m,4}$& $[1, i + j]$\\
\hline
$\grp X^{\mathbf{p2gg}}_{4 \, m - 2,2}$& $\grp\setminus^{\mathbf{pm}}_{4 \, m - 2,4}$& $[j + k, i + j]$&$\grp X^{\mathbf{p2gg}}_{4 \, m,2}$& $\grp\setminus^{\mathbf{cm}}_{4 \, m,2}$& $[1, i + j]$\\
$\grp X^{\mathbf{p2gg}}_{2,4 \, m - 2}$& $\grp/^{\mathbf{pm}}_{4,4 \, m - 2}$& $[i + j, j + k]$&$\grp X^{\mathbf{p2gg}}_{2,4 \, m}$& $\grp/^{\mathbf{cm}}_{2,4 \, m}$& $[i + j, 1]$\\
\hline
$\grp X^{\mathbf{c2mm}}_{4 \, m - 3,1}$& $\grp._{4 \, m - 3,2}     $& $[1, 1 - j]$&$\grp X^{\mathbf{c2mm}}_{4 \, m - 2,2}$& $\grp X^{\mathbf{p2mm}}_{4 \, m - 2,4}$& $[j + k, i + j]$\\
$\grp X^{\mathbf{c2mm}}_{4 \, m - 1,1}$& $\grp._{4 \, m - 1,2}     $& $[1, 1 + j]$&$\grp X^{\mathbf{c2mm}}_{4 \, m,2}$& $\grp X^{\mathbf{p2gm}}_{4 \, m,4}$& $[1, i + k]$\\
$\grp X^{\mathbf{c2mm}}_{1,4 \, m - 3}$& $\grp._{1,8 \, m - 6}^{(2 \, m - 2)}$& $[1 + j, 1]$&$\grp X^{\mathbf{c2mm}}_{2,4 \, m - 2}$& $\grp X^{\mathbf{p2mm}}_{4,4 \, m - 2}$& $[i + j, j + k]$\\
$\grp X^{\mathbf{c2mm}}_{1,4 \, m - 1}$& $\grp._{1,8 \, m - 2}^{(2 \, m - 1)}$& $[1 - j, 1]$&$\grp X^{\mathbf{c2mm}}_{2,4 \, m}$& $\grp X^{\mathbf{p2mg}}_{4,4 \, m}$& $[i + k, 1]$\\
\hline
\multicolumn6{|l|}{achiral groups}\\
\hline
$\grp+^{\mathbf{p2mm}}_{m,n}$& $\grp+^{\mathbf{p2mm}}_{n,m}$& $[i, k]=\sym/$&$\grp L_{a,b}             $& $\grp L_{b,a}             $& $[i, k]=\sym/$\\
$\grp+^{\mathbf{p2gg}}_{m,n}$& $\grp+^{\mathbf{p2gg}}_{n,m}$& $[i, k]=\sym/$&$\grp L_{1,0}             $& $\grp|^{\mathbf{pg}}_{2,1}$& $[1 + k, 1 - i + j + k]$\\
$\grp+^{\mathbf{c2mm}}_{m,n}$& $\grp+^{\mathbf{c2mm}}_{n,m}$& $[i, k]=\sym/$&$\grp L_{1,1}             $& $\grp|^{\mathbf{pg}}_{2,2}$& $[1 + k, 1 + i - j + k]$\\
$\grp+^{\mathbf{p2mm}}_{1,1}$& $\grp|^{\mathbf{pm}}_{1,2}$& $[1 + k, 1 - k]$&$\grp L_{2,0}             $& $\grp+^{\mathbf{p2gg}}_{2,2}$& $[1 + k, 1 + k]$\\
$\grp+^{\mathbf{p2mg}}_{1,1}$& $\grp|^{\mathbf{pm}}_{2,1}$& $[1 + k, i - j]$&&&\\
$\grp+^{\mathbf{p2gg}}_{1,1}$& $\grp|^{\mathbf{pg}}_{1,2}$& $[1 + k, 1 - k]$&&&\\
$\grp+^{\mathbf{c2mm}}_{1,1}$& $\grp|^{\mathbf{pm}}_{2,2}$& $[1 + k, 1 - k]$&&&\\
\hline
$\grp*^{\mathbf{p4mm}\textrm{U}}_{1}$& $\grp+^{\mathbf{p2mg}}_{1,2}$& $[1 + k, 1 - i - j - k]$&$\grp*^{\mathbf{p4gm}\textrm{U}}_{1}$& $\grp+^{\mathbf{p2gg}}_{2,1}$& $[1 + k, 1 + i - j + k]$\\
$\grp*^{\mathbf{p4mm}\textrm{U}}_{2}$& $\grp+^{\mathbf{c2mm}}_{2,2}$& $[1 + k, 1 + k]$&$\grp*^{\mathbf{p4gm}\textrm{U}}_{2}$& $\grp L_{2,2}             $& $[1 + j, 1 + j]$\\
$\grp*^{\mathbf{p4mm}\textrm{S}}_{1}$& $\grp+^{\mathbf{p2mg}}_{2,2}$& $[1 + k, 1 + i + j - k]$&$\grp*^{\mathbf{p4gm}\textrm{S}}_{1}$& $\grp|^{\mathbf{cm}}_{2,2}$& $[1 + k, 1 + k]$\\
\hline
\end{tabular}
\caption[The duplications among toroidal groups]{Duplications.
The range of the parameter $m$ is $m \ge 1$ in all cases.
The group $G_1$ is obtained from $G_2$ by conjugation with
$h:=[\frac{\hat{l}}{\|\hat{l}\|}, \frac{\hat{r}}{\|\hat{r}\|}]$.
That is, $G_1=h^{-1}G_2h$.
  }
  \label{tab:duplications}
\end{table}

\begin{figure}[htb]
  \centering
  \includegraphics{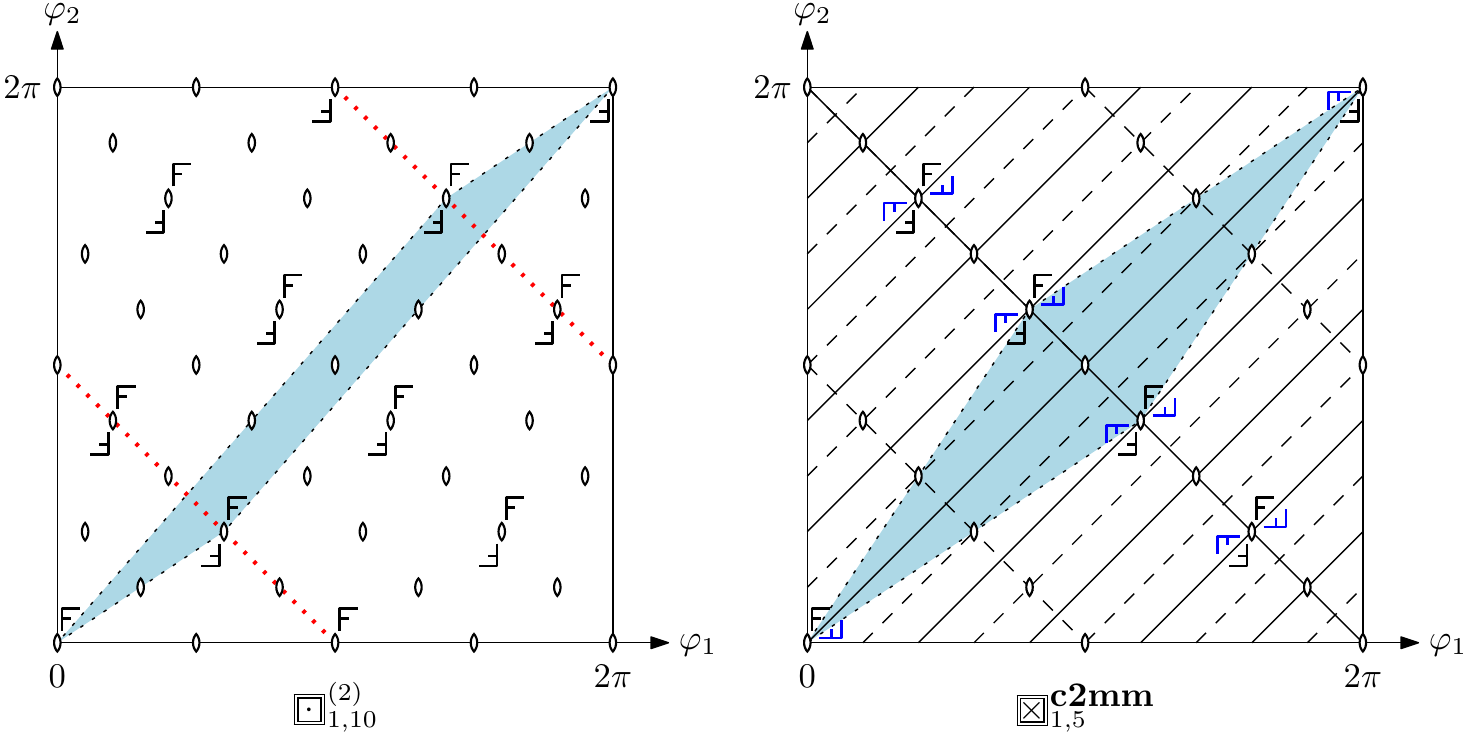}
  \caption{Duplication example, $ \grp{.}_{1,10}^{(2)} \doteq
    \grp{X}_{1,5}^{\textbf{c2mm}} $
  }
  \label{fig:dup}
\end{figure}

\subsubsection{A duplication example}
\label{sec:dup-example}
By way of example, we treat one
duplication in detail:
\begin{equation}
\label{eq:DUP}
\grp{X}_{1,n}^{\textbf{c2mm}} \doteq \grp{.}_{1,2n}^{(\frac{n-1}{2})}
\text{, for odd $n$.}
\end{equation}
Figure~\ref{fig:dup} shows the action of these groups on the torus
for $n=5$.
We can confirm that, in accordance with
Corollary~\ref{cor:several-invariant-tori},
the 10 torus translations of
$\grp{.}_{1,10}^{(2)}$ lie only on a diagonal and on the line
$\alpha_1+\alpha_2=\pm\pi$.
The latter 5 translations become reflections and glide reflections
in $\grp{X}_{1,5}^{\textbf{c2mm}}$.
More precisely,
in accordance with
Figure~\ref{fig:ambiguities}, they are the reflections at the
\sym\setminus\ diagonal (4 glide reflections and one reflection). The
picture shows actually more glide reflection and  reflection axes than
the order of the group would allow. The reason is that every glide
reflection in this group can also be interpreted as a reflection, at a
different axis.

We now prove the conjugacy formally.
Since these groups have the same order $4n$, it is enough to show
that  $G_2 = \grp{.}_{1,2n}^{(\frac{n-1}{2})}$ is contained in
$G_1 = \grp{X}_{1,n}^{\textbf{c2mm}}$. We do this by
checking
that the generators of $G_2$, %
under conjugation
by the element $h$ from Table~\ref{tab:duplications}, are elements of
$G_1$. %
Here are the generators we gave for these groups:
\begin{align*}
G_1=    \grp{X}_{1,n}^{\textbf{c2mm}} &= \langle [1,1],
        [1, e^{\frac{i2\pi}{n}}], [-1, e^{\frac{\pi i}{n}}], [i, k], [-k, i]
    \rangle \text{\quad (see Table~\ref{tab:full-generators})} \\
  G_2=    \grp{.}_{1,2n}^{(\frac{n-1}{2})} &= \langle
[e^{-2\pi i},1],                                             
        [-i, e^{\frac{\pi i}{2n}}], [j, j]
    \rangle
= \langle
        [-i, e^{\frac{\pi i}{2n}}], [j, j]
                                             \rangle
\text{\quad (see Section~\ref{sec:translations+flip})}
\end{align*}
We have to choose different conjugations depending on the value of $n$
modulo $4$.
\begin{itemize}
    \item For $\grp{X}_{1,4m-1}^{\textbf{c2mm}} \doteq \grp{.}_{1,8m-2}^{(2m-1)}$,
        we do conjugation by $h_1=[1-j, 1]$:
        \begin{align*}
          &[\tfrac{1+j}2, 1][-i, e^{\frac{\pi i}{8m-2}}][1-j, 1] =
             [k, e^{\frac{\pi i}{8m-2}}]
            = [k, e^{\frac{i(14m-3)\pi}{8m-2}}] 
              = [k, -i] [1, e^{\frac{i2\pi}{4m-1}}]^m \in G_1
          \\
          &[\tfrac{1+j}2, 1][j, j][1-j, 1] = [j, j] = [i, k][-k,i] \in G_1
        \end{align*}
    \item For $\grp{X}_{1,4m-3}^{\textbf{c2mm}} \doteq \grp{.}_{1,8m-6}^{(2m-2)}$,
        we do conjugation by $h_2=[1+j, 1]$:
        \begin{align*}
            &[\tfrac{1-j}2, 1][-i, e^{\frac{\pi i}{8m-6}}][1+j, 1] = [-k, e^{\frac{\pi i}{8m-6}}]
            = [j, j] [1, e^{\frac{i2\pi}{4m-3}}]^{m-1} [i, k] \in G_1
          \\
          &[\tfrac{1-j}2, 1][j, j][1+j, 1] = [j, j] = [i, k][-k,i] \in  G_1
        \end{align*}
      \end{itemize}

We can also study this transformation geometrically: What happens to
the torus under this coordinate transformation? On which other torus
do the glide reflections of $\grp{X}_{1,n}^{\textbf{c2mm}}$
appear as torus translations?
Indeed, there is another simple equation for a
Clifford torus that is commonly used.
We can transform our equation for the torus $\T$ as follows:
\begin{align}
  x_1^2+x_2^2 &= x_3^2+x_4^2\nonumber\\
  x_2^2-x_4^2 &= x_3^2-x_1^2\nonumber\\
  (x_2-x_4)(x_2+x_4) &= (x_3+x_1)(x_3-x_1)
                       \label{clif2a}\\
  \tilde x_2\tilde x_4&=\tilde x_1\tilde x_3,
                        \label{clif2}
\end{align}
with transformed coordinates $(\tilde x_1,\tilde x_2,\tilde x_3,\tilde x_4)$.
This is, for example, how the torus is introduced in
Coxeter~\cite[Eq.~(4.41)]{Cox-complex},
who has a separate section on
``the spherical torus''~\cite[\S 4.4, p.~35--37]{Cox-complex}.

Now, the coordinate change from~\eqref{clif2a} to~\eqref{clif2} is
precisely what the
transformation $h_1=[1-j,1]$
in our example achieves:
$[1-j,1]$
maps the quaternion units $(1,i,j,k)\equiv(x_1,x_2,x_3,x_4)$ to
$(1 + j, i - k, -1 + j, i + k)\equiv
(x_1+x_3,x_2-x_4,-x_1+x_3,x_2+x_4)=
(\tilde x_1,\tilde x_2,\tilde x_3,\tilde x_4)$.
Many conjugations in Table~\ref{tab:duplications} are of this form.

      The reason why we have chosen the example
\eqref{eq:DUP}
      for manual confirmation is that it corresponds to
one of two duplications in the Conway-Smith classification that are not literally mentioned there:
\begin{align*}
+\tfrac14[D_4  \times\bar D_{4n}]
&\doteq+\tfrac14[D_4  \times D_{4n}^{(1)}]\text{ for odd $n$.}
\\
\pm\tfrac14[D_4  \times\bar D_{4n}] &
\doteq\pm\tfrac14[D_4  \times D_{4n}^{(1)}]
\end{align*}
The second %
equality
appears in Table~\ref{tab:duplications} as %
$\grp{X}^{\mathbf{p2mm}}_{2,2m}$ for odd $m$ and
$\grp{X}^{\mathbf{p2gm}}_{2,2m}$ for even~$m$.
The reason behind these duplications is discussed in %
Section~\ref{sec:index4}.

\subsection{Comparison with the classification of Conway and Smith}
\label{sec:compare-CS}

Looking at the right column of  Table~\ref{tab:overview}, we see that
our classification and the classification of Conway and Smith~\cite{CS}
have some similarity
in the rough categorization.
For example the ``mixed'' groups of type $[C\times D]$ are
the torus swap groups (type $\grp/$).
 In the finer details, however, the two classifications
 are
 often quite at odds with each other.
Groups that come from one geometric family correspond to different
classes in the CS classification from the algebraic viewpoint,
depending on parity conditions. On the other hand,
some groups that belong together algebraically appear in different categories
of our classification.

While we acquired some understanding
of the classic classification of the toroidal groups
according to Conway and Smith~\cite{CS}, in particular, of the simplest
 case of the torus translation groups
  (type $\grp1$, corresponding to $[C\times C]$,
 see Appendix~\ref{conway-smith}),
 most entries in the right column of Table~\ref{tab:overview}
 were filled with the help of a computer, by generating the groups
 from the specified generators and
 comparing them by the fingerprints described in
Section~\ref{Fingerprinting},
 and recognizing patterns. %

One reason for the difficulty is the distinction between haploid and diploid groups,
a term borrowed from biology by Conway and Smith~\cite{CS}.
A group is \emph{diploid} if it contains the central reflection $-\id$;
otherwise, it is \emph{haploid}.\footnote
{Threlfall and Seifert~\cite[\S\,5]{ThrS-I} used the terms
  \emph{zweistufig} and \emph{einstufig} for these groups.}
In the classic classification,
the diploid groups arise easily, but the haploid groups
must be specially constructed as index-2 subgroups of  diploid groups.
Thus, 
the presence or absence of $-\id$ appears at the very beginning of
the classic classification by  quaternions.
In the notation of \cite{CS}, diploid and haploid groups are
distinguished by the prefix $\pm$ and $+$.

For our geometric construction of the toroidal groups, this
distinction is ephemeral.
The central reflection $-\id$ is the torus translation
$R_{\pi,\pi}$ in the center of the parameter square.
It depends on some parity conditions of the translation parameters whether this
element belongs to~$G_\Box$.
(For example, one can easily work out
from Figure~\ref{fig:grid-parameters} that
the groups
$\grp|^{\textbf{pm}}$ and
$\grp|^{\textbf{pg}}$
are diploid if $m$ and $n$ are even.
The groups
$\grp|^{\textbf{cm}}$
are diploid if $m$ and $n$ have the same parity.)

In elliptic geometry, where opposite points of $S^3$ are identified,
the distinction between haploid and the corresponding diploid groups
disappears, or in other words, only diploid groups play a role in
elliptic space.

\section{The polyhedral groups}
\label{sec:polyhedral}

We will now explain the polyhedral groups,
which are related to the regular 4-dimensional polytopes. 
The regular 4-dimensional polytopes have a rich and beautiful structure.
They and their symmetry groups have been amply discussed in the literature,
see for example \cite[Chapters VIII and~XIII]{Coxeter},
\cite[\S26, \S27]{duval},
and therefore we will be brief, except that we
study in some more detail the groups that come in enantiomorphic pairs.
Table~\ref{tab:polyhedral} gives an overview,\footnote
{In Du Val's enumeration of the achiral groups \cite[p.~61]{duval},
\label{duval-41-42}
  the
  descriptions of the orientation-reversing elements of the groups
\#41 $(T/V;T/V)^*$ and
\#42 $(T/V;T/V)^*_-$
  are
  swapped by mistake.
  We follow Goursat and Hurley and go with
  the convention that
the group with the more natural choice of elements should be associated to
the name
without a distinguishing subscript.
Du Val himself,
in the detailed
discussion of these groups
 \cite[%
p.~73]{duval},
follows the same (correct) interpretation. %
}
 and
Table~\ref{tab:polyhedral-appendix}
in Appendix~\ref{sec:polyhedral-and-axial-table}
lists these groups
with
generators
and cross references to other classifications.

We mention that pictures of the cube, the 120-cell, the 24-cell, and the
bitruncated 24-cell (also known as the 48-cell, defined in
Section~\ref{48-cell}) arise among the illustrations for the tubical
groups, see Section~\ref{small-n}.

\subsection{The Coxeter notation for groups}
For the geometric description of the groups, we will use the notations
of Coxeter, with adaptations by Conway and Smith~\cite[\S 4.4]{CS}.

In the basic Coxeter group notation,
a sequence of $n-1$ numbers
$[p, q, \ldots , r, s]$ stands
for the symmetry group of the $n$-dimensional
   polytope $\{p, q, \ldots , r, s\}$.
 This is generated by $n$ reflections $R_1,\ldots , R_n$.
Each reflection is its own mirror: $(R_i)^2=1$, and   
   any two adjacent reflections generate a rotation whose order is
   specified in the sequence:
   $(R_1R_2)^p = (R_2R_3)^q = \cdots = (R_{n-1}R_n)^s=1$.
   Nonadjacent mirrors are perpendicular:
   $R_iR_j=R_jR_i$ for $|i-j|\ge 2$.

   $G^+$ denotes the chiral part of the group $G$, which contains products of an even number
   of reflections.
  When just one of the numbers $p, q, \ldots , r, s$ is even, say that between $R_k
  $ and $R_{k+1}$, there are three further subgroups.
  The two subgroups
                    $[^+ p, q, \ldots , r, s]$ and $[p, q, \ldots , r, s^+ ]$
consist of words that use respectively
                       $R_1 , \ldots , R_k$ and $R_{k+1} , \ldots , R_n$
an even number of times. Their intersection is the index-4 subgroup
$[^+ p, q, \ldots , r, s^+]$. Coxeter's original notation
for $[^+ p, q, \ldots]$ is
$[p^+, q, \ldots]$.

A
second pair of brackets, like in $[[3,3,3]]$, indicates a swap between
a polytope and its polar, following~\cite{semiregular2}.
Some further extensions of the notation will be needed for the axial
groups in Section~\ref{sec:axial}, see Table~\ref{tbl:axial_full}.
In some cases, we have
extended the Coxeter notations in an ad-hoc manner, allowing us to
avoid other ad-hoc extension of~\cite{CS}.

\begin{table}[htb]
  \centering
  \setlength{\extrarowheight}{2.1pt}
      \begin{tabular}[t]{|lllr|l|}\hline
    CS name & Du Val \# and
              name& Coxeter name& \llap{ord}er &method\\
    [2pt]\hline \multicolumn 5{|l|}{symmetries of the 120-cell $Q_{120}=\{5,3,3\}$ / the 600-cell $P_{600}=\{3,3,5\}$}\\[1pt]\hline
$\pm[I\times I]\cdot 2$& 50. $(I/I;I/I)^*$& $[3,3,5]$&{14400}&\\
$\pm[I\times I]$& 30. $(I/I;I/I)$& $[3,3,5]^+$&{7200}&chiral part\\
$\pm[I\times O]$& 29. $(I/I;O/O)$& $[[3,3,5]^+_{\frac15L}]$&{2880}&inscribed polar \& swap\\
$\pm[O\times I]$& 29. $(O/O;I/I)$& $[[3,3,5]^+_{\frac15R}]$&{2880}&inscribed polar \& swap\\
$\pm[I\times T]$& 24. $(I/I;T/T)$& $[3,3,5]^+_{\frac15L}$&{1440}&inscribed polar\\
$\pm[T\times I]$& 24. $(T/T;I/I)$& $[3,3,5]^+_{\frac15R}$&{1440}&inscribed polar\\
[2pt]\hline \multicolumn 5{|l|}{symmetries of the 24-cell $P_T=\{3,4,3\}$ and its polar 24-cell $P_{T_1}$}\\[1pt]\hline
$\pm[O\times O]\cdot 2$& 48. $(O/O;O/O)^*$& $[[3,4,3]]$&{2304}&\\
$\pm[O\times O]$& 25. $(O/O;O/O)$& $[[3,4,3]]^+$&{1152}&chiral part\\
$\pm\frac1{2}[O\times O]\cdot 2$& 45. $(O/T;O/T)^*$& $[3,4,3]$&{1152}&nonswapping\\
$\pm\frac1{2}[O\times O]\cdot \bar2$& 46. $(O/T;O/T)^*_-$& $[[3,4,3]^+]$&{1152}&swap with mirror\\
$\pm\frac1{2}[O\times O]$& 28. $(O/T;O/T)$& $[3,4,3]^+$&{576}&chiral \& nonswapping\\
[2pt]\hline
$\pm[T\times T]\cdot 2$& 43. $(T/T;T/T)^*$& $[3,4,3^+]$&{576}&edge orientation\\
$\pm[O\times T]$& 23. $(O/O;T/T)$& $[[^+3,4,3^+]]_L$&{576}&diagonal marking\\
$\pm[T\times O]$& 23. $(T/T;O/O)$& $[[^+3,4,3^+]]_R$&{576}&diagonal marking\\
$\pm[T\times T]$& 20. $(T/T;T/T)$& $[^+3,4,3^+]$&{288}&2 dual edge orientations\\
[2pt]\hline \multicolumn 5{|l|}{symmetries of the hypercube $\{4,3,3\}$ / the cross-polytope $\{3,3,4\}$}\\[1pt]\hline
$\pm\frac1{6}[O\times O]\cdot 2$& 47. $(O/V;O/V)^*$& $[3,3,4]$&{384}&\\
$\pm\frac1{6}[O\times O]$& 27. $(O/V;O/V)$& $[3,3,4]^+$&{192}&chiral part\\
$\pm\frac1{3}[T\times T]\cdot 2$& 41. $(T/V;T/V)^*$& $[^+3,3,4]$&{192}&even permutations\\
$\pm\frac1{3}[T\times \overline T]\cdot 2$& 42. $(T/V;T/V)^*_-$& $[3,3,4^+]$&{192}&2-coloring\\
$\pm\frac1{3}[T\times T]$& 22. $(T/V;T/V)$& $[^+3,3,4^+]$&{96}&2-coloring \& chiral\\
[2pt]\hline \multicolumn 5{|l|}{symmetries of the simplex $\{3,3,3\}$ and its polar}\\[1pt]\hline
$\pm\frac1{60}[I\times \overline I]\cdot 2$& 51. $(I^\dag/C_2;I/C_2)^{\dag* }$& $[[3,3,3]]$&{240}&\\
$\pm\frac1{60}[I\times \overline I]$& 32. $(I^\dag/C_2;I/C_2)^{\dag}$& $[[3,3,3]]^+$&{120}&chiral part\\
$+\frac1{60}[I\times \overline I]\cdot 2_1$& 51\rlap{$'$}. $(I^\dag/C_1;I/C_1)^{\dag*}$& $[3,3,3]$&{120}&nonswapping\\
$+\frac1{60}[I\times \overline I]\cdot 2_3$& 51\rlap{$'$}. $(I^\dag/C_1;I/C_1)^{\dag*}_ -$& $[[3,3,3]^+]$&{120}&swap with mirror\\
$+\frac1{60}[I\times \overline I]$& 32\rlap{$'$}. $(I^\dag/C_1;I/C_1)^{\dag}$& $[3,3,3]^+$&{60}&chiral \& nonswapping\\
[2pt]\hline \end{tabular}

  \caption[The 25 polyhedral groups]{The polyhedral groups}
  \label{tab:polyhedral}
\end{table}

\subsection{Strongly inscribed polytopes}
\label{sec:strongly-inscribed}

We say that a polytope $P$ is \emph{strongly inscribed} in a
polytope $Q$ if every vertex of $P$ is a vertex of~$Q$, and every
facet of $Q$ contains a facet of $P$. Figure~\ref{fig:strongly-inscribed}
shows two three-dimensional examples.
This relation between $P$ and $Q$ is reversed under polarity:
With respect to an origin that lies inside $P$, the polar polytope
$Q^\Delta$ will be strongly inscribed in $P^\Delta$.

\begin{figure}[htb]
\begin{subfigure}{0.5\textwidth}
\centering
\includegraphics{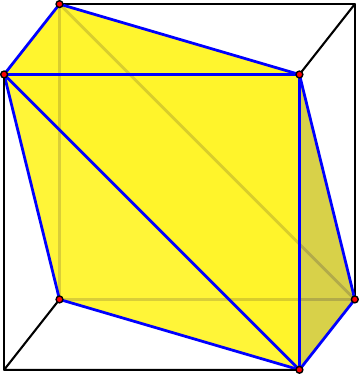}
\end{subfigure}%
\begin{subfigure}{0.5\textwidth}
\centering
\includegraphics{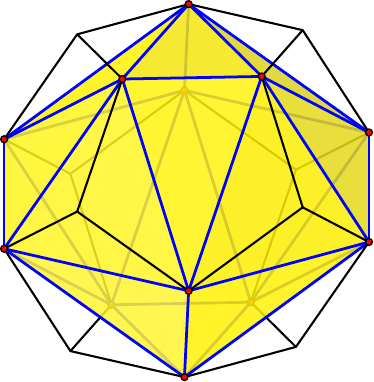}
\end{subfigure}
\caption{A cube with a
strongly inscribed (non-regular) octahedron (left).
A dodecahedron with a strongly inscribed (non-regular) icosahedron (right).}
  \label{fig:strongly-inscribed}
\end{figure}

In four dimensions, we will show two instances of this phenomenon
where a rotated copy of the polar polytope $P^\Delta$ of a polytope $P$ can be
strongly inscribed into $P$.
Among the \emph{regular} polytopes in three dimensions, there are just some
degenerate cases, %
where every facet of $Q$ contains
only an \emph{edge} of $P$:
In a cube $Q$, a regular tetrahedron $P$ can
be inscribed, with the six edges of $P$ on the six square sides of $Q$.
In a dodecahedron $Q$, a cube $P$ can
be inscribed, with its twelve edges on the twelve pentagons of $Q$.
The tetrahedron inscribed in a dodecahedron does not fall in this
category, since its edges go through the interior of the dodecahedron.

\subsection{Symmetries of the simplex}
\label{sec:4-simplex}

The full symmetry group of the 4-simplex is $[3,3,3]$. 
The group $[[3,3,3]]$ %
additionally swaps (by negation) the simplex with its polar.
The chiral versions are $[3,3,3]^+$ and $[[3,3,3]]^+$. %
The group $[[3,3,3]^+]$ %
allows the flip to the polar only in connection
with a reversal of orientation.

\subsection{Symmetries of the hypercube (and its polar, the cross-polytope)}
\label{subsec:hypercube}
The full symmetry group of the hypercube is $[3, 3, 4]$.
It is isomorphic to the semidirect product of
coordinate permutations with sign flips
$\{(\pm1,\pm1,\pm1,\pm1)\} \rtimes S_4$.
This group has four subgroups.

The cube has a natural 2-coloring
of the vertices that gives alternating colors
to adjacent vertices. One can check that
the vertices of each color form a cross-polytope. This cross-polytope is
strongly inscribed in the %
cube: %
Each facet of the hypercube contains exactly one
(tetrahedral) facet of that cross-polytope.
The subgroup $[3,3,4^+]$ contains those elements
that preserve the 2-coloring. Equivalently, these are the elements
that have an even number of sign changes.

The subgroup $[^+3,3,4]$ contains those elements
that have an even permutation of coordinates.
It is isomorphic to
$\{(\pm1,\pm1,\pm1,\pm1)\} \rtimes A_4$.
The subgroup $[^+3,3,4^+]$ is their intersection.
The subgroup $[3,3,4]^+$ contains the \OP\ transformations.
These are the transformations where the parity of
the sign changes matches the parity of the permutation.

It is interesting to note that
the 3-dimensional group $[3,4]$ closely mirrors the picture for $[3,3,4]$,
see Table~\ref{tab:cube}.
Both in three and four dimensions,
the ``half-cube'' is itself a regular polytope: %
in 3 dimensions, it is the regular tetrahedron,
while in 4 dimensions, it is the cross-polytope.
The subgroup $[3,4^+]=TO$ preserves the 2-coloring of the vertices,
i.e. it contains all symmetries of the tetrahedron.
Its subgroup $[^+3,4^+]=+T$ contains the \OP\ symmetries of the
tetrahedron.
The group $[^+3,4]=\pm T$ contains the \OP\ symmetries of the
tetrahedron together with its central reflection.
It is also characterized as those symmetries that subject the three space axes
to an even permutation.
The group $[3,4]^+$ contains all \OP\ transformations in $[3, 4]$.
For the groups $+T$ and $TO$ we have used alternate Coxeter names,
which are equivalent to %
the standard ones, in order to highlight the analogy with 4~dimensions,
cf.~\cite[p.~390]{Symmetries}.

\begin{table}[htb]
  \centering
  \begin{tabular}{|l|r|@{\,}|l|r||l|}
    \hline
    4 dimensions & order & 3 dimensions& order & description\\\hline
    $[3,3,4]$&384&$[3,4]=\pm O$&48&the full symmetry group\\
    $[3,3,4]^+$&192&$[3,4]^+=+O$&24&chiral part (preserves orientation)\\
    $[^+3,3,4]$&192&$[^+3,4]=\pm T$&24&even permutation of coordinates\\
    $[3,3,4^+]$&192&$[3,4^+]=[3,3]=TO$&24&preserves the 2-coloring\\
    $[^+3,3,4^+]$&96&$[^+3,4^+]=[3,3]^+=+T$&12&all three constraints
                                                above %
    \\
    \hline
  \end{tabular}
  \caption{Analogy between symmetries of the four-dimensional and three-dimensional
     cube}
  \label{tab:cube}
\end{table}

\subsection{Symmetries of the 600-cell (and its polar, the 120-cell)}
\label{sec:600-cell}

The 120 quaternions $2I$ form the vertices of
a 600-cell $%
P_{600}
=\{3,3,5\}$.
These quaternions are the centers of the 120 dodecahedra of
the polar 120-cell $Q_{120}=\{5,3,3\}$, which has 600 vertices.
The full symmetry group of $P_{600}$ (or $Q_{120}$) is $[3, 3, 5]$.
Its chiral version is $[3,3,5]^+$.

The group has four interesting subgroups, which come in
enantiomorphic versions.
Under the left rotations by elements of $2I$, or in other words, under
the group $\pm[I\times C_1]$, the 600 vertices of $Q_{120}$
decompose into five orbits, as shown by the
five labels $A,B,C,D,E$ for the cell $F_0$ in
Figure~\ref{fig:inscribed_tetrahedra_L}, cf.~\cite[Figure~22,
p.~84]{duval}.
We can regard this as a 5-coloring of the vertices.
(The points of each color are labeled $X,X',X'',X'''$ according to the
horizontal levels in this picture, but this grouping has otherwise no significance.)
One can indeed check that the mapping from a pentagonal face to the
opposite face with a left screw by $\pi/5$, as effected by the
elements of $\pm[I\times C_1]$, preserves the coloring.

The vertices of one color form
a regular tetrahedron inscribed in a regular dodecahedron, and
there are thus five ways inscribe such a ``left'' tetrahedron
in a regular dodecahedron.
There is an analogous ``right'' 5-coloring by the orbits under
$\pm[ C_1\times I]$, and correspondingly, there are five
 ways of inscribing a ``right'' tetrahedron
 in a regular dodecahedron.
One such tetrahedron is shown in Figure~\ref{fig:inscribed_tetrahedra_R}.\footnote
{The unions of these five or ten tetrahedra inside a dodecahedron form nice
nonconvex star-like polyhedral compounds,
see \cite[Figures~14 and 15a--b]{duval}. %
See also %
\url{https://blogs.ams.org/visualinsight/2015/05/15/dodecahedron-with-5-tetrahedra/}
from the AMS blog ``Visual Insight''.
}
The left and right tetrahedra
are mirrors of each other, and they
can be distinguished by
looking at the
paths of length 3 on the dodecahedron between %
 vertices of a tetrahedron: These paths are either S-shaped zigzag paths
 (for left tetrahedra) or
they have the shape of an inverted~S (for right tetrahedra).

\begin{figure}
\begin{subfigure}{.5\textwidth}
\centering
    \includegraphics{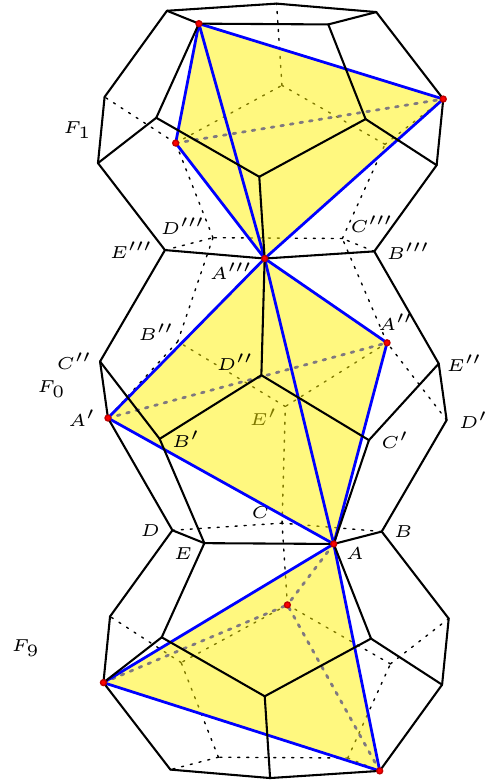}
    \caption{left tetrahedra}
    \label{fig:inscribed_tetrahedra_L}
\end{subfigure}
\begin{subfigure}{.5\textwidth}
\centering
    \includegraphics{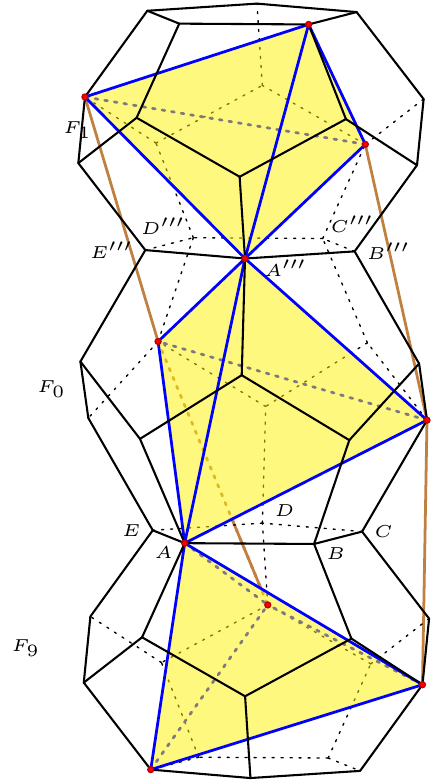}
    \caption{right tetrahedra}
    \label{fig:inscribed_tetrahedra_R}
\end{subfigure}%
\caption{A sequence of inscribed tetrahedra in three successive
  dodecahedra %
  of the 120-cell $Q_{120}$. The red vertices form a 600-cell
  $P'_{600}$.
This is an orthogonal projection to the tangent space in the center of
the middle cell, and this is why the adjacent cells are foreshortened.
}
\label{fig:inscribed_tetrahedra}
\end{figure}

Every color class consists of the points $2I\cdot p_0$ for some starting
point $p_0$, and hence it forms a rotated copy $P_{600}'$ the 600-cell
$P_{600}$. This polytope is strongly inscribed in~$Q_{120}$:
For each dodecahedron of $Q_{120}$, there is a unique
left rotation in $\pm[I\times C_1]$ mapping $F_0$ to this
dodecahedron,
and in this way we get 120 images of the starting tetrahedron.
Figure~\ref{fig:inscribed_tetrahedra_L} shows these
tetrahedra in  three adjacent dodecahedra.
(As a sanity check, one can perform a small calculation:
 A vertex is shared by four tetrahedra---one
tetrahedron in each of the four dodecahedra meeting in the vertex---,
and this gives a consistent vertex count, since every tetrahedron has
four vertices and $120\cdot4/4=120$.)

The red points in Figure~\ref{fig:inscribed_tetrahedra_R} form part of
an analogous 600-cell
 $P_{600}^{\prime R}$ spanned by right inscribed tetrahedra.
Some additional edges of this $P_{600}^{\prime R}$, which don't lie in the three
dodecahedra that are shown, are drawn in brown.

The group $%
\pm[I\times T]$ consists of those
symmetries of that simultaneously preserve the 120-cell $Q_{120}$ and
 its strongly inscribed ``left'' 600-cell $P_{600}'$.  To see
this, consider the dodecahedral cell $F_0$ that is centered at the
quaternion~$1$. As mentioned, each left multiplication by an element $2I$ maps
$F_0$, together with its inscribed tetrahedron $AA'A''A'''$ to a
unique dodecahedral cell of $Q_{120}$ with the corresponding tetrahedron.
To understand the full group, we have to consider those group elements
that keep $F_0$ fixed.
$\pm[I\times T]$ consists of the elements $[l,r]$ with $(l,r)\in 2I \times 2T$.
The transformation $[l,r]$ keeps $F_0$ fixed iff it maps $1$ to $1$,
and this is the case iff $l=r$. These elements are the elements
$[r,r]=[r]$ with $r\in 2T$, in other words, they form the tetrahedral
group $\pm T$. And indeed, the symmetries of $F_0$ that keep the
tetrahedron $AA'A''A'''$ invariant form a tetrahedral group.

We chose $[3, 3, 5]_{\frac15L}^+$ as an ad-hoc extension of Coxeter's notation
for the group $\pm[I\times T]$, to indicate a $1/5$ fraction of the group
$[3, 3, 5]^+$.

Now, there is also the original 600-cell  $P_{600}$, the polar of the
$Q_{120}$,
having one vertex in the center of each dodecahedron.
This gives rise to a larger group
 $[[3, 3, 5]_{\frac15L}^+]=\pm[I\times O]$
 where the two 600-cells  $P_{600}$ and
 $P_{600}'$ (properly scaled) are swapped.
This group is not a subgroup of any other 4-dimensional point group.

When the starting point $s$ is chosen in the center of
the dodecahedral cell
of $Q_{120}$, the polar orbit polytope
of
this group %
has 240 cells.
Figure~\ref{fig:IxO} shows such a cell $C$. %
The points of the orbit
 closest to~$s$
are four vertices of the
dodecahedron (say, those of color $A$, the red points in
Figure~\ref{fig:inscribed_tetrahedra_L}).
They form a tetrahedral cell of
$P_{600}'$, and
they are responsible for the rough tetrahedral shape of~$C$.
The centers of the twelve neighboring dodecahedra in
$Q_{120}$ give rise to the twelve small triangular faces, which are
the remainders of the twelve pentagons of the original dodecahedral
cell, when the polar is not present.
In addition, there are four neighboring cells that are adjacent
through hexagonal faces,
opposite the large 12-gons.
They centered at
vertices of
$P_{600}'$.
Two of these are shown as red points
in Figure~\ref{fig:inscribed_tetrahedra_L},
the point adjacent to $C$
in the lower cell $F_9$,
and the point adjacent to $D'''$
in the upper cell $F_1$.
The cell has chiral tetrahedral symmetry $+T$. In particular, it is
not mirror-symmetric.
In \cite[Figure~9]{dunbar94_f2}, this cell is shown together with a
fundamental domain inside it.
Incidentally, this cell (and the orbit polytope) coincides with that of the tubical group
$\pm[I\times C_4]$ when the starting point is chosen on a two-fold rotation
center
(Figure~\ref{fig:IxCn_2fold}).

\begin{figure}[htb]
  \centering
\includegraphics{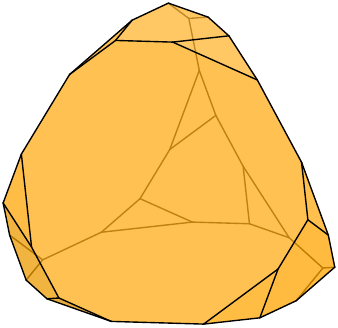}
  \caption{A cell $C$ of the polar orbit polytope of the group~$\pm[I\times O]$}
  \label{fig:IxO}
\end{figure}

If we use the ``right'' 5-coloring
we get the corresponding groups $[3, 3, 5]_{\frac15R}^+=\pm[T\times I]$
and $[[3, 3, 5]_{\frac15R}^+]=\pm[O\times I]$.
See Figure~\ref{fig:inscribed_tetrahedra_R}.
These four groups come in two enantiomorphic pairs. The two
corresponding groups are mirrors of each other. (They are therefore
\emph{metachiral groups} in the terminology of Conway and
Smith~\cite[\S 4.6]{CS}.)

\subsection{Symmetries of the 24-cell}
\label{sec:24-cell}
The set of 24 quaternions of $2T$ form the vertices of a regular
24-cell $P_T$.
The complete symmetry group of $P_T$ is $[3, 4, 3]$, and its chiral version is $[3, 4, 3]^+$.

The points of $P_T$ can be 3-colored:
There are 8 vertices of $P_T$
whose coordinates are the permutations of $(\pm1,0,0,0)$.
They form a cross-polytope.
The 16 remaining vertices are of the form $(\pm1/2,\pm1/2,\pm1/2,\pm1/2)$.
 They are the vertices of a 4-cube, and they
 can be naturally divided into two color groups of 8, as mentioned
 in Section~\ref{subsec:hypercube}.
In total, we have three groups of 8 vertices, which we interpret as
a \emph{3-coloring} of the vertices by the colors $a,b,c$, see
Figure~\ref{fig:truncated-cube}a.
Every triangular face contains vertices from all three colors. Thus,
every symmetry of $P_T$ induces a permutation of the colors.

We can look at those symmetries for which the permutations of the
colors is even. In other words, besides the identity, we allow only
cyclic shifts. These form the subgroup $[3,4,3^+]$.
Another way to express this is to establish an \emph{orientation of the edges} according to
some cyclic ordering of the colors $a\to b\to c\to a$ (a \emph{coherent orientation} %
\cite[\S 8.3]{Coxeter}).
The subgroup $[3,4,3^+]$ consists of those elements that preserve this
edge orientation.
(This is analogous to the pyritohedral group $\pm T$ in three dimensions,
which can also be described as preserving the orientation of the edges
of the octahedron shown in Figure~\ref{fig:truncated-cube}a.)

The 24-cell is a self-dual polytope. In fact, the vertices of the
polar polytope
 $P_{T_1}$ (properly scaled) are 
 the quaternions in the coset of $2T$ in $2O$.
If we add to $[3,4,3]$ the symmetries that swap $P_T$ and $P_{T_1}$,
we get the group $[[3,4,3]]$, the symmetry group of the joint configuration
$P_O=P_{T} \cup P_{T_1}$. Its chiral version is $[[3, 4, 3]]^+$.
The subgroup $[[3, 4, 3]^+]$ contains the symmetries that
exchange $P_T$ and $P_{T_1}$ only in
combination with a reversal of orientation.
This group is interesting, because it is achiral, but it contains no reflections.

The polar polytope also has a three-coloring of its vertices.
(One can give the partition explicitly in terms of the coordinates, as
for $P_T$:
The vertices of $P_{T_1}$ are the centers of the facets of~$P_T$, properly scaled,
and their coordinates $(x_1,x_2,x_3,x_4)$ are all permutations of the coordinates 
$(\pm1,\pm1,0,0)/\sqrt2$. The three color classes are characterized by
the
condition
$|x_1|=|x_2|$,
$|x_1|=|x_3|$, and
$|x_1|=|x_4|$, respectively.)
We can interpret this 3-coloring as
a 3-coloring of the \emph{cells} of $P_T$, which we denote by $A,B,C$.
The group $[^+3,4,3]$ contains those symmetries of $P_T$
for which the permutation of the colors of the \emph{cells} is even.
This group %
is of course geometrically the same as $[3,4,3^+]$, but we can also
have both conditions: $[^+3,4,3^+]$.

\begin{figure}[htb]
  \centering
  \includegraphics{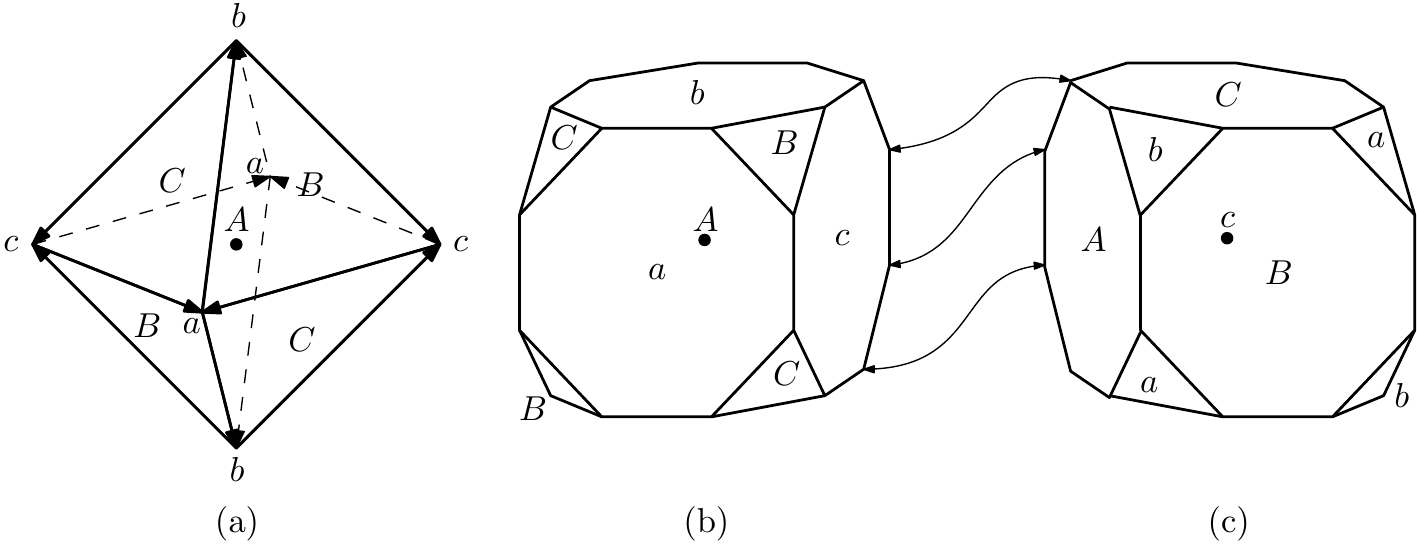}
  \caption{(a) An octahedral cell of the 24-cell with a consistent
    edge orientation. (b) The 48-cell
    consists of 48 truncated cubes.}
  \label{fig:truncated-cube}
\end{figure}

\begin{figure}[htb]

  \begin{minipage}[b]{0.46\linewidth}
  \centering
  \includegraphics{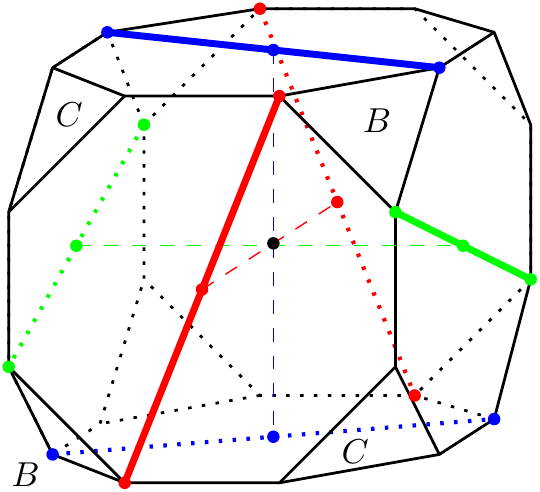}
  \caption{Decoration of the truncated cube by diagonals.}
  \label{fig:24-cell-decorated}
\end{minipage}
\hfill
  \begin{minipage}[b]{0.46\linewidth}
  \centering
  \includegraphics{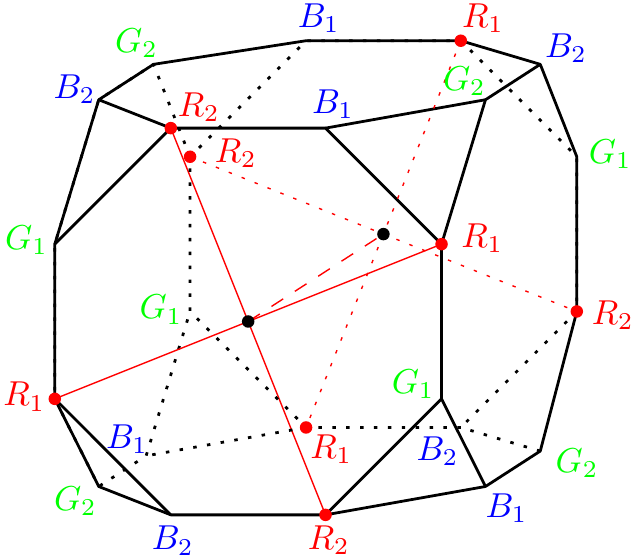}
  \caption{The 6 orbits of the vertices under $\pm[O\times C_1]$
    (left multiplication with $2O$)
  }
  \label{fig:truncated-cube-coloring}   
  \end{minipage}
\end{figure}

\begin{figure}[htb]
  \centering
  \includegraphics{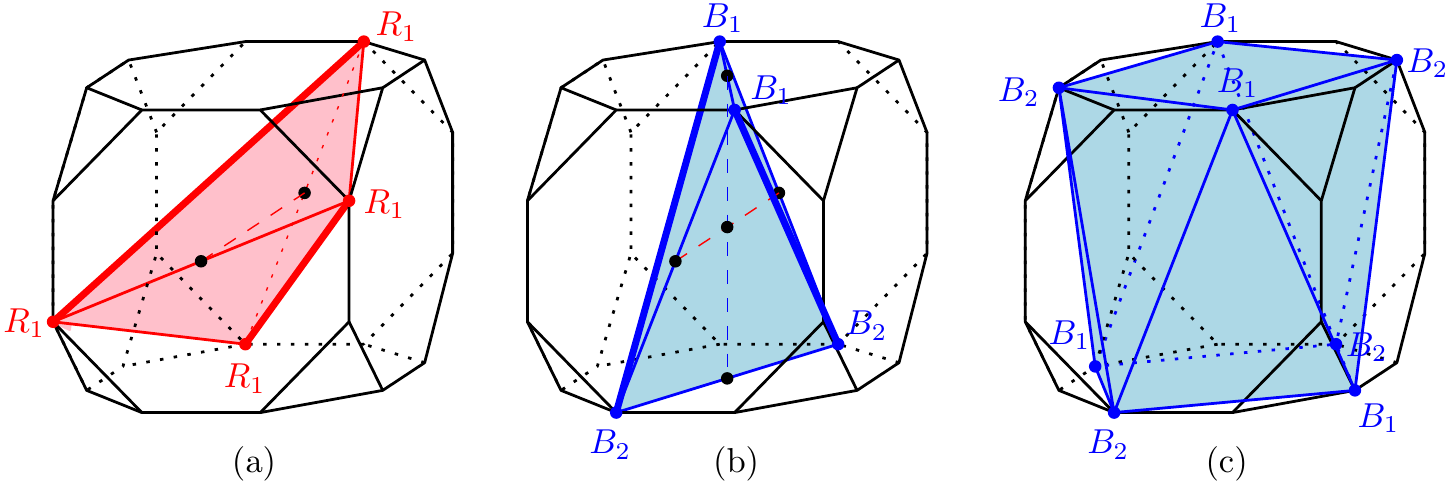}
  \caption{Facets inscribed in the truncated cube}
  \label{fig:inscribed}
\end{figure}

\subsubsection{A pair of \ena\ groups}
\label{48-cell}
Finally, we have two more groups, which are mirrors of each other.
To understand these groups, let us look at
 the polar orbit polytope of $P_O = P_{T} \cup P_{T_1}$:
The octahedral cells of the 24-cell shrink to
truncated cubes with 6
regular octagons and 8 triangles as faces,
see Figure~\ref{fig:truncated-cube}b.
This %
polytope
is sometimes called the bitruncated 24-cell, or
truncated-cubical
tetracontaoctachoron.
We will simply refer to it as
the \emph{48-cell}.
The small triangles are remainders from the triangular faces of the original
octahedral cells of the 24-cell, which are centered at the points $P_T$.

Figure~\ref{fig:truncated-cube}b
shows a cell of color $A$.
The triangles lead to adjacent cells,
colored $B$ or $C$, and we have labeled the
triangles accordingly. The octagons %
lead to cells centered at points of $P_{T}$, and we have labeled
them with the corresponding color $a$, $b$, or~$c$.

Figure~\ref{fig:truncated-cube}c shows an adjacent ``dual'' cell of the 48-cell,
centered at a point of color $c$. Note that these two cells are not
attached in a straight way, but by a %
screw of
$45^\circ$.
We can enforce the screw to be a left screw by decorating each of the six
octagonal faces with a diagonal, as shown in
Figure~\ref{fig:24-cell-decorated}.
The group $\pm[O \times C_1]$ will map one selected cell to each cell
by a unique left multiplication with an element of $2O$ and hence will
carry the diagonal pattern to every truncated cube of the 48-cell.
The diagonals on adjacent cells match: A left rotation that maps a
cell to the adjacent cell performs a left screw by
$45^\circ$, and one can check in 
Figure~\ref{fig:24-cell-decorated} that the screw that maps an octagon
to the opposite octagon while maintaining the diagonal is a left screw.

The group $\pm[O \times T]$ is the group that preserves the set of
diagonals (ignoring the colors). This can be confirmed as in the case
$\pm[I \times T]$
in Section~\ref{sec:600-cell}: The group that fixes a cell should be
the tetrahedral group $+T$, and indeed,
the diagonal pattern of Figure~\ref{fig:24-cell-decorated} has
tetrahedral symmetry: The diagonals connect only the $B$-triangles,
and the $B$-triangles form a tetrahedral pattern.
We have chosen the ad-hoc extension of Coxeter's notation $[[^+3,4,3^+]]_L$
for the group $\pm[O\times T]$ to indicate that it extends the operations
$[^+3, 4,3^+]$ by a swap between $P_T$ and the polar polytope
$P_{T_1}$, and this swap is effected by \emph{left} rotations.

Of course, there is a mirror pattern
of Figure~\ref{fig:24-cell-decorated}, which leads
to the mirror group
 $\pm[T \times O]=[[^+3,4,3^+]]_R$, and these two groups are \ena.

 \paragraph{Analogies with three dimensions.}
 As pointed out by Du Val \cite[p.~71]{duval}, 
 there is a strong analogy between the symmetries of the different self-dual
 polytopes in three and in four dimensions,
as shown in Table~\ref{tab:simplex}.
The simplex is a self-dual regular polytope,  both in
4 dimensions  (Section~\ref{sec:4-simplex})
and in 3 dimensions.
 In 3 dimensions,
 moreover,
 the simplex and its polar form the cube,
 and thus we have used alternate Coxeter notations
 to highlight the analogy (opposite ones from
 Table~\ref{tab:cube}, where the analogy with the cube is emphasized).
 Only five of the symmetries of the 24-cell and its polar are used.
 
 From the viewpoint of the cross-polytope, one could also match
 the group $\pm[T\times T]\cdot 2 = [3,4,3^+] =  [^+3,4,3]
 $ of order {576} with
 the pyritohedral group
 $\pm T=[^+3,4]$ of order 24, because they are both based on consistent
 edge orientations.

\begin{table}[htb]
  \centering
  \begin{tabular}{|l@{ }r|@{\,}|l@{ }r|@{\,}|l@{}r||l|}
    \hline
    4-simplex & \,order & 24-cell& order & 3-simplex& order & description\\\hline
    $[[3,3,3]]$&240&
    $[[3,4,3]]$&2304&
                     $[[3,3]]=[3,4]=\pm O$&48&all symmetries\\
    $[[3,3,3]]^+$&120&
    $[[3,4,3]]^+$&1152&
                        $[[3,3]]^+=[3,4]^+=+O$&24
                                                                     &chiral part\\
    $[3,3,3]$&120&
    $[3,4,3]$&1152&
                    $[3,3]=TO$&24
                                                                     &nonswapping\\
    $[[3,3,3]^+]$&120&
    $[[3,4,3]^+]$&1152&
                       $[[3,3]^+]=[^+3,4]=\pm T$&24&swap with mirror\\
    $[3,3,3]^+$&60&
    $[3,4,3]^+$&576&
                    $[3,3]^+=+T$&12&chiral \& nonswapping\\
    \hline
  \end{tabular}
  \caption{Analogies between %
    symmetries of self-dual polytopes}
  \label{tab:simplex}
\end{table}

\paragraph{A strongly inscribed polar polytope.} 

The convex hull of the points
$P_O = P_{T} \cup P_{T_1}$ is a polytope with 288 equal tetrahedral
facets,
which we call the \emph{288-cell}.
It is polar to the 48-cell.
We perform the same procedure as in Section~\ref{sec:600-cell}
 and split the vertices of the 48-cell into orbits under the action of
 $\pm[O \times C_1]$. 
We will see that this leads to another instance of polytope with a strongly
inscribed copy of its polar. However, we won't get any new groups.

The 48-cell has 288 vertices, and they are partitioned into
 6 orbits of size 48, as
shown in
Figure~\ref{fig:truncated-cube-coloring}, cf.\
Du Val~\cite[Figure~24, p.~85]{duval}:
There is a natural partition of the colors into three pairs
$R_1,R_2$; $G_1,G_2$; and $B_1,B_2$, according to the
opposite octagons to which the colors belong. (The partition of each
pair into
$R_1$ and $R_2$, etc., is arbitrary.)
Indeed, one can check that the transition from an octagon to the
opposite octagon with a left screw of $45^\circ$ preserves the six colors
(indicated for the red colors by two corresponding crosses.)
Likewise, the transition from a triangle to the opposite triangle
with a left screw of $60^\circ$ preserves the colors.

Now, as in Section~\ref{sec:600-cell}, the points of one color form a
right coset of $2O$, and hence they form a rotated and scaled copy
$P_O'$ of the 288-cell $P_O$.
This polytope is strongly inscribed in the 48-cell: Each truncated
cube of the 48-cell contains one tetrahedron of $P_O'$.
Figure~\ref{fig:inscribed}a shows one such tetrahedron, spanned by the
vertices of color $R_1$.

The geometry of this tetrahedron becomes clearer after
rotating it
by $45^\circ$ around the midpoints of the front and back octagons, as in
Figure~\ref{fig:inscribed}b.
We see that the tetrahedron has four equal sides, whose length is the diagonal of
the octagons, and two  opposite sides of
larger length, equal to the diagonal of a circumscribed square.
The 2-faces %
are therefore
congruent
isosceles triangles. Such a tetrahedron is called a
\emph{tetragonal disphenoid}.\footnote
{The side length of the ``untruncated'' cube is 
$\sqrt8-2\approx 0.8$, which equals the edge length of a circumscribed
8-gon around a unit circle.
Hence the two long edges of the tetrahedra, highlighted in bold, have length
$\sqrt2(
\sqrt8-2) =4-\sqrt8
\approx 1.17$. The four short edges have length
$\sqrt{8(10-\sqrt{98})}\approx 0.9$,
and the edge length of the 48-cell is $6-\sqrt{32}\approx0.34$.}

The symmetry group of the 48-cell together with its strongly inscribed
288-cell $P_O'$ is the tubical group $\pm[O\times D_4]$, because the symmetry group of the disphenoid
inside the truncated cube is only the vierergruppe $D_4$, consisting
of half-turns through edge midpoints.

We can try to start with the rotated tetrahedra of
Figure~\ref{fig:inscribed}b, spanned by two opposite diagonals used
for the decoration in
Figure~\ref{fig:truncated-cube-coloring},
hoping to recover the group $\pm[O\times T]$.
However, this tetrahedron
contains vertices of \emph{two} colors $B_1$ and $B_2$, and its orbit will
thus contain the union of the orbits $B_1$ and $B_2$. Inside each
truncated cube, the convex hull forms a quadratic antiprism, as shown
Figure~\ref{fig:inscribed}c. (The convex hull contains 48 such
antiprisms plus 192 tetrahedral cells, for a total of 240 facets.)

\section{The axial groups}
\label{sec:axial}

These are the finite subgroups of the direct product $\O(3) \times \O(1)$.
The subgroup $\O(1)$ operates on the 4-th coordinate $x_4$,
and we denote its elements by $\O(1)=\{+x_4,-x_4\}$.
Here $+x_4$ is the identity,
and $-x_4$ denotes the reflection of the 4-th coordinate.

Let $G$ be such an axial group.
Let $G_3 \in \O(3)$ be the ``projection'' of $G$ on $\O(3)$.
That is,
$$G_3 := \{\,g\in \O(3)\mid (g,+x_4)\in G \text{ or }  (g,-x_4)\in G \,\}.$$
If $G_3$ itself is a 3-dimensional axial group,
i.e. $G \leqslant \O(2)\times \O(1)$,
then we may call $G$ a \emph{doubly axial group}.
In this case, we prefer to regard $G$ as a toroidal group in 
$\O(2) \times \O(1) \times \O(1) \leqslant \O(2)\times \O(2)$
and classify it as such. %
(These groups are the subgroups of
$\grp+^{\textbf{p2mm}}_{m,2}$.)
Hence from now on,
we assume that $G_3$ is not an axial 3-dimensional group, i.e.,
we assume that $G_3\leqslant O(3)$ is one of the seven \emph{polyhedral} 3-dimensional groups.
These are well-understood, and thus
the axial groups are quite easy to classify.
There are 21 axial groups
(excluding the doubly axial groups), and their full list is given
below in
Table~\ref{tbl:axial_full},
 with references to other classifications from the literature.
Together with the polyhedral groups in
Table~\ref{tab:polyhedral}, these groups exhaust all entries in \cite[Tables 4.2
and 4.3]{CS} except the toroidal groups.
Table~\ref{tab:polyhedral-appendix}
in Appendix~\ref{sec:polyhedral-and-axial-table}
lists them %
with
generators
and cross references to other classifications.

Note that the product $\O(3) \times \O(1)$ used here is different from
the product $\pm[L\times R]$ on which the classic classification is
based. Both are direct products in the group-theoretic sense, but
$\O(3) \times \O(1)$ is a direct sum, a ``Cartesian'' product in a straightforward
geometric sense, consisting of pairs of independent transformations in
orthogonal subspaces, whereas the product $\pm[L\times R]$, which is
specific to $\SO(4)$, refers to the representation $[l,r]$ by pairs of
quaternions, which have by themselves a significance as operations
$[l]$ and $[r]$ in $\SO(3)$.

We will now derive the axial groups systematically.
Let $G_3^{+x_4} \leqslant \O(3)$ be the subgroup of $G_3$ of
those elements that don't negate the 4-th coordinate.
That is,
$$G^{+x_4}_3 := \{\,g\in \O(3)\mid (g,+x_4)\in G\,\}.$$
The subgroup $G^{+x_4}_3$ is either equal to $G_3$,
or it is an index-2 subgroup of $G_3$.

If $G^{+x_4}_3=G_3$, %
there are two cases, which are both easy: we can form the
``pyramidal'' group $G_3\times
\{+x_4\}$, which does not move the 4-th dimension at all, or the full
``prismatic'' group $G_3\times
\{+x_4,-x_4\}$.  This gives two axial groups for each
three-dimensional polyhedral group $G_3\leqslant
\SO(3)$, and they are listed in Table~\ref{tab:axial_pap}, together
with their ``CS names'' following Conway and Smith~\cite{CS}, and
their ``Coxeter names''.

The prismatic groups are never chiral.
The pyramidal group $G_3\times \{+x_4\}$ is chiral iff $G_3$ is:
These are the groups $+I$, $+O$, and $+T$.

\begin{table}[htb]
\setlength{\extrarowheight}{1.8pt}
\centering

  \begin{tabular}[t]{|ccc|llr|llr|}
    \hline
    \multicolumn 3{|c|}{$G_3$}
&   \multicolumn 3{c|}{pyramidal groups $G_3 \times \{+x_4\}$}
&   \multicolumn 3{c|}{prismatic groups  $G_3 \times \{+x_4,-x_4\}$}
\\    \hline
    name %
    & \hbox to 0pt{\hss orbitope\hss}& I.T.& CS name & Cox.%
          & \llap{ord}er & CS  name & Cox.%
                                      \ name
          & \llap{o}rder \\\hline
$\pm I$& $\mathbf{*532}$& $53m$%
&$+\frac1{60}[I\times I]\cdot 2_3$& $[3,5]$ & $120$&$\pm\frac1{60}[I\times I]\cdot 2$& $2.[3,5]$ & $240$\\
$+I$& $\mathbf{532}$& $532$%
&$+\frac1{60}[I\times I]$& $[3,5]^+$ & $60$&$+\frac1{60}[I\times I]\cdot 2_1$& $[3,5]^\circ$ & $120$\\
$\pm O$& $\mathbf{*432}$& $m3m$%
&$+\frac1{24}[O\times O]\cdot 2_3$& $[3,4]$ & $48$&$\pm\frac1{24}[O\times O]\cdot 2$& $2.[3,4]$ & $96$\\
$+O$& $\mathbf{432}$& $432$%
&$+\frac1{24}[O\times O]$& $[3,4]^+$ & $24$&$+\frac1{24}[O\times O]\cdot 2_1$& $[3,4]^\circ$ & $48$\\
$TO$& $\mathbf{*332}$& $\bar43m$%
&$+\frac1{12}[T\times \overline T]\cdot 2_1$& $[3,3]$ & $24$&$+\frac1{24}[O\times \overline O]\cdot 2_1$& $[2,3,3]$ & $48$\\
$\pm T$& $\mathbf{3{*}2}$& $m3$%
&$+\frac1{12}[T\times T]\cdot 2_3$& $[^+3,4]$ & $24$&$\pm\frac1{12}[T\times T]\cdot 2$& $2.[^+3,4]$ & $48$\\
$+T$& $\mathbf{332}$& $23$%
&$+\frac1{12}[T\times T]$& $[3,3]^+$ & $12$&$+\frac1{12}[T\times T]\cdot 2_1$& $[^+3,4]^\circ$ & $24$\\[2pt]
\hline
\end{tabular}

\caption
[The 14 pyramidal and prismatic axial groups]
{Pyramidal and prismatic axial groups (except doubly axial groups)}
\label{tab:axial_pap}
\end{table}

\begin{table}[htb]
\setlength{\extrarowheight}{1.6pt}
\centering
\begin{tabular}[t]
  {|l|llr|l|}\hline
    \multicolumn 5{|c|}{hybrid axial groups}\\
    \hline
    $G_3^{+x_4}$ in $G_3$ &    %
                            CS name %
    & Coxeter name 
    & order & methods
    \\\hline$+I$ in $\pm I$ &$\pm\frac1{60}[I\times I]$ %
&$2.[3,5]^+$ & $120$ & center, chirality\\
$\pm T$ in $\pm O$ &$+\frac1{24}[O\times \overline O]\cdot 2_3$ %
&$[2,3,3]^\circ$ & $48$ & edge orientation\\
$+O$ in $\pm O$ &$\pm\frac1{24}[O\times O]$ %
&$2.[3,4]^+$ & $48$ & center, chirality\\
$TO$ in $\pm O$ &$\pm\frac1{12}[T\times \overline T]\cdot 2$ %
&$2.[3,3]$ & $48$ & center, alternation\\
$+T$ in $\pm T$ &$\pm\frac1{12}[T\times T]$ %
&$2.[3,3]^+$ & $24$ & center, chirality\\
$+T$ in $+O$ &$+\frac1{12}[T\times \overline T]\cdot 2_3$ %
&$[3,3]^\circ$ & $24$ & alternation\\
$+T$ in $TO$ &$+\frac1{24}[O\times \overline O]$ %
&$[2,3,3]^+$ & $24$ & chirality\\[2pt]
\hline
\end{tabular}%

\caption
[The 7 hybrid axial groups]
{Hybrid axial groups (except doubly axial groups)
}
\label{tab:axial_hybrid}
\end{table}

\begin{table}[htbp]
\setlength{\extrarowheight}{1.8pt}
\centering
  \begin{tabular}[t]{|l|llllr|}
    \hline
 \multicolumn 6{|c|}{The 21 axial groups}
\\\hline
 \multicolumn 6{|c|}{pyramidal groups $G_3 \times \{+x_4\}$}
\\\hline $G_3$  &CS name & Du Val \# and name& Cox.%
 & {BBNZW} %
 & \llap{order} \\
\hline
$\pm I$&$+\frac1{60}[I\times I]\cdot 2_3$& 49\rlap{$'$}. $(I/C_1;I/C_1)^*$&  $[3,5]$ & n.cryst.%
& $120$\\
$+I$&$+\frac1{60}[I\times I]$& 31\rlap{$'$}. $(I/C_1;I/C_1)$&  $[3,5]^+$ & n.cryst.%
& $60$\\
$\pm O$&$+\frac1{24}[O\times O]\cdot 2_3$& 44\rlap{$'$}. $(O/C_1;O/C_1)^{*\prime}$&  $[3,4]$ & 25/10%
& $48$\\
$+O$&$+\frac1{24}[O\times O]$& 26\rlap{$'$}. $(O/C_1;O/C_1)'$&  $[3,4]^+$ & 25/03%
& $24$\\
$TO$&$+\frac1{12}[T\times \overline T]\cdot 2_1$& 40\rlap{$'$}. $(T/C_1;T/C_1)^* $&  $[3,3]$ & 24/04%
& $24$\\
$\pm T$&$+\frac1{12}[T\times T]\cdot 2_3$& 39\rlap{$'$}. $(T/C_1;T/C_1)^*_c$&  $[^+3,4]$ & 25/02%
& $24$\\
$+T$&$+\frac1{12}[T\times T]$& 21\rlap{$'$}. $(T/C_1;T/C_1)$&  $[3,3]^+$ & 24/01%
& $12$\\[2pt]
\hline
 \multicolumn 6{|c|}{prismatic groups  $G_3 \times \{+x_4,-x_4\}$}
\\\hline $G_3$  &CS name & Du Val \# and name& Cox.%
 & {BBNZW} %
 & \llap{order} \\
\hline
$\pm I$&$\pm\frac1{60}[I\times I]\cdot 2$& 49. $(I/C_2;I/C_2)^*$&  $2.[3,5]$ & n.cryst.%
& $240$\\
$+I$&$+\frac1{60}[I\times I]\cdot 2_1$& 49\rlap{$'$}. $(I/C_1;I/C_1)^*_-$&  $[3,5]^\circ$ & n.cryst.%
& $120$\\
$\pm O$&$\pm\frac1{24}[O\times O]\cdot 2$& 44. $(O/C_2;O/C_2)^*$&  $2.[3,4]$ & 25/11%
& $96$\\
$+O$&$+\frac1{24}[O\times O]\cdot 2_1$& 44\rlap{$'$}. $(O/C_1;O/C_1)^{*\prime}_-$&  $[3,4]^\circ$ & 25/07%
& $48$\\
$TO$&$+\frac1{24}[O\times \overline O]\cdot 2_1$& 44\rlap{$''$}. $(O/C_1;O/C_1)^{*\prime\prime}_-$&  $[2,3,3]$ & 25/08%
& $48$\\
$\pm T$&$\pm\frac1{12}[T\times T]\cdot 2$& 39. $(T/C_2;T/C_2)^*_c$&  $2.[^+3,4]$ & 25/05%
& $48$\\
$+T$&$+\frac1{12}[T\times T]\cdot 2_1$& 39\rlap{$'$}. $(T/C_1;T/C_1)^*_{c-}$&  $[^+3,4]^\circ$ & 25/01%
& $24$\\[2pt]
\hline
 \multicolumn 6{|c|}{hybrid axial groups $G_3^{+x_4}$ in $G_3$}
\\\hline $G_3^{+x_4}$ in $G_3$  &CS name & Du Val \# and name& Cox.%
 & {BBNZW} %
 & \llap{order} \\
\hline
$+I$ in $\pm I$&$\pm\frac1{60}[I\times I]$& 31. $(I/C_2;I/C_2)$&  $2.[3,5]^+$ & n.cryst.%
& $120$\\
$\pm T$ in $\pm O$&$+\frac1{24}[O\times \overline O]\cdot 2_3$& 44\rlap{$''$}. $(O/C_1;O/C_1)^{*\prime\prime}$&  $[2,3,3]^\circ$ & 25/09%
& $48$\\
$+O$ in $\pm O$&$\pm\frac1{24}[O\times O]$& 26. $(O/C_2;O/C_2)$&  $2.[3,4]^+$ & 25/06%
& $48$\\
$TO$ in $\pm O$&$\pm\frac1{12}[T\times \overline T]\cdot 2$& 40. $(T/C_2;T/C_2)^*$&  $2.[3,3]$ & 24/05%
& $48$\\
$+T$ in $\pm T$&$\pm\frac1{12}[T\times T]$& 21. $(T/C_2;T/C_2)$&  $2.[3,3]^+$ & 24/02%
& $24$\\
$+T$ in $+O$&$+\frac1{12}[T\times \overline T]\cdot 2_3$& 40\rlap{$'$}. $(T/C_1;T/C_1)^*_-$&  $[3,3]^\circ$ & 24/03%
& $24$\\
$+T$ in $TO$&$+\frac1{24}[O\times \overline O]$& 26\rlap{$''$}. $(O/C_1;O/C_1)''$&  $[2,3,3]^+$ & 25/04%
& $24$\\[2pt]
\hline
\end{tabular}

\caption[Summary of the 21 axial groups]
  {Summary of the 21 axial groups (except doubly axial groups).
We have included references
to the list of crystallographic 4-dimensional groups by
Brown, Bülow, Neubüser, Wondratschek, Zassenhaus (BBNWZ)  
\cite{BBNWZ},
and the names of Du Val~\cite{duval},
together with his numbering which extends the numbering of Goursat.
\\
We use two further adaptations of Coxeter's
notation, following~\cite{CS}:
$G^\circ$ is obtained by replacing the \OR\ elements $g$ of
$G$ by $-g$.  An initial
``$2.$'' indicates doubling the group by adjoining negatives.  The
$2$ in %
$[2,3,3]$ indicates the presence of an extra ``perpendicular'' mirror
$R_1$ that commutes with the other reflections.
\\
In Du Val's notation, achiral groups can be recognized by the $^*$ superscript.
Haploid groups (those whose CS name begins with a +), which were not
considered by Goursat, and Du~Val denotes them by adding primes to the numbers of
the corresponding diploid groups, such as 44$'$ and 44$''$.
Variations are indicated by various subscript and superscript
decorations of the group names.
In some cases, a unique notation is only achieved
by considering the number and the name together.
Thus, we are deviating from Du Val's notation by attaching the primes also to the names.
For example, Du Val distinguishes two groups 26$'$ and 26$''$
with the same name $(O/C_1;O/C_1)$.
Accordingly, although this is overlooked in Du Val~\cite[p.~61]{duval}, 
one must also make a distinction between the corresponding achiral groups
44$'$ and 44$''$.
\emph{Each of} these two achiral extensions comes in two variations:
$(O/C_1;O/C_1)^*$ and $(O/C_1;O/C_1)^*_-$.
This omission in Du Val's list was already noted by Dunbar
\cite[p.~141, last paragraph]{dunbar94_f2}.
}
\label{tbl:axial_full}
\end{table}

We are left with the case that $G_3^{+x_4}$ is
an index-2 subgroup $H$ of $G_3$.
In this case, the group $G$ is uniquely determined by $H$ and $G_3$:
It consists of the elements $(g,+x_4)$ for
$g\in H$ and $(g,-x_4)$ for $g\in G_3-H$.
We denote this group as
``$H$  in $G_3$''.
As an abstract group, it is isomorphic to~$G_3$.
There are seven index-2 containments among the
three-dimensional polyhedral groups. %
(See \cite[Figures 3.9 and~3.10]{CS} for an overview
about all index-2 containments in $\O(3)$.)
They lead to seven ``hybrid axial groups'',
which are listed in Table~\ref{tab:axial_hybrid}.

There are several methods by which such an index-2
containment can be constructed, and we indicate in the table
which methods are applicable:

\begin{enumerate}
\item Chirality:
$G_3^{+x_4}$ is the chiral part of an achiral group $G_3$.
In this case, the resulting group will be chiral,
because the \OR\ elements of $G_3$ are composed with
the reflection of the axis.
In other words, $G$ is the chiral part
$(G_3\times \{x_4,-x_4\})^+$
of the prismatic group
$G_3\times \{x_4,-x_4\}$.

\item Center:
$G_3^{+x_4}$ does not contain the central reflection.
In this case, an index-2 extension $G_3$
of $G_3^{+x_4}$
can always be obtained by adjoining
the central reflection (in $\mathbb{R}^3$).
The resulting group ``$G_3^{+x_4}$  in $G_3$''
is equivalently thought of as simply adjoining
the central reflection (in $\mathbb{R}^4$) to $G_3^{+x_4}$.
These groups can be recognized as having
their Coxeter names prefixed with ``$2.$''.
$G$ is achiral iff $G_3^{+x_4}$ is achiral, and in this case, the
construction is simultaneously a case of the chirality method.

\item Alternation:
This applies to the octahedral groups,
which are symmetries of the cube.
The vertices of the cube can be two-colored.
The subgroup consists of those transformations
that preserve the coloring.

\item Edge orientation:
There is only one case where this applies,
namely the pyritohedral group $\pm T$
as a subgroup of the full octahedral group $\pm O$.
The edges of the octahedron can be \emph{coherently} oriented in such
a way that the boundary of every face is a directed cycle.
The subgroup consists of
those transformations that preserve this orientation
(cf.\ the use of the edge orientation for the 24-cell and its polar,
Section~\ref{sec:24-cell}). 
\end{enumerate}

Often, the same result can be obtained by  two methods. For example,
$TO$ in $\pm O$ results both from alternation and from center.

The group ``$G_3^{+x_4}$ in $G_3$''
is chiral if and only if $G_3^{+x_4}$ is chiral and $G_3$ is achiral,
because the elements of $G_3 \setminus G_3^{+x_4}$
are flipped by the $x_4$-reflection.
These are the case of the form ``$+G$ in $\pm G$'' in the table, plus
the group ``$+T$ in $TO$''.

The situation is very much analogous to the construction of
the achiral groups in $\O(3)$ from the chiral groups in $\SO(3)$
and their index-2 subgroups in \cite[\S 3.8]{CS},
except that Conway and Smith prefer to extend by
the algebraically simpler central inversion
$-\id$  %
instead of the geometrically more natural reflection of the
axial coordinate.

The maximal axial groups are
$\pm\frac1{60}[I\times I]\cdot 2  = 2.[3,5]$
and
$\pm\frac1{24}[O\times O]\cdot 2 = 2.[3,4]$.
Hence, the axial groups can be characterized as the symmetries
of a 4-dimensional prism over an icosahedron
or over an octahedron, %
and the subgroups of these.
(This includes, however, the doubly axial groups, which we have
classified under the toroidal groups.)

We mention that, among the $3\times 7=21$ axial groups, there are 7
chiral ones and 14 achiral ones.
Among the polyhedral groups, there are 14 chiral ones.
We have no explanation for the frequent appearance of the magic number
7 and its multiples.

\section{Computer calculations}
\label{Computer calculations}

We used the help of computers for investigating the groups and
checking the results, as well as for the preparation of the figures
and tables.
We used SageMath~\cite{sagemath} and its
interface to the GAP~\cite{GAP4} software for group-theoretic calculations.
The computer code is available in
\url{https://github.com/LaisRast/point-groups}.

\subsection{Representation of transformations and groups}
\label{sec:representation}

We represent the orthogonal transformations
$[l,r]$ and $*[l,r]$ by
the quaternion pair $(l,r)$ and a bit for indicating orientation reversal.
In a group, each transformation is represented twice, by the
equivalent pairs
$(l,r)$ and $(-l,-r)$.

We used two different representations for quaternions:
For the elements of $2I$, $2O$, and $2T$, the quaternions
$x_1+x_2i+x_3j+x_4$
are
represented in the natural way with precise algebraic coefficients, using SageMath's
support for %
algebraic extension fields. %
For the elements of $2D_{2n}$, %
we used a tailored representation: These elements are of the form
$e_n^s$ or
$e_n^sj$, and we represent and manipulate them using the %
fraction $s/n$, and a bit that indicates whether the factor $j$ is
present. (An exact algebraic representation would have required extension
fields of arbitrarily high degree.)

The left group and the right group don't have to use the same
representation: For elements of tubical groups, like $[l,r]\in\pm[I\times C_n]$,
each of $l$ and $r$ uses its own appropriate representation.

\subsection{Fingerprinting}
\label{Fingerprinting}
For preparing a catalog of groups, it is useful to have some easily
computable invariants. We used the number of elements of each
geometric type as a \emph{fingerprint}.
This technique was initiated by Hurley~\cite{hurley} in his
classification of the 4-dimensional crystallographic groups.

We first discuss the classification of the individual
4-dimensional orthogonal transformations, as introduced
in Section~\ref{sec:the-4d-transformations}.
Every \OP\ orthogonal transformation can be written as a
block diagonal matrix $R_{\alpha_1,\alpha_2}$ of two rotation matrices
\eqref{rotation}.
We must be aware of other angle parameters
$R_{\alpha_1',\alpha'_2}$ that describe geometrically the same operation,
in other words, that are conjugate by an \OP\
transformation (see Section~\ref{sec:coordinate-system}).
If we swap the two invariant coordinate planes
$(x_1,x_2)\leftrightarrow(x_3,x_4)$, this is an \OP\ transformation,
and it turns $R_{\alpha_1,\alpha_2}$ into $R_{\alpha_2,\alpha_1}$.
A simultaneous reflection in both coordinate planes
($x_1\leftrightarrow x_2$
and $x_3\leftrightarrow x_4$)
is also \OP,
and it turns $R_{\alpha_1,\alpha_2}$ into  $R_{-\alpha_1,-\alpha_2}$.

Thus,
 $R_{\alpha_1,\alpha_2}\doteq R_{\alpha_2,\alpha_1}\doteq
 R_{-\alpha_1,-\alpha_2}
 \doteq R_{-\alpha_2,-\alpha_1}$.
On the other hand,
 $R_{\alpha_1,\alpha_2}$ and $R_{\alpha_1,-\alpha_2}$ are
 distinct unless one of the angles is $0$ or $\pm \pi$. They are
 mirrors of each other.

The \OR\ transformations $\bar R_{\alpha}$ of \eqref{eq:OR} are characterized by a
single angle $\alpha$.
Since
the simultaneous negation of $x_1$ and $x_4$ turns $\bar R_{\alpha}$
into $\bar R_{-\alpha}$, the parameter $\alpha$ can be normalized to
the range $0\le\alpha\le \pi/2$.

Since the angles are rational multiples of $\pi$, it is possible to
encode the data about the operation into a short code.
By collecting the codes of the elements in a group into a
string, we obtained a ``fingerprint'' of the group,
which we
used as a key for our catalog.\footnote
{Here are some details: We
  actually use the quaternion pair $[l,r]$ for computing the code for
  a rotation: If $[l,1]$ and $[1,r]$ are rotations by $a\pi$ and
  $b\pi$, respectively, we use the pair of rational numbers
  $(a,b)$ with $0\le a,b\le 1$.  The pair $[-l,-r]$, which represents
  the same rotation, gives the pair $(1-a,1-b)$, and hence we normalize
  by requiring that $a<b$ or $a=b\le 1/2$.

  For example, the group $\grp|^{\mathrm{pg}}_{2,4}$ has the fingerprint
 \texttt{0|0:2 0|1:2 1|1/4:4 1|3/4:4 1|1/2:4 *1/2:16}.
  We tried to make the code concise while keeping it readable.
  The term \texttt{/4} in
\texttt{1|3/4:4} is a common denominator for both components, and
hence
\texttt{1|3/4} stands for the pair $(a,b)=(\frac14,\frac34)$,
denoting a rotation of the form $[\exp \frac\pi4,\exp \frac{3\pi}4]
\doteq R_{-\pi/2,\pi}$.
The number \texttt{:4} after the colon denotes the multiplicity.
Since our group representation contains both pairs $[l,r]$ and
$[-l,-r]$ for each rotation, the multiplicity is always overcounted by
a factor of~2, the group actually contains only two operations
$R_{-\pi/2,\pi}$. (The reader may wish to identify them as torus
translations of this group, see Figure~\ref{fig:grid-parameters}.)
The symbol \texttt{0|0} denotes the identity.
The \OR\ transformations are written with a star.
 The sign \texttt{*}$a$ with a fraction $a$ denotes %
 $\bar R_{(1-a)\pi}$. In our example,
\texttt{*1/2:16} denotes eight operations of the form $\bar
R_{\pi/2}$.
The sum of the written multiplicities is 32, in accordance with the
fact that
the group has order $32/2=16$.}
Experimentally,
in all cases that we encountered, this method was sufficient to
distinguish groups up to conjugacy. (As reported below, we
considered,
from the infinite families of groups, at least all groups of order up
to 100.)
  
 The classification of the elements by Hurley~\cite{hurley} is almost equivalent, except that
 it disregards the orientation:
 He
classified a transformation by the
triplet of coefficients $(c_3,c_2,c_0)$ of
its characteristic equation
 $\lambda^4
-c_3 \lambda^3
+c_2 \lambda^2
-c_1 \lambda
+c_0=0$:
the trace $c_3$, the second invariant $c_2$, and the determinant $c_0$.
 Since all eigenvalues have absolute value 1, the linear coefficient
 $c_1$
 is determined by the others %
 through the
 formula
 $c_1=-c_0c_3$.
 The Hurley triplet determines the eigenvalues and thus the geometric
 conjugacy type and the rotation angles $\alpha_1,\alpha_2$, but only
 up to orientation.
 $R_{\alpha_1,\alpha_2}$ and $R_{\alpha_1,-\alpha_2}$ have the same
 spectrum and the same Hurley symbol.

 \paragraph{The Hurley symbol.}
 Hurley was interested in the crystallographic groups, and the
 operations in these groups must have integer coefficients in their
 characteristic polynomial. This restricts the operations to a finite
 set. Hurley denoted them by 24 letters %
 (the Hurley symbols).

 They were also used
in %
 the monumental
classification of the four-dimensional crystallographic space groups by
Brown, Bülow, Neubüser, Wondratschek, Zassenhaus
\cite{BBNWZ}.
 Brown et al.\
 refined the classification by splitting the groups into
 conjugacy classes \emph{under the group operations},
resulting in the \emph{Hurley pattern}.
 It may happen that several operations are geometrically the same
 but not conjugate to each other by a transformation of the group that
 is under consideration.\footnote
 {For example, the group 21/03
in \cite{BBNWZ} of order 12
   has the Hurley pattern
  \texttt{1*1I, 1*1E, 2*3E, 1*2S', 1*2B;}
  in our classification, it corresponds to two \ena\ groups,
  $\grp X^{\mathbf{c2mm}}_{1,3}$ and
  $\grp X^{\mathbf{c2mm}}_{3,1}$.
The fingerprints of these groups are
\texttt{0|0:2 0|2/3:4 1|1/2:14 3|5/6:4} and
\texttt{0|0:2 1|3/6:4 1|3/3:4 1|1/2:14}.
Both groups contain 7 half-turns (code \texttt{1|1/2}, Hurley symbol
\texttt{E}).
The second group, for example, is actually also a torus flip group:
$\grp X^{\mathbf{c2mm}}_{3,1} \doteq
\grp ._{3,2}$.
In this representation, it has 6 flip operations, which are
half-turns. In addition, it contains the torus translation
$R_{\pi,0}$, which is another half-turn.
This half-turn is not conjugate to the other half-turns by operations
of the group.
It forms a conjugacy class of its own,
as indicated by the code  \texttt{1*1E} in the
 Hurley pattern.
The 6 flip operations split into two
conjugacy classes of size 3,
as indicated by the code \texttt{2*3E}.
}
 
 Brown et al.\ \cite[p.~9]{BBNWZ} report that their classification,
 which is more refined than ours but in another respect coarser, since
 it does not distinguish enantiomorphic groups, was also found to be
 sufficient to characterize the crystallographic point groups uniquely
 (up to mirror congruence).

We could use the data in the Tables of \cite{BBNWZ}
to match them with our classification. The results are tabulated in
Tables \ref{tab:crystallographic}--\ref{tab:crystallographic2} in
Appendix~\ref{sec:crystallographic}.

\subsection{Computer checks}

As mentioned, the classic approach to the classification
following Goursat's method
yields
the chiral groups,
and with the exception of the toroidal groups, they are obtained
 quite painlessly.
However, the achiral groups must be found and classified
as index-2 extensions of the chiral groups.

This task has been carried out
by Du Val~\cite{duval} and Conway and Smith~\cite{CS}, but they
only gave the results.
Du Val~\cite[p.~61]{duval} explicitly lists the \OR\ elements of each
achiral group. %
Conway and Smith~\cite[Tables~4.1--4.3]{CS} provide generating elements for
each group.

A detailed derivation is not presented in the literature.  The
considerations about the extension from chiral groups to achiral ones
are only briefly sketched by Conway and Smith \cite[p.~51--52]{CS}, see
Figures~\ref{fig:p51}--\ref{fig:p52}.
Since we found this situation unsatisfactory,
 we
ran a brute-force computer check.
We generated
 all subgroups of the groups $\pm[I\times I]$,
$\pm[O\times O]$ and
$\pm[T\times T]$
and their achiral extensions. No missing groups were discovered.
More details are given below.

For the achiral extension of the subgroups of $\pm[C_n\times C_n]$,
and $\pm[D_{2n}\times D_{2n}]$, we have supplanted the classic
classification by own classification as
toroidal groups.
Nevertheless, we ran some computer checks also for these groups, see
Section~\ref{sec:compute-toroidal}.

\subsection{Checking the achiral polyhedral and axial groups}
\label{checking-achiral}
\label{achiral-polyhedral}

For each group $\pm[I\times I]$,
$\pm[O\times O]$ and
$\pm[T\times T]$ in turn,
we generate all subgroups,
We kept only those subgroups for which the left and right subgroup is
the full group $2I$, $2O$, or $2T$ respectively.
(For an achiral group, we must extend a group whose left group is
equal to its right group.)

For each obtained subgroup, we identified the possible
extending elements, using the considerations of Section~\ref{sec:achiral}.
Each achiral group was classified by
its fingerprint (the conjugacy types of its
elements), and for each class,
we managed to find geometric conjugations to show that all groups in
that class are geometrically the same.

We mention some details for the largest group $[I\times I]$.
The group $\pm[I\times I]$ was represented by its double-cover
$2I\times 2I$, and converted to a permutation group,
in order to let GAP generate
 the subgroups.
There are 19,987 subgroups in total, and they were found in about 5 minutes.
14,896 subgroups of them contain the pair $(-1,-1)$, which
is necessary to have a double cover of a rotation group in $\pm[I\times I]$,
and
only 241 of these groups have the left and right subgroups equal to
$2I$.
These represented the group
$\pm[I\times I]$ itself,
and 60 different copies of each group
$\pm\frac1{60}[I\times I]$,
$\pm\frac1{60}[I\times \bar I]$,
$+\frac1{60}[I\times I]$,
$+\frac1{60}[I\times \bar I]$.

For each of the 241 groups, we tried to extend it by an element
$*[1,c]$ in all possible ways, following
Proposition~\ref{1,c}.
Actually, it is easy to see that elements $c$ and $c'=cx$ that are
related by an element $x$ in the
kernel %
lead to the same extension, and thus
they need not be tried separately.

This leads to 361 distinct groups.
Again there are 60 representatives of each of the six achiral groups
with fraction $\frac1{60}$, plus one for the group
$\pm[I\times I]\cdot 2$ itself.

Since we searched for conjugacies in a systematic but somewhat
ad-hoc manner, it took about half a week for the computer to show that
all 60 groups in each class are geometrically the same.  With
hindsight, the multiplicity of 60 is not surprising, since there are
60 conjugacies that map the elements of $2I$ to themselves.

\subsection{Checking the toroidal groups}
\label{sec:compute-toroidal}

The toroidal groups form an infinite family, and hence we can only
generate them up to some limit.
We set the goal of checking all chiral toroidal groups 
 up to order 200 and all
 achiral groups up to order 400.
For this purpose, we generated all groups
$\pm[D_{n}\times D_{n}]\cdot 2$ (for even $n$) and
$\pm[C_{n}\times C_{n}]\cdot 2$ in the range $100<n\le 200$,
together with their subgroups.

For generating the subgroups, we
took a different approach than for the polyhedral groups: We
constructed a permutation group representation of the \emph{achiral} group
and computed all its subgroups. We took all subgroups, regardless
of whether the left and right group is the full group $C_n$ or $D_n$.
 For each
 chiral group up to order 200 and each
 achiral group up to order 400
that was generated,
we checked
 that it is conjugate to one on the
 known groups according to our classification.
 We also checked whether all known toroidal groups within these size
 bounds are found. This turned out to be the case
with a few exceptions.
The exceptions were the chiral groups
$\grp\setminus^{\mathrm{cm}}_{m,n}$,
$\grp\setminus^{\mathrm{cm}}_{n,m}$
$\grp/^{\mathrm{cm}}_{m,n}$, and
$\grp/^{\mathrm{cm}}_{n,m}$,
for 13 pairs
$(m,n)=(3,17),(3,19),\ldots
,(7,13),(9,11)$
of orders $2mn$ between 100 and 200.
The reason that these groups were missed is that they are of the form $+\frac12[D_{2m}\times C_{2n}]
\leqslant+\frac12[D_{2m}\times D_{4n}]$,
and the smallest group
$\pm[D_{n'}\times D_{n'}]\cdot 2$ that contains them has $n'=4\cdot\mathrm{lcm}(m,n)$,
which exceeds 200.

The group with the largest number of subgroups was %
$\pm[D_{192}\times D_{192}]\cdot 2$.
It has
1,361,642 subgroups. For
1,249,563 of these groups, the order was within the limits.
This computation requires a workstation with large memory, on the
order of about 100 gigabytes.
The whole computation took about 10 days.

\section{Higher dimensions}
\label{Higher dimensions}

In the classification of Theorem~\ref{classification},
there are categories that we expect in any dimension: the polyhedral groups, which are
related to the regular polytopes,
 the toroidal groups, and the axial groups, which come from direct sums
of lower-dimensional groups.
On the other hand, the tubical groups are more surprising.
They rely on the covering
$\SO(3) \times \SO(3) \stackrel{2:1\,}\longrightarrow \SO(4) $, which provides
a different product structure in terms of lower-dimensional groups
than the direct sum.

The scarcity of regular polytopes in high dimensions might
be an indication that these groups are not very exciting.
On the other hand, the root systems $E_6$, $E_7$, and $E_8$ in 6, 7, and 8
dimensions promise some richer structure in certain dimensions.

In five dimensions,
the \OP\ case has been settled by Mecchia and Zimmermann~\cite{MZimm}, see \cite[Corollary 2]{Zi}:
\begin{theorem}
  The finite subgroups of the orthogonal group $\SO(5)$ are
  \begin{enumerate}[\rm(i)]
  \item 
subgroups
of $\Orth4
\times \Orth1$ or $\Orth3 \times \Orth2$ \textup(the reducible case\textup)\textup;
\item 
  subgroups of the symmetry group
  $(\mathbb Z_2 )^5 \rtimes S_5$ of the hypercube\textup;
\item or isomorphic to $A_5$, $S_5$, $A_6$ or $S_6$. \textup(This includes
  symmetries of the simplex and its polar.\textup)
\end{enumerate}

\end{theorem}

The irreducible representations of the groups in (iii) can be looked up in
the character tables of the books on Representation Theory. It
would be interesting to know what the 5-dimensional representations
are in geometric terms (besides the symmetries of the simplex).

This theorem gives only the chiral groups,
but in odd dimensions like 5, it is in principle straightforward to
derive
the achiral groups from the chiral ones:
All one needs to know are the chiral groups and their index-2 subgroups.
See \cite[\S3.8%
]{CS} for the three-dimensional case.
Briefly, one can say
that nothing unexpected happens for the point groups in 5
dimensions.

\paragraph{Six dimensions.}

The richest part of the 4-dimensional groups were the toroidal groups,
which
have an invariant Clifford torus.
The sphere $S^5$ contains an analogous
three-dimensional %
torus
\begin{displaymath}
  x_1^2+x_2^2 =
  x_3^2+x_4^2 =
  x_5^2+x_6^2 = {1/3}
\end{displaymath}
A group that leaves this torus invariant behaves similarly to a
three-dimensional space group, involving translations, reflections,
and rotations in terms of torus
coordinates $\phi_1,\phi_2,\phi_3$.
Thus, %
the %
three-dimensional space
groups %
will make their appearance in the classification of 6-dimensional
point groups.

The situation in 4 dimensions was similar: We have
studied the toroidal groups in analogy to
the wallpaper groups (the two-dimensional space groups).
In contrast to the situation in the plane, a 6-fold rotation in 3-space is not inconsistent
with the requirement that the
lattice of translations contains a cubical lattice. Thus, we may expect that
 all of the 230 three-dimensional space
groups show up in the 6-dimensional point groups.
(In one dimension lower, we have %
another instance of this phenomenon: The frieze groups appear as the
3-dimensional axial point groups.)

Thus, a classification of the point groups in 6 dimensions will be much more
laborious than in 5 dimensions.
It has already been observed by Carl Hermann in 1952 \cite[p.~33]{Hermann1952}, in connection with the
crystallographic groups, that
``going up from an odd dimension to the next higher even one leads by far
to more surprises than the opposite case''.

\bibliographystyle{plainurl}
\markright{References}
{%
  \let\oldsection=\section
  \def\section*#1{\oldsection*#1%
    \addcontentsline {toc}{section}{\numberline {}References}%
  }
 \bibliography{main}
}  
\markboth{Laith Rastanawi and Günter Rote: 4-Dimensional Point Groups}
         {Laith Rastanawi and Günter Rote: 4-Dimensional Point Groups}

\appendix

\section{Generators for the polyhedral and axial groups}
\label{sec:polyhedral-and-axial-table}

Table~\ref{tab:polyhedral-appendix}
gives a complete summary of the polyhedral
(Table~\ref{tab:polyhedral})
and axial groups (Table~\ref{tbl:axial_full}),
following the numbering by Goursat~\cite{goursat-1889}, as extended
to the haploid groups by Du~Val~\cite{duval}, together with
a set of generators for each group.
  The axial groups can be recognized as having only two numbers
  different from 2
  in their Coxeter name. %
Our adaptations of Du Val's names was explained in
Table~\ref{tbl:axial_full} and footnote~\ref{duval-41-42} on 
p.~\pageref{duval-41-42}.
The top part contains the chiral groups (\#20--\#32)
and the bottom part the achiral ones (\#39--\#51).\footnote
{A similar table, containing some four-dimensional reflection groups
  and their subgroups,
  appears in
  Coxeter~\cite[p.~571]{semiregular2},
  with correspondences between Coxeter's own notation and Du~Val's
  names.
  The very first entry in that table, $[3,3,2]^+$, mistakenly refers to Du Val's group
  \#21 $(T/C_2;T/C_2)
=\pm\frac1{12}[T\times T] $, while it is actually
  \#26{$''$} $(O/C_1;O/C_1)''=
  +\frac1{24}[O\times \overline O]$.
  The fifth entry, $[3,3,2]$, refers to Du Val's group 
  $(O/C_1;O/C_1)^{*}$,
while it should actually be $(O/C_1;O/C_1)^{*}_-$, or more
precisely
\#44{$''$} $(O/C_1;O/C_1)^{*\prime\prime}_- =
+\frac1{24}[O\times\overline O]\cdot 2_1$.
The confusing ambiguity of Du Val's names for the groups   
44{$'$} and
44{$''$}
mentioned in the caption of Table~\ref{tbl:axial_full}
was apparently not realized by Coxeter.
  }

Where appropriate, we include
 a reference to the numbering of
crystallographic point groups according to
Brown, Bülow, Neubüser, Wondratschek, Zassenhaus (BBNWZ)  
\cite{BBNWZ}, see also Appendix~\ref{sec:crystallographic}.

\begin{table}[p]
  \centering
\setlength{\extrarowheight}{0,6pt}  
\setlength{\extrarowheight}{0,1pt}  

    \begin{tabular}{|l@{ }l@{ \ }ll@{ }r@{ \ }l|}
      \hline
      D\rlap{u Val \# \& name} &
      CS name & generators& Cox.\ name & \llap{ord}er & BBNWZ\\\hline
    20. $(T/T;T/T)$&$\pm[T\times T]$& $[i,\omega],[\omega,i]$& $[^+3,4,3^+]$&288&33/13\\
21. $(T/C_2;T/C_2)$&$\pm\frac1{12}[T\times T]$& $[\omega,-\omega],[i,i]$& $2.[3,3]^+$&24&24/02\\
21\rlap{$'$}. $(T/C_1;T/C_1)$&$+\frac1{12}[T\times T]$& $[\omega,\omega],[i,i]$& $[3,3]^+$&12&24/01\\
22. $(T/V;T/V)$&$\pm\frac1{3}[T\times T]$& $[i,1],[1,i],[\omega,\omega]$& $[^+3,3,4^+]$&96&32/16\\
23. $(O/O;T/T)$&$\pm[O\times T]$& $[i_O,\omega],[\omega,i]$& $[[^+3,4,3^+]]_L$&576&not cryst.\\
23. $(T/T;O/O)$&$\pm[T\times O]$& $[i,\omega],[\omega,i_O]$& $[[^+3,4,3^+]]_R$&576&not cryst.\\
24. $(I/I;T/T)$&$\pm[I\times T]$& $[i_I,\omega],[\omega,i]$& $[3,3,5]^+_{\frac15L}$&1440&not cryst.\\
24. $(T/T;I/I)$&$\pm[T\times I]$& $[i,\omega],[\omega,i_I]$& $[3,3,5]^+_{\frac15R}$&1440&not cryst.\\
25. $(O/O;O/O)$&$\pm[O\times O]$& $[i_O,\omega],[\omega,i_O]$& $[[3,4,3]]^+$&1152&not cryst.\\
26. $(O/C_2;O/C_2)$&$\pm\frac1{24}[O\times O]$& $[\omega,-\omega],[i_O,i_O]$& $2.[3,4]^+$&48&25/06\\
26\rlap{$'$}. $(O/C_1;O/C_1)'$&$+\frac1{24}[O\times O]$& $[\omega,\omega],[i_O,i_O]$& $[3,4]^+$&24&25/03\\
26\rlap{$''$}. $(O/C_1;O/C_1)''$&$+\frac1{24}[O\times \overline O]$& $[\omega,\omega],[i_O,-i_O]$& $[2,3,3]^+$&24&25/04\\
27. $(O/V;O/V)$&$\pm\frac1{6}[O\times O]$& $[i,j],[\omega,\omega],[i_O,i_O]$& $[3,3,4]^+$&192&32/20\\
28. $(O/T;O/T)$&$\pm\frac1{2}[O\times O]$& $[\omega,1],[1,\omega],[i_O,i_O]$& $[3,4,3]^+$&576&33/15\\
29. $(I/I;O/O)$&$\pm[I\times O]$& $[i_I,\omega],[\omega,i_O]$& $[[3,3,5]^+_{\frac15L}]$&2880&not cryst.\\
29. $(O/O;I/I)$&$\pm[O\times I]$& $[i_O,\omega],[\omega,i_I]$& $[[3,3,5]^+_{\frac15R}]$&2880&not cryst.\\
30. $(I/I;I/I)$&$\pm[I\times I]$& $[i_I,\omega],[\omega,i_I]$& $[3,3,5]^+$&7200&not cryst.\\
31. $(I/C_2;I/C_2)$&$\pm\frac1{60}[I\times I]$& $[\omega,\omega],[i_I,-i_I]$& $2.[3,5]^+$&120&not cryst.\\
31\rlap{$'$}. $(I/C_1;I/C_1)$&$+\frac1{60}[I\times I]$& $[\omega,\omega],[i_I,i_I]$& $[3,5]^+$&60&not cryst.\\
32. $(I^\dag/C_2;I/C_2)^{\dag}$&& $[\omega,\omega],[i_I,-i_I^\dag]$&\smash{\lower 1,4ex\hbox{\llap{$\biggr\}$ }$[[3,3,3]]^+$}}&\smash{\lower 1,4ex\hbox{120}}&\smash{\lower 1,4ex\hbox{31/06}}\\
&$\pm\frac1{60}[I\times \overline I]$& $[\omega,\omega],[i_I,-i'_I]$&&&\\
32\rlap{$'$}. $(I^\dag/C_1;I/C_1)^{\dag}$&& $[\omega,\omega],[i_I,i_I^\dag]$&\smash{\lower 1,4ex\hbox{\llap{$\biggr\}$ }$[3,3,3]^+$}}&\smash{\lower 1,4ex\hbox{60}}&\smash{\lower 1,4ex\hbox{31/03}}\\
&$+\frac1{60}[I\times \overline I]$& $[\omega,\omega],[i_I,i'_I]$&&&\\
[2pt]\hline
39. $(T/C_2;T/C_2)^*_c$&$\pm\frac1{12}[T\times T]\cdot 2$& $[\omega,-\omega],{*}[i,-i]$& $2.[^+3,4]$&48&25/05\\
39\rlap{$'$}. $(T/C_1;T/C_1)^*_c$&$+\frac1{12}[T\times T]\cdot 2_3$& $[\omega,\omega],{*}[i,i]$& $[^+3,4]$&24&25/02\\
39\rlap{$'$}. $(T/C_1;T/C_1)^*_{c-}$&$+\frac1{12}[T\times T]\cdot 2_1$& $[\omega,\omega],{*}[i,-i]$& $[^+3,4]^\circ$&24&25/01\\
40. $(T/C_2;T/C_2)^*$&& $[\omega,-\omega],{*}[i_O,-i_O]$&\smash{\lower 1,4ex\hbox{\llap{$\biggr\}$ }$2.[3,3]$}}&\smash{\lower 1,4ex\hbox{48}}&\smash{\lower 1,4ex\hbox{24/05}}\\
&$\pm\frac1{12}[T\times \overline T]\cdot 2$& $[\omega,-\overline \omega],{*}[i,-i]$&&&\\
40\rlap{$'$}. $(T/C_1;T/C_1)^* $&& $[\omega,\omega],{*}[i_O,i_O]$&\smash{\lower 1,4ex\hbox{\llap{$\biggr\}$ }$[3,3]$}}&\smash{\lower 1,4ex\hbox{24}}&\smash{\lower 1,4ex\hbox{24/04}}\\
&$+\frac1{12}[T\times \overline T]\cdot 2_1$& $[\omega,\overline \omega],{*}[i,i]$&&&\\
40\rlap{$'$}. $(T/C_1;T/C_1)^*_-$&& $[\omega,\omega],{*}[i_O,-i_O]$&\smash{\lower 1,4ex\hbox{\llap{$\biggr\}$ }$[3,3]^\circ$}}&\smash{\lower 1,4ex\hbox{24}}&\smash{\lower 1,4ex\hbox{24/03}}\\
&$+\frac1{12}[T\times \overline T]\cdot 2_3$& $[\omega,\overline \omega],{*}[i,-i]$&&&\\
41. $(T/V;T/V)^*$&$\pm\frac1{3}[T\times T]\cdot 2$& ${*}[i,1],[\omega,\omega]$& $[^+3,3,4]$&192&32/18\\
42. $(T/V;T/V)^*_-$&$\pm\frac1{3}[T\times \overline T]\cdot 2$& ${*}[i,1],[\omega,\overline \omega]$& $[3,3,4^+]$&192&32/19\\
43. $(T/T;T/T)^*$&$\pm[T\times T]\cdot 2$& $[i,\omega],{*}[\omega,i]$& $[3,4,3^+]$&576&33/14\\
44. $(O/C_2;O/C_2)^*$&$\pm\frac1{24}[O\times O]\cdot 2$& $[\omega,-\omega],[i_O,i_O],-{*}$& $2.[3,4]$&96&25/11\\
44\rlap{$'$}. $(O/C_1;O/C_1)^{*\prime}$&$+\frac1{24}[O\times O]\cdot 2_3$& $[\omega,\omega],[i_O,i_O],{*}$& $[3,4]$&48&25/10\\
44\rlap{$'$}. $(O/C_1;O/C_1)^{*\prime}_-$&$+\frac1{24}[O\times O]\cdot 2_1$& $[\omega,\omega],[i_O,i_O],-{*}$& $[3,4]^\circ$&48&25/07\\
44\rlap{$''$}. $(O/C_1;O/C_1)^{*\prime\prime}$&$+\frac1{24}[O\times \overline O]\cdot 2_3$& $[\omega,\omega],[i_O,-i_O],{*}$& $[2,3,3]^\circ$&48&25/09\\
44\rlap{$''$}. $(O/C_1;O/C_1)^{*\prime\prime}_-$&$+\frac1{24}[O\times \overline O]\cdot 2_1$& $[\omega,\omega],[i_O,-i_O],-{*}$& $[2,3,3]$&48&25/08\\
45. $(O/T;O/T)^*$&$\pm\frac1{2}[O\times O]\cdot 2$& ${*}[\omega,1],[i_O,i_O]$& $[3,4,3]$&1152&33/16\\
46. $(O/T;O/T)^*_-$&$\pm\frac1{2}[O\times O]\cdot \bar2$& $[\omega,1],{*}[1,i_O]$& $[[3,4,3]^+]$&1152&not cryst.\\
47. $(O/V;O/V)^*$&$\pm\frac1{6}[O\times O]\cdot 2$& ${*}[i\omega,\omega],[i_O,i_O]$& $[3,3,4]$&384&32/21\\
48. $(O/O;O/O)^*$&$\pm[O\times O]\cdot 2$& ${*}[1,\omega],[\omega,i_O]$& $[[3,4,3]]$&2304&not cryst.\\
49. $(I/C_2;I/C_2)^*$&$\pm\frac1{60}[I\times I]\cdot 2$& $[\omega,-\omega],{*}[i_I,-i_I]$& $2.[3,5]$&240&not cryst.\\
49\rlap{$'$}. $(I/C_1;I/C_1)^*$&$+\frac1{60}[I\times I]\cdot 2_3$& $[\omega,\omega],{*}[i_I,i_I]$& $[3,5]$&120&not cryst.\\
49\rlap{$'$}. $(I/C_1;I/C_1)^*_-$&$+\frac1{60}[I\times I]\cdot 2_1$& $[\omega,\omega],{*}[i_I,-i_I]$& $[3,5]^\circ$&120&not cryst.\\
50. $(I/I;I/I)^*$&$\pm[I\times I]\cdot 2$& $[i_I,\omega],[\omega,i_I],{*}$& $[3,3,5]$&14400&not cryst.\\
51. $(I^\dag/C_2;I/C_2)^{\dag*}$&& $[\omega,-\omega],{*}[i_I i_O i,i_I^\dag i_O i]$&\smash{\lower 1,4ex\hbox{\llap{$\biggr\}$ }$[[3,3,3]]$}}&\smash{\lower 1,4ex\hbox{240}}&\smash{\lower 1,4ex\hbox{31/07}}\\
&$\pm\frac1{60}[I\times \overline I]\cdot 2$& $[\omega,-\omega],{*}[i_I,i'_I]$&&&\\
51\rlap{$'$}. $(I^\dag/C_1;I/C_1)^{\dag*}$&& $[\omega,\omega],{*}[i_I i_O i,i_I^\dag i_O i]$&\smash{\lower 1,4ex\hbox{\llap{$\biggr\}$ }$[3,3,3]$}}&\smash{\lower 1,4ex\hbox{120}}&\smash{\lower 1,4ex\hbox{31/05}}\\
&$+\frac1{60}[I\times \overline I]\cdot 2_1$& $[\omega,\omega],[i_I,i'_I],-{*}$&&&\\
51\rlap{$'$}. $(I^\dag/C_1;I/C_1)^{\dag*}_-$&& $[\omega,\omega],{*}[i_I i_O i,-i_I^\dag i_O i]$&\smash{\lower 1,4ex\hbox{\llap{$\biggr\}$ }$[[3,3,3]^+]$}}&\smash{\lower 1,4ex\hbox{120}}&\smash{\lower 1,4ex\hbox{31/04}}\\
&$+\frac1{60}[I\times \overline I]\cdot 2_3$& $[\omega,\omega],{*}[i_I,i'_I]$&&&\\
[2pt]\hline
    \end{tabular}

\caption[The 46 polyhedral and axial groups with generators]
{Polyhedral and axial groups with generators} %
  \label{tab:polyhedral-appendix}
\end{table}

In addition to the quaternions defined in
\eqref{eq:defining_quaternions}
in Section~\ref{sec:quaternion-groups},
 the following elements are used for generating the groups:
\begin{align}
  \nonumber
    \bar\w &= \tfrac12(-1-i-j-k) &&\text{(order 3)}\\
    i_I^\dag &= \tfrac12\bigl(i + \tfrac{-\sqrt5-1}2 j +
    \tfrac{-\sqrt5+1}2k\bigr) \label{iIdag}
     &&\text{(order 4)}\\
  i_I' &= \tfrac12\bigl( -\tfrac{\sqrt5-1}2 i -  \tfrac{\sqrt5+1}2 j + k\bigr)  
     &&\text{(order 4)} \label{iI'}
\end{align}
$\bar \w$ is simply the conjugate quaternion of $\w$.
We tried to reduce the number of generators by trial and error,
confirming by computer whether the generated groups did not change.

For a few groups, the groups given by Conway and Smith are not
identical to the groups of Du Val, and our table
lists both possibilities.

Conway and Smith \cite[Tables 4.2--4.3]{CS} specified the five groups
of type $[I\times \bar I]$
(\#32, \#32$'$ and \#51--\#51$''$)
by the generating set
``$[\omega,\omega],[i_I,\pm i_I']$'',
possibly extended by $*$ or $-{*}$ for the achiral groups,
but they did not define what $i_I'$ is.\footnote
{Five years later, the tables were almost literally reproduced in
  another book~\cite[Chapter~26]{Symmetries}, still without a
  definition of $i_I'$.
}
We tried all 120 elements of $2I$, and it turned out that
\eqref{iI'} is the
only value that works in this way. We don't see %
how we could have
predicted 
precisely this element, and we have no explanation for it.

Du Val \cite{duval}, on the other hand,
specifies generators
for these five groups
in terms of
the quaternion $i_I^\dag$ defined in~\eqref{iIdag}, which is obtained
by flipping the sign of $\sqrt5$ in the
expression for $i_I
=\tfrac12(i + \tfrac{\sqrt5-1}2 j +   \tfrac{\sqrt5+1}2
        k)
$.
This alternative choice generates a %
group $2I^\dag$ that is
different from~$2I$.
With this setup, it is not possible to use
the simple extending elements $*$ and $-{*}$ for the three achiral extensions \#51--\#51$''$:
For example, the square of the element ${*}[i_I^\dag,i_I]$ is
$[i_I i_I^\dag,i^\dag_Ii_I]$ with
$i_Ii^\dag_I = \frac14 + \frac{\sqrt5}4(i+j-k)$, and this element is
in neither of the groups $2I$ or~$2I^\dag$.
Du Val \cite[p.~55--56]{duval} gives a thorough and transparent
exposition of these groups and explains why they represent
the symmetries of the 4-simplex.

For the axial groups of type $\frac1{12}[T\times \bar T]$ (\#40 and
\#40$'$), the natural
generators from an algebraic viewpoint involve the quaternion $\bar
\omega$, and these were chosen by Conway and Smith. However, the axis
that is kept invariant by the groups is then spanned by the quaternion $j-k$.
With $*[i_O,\pm i_O]$ as the \OR\ generator,
the invariant axis becomes the real axis, and only in this representation,
the groups are subgroups of the larger axial group
 $\pm\frac1{24}[O\times O]$ (\#44).

\section{Orbit polytopes for tubical groups with special starting points}
\label{sec:special_starting_points}

We show polar
orbit polytopes for the tubical groups of cyclic type with all choices
of special starting points.

Each subsection considers a left tubical group $G$ together with a representative $f$-fold rotation center $p$ of $G^h$, corresponding to an entry in 
Table~\ref{tbl:tubical_supergroups}.
The particular data are given in the caption.
In addition, we indicate the subgroup $H$ of $G$
of elements that preserve $K_p$.
An \emph{alternate group} refers to an index-2 dihedral-type
supergroup of $G$ that, for an appropriate starting point on~$K_p$,
produces the same orbit as $G$. %

Two of these groups were already illustrated in the main text
(Figures~\ref{fig:IxCn_5fold}~and~\ref{fig:hOxC2n_4fold}),
and we 
follow the same conventions as in these figures:
On the top left, we show the $G^h$-orbit polytope of $p$,
and on the top right the spherical Voronoi diagram of that orbit.
Then we show the cells of the polar $G$-orbit polytopes of 
a starting point on $K_p$, for different values of $n$,
in increasing order of the size of the orbit.
For each cell, we indicate the values of $n$, and in
addition, the counterclockwise angle (as seen from the top) by which the
group rotates the cell as it proceeds to the next cell
above.
A blue vertical line
indicates the cell axis, the direction towards the next cell along~$K_p$.
For small values of $n$, this axis sometimes exits through a vertex or
an edge of the cell, but for large enough $n$ it goes through the top face
where the next cell is attached.

When the same orbit arises for several values of $n$, then the specified rotation
angle is the unique valid angle only for the smallest value $n_0$ that
is given. For a larger value $n=n_0f$, this can be combined with
arbitrary multiples of an $f$-fold rotation. For example, in
Figure~\ref{fig:IxCn_3fold}, we have the same cell for $n=5$ and
$n=15$. The specified rotation angle $(\frac13+\frac1{30})\cdot2\pi$ is the
unique valid angle between consecutive cells in the group $\pm[I\times
C_5]$,
but in the larger group $\pm[I\times C_{15}]$, it can be combined with
all multiples of $\frac 23\pi$. That is, all three rotation angles
$\frac1{15}\pi$,
$(\frac23+\frac1{15})\pi$, and
$(\frac43+\frac1{15})\pi$ are valid.
In some cases, such as $n=18$, the angle is never unique, and this is
indicated by a free parameter $k$ in the angle specification, which
can take any integer value.

By observing the rotation angles for the successive cells in the
figures, one can recognize the pattern that they follow.
\newpage

\subsection{\texorpdfstring{$\pm[I\times C_n]$}{+-[IxCn]}}
\subsubsection{\texorpdfstring{$\pm[I\times C_n]$}{+-[IxCn]}, 3-fold rotation center}
\begin{minipage}{0.5\textwidth}
\centering
\includegraphics[scale=1]{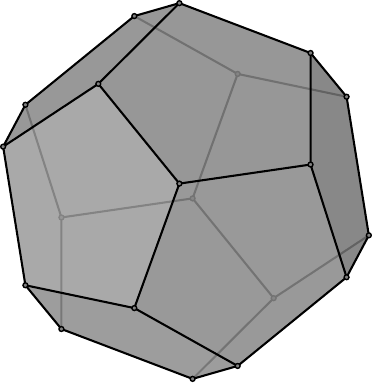}
\end{minipage}%
\begin{minipage}{0.5\textwidth}
\centering
\includegraphics[scale=1]{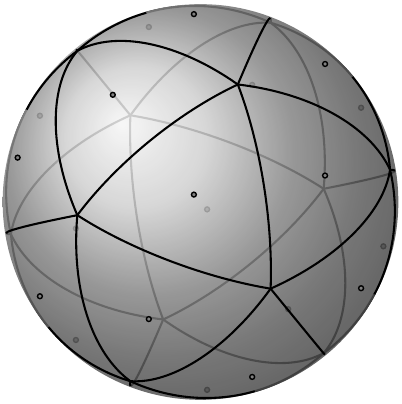}
\end{minipage}
\vskip 5pt plus 0.5fill
\begingroup
\setlength{\tabcolsep}{0pt}
\noindent\begin{tabular}{
>{\centering\arraybackslash}m{0.33\textwidth}
>{\centering\arraybackslash}m{0.33\textwidth}
>{\centering\arraybackslash}m{0.33\textwidth}}
\href{https://www.inf.fu-berlin.de/inst/ag-ti/software/DiscreteHopfFibration/gallery.html?f=IxCn/3/120cells_20tubes}{\vbox{\vskip-3mm \includegraphics[scale=1.1]{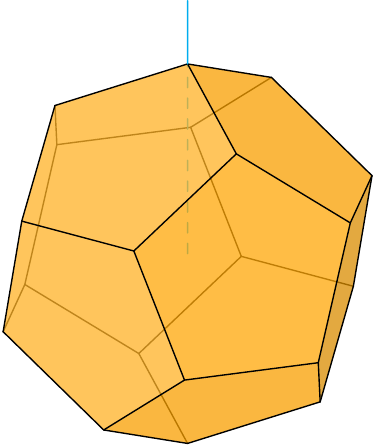}}}
&
\href{https://www.inf.fu-berlin.de/inst/ag-ti/software/DiscreteHopfFibration/gallery.html?f=IxCn/3/240cells_20tubes}{\includegraphics[scale=1.1]{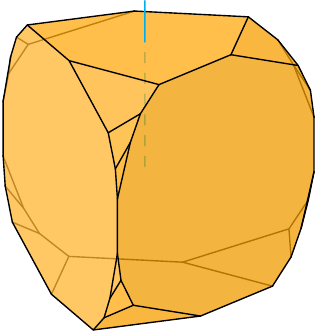}}
&
\href{https://www.inf.fu-berlin.de/inst/ag-ti/software/DiscreteHopfFibration/gallery.html?f=IxCn/3/360cells_20tubes}{\includegraphics[scale=1.1]{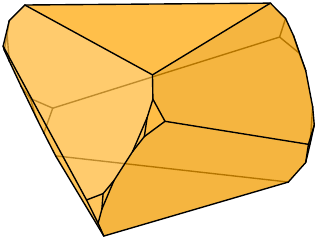}}
\\[1.5mm]
$n = 1, 3$\break
$(\frac{2}{3} + \frac{1}{6})\cdot 2\pi$
&$n = 2, 6$\break
$(\frac{1}{3} + \frac{1}{12})\cdot 2\pi$
&$n = 9$\break
$(\frac{k}{3} + \frac{1}{18})\cdot 2\pi$
\end{tabular}\nobreak

\vfill\nobreak
\noindent\begin{tabular}{
>{\centering\arraybackslash}m{0.33\textwidth}
>{\centering\arraybackslash}m{0.33\textwidth}
>{\centering\arraybackslash}m{0.33\textwidth}}
\href{https://www.inf.fu-berlin.de/inst/ag-ti/software/DiscreteHopfFibration/gallery.html?f=IxCn/3/480cells_20tubes}{\includegraphics[scale=1.1]{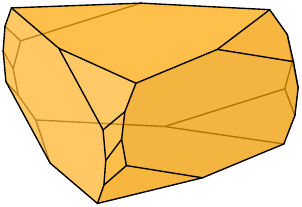}}
&
\href{https://www.inf.fu-berlin.de/inst/ag-ti/software/DiscreteHopfFibration/gallery.html?f=IxCn/3/600cells_20tubes}{\includegraphics[scale=1.1]{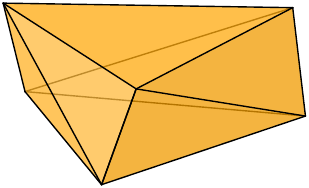}}
&
\href{https://www.inf.fu-berlin.de/inst/ag-ti/software/DiscreteHopfFibration/gallery.html?f=IxCn/3/720cells_20tubes}{\includegraphics[scale=1.1]{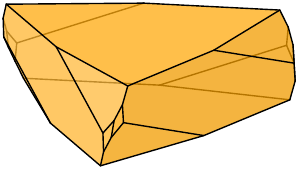}}
\\[1.5mm]
$n = 4, 12$\break
$(\frac{2}{3} + \frac{1}{24})\cdot 2\pi$
&$n = 5, 15$\break
$(\frac{1}{3} + \frac{1}{30})\cdot 2\pi$
&$n = 18$\break
$(\frac{k}{3} + \frac{1}{36})\cdot 2\pi$
\end{tabular}\nobreak

\vfill\nobreak
\noindent\begin{tabular}{
>{\centering\arraybackslash}m{0.33\textwidth}
>{\centering\arraybackslash}m{0.33\textwidth}
>{\centering\arraybackslash}m{0.33\textwidth}}
\href{https://www.inf.fu-berlin.de/inst/ag-ti/software/DiscreteHopfFibration/gallery.html?f=IxCn/3/840cells_20tubes}{\includegraphics[scale=1.1]{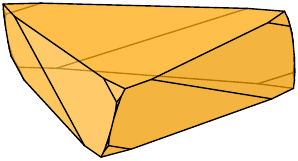}}
&
\href{https://www.inf.fu-berlin.de/inst/ag-ti/software/DiscreteHopfFibration/gallery.html?f=IxCn/3/960cells_20tubes}{\includegraphics[scale=1.1]{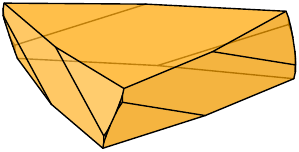}}
&
\href{https://www.inf.fu-berlin.de/inst/ag-ti/software/DiscreteHopfFibration/gallery.html?f=IxCn/3/1080cells_20tubes}{\includegraphics[scale=1.1]{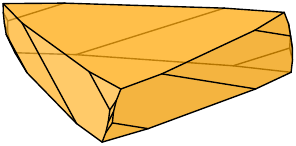}}
\\[1.5mm]
$n = 7, 21$\break
$(\frac{2}{3} + \frac{1}{42})\cdot 2\pi$
&$n = 8, 24$\break
$(\frac{1}{3} + \frac{1}{48})\cdot 2\pi$
&$n = 27$\break
$(\frac{k}{3} + \frac{1}{54})\cdot 2\pi$
\end{tabular}\nobreak

\vfill\nobreak
\noindent\begin{tabular}{
>{\centering\arraybackslash}m{0.33\textwidth}
>{\centering\arraybackslash}m{0.33\textwidth}
>{\centering\arraybackslash}m{0.33\textwidth}}
\href{https://www.inf.fu-berlin.de/inst/ag-ti/software/DiscreteHopfFibration/gallery.html?f=IxCn/3/1200cells_20tubes}{\includegraphics[scale=1.1]{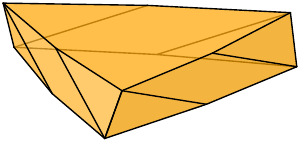}}
&
\href{https://www.inf.fu-berlin.de/inst/ag-ti/software/DiscreteHopfFibration/gallery.html?f=IxCn/3/1320cells_20tubes}{\includegraphics[scale=1.1]{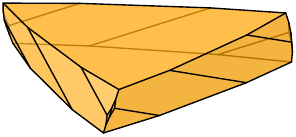}}
&
\href{https://www.inf.fu-berlin.de/inst/ag-ti/software/DiscreteHopfFibration/gallery.html?f=IxCn/3/1440cells_20tubes}{\includegraphics[scale=1.1]{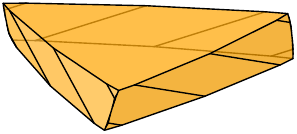}}
\\[1.5mm]
$n = 10, 30$\break
$(\frac{2}{3} + \frac{1}{60})\cdot 2\pi$
&$n = 11, 33$\break
$(\frac{1}{3} + \frac{1}{66})\cdot 2\pi$
&$n = 36$\break
$(\frac{k}{3} + \frac{1}{72})\cdot 2\pi$
\end{tabular}\nobreak

\vfill\nobreak
\noindent\begin{tabular}{
>{\centering\arraybackslash}m{0.33\textwidth}
>{\centering\arraybackslash}m{0.33\textwidth}
>{\centering\arraybackslash}m{0.33\textwidth}}
\href{https://www.inf.fu-berlin.de/inst/ag-ti/software/DiscreteHopfFibration/gallery.html?f=IxCn/3/1560cells_20tubes}{\includegraphics[scale=1.1]{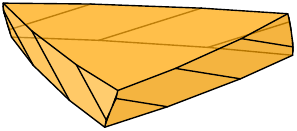}}
&
\href{https://www.inf.fu-berlin.de/inst/ag-ti/software/DiscreteHopfFibration/gallery.html?f=IxCn/3/1680cells_20tubes}{\includegraphics[scale=1.1]{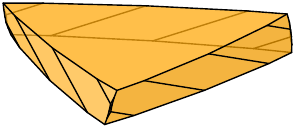}}
&
\href{https://www.inf.fu-berlin.de/inst/ag-ti/software/DiscreteHopfFibration/gallery.html?f=IxCn/3/1800cells_20tubes}{\includegraphics[scale=1.1]{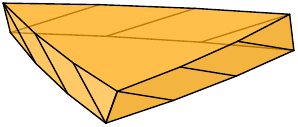}}
\\[1.5mm]
$n = 13, 39$\break
$(\frac{2}{3} + \frac{1}{78})\cdot 2\pi$
&$n = 14, 42$\break
$(\frac{1}{3} + \frac{1}{84})\cdot 2\pi$
&$n = 45$\break
$(\frac{k}{3} + \frac{1}{90})\cdot 2\pi$
\end{tabular}\nobreak

\begin{figure}[H]
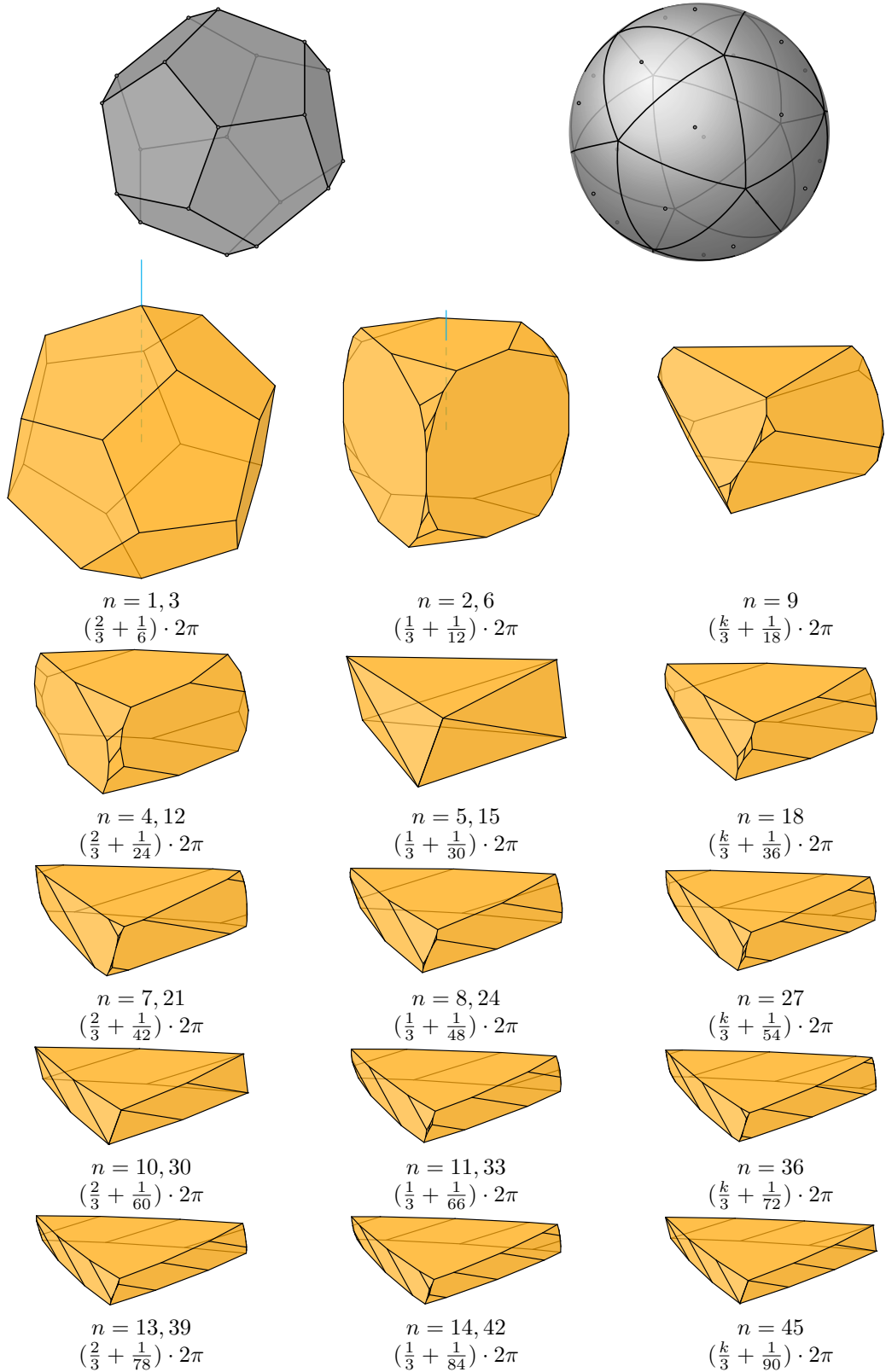

\caption{
$G=\pm[I\times C_n]$,
$G^h={+I}$,
3-fold rotation center
$p=\frac1{\sqrt3}(-1, -1, -1)$.
$H=\langle [-\omega, 1], [1, e_n] \rangle$.
20 tubes,
each with $\mathrm{lcm}(2n, 6)$ cells.
Alternate group: $\pm[I\times D_{2n}]$.
When $n=1$ or $n=3$,
the cells of a tube are disconnected from each other.
}
\label{fig:IxCn_3fold}
\end{figure}\endgroup
\newpage

\subsubsection{\texorpdfstring{$\pm[I\times C_n]$}{+-[IxCn]}, 2-fold rotation center}
\hrule height 0pt
\vskip -1mm  \vfill
\hrule height 0pt \nobreak
\begin{minipage}{0.5\textwidth}
\centering
\includegraphics[scale=1]{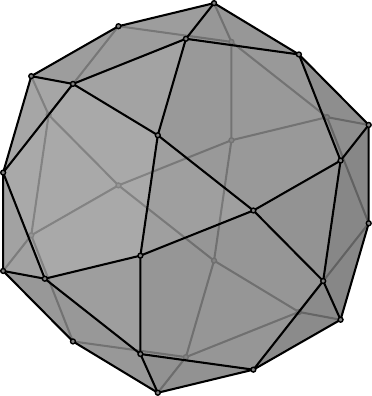}
\end{minipage}%
\begin{minipage}{0.5\textwidth}
\centering
\includegraphics[scale=1]{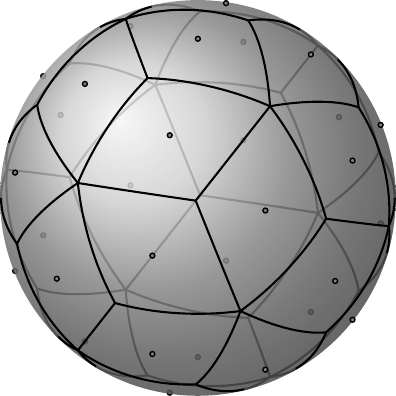}
\end{minipage}
\vskip 5pt plus 0.5fill
\begingroup
\setlength{\tabcolsep}{0pt}
\noindent\begin{tabular}{
>{\centering\arraybackslash}m{0.33\textwidth}
>{\centering\arraybackslash}m{0.33\textwidth}
>{\centering\arraybackslash}m{0.33\textwidth}}
\href{https://www.inf.fu-berlin.de/inst/ag-ti/software/DiscreteHopfFibration/gallery.html?f=IxCn/2/120cells_30tubes}{\vbox{\vskip-3mm \includegraphics[scale=1.1]{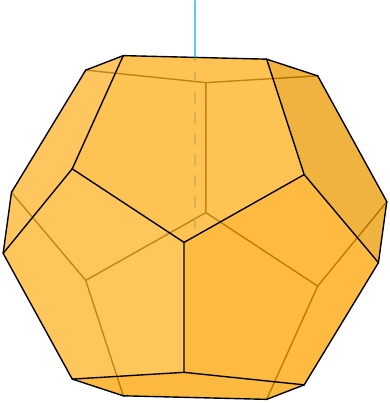}}}
&
\href{https://www.inf.fu-berlin.de/inst/ag-ti/software/DiscreteHopfFibration/gallery.html?f=IxCn/2/240cells_30tubes}{\includegraphics[scale=1.1]{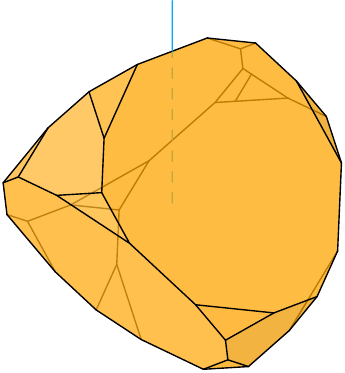}}
&
\href{https://www.inf.fu-berlin.de/inst/ag-ti/software/DiscreteHopfFibration/gallery.html?f=IxCn/2/360cells_30tubes}{\includegraphics[scale=1.1]{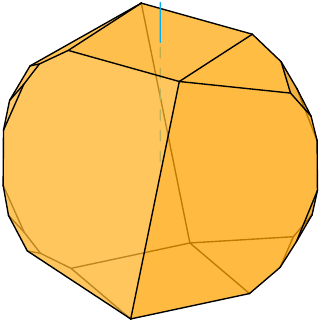}}
\\[1.5mm]
$n = 1, 2$\break
$(\frac{1}{2} + \frac{1}{4})\cdot 2\pi$
&$n = 4$\break
$(\frac{k}{2} + \frac{1}{8})\cdot 2\pi$
&$n = 3, 6$\break
$(\frac{1}{2} + \frac{1}{12})\cdot 2\pi$
\end{tabular}\nobreak

\vfill\nobreak
\noindent\begin{tabular}{
>{\centering\arraybackslash}m{0.33\textwidth}
>{\centering\arraybackslash}m{0.33\textwidth}
>{\centering\arraybackslash}m{0.33\textwidth}}
\href{https://www.inf.fu-berlin.de/inst/ag-ti/software/DiscreteHopfFibration/gallery.html?f=IxCn/2/480cells_30tubes}{\includegraphics[scale=1.1]{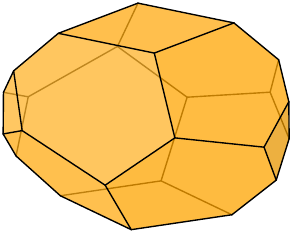}}
&
\href{https://www.inf.fu-berlin.de/inst/ag-ti/software/DiscreteHopfFibration/gallery.html?f=IxCn/2/600cells_30tubes}{\includegraphics[scale=1.1]{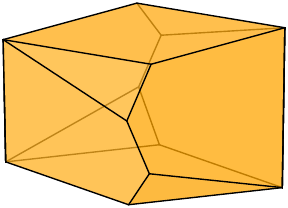}}
&
\href{https://www.inf.fu-berlin.de/inst/ag-ti/software/DiscreteHopfFibration/gallery.html?f=IxCn/2/720cells_30tubes}{\includegraphics[scale=1.1]{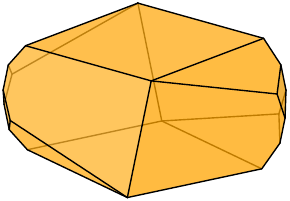}}
\\[1.5mm]
$n = 8$\break
$(\frac{k}{2} + \frac{1}{16})\cdot 2\pi$
&$n = 5, 10$\break
$(\frac{1}{2} + \frac{1}{20})\cdot 2\pi$
&$n = 12$\break
$(\frac{k}{2} + \frac{1}{24})\cdot 2\pi$
\end{tabular}\nobreak

\vfill\nobreak
\noindent\begin{tabular}{
>{\centering\arraybackslash}m{0.33\textwidth}
>{\centering\arraybackslash}m{0.33\textwidth}
>{\centering\arraybackslash}m{0.33\textwidth}}
\href{https://www.inf.fu-berlin.de/inst/ag-ti/software/DiscreteHopfFibration/gallery.html?f=IxCn/2/840cells_30tubes}{\includegraphics[scale=1.1]{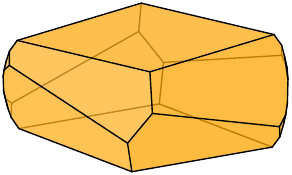}}
&
\href{https://www.inf.fu-berlin.de/inst/ag-ti/software/DiscreteHopfFibration/gallery.html?f=IxCn/2/960cells_30tubes}{\includegraphics[scale=1.1]{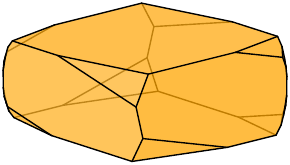}}
&
\href{https://www.inf.fu-berlin.de/inst/ag-ti/software/DiscreteHopfFibration/gallery.html?f=IxCn/2/1080cells_30tubes}{\includegraphics[scale=1.1]{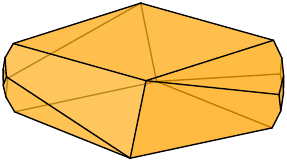}}
\\[1.5mm]
$n = 7, 14$\break
$(\frac{1}{2} + \frac{1}{28})\cdot 2\pi$
&$n = 16$\break
$(\frac{k}{2} + \frac{1}{32})\cdot 2\pi$
&$n = 9, 18$\break
$(\frac{1}{2} + \frac{1}{36})\cdot 2\pi$
\end{tabular}\nobreak

\vfill\nobreak
\noindent\begin{tabular}{
>{\centering\arraybackslash}m{0.33\textwidth}
>{\centering\arraybackslash}m{0.33\textwidth}
>{\centering\arraybackslash}m{0.33\textwidth}}
\href{https://www.inf.fu-berlin.de/inst/ag-ti/software/DiscreteHopfFibration/gallery.html?f=IxCn/2/1200cells_30tubes}{\includegraphics[scale=1.1]{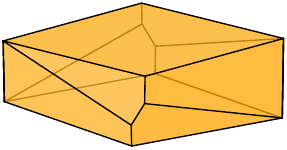}}
&
\href{https://www.inf.fu-berlin.de/inst/ag-ti/software/DiscreteHopfFibration/gallery.html?f=IxCn/2/1320cells_30tubes}{\includegraphics[scale=1.1]{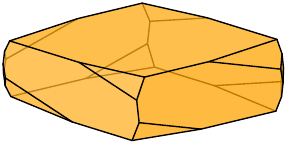}}
&
\href{https://www.inf.fu-berlin.de/inst/ag-ti/software/DiscreteHopfFibration/gallery.html?f=IxCn/2/1440cells_30tubes}{\includegraphics[scale=1.1]{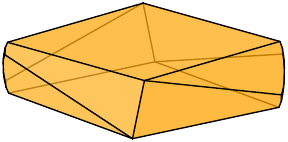}}
\\[1.5mm]
$n = 20$\break
$(\frac{k}{2} + \frac{1}{40})\cdot 2\pi$
&$n = 11, 22$\break
$(\frac{1}{2} + \frac{1}{44})\cdot 2\pi$
&$n = 24$\break
$(\frac{k}{2} + \frac{1}{48})\cdot 2\pi$
\end{tabular}\nobreak

\vfill\nobreak
\noindent\begin{tabular}{
>{\centering\arraybackslash}m{0.33\textwidth}
>{\centering\arraybackslash}m{0.33\textwidth}
>{\centering\arraybackslash}m{0.33\textwidth}}
\href{https://www.inf.fu-berlin.de/inst/ag-ti/software/DiscreteHopfFibration/gallery.html?f=IxCn/2/1560cells_30tubes}{\includegraphics[scale=1.1]{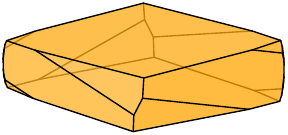}}
&
\href{https://www.inf.fu-berlin.de/inst/ag-ti/software/DiscreteHopfFibration/gallery.html?f=IxCn/2/1680cells_30tubes}{\includegraphics[scale=1.1]{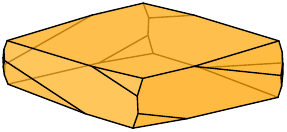}}
&
\href{https://www.inf.fu-berlin.de/inst/ag-ti/software/DiscreteHopfFibration/gallery.html?f=IxCn/2/1800cells_30tubes}{\includegraphics[scale=1.1]{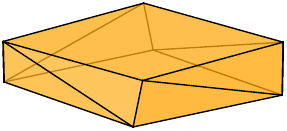}}
\\[1.5mm]
$n = 13, 26$\break
$(\frac{1}{2} + \frac{1}{52})\cdot 2\pi$
&$n = 28$\break
$(\frac{k}{2} + \frac{1}{56})\cdot 2\pi$
&$n = 15, 30$\break
$(\frac{1}{2} + \frac{1}{60})\cdot 2\pi$
\end{tabular}\nobreak

\begin{figure}[H]
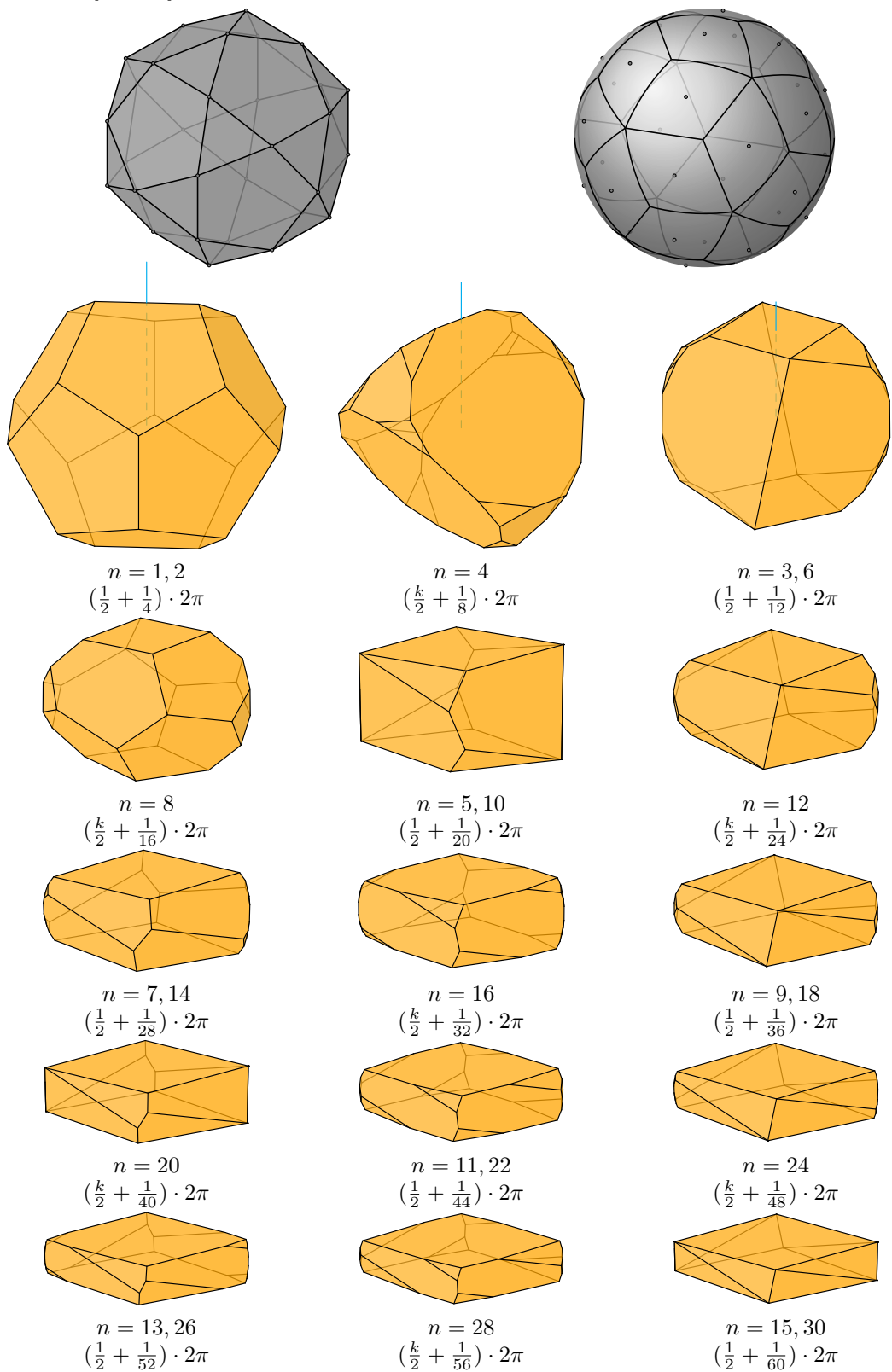

\caption{
$G=\pm[I\times C_n]$,
$G^h={+I}$,
2-fold rotation center
$p=\frac12(1, \frac1\varphi, \varphi)$, where $\varphi=\frac{1+\sqrt5}{2}$.
The $G^h$-orbit polytope is an icosidodecahedron.
The corresponding Voronoi diagram on the 2-sphere
has the structure of a rhombic triacontahedron.
$H=\langle [i_I, 1], [1, e_n] \rangle$.
30 tubes,
each with $\mathrm{lcm}(2n, 4)$ cells.
Alternate group: $\pm[I\times D_{2n}]$.
When $n=1$, $2$, or $4$,
the cells of a tube are disconnected from each other.
}
\label{fig:IxCn_2fold}
\end{figure}\endgroup
\newpage

\subsection{\texorpdfstring{$\pm[O\times C_n]$}{+-[OxCn]}}
\subsubsection{\texorpdfstring{$\pm[O\times C_n]$}{+-[OxCn]}, 4-fold
  rotation center}
\label{O-4fold}
\begin{minipage}{0.5\textwidth}
\centering
\includegraphics[scale=1]{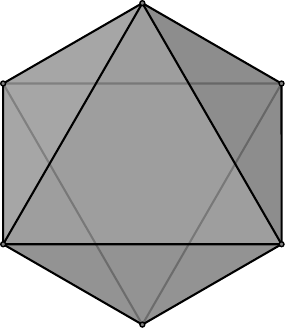}
\end{minipage}%
\begin{minipage}{0.5\textwidth}
\centering
\includegraphics[scale=1]{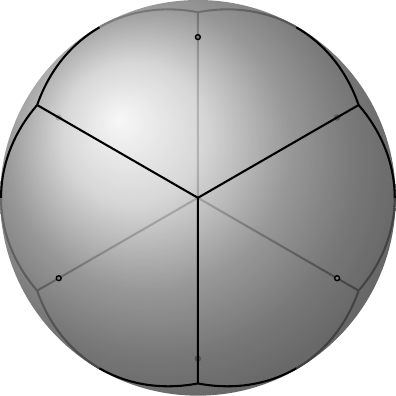}
\end{minipage}
\vskip 5pt plus 0.5fill
\begingroup
\setlength{\tabcolsep}{0pt}
\noindent\begin{tabular}{
>{\centering\arraybackslash}m{0.33\textwidth}
>{\centering\arraybackslash}m{0.33\textwidth}
>{\centering\arraybackslash}m{0.33\textwidth}}
\href{https://www.inf.fu-berlin.de/inst/ag-ti/software/DiscreteHopfFibration/gallery.html?f=OxCn/4/48cells_6tubes}{\includegraphics[scale=0.8]{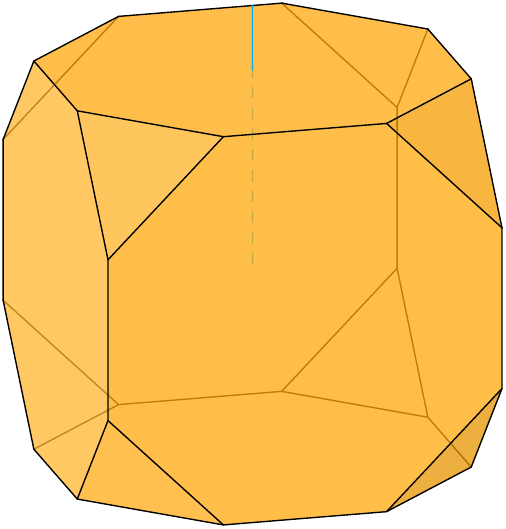}}
&
\href{https://www.inf.fu-berlin.de/inst/ag-ti/software/DiscreteHopfFibration/gallery.html?f=OxCn/4/96cells_6tubes}{\includegraphics[scale=0.8]{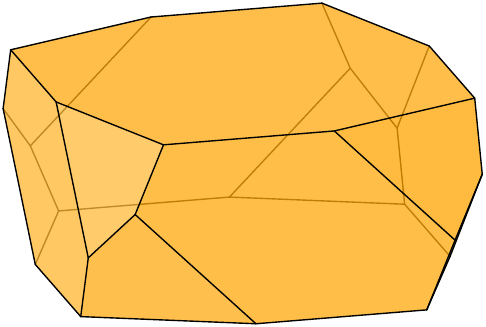}}
&
\href{https://www.inf.fu-berlin.de/inst/ag-ti/software/DiscreteHopfFibration/gallery.html?f=OxCn/4/144cells_6tubes}{\includegraphics[scale=0.8]{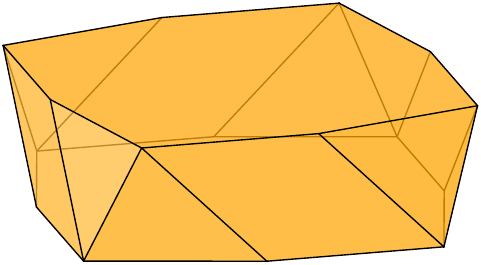}}
\\[1.5mm]
$n = 1, 2, 4$\break
$(\frac{3}{4} + \frac{1}{8})\cdot 2\pi$
&$n = 8$\break
$(\frac{k}{4} + \frac{1}{16})\cdot 2\pi$
&$n = 3, 6, 12$\break
$(\frac{1}{4} + \frac{1}{24})\cdot 2\pi$
\end{tabular}\nobreak

\vfill\nobreak
\noindent\begin{tabular}{
>{\centering\arraybackslash}m{0.33\textwidth}
>{\centering\arraybackslash}m{0.33\textwidth}
>{\centering\arraybackslash}m{0.33\textwidth}}
\href{https://www.inf.fu-berlin.de/inst/ag-ti/software/DiscreteHopfFibration/gallery.html?f=OxCn/4/192cells_6tubes}{\includegraphics[scale=0.8]{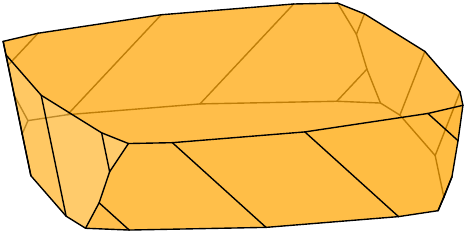}}
&
\href{https://www.inf.fu-berlin.de/inst/ag-ti/software/DiscreteHopfFibration/gallery.html?f=OxCn/4/240cells_6tubes}{\includegraphics[scale=0.8]{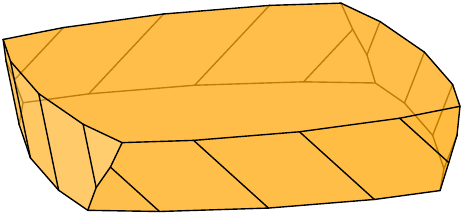}}
&
\href{https://www.inf.fu-berlin.de/inst/ag-ti/software/DiscreteHopfFibration/gallery.html?f=OxCn/4/288cells_6tubes}{\includegraphics[scale=0.8]{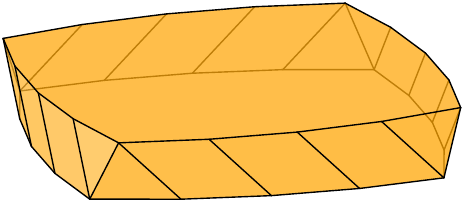}}
\\[1.5mm]
$n = 16$\break
$(\frac{k}{4} + \frac{1}{32})\cdot 2\pi$
&$n = 5, 10, 20$\break
$(\frac{3}{4} + \frac{1}{40})\cdot 2\pi$
&$n = 24$\break
$(\frac{k}{4} + \frac{1}{48})\cdot 2\pi$
\end{tabular}\nobreak

\vfill\nobreak
\noindent\begin{tabular}{
>{\centering\arraybackslash}m{0.33\textwidth}
>{\centering\arraybackslash}m{0.33\textwidth}
>{\centering\arraybackslash}m{0.33\textwidth}}
\href{https://www.inf.fu-berlin.de/inst/ag-ti/software/DiscreteHopfFibration/gallery.html?f=OxCn/4/336cells_6tubes}{\includegraphics[scale=0.8]{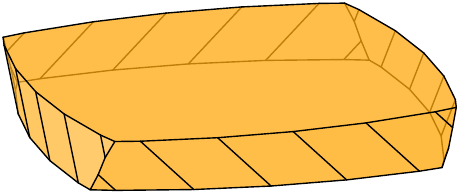}}
&
\href{https://www.inf.fu-berlin.de/inst/ag-ti/software/DiscreteHopfFibration/gallery.html?f=OxCn/4/384cells_6tubes}{\includegraphics[scale=0.8]{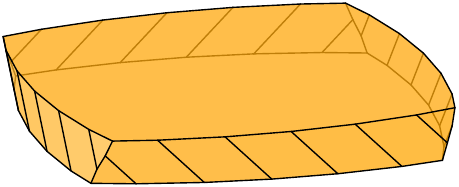}}
&
\href{https://www.inf.fu-berlin.de/inst/ag-ti/software/DiscreteHopfFibration/gallery.html?f=OxCn/4/432cells_6tubes}{\includegraphics[scale=0.8]{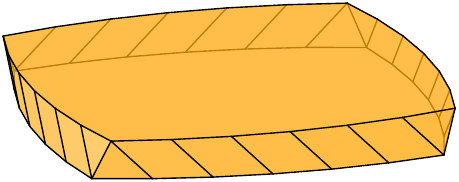}}
\\[1.5mm]
$n = 7, 14, 28$\break
$(\frac{1}{4} + \frac{1}{56})\cdot 2\pi$
&$n = 32$\break
$(\frac{k}{4} + \frac{1}{64})\cdot 2\pi$
&$n = 9, 18, 36$\break
$(\frac{3}{4} + \frac{1}{72})\cdot 2\pi$
\end{tabular}\nobreak

\vfill\nobreak
\noindent\begin{tabular}{
>{\centering\arraybackslash}m{0.33\textwidth}
>{\centering\arraybackslash}m{0.33\textwidth}
>{\centering\arraybackslash}m{0.33\textwidth}}
\href{https://www.inf.fu-berlin.de/inst/ag-ti/software/DiscreteHopfFibration/gallery.html?f=OxCn/4/480cells_6tubes}{\includegraphics[scale=0.8]{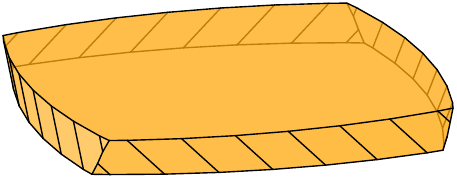}}
&
\href{https://www.inf.fu-berlin.de/inst/ag-ti/software/DiscreteHopfFibration/gallery.html?f=OxCn/4/528cells_6tubes}{\includegraphics[scale=0.8]{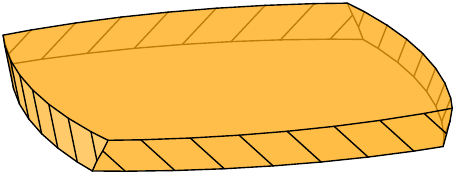}}
&
\href{https://www.inf.fu-berlin.de/inst/ag-ti/software/DiscreteHopfFibration/gallery.html?f=OxCn/4/576cells_6tubes}{\includegraphics[scale=0.8]{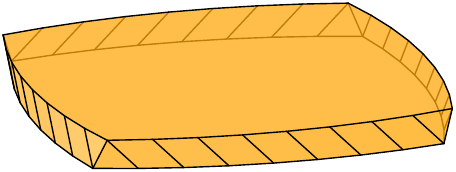}}
\\[1.5mm]
$n = 40$\break
$(\frac{k}{4} + \frac{1}{80})\cdot 2\pi$
&$n = 11, 22, 44$\break
$(\frac{1}{4} + \frac{1}{88})\cdot 2\pi$
&$n = 48$\break
$(\frac{k}{4} + \frac{1}{96})\cdot 2\pi$
\end{tabular}\nobreak

\vfill\nobreak
\noindent\begin{tabular}{
>{\centering\arraybackslash}m{0.33\textwidth}
>{\centering\arraybackslash}m{0.33\textwidth}
>{\centering\arraybackslash}m{0.33\textwidth}}
\href{https://www.inf.fu-berlin.de/inst/ag-ti/software/DiscreteHopfFibration/gallery.html?f=OxCn/4/624cells_6tubes}{\includegraphics[scale=0.8]{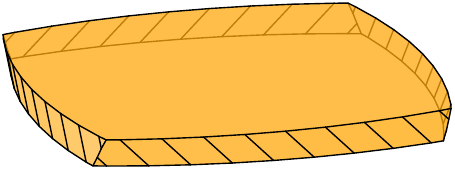}}
&
\href{https://www.inf.fu-berlin.de/inst/ag-ti/software/DiscreteHopfFibration/gallery.html?f=OxCn/4/672cells_6tubes}{\includegraphics[scale=0.8]{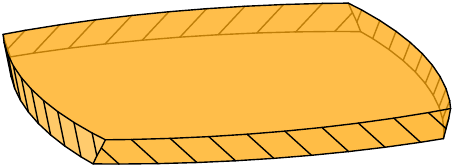}}
&
\href{https://www.inf.fu-berlin.de/inst/ag-ti/software/DiscreteHopfFibration/gallery.html?f=OxCn/4/720cells_6tubes}{\includegraphics[scale=0.8]{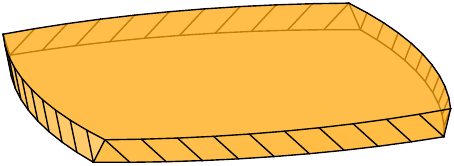}}
\\[1.5mm]
$n = 13, 26, 52$\break
$(\frac{3}{4} + \frac{1}{104})\cdot 2\pi$
&$n = 56$\break
$(\frac{k}{4} + \frac{1}{112})\cdot 2\pi$
&$n = 15, 30, 60$\break
$(\frac{1}{4} + \frac{1}{120})\cdot 2\pi$
\end{tabular}\nobreak

\begin{figure}[H]
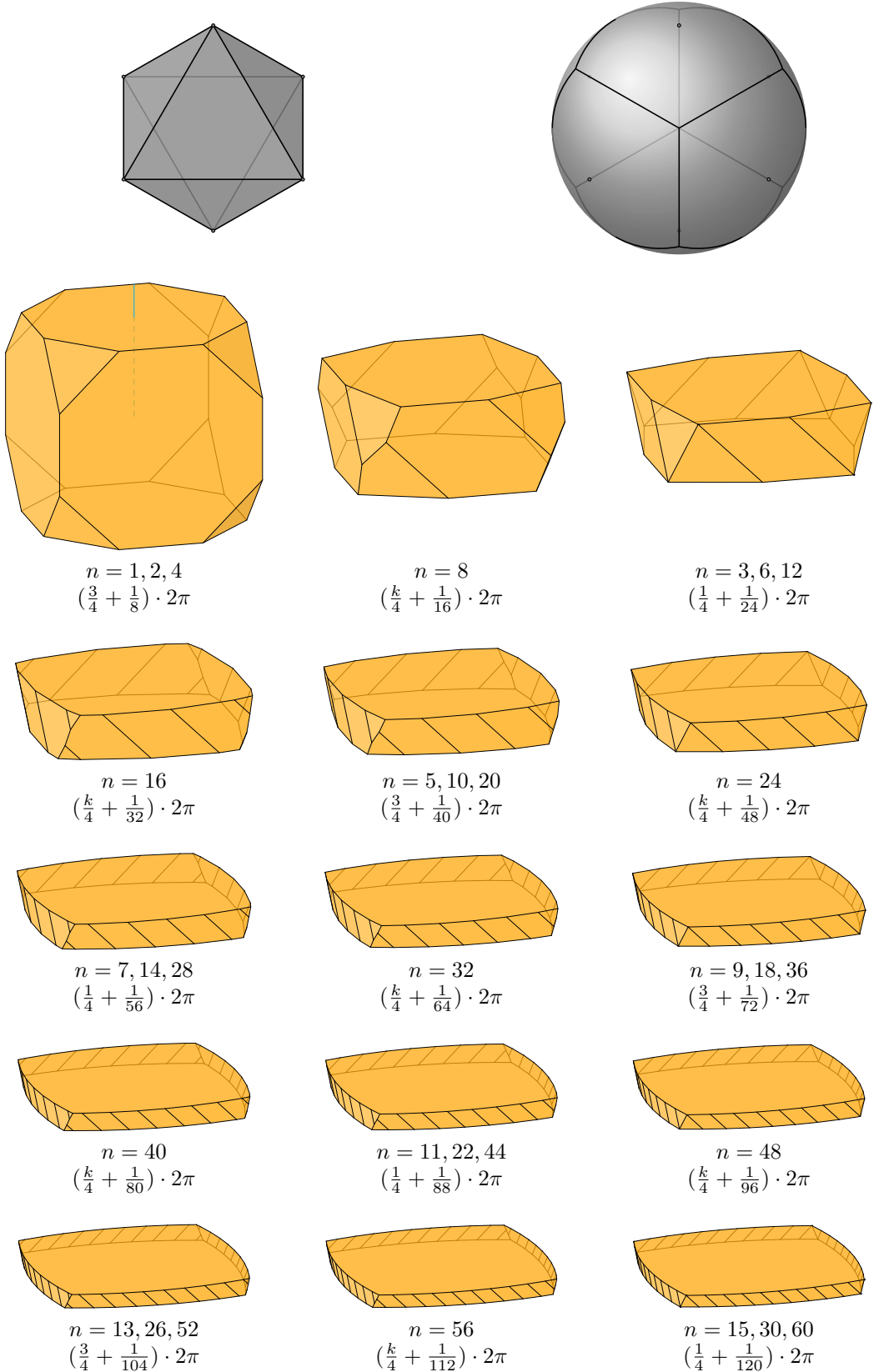

\caption{
$G=\pm[O\times C_n]$,
$G^h={+O}$,
4-fold rotation center
$p=(0, 1, 0)$.
$H=\langle [-\omega i_O, 1], [1, e_n] \rangle$.
6 tubes,
each with $\mathrm{lcm}(2n, 8)$ cells.
Alternate group: $\pm[O\times D_{2n}]$.
}
\label{fig:OxCn_4fold}
\end{figure}\endgroup
\newpage

\subsubsection{\texorpdfstring{$\pm[O\times C_n]$}{+-[OxCn]}, 3-fold rotation center}
\begin{minipage}{0.5\textwidth}
\centering
\includegraphics[scale=1]{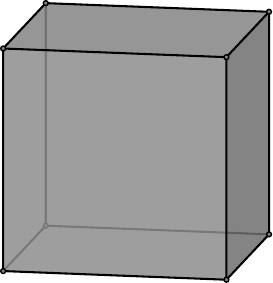}
\end{minipage}%
\begin{minipage}{0.5\textwidth}
\centering
\includegraphics[scale=1]{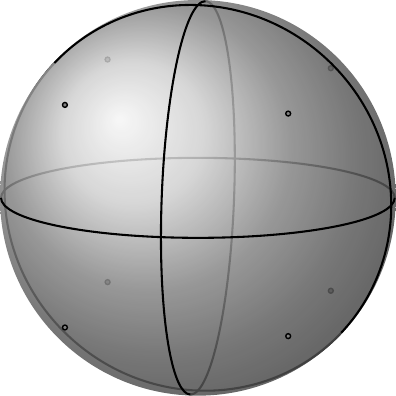}
\end{minipage}
\vskip 5pt plus 0.5fill
\begingroup
\setlength{\tabcolsep}{0pt}
\noindent\begin{tabular}{
>{\centering\arraybackslash}m{0.33\textwidth}
>{\centering\arraybackslash}m{0.33\textwidth}
>{\centering\arraybackslash}m{0.33\textwidth}}
\href{https://www.inf.fu-berlin.de/inst/ag-ti/software/DiscreteHopfFibration/gallery.html?f=OxCn/3/48cells_8tubes}{\includegraphics[scale=0.8]{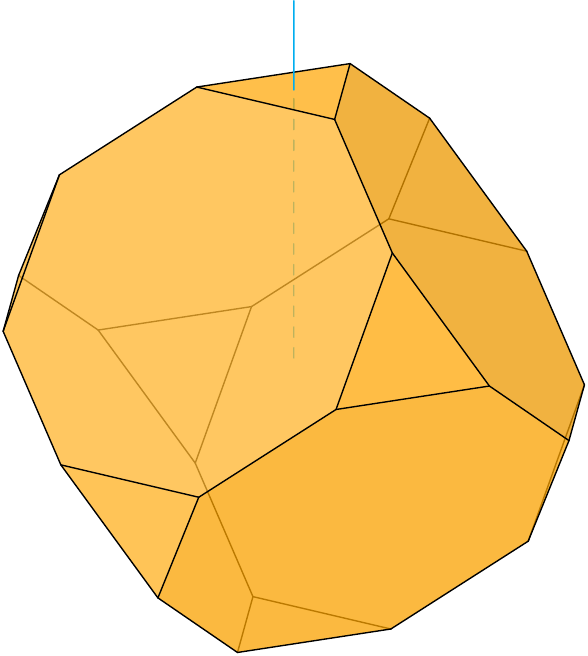}}
&
\href{https://www.inf.fu-berlin.de/inst/ag-ti/software/DiscreteHopfFibration/gallery.html?f=OxCn/3/96cells_8tubes}{\includegraphics[scale=0.8]{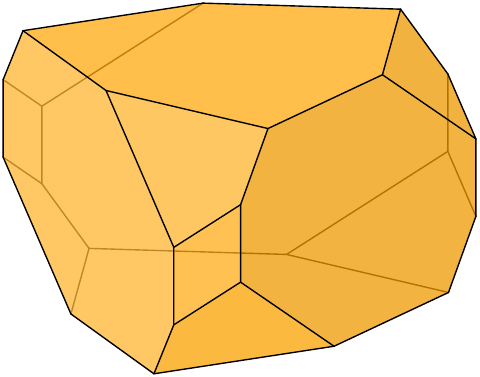}}
&
\href{https://www.inf.fu-berlin.de/inst/ag-ti/software/DiscreteHopfFibration/gallery.html?f=OxCn/3/144cells_8tubes}{\includegraphics[scale=0.8]{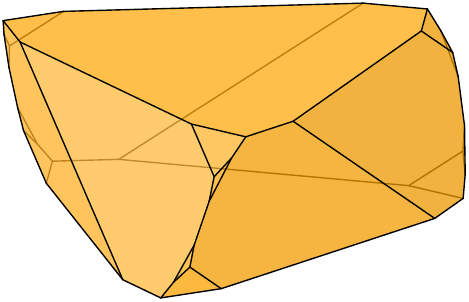}}
\\[1.5mm]
$n = 1, 3$\break
$(\frac{2}{3} + \frac{1}{6})\cdot 2\pi$
&$n = 2, 6$\break
$(\frac{1}{3} + \frac{1}{12})\cdot 2\pi$
&$n = 9$\break
$(\frac{k}{3} + \frac{1}{18})\cdot 2\pi$
\end{tabular}\nobreak

\vfill\nobreak
\noindent\begin{tabular}{
>{\centering\arraybackslash}m{0.33\textwidth}
>{\centering\arraybackslash}m{0.33\textwidth}
>{\centering\arraybackslash}m{0.33\textwidth}}
\href{https://www.inf.fu-berlin.de/inst/ag-ti/software/DiscreteHopfFibration/gallery.html?f=OxCn/3/192cells_8tubes}{\includegraphics[scale=0.8]{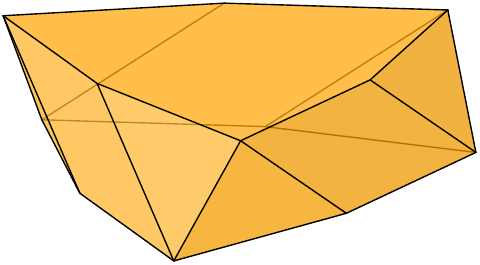}}
&
\href{https://www.inf.fu-berlin.de/inst/ag-ti/software/DiscreteHopfFibration/gallery.html?f=OxCn/3/240cells_8tubes}{\includegraphics[scale=0.8]{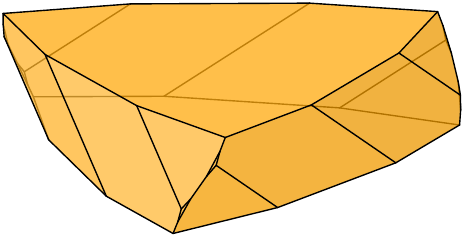}}
&
\href{https://www.inf.fu-berlin.de/inst/ag-ti/software/DiscreteHopfFibration/gallery.html?f=OxCn/3/288cells_8tubes}{\includegraphics[scale=0.8]{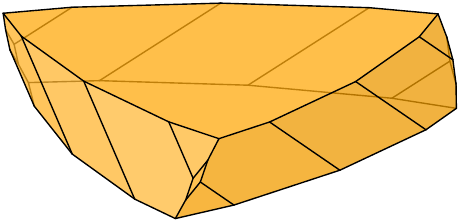}}
\\[1.5mm]
$n = 4, 12$\break
$(\frac{2}{3} + \frac{1}{24})\cdot 2\pi$
&$n = 5, 15$\break
$(\frac{1}{3} + \frac{1}{30})\cdot 2\pi$
&$n = 18$\break
$(\frac{k}{3} + \frac{1}{36})\cdot 2\pi$
\end{tabular}\nobreak

\vfill\nobreak
\noindent\begin{tabular}{
>{\centering\arraybackslash}m{0.33\textwidth}
>{\centering\arraybackslash}m{0.33\textwidth}
>{\centering\arraybackslash}m{0.33\textwidth}}
\href{https://www.inf.fu-berlin.de/inst/ag-ti/software/DiscreteHopfFibration/gallery.html?f=OxCn/3/336cells_8tubes}{\includegraphics[scale=0.8]{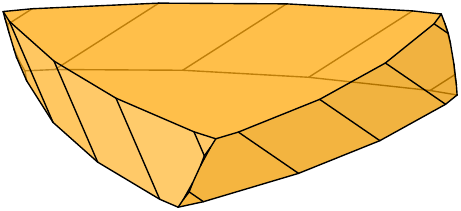}}
&
\href{https://www.inf.fu-berlin.de/inst/ag-ti/software/DiscreteHopfFibration/gallery.html?f=OxCn/3/384cells_8tubes}{\includegraphics[scale=0.8]{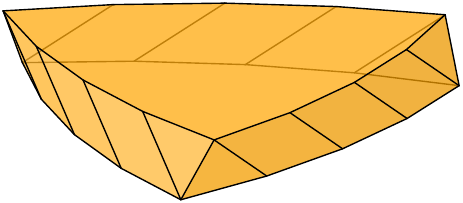}}
&
\href{https://www.inf.fu-berlin.de/inst/ag-ti/software/DiscreteHopfFibration/gallery.html?f=OxCn/3/432cells_8tubes}{\includegraphics[scale=0.8]{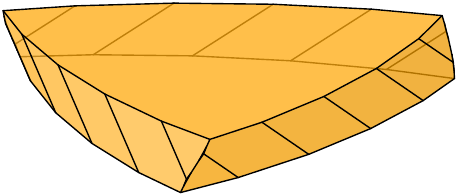}}
\\[1.5mm]
$n = 7, 21$\break
$(\frac{2}{3} + \frac{1}{42})\cdot 2\pi$
&$n = 8, 24$\break
$(\frac{1}{3} + \frac{1}{48})\cdot 2\pi$
&$n = 27$\break
$(\frac{k}{3} + \frac{1}{54})\cdot 2\pi$
\end{tabular}\nobreak

\vfill\nobreak
\noindent\begin{tabular}{
>{\centering\arraybackslash}m{0.33\textwidth}
>{\centering\arraybackslash}m{0.33\textwidth}
>{\centering\arraybackslash}m{0.33\textwidth}}
\href{https://www.inf.fu-berlin.de/inst/ag-ti/software/DiscreteHopfFibration/gallery.html?f=OxCn/3/480cells_8tubes}{\includegraphics[scale=0.8]{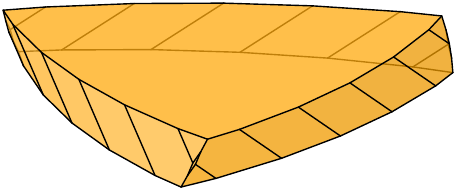}}
&
\href{https://www.inf.fu-berlin.de/inst/ag-ti/software/DiscreteHopfFibration/gallery.html?f=OxCn/3/528cells_8tubes}{\includegraphics[scale=0.8]{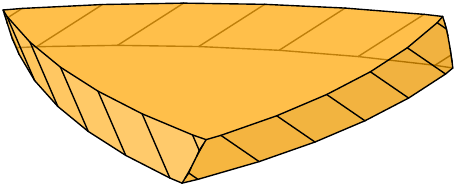}}
&
\href{https://www.inf.fu-berlin.de/inst/ag-ti/software/DiscreteHopfFibration/gallery.html?f=OxCn/3/576cells_8tubes}{\includegraphics[scale=0.8]{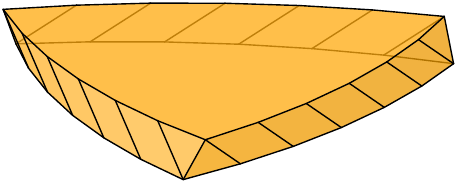}}
\\[1.5mm]
$n = 10, 30$\break
$(\frac{2}{3} + \frac{1}{60})\cdot 2\pi$
&$n = 11, 33$\break
$(\frac{1}{3} + \frac{1}{66})\cdot 2\pi$
&$n = 36$\break
$(\frac{k}{3} + \frac{1}{72})\cdot 2\pi$
\end{tabular}\nobreak

\vfill\nobreak
\noindent\begin{tabular}{
>{\centering\arraybackslash}m{0.33\textwidth}
>{\centering\arraybackslash}m{0.33\textwidth}
>{\centering\arraybackslash}m{0.33\textwidth}}
\href{https://www.inf.fu-berlin.de/inst/ag-ti/software/DiscreteHopfFibration/gallery.html?f=OxCn/3/624cells_8tubes}{\includegraphics[scale=0.8]{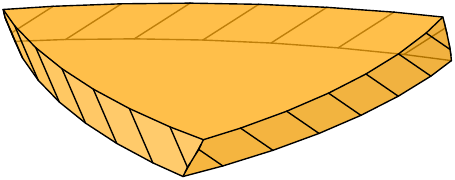}}
&
\href{https://www.inf.fu-berlin.de/inst/ag-ti/software/DiscreteHopfFibration/gallery.html?f=OxCn/3/672cells_8tubes}{\includegraphics[scale=0.8]{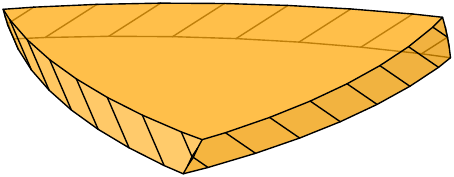}}
&
\href{https://www.inf.fu-berlin.de/inst/ag-ti/software/DiscreteHopfFibration/gallery.html?f=OxCn/3/720cells_8tubes}{\includegraphics[scale=0.8]{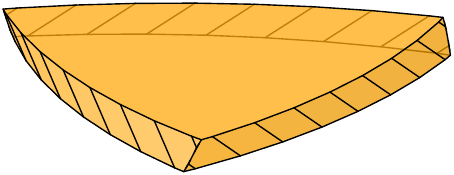}}
\\[1.5mm]
$n = 13, 39$\break
$(\frac{2}{3} + \frac{1}{78})\cdot 2\pi$
&$n = 14, 42$\break
$(\frac{1}{3} + \frac{1}{84})\cdot 2\pi$
&$n = 45$\break
$(\frac{k}{3} + \frac{1}{90})\cdot 2\pi$
\end{tabular}\nobreak

\begin{figure}[H]
\caption{
$G=\pm[O\times C_n]$,
$G^h={+O}$,
3-fold rotation center
$p=\frac1{\sqrt3}(-1, -1, -1)$.
$H=\langle [-\omega, 1], [1, e_n] \rangle$.
8 tubes,
each with $\mathrm{lcm}(2n, 4)$ cells.
Alternate group: $\pm[O\times D_{2n}]$.
}
\label{fig:OxCn_3fold}
\end{figure}\endgroup
\newpage

\subsubsection{\texorpdfstring{$\pm[O\times C_n]$}{+-[OxCn]}, 2-fold rotation center}
\hrule height 0pt
\vskip -2mm  \vfill
\hrule height 0pt \nobreak
\begin{minipage}{0.5\textwidth}
\centering
\includegraphics[scale=1]{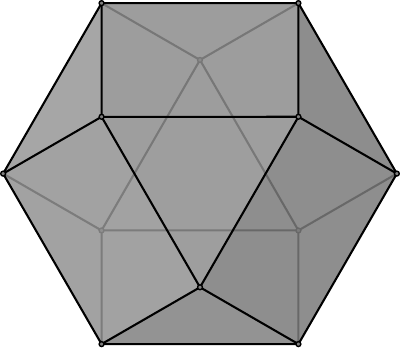}
\end{minipage}%
\begin{minipage}{0.5\textwidth}
\centering
\includegraphics[scale=1]{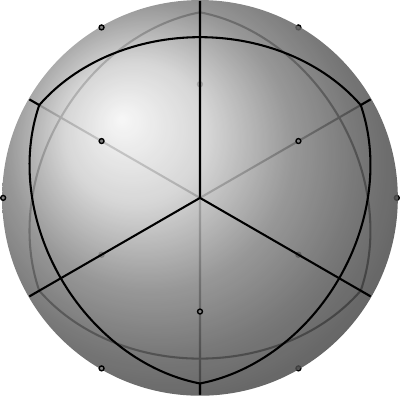}
\end{minipage}
\vskip 5pt plus 0.5fill
\begingroup
\setlength{\tabcolsep}{0pt}
\noindent\begin{tabular}{
>{\centering\arraybackslash}m{0.33\textwidth}
>{\centering\arraybackslash}m{0.33\textwidth}
>{\centering\arraybackslash}m{0.33\textwidth}}
\href{https://www.inf.fu-berlin.de/inst/ag-ti/software/DiscreteHopfFibration/gallery.html?f=OxCn/2/48cells_12tubes}{\vbox{\vskip-3mm \includegraphics[scale=0.8]{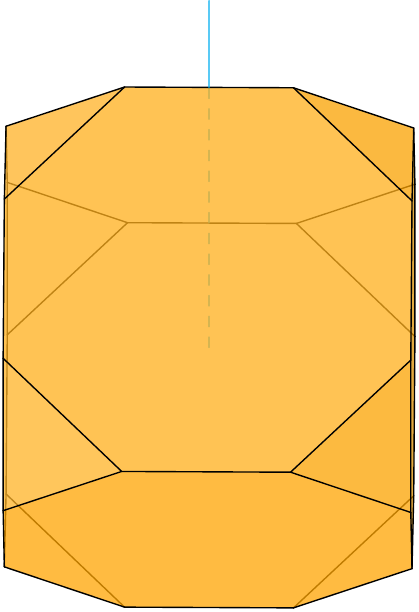}}}
&
\href{https://www.inf.fu-berlin.de/inst/ag-ti/software/DiscreteHopfFibration/gallery.html?f=OxCn/2/96cells_12tubes}{\vbox{\vskip-3mm \includegraphics[scale=0.8]{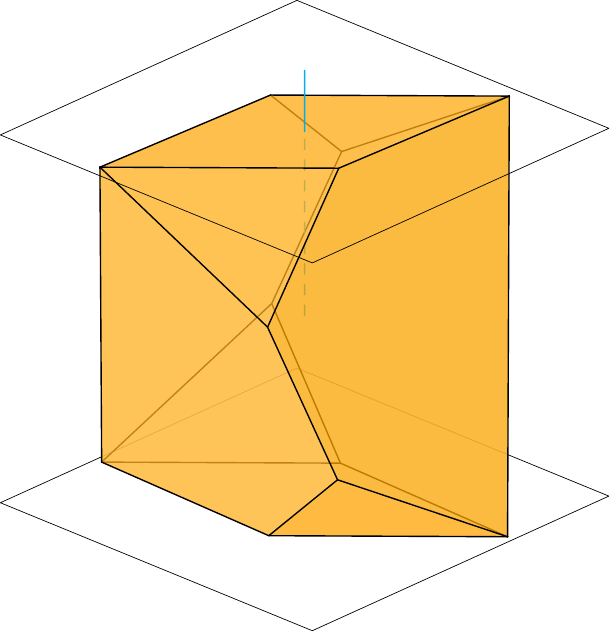}}}
&
\href{https://www.inf.fu-berlin.de/inst/ag-ti/software/DiscreteHopfFibration/gallery.html?f=OxCn/2/144cells_12tubes}{\includegraphics[scale=0.8]{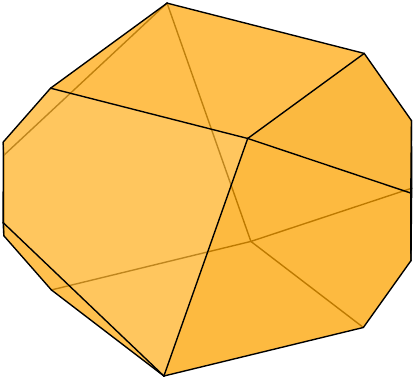}}
\\[1.5mm]
$n = 1, 2$\break
$(\frac{1}{2} + \frac{1}{4})\cdot 2\pi$
&$n = 4$\break
$(\frac{k}{2} + \frac{1}{8})\cdot 2\pi$
&$n = 3, 6$\break
$(\frac{1}{2} + \frac{1}{12})\cdot 2\pi$
\end{tabular}\nobreak

\vfill\nobreak
\noindent\begin{tabular}{
>{\centering\arraybackslash}m{0.33\textwidth}
>{\centering\arraybackslash}m{0.33\textwidth}
>{\centering\arraybackslash}m{0.33\textwidth}}
\href{https://www.inf.fu-berlin.de/inst/ag-ti/software/DiscreteHopfFibration/gallery.html?f=OxCn/2/192cells_12tubes}{\includegraphics[scale=0.8]{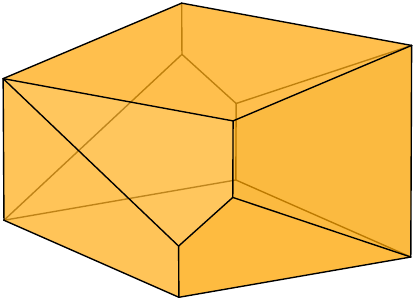}}
&
\href{https://www.inf.fu-berlin.de/inst/ag-ti/software/DiscreteHopfFibration/gallery.html?f=OxCn/2/240cells_12tubes}{\includegraphics[scale=0.8]{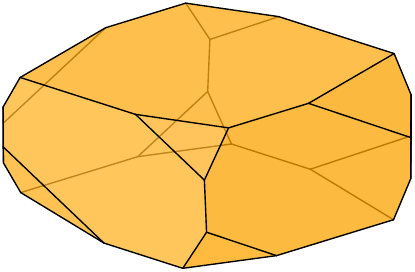}}
&
\href{https://www.inf.fu-berlin.de/inst/ag-ti/software/DiscreteHopfFibration/gallery.html?f=OxCn/2/288cells_12tubes}{\includegraphics[scale=0.8]{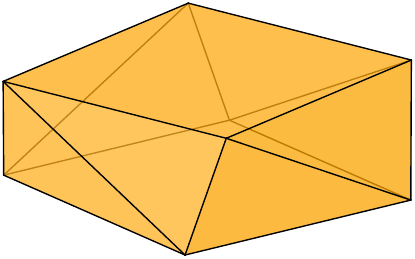}}
\\[1.5mm]
$n = 8$\break
$(\frac{k}{2} + \frac{1}{16})\cdot 2\pi$
&$n = 5, 10$\break
$(\frac{1}{2} + \frac{1}{20})\cdot 2\pi$
&$n = 12$\break
$(\frac{k}{2} + \frac{1}{24})\cdot 2\pi$
\end{tabular}\nobreak

\vfill\nobreak
\noindent\begin{tabular}{
>{\centering\arraybackslash}m{0.33\textwidth}
>{\centering\arraybackslash}m{0.33\textwidth}
>{\centering\arraybackslash}m{0.33\textwidth}}
\href{https://www.inf.fu-berlin.de/inst/ag-ti/software/DiscreteHopfFibration/gallery.html?f=OxCn/2/336cells_12tubes}{\includegraphics[scale=0.8]{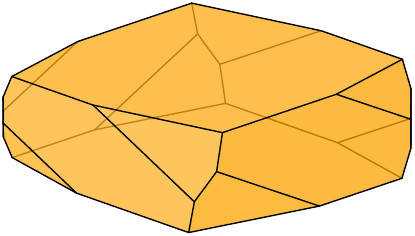}}
&
\href{https://www.inf.fu-berlin.de/inst/ag-ti/software/DiscreteHopfFibration/gallery.html?f=OxCn/2/384cells_12tubes}{\includegraphics[scale=0.8]{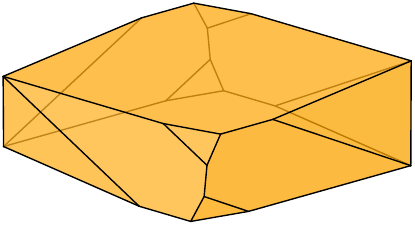}}
&
\href{https://www.inf.fu-berlin.de/inst/ag-ti/software/DiscreteHopfFibration/gallery.html?f=OxCn/2/432cells_12tubes}{\includegraphics[scale=0.8]{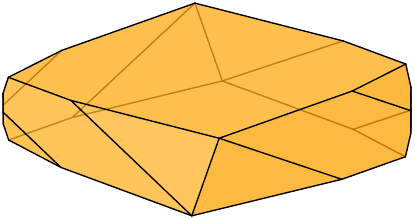}}
\\[1.5mm]
$n = 7, 14$\break
$(\frac{1}{2} + \frac{1}{28})\cdot 2\pi$
&$n = 16$\break
$(\frac{k}{2} + \frac{1}{32})\cdot 2\pi$
&$n = 9, 18$\break
$(\frac{1}{2} + \frac{1}{36})\cdot 2\pi$
\end{tabular}\nobreak

\vfill\nobreak
\noindent\begin{tabular}{
>{\centering\arraybackslash}m{0.33\textwidth}
>{\centering\arraybackslash}m{0.33\textwidth}
>{\centering\arraybackslash}m{0.33\textwidth}}
\href{https://www.inf.fu-berlin.de/inst/ag-ti/software/DiscreteHopfFibration/gallery.html?f=OxCn/2/480cells_12tubes}{\includegraphics[scale=0.8]{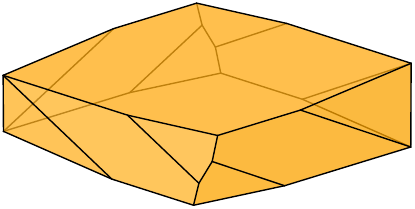}}
&
\href{https://www.inf.fu-berlin.de/inst/ag-ti/software/DiscreteHopfFibration/gallery.html?f=OxCn/2/528cells_12tubes}{\includegraphics[scale=0.8]{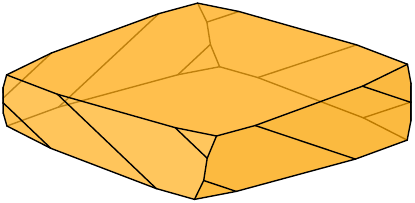}}
&
\href{https://www.inf.fu-berlin.de/inst/ag-ti/software/DiscreteHopfFibration/gallery.html?f=OxCn/2/576cells_12tubes}{\includegraphics[scale=0.8]{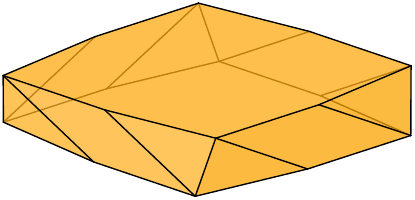}}
\\[1.5mm]
$n = 20$\break
$(\frac{k}{2} + \frac{1}{40})\cdot 2\pi$
&$n = 11, 22$\break
$(\frac{1}{2} + \frac{1}{44})\cdot 2\pi$
&$n = 24$\break
$(\frac{k}{2} + \frac{1}{48})\cdot 2\pi$
\end{tabular}\nobreak

\vfill\nobreak
\noindent\begin{tabular}{
>{\centering\arraybackslash}m{0.33\textwidth}
>{\centering\arraybackslash}m{0.33\textwidth}
>{\centering\arraybackslash}m{0.33\textwidth}}
\href{https://www.inf.fu-berlin.de/inst/ag-ti/software/DiscreteHopfFibration/gallery.html?f=OxCn/2/624cells_12tubes}{\includegraphics[scale=0.8]{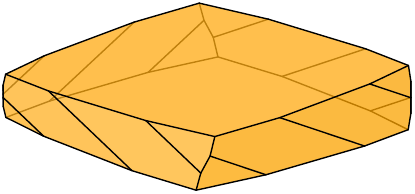}}
&
\href{https://www.inf.fu-berlin.de/inst/ag-ti/software/DiscreteHopfFibration/gallery.html?f=OxCn/2/672cells_12tubes}{\includegraphics[scale=0.8]{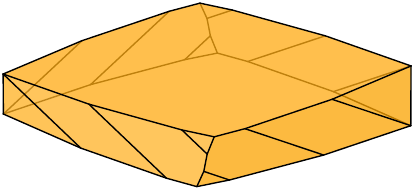}}
&
\href{https://www.inf.fu-berlin.de/inst/ag-ti/software/DiscreteHopfFibration/gallery.html?f=OxCn/2/720cells_12tubes}{\includegraphics[scale=0.8]{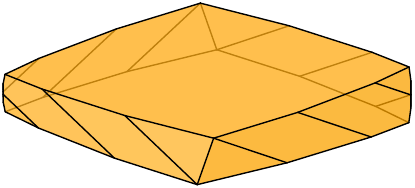}}
\\[1.5mm]
$n = 13, 26$\break
$(\frac{1}{2} + \frac{1}{52})\cdot 2\pi$
&$n = 28$\break
$(\frac{k}{2} + \frac{1}{56})\cdot 2\pi$
&$n = 15, 30$\break
$(\frac{1}{2} + \frac{1}{60})\cdot 2\pi$
\end{tabular}\nobreak

\begin{figure}[H]
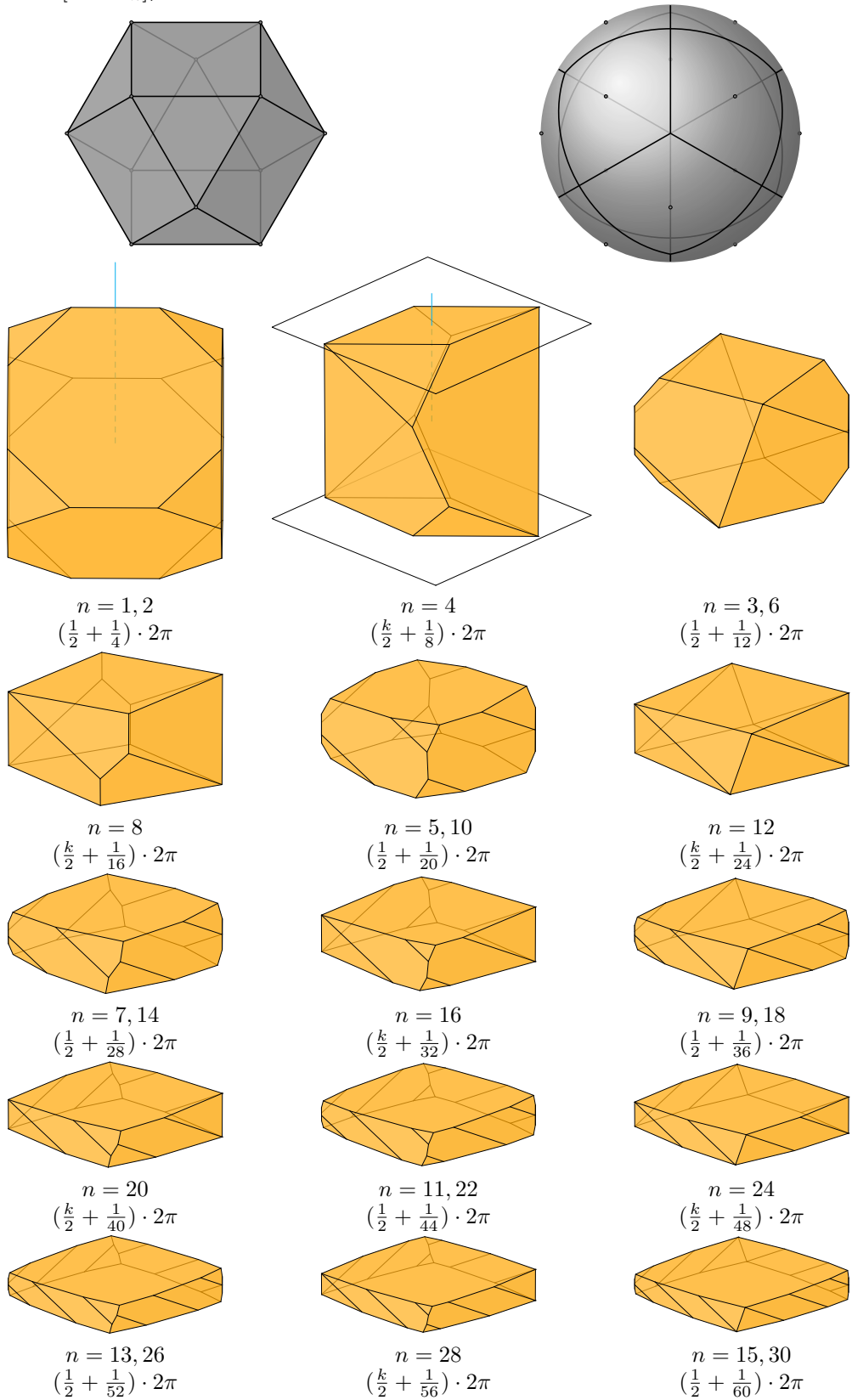

\caption{
$G=\pm[O\times C_n]$,
$G^h={+O}$,
2-fold rotation center
$p=\frac1{\sqrt2}(0, 1, 1)$.
$H=\langle [i_O, 1], [1, e_n] \rangle$.
12 tubes,
each with $\mathrm{lcm}(2n, 4)$ cells.
Alternate group: $\pm[O\times D_{2n}]$.
When $n=1$ or $n=2$,
the cells of a tube are disconnected from each other.
For $n = 4$, we have drawn squares in the planes around the
top and bottom face, to indicate that these faces are horizontal and parallel.
}
\label{fig:OxCn_2fold}
\end{figure}\endgroup
\newpage

\subsection{\texorpdfstring{$\pm\frac12[O\times C_{2n}]$}{+-1/2[OxC2n]}}
\subsubsection{\texorpdfstring{$\pm\frac12[O\times C_{2n}]$}{+-1/2[OxC2n]}, 3-fold rotation center}
\hrule height 0pt
\vskip -1mm  \vfill
\hrule height 0pt \nobreak
\begin{minipage}{0.5\textwidth}
\centering
\includegraphics[scale=1]{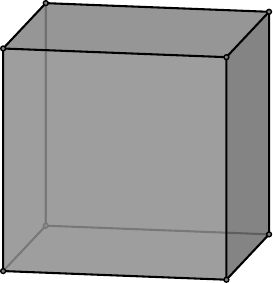}
\end{minipage}%
\begin{minipage}{0.5\textwidth}
\centering
\includegraphics[scale=1]{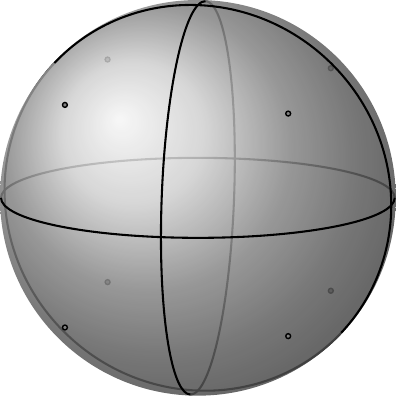}
\end{minipage}
\vskip 5pt plus 0.5fill
\begingroup
\setlength{\tabcolsep}{0pt}
\noindent\begin{tabular}{
>{\centering\arraybackslash}m{0.33\textwidth}
>{\centering\arraybackslash}m{0.33\textwidth}
>{\centering\arraybackslash}m{0.33\textwidth}}
\href{https://www.inf.fu-berlin.de/inst/ag-ti/software/DiscreteHopfFibration/gallery.html?f=hOxC2n/3/48cells_8tubes}{\includegraphics[scale=0.8]{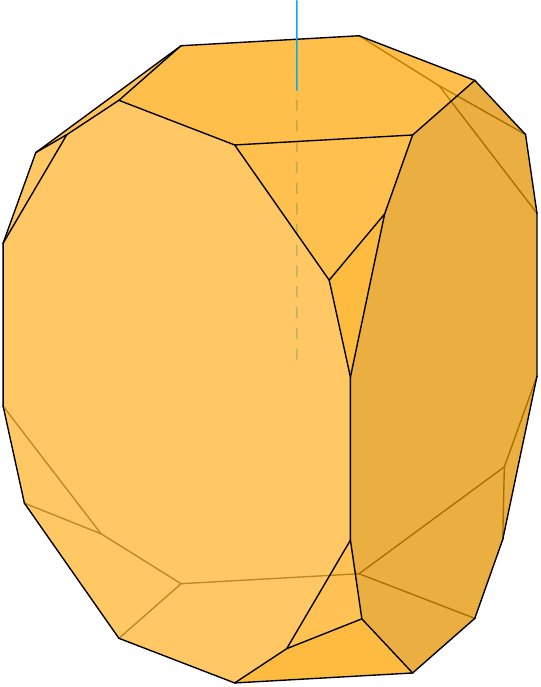}}
&
\href{https://www.inf.fu-berlin.de/inst/ag-ti/software/DiscreteHopfFibration/gallery.html?f=hOxC2n/3/96cells_8tubes}{\includegraphics[scale=0.8]{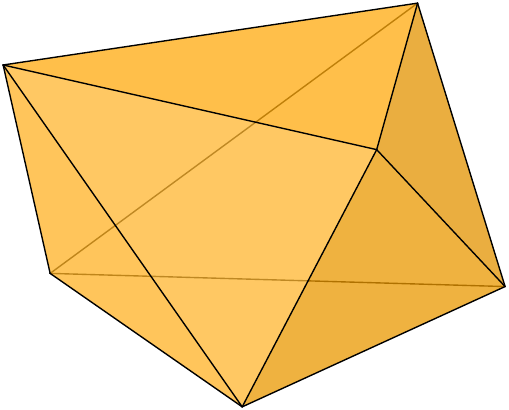}}
&
\href{https://www.inf.fu-berlin.de/inst/ag-ti/software/DiscreteHopfFibration/gallery.html?f=hOxC2n/3/144cells_8tubes}{\includegraphics[scale=0.8]{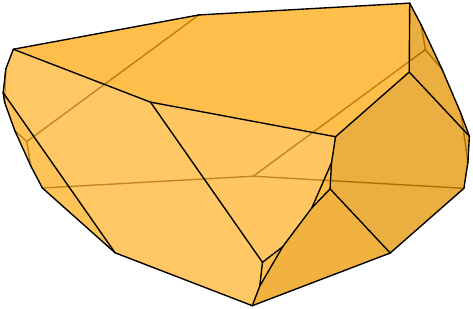}}
\\[1.5mm]
$n = 1, 3$\break
$(\frac{2}{3} + \frac{1}{6})\cdot 2\pi$
&$n = 2, 6$\break
$(\frac{1}{3} + \frac{1}{12})\cdot 2\pi$
&$n = 9$\break
$(\frac{k}{3} + \frac{1}{18})\cdot 2\pi$
\end{tabular}\nobreak

\vfill\nobreak
\noindent\begin{tabular}{
>{\centering\arraybackslash}m{0.33\textwidth}
>{\centering\arraybackslash}m{0.33\textwidth}
>{\centering\arraybackslash}m{0.33\textwidth}}
\href{https://www.inf.fu-berlin.de/inst/ag-ti/software/DiscreteHopfFibration/gallery.html?f=hOxC2n/3/192cells_8tubes}{\includegraphics[scale=0.8]{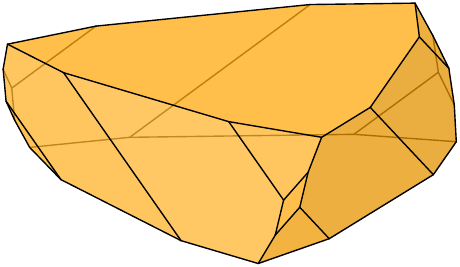}}
&
\href{https://www.inf.fu-berlin.de/inst/ag-ti/software/DiscreteHopfFibration/gallery.html?f=hOxC2n/3/240cells_8tubes}{\includegraphics[scale=0.8]{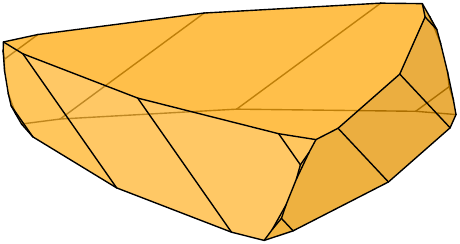}}
&
\href{https://www.inf.fu-berlin.de/inst/ag-ti/software/DiscreteHopfFibration/gallery.html?f=hOxC2n/3/288cells_8tubes}{\includegraphics[scale=0.8]{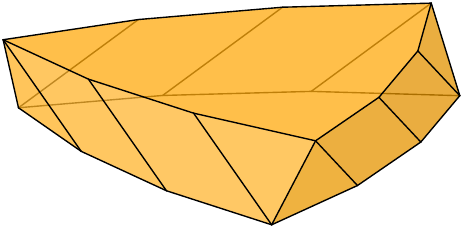}}
\\[1.5mm]
$n = 4, 12$\break
$(\frac{2}{3} + \frac{1}{24})\cdot 2\pi$
&$n = 5, 15$\break
$(\frac{1}{3} + \frac{1}{30})\cdot 2\pi$
&$n = 18$\break
$(\frac{k}{3} + \frac{1}{36})\cdot 2\pi$
\end{tabular}\nobreak

\vfill\nobreak
\noindent\begin{tabular}{
>{\centering\arraybackslash}m{0.33\textwidth}
>{\centering\arraybackslash}m{0.33\textwidth}
>{\centering\arraybackslash}m{0.33\textwidth}}
\href{https://www.inf.fu-berlin.de/inst/ag-ti/software/DiscreteHopfFibration/gallery.html?f=hOxC2n/3/336cells_8tubes}{\includegraphics[scale=0.8]{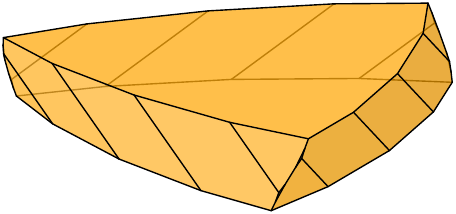}}
&
\href{https://www.inf.fu-berlin.de/inst/ag-ti/software/DiscreteHopfFibration/gallery.html?f=hOxC2n/3/384cells_8tubes}{\includegraphics[scale=0.8]{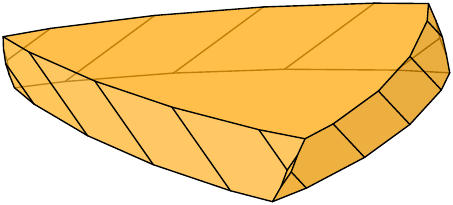}}
&
\href{https://www.inf.fu-berlin.de/inst/ag-ti/software/DiscreteHopfFibration/gallery.html?f=hOxC2n/3/432cells_8tubes}{\includegraphics[scale=0.8]{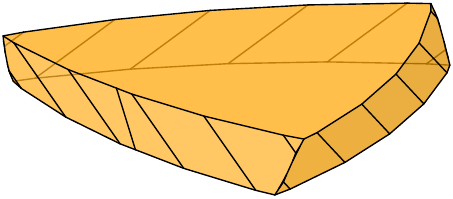}}
\\[1.5mm]
$n = 7, 21$\break
$(\frac{2}{3} + \frac{1}{42})\cdot 2\pi$
&$n = 8, 24$\break
$(\frac{1}{3} + \frac{1}{48})\cdot 2\pi$
&$n = 27$\break
$(\frac{k}{3} + \frac{1}{54})\cdot 2\pi$
\end{tabular}\nobreak

\vfill\nobreak
\noindent\begin{tabular}{
>{\centering\arraybackslash}m{0.33\textwidth}
>{\centering\arraybackslash}m{0.33\textwidth}
>{\centering\arraybackslash}m{0.33\textwidth}}
\href{https://www.inf.fu-berlin.de/inst/ag-ti/software/DiscreteHopfFibration/gallery.html?f=hOxC2n/3/480cells_8tubes}{\includegraphics[scale=0.8]{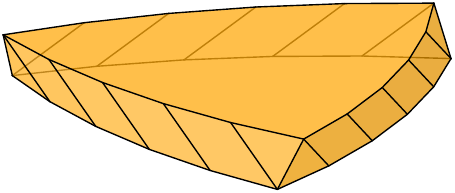}}
&
\href{https://www.inf.fu-berlin.de/inst/ag-ti/software/DiscreteHopfFibration/gallery.html?f=hOxC2n/3/528cells_8tubes}{\includegraphics[scale=0.8]{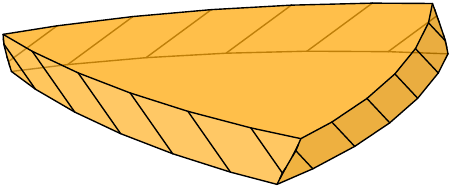}}
&
\href{https://www.inf.fu-berlin.de/inst/ag-ti/software/DiscreteHopfFibration/gallery.html?f=hOxC2n/3/576cells_8tubes}{\includegraphics[scale=0.8]{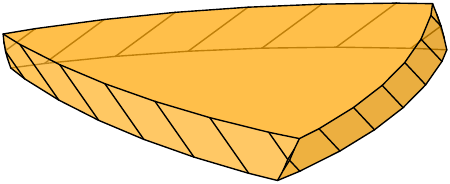}}
\\[1.5mm]
$n = 10, 30$\break
$(\frac{2}{3} + \frac{1}{60})\cdot 2\pi$
&$n = 11, 33$\break
$(\frac{1}{3} + \frac{1}{66})\cdot 2\pi$
&$n = 36$\break
$(\frac{k}{3} + \frac{1}{72})\cdot 2\pi$
\end{tabular}\nobreak

\vfill\nobreak
\noindent\begin{tabular}{
>{\centering\arraybackslash}m{0.33\textwidth}
>{\centering\arraybackslash}m{0.33\textwidth}
>{\centering\arraybackslash}m{0.33\textwidth}}
\href{https://www.inf.fu-berlin.de/inst/ag-ti/software/DiscreteHopfFibration/gallery.html?f=hOxC2n/3/624cells_8tubes}{\includegraphics[scale=0.8]{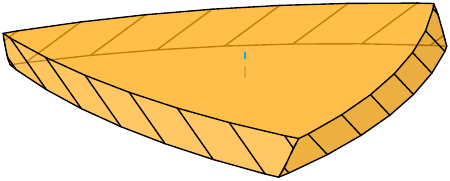}}
&
\href{https://www.inf.fu-berlin.de/inst/ag-ti/software/DiscreteHopfFibration/gallery.html?f=hOxC2n/3/672cells_8tubes}{\includegraphics[scale=0.8]{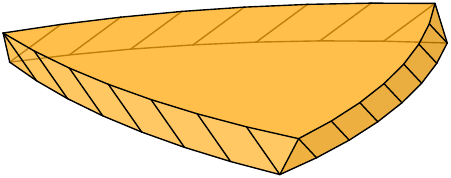}}
&
\href{https://www.inf.fu-berlin.de/inst/ag-ti/software/DiscreteHopfFibration/gallery.html?f=hOxC2n/3/720cells_8tubes}{\includegraphics[scale=0.8]{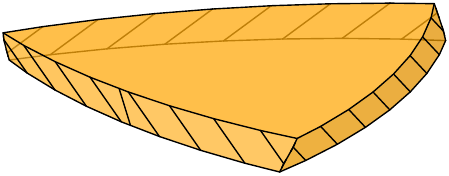}}
\\[1.5mm]
$n = 13, 39$\break
$(\frac{2}{3} + \frac{1}{78})\cdot 2\pi$
&$n = 14, 42$\break
$(\frac{1}{3} + \frac{1}{84})\cdot 2\pi$
&$n = 45$\break
$(\frac{k}{3} + \frac{1}{90})\cdot 2\pi$
\end{tabular}\nobreak

\begin{figure}[H]
\caption{
$G=\pm\frac12[O\times C_{2n}]$,
$G^h={+O}$,
3-fold rotation center
$p=\frac1{\sqrt3}(-1, -1, -1)$.
$H=\langle [-\omega, 1], [1, e_n] \rangle$.
8 tubes,
each with $\mathrm{lcm}(2n, 6)$ cells.
Alternate group: $\pm\frac12[O\times \overline{D}_{4n}]$.
}
\label{fig:hOxC2n_3fold}
\end{figure}\endgroup
\newpage

\newpage
\subsubsection{\texorpdfstring{$\pm\frac12[O\times C_{2n}]$}{+-1/2[OxC2n]}, 2-fold rotation center}
\hrule height 0pt
\vskip -3mm \vfill
\hrule height 0pt \nobreak
\begin{minipage}{0.5\textwidth}
\centering
\includegraphics[scale=1]{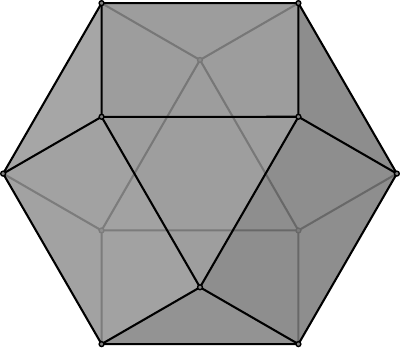}
\end{minipage}%
\begin{minipage}{0.5\textwidth}
\centering
\includegraphics[scale=1]{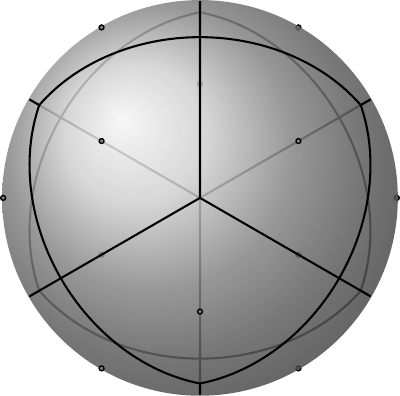}
\end{minipage}
\vskip -1pt plus 0.5fill
\begingroup
\setlength{\tabcolsep}{0pt}
\noindent\begin{tabular}{
>{\centering\arraybackslash}m{0.33\textwidth}
>{\centering\arraybackslash}m{0.33\textwidth}
>{\centering\arraybackslash}m{0.33\textwidth}}
\href{https://www.inf.fu-berlin.de/inst/ag-ti/software/DiscreteHopfFibration/gallery.html?f=hOxC2n/2/24cells_12tubes}{\includegraphics[scale=0.5]{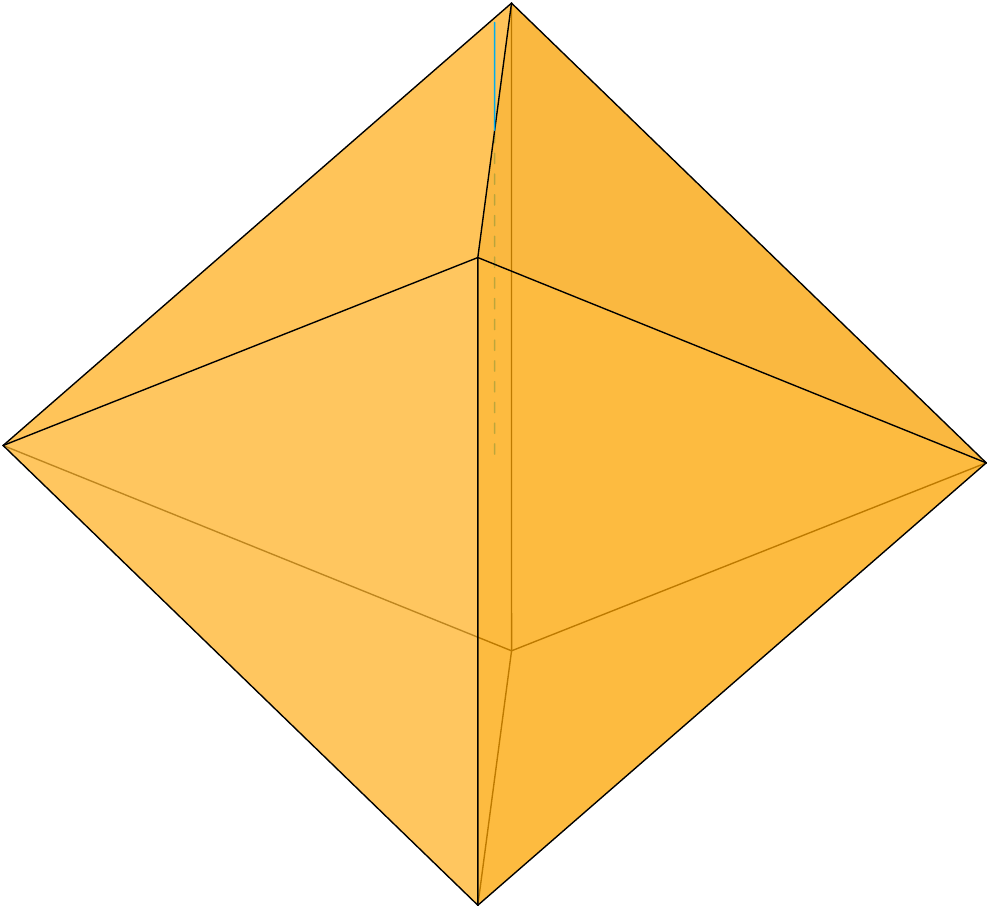}}
&
\href{https://www.inf.fu-berlin.de/inst/ag-ti/software/DiscreteHopfFibration/gallery.html?f=hOxC2n/2/72cells_12tubes}{\includegraphics[scale=0.8]{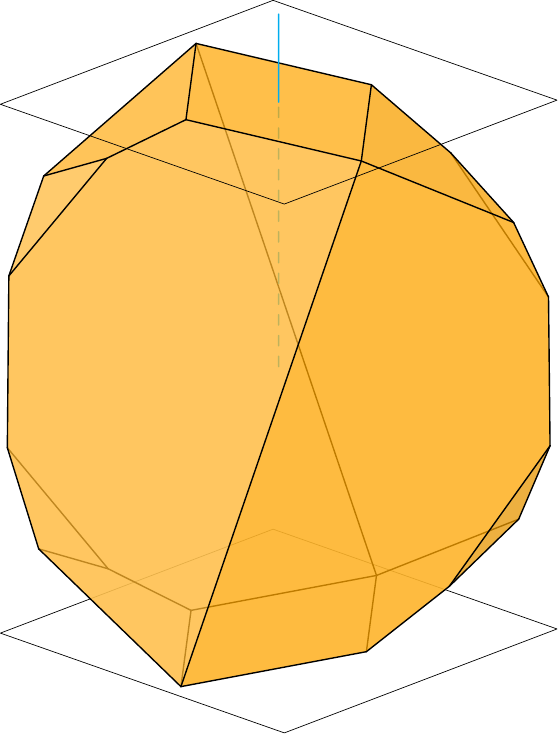}}
&
\href{https://www.inf.fu-berlin.de/inst/ag-ti/software/DiscreteHopfFibration/gallery.html?f=hOxC2n/2/96cells_12tubes}{\includegraphics[scale=0.8]{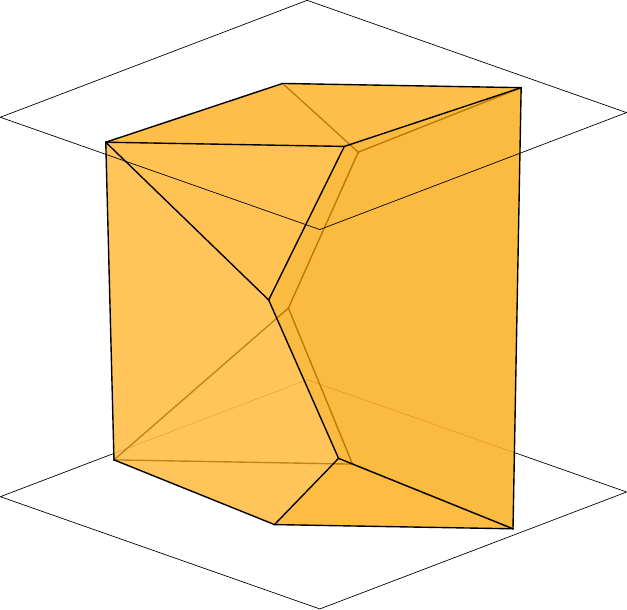}}
\\[1.5mm]
$n = 1$\break
$(\frac{k}{2} + \frac{1}{2})\cdot 2\pi$
&$n = 3$\break
$(\frac{k}{2} + \frac{1}{6})\cdot 2\pi$
&$n = 2$\break
$(\frac{1}{2} + \frac{1}{8})\cdot 2\pi$
\end{tabular}\nobreak

\vfill\nobreak
\noindent\begin{tabular}{
>{\centering\arraybackslash}m{0.33\textwidth}
>{\centering\arraybackslash}m{0.33\textwidth}
>{\centering\arraybackslash}m{0.33\textwidth}}
\href{https://www.inf.fu-berlin.de/inst/ag-ti/software/DiscreteHopfFibration/gallery.html?f=hOxC2n/2/120cells_12tubes}{\includegraphics[scale=0.8]{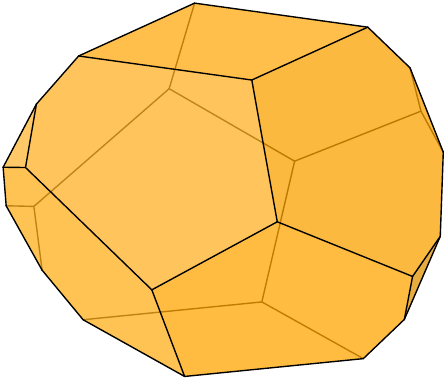}}
&
\href{https://www.inf.fu-berlin.de/inst/ag-ti/software/DiscreteHopfFibration/gallery.html?f=hOxC2n/2/168cells_12tubes}{\includegraphics[scale=0.8]{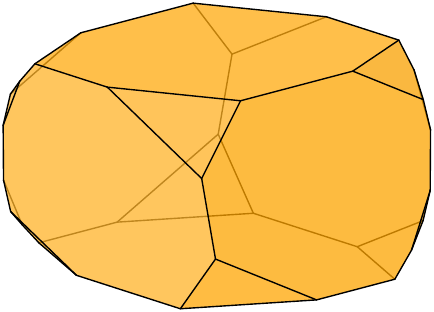}}
&
\href{https://www.inf.fu-berlin.de/inst/ag-ti/software/DiscreteHopfFibration/gallery.html?f=hOxC2n/2/192cells_12tubes}{\includegraphics[scale=0.8]{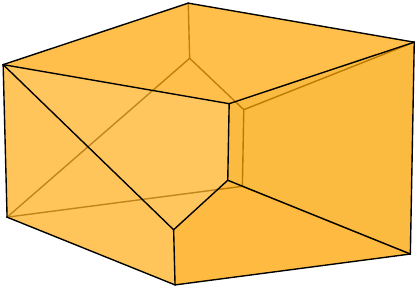}}
\\[1.5mm]
$n = 5$\break
$(\frac{k}{2} + \frac{1}{10})\cdot 2\pi$
&$n = 7$\break
$(\frac{k}{2} + \frac{1}{14})\cdot 2\pi$
&$n = 4$\break
$(\frac{1}{2} + \frac{1}{16})\cdot 2\pi$
\end{tabular}\nobreak

\vfill\nobreak
\noindent\begin{tabular}{
>{\centering\arraybackslash}m{0.33\textwidth}
>{\centering\arraybackslash}m{0.33\textwidth}
>{\centering\arraybackslash}m{0.33\textwidth}}
\href{https://www.inf.fu-berlin.de/inst/ag-ti/software/DiscreteHopfFibration/gallery.html?f=hOxC2n/2/216cells_12tubes}{\includegraphics[scale=0.8]{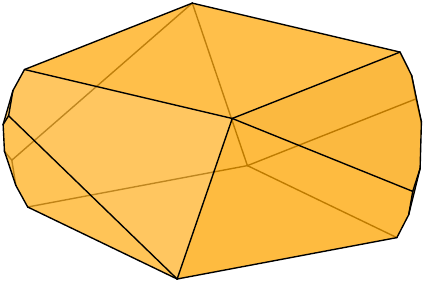}}
&
\href{https://www.inf.fu-berlin.de/inst/ag-ti/software/DiscreteHopfFibration/gallery.html?f=hOxC2n/2/264cells_12tubes}{\includegraphics[scale=0.8]{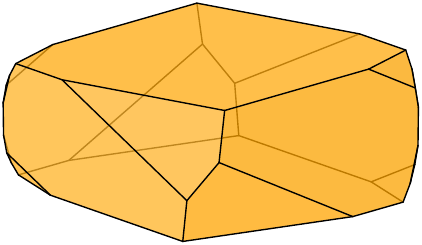}}
&
\href{https://www.inf.fu-berlin.de/inst/ag-ti/software/DiscreteHopfFibration/gallery.html?f=hOxC2n/2/288cells_12tubes}{\includegraphics[scale=0.8]{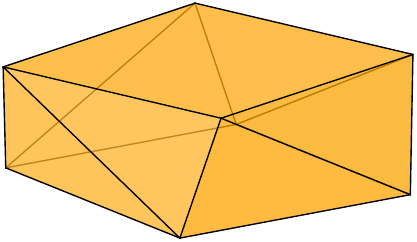}}
\\[1.5mm]
$n = 9$\break
$(\frac{k}{2} + \frac{1}{18})\cdot 2\pi$
&$n = 11$\break
$(\frac{k}{2} + \frac{1}{22})\cdot 2\pi$
&$n = 6$\break
$(\frac{1}{2} + \frac{1}{24})\cdot 2\pi$
\end{tabular}\nobreak

\vfill\nobreak
\noindent\begin{tabular}{
>{\centering\arraybackslash}m{0.33\textwidth}
>{\centering\arraybackslash}m{0.33\textwidth}
>{\centering\arraybackslash}m{0.33\textwidth}}
\href{https://www.inf.fu-berlin.de/inst/ag-ti/software/DiscreteHopfFibration/gallery.html?f=hOxC2n/2/312cells_12tubes}{\includegraphics[scale=0.8]{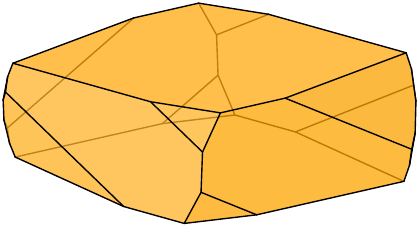}}
&
\href{https://www.inf.fu-berlin.de/inst/ag-ti/software/DiscreteHopfFibration/gallery.html?f=hOxC2n/2/360cells_12tubes}{\includegraphics[scale=0.8]{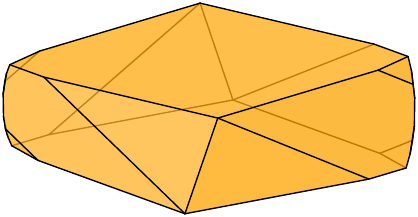}}
&
\href{https://www.inf.fu-berlin.de/inst/ag-ti/software/DiscreteHopfFibration/gallery.html?f=hOxC2n/2/384cells_12tubes}{\includegraphics[scale=0.8]{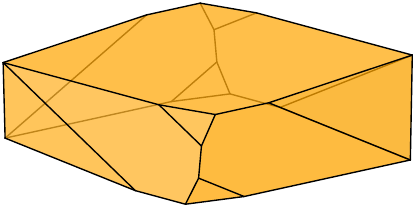}}
\\[1.5mm]
$n = 13$\break
$(\frac{k}{2} + \frac{1}{26})\cdot 2\pi$
&$n = 15$\break
$(\frac{k}{2} + \frac{1}{30})\cdot 2\pi$
&$n = 8$\break
$(\frac{1}{2} + \frac{1}{32})\cdot 2\pi$
\end{tabular}\nobreak

\begin{figure}[H]
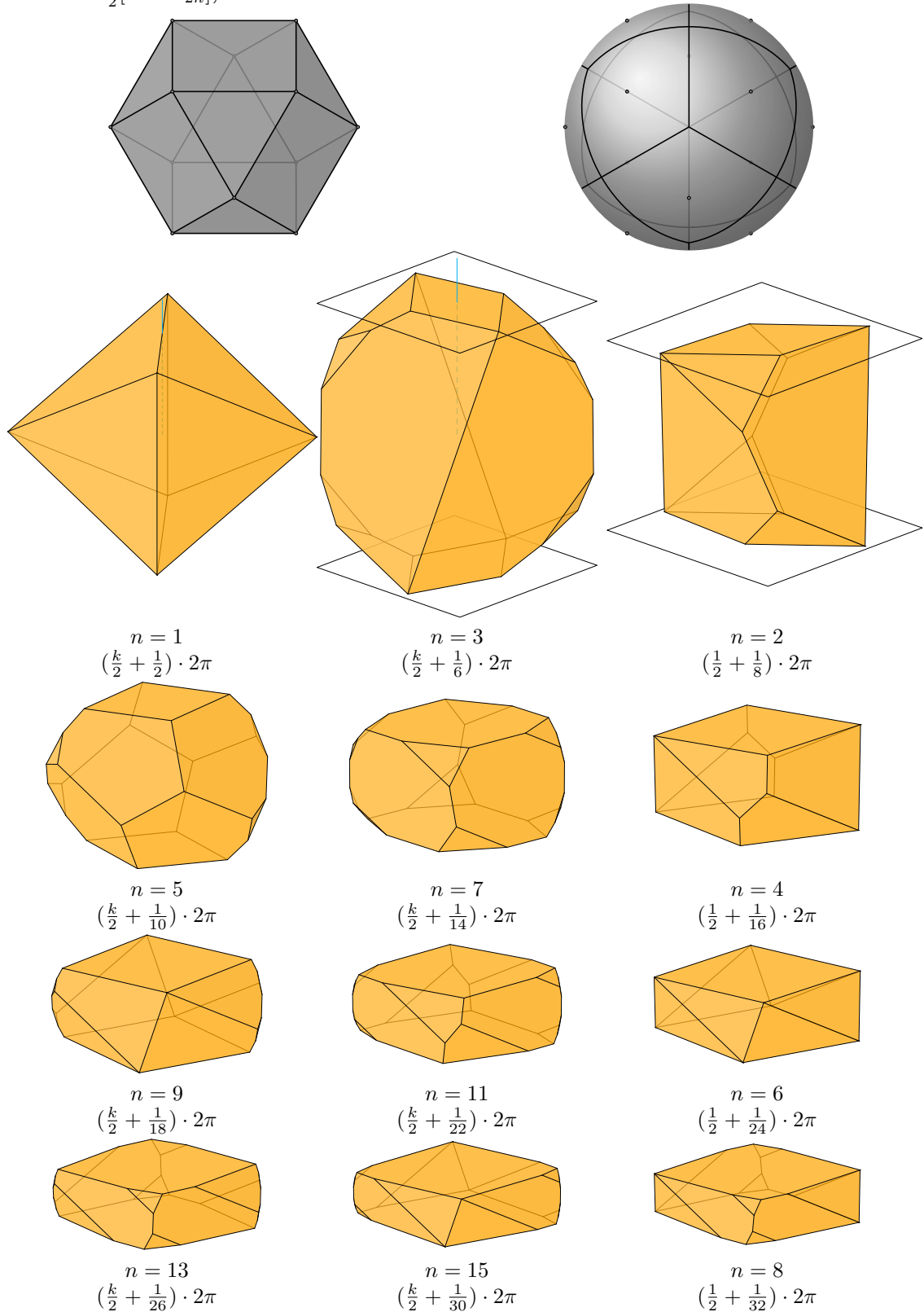

\caption{
$G=\pm\frac12[O\times C_{2n}]$,
$G^h={+O}$,
2-fold rotation center
$p=\frac1{\sqrt2}(0, 1, 1)$.
The $G^h$-orbit polytope is a cuboctahedron.
The corresponding Voronoi diagram on the 2-sphere
has the structure of a rhombic dodecahedron.
$H=\langle [i_O, e_{2n}], [1, e_n] \rangle$.
12 tubes,
each with $\frac{4n}{\gcd(n-1, 2)}$ cells.
Alternate group: $\pm\frac12[O\times \overline{D}_{4n}]$.
When $n = 1$,
the cells of a tube are disconnected from each other.
For $n=2$ and $n=3$,
we have drawn squares in the planes around the top and bottom face,
to indicate that these faces are horizontal and parallel.
}
\label{fig:hOxC2n_2fold}
\end{figure}\endgroup
\newpage

\subsection{\texorpdfstring{$\pm[T\times C_n]$}{+-[TxCn]}}
\subsubsection{\texorpdfstring{$\pm[T\times C_n]$}{+-[TxCn]}, 3-fold rotation center}
\begin{minipage}{0.5\textwidth}
\centering
\includegraphics[scale=1]{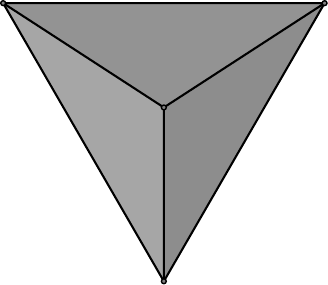}
\end{minipage}%
\begin{minipage}{0.5\textwidth}
\centering
\includegraphics[scale=1]{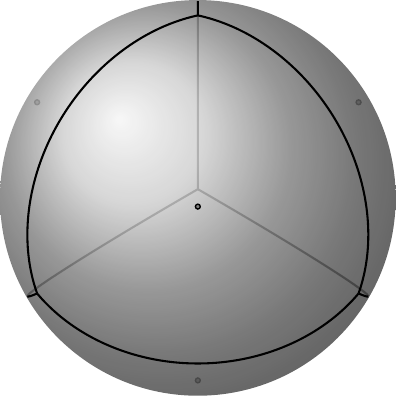}
\end{minipage}
\vskip 5pt plus 0.5fill
\begingroup
\setlength{\tabcolsep}{0pt}
\noindent\begin{tabular}{
>{\centering\arraybackslash}m{0.33\textwidth}
>{\centering\arraybackslash}m{0.33\textwidth}
>{\centering\arraybackslash}m{0.33\textwidth}}
\href{https://www.inf.fu-berlin.de/inst/ag-ti/software/DiscreteHopfFibration/gallery.html?f=TxCn/3/24cells_4tubes}{\includegraphics[scale=0.6]{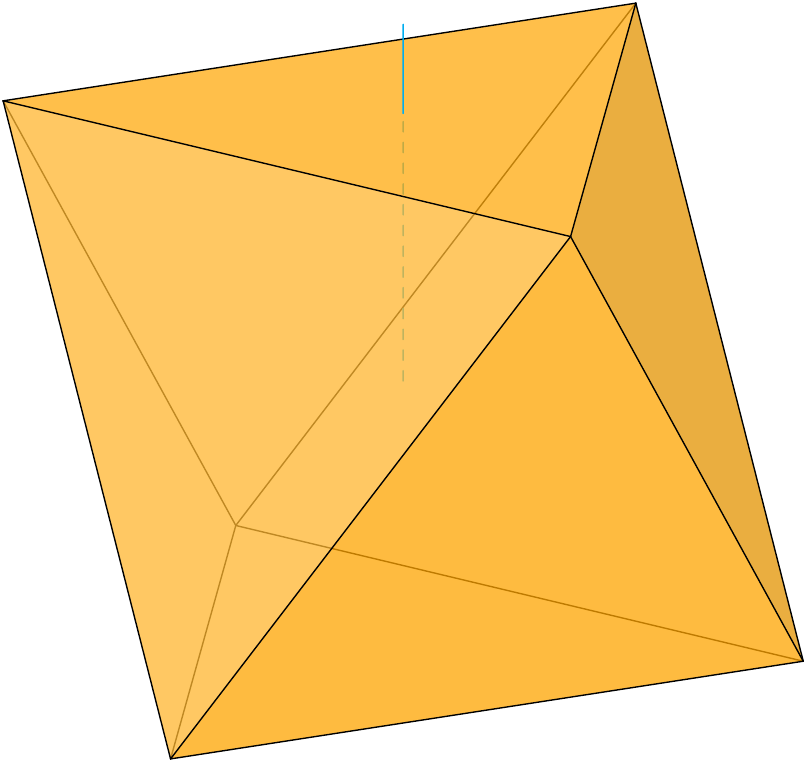}}
&
\href{https://www.inf.fu-berlin.de/inst/ag-ti/software/DiscreteHopfFibration/gallery.html?f=TxCn/3/48cells_4tubes}{\includegraphics[scale=0.6]{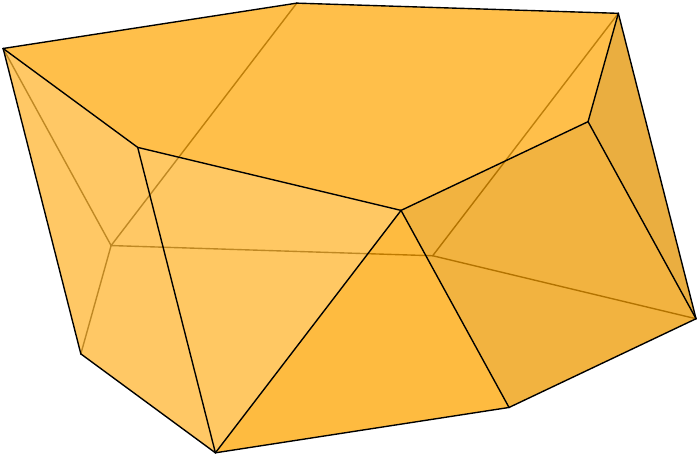}}
&
\href{https://www.inf.fu-berlin.de/inst/ag-ti/software/DiscreteHopfFibration/gallery.html?f=TxCn/3/72cells_4tubes}{\includegraphics[scale=0.6]{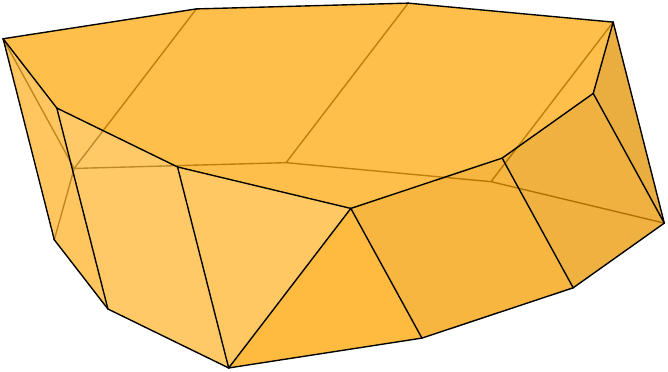}}
\\[1.5mm]
$n = 1, 3$\break
$(\frac{2}{3} + \frac{1}{6})\cdot 2\pi$
&$n = 2, 6$\break
$(\frac{1}{3} + \frac{1}{12})\cdot 2\pi$
&$n = 9$\break
$(\frac{k}{3} + \frac{1}{18})\cdot 2\pi$
\end{tabular}\nobreak

\vfill\nobreak
\noindent\begin{tabular}{
>{\centering\arraybackslash}m{0.33\textwidth}
>{\centering\arraybackslash}m{0.33\textwidth}
>{\centering\arraybackslash}m{0.33\textwidth}}
\href{https://www.inf.fu-berlin.de/inst/ag-ti/software/DiscreteHopfFibration/gallery.html?f=TxCn/3/96cells_4tubes}{\includegraphics[scale=0.6]{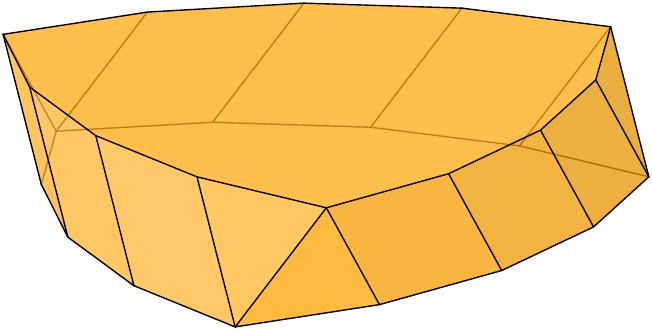}}
&
\href{https://www.inf.fu-berlin.de/inst/ag-ti/software/DiscreteHopfFibration/gallery.html?f=TxCn/3/120cells_4tubes}{\includegraphics[scale=0.6]{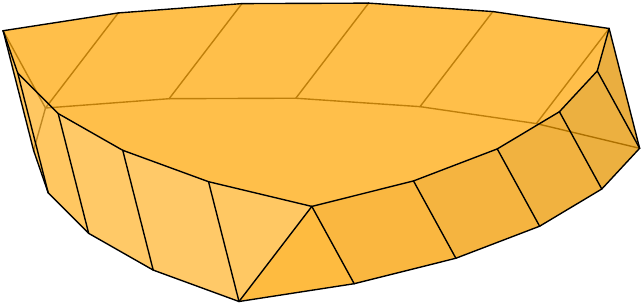}}
&
\href{https://www.inf.fu-berlin.de/inst/ag-ti/software/DiscreteHopfFibration/gallery.html?f=TxCn/3/144cells_4tubes}{\includegraphics[scale=0.6]{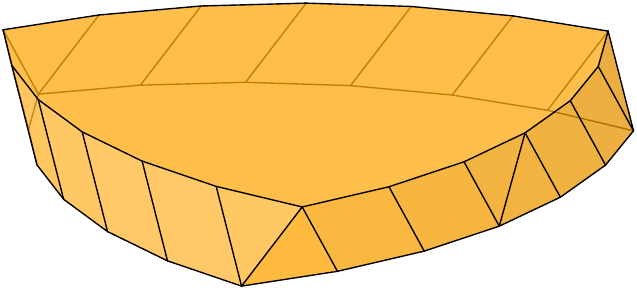}}
\\[1.5mm]
$n = 4, 12$\break
$(\frac{2}{3} + \frac{1}{24})\cdot 2\pi$
&$n = 5, 15$\break
$(\frac{1}{3} + \frac{1}{30})\cdot 2\pi$
&$n = 18$\break
$(\frac{k}{3} + \frac{1}{36})\cdot 2\pi$
\end{tabular}\nobreak

\vfill\nobreak
\noindent\begin{tabular}{
>{\centering\arraybackslash}m{0.33\textwidth}
>{\centering\arraybackslash}m{0.33\textwidth}
>{\centering\arraybackslash}m{0.33\textwidth}}
\href{https://www.inf.fu-berlin.de/inst/ag-ti/software/DiscreteHopfFibration/gallery.html?f=TxCn/3/168cells_4tubes}{\includegraphics[scale=0.6]{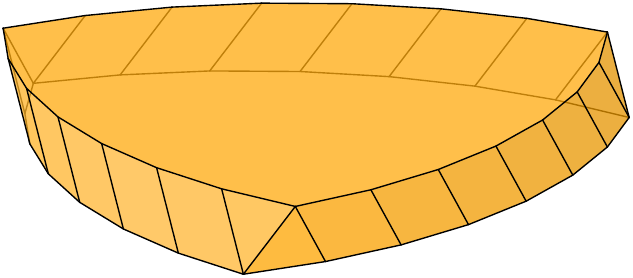}}
&
\href{https://www.inf.fu-berlin.de/inst/ag-ti/software/DiscreteHopfFibration/gallery.html?f=TxCn/3/192cells_4tubes}{\includegraphics[scale=0.6]{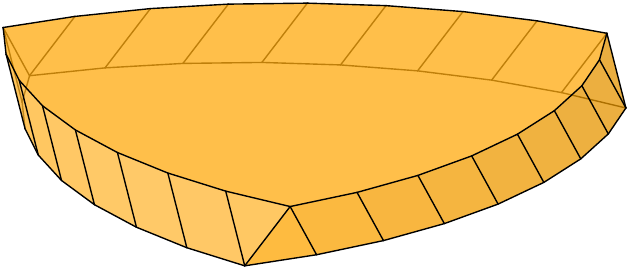}}
&
\href{https://www.inf.fu-berlin.de/inst/ag-ti/software/DiscreteHopfFibration/gallery.html?f=TxCn/3/216cells_4tubes}{\includegraphics[scale=0.6]{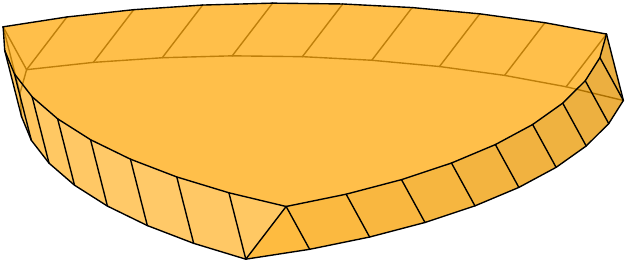}}
\\[1.5mm]
$n = 7, 21$\break
$(\frac{2}{3} + \frac{1}{42})\cdot 2\pi$
&$n = 8, 24$\break
$(\frac{1}{3} + \frac{1}{48})\cdot 2\pi$
&$n = 27$\break
$(\frac{k}{3} + \frac{1}{54})\cdot 2\pi$
\end{tabular}\nobreak

\vfill\nobreak
\noindent\begin{tabular}{
>{\centering\arraybackslash}m{0.33\textwidth}
>{\centering\arraybackslash}m{0.33\textwidth}
>{\centering\arraybackslash}m{0.33\textwidth}}
\href{https://www.inf.fu-berlin.de/inst/ag-ti/software/DiscreteHopfFibration/gallery.html?f=TxCn/3/240cells_4tubes}{\includegraphics[scale=0.6]{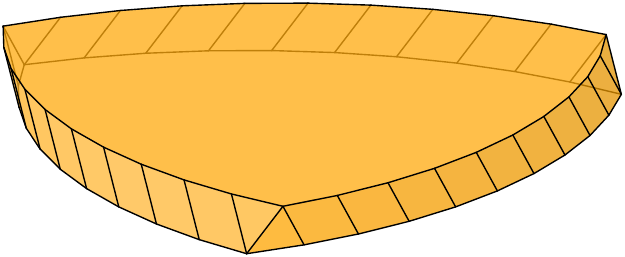}}
&
\href{https://www.inf.fu-berlin.de/inst/ag-ti/software/DiscreteHopfFibration/gallery.html?f=TxCn/3/264cells_4tubes}{\includegraphics[scale=0.6]{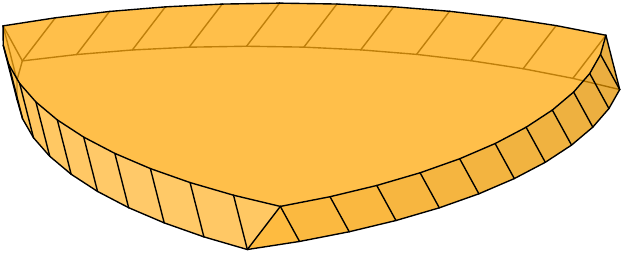}}
&
\href{https://www.inf.fu-berlin.de/inst/ag-ti/software/DiscreteHopfFibration/gallery.html?f=TxCn/3/288cells_4tubes}{\includegraphics[scale=0.6]{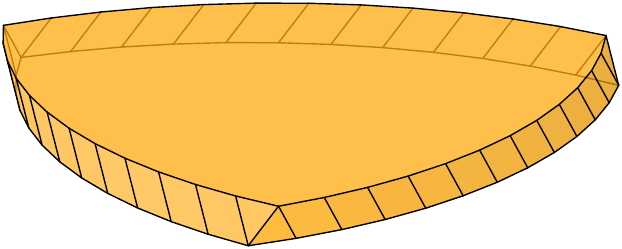}}
\\[1.5mm]
$n = 10, 30$\break
$(\frac{2}{3} + \frac{1}{60})\cdot 2\pi$
&$n = 11, 33$\break
$(\frac{1}{3} + \frac{1}{66})\cdot 2\pi$
&$n = 36$\break
$(\frac{k}{3} + \frac{1}{72})\cdot 2\pi$
\end{tabular}\nobreak

\vfill\nobreak
\noindent\begin{tabular}{
>{\centering\arraybackslash}m{0.33\textwidth}
>{\centering\arraybackslash}m{0.33\textwidth}
>{\centering\arraybackslash}m{0.33\textwidth}}
\href{https://www.inf.fu-berlin.de/inst/ag-ti/software/DiscreteHopfFibration/gallery.html?f=TxCn/3/312cells_4tubes}{\includegraphics[scale=0.6]{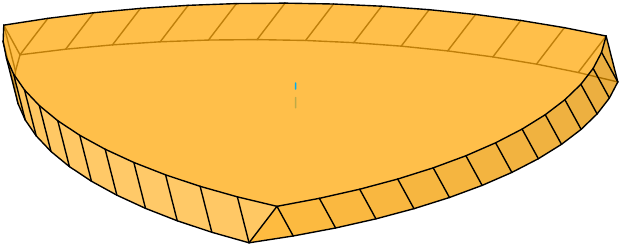}}
&
\href{https://www.inf.fu-berlin.de/inst/ag-ti/software/DiscreteHopfFibration/gallery.html?f=TxCn/3/336cells_4tubes}{\includegraphics[scale=0.6]{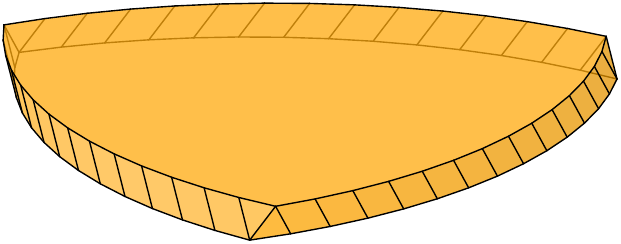}}
&
\href{https://www.inf.fu-berlin.de/inst/ag-ti/software/DiscreteHopfFibration/gallery.html?f=TxCn/3/360cells_4tubes}{\includegraphics[scale=0.6]{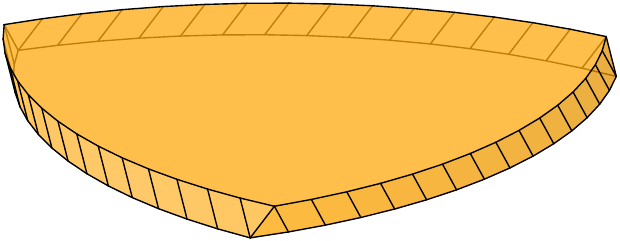}}
\\[1.5mm]
$n = 13, 39$\break
$(\frac{2}{3} + \frac{1}{78})\cdot 2\pi$
&$n = 14, 42$\break
$(\frac{1}{3} + \frac{1}{84})\cdot 2\pi$
&$n = 45$\break
$(\frac{k}{3} + \frac{1}{90})\cdot 2\pi$
\end{tabular}\nobreak

\begin{figure}[H]
\caption{
$G=\pm[T\times C_n]$,
$G^h={+T}$,
3-fold (type I) rotation center
$p=\frac1{\sqrt3}(-1, -1, -1)$.
$H=\langle [-\omega, 1], [1, e_n] \rangle$.
4 tubes,
each with $\mathrm{lcm}(2n, 6)$ cells.
Alternate group: $\pm\frac12[O\times D_{2n}]$.
}
\label{fig:TxCn_3fold}
\end{figure}\endgroup
\newpage

\subsubsection{\texorpdfstring{$\pm[T\times C_n]$}{+-[TxCn]}, 2-fold
  rotation center}
\label{T-2fold}
\begin{minipage}{0.5\textwidth}
\centering
\includegraphics[scale=1]{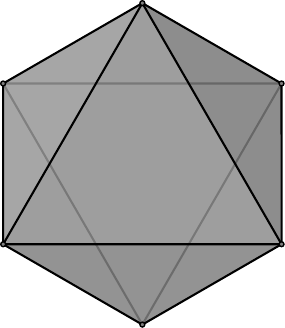}
\end{minipage}%
\begin{minipage}{0.5\textwidth}
\centering
\includegraphics[scale=1]{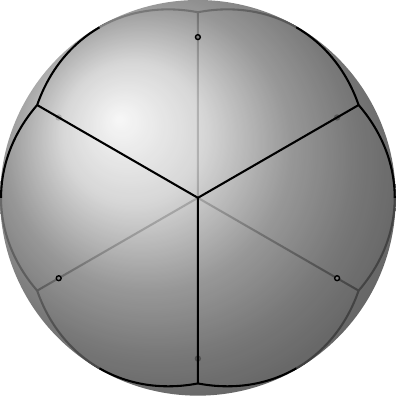}
\end{minipage}
\vskip 5pt plus 0.5fill
\begingroup
\setlength{\tabcolsep}{0pt}
\noindent\begin{tabular}{
>{\centering\arraybackslash}m{0.33\textwidth}
>{\centering\arraybackslash}m{0.33\textwidth}
>{\centering\arraybackslash}m{0.33\textwidth}}
\href{https://www.inf.fu-berlin.de/inst/ag-ti/software/DiscreteHopfFibration/gallery.html?f=TxCn/2/24cells_6tubes}{\includegraphics[scale=0.4]{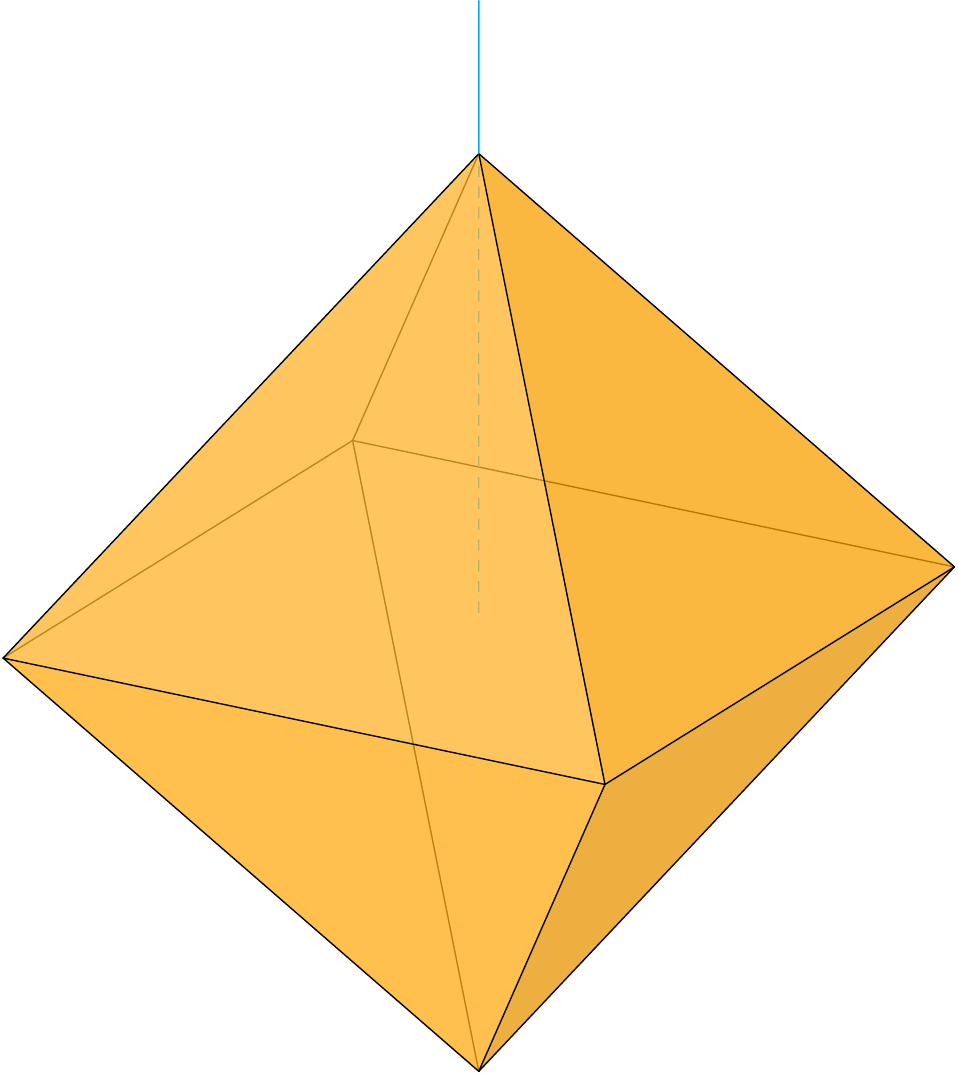}}
&
\href{https://www.inf.fu-berlin.de/inst/ag-ti/software/DiscreteHopfFibration/gallery.html?f=TxCn/2/48cells_6tubes}{\includegraphics[scale=0.6]{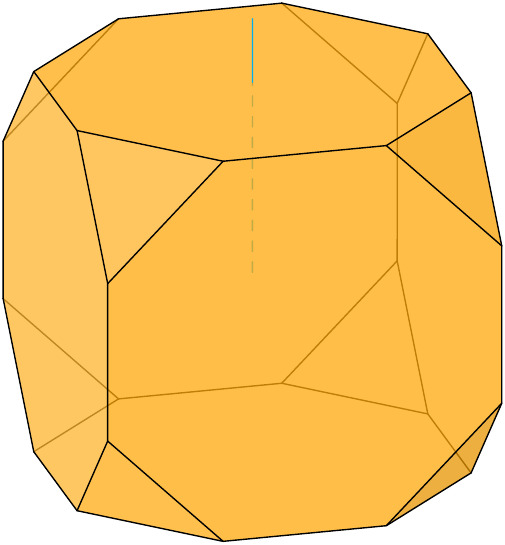}}
&
\href{https://www.inf.fu-berlin.de/inst/ag-ti/software/DiscreteHopfFibration/gallery.html?f=TxCn/2/72cells_6tubes}{\includegraphics[scale=0.6]{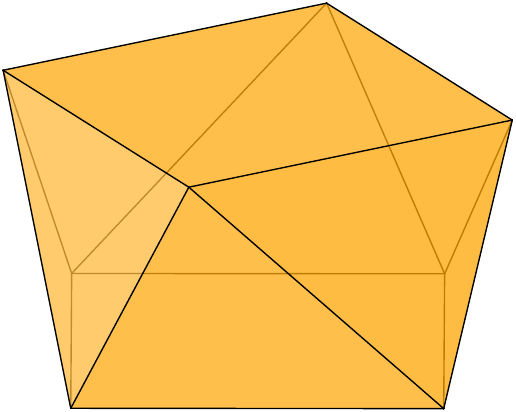}}
\\[1.5mm]
$n = 1, 2$\break
$(\frac{1}{2} + \frac{1}{4})\cdot 2\pi$
&$n = 4$\break
$(\frac{k}{2} + \frac{1}{8})\cdot 2\pi$
&$n = 3, 6$\break
$(\frac{1}{2} + \frac{1}{12})\cdot 2\pi$
\end{tabular}\nobreak

\vfill\nobreak
\noindent\begin{tabular}{
>{\centering\arraybackslash}m{0.33\textwidth}
>{\centering\arraybackslash}m{0.33\textwidth}
>{\centering\arraybackslash}m{0.33\textwidth}}
\href{https://www.inf.fu-berlin.de/inst/ag-ti/software/DiscreteHopfFibration/gallery.html?f=TxCn/2/96cells_6tubes}{\includegraphics[scale=0.6]{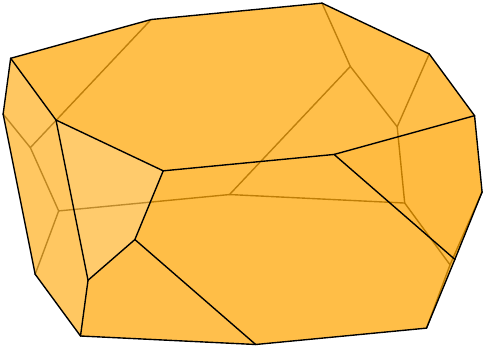}}
&
\href{https://www.inf.fu-berlin.de/inst/ag-ti/software/DiscreteHopfFibration/gallery.html?f=TxCn/2/120cells_6tubes}{\includegraphics[scale=0.6]{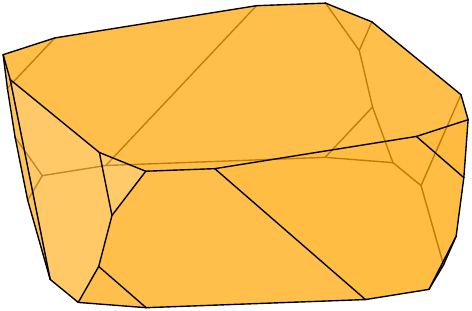}}
&
\href{https://www.inf.fu-berlin.de/inst/ag-ti/software/DiscreteHopfFibration/gallery.html?f=TxCn/2/144cells_6tubes}{\includegraphics[scale=0.6]{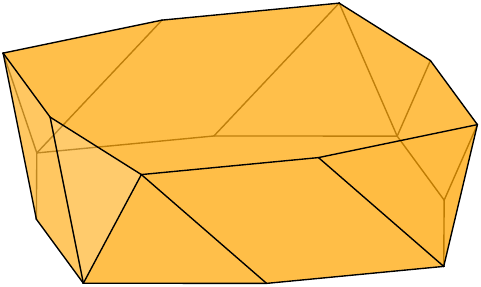}}
\\[1.5mm]
$n = 8$\break
$(\frac{k}{2} + \frac{1}{16})\cdot 2\pi$
&$n = 5, 10$\break
$(\frac{1}{2} + \frac{1}{20})\cdot 2\pi$
&$n = 12$\break
$(\frac{k}{2} + \frac{1}{24})\cdot 2\pi$
\end{tabular}\nobreak

\vfill\nobreak
\noindent\begin{tabular}{
>{\centering\arraybackslash}m{0.33\textwidth}
>{\centering\arraybackslash}m{0.33\textwidth}
>{\centering\arraybackslash}m{0.33\textwidth}}
\href{https://www.inf.fu-berlin.de/inst/ag-ti/software/DiscreteHopfFibration/gallery.html?f=TxCn/2/168cells_6tubes}{\includegraphics[scale=0.6]{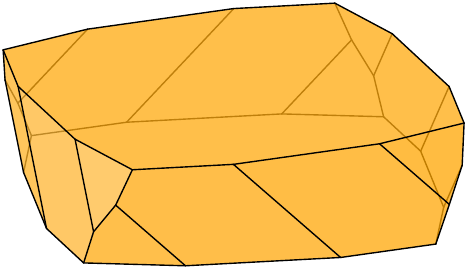}}
&
\href{https://www.inf.fu-berlin.de/inst/ag-ti/software/DiscreteHopfFibration/gallery.html?f=TxCn/2/192cells_6tubes}{\includegraphics[scale=0.6]{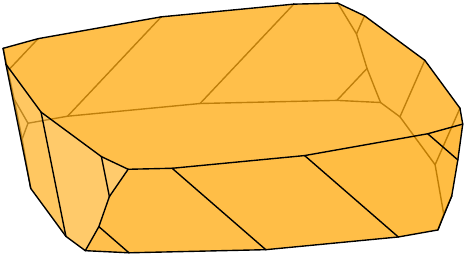}}
&
\href{https://www.inf.fu-berlin.de/inst/ag-ti/software/DiscreteHopfFibration/gallery.html?f=TxCn/2/216cells_6tubes}{\includegraphics[scale=0.6]{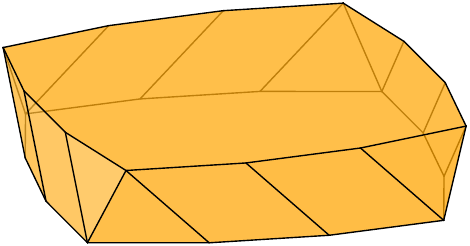}}
\\[1.5mm]
$n = 7, 14$\break
$(\frac{1}{2} + \frac{1}{28})\cdot 2\pi$
&$n = 16$\break
$(\frac{k}{2} + \frac{1}{32})\cdot 2\pi$
&$n = 9, 18$\break
$(\frac{1}{2} + \frac{1}{36})\cdot 2\pi$
\end{tabular}\nobreak

\vfill\nobreak
\noindent\begin{tabular}{
>{\centering\arraybackslash}m{0.33\textwidth}
>{\centering\arraybackslash}m{0.33\textwidth}
>{\centering\arraybackslash}m{0.33\textwidth}}
\href{https://www.inf.fu-berlin.de/inst/ag-ti/software/DiscreteHopfFibration/gallery.html?f=TxCn/2/240cells_6tubes}{\includegraphics[scale=0.6]{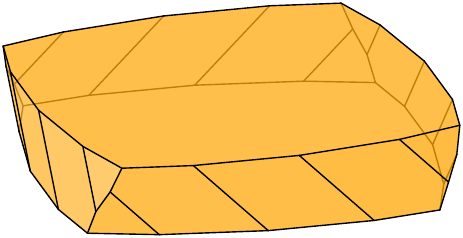}}
&
\href{https://www.inf.fu-berlin.de/inst/ag-ti/software/DiscreteHopfFibration/gallery.html?f=TxCn/2/264cells_6tubes}{\includegraphics[scale=0.6]{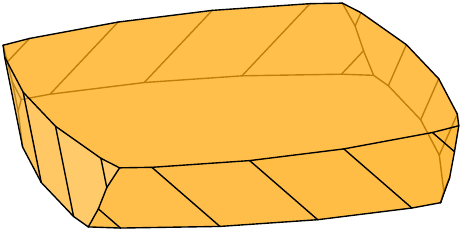}}
&
\href{https://www.inf.fu-berlin.de/inst/ag-ti/software/DiscreteHopfFibration/gallery.html?f=TxCn/2/288cells_6tubes}{\includegraphics[scale=0.6]{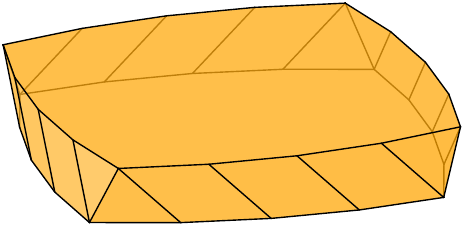}}
\\[1.5mm]
$n = 20$\break
$(\frac{k}{2} + \frac{1}{40})\cdot 2\pi$
&$n = 11, 22$\break
$(\frac{1}{2} + \frac{1}{44})\cdot 2\pi$
&$n = 24$\break
$(\frac{k}{2} + \frac{1}{48})\cdot 2\pi$
\end{tabular}\nobreak

\vfill\nobreak
\noindent\begin{tabular}{
>{\centering\arraybackslash}m{0.33\textwidth}
>{\centering\arraybackslash}m{0.33\textwidth}
>{\centering\arraybackslash}m{0.33\textwidth}}
\href{https://www.inf.fu-berlin.de/inst/ag-ti/software/DiscreteHopfFibration/gallery.html?f=TxCn/2/312cells_6tubes}{\includegraphics[scale=0.6]{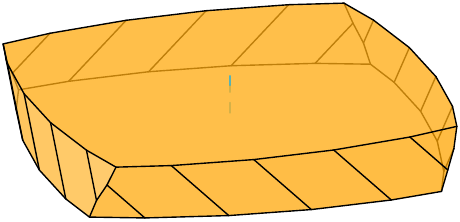}}
&
\href{https://www.inf.fu-berlin.de/inst/ag-ti/software/DiscreteHopfFibration/gallery.html?f=TxCn/2/336cells_6tubes}{\includegraphics[scale=0.6]{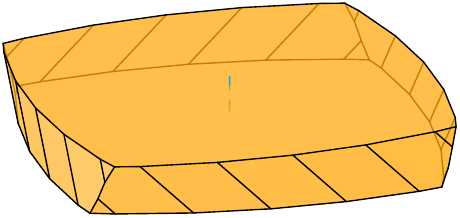}}
&
\href{https://www.inf.fu-berlin.de/inst/ag-ti/software/DiscreteHopfFibration/gallery.html?f=TxCn/2/360cells_6tubes}{\includegraphics[scale=0.6]{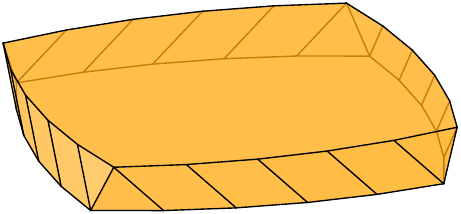}}
\\[1.5mm]
$n = 13, 26$\break
$(\frac{1}{2} + \frac{1}{52})\cdot 2\pi$
&$n = 28$\break
$(\frac{k}{2} + \frac{1}{56})\cdot 2\pi$
&$n = 15, 30$\break
$(\frac{1}{2} + \frac{1}{60})\cdot 2\pi$
\end{tabular}\nobreak

\begin{figure}[H]
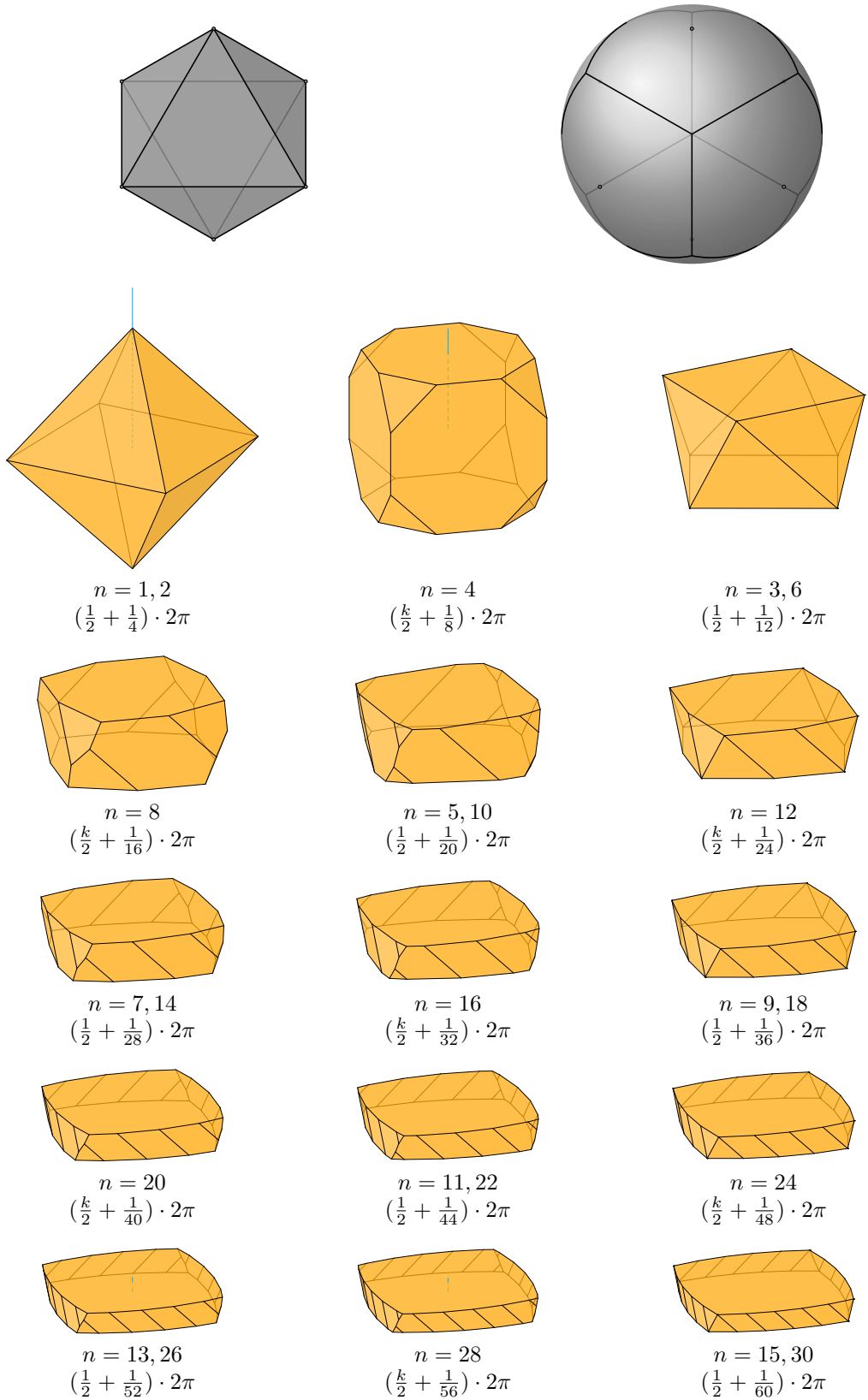

\caption{
$G=\pm[T\times C_n]$,
$G^h={+T}$,
2-fold rotation center
$p=(1, 0, 0)$.
$H=\langle [i, 1], [1, e_n] \rangle$.
6 tubes,
each with $\mathrm{lcm}(2n, 4)$ cells.
Alternate groups: $\pm[T\times D_{2n}]$ and $\pm\frac12[O\times D_{2n}]$ (also their common supergroup $\pm[O\times D_{2n}]$)
if $n \equiv 0 \mod 4$, else $\pm[T\times D_{2n}]$ (and its supergroup $\pm\frac12[O\times\overline{D}_{4n}]$).
When $n=1$ or $n=2$,
consecutive cells of a tube touch only via vertices.
}
\label{fig:TxCn_2fold}
\end{figure}\endgroup
\newpage

\subsection{\texorpdfstring{$\pm\frac13[T\times C_{3n}]$}{+-1/3[TxC3n]}}
\subsubsection{\texorpdfstring{$\pm\frac13[T\times C_{3n}]$}{+-1/3[TxC3n]}, 3-fold (type I) rotation center}
\begin{minipage}{0.5\textwidth}
\centering
\includegraphics[scale=1]{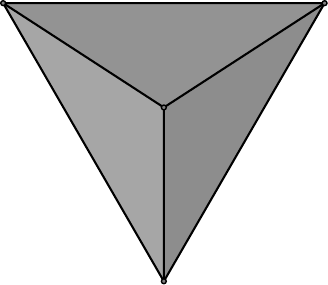}
\end{minipage}%
\begin{minipage}{0.5\textwidth}
\centering
\includegraphics[scale=1]{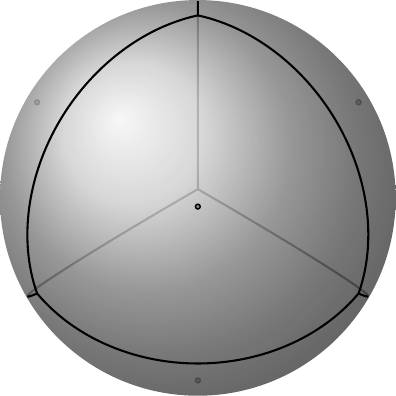}
\end{minipage}
\vskip 5pt plus 0.5fill
\begingroup
\setlength{\tabcolsep}{0pt}
\noindent\begin{tabular}{
>{\centering\arraybackslash}m{0.33\textwidth}
>{\centering\arraybackslash}m{0.33\textwidth}
>{\centering\arraybackslash}m{0.33\textwidth}}
\href{https://www.inf.fu-berlin.de/inst/ag-ti/software/DiscreteHopfFibration/gallery.html?f=tTxC3n/3/8cells_4tubes}{\vbox{\vskip-3mm \includegraphics[scale=0.3]{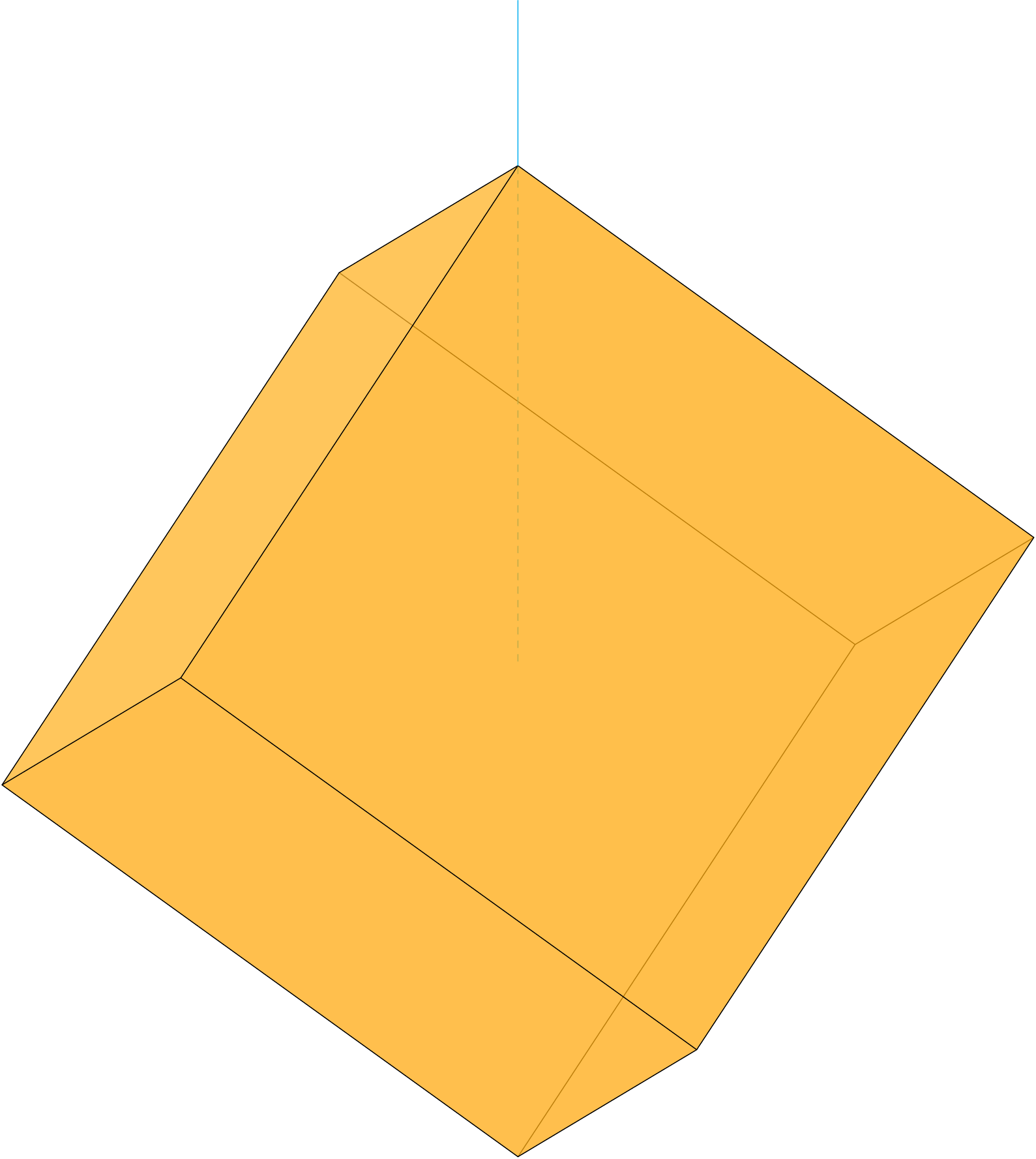}}}
&
\href{https://www.inf.fu-berlin.de/inst/ag-ti/software/DiscreteHopfFibration/gallery.html?f=tTxC3n/3/32cells_4tubes}{\vbox{\vskip-3mm \includegraphics[scale=0.6]{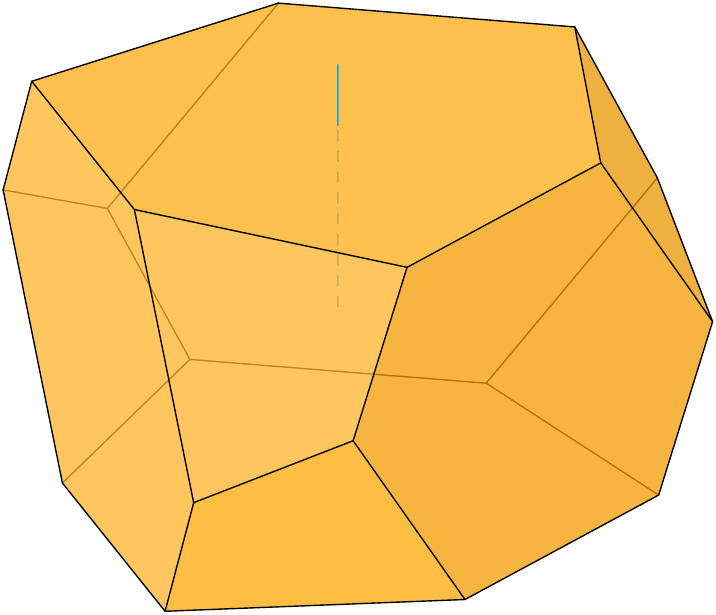}}}
&
\href{https://www.inf.fu-berlin.de/inst/ag-ti/software/DiscreteHopfFibration/gallery.html?f=tTxC3n/3/48cells_4tubes}{\includegraphics[scale=0.6]{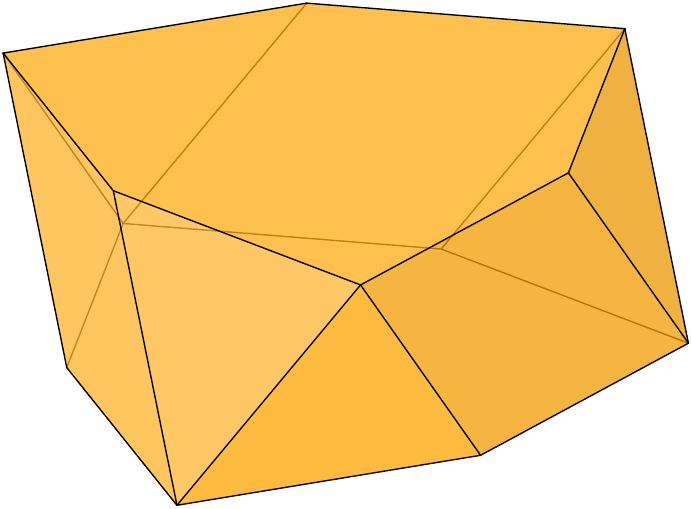}}
\\[1.5mm]
$n = 1$\break
$(\frac{k}{3} + \frac{1}{2})\cdot 2\pi$
&$n = 4$\break
$(\frac{k}{3} + \frac{1}{8})\cdot 2\pi$
&$n = 2$\break
$(\frac{2}{3} + \frac{1}{12})\cdot 2\pi$
\end{tabular}\nobreak

\vfill\nobreak
\noindent\begin{tabular}{
>{\centering\arraybackslash}m{0.33\textwidth}
>{\centering\arraybackslash}m{0.33\textwidth}
>{\centering\arraybackslash}m{0.33\textwidth}}
\href{https://www.inf.fu-berlin.de/inst/ag-ti/software/DiscreteHopfFibration/gallery.html?f=tTxC3n/3/56cells_4tubes}{\includegraphics[scale=0.6]{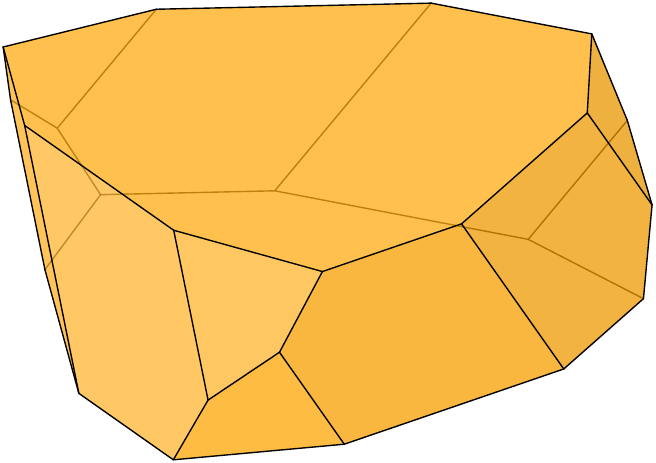}}
&
\href{https://www.inf.fu-berlin.de/inst/ag-ti/software/DiscreteHopfFibration/gallery.html?f=tTxC3n/3/72cells_4tubes}{\includegraphics[scale=0.6]{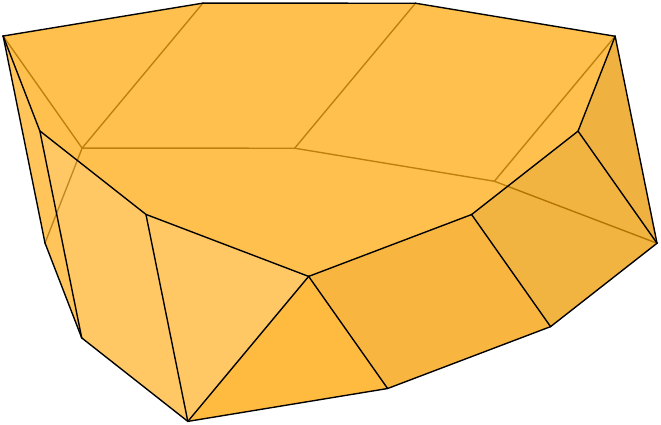}}
&
\href{https://www.inf.fu-berlin.de/inst/ag-ti/software/DiscreteHopfFibration/gallery.html?f=tTxC3n/3/80cells_4tubes}{\includegraphics[scale=0.6]{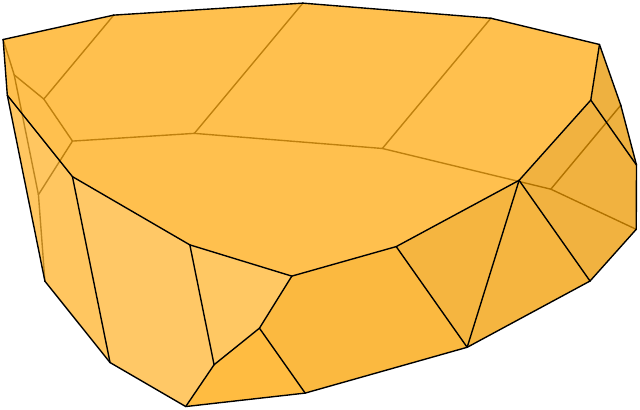}}
\\[1.5mm]
$n = 7$\break
$(\frac{k}{3} + \frac{1}{14})\cdot 2\pi$
&$n = 3$\break
$(\frac{1}{3} + \frac{1}{18})\cdot 2\pi$
&$n = 10$\break
$(\frac{k}{3} + \frac{1}{20})\cdot 2\pi$
\end{tabular}\nobreak

\vfill\nobreak
\noindent\begin{tabular}{
>{\centering\arraybackslash}m{0.33\textwidth}
>{\centering\arraybackslash}m{0.33\textwidth}
>{\centering\arraybackslash}m{0.33\textwidth}}
\href{https://www.inf.fu-berlin.de/inst/ag-ti/software/DiscreteHopfFibration/gallery.html?f=tTxC3n/3/104cells_4tubes}{\includegraphics[scale=0.6]{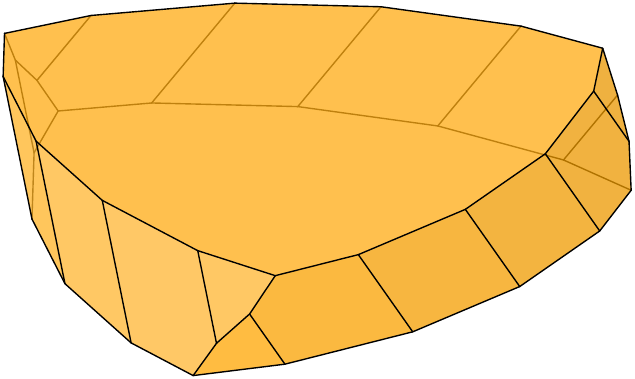}}
&
\href{https://www.inf.fu-berlin.de/inst/ag-ti/software/DiscreteHopfFibration/gallery.html?f=tTxC3n/3/120cells_4tubes}{\includegraphics[scale=0.6]{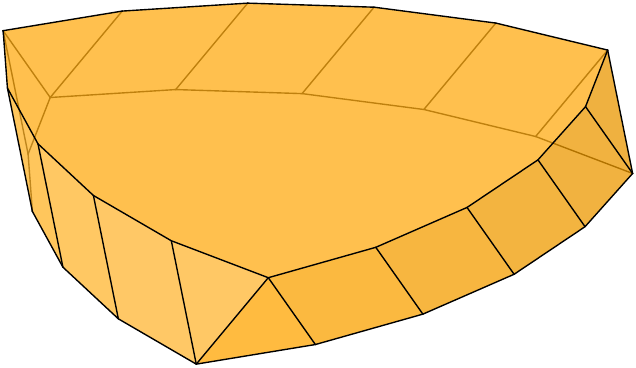}}
&
\href{https://www.inf.fu-berlin.de/inst/ag-ti/software/DiscreteHopfFibration/gallery.html?f=tTxC3n/3/128cells_4tubes}{\includegraphics[scale=0.6]{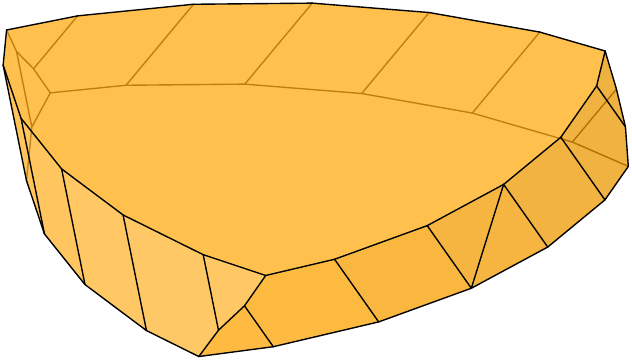}}
\\[1.5mm]
$n = 13$\break
$(\frac{k}{3} + \frac{1}{26})\cdot 2\pi$
&$n = 5$\break
$(\frac{2}{3} + \frac{1}{30})\cdot 2\pi$
&$n = 16$\break
$(\frac{k}{3} + \frac{1}{32})\cdot 2\pi$
\end{tabular}\nobreak

\vfill\nobreak
\noindent\begin{tabular}{
>{\centering\arraybackslash}m{0.33\textwidth}
>{\centering\arraybackslash}m{0.33\textwidth}
>{\centering\arraybackslash}m{0.33\textwidth}}
\href{https://www.inf.fu-berlin.de/inst/ag-ti/software/DiscreteHopfFibration/gallery.html?f=tTxC3n/3/144cells_4tubes}{\includegraphics[scale=0.6]{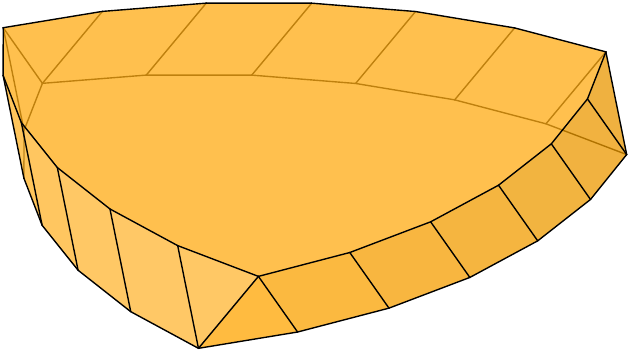}}
&
\href{https://www.inf.fu-berlin.de/inst/ag-ti/software/DiscreteHopfFibration/gallery.html?f=tTxC3n/3/152cells_4tubes}{\includegraphics[scale=0.6]{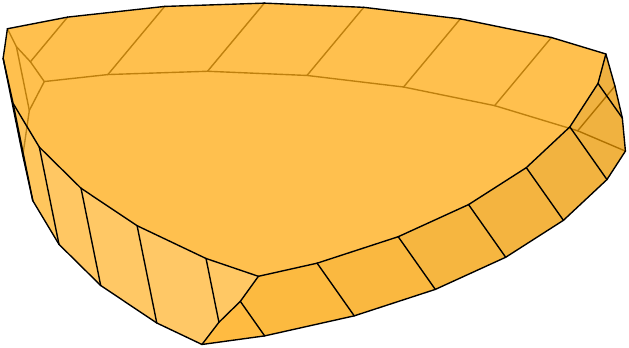}}
&
\href{https://www.inf.fu-berlin.de/inst/ag-ti/software/DiscreteHopfFibration/gallery.html?f=tTxC3n/3/176cells_4tubes}{\includegraphics[scale=0.6]{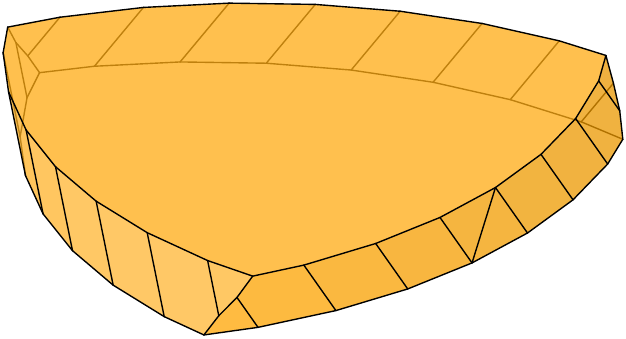}}
\\[1.5mm]
$n = 6$\break
$(\frac{1}{3} + \frac{1}{36})\cdot 2\pi$
&$n = 19$\break
$(\frac{k}{3} + \frac{1}{38})\cdot 2\pi$
&$n = 22$\break
$(\frac{k}{3} + \frac{1}{44})\cdot 2\pi$
\end{tabular}\nobreak

\begin{figure}[H]
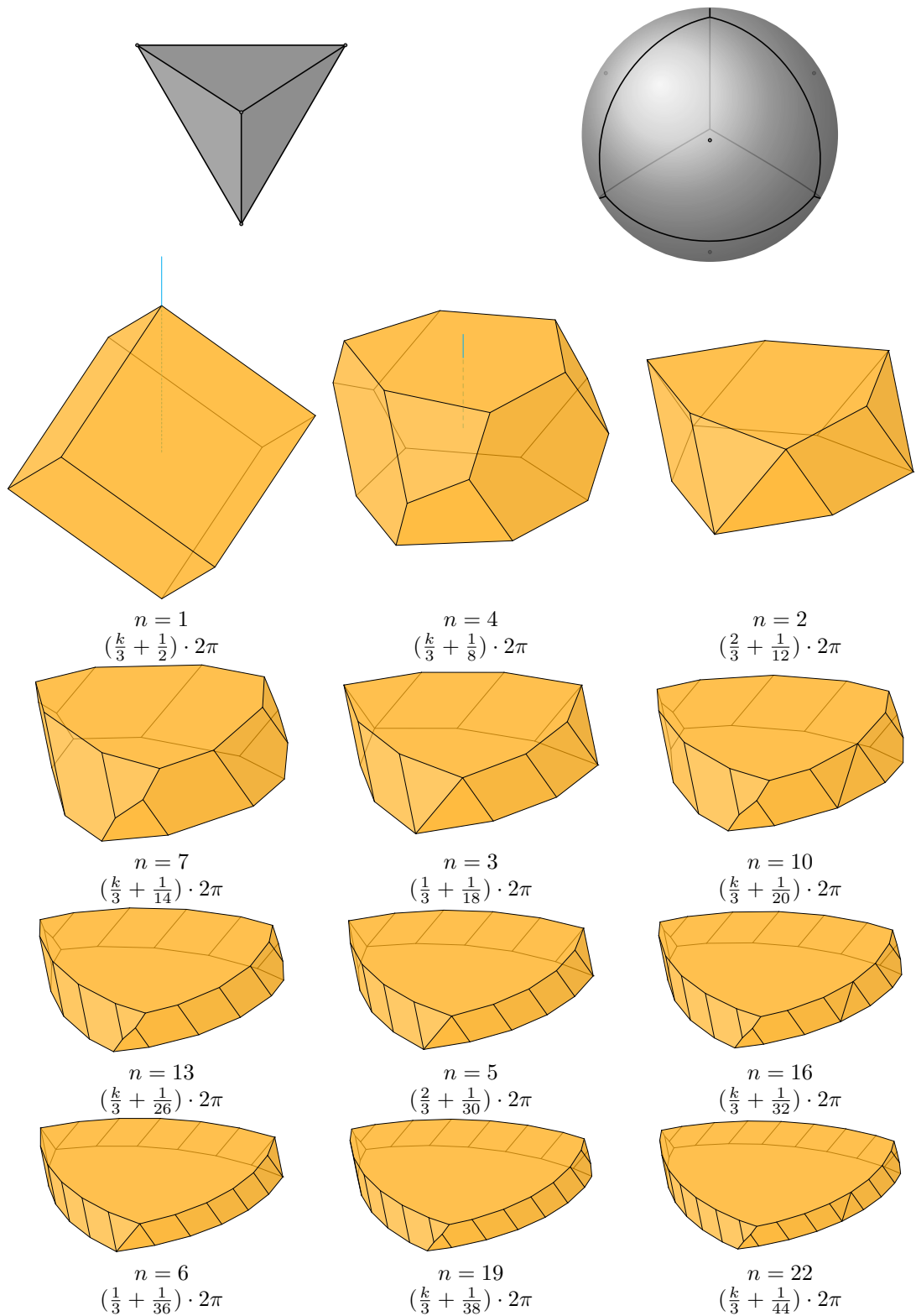

\caption{
$G=\pm\frac13[T\times C_{3n}]$,
$G^h={+T}$,
3-fold (type I) rotation center
$p=\frac1{\sqrt3}(-1, -1, -1)$.
$H=\langle [-\omega, e_{3n}], [1, e_n] \rangle$.
4 tubes,
each with $\frac{6n}{\gcd(n-1, 3)}$ cells.
Alternate groups: $\pm\frac16[O\times D_{6n}]$ (and its supergroup $\pm\frac12[O\times D_{6n}]$ if $n\not\equiv 1\mod 3$).
When $n = 1$,
the cells of a tube are disconnected from each other.
}
\label{fig:tTxC3n_3fold}
\end{figure}\endgroup
\newpage

\subsubsection{\texorpdfstring{$\pm\frac13[T\times C_{3n}]$}{+-1/3[TxC3n]}, 3-fold (type II) rotation center}
\begin{minipage}{0.5\textwidth}
\centering
\includegraphics[scale=1]{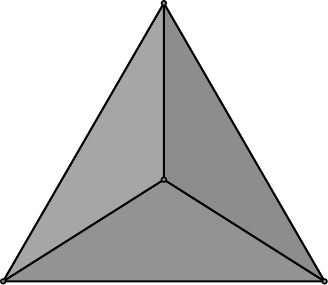}
\end{minipage}%
\begin{minipage}{0.5\textwidth}
\centering
\includegraphics[scale=1]{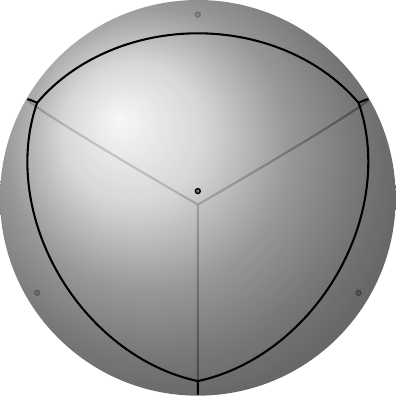}
\end{minipage}
\vskip 5pt plus 0.5fill
\begingroup
\setlength{\tabcolsep}{0pt}
\noindent\begin{tabular}{
>{\centering\arraybackslash}m{0.33\textwidth}
>{\centering\arraybackslash}m{0.33\textwidth}
>{\centering\arraybackslash}m{0.33\textwidth}}
\href{https://www.inf.fu-berlin.de/inst/ag-ti/software/DiscreteHopfFibration/gallery.html?f=tTxC3n/3p/16cells_4tubes}{\includegraphics[scale=0.5]{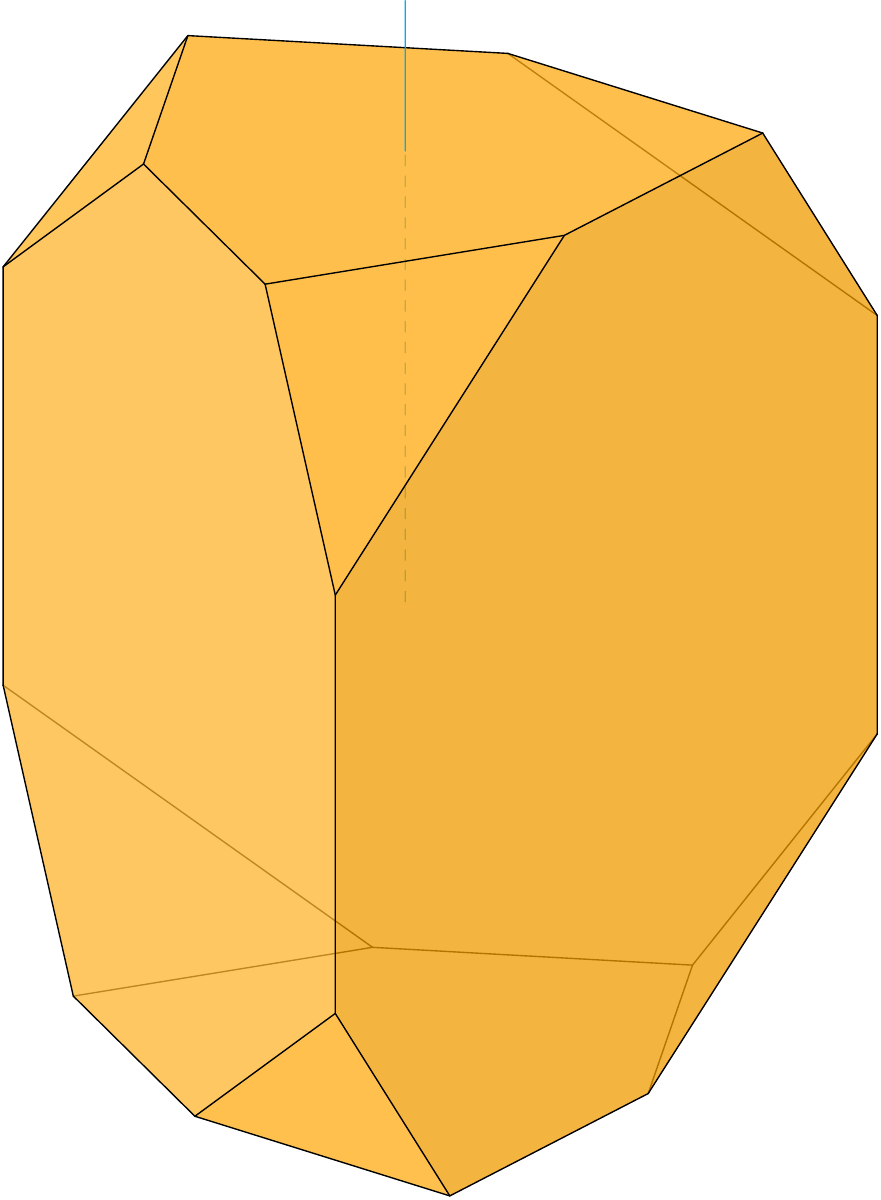}}
&
\href{https://www.inf.fu-berlin.de/inst/ag-ti/software/DiscreteHopfFibration/gallery.html?f=tTxC3n/3p/24cells_4tubes}{\includegraphics[scale=0.5]{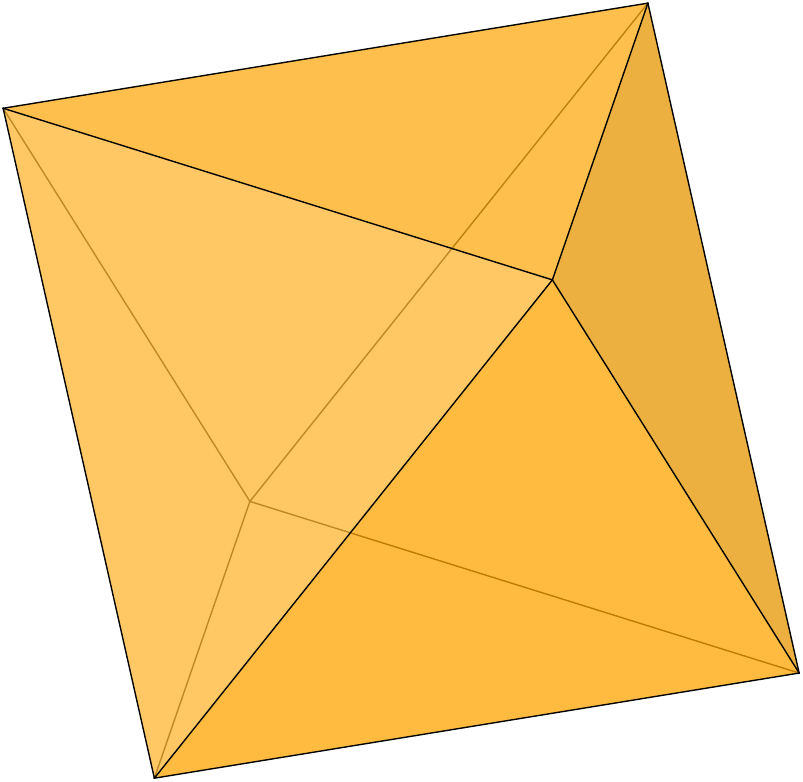}}
&
\href{https://www.inf.fu-berlin.de/inst/ag-ti/software/DiscreteHopfFibration/gallery.html?f=tTxC3n/3p/40cells_4tubes}{\includegraphics[scale=0.5]{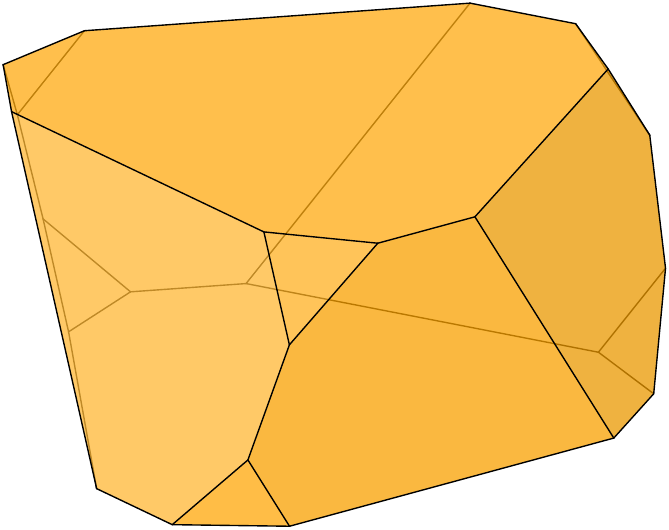}}
\\[1.5mm]
$n = 2$\break
$(\frac{k}{3} + \frac{1}{4})\cdot 2\pi$
&$n = 1$\break
$(\frac{1}{3} + \frac{1}{6})\cdot 2\pi$
&$n = 5$\break
$(\frac{k}{3} + \frac{1}{10})\cdot 2\pi$
\end{tabular}\nobreak

\vfill\nobreak
\noindent\begin{tabular}{
>{\centering\arraybackslash}m{0.33\textwidth}
>{\centering\arraybackslash}m{0.33\textwidth}
>{\centering\arraybackslash}m{0.33\textwidth}}
\href{https://www.inf.fu-berlin.de/inst/ag-ti/software/DiscreteHopfFibration/gallery.html?f=tTxC3n/3p/64cells_4tubes}{\includegraphics[scale=0.5]{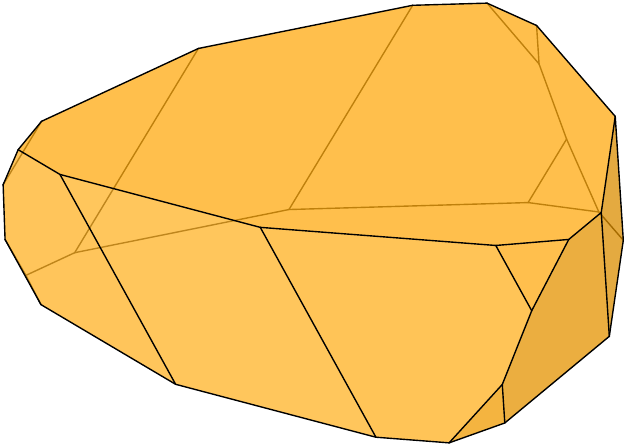}}
&
\href{https://www.inf.fu-berlin.de/inst/ag-ti/software/DiscreteHopfFibration/gallery.html?f=tTxC3n/3p/72cells_4tubes}{\includegraphics[scale=0.5]{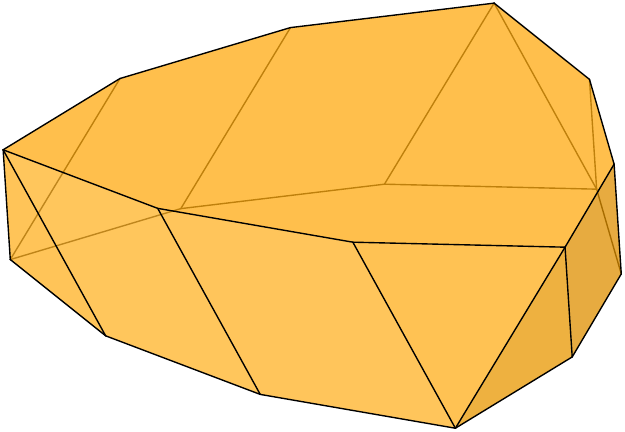}}
&
\href{https://www.inf.fu-berlin.de/inst/ag-ti/software/DiscreteHopfFibration/gallery.html?f=tTxC3n/3p/88cells_4tubes}{\includegraphics[scale=0.5]{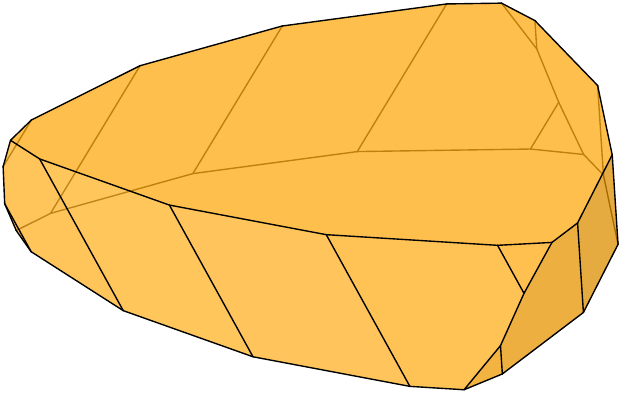}}
\\[1.5mm]
$n = 8$\break
$(\frac{k}{3} + \frac{1}{16})\cdot 2\pi$
&$n = 3$\break
$(\frac{2}{3} + \frac{1}{18})\cdot 2\pi$
&$n = 11$\break
$(\frac{k}{3} + \frac{1}{22})\cdot 2\pi$
\end{tabular}\nobreak

\vfill\nobreak
\noindent\begin{tabular}{
>{\centering\arraybackslash}m{0.33\textwidth}
>{\centering\arraybackslash}m{0.33\textwidth}
>{\centering\arraybackslash}m{0.33\textwidth}}
\href{https://www.inf.fu-berlin.de/inst/ag-ti/software/DiscreteHopfFibration/gallery.html?f=tTxC3n/3p/96cells_4tubes}{\includegraphics[scale=0.5]{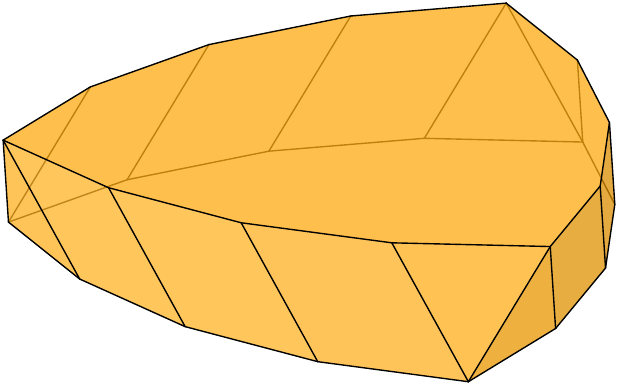}}
&
\href{https://www.inf.fu-berlin.de/inst/ag-ti/software/DiscreteHopfFibration/gallery.html?f=tTxC3n/3p/112cells_4tubes}{\includegraphics[scale=0.5]{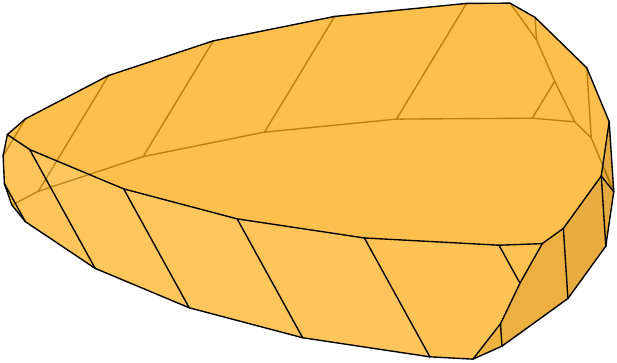}}
&
\href{https://www.inf.fu-berlin.de/inst/ag-ti/software/DiscreteHopfFibration/gallery.html?f=tTxC3n/3p/136cells_4tubes}{\includegraphics[scale=0.5]{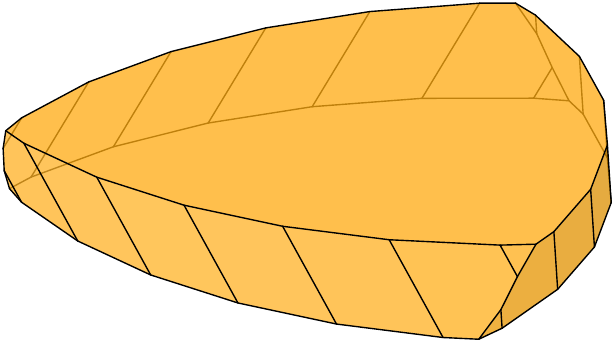}}
\\[1.5mm]
$n = 4$\break
$(\frac{1}{3} + \frac{1}{24})\cdot 2\pi$
&$n = 14$\break
$(\frac{k}{3} + \frac{1}{28})\cdot 2\pi$
&$n = 17$\break
$(\frac{k}{3} + \frac{1}{34})\cdot 2\pi$
\end{tabular}\nobreak

\vfill\nobreak
\noindent\begin{tabular}{
>{\centering\arraybackslash}m{0.33\textwidth}
>{\centering\arraybackslash}m{0.33\textwidth}
>{\centering\arraybackslash}m{0.33\textwidth}}
\href{https://www.inf.fu-berlin.de/inst/ag-ti/software/DiscreteHopfFibration/gallery.html?f=tTxC3n/3p/144cells_4tubes}{\includegraphics[scale=0.5]{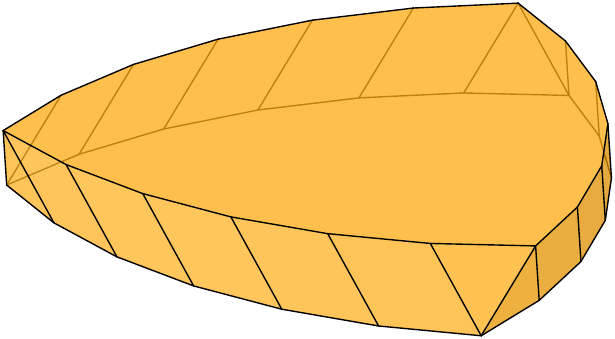}}
&
\href{https://www.inf.fu-berlin.de/inst/ag-ti/software/DiscreteHopfFibration/gallery.html?f=tTxC3n/3p/160cells_4tubes}{\includegraphics[scale=0.5]{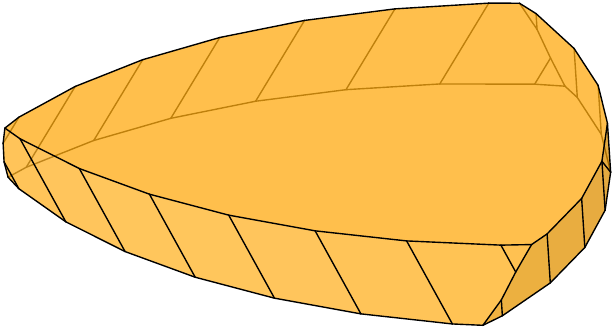}}
&
\href{https://www.inf.fu-berlin.de/inst/ag-ti/software/DiscreteHopfFibration/gallery.html?f=tTxC3n/3p/168cells_4tubes}{\includegraphics[scale=0.5]{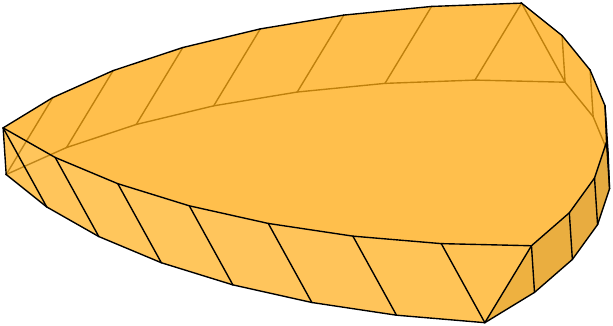}}
\\[1.5mm]
$n = 6$\break
$(\frac{2}{3} + \frac{1}{36})\cdot 2\pi$
&$n = 20$\break
$(\frac{k}{3} + \frac{1}{40})\cdot 2\pi$
&$n = 7$\break
$(\frac{1}{3} + \frac{1}{42})\cdot 2\pi$
\end{tabular}\nobreak

\begin{figure}[H]
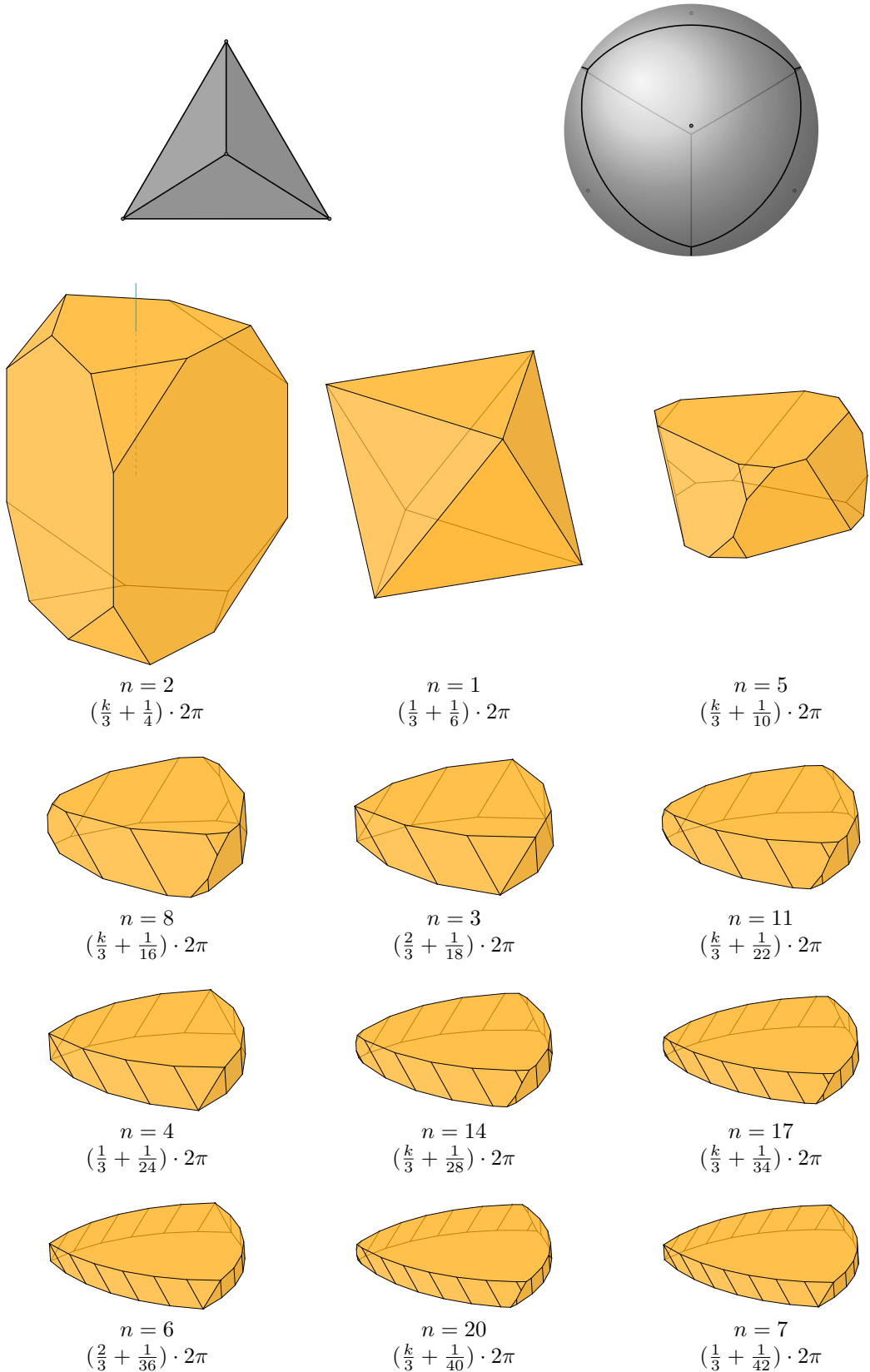

\caption{
$G=\pm\frac13[T\times C_{3n}]$,
$G^h={+T}$,
3-fold (type II) rotation center
$p=\frac1{\sqrt3}(1, 1, 1)$.
$H=\langle [-\omega^2, e_{3n}^2], [1, e_n] \rangle$.
4 tubes,
each with $\frac{6n}{\gcd(n-2, 3)}$ cells.
Alternate groups: $\pm\frac16[O\times D_{6n}]$ (and its supergroup $\pm\frac12[O\times D_{6n}]$ if $n\not\equiv 2\mod 3$).
}
\label{fig:tTxC3n_3pfold}
\end{figure}\endgroup
\newpage

\subsubsection{\texorpdfstring{$\pm\frac13[T\times C_{3n}]$}{+-1/3[TxC3n]}, 2-fold rotation center}
\begin{minipage}{0.5\textwidth}
\centering
\includegraphics[scale=1]{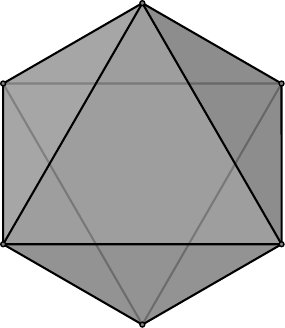}
\end{minipage}%
\begin{minipage}{0.5\textwidth}
\centering
\includegraphics[scale=1]{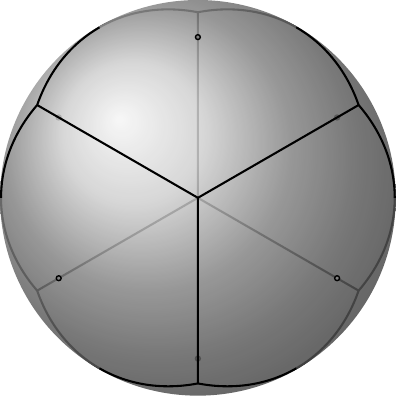}
\end{minipage}
\vskip 5pt plus 0.5fill
\begingroup
\setlength{\tabcolsep}{0pt}
\noindent\begin{tabular}{
>{\centering\arraybackslash}m{0.33\textwidth}
>{\centering\arraybackslash}m{0.33\textwidth}
>{\centering\arraybackslash}m{0.33\textwidth}}
\href{https://www.inf.fu-berlin.de/inst/ag-ti/software/DiscreteHopfFibration/gallery.html?f=tTxC3n/2/24cells_6tubes}{\includegraphics[scale=0.5]{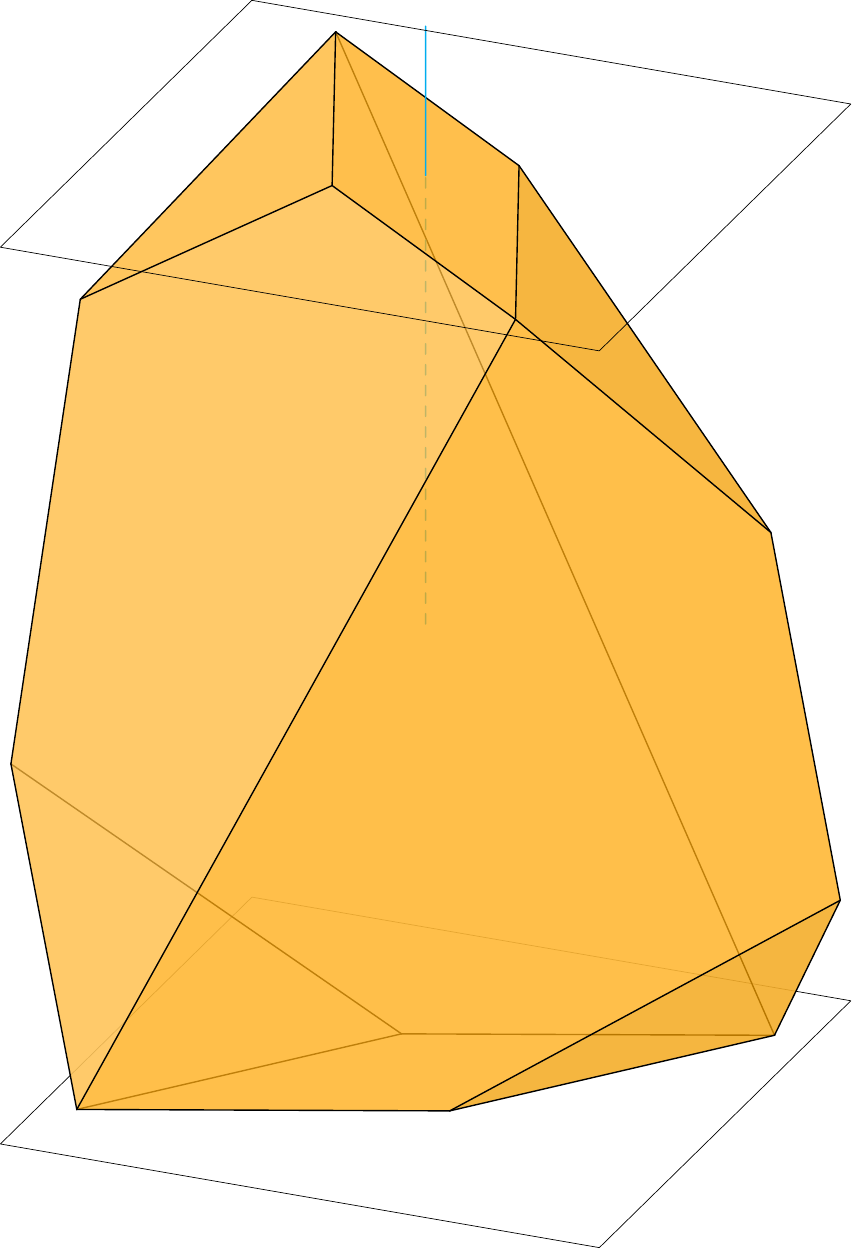}}
&
\href{https://www.inf.fu-berlin.de/inst/ag-ti/software/DiscreteHopfFibration/gallery.html?f=tTxC3n/2/48cells_6tubes}{\includegraphics[scale=0.6]{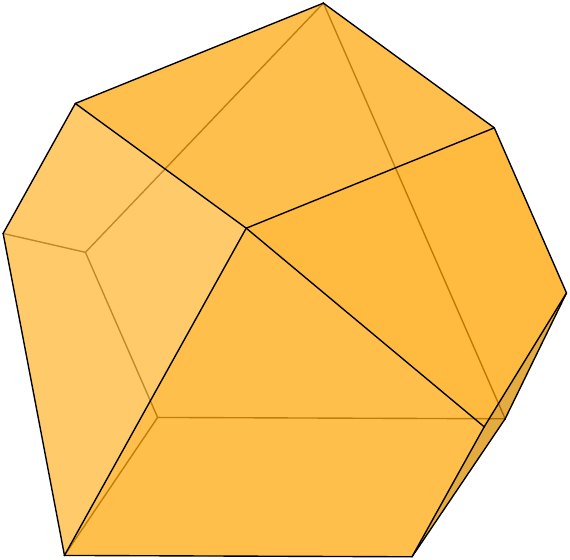}}
&
\href{https://www.inf.fu-berlin.de/inst/ag-ti/software/DiscreteHopfFibration/gallery.html?f=tTxC3n/2/72cells_6tubes}{\includegraphics[scale=0.6]{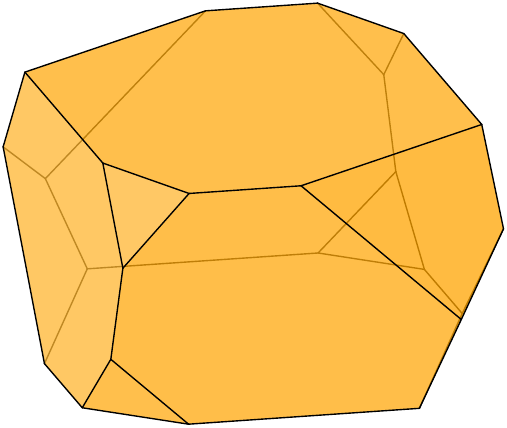}}
\\[1.5mm]
$n = 1, 2$\break
$(\frac{1}{2} + \frac{1}{4})\cdot 2\pi$
&$n = 4$\break
$(\frac{k}{2} + \frac{1}{8})\cdot 2\pi$
&$n = 3, 6$\break
$(\frac{1}{2} + \frac{1}{12})\cdot 2\pi$
\end{tabular}\nobreak

\vfill\nobreak
\noindent\begin{tabular}{
>{\centering\arraybackslash}m{0.33\textwidth}
>{\centering\arraybackslash}m{0.33\textwidth}
>{\centering\arraybackslash}m{0.33\textwidth}}
\href{https://www.inf.fu-berlin.de/inst/ag-ti/software/DiscreteHopfFibration/gallery.html?f=tTxC3n/2/96cells_6tubes}{\includegraphics[scale=0.6]{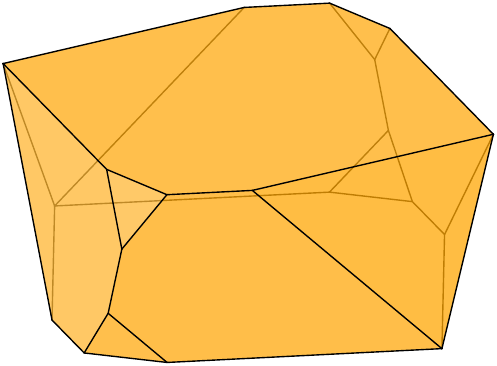}}
&
\href{https://www.inf.fu-berlin.de/inst/ag-ti/software/DiscreteHopfFibration/gallery.html?f=tTxC3n/2/120cells_6tubes}{\includegraphics[scale=0.6]{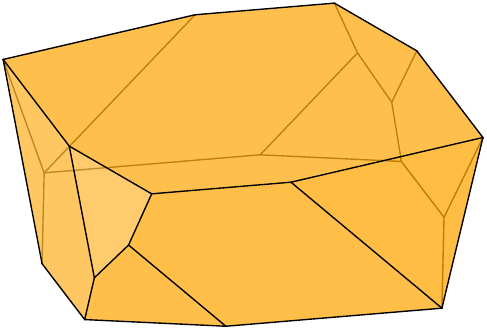}}
&
\href{https://www.inf.fu-berlin.de/inst/ag-ti/software/DiscreteHopfFibration/gallery.html?f=tTxC3n/2/144cells_6tubes}{\includegraphics[scale=0.6]{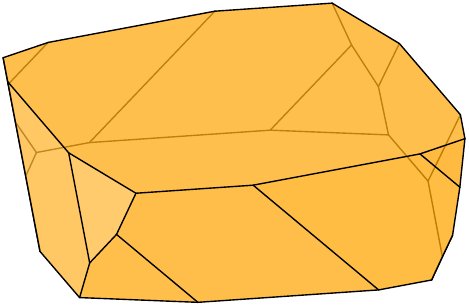}}
\\[1.5mm]
$n = 8$\break
$(\frac{k}{2} + \frac{1}{16})\cdot 2\pi$
&$n = 5, 10$\break
$(\frac{1}{2} + \frac{1}{20})\cdot 2\pi$
&$n = 12$\break
$(\frac{k}{2} + \frac{1}{24})\cdot 2\pi$
\end{tabular}\nobreak

\vfill\nobreak
\noindent\begin{tabular}{
>{\centering\arraybackslash}m{0.33\textwidth}
>{\centering\arraybackslash}m{0.33\textwidth}
>{\centering\arraybackslash}m{0.33\textwidth}}
\href{https://www.inf.fu-berlin.de/inst/ag-ti/software/DiscreteHopfFibration/gallery.html?f=tTxC3n/2/168cells_6tubes}{\includegraphics[scale=0.6]{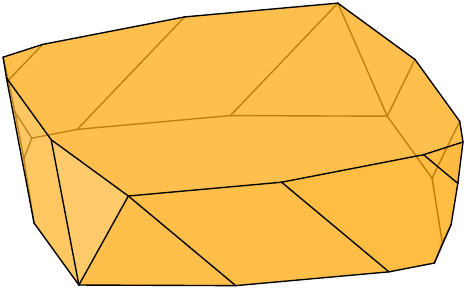}}
&
\href{https://www.inf.fu-berlin.de/inst/ag-ti/software/DiscreteHopfFibration/gallery.html?f=tTxC3n/2/192cells_6tubes}{\includegraphics[scale=0.6]{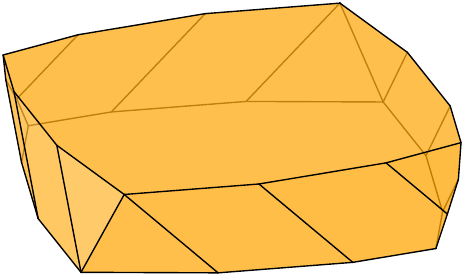}}
&
\href{https://www.inf.fu-berlin.de/inst/ag-ti/software/DiscreteHopfFibration/gallery.html?f=tTxC3n/2/216cells_6tubes}{\includegraphics[scale=0.6]{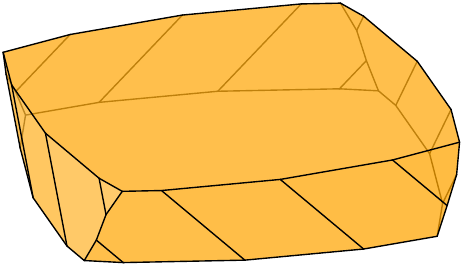}}
\\[1.5mm]
$n = 7, 14$\break
$(\frac{1}{2} + \frac{1}{28})\cdot 2\pi$
&$n = 16$\break
$(\frac{k}{2} + \frac{1}{32})\cdot 2\pi$
&$n = 9, 18$\break
$(\frac{1}{2} + \frac{1}{36})\cdot 2\pi$
\end{tabular}\nobreak

\vfill\nobreak
\noindent\begin{tabular}{
>{\centering\arraybackslash}m{0.33\textwidth}
>{\centering\arraybackslash}m{0.33\textwidth}
>{\centering\arraybackslash}m{0.33\textwidth}}
\href{https://www.inf.fu-berlin.de/inst/ag-ti/software/DiscreteHopfFibration/gallery.html?f=tTxC3n/2/240cells_6tubes}{\includegraphics[scale=0.6]{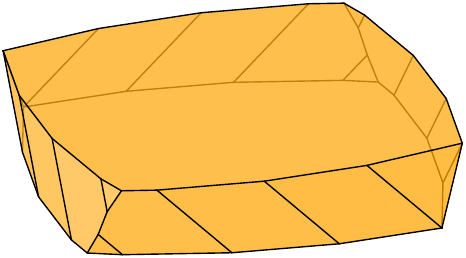}}
&
\href{https://www.inf.fu-berlin.de/inst/ag-ti/software/DiscreteHopfFibration/gallery.html?f=tTxC3n/2/264cells_6tubes}{\includegraphics[scale=0.6]{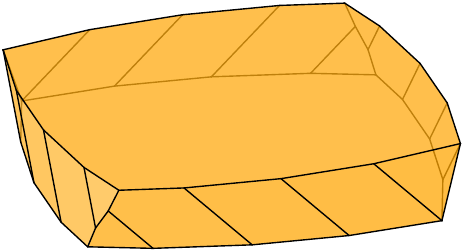}}
&
\href{https://www.inf.fu-berlin.de/inst/ag-ti/software/DiscreteHopfFibration/gallery.html?f=tTxC3n/2/288cells_6tubes}{\includegraphics[scale=0.6]{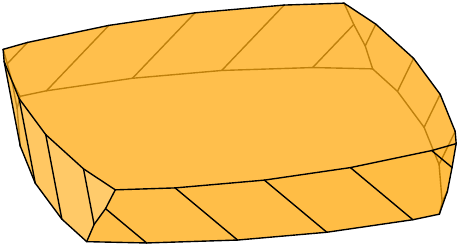}}
\\[1.5mm]
$n = 20$\break
$(\frac{k}{2} + \frac{1}{40})\cdot 2\pi$
&$n = 11, 22$\break
$(\frac{1}{2} + \frac{1}{44})\cdot 2\pi$
&$n = 24$\break
$(\frac{k}{2} + \frac{1}{48})\cdot 2\pi$
\end{tabular}\nobreak

\begin{figure}[H]
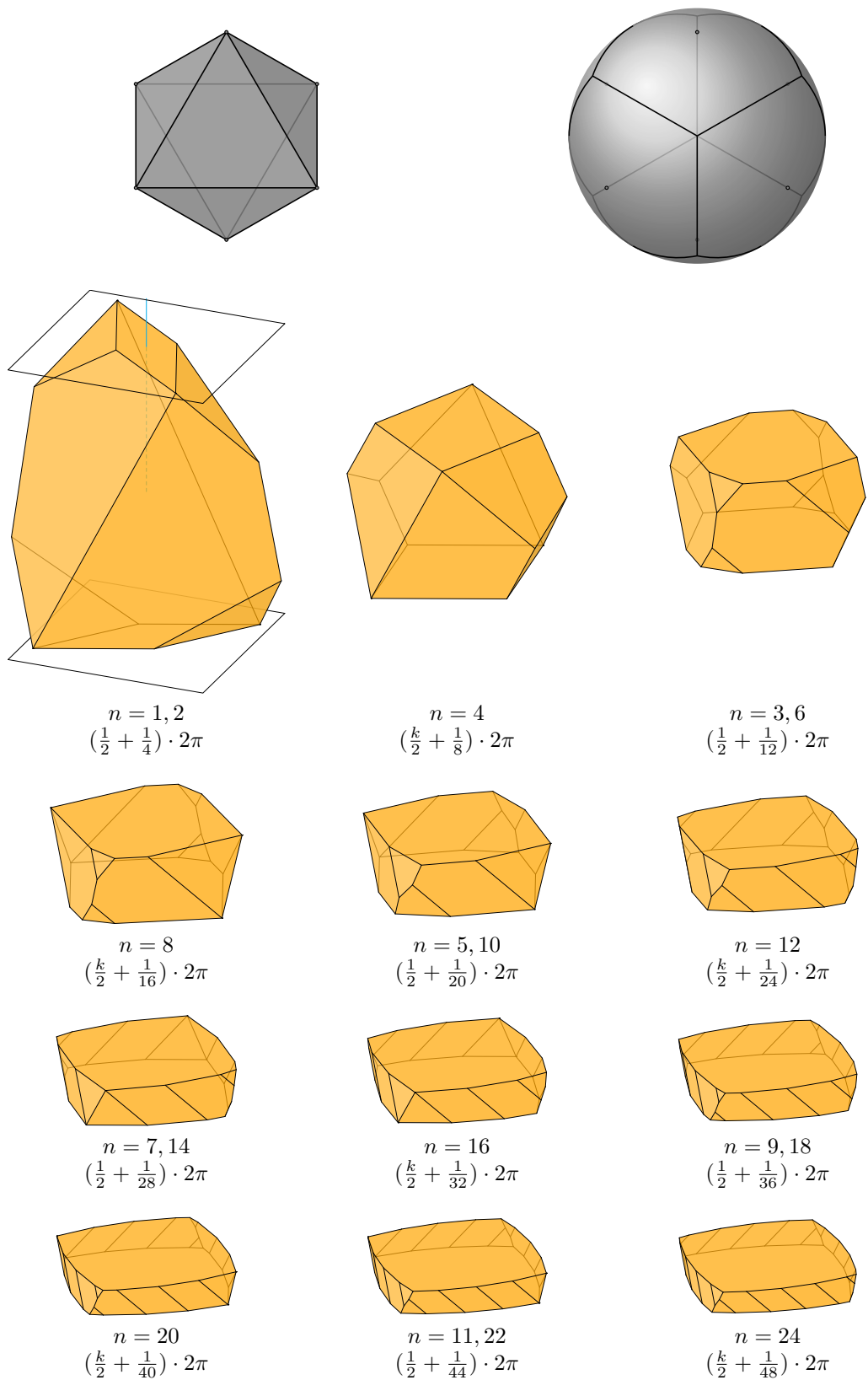

\caption{
$G=\pm\frac13[T\times C_{3n}]$,
$G^h={+T}$,
2-fold rotation center
$p=(1, 0, 0)$.
$H=\langle [i, 1], [1, e_n] \rangle$.
6 tubes,
each with $\mathrm{lcm}(2n, 4)$ cells.
Alternate group: $\pm\frac16[O\times D_{6n}]$.
For $n=1$ and $n=2$,
we have drawn squares in the planes around the top and bottom face,
to indicate that these faces are horizontal and parallel.
}
\label{fig:tTxC3n_2fold}
\end{figure}\endgroup
\newpage

\section{The number of groups of given order}
\label{sec:counting}

We will see that
the number of groups of order $N$ is always at least $N/2$,
and less than $O(N^2)$.
If $N$ is an odd prime,  there are exactly $(N+3)/2$
groups,
namely the torus translation groups
$\grp1_{1,N}^{(s)}$ for $0\le s \le (N-1)/2$ and %
$\grp1_{N,1}^{((1-N)/2)}$.

The richest class of groups are the toroidal groups, and among them,
the most numerous groups are the torus translation groups, of type $\grp1$: For each
divisor $m$ of $N$, there are $\sim n/2$ groups $\grp1_{m,n}^{(s)}$,
where
$n=N/m$. Thus, the number of groups is about $1/2$ times the
sum $\sigma(N)$
of
divisors of $N$, which is bounded by
$N^{1+\frac{1+O(1/\log\log n)}{\log_2\ln N}}\le N^2$
\cite{nicolas_robin_1983}.
 The upper bound of $O(N^2)$ is very weak; the actual bound is
 slightly superlinear.
 
The number of groups of type $\grp.$ is of similar magnitude, provided
that $N$ is even. For all the other types, there is at most one group
for every divisor of $N$, except for the swapturn groups, whose number
is related to the number of integer points on the circle
$a^2+b^2=N/4$, and this number is at most~$N$.

 From all the remaining classes of groups (tubical, polyhedral, or axial), there can be
 only a constant number of groups of a given order.

\paragraph{The number of groups of order 100.}
As an exercise, let us compute the number of point groups of order $N=100$.

We proceed through the toroidal classes of groups
in
Table~\ref{tab:overview}
one by one.
For the pure translation groups of type $\grp1$, we can write
$100=mn
=1\cdot 100
=2\cdot 50
=4\cdot 25
=5\cdot 20
=10\cdot 10
=20\cdot 5
=25\cdot 4
=50\cdot 2
=100\cdot 1$
with accordingly
$50+26+13+10+6+3+2+2+1=113$ choices of~$s$,
see the remark after \eqref{eq:ks}
in Section~\ref{sec:translations}.
For the flip groups of type $\grp.$ of order
$100=2mn$, we have to factor 50 instead of 100.
The possibilities are
$50
=1\cdot 50
=2\cdot 25
=5\cdot 10
=10\cdot 5
=25\cdot 2
=50\cdot 1$
with $
25+13+5+3+1+1=48$
choices of~$s$.

For the swap groups
$\grp\setminus^{\mathbf{pm}}_{m,n}$
of order $4mn$, we have to split
$25=mn$ into two factors $mn$ larger than 1. There is one possibility:
$25=5\times5$.
For the groups
$\grp\setminus^{\mathbf{pg}}_{m,n}$,
only the first factor $m$ must be larger than~1. This gives 2 choices.
For
$\grp\setminus^{\mathbf{cm}}_{m,n}$ of order $2mn$, $mn=50$ must be
split into two factors of the same parity. This is impossible since
$mn\equiv 2\pmod4$.
Thus, in total we have 3 swap groups of type
$\grp\setminus$. 
Clearly, there is the same number of 3 swap groups of type
$\grp/$.

Finally, for the full torus swap groups,
almost all
types have order $8mn$, which cannot equal 100.
We only need to consider the groups of
type
$\grp X^{\mathbf{c2mm}}_{m,n}$, of order $4mn$,
We have to split $100/4=25$ into two factors $\ge3$ of the same
parity.
There is one possibility: $25=5\times 5$.

In total, we get $113+48+3+3+1=168$ chiral toroidal groups of order 100.

Let us turn to the achiral groups:
For the reflection groups $\grp|$, we have to consider all
factorizations $100=2mn$
(types $\mathbf{pm}$ and $\mathbf{pg}$)
or $100=4mn$ (type
$\mathbf{cm}$). This gives $2\times \sigma_0(50) + \sigma_0(25) = 2\times 6 +
3=15$ groups, where $\sigma_0$ denotes the number of divisors of a number.

For the full reflection groups $\grp+$, we have to consider all
factorizations $100=4mn$ or $100=8mn$, respectively, where in one case
($\mathbf{p2mg}$), we distinguish the order of the factors.  We get
$2+3+2+0=7$ possibilities. For general $N$, there are
$2\lceil\sigma_0(\frac N4)/2\rceil + \sigma_0(\frac N4) + \lceil\sigma_0(\frac
N8)/2\rceil$ full reflection groups of order~$N$, where $\sigma_0(x)=0$ if
$x$ is not an integer.

For $\grp L$, we must have $100=4(a^2+b^2)$ with $a\ge b \ge 0$. There
are two possibilities: $(a,b)=(5,0)$ or $(4,3)$.

For the full torus groups $\grp*$, the order would have to be a
multiple of 8; so there are no such groups of order $100$.

In total, we get $15+7+2=24$ achiral toroidal groups of order 100,
and 192 toroidal groups altogether.

$N=100$ does not occur as the order of any of the other types of
groups. So 192 is
the total number of 4-dimensional point groups of order $100$.

\paragraph{Enantiomorphic pairs.}
As an advanced exercise,
we can ask, how many of the 168 chiral groups or order 100 are their own mirror
image?

For the groups of type $\grp1$, we are looking for a lattice of
translations of size 100 that
has an \OR\ symmetry.
If it
is symmetric with respect to a horizontal axis, then,
according to
Lemma~\ref{lattice-with-axis},
the possibilities are an $m\times n$ rectangular grid of $mn$ points
or
a rhombic grid of $2mn$ points.
In this case, it is also symmetric with respect to a vertical axis.

Thus, we have to split $100=mn$ and $50=mn$ into two factors $m$ and
$n$. The order of the factors plays no role, because the reflection
\sym/ swaps the factors.
We have 5 possibilities for
$100
=1\cdot 100
=2\cdot 50
=4\cdot 25
=5\cdot 20
=10\cdot 10
$ and 3 possibilities for
$50
=1\cdot 50
=2\cdot 25
=5\cdot 10
$,
which gives
$5+3=8$ possibilities in total.
(Alternatively,
adding a vertical and horizontal mirror to such a translational subgroup
will
produce a group of type $\grp+^{\mathbf{p2mm}}$
or $\grp+^{\mathbf{c2mm}}$. So we can equivalently count the groups of these
types of order $4N=400$.)

There is also the possibility that the lattice is symmetric with
respect to a swapturn operation~\sym L. The number of these
groups equals the number of groups of type \grp L of order $4N=400$.
It can be computed as the number of integer points $(a,b)$ on the circle
$100=a^2+b^2$ with $a\ge b \ge 0$. There are
two possibilities: $(10,0)$ and $(8,6)$.

We have overcounted the lattices that are symmetric with respect to
both \sym+ and \sym L, in other words, the upright or slanted square
lattices.
There is one lattice of this type: the $10\times 10$ upright lattice.

In total, $8+2-1=9$ groups among the 113 groups of type \grp1 are equal to their own mirror.

For the groups of type
 $\grp.$, we can repeat the same game, except that we are looking for a translation lattice
 of half the size, 50.
For the lattices with \sym+ symmetry,
we have
3 possibilities for
$50
=1\cdot 50
=2\cdot 25
=5\cdot 10
$,
and 2 possibilities for
$25
=1\cdot 25
=5\cdot 5
$, giving $3+2=5$ possibilities in total.
There are two possibilities for
$50=a^2+b^2$ with $a\ge b \ge 0$: $(7,1)$ and $(5,5)$.
We have to subtract 1 for the slanted $5\times 5$ grid, for
a total of $5+2-1=6$ groups among the 48 flip groups.

The mirrors of the groups of type $\grp/$ are the groups of type
$\grp\setminus$, and hence none of them is its own mirror.
The groups of type $\grp X$ are easy to
handle:
The two parameters $m$ and $n$ must be equal.
 We have one such group, $\grp X^{\mathbf{c2mm}}_{5,5}$.
In total, $9+6+1=16$ chiral
 groups are their own mirror images.
The remaining $168-16=152$ chiral groups consist of enantiomorphic pairs.

\paragraph{The number of groups of order 7200.}

To look at a more interesting example, let us count
the groups of order 7200.
The count of toroidal groups follows the same calculation as above, and it
amounts to 19,319 chiral and 216 achiral groups.
In addition, we have 22 tubical groups:
$
\pm[I\times C_{60}],
\pm[I\times D_{60}],
\pm[O\times C_{150}],
\pm[O\times D_{150}],
\pm[T\times C_{300}],
\pm[T\times D_{300}],
\pm\frac1{2}[O\times D_{300}],
\pm\frac1{2}[O\times \overline D_{300}],
\pm\frac1{2}[O\times C_{300}],
\pm\frac1{6}[O\times D_{900}],
\pm\frac1{3}[T\times C_{900}],
$
and their mirrors.
Finally, there is one polyhedral group $\pm[I\times I]$.
In total, we have $19{,}319+22+1=19{,}342$ chiral groups and 216 achiral ones.

\paragraph{The number of groups of order at most $M$.}

While the number of groups of a given order~$N$ fluctuates between
a linear lower bound and a slightly superlinear upper bound,
the ``average number'' can be estimated quite precisely:
We have seen that the number of groups of order $N$ is of order
$\Theta(\sigma(N))$, where 
$\sigma(N)$ is the sum of divisors of~$N$.
If we look at all groups of order at most $M$, we can sum over all
potential divisors $d$ and get
\begin{displaymath}
  \sum_{N=1}^M \sigma(N) =
  \sum_{d=1}^M d \lfloor M/d\rfloor = \Theta (M^2).
\end{displaymath}
Thus, the number of four-dimensional groups of order at most $M$
is $\Theta (M^2)$. The majority of these groups is chiral, but the
achiral ones alone are already of the order
$\Theta (M^2)$: There is essentially one swapturn
group for each integer point $(a,b)$ in the disk $a^2+b^2\le M/4$,
with roughly a factor 8 of overcounting of symmetric points, and this gives
$\Theta (M^2)$ chiral groups.

\section{The crystallographic point groups}
\label{sec:crystallographic}

\begingroup\setlength{\extrarowheight}{1.8pt}
   \begin{table}[phtb]
     \centering
     \begin{tabular}[t]{|ll@{}r|}
\hline
&&\llap{order} \\
\hline
01/01& $\grp1_{1,1}$&1\\
01/02& $\grp1_{2,1}$&2\\
\hline
02/01& $\grp|^{\mathbf{pg}}_{1,1}$&2\\
02/02& $\grp|^{\mathbf{pm}}_{1,1}$&2\\
02/03& $\grp|^{\mathbf{cm}}_{1,1}$&4\\
\hline
03/01& $\grp/^{\mathbf{cm}}_{1,1}$&2\\
03/02& $\grp/^{\mathbf{pm}}_{2,2}$&4\\
\hline
04/01& $\grp+^{\mathbf{p2gg}}_{1,1}$&4\\
04/02& $\grp+^{\mathbf{p2mg}}_{1,1}$&4\\
04/03& $\grp+^{\mathbf{p2mm}}_{1,1}$&4\\
04/04& $\grp+^{\mathbf{c2mm}}_{1,1}$&8\\
\hline
05/01& $\grp X^{\mathbf{c2mm}}_{1,1}$&4\\
05/02& $\grp X^{\mathbf{p2mm}}_{2,2}$&8\\
\hline
06/01& $\grp+^{\mathbf{p2mg}}_{2,1}$&8\\
06/02& $\grp+^{\mathbf{p2mm}}_{2,1}$&8\\
06/03& $\grp+^{\mathbf{p2mm}}_{2,2}$&16\\
\hline
07/01& $\grp1_{1,4}^{(1)}$&4\\
07/02& $\grp1_{1,4}^{(0)}$&4\\
07/03& $\grp1_{2,4}^{(0)}$&8\\
07/04& $\grp|^{\mathbf{cm}}_{1,2}$&8\\
07/05& $\grp|^{\mathbf{pm}}_{1,4}$&8\\
07/06& $\grp|^{\mathbf{pg}}_{1,4}$&8\\
07/07& $\grp|^{\mathbf{pm}}_{2,4}$&16\\
\hline
08/01& $\grp1_{1,3}^{(0)}$&3\\
08/02& $\grp1_{2,3}^{(0)}$&6\\
08/03& $\grp|^{\mathbf{pg}}_{1,3}$&6\\
08/04& $\grp|^{\mathbf{pm}}_{1,3}$&6\\
08/05& $\grp|^{\mathbf{cm}}_{1,3}$&12\\
\hline
09/01& $\grp1_{1,6}^{(0)}$&6\\
09/02& $\grp1_{1,6}^{(2)}$&6\\
09/03& $\grp1_{2,6}^{(0)}$&12\\
09/04& $\grp|^{\mathbf{pg}}_{1,6}$&12\\
09/05& $\grp|^{\mathbf{pm}}_{1,6}$&12\\
09/06& $\grp|^{\mathbf{pm}}_{2,3}$&12\\
09/07& $\grp|^{\mathbf{pm}}_{2,6}$&24\\
\hline
10/01& $\grp/^{\mathbf{pg}}_{2,2} \bigm| \grp\setminus^{\mathbf{pg}}_{2,2}$&4\\
\hline
11/01& $\grp1_{1,3}^{(1)} \bigm| \grp1_{3,1}$&3\\
11/02& $\grp1_{2,3}^{(-1)}\! \bigm| \grp1_{6,1}$&6\\
\hline
12/01& $\grp L_{1,0}$&4\\
12/02& $\grp L_{1,1}$&8\\
12/03& $\grp*^{\mathbf{p4gm}\textrm{U}}_{1}$&8\\
12/04& $\grp*^{\mathbf{p4mm}\textrm{U}}_{1}$&8\\
12/05& $\grp*^{\mathbf{p4mm}\textrm{S}}_{1}$&16\\
\hline
\end{tabular}
\nobreak %
\hfill\nobreak
\begin{tabular}[t]{|ll@{\,}r|}
\hline
&&\llap{order} \\
\hline
13/01& $\grp|^{\mathbf{pm}}_{4,1}$&8\\
13/02& $\grp|^{\mathbf{cm}}_{2,1}$&8\\
13/03& $\grp._{1,4}^{(1)}$&8\\
13/04& $\grp._{1,4}^{(0)}$&8\\
13/05& $\grp|^{\mathbf{pm}}_{4,2}$&16\\
13/06& $\grp+^{\mathbf{p2mg}}_{4,1}$&16\\
13/07& $\grp+^{\mathbf{c2mm}}_{2,1}$&16\\
13/08& $\grp+^{\mathbf{p2mm}}_{4,1}$&16\\
13/09& $\grp._{2,4}^{(0)}$&16\\
13/10& $\grp+^{\mathbf{p2mm}}_{4,2}$&32\\
\hline
14/01& $\grp|^{\mathbf{pm}}_{3,1}$&6\\
14/02& $\grp|^{\mathbf{pg}}_{3,1}$&6\\
14/03& $\grp._{1,3}^{(0)}$&6\\
14/04& $\grp|^{\mathbf{cm}}_{3,1}$&12\\
14/05& $\grp._{2,3}^{(0)}$&12\\
14/06& $\grp+^{\mathbf{p2mg}}_{3,1}$&12\\
14/07& $\grp+^{\mathbf{p2mm}}_{3,1}$&12\\
14/08& $\grp+^{\mathbf{p2gg}}_{3,1}$&12\\
14/09& $\grp+^{\mathbf{p2mg}}_{1,3}$&12\\
14/10& $\grp+^{\mathbf{c2mm}}_{3,1}$&24\\
\hline
15/01& $\grp|^{\mathbf{pm}}_{6,1}$&12\\
15/02& $\grp|^{\mathbf{pg}}_{3,2}$&12\\
15/03& $\grp|^{\mathbf{pm}}_{3,2}$&12\\
15/04& $\grp._{1,6}^{(0)}$&12\\
15/05& $\grp._{1,6}^{(2)}$&12\\
15/06& $\grp+^{\mathbf{p2mg}}_{6,1}$&24\\
15/07& $\grp+^{\mathbf{p2mg}}_{2,3}$&24\\
15/08& $\grp|^{\mathbf{pm}}_{6,2}$&24\\
15/09& $\grp+^{\mathbf{p2mm}}_{6,1}$&24\\
15/10& $\grp+^{\mathbf{p2mm}}_{3,2}$&24\\
15/11& $\grp._{2,6}^{(0)}$&24\\
15/12& $\grp+^{\mathbf{p2mm}}_{6,2}$&48\\
\hline
16/01& $\grp X^{\mathbf{p2gm}}_{2,2} \bigm| \grp X^{\mathbf{p2mg}}_{2,2}$&8\\
\hline
17/01& $\grp/^{\mathbf{cm}}_{1,3} \bigm| \grp\setminus^{\mathbf{cm}}_{3,1}$&6\\
17/02& $\grp/^{\mathbf{pm}}_{2,6} \bigm| \grp\setminus^{\mathbf{pm}}_{6,2}$&12\\
\hline
18/01& $\grp X^{\mathbf{p2gg}}_{2,2}$&8\\
18/02& $\grp*^{\mathbf{p4gm}\textrm{S}}_{1}$&16\\
18/03& $\grp L_{2,0}$&16\\
18/04& $\grp X^{\mathbf{c2mm}}_{2,2}$&16\\
18/05& $\grp*^{\mathbf{p4mm}\textrm{U}}_{2}$&32\\
\hline
19/01& $\grp|^{\mathbf{pg}}_{2,4}$&16\\
19/02& $\grp1_{4,4}^{(0)}$&16\\
19/03& $\grp|^{\mathbf{pm}}_{4,4}$&32\\
19/04& $\grp+^{\mathbf{p2mg}}_{4,2}$&32\\
19/05& $\grp._{4,4}^{(0)}$&32\\
19/06& $\grp+^{\mathbf{p2mm}}_{4,4}$&64\\
\hline
\end{tabular}
\nobreak %
\hfill\nobreak
\begin{tabular}[t]{|ll@{\,}r|}
\hline
&&\llap{order} \\
\hline
20/01& $\grp1_{1,12}^{(3)}$&12\\
20/02& $\grp1_{1,12}^{(2)}$&12\\
20/03& $\grp|^{\mathbf{pg}}_{2,3}$&12\\
20/04& $\grp|^{\mathbf{pm}}_{4,3}$&24\\
20/05& $\grp1_{2,12}^{(4)}$&24\\
20/06& $\grp|^{\mathbf{pm}}_{3,4}$&24\\
20/07& $\grp._{1,12}^{(3)}$&24\\
20/08& $\grp|^{\mathbf{pg}}_{3,4}$&24\\
20/09& $\grp|^{\mathbf{cm}}_{2,3}$&24\\
20/10& $\grp|^{\mathbf{cm}}_{3,2}$&24\\
20/11& $\grp._{1,12}^{(2)}$&24\\
20/12& $\grp+^{\mathbf{p2gg}}_{3,2}$&24\\
20/13& $\grp+^{\mathbf{p2mg}}_{3,2}$&24\\
20/14& $\grp|^{\mathbf{pg}}_{2,6}$&24\\
20/15& $\grp|^{\mathbf{pm}}_{4,6}$&48\\
20/16& $\grp+^{\mathbf{p2mm}}_{4,3}$&48\\
20/17& $\grp+^{\mathbf{p2mg}}_{4,3}$&48\\
20/18& $\grp|^{\mathbf{pm}}_{6,4}$&48\\
20/19& $\grp._{2,12}^{(4)}$&48\\
20/20& $\grp+^{\mathbf{c2mm}}_{3,2}$&48\\
20/21& $\grp+^{\mathbf{p2mg}}_{6,2}$&48\\
20/22& $\grp+^{\mathbf{p2mm}}_{6,4}$&96\\
\hline
21/01& $\grp\setminus^{\mathbf{cm}}_{1,3} \bigm| \grp/^{\mathbf{cm}}_{3,1}$&6\\
21/02& $\grp\setminus^{\mathbf{pm}}_{2,6} \bigm| \grp/^{\mathbf{pm}}_{6,2}$&12\\
21/03& $\grp X^{\mathbf{c2mm}}_{1,3} \bigm| \grp X^{\mathbf{c2mm}}_{3,1}$&12\\
21/04& $\grp X^{\mathbf{p2mm}}_{2,6} \bigm| \grp X^{\mathbf{p2mm}}_{6,2}$&24\\
\hline
22/01& $\grp1_{3,3}^{(0)}$&9\\
22/02& $\grp1_{6,3}^{(-3)}$&18\\
22/03& $\grp|^{\mathbf{pg}}_{3,3}$&18\\
22/04& $\grp|^{\mathbf{pm}}_{3,3}$&18\\
22/05& $\grp._{3,3}^{(0)}$&18\\
22/06& $\grp|^{\mathbf{cm}}_{3,3}$&36\\
22/07& $\grp._{6,3}^{(-3)}$&36\\
22/08& $\grp+^{\mathbf{p2gg}}_{3,3}$&36\\
22/09& $\grp+^{\mathbf{p2mg}}_{3,3}$&36\\
22/10& $\grp+^{\mathbf{p2mm}}_{3,3}$&36\\
22/11& $\grp+^{\mathbf{c2mm}}_{3,3}$&72\\
\hline
23/01& $\grp1_{3,6}^{(0)}$&18\\
23/02& $\grp1_{6,6}^{(0)}$&36\\
23/03& $\grp|^{\mathbf{pm}}_{3,6}$&36\\
23/04& $\grp|^{\mathbf{pg}}_{3,6}$&36\\
23/05& $\grp._{3,6}^{(0)}$&36\\
23/06& $\grp|^{\mathbf{pm}}_{6,3}$&36\\
23/07& $\grp|^{\mathbf{pm}}_{6,6}$&72\\
23/08& $\grp._{6,6}^{(0)}$&72\\
23/09& $\grp+^{\mathbf{p2mm}}_{6,3}$&72\\
23/10& $\grp+^{\mathbf{p2mg}}_{6,3}$&72\\
23/11& $\grp+^{\mathbf{p2mm}}_{6,6}$&144\\
\hline
\end{tabular}

\caption{The 227 crystallographic point groups in four dimensions, part 1}
     \label{tab:crystallographic}
   \end{table}

   \begin{table}[htbp]
     \centering
\begin{tabular}[t]{|llr|}
\hline
&&\llap{order} \\
\hline
24/01& $+\frac1{12}[T\times T]$&12\\
24/02& $\pm\frac1{12}[T\times T]$&24\\
24/03& $+\frac1{12}[T\times \bar T]\cdot 2_3$&24\\
24/04& $+\frac1{12}[T\times \bar T]\cdot 2_1$&24\\
24/05& $\pm\frac1{12}[T\times \bar T]\cdot 2$&48\\
\hline
25/01& $+\frac1{12}[T\times T]\cdot 2_1$&24\\
25/02& $+\frac1{12}[T\times T]\cdot 2_3$&24\\
25/03& $+\frac1{24}[O\times O]$&24\\
25/04& $+\frac1{24}[O\times \bar O]$&24\\
25/05& $\pm\frac1{12}[T\times T]\cdot 2$&48\\
25/06& $\pm\frac1{24}[O\times O]$&48\\
25/07& $+\frac1{24}[O\times O]\cdot 2_1$&48\\
25/08& $+\frac1{24}[O\times \bar O]\cdot 2_1$&48\\
25/09& $+\frac1{24}[O\times \bar O]\cdot 2_3$&48\\
25/10& $+\frac1{24}[O\times O]\cdot 2_3$&48\\
25/11& $\pm\frac1{24}[O\times O]\cdot 2$&96\\
\hline
26/01& $\grp\setminus^{\mathbf{pg}}_{2,4} \bigm| \grp/^{\mathbf{pg}}_{4,2}$&8\\
26/02& $\grp X^{\mathbf{p2mg}}_{2,4} \bigm| \grp X^{\mathbf{p2gm}}_{4,2}$&16\\
\hline
27/01& $\grp1_{1,5}^{(1)}$&5\\
27/02& $\grp1_{2,5}^{(1)}$&10\\
27/03& $\grp._{1,5}^{(1)}$&10\\
27/04& $\grp._{2,5}^{(1)}$&20\\
\hline
28/01& $\grp\setminus^{\mathbf{pg}}_{2,6} \bigm| \grp/^{\mathbf{pg}}_{6,2}$&12\\
28/02& $\grp X^{\mathbf{p2mg}}_{2,6} \bigm| \grp X^{\mathbf{p2gm}}_{6,2}$&24\\
\hline
29/01& $\grp\setminus^{\mathbf{cm}}_{3,3} \bigm| \grp/^{\mathbf{cm}}_{3,3}$&18\\
29/02& $\grp\setminus^{\mathbf{pm}}_{6,6} \bigm| \grp/^{\mathbf{pm}}_{6,6}$&36\\
29/03& $\grp X^{\mathbf{c2mm}}_{3,3}$&36\\
29/04& $\grp L_{3,0}$&36\\
29/05& $\grp X^{\mathbf{p2mm}}_{6,6}$&72\\
29/06& $\grp L_{3,3}$&72\\
29/07& $\grp*^{\mathbf{p4gm}\textrm{U}}_{3}$&72\\
29/08& $\grp*^{\mathbf{p4mm}\textrm{U}}_{3}$&72\\
29/09& $\grp*^{\mathbf{p4mm}\textrm{S}}_{3}$&144\\
\hline
30/01& $\grp/^{\mathbf{pg}}_{2,6} \bigm| \grp\setminus^{\mathbf{pg}}_{6,2}$&12\\
30/02& $\grp X^{\mathbf{p2gg}}_{2,6} \bigm| \grp X^{\mathbf{p2gg}}_{6,2}$&24\\
30/03& $\grp\setminus^{\mathbf{cm}}_{2,6} \bigm| \grp/^{\mathbf{cm}}_{6,2}$&24\\
30/04& $\grp X^{\mathbf{p2gm}}_{2,6} \bigm| \grp X^{\mathbf{p2mg}}_{6,2}$&24\\
30/05& $\grp\setminus^{\mathbf{pg}}_{6,6} \bigm| \grp/^{\mathbf{pg}}_{6,6}$&36\\
30/06& $\grp X^{\mathbf{c2mm}}_{2,6} \bigm| \grp X^{\mathbf{c2mm}}_{6,2}$&48\\
30/07& $\grp\setminus^{\mathbf{cm}}_{6,6} \bigm| \grp/^{\mathbf{cm}}_{6,6}$&72\\
30/08& $\grp X^{\mathbf{p2mg}}_{6,6} \bigm| \grp X^{\mathbf{p2gm}}_{6,6}$&72\\
30/09& $\grp X^{\mathbf{p2gg}}_{6,6}$&72\\
30/10& $\grp X^{\mathbf{c2mm}}_{6,6}$&144\\
30/11& $\grp L_{6,0}$&144\\
30/12& $\grp*^{\mathbf{p4gm}\textrm{S}}_{3}$&144\\
30/13& $\grp*^{\mathbf{p4mm}\textrm{U}}_{6}$&288\\
\hline
\end{tabular}
\qquad
\begin{tabular}[t]{|llr|}
\hline
&&\llap{order} \\
\hline
31/01& $\grp L_{2,1}$&20\\
31/02& $\grp L_{3,1}$&40\\
31/03& $+\frac1{60}[I\times \bar I]$&60\\
31/04& $+\frac1{60}[I\times \bar I]\cdot 2_3$&120\\
31/05& $+\frac1{60}[I\times \bar I]\cdot 2_1$&120\\
31/06& $\pm\frac1{60}[I\times \bar I]$&120\\
31/07& $\pm\frac1{60}[I\times \bar I]\cdot 2$&240\\
\hline
32/01& $\grp/^{\mathbf{pg}}_{2,4} \bigm| \grp\setminus^{\mathbf{pg}}_{4,2}$&8\\
32/02& $\grp\setminus^{\mathbf{cm}}_{2,4} \bigm| \grp/^{\mathbf{cm}}_{4,2}$&16\\
32/03& $\grp X^{\mathbf{p2gg}}_{2,4} \bigm| \grp X^{\mathbf{p2gg}}_{4,2}$&16\\
32/04& $\grp X^{\mathbf{p2mg}}_{4,2} \bigm| \grp X^{\mathbf{p2gm}}_{2,4}$&16\\
32/05& $\pm\frac1{3}[T\times C_{3}] \bigm| \pm\frac1{3}[C_{3}\times T]$&24\\
32/06& $\grp X^{\mathbf{c2mm}}_{2,4} \bigm| \grp X^{\mathbf{c2mm}}_{4,2}$&32\\
32/07& $\grp X^{\mathbf{p2gg}}_{4,4}$&32\\
32/08& $\grp\setminus^{\mathbf{cm}}_{4,4} \bigm| \grp/^{\mathbf{cm}}_{4,4}$&32\\
32/09& $\grp*^{\mathbf{p4gm}\textrm{U}}_{2}$&32\\
32/10& $\grp X^{\mathbf{p2mm}}_{4,4}$&32\\
32/11& $\pm\frac1{6}[O\times D_{6}] \bigm| \pm\frac1{6}[D_{6}\times O]$&48\\
32/12& $\grp X^{\mathbf{c2mm}}_{4,4}$&64\\
32/13& $\grp*^{\mathbf{p4gm}\textrm{S}}_{2}$&64\\
32/14& $\grp*^{\mathbf{p4mm}\textrm{S}}_{2}$&64\\
32/15& $\grp L_{4,0}$&64\\
32/16& $\pm\frac1{3}[T\times T]$&96\\
32/17& $\grp*^{\mathbf{p4mm}\textrm{U}}_{4}$&128\\
32/18& $\pm\frac1{3}[T\times T]\cdot 2$&192\\
32/19& $\pm\frac1{3}[T\times \bar T]\cdot 2$&192\\
32/20& $\pm\frac1{6}[O\times O]$&192\\
32/21& $\pm\frac1{6}[O\times O]\cdot 2$&384\\
\hline
33/01& $\grp\setminus^{\mathbf{pg}}_{4,6} \bigm| \grp/^{\mathbf{pg}}_{6,4}$&24\\
33/02& $\grp/^{\mathbf{pg}}_{4,6} \bigm| \grp\setminus^{\mathbf{pg}}_{6,4}$&24\\
33/03& $\pm[C_{1}\times T] \bigm| \pm[T\times C_{1}]$&24\\
33/04& $\grp X^{\mathbf{p2gg}}_{4,6} \bigm| \grp X^{\mathbf{p2gg}}_{6,4}$&48\\
33/05& $\pm[C_{2}\times T] \bigm| \pm[T\times C_{2}]$&48\\
33/06& $\pm\frac1{2}[O\times C_{2}] \bigm| \pm\frac1{2}[C_{2}\times O]$&48\\
33/07& $\pm[C_{3}\times T] \bigm| \pm[T\times C_{3}]$&72\\
33/08& $\pm[D_{4}\times T] \bigm| \pm[T\times D_{4}]$&96\\
33/09& $\pm\frac1{2}[O\times D_{4}] \bigm| \pm\frac1{2}[D_{4}\times O]$&96\\
33/10& $\pm\frac1{2}[O\times C_{4}] \bigm| \pm\frac1{2}[C_{4}\times O]$&96\\
33/11& $\pm\frac1{2}[O\times D_{6}] \bigm| \pm\frac1{2}[D_{6}\times O]$&144\\
33/12& $\pm\frac1{2}[O\times \bar D_{8}] \bigm| \pm\frac1{2}[\bar D_{8}\times O]$&192\\
33/13& $\pm[T\times T]$&288\\
33/14& $\pm[T\times T]\cdot 2$&576\\
33/15& $\pm\frac1{2}[O\times O]$&576\\
33/16& $\pm\frac1{2}[O\times O]\cdot 2$&1152\\
\hline
---& $\grp X^{\mathbf{p2gm}}_{4,6} \bigm| \grp X^{\mathbf{p2mg}}_{6,4}$&48\\
---& $\grp X^{\mathbf{p2mg}}_{4,6} \bigm| \grp X^{\mathbf{p2gm}}_{6,4}$&48\\
---& $\pm[D_{6}\times T] \bigm| \pm[T\times D_{6}]$&144\\
\hline
\end{tabular}

     \caption{The 227 crystallographic point groups, part~2, and three
       pseudo-crystal groups}
     \label{tab:crystallographic2}
   \end{table}
\endgroup

Brown, Bülow, Neubüser, Wondratschek, Zassenhaus %
classified the four-dimensional crystallographic space groups in 1978~\cite{BBNWZ}.
They grouped them by the underlying point groups (geometric crystal
classes, or $\mathbb Q$-classes), and
assigned numbers to these groups.  The \cry\ point groups are
characterized as having some lattice that they leave invariant.

There are 227 crystallographic points groups, sorted into 33
crystal systems according to the holohedry, i.e., the symmetry group
of the underlying lattice.  Tables
\ref{tab:crystallographic}--\ref{tab:crystallographic2} give a
reference from the 227 groups in the list of~\cite[Table~1C, pp.~79--260]
{BBNWZ} to our notation (for the toroidal groups) or Conway and
Smith's notation (for the remaining groups). When appropriate, we list
two \ena\ groups.

The first classification of the four-dimensional \cry\ point groups
was obtained by Hurley in 1951~\cite{hurley}, see
Section~\ref{Fingerprinting}. A few mistakes were
later corrected~\cite{hurley66}.

\goodbreak
All these groups are subgroups of only four maximal groups:
\begin{itemize}
\item 
$31/07=\pm\frac1{60}[I\times \bar I]\cdot 2=[[3,3,3]]$ (the simplex and its
polar, order 240)
\item 
$33/16=\pm\frac1{2}[O\times O]\cdot 2=[3,4,3]$ (the 24-cell, order 1152).
Taking the permutations of
$(\pm1,\pm1,0,0)$ as the vertices of a 24-cell, this set
generates a lattice, and this lattice is invariant under the group.
The symmetries of the hypercube/cross-polytope,
$32/21=\pm\frac1{6}[O\times O]\cdot 2=[3,3,4]$, are contained in this
group as a subgroup.
\item 
$30/13=\grp*^{\mathbf{p4mm}\textrm{U}}_{6}=\pm\frac12[\bar D_{12} \times
\bar D_{12}]\cdot 2$, order 288.
The invariant lattice is the
Cartesian product of two hexagonal plane lattices.

\item 
$20/22=\grp+^{\mathbf{p2mm}}_{6,4}
=\pm\frac1{24}[D_{24} \times D_{24}^{(5)}]\cdot 2^{(0,0)}
$, order 96.
The invariant lattice is the
Cartesian product
of a hexagonal lattice and a square lattice.
\end{itemize}

The last three items in Table~\ref{tab:crystallographic2} are the
``pseudo crystal groups'' of Hurley~\cite{hurley66}: Each such group
consists of transformations that can individually occur in
crystallographic groups, but as a whole, it is not a crystallographic
group.  All its proper subgroups are crystallographic groups.

\section{Geometric interpretation of oriented great circles}
\label{sec:oriented-circles-interpretation}

Section~\ref{sec:oriented} introduced the notation
$\vec K_p^q$ to denote {oriented great circles} on $S^3$.
Here we give a geometric interpretation of the orientation.
In fact, we will give two equivalent
geometric interpretations. However, at the
boundary cases $p=q$ and $q=-p$, one or the other of the
interpretations loses its meaning, and
only by combining both
interpretations we get a consistent definition that covers all cases.

\begin{figure}[htb]
  \centering
  \includegraphics{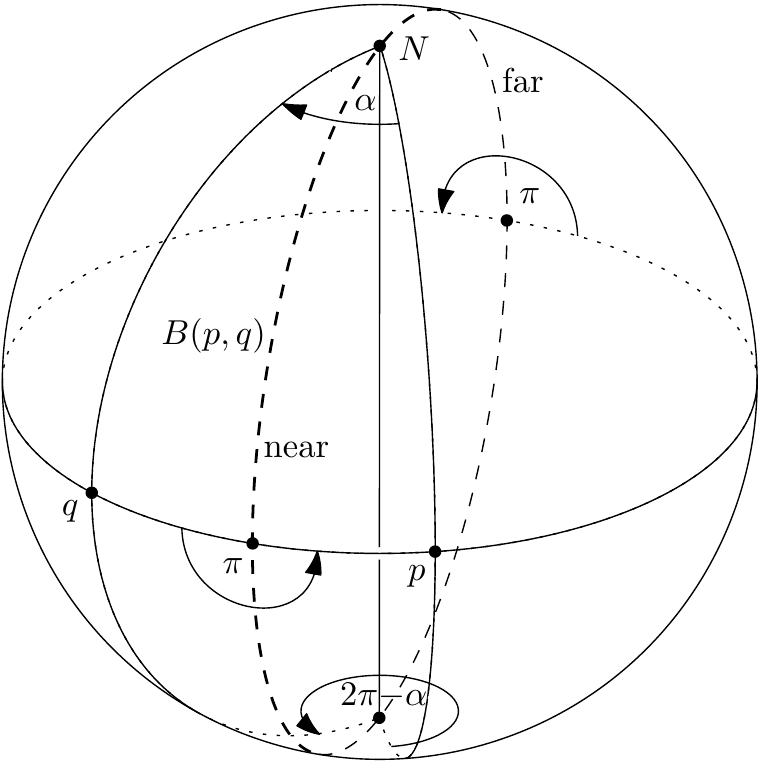}
\caption{The centers of the rotations mapping $p$ to $q$ lie on the
 bisecting circle $B(p,q)$.
}
  \label{fig:rotation-centers}
\end{figure}

We start from the definition~\eqref{eq:great-circles} of $K_p^q$ as
the
set of rotations $[x]$ that map $p$ to $q$ in $S^2$.
The centers $r$ of these rotations
lie on the
bisecting circle $B(p,q)$ between $p$ and $q$.
In Figure~\ref{fig:rotation-centers},
 we have drawn
$p$ and $q$ on the equator, with
 $p$ east of $q$.
If we observe the clockwise rotation angle $\phi$
as $r$ moves along $B(p,q)$, we see that $\phi$ has two extrema:
If the angular distance between $p$ and $q$ is $\alpha$,
the minimum clockwise angle $\phi=\alpha$ is achieved when $r$ is at the
North Pole.
The maximum $2\pi-\alpha$ is achieved at the South Pole.
The poles bisect $B(p,q)$ into two semicircles, the
\emph{near semicircle} and the
\emph{far semicircle}, according to the distance from $p$ and~$q$.

 \begin{figure}[htb]
   \centering
   \includegraphics{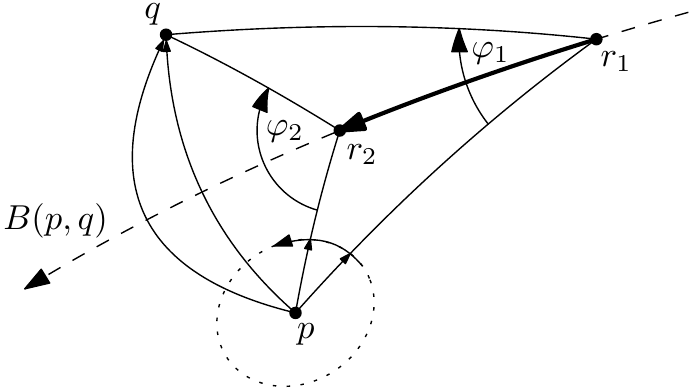}
   \caption{Orienting the great circle  $K^q_p$}
   \label{fig:oriented-circles}
 \end{figure}

To define an orientation,
we let %
$r$
move  continuously on $B(p,q)$, see
Figure~\ref{fig:oriented-circles} for an illustration on a small patch
of~$S^2$.
We make the movement %
in such a way that
\begin{compactenum}[(i)]
\item
    the rotation center $r$ moves in counterclockwise direction around $p$;
  \item simultaneously,
    the  clockwise rotation angle $\phi$
    increases when $r$ is on the near semicircle and
    decreases when $r$ is on the far semicircle.
  \end{compactenum}
  In Figure~\ref{fig:oriented-circles},
as $r$ moves from $r_1$ to $r_2$ along the thick arrow,
the angle $\phi$ increases from $\phi_1$ to~$\phi_2$.
These rules define an orientation of $B(p,q)$.

When we want to transfer
this orientation to $K_p^q$,
we must be aware of the $2:1$ relation between quaternions
$x=\cos \frac\phi2 + r\cdot \sin \frac\phi2$
and
rotations $[x]$ of $S^2$.
The angle $\phi$ is
defined only up to multiples of~$2\pi$, and hence a rotation
corresponds to two opposite quaternions $x$ and $-x$.  Thus, there are two
ways of defining a continuous dependence from $r$ via $\phi$ to~$x$.
Both possibilities lead to the same orientation of $K_p^q$, but
we can select one of them by restricting $\phi$ to the interval
$0\le\phi<2\pi$.
Once this mapping
is chosen, two opposite points $r$ and $-r$ on $B(p,q)$, which define
the same rotation $[r]$ of $S^2$, correspond to opposite quaternions
$x$ and $-x$ on $K_p^q$.
(The easiest way to check this is for the midpoint of $p$ and $q$ in
Figure~\ref{fig:rotation-centers} and the opposite point.
Both have the same rotation angle $\phi=\pi$.
Generally, the transition from $\phi$ to $2\pi-\phi$ changes
the sign of $\cos \frac\phi2$ and leaves $\sin \frac\phi2$ unchanged.)
Thus, as $r$ traverses $B(p,q)$, $x$ traverses $K_p^q$ once, and
this traversal defines the orientation $\vec K_p^q$.

The rules break down in the degenerate situations when $q=\pm
p$. Luckily, in each situation, there is one rule that works.
\begin{itemize}
\item When $p=q$, the only rotations centers are $r=p$ and $r=-p$. In
  this case, we can maintain rule (ii): We consider increasing rotation
  angles around $r=p$.
\item When $p=-q$, the rotation angle $\phi=180^\circ$ is constant, but
  we can stick to rule (i): The rotation centers $r$ lie on the %
  circle $B(p,-p)$ that has $p$ and $-p$ as poles, and we let them move
  counterclockwise around $p$.
\end{itemize}
Considering the definition~\eqref{eq:great-circles} of $K_p^q$, it is
actually surprising that 
$K_p^q$ makes a smooth transition when $q$ approaches~$p$: The locus
$B(p,q)$ of rotation centers changes discontinuously from a circle to a
set of opposite points.

When $p$ and $q$ are exchanged with $-p$ and $-q$, everything changes
its direction: A counterclockwise
movement of $r$ around $p$ becomes a clockwise movement when seen from
$-p$, and $r$ is on the near semicircle of $p$ and $q$ if it is on the
far semicircle of $-p$
and $-q$. Thus, $\vec K_{-p} ^{-q}$ has the opposite orientation.

\section{Subgroup relations between tubical groups}
\label{sec:subgroups-tubical}

Figure~\ref{fig:tubical_subgroup_structure} shows the subgroup
structure between different tubical groups.
Some types are included multiple times with different parameters
to indicate common supergroups.
However, all the types appear at least once with the parameter ``$n$''.
(Those are the ones in red.)

\begin{figure}[htb]
  \centering%
  {
  \includegraphics{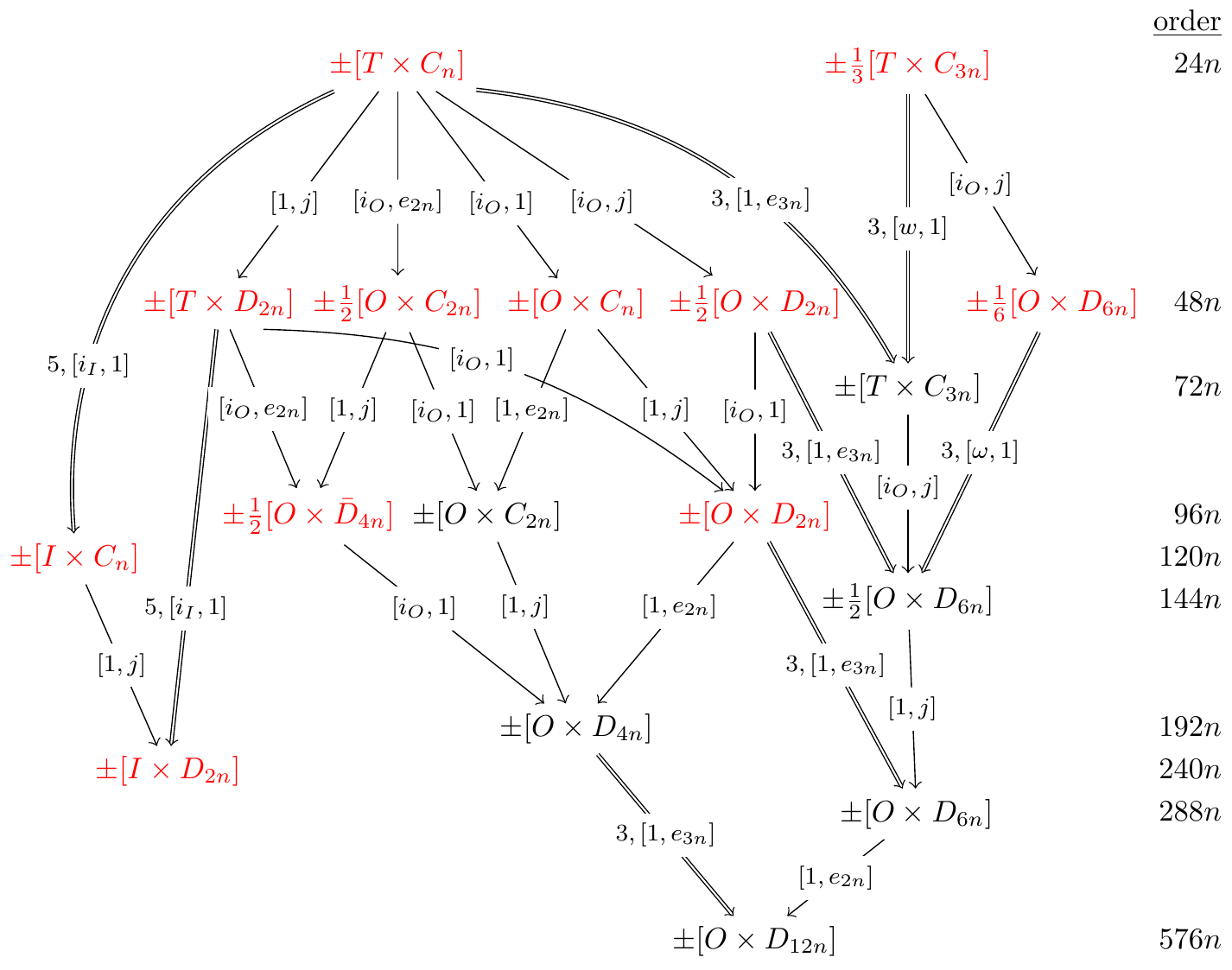}
}  \caption{
  Small-index containments between left tubical groups.
  Each arrow is marked with an extending element.
  Single arrows indicate index-2 containments.
  Double arrows denote index-3 or index-5 containments,
as specified with the extending element.
The red groups have the ``natural'' parameter $n$ (as in
Table~\ref{tbl:left_tubical_groups}).
Groups at the same horizontal level have the same order,
which is given in the rightmost column.
}
  \label{fig:tubical_subgroup_structure}
\end{figure}

\section{Conway and Smith's classification of the toroidal groups}
\label{conway-smith}

\begin{figure}[htb]
  \centering
  \includegraphics {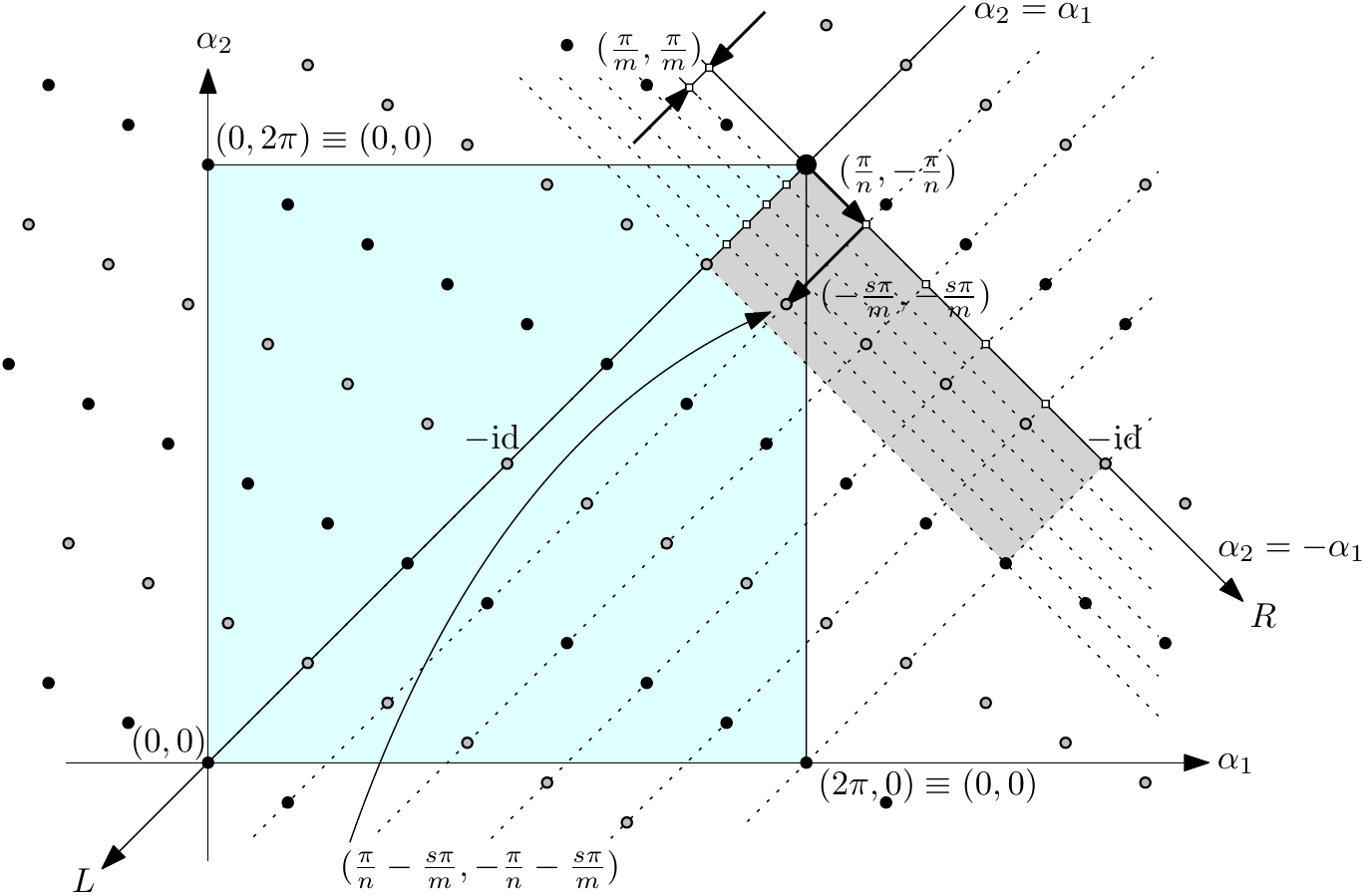}
  \caption{Parameterization of the translation groups in Conway and
    Smith.
    The black and gray points together form the diploid
    group $\pm \frac15[C_{15}^{(4)}\times C_5]
    =\grp1_{6,5}^{(-2)}$
    of order 30.
    The black points alone form the haploid
    group $+ \frac15[C_{15}^{(9)}\times C_5]
    =\grp1_{3,5}^{(-1)}$
    of order 15.}
  \label{fig:cs-translations}
\end{figure}

We describe the parameterization of the lattice translations for the
Conway--Smith classification of the groups of types $\pm[C\times C]$
and $+[C\times C]$
in geometric terms and relate them to our torus translations groups
(type \grp1). This might be interesting for readers who want to
study the classic classification for the toroidal groups and 
understand the connections.

As before, we describe the groups in terms of the lattice of torus
translations in the
$(\alpha_1,\alpha_2)$ coordinate system,
see Figure~\ref{fig:cs-translations}.
We put the origin at the top right corner $(2\pi,2\pi)$ because
the left rotations $[e_m,1]$ is a shift by $\pi/m$ along the negative
$\alpha_1=\alpha_2$ axis. This is the axis for the left rotations, and
we call it the $L$-axis.
The right rotations move on the
$\alpha_2=-\alpha_1$ axis in the southeast direction, and we call
this the $R$-axis.

We first describe the diploid groups
$\pm \frac1f[C_{m}^{(s)}\times C_n]
$, and we related them to our
groups
$\grp1_{m',n}^{(s')}$.
The left and right groups are determined by the grid formed by drawing
$\pm45^\circ$ lines through all points. If $2m$ grid lines cross
 the $L$-axis between $(0,0)$ and
$(2\pi,2\pi)$, then the left group is $C_m$.  Similarly, if there are
$2n$ grid intervals on the %
$R$-axis
between $(2\pi,2\pi)$ and $(4\pi,0)$,
(or equivalently, on the $-45^\circ$ diagonal of the square),
the %
right group is $C_n$.  The translation vectors on these diagonals form the left
kernel $C_{m/f}$ and the right kernel $C_{n/f}$.  The factor $f$ is
determined by the number of grid steps along the diagonal from one
point to the next.  In the picture, these are $f=5$ steps.  The
parameter $m'$ for our parameterization is hence $2m/f$.  The kernels
span a slanted rectangular grid; one rectangular box of this grid is shaded in the
picture. In terms of grid lines, the diagonal is an $f\times f$
square, and it contains exactly one point per grid line of either
direction, for a total of $f$ points (counting the four corners
only once).
In geometric terms,
Conway and Smith parameterize the lattice by looking at the first
grid line below the $L$-axis, as in our parameterization. They
measure $s$ as the number of grid steps to the first lattice point,
 starting from the $R$-axis in southwest direction.
The number $s$ must be relatively prime to~$f$, because otherwise,
 additional points on the $R$-axis would be generated.

By contrast, the parameter $s'$ in our setup (Figure~\ref{fig:translations})
is effectively measured
in the same units along the same diagonal line, but starting from the intersection with
the $\alpha_1$-axis, in the northeast direction.
Our parameterization is simpler because we don't specify in advance
the number of points on the $R$-axis. This allows us to freely choose
$s'$ within some range.

The group $\pm \frac1f[C_{m}^{(s)}\times C_n]$ is therefore generated by the
translation vectors
$[e_m^f,1]$ along the $L$-axis,
 $[1,e_n^f]$ along the $R$-axis,
and the additional vector
$[e_m^s,e_n]$.
(The second generator $[1,e_n^f]$ is actually redundant because
$[e_m^s,e_n]^f[e_m^f,1]^{-s} =
[1,e_n^f]$.)

For our group $\grp1_{m',n}^{(s')}$, the parameter $n$ is the same,
and $m'=2m/f$.
The parameter $s'$ can be computed as follows.
Choose generators for $\pm \frac1f[C_{m}^{(s)}\times C_n]$
as in Figure~\ref{fig:translations}.
These generators are then
$ t_1 = (\frac{f\pi}m, \frac{f\pi}m) $
and
$ t_2 = (\tfrac\pi{n} - \frac{s\pi}m + \frac{f\pi}m, 
-\tfrac\pi{n} - \frac{s\pi}m + \frac{f\pi}m)$.
Comparing them with the generators in
Proposition~\ref{prop:torus_translations}, we get
$s' = \frac{-m + (f-s)n}f$.

As mentioned in footnote~\ref{fn-mirror},
we have swapped the roles of the left and right groups with respect to Conway
and Smith's convention, to get a closer correspondence.
In the original convention of Conway and Smith,
the group
$\pm \frac1f[C_{m}\times C_n^{(s)}]$ is considered, whose third generator is
$[e_m,e_n^s]$. This group is the mirror of the
group $\pm \frac1f[C_n^{(s)}\times C_m]$.

A haploid group
$+ \frac1f[C_{m}^{(s)}\times C_n]
$ exists if both $m/f$ and $n/f$ are odd.
We modify the first generator to
$[e_m^{2f},1]$. This
omits every other point
on the $L$-axis (and on every line parallel to it) and thus avoids
 the point $(\pi,\pi)=-\id$.
In addition to being relatively prime to $f$,
$s$ must be odd, because otherwise, since
$[e_m^s,e_n]^{n}[e_m^{2f},1]^{-n/f\cdot s/2} = [1,e_n^{n}] =  [1,-1]$,
we would nevertheless generate the point $(\pi,\pi)=-\id$.

Reflection in the $L$-axis gives the same group. Hence
$\pm \frac15[C_{15}^{(4)}\times C_5]
\doteq
\pm \frac15[C_{15}^{(1)}\times C_5]
=\grp1_{6,5}^{(1)}$,
and
$+ \frac15[C_{15}^{(9)}\times C_5]
\doteq
+ \frac15[C_{15}^{(1)}\times C_5]
=\grp1_{3,5}^{(2)}$.
This reflection changes
the parameter $s$ to $f-s$
for the diploid groups
and to $2f-s$ for the haploid groups.
To eliminate these duplications,
the parameter $s$ should be constrained to
the interval
$0 \le s \le f/2$
for the diploid groups
and
$0 \le s \le f$
for the haploid groups.
As mentioned in footnote~\ref{fn-CS-escaped}, these constraints
are not stated in Conway and Smith. %
This concerns the last
four entries of~\cite[Table~4.2]{CS},
see Figure~\ref{fig:t42}.

With the help of the geometric picture of
 Figure~\ref{fig:cs-translations} for
 the parameterization of
Conway and Smith, one can give
 a geometric
 interpretation to the conditions
$s=fg\pm1$
of~\cite[pp.~52--53]{CS} for the last 4 lines of Table~4.3:
The condition
$s=fg-1$ expresses the fact that a square lattice is generated, as is necessary
for the torus swapturn groups \grp L (type $[D\times D]\cdot \bar2$).
The condition
$s=fg+1$
characterizes a rectangular lattice, as required for the
groups of type $\grp|$ and $\grp+$.
(Accordingly, for the two types of  groups
$\pm[C\times C]\cdot 2^{(\gamma)}$
and $+[C\times C]\cdot 2^{(\gamma)}$ in the upper half
of~\cite[pp.~53]{CS},
the condition $s=fg-1$ must be corrected to
$s=fg+1$, see~Figure~\ref{fig:p53}.)

\subsection{Index-4 subgroups of \texorpdfstring{$D_{4m}$}{D4m}}
\label{sec:index4}

There is one ambiguity that is notorious for causing oversights and omissions.
It arises when the group
$C_m$ is used as an 
index-4 subgroup of $D_{4m}$.

$D_{4m}$ is the chiral symmetry group of
a regular $2m$-gon $P_{2m}$ in space.  In Figure~\ref{fig:D20} we show such a $2m$-gon
with an alternating 2-coloring of its vertices.  $C_m$ is the normal subgroup of
rotations around the principal
axis, %
perpendicular to the polygon,
by multiples of $2\pi/m$ (those that respect the coloring).
$C_m$ has three cosets in $D_{4m}$: The ``cyclic coset'' $C_m'$ of rotations
by odd multiples of $\pi/m$ (those that swap the coloring), and two
``half-turn cosets''
$C_m^0$ and
$C_m^1$. %
One of these contains
the
half-turns through the vertices of $P_{2m}$ (the dashed axes, keeping
the colors), and the other %
the half-turns
through the edge midpoints of $P_{2m}$ (the dotted axes, swapping colors). However, when we rotate
$P_{2m}$ by $\pi/(2m)$, the involved groups and subgroups don't
change, and hence we see that 
$C_m^0$ and
$C_m^1$ are geometrically the same, whereas $C_m'$ is clearly
distinguishable (unless $m=1$).

\begin{figure}[htb]
  \centering
  \includegraphics{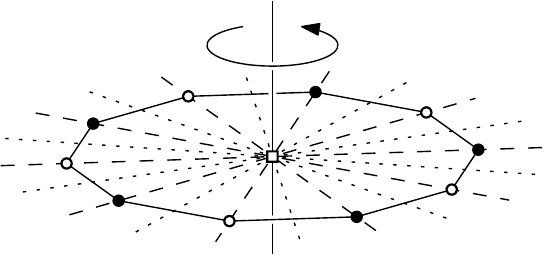}
  \caption{The operations of $D_{20}$ on a regular 10-gon $P_{10}$}
  \label{fig:D20}
\end{figure}

The case of
 the index-4 subgroups
 $C_m$ and $C_n$ of
$D_{4m}$ and $D_{4n}$ is denoted
in Conway and Smith~\cite{CS} by the notation
$\frac 14[D_{4m}\times D_{4n}]$, possibly with some decoration to
distinguish different cases.

The actual group is determined by an isomorphism between the cosets of
$D_{4m}/C_m$ and
$D_{4n}/C_n$.
For this there are two possibilities.
\begin{enumerate}[(a)]
\item 
 The cyclic coset $C_m'$ is matched with the
cyclic coset $C_n'$.
\item 
 The cyclic coset $C_m'$ and the
cyclic coset $C_n'$ are not matched to each other.  
\end{enumerate}

\paragraph{Goursat's omission.}

In the earliest enumeration by Goursat from 1889, the less natural possibility~(b)
has been overlooked. This was noted by
Threlfall and Seifert in 1931,
\cite[footnote 9 on p.~14%
]
{ThrS-I}\footnote
{Referring to Goursat's work:
  ``Gruppen dieser Substitutionen -- mit unseren Paargruppen 1-isomorph
 --
 sind mit
 einer Ausnahme (\S\,4 S.~18 Fu\ss note und \S\,4 S.~22) vollst\"andig
 angegeben.''
 (Groups of these substitutions -- which are 1-isomorphic to our pair
 groups -- are completely specified with one exception,
 see \S\,4 p.~18 footnote 13 and \S\,4 p.~22.)
In fact, in
footnote 13 on p.~18,
they %
use two such groups as an example of groups with equal
normal subgroups $L_0$ and $R_0$ that
are different already as abstract groups.
It is curious that Threlfall and Seifert, in the same paper, when they
came to the actual classification, overlooked this class of groups again. 
 They noted the gap themselves
 and filled it in part~II %
 \cite[pp.~585–586, Appendix II, Note 5]{ThrS-II}.
}
and by Hurley in 1951~\cite[bottom of p.~652]{hurley},\footnote
{``In the course of this calculation we find that Goursat has omitted one family
of groups. This omission appears to have passed unnoticed by subsequent writers.''}
who consequently extended the classification by adding an additional class
 XIII$'$
 of groups to Goursat's list.
 Du Val~\cite{duval} %
 followed
 Goursat and omitted case~(b) again.

\paragraph{A missed duplication in Conway and Smith.}

Conway and Smith \cite{CS} denote case (b) by adding a bar to the
second factor as follows:
\begin{equation}
\nonumber %
  \pm \tfrac 14[D_{4m}\times \bar D_{4n}]
  \text{ or }
 +\tfrac 14[D_{4m}\times \bar D_{4n}]
\end{equation}
When $n=1$, the distinction between case (a) and (b) disappears. $D_4$ is the
Vierergruppe,
whose nontrivial operations are half-turns around three perpendicular
axes,
and these elements are geometrically indistinguishable.

Conway and Smith 
express this succinctly in
the concluding sentence of their classification (see
Figure~\ref{fig:p53}):
``In the last eight lines, it is always permissible to replace
$D_2$ by $C_2$ and $\bar D_4$ by~$D_4$.''
However, this formulation in connection with the choice of notation
might lead an unwary reader into a trap:\footnote{Besides, the rule
  should also apply to entries that are not in the last eight lines of
  the tables. Accordingly, the constraint $n\ge 2$, which is stated
  for five of the eleven tubical groups in
Table~\ref{tbl:left_tubical_groups},
  should also be applied to the corresponding groups in
  \cite[Table 4.1]{CS}.
  For the group
$+ \frac12[D_{2m}\times C_{2n}]$
  in the penultimate line of Table 4.1, the obvious condition that $m$ and $n$
should be odd was forgotten. This omission has already been noted
by
Medeiros and Figueroa-O'Farrill~\cite[p.~1405]{atmp/spin}.
}
The choice (b) of an alternative mapping between the index-4 cosets in
$\frac14[D_{4m}\times  D_{4n}]$
 is not a property associated to
 $D_{4n}$ and its chosen normal subgroup, and it would
 more appropriate
to add the bar
to
the $\times$ operator or the whole
expression.
 The distinction disappears when
 at least \emph{one} of $D_{4m}$ and $D_{4n}$ is $D_4$, and hence,
 the bar can also be removed in a case like $[D_4\times \bar D_{4n}]$
 when the first factor is $D_4$.
 This duplication example has been treated in detail in
 Section~\ref{sec:dup-example}.

Conway and Smith use the bar notation $\bar D_{4n}$ also for something
different, namely in the index-2 case,
for example in
$\pm\frac{1}{2}[O\times\bar{D}_{4n}]$, see Table~\ref{tbl:left_tubical_groups}.
It indicates that,
as the kernel $R_0$ (or $L_0$)
of $D_{4n}$,
the normal
subgroup $D_{2n}$ is used,
as opposed to $C_{2n}$.
Also in this case, the
distinction disappears for $n=1$, but this time, it is a property of
the group $ D_{4n}$ and its normal subgroup, and hence the notation of
attaching the bar to
$ D_{4n}$ causes no confusion.

\paragraph{Another duplication in Conway and Smith.}

Our computer check unveiled another duplication
in Conway and Smith's classification.
It concerns the groups
$\grp+^{\mathbf{p2mg}}_{m,n}$ for $m=n$:
\begin{align*}
\grp+^{\mathbf{p2mg}}_{n,n} &\doteq
\pm\tfrac14[D_{2n}\times D_{2n}^{(1)}]\cdot 2^{(1,0)}
\doteq
\pm\tfrac14[D_{2n}\times D_{2n}^{(1)}]\cdot 2^{(1,1)}
&&\text{for even $n$}\\
\grp+^{\mathbf{p2mg}}_{n,n} &\doteq
+\tfrac12[D_{2n}\times D_{2n}^{(1)}]\cdot 2^{(0,0)}
\doteq
+\tfrac12[D_{2n}\times D_{2n}^{(1)}]\cdot 2^{(0,2)}
&&\text{for odd $n$}
\end{align*}
Neither of these duplications is warranted
according to the equalities listed
in~\cite[pp.~52--53]{CS}.
For example, for
$\pm\tfrac14[D_{2n}\times D_{2n}^{(s)}]\cdot 2^{(\alpha,\beta)}$ in the
first line,
we need a transition from $(\alpha,\beta)=(1,0)$
to  $(\alpha,\beta)=(1,1)$.
In this example, $f=2$ and $g=0$.
The only rule
according to~\cite[bottom of p.~52]{CS} that allows this change is the
transition
from $\langle s,\alpha,\beta\rangle$
to
$\langle s+f,\alpha,\beta-\alpha\rangle$
(see Figure~\ref{fig:p52}),
but it comes with a simultaneous change of  $s$ from
$s=1$ to $s+f=3$. The parameter $s$~is regarded modulo $2f=4$.

We did not investigate the reason for this duplication.
Since $f=2$ in both cases, it may have to do with
``\dots\  the easy cases when $f\le2$, which we exclude''
\cite[p.~52, line 2]{CS},
see Figure~\ref{fig:p52}.

The book of Conway and Smith~\cite{CS} is otherwise a very nice book
on topics related to quaternions and octonions, but it suffers from a %
concentration of mistakes near the end of Chapter~4, in particular,
concerning the achiral groups.  As an
``erratum'' to \cite[\S4]{CS}, we attach in
Figures~\ref{fig:t42}--\ref{fig:p53} the Tables 4.1--4.2 and the last
three pages of Chapter~4 of \cite{CS} with our additional
explanations and corrections, as far
as we could ascertain them, but we certainly did not fix all problems.

\begin{figure}[p]
  \centering
  \fbox{\vbox{\vskip-7mm
      \includegraphics[page=1,scale=0.5]{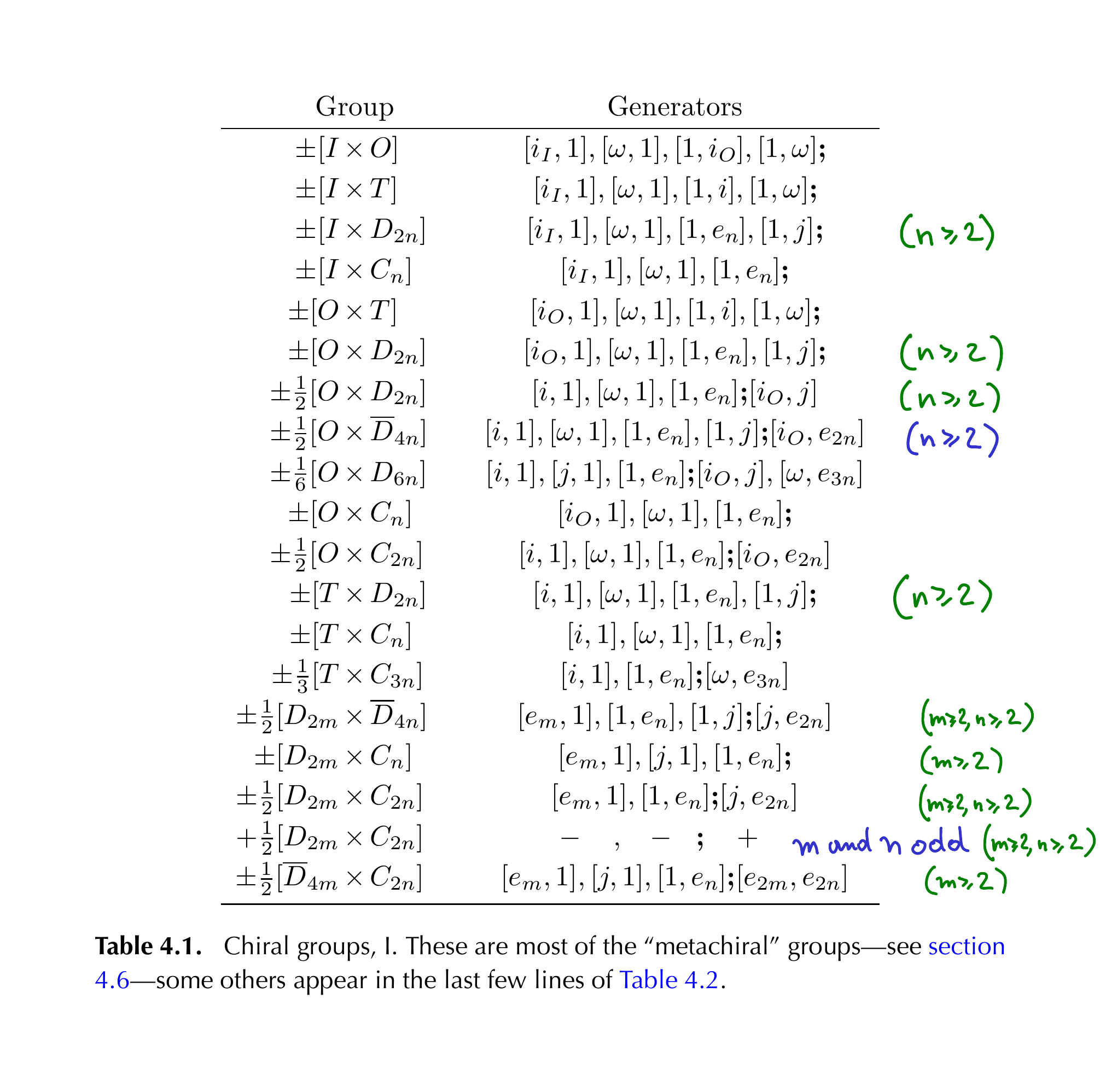}\vskip-7mm}}
  \bigbreak\vfill
  \fbox{\vbox{\vskip-13mm
      \includegraphics[page=2,scale=0.9]{figures/Conway-Smith-corrected-Tables.pdf}
      \vskip-12mm}}
  \caption{Corrections and remarks for \cite[Tables 4.1 and 4.2, p.~44 and 46]{CS}.}
  \label{fig:t42}
\end{figure}

\begin{figure}[p]
  \centering
  \fbox{\vbox{\vskip-1mm
      \includegraphics[page=1,scale=0.75]{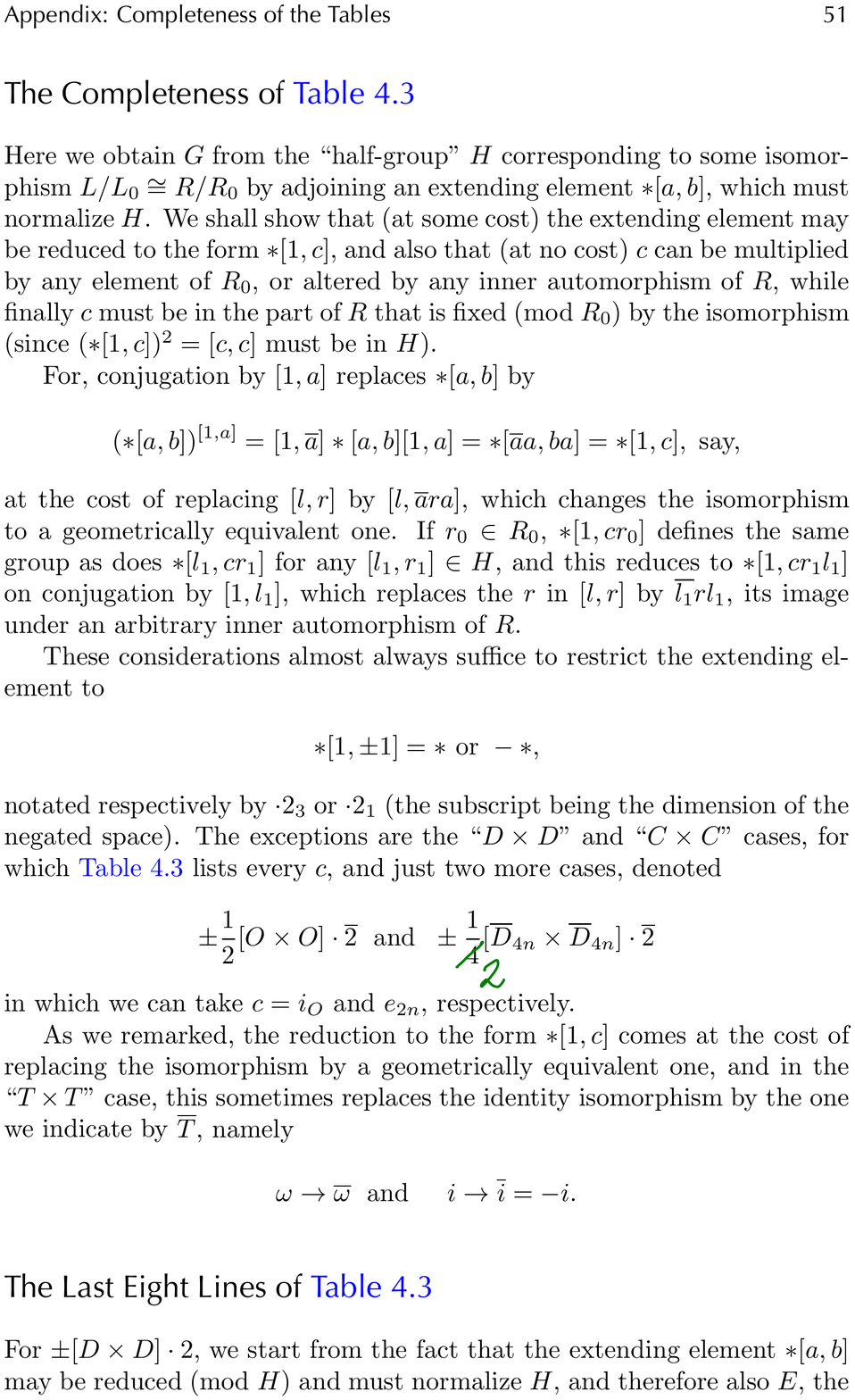}
\vskip-3mm}}
  \caption{Corrections for \cite[p.~51]{CS}.}
  \label{fig:p51}
\end{figure}
\begin{figure}[p]
  \centering
  \fbox{\includegraphics[page=2,scale=0.75]{figures/Conway-Smith-corrected.pdf}}
  \caption{Corrections and remarks for \cite[p.~52]{CS}.}
  \label{fig:p52}
\end{figure}
\begin{figure}[p]
  \centering
  \fbox{\includegraphics[page=3,scale=0.75]{figures/Conway-Smith-corrected.pdf}}
  \caption{Corrections and remarks for \cite[p.~53]{CS}.}
  \label{fig:p53}
\end{figure}

\end{document}